\pdfoutput=1 
\documentclass[12pt,twoside]{article}
\usepackage{amsmath,amssymb,mathtools}
\usepackage{verbatim}
\usepackage{dsfont,bm,mathrsfs}
\usepackage{color} 
\usepackage{graphicx}
\usepackage{amsthm}
\usepackage[shortlabels]{enumitem}
\usepackage{esint} 
\usepackage[title]{appendix}
\usepackage[T1]{fontenc}

\usepackage{imakeidx}   
\makeindex[title=Index of words and phrases,intoc,options=-s \jobname.ist]

\usepackage{xcolor}
\definecolor{darkgreen}{rgb}{0,0.75,0}
\definecolor{darkred}{rgb}{0.75,0,0}
\definecolor{darkmagenta}{rgb}{0.5,0,0.5}
\usepackage[colorlinks,citecolor=darkgreen,linkcolor=darkred,urlcolor=darkmagenta]{hyperref}
\usepackage{url} 

\newtheorem{thm}{Theorem}[section]
\newtheorem{cor}[thm]{Corollary}
\newtheorem{lem}[thm]{Lemma}
\newtheorem{prop}[thm]{Proposition}

\theoremstyle{definition}
\newtheorem{defn}[thm]{Definition}
\newtheorem{assum}[thm]{Assumption}
\newtheorem{rmk}[thm]{Remark}

\newtheorem{example}[thm]{Example}

\newtheorem{framework}[thm]{Framework}

\newtheorem{notation}[thm]{Notation}
\numberwithin{equation}{section}
\numberwithin{figure}{section}

\allowdisplaybreaks 

\newcommand{\norm}[1]{\left\lVert#1\right\rVert}
\newcommand{\trinorm}[1]{{\left\vert\kern-0.25ex\left\vert\kern-0.25ex\left\vert #1
    \right\vert\kern-0.25ex\right\vert\kern-0.25ex\right\vert}}
\newcommand{\indicator}[1]{\mathds{1}_{#1}} 
\newcommand{\closure}[1]{\overline{#1}}  
\DeclareMathOperator*{\osc}{osc}
\newcommand{\abs}[1]{\left\lvert#1\right\rvert} 
\newcommand{\id}{\operatorname{id}}
\newcommand{\supp}{\operatorname{supp}}
\newcommand{\sgn}{\operatorname{sgn}} 
\newcommand{\contfunc}{C}
\newcommand{\SigmaAlgEM}{\mathcal{B}_{0}} 

\newcommand{\diam}{\operatorname{diam}}
\newcommand{\dist}{\operatorname{dist}}
\newcommand{\measure}{m}
\newcommand{\metric}{d}
\newcommand{\pmetric}{\widehat{R}}

\newcommand{\GSC}{\mathrm{GSC}}
\newcommand{\SG}{\mathrm{SG}} 
\newcommand{\rweight}{\rho} 
\newcommand{\hdim}{d_{\mathrm{f}}}
\newcommand{\pwalk}{d_{\mathrm{w},p}}

\newcommand{\whdim}{d_{\mathrm{f}}(\bm{\rweight})}

\newcommand{\core}{\mathcal{D}}
\newcommand{\ExtD}{\core^{\#}}
\newcommand{\CoreClosure}{\mathcal{F}^{0}}
 
\newcommand{\zero}{\mathbf{0}} 
\newcommand{\interior}{\operatorname{int}} 

\usepackage{fancyhdr}
\fancypagestyle{mypagestyle}{%
  \fancyhf{}
  \fancyhead[EC]{N. Kajino and R. Shimizu}
  \fancyhead[OC]{Contraction properties and differentiability of $p$-energy forms}
  \fancyhead[EL,OR]{\thepage} 
  \fancyfoot{}%

}
\newcommand{\mr}[1]{{\tt \href{http://mathscinet.ams.org/mathscinet-getitem?mr=#1}{MR#1}}}
\newcommand{\arxiv}[1]{{\tt \href{http://arxiv.org/abs/#1}{arXiv:#1}}}

\begin{document}

	\font\titlefont=cmbx14 scaled\magstep1
	\title{\titlefont Contraction properties and differentiability of \texorpdfstring{$p$-energy}{p-energy} forms with applications to nonlinear potential theory on self-similar sets}

	\author{
	Naotaka Kajino\footnote{\scriptsize Research Institute for Mathematical Sciences, Kyoto University, Kyoto, Japan. \texttt{nkajino@kurims.kyoto-u.ac.jp}}
	\, and \,
	Ryosuke Shimizu\footnote{\scriptsize Waseda Research Institute for Science and Engineering, Waseda University, Tokyo, Japan.}\ \footnote{\scriptsize Graduate School of Informatics, Kyoto University, Kyoto, Japan (current address). \texttt{r.shimizu@acs.i.kyoto-u.ac.jp}}
	}
	\date{November 28, 2026}
	\maketitle
	\vspace{-1.0cm}
	\begin{abstract} 
        We introduce a new contraction property, which we call the \emph{generalized $p$-contraction property}, for $p$-energy forms as generalizations of many well-known inequalities, such as $p$-Clarkson's inequality, the strong subadditivity and the Markov property in the theory of nonlinear Dirichlet forms, and show that any $p$-energy form satisfying $p$-Clarkson's inequality is Fr\'{e}chet differentiable. 
        We also verify the generalized $p$-contraction property for $p$-energy forms on fractals constructed by Kigami [\emph{Mem.\ Eur.\ Math.\ Soc.}\ \textbf{5} (2023)] and by Cao--Gu--Qiu [\emph{Adv.\ Math.}\ \textbf{405} (2022), no.\ 108517]. 
        As a general framework of $p$-energy forms taking the generalized $p$-contraction property into consideration, we introduce the notion of \emph{$p$-resistance form} and investigate fundamental properties of $p$-harmonic functions with respect to $p$-resistance forms. 
        In particular, some new estimates on scaling factors of self-similar $p$-energy forms on self-similar sets are obtained by establishing H\"{o}lder regularity estimates for $p$-harmonic functions, and the $p$-walk dimensions of any generalized Sierpi\'{n}ski carpet and the $D$-dimensional level-$l$ Sierpi\'{n}ski gasket are shown to be strictly greater than $p$. 
	\end{abstract}

\newpage 
\tableofcontents

\newpage 
\section{Introduction}\label{sec:intro}
In the late 1980s, Goldstein \cite{Gol87} and Kusuoka \cite{Kus87} independently constructed a Brownian motion (a canonical diffusion process) on the Sierpi\'{n}ski gasket (the left of Figure \ref{fig.fractals}) as a scaling limit of the simple random walks on pre-gaskets (approximating graphs), and Barlow--Perkins \cite{BP88} established detailed estimates called the sub-Gaussian heat kernel estimates for its transition density.
Subsequently, Kigami \cite{Kig89} directly constructed the Laplacian on the Sierpi\'{n}ski gasket as a scaling limit of the discrete Laplacians on pre-gaskets, and Fukushima--Shima \cite{FS92} indicated that the theory of Dirichlet forms was well-applicable to the field of analysis on fractals; more precisely, Fukushima and Shima gave a direct description of the regular symmetric Dirichlet form $(\mathcal{E}_{2},\mathcal{F}_{2})$ corresponding to the Friedrichs extension of Kigami's Laplacian, which is an analogue of the pair of the Dirichlet $2$-energy $\int\abs{\nabla u}^{2}\,dx \eqqcolon \mathcal{E}_{2}(u)$ and the associated $(1,2)$-Sobolev space $W^{1,2} \eqqcolon \mathcal{F}_{2}$ on smooth spaces, and used it to investigate the eigenvalue problems for Kigami's Laplacian\footnote{The results in \cite{FS92,Kig89} were proved for the $D$-dimensional level-$2$ Sierpi\'{n}ski gasket (Framework \ref{frmwrk:SG}), where $D \in \mathbb{N}$ with $D \ge 2$.}.
Later, Kigami \cite{Kig93} extended the method in \cite{FS92} to \emph{post-critically finite self-similar sets} (Definition \ref{d:V0Vstar}), and Kusuoka--Zhou \cite{KZ92} constructed regular symmetric Dirichlet forms $(\mathcal{E}_{2},\mathcal{F}_{2})$ on a large class of self-similar sets including the Sierpi\'{n}ski carpet (the right of Figure \ref{fig.fractals}) through a subsequential scaling limit of discrete Dirichlet forms.
(The first construction of a Brownian motion on the Sierpi\'{n}ski carpet was done by Barlow--Bass \cite{BB89} by establishing a subsequential convergence of scaled Brownian motions on pre-carpets.)
See, e.g., \cite{Bar13,Kig01} for further background on the field of analysis on fractals.
As another advantage of the theory of Dirichlet forms, once we obtain a regular symmetric Dirichlet form $(\mathcal{E}_{2},\mathcal{F}_{2})$, we can capture the associated energy measure $\Gamma_{2}\langle u \rangle$ playing the role of $\abs{\nabla u}^{2}\,dx$ although the density ``$\abs{\nabla u}$'' usually does not make sense on fractals due to the singularity of $\Gamma_{2}\langle u \rangle$ with respect to the canonical volume measure (see \cite{Hin05,KM20} for details of this singularity of $\Gamma_{2}\langle u \rangle$). 

The main purpose of this article is to develop a general theory of $L^{p}$-analogues of $(\mathcal{E}_{2},\mathcal{F}_{2},\Gamma_{2}\langle \,\cdot\,\rangle)$, where $p \in (1,\infty)$, on the basis of a new contraction property which we call the \emph{generalized $p$-contraction property}.  
For a large class of triples $(K,m,p)$ of a self-similar set $K$, a natural self-similar measure $m$ on $K$ and $p \in (1,\infty)$, an $L^{p}$-analogue of $(\mathcal{E}_{2},\mathcal{F}_{2})$ on $(K,m)$, namely a $p$-energy form $(\mathcal{E}_{p},\mathcal{F}_{p})$ playing the role of $\int\abs{\nabla u}^{p}\,dx$ and the associated $(1,p)$-Sobolev space $W^{1,p}$, where $\mathcal{F}_{p}$ is a linear subspace of $L^{p}(K,m)$ and $\mathcal{E}_{p} \colon \mathcal{F}_{p} \to [0,\infty)$ is such that $\mathcal{E}_{p}^{1/p}$ is a seminorm on $\mathcal{F}_{p}$, has been constructed in several works \cite{CGQ22,HPS04,Kig23,KO+,MS+,Shi24}\footnote{The main difference among these works is the classes of $(K,m,p)$ on which $(\mathcal{E}_{p},\mathcal{F}_{p})$ is constructed. Let us briefly summarize what classes of $(K,m,p)$ are treated in these works (see \cite[Introduction]{KS.survey} for details). In \cite{CGQ22,HPS04}, $K$ is assumed to be a post-critically finite self-similar set (Definition \ref{d:V0Vstar}) so that the Sierpi\'{n}ski gasket is included while the Sierpi\'{n}ski carpet is excluded. The case where $K$ is the Sierpi\'{n}ski carpet is allowed in \cite{Kig23,KO+,MS+,Shi24}, but we need to assume that $p$ is strictly greater than the Ahlfors regular conformal dimension of $K$ (Definition \hyperref[it:defn.VD]{\ref{defn.AR}}-\ref{it:ARCdim}) in \cite{Kig23,KO+,Shi24}.}, most of which are very recent.
Furthermore, the associated $p$-energy measure $\Gamma_{p}\langle u \rangle$, which is a finite Borel measure on $K$ and an analogue of $\abs{\nabla u}^{p}\,dx$, has been introduced in \cite{MS+,Shi24} with the help of the self-similarity of $(\mathcal{E}_{p},\mathcal{F}_{p})$. See Section \ref{sec.ss} for details on the self-similarity of a $p$-energy form, and Example \ref{ex.em} for examples of $p$-energy measures which do not rely on the self-similarity.

Compared with the case of $p = 2$, where the theory of symmetric Dirichlet forms is applicable, very little has been established for $p \in (1,\infty) \setminus \{2\}$ in the direction of dealing with $(\mathcal{E}_{p},\mathcal{F}_{p},\Gamma_{p}\langle \,\cdot\, \rangle)$ in a general framework.
Such general developments that do not rely on an explicit representation of $(\mathcal{E}_{p},\mathcal{F}_{p})$ or the existence of the density ``$\abs{\nabla u}$'' are significant because the $p$-energy measure $\Gamma_{p}\langle u \rangle$ is typically singular with respect to the canonical volume measure in the case of fractals (see \cite{Hin05,KM20,Yan25+.sing}). To develop a reasonably satisfactory theory on $(\mathcal{E}_{p},\mathcal{F}_{p},\Gamma_{p}\langle \,\cdot\, \rangle)$, there are currently two missing pieces: first, useful contraction properties of it, and secondly, the (Fr\'{e}chet) differentiability of $\mathcal{E}_{p}$ and of $\Gamma_{p}$. 
In the first half of this paper (Sections \ref{sec.GC}--\ref{sec.ss}), we aim at establishing general results filling these missing pieces. We shall explain more details of the main results of these sections below.

\begin{figure}[tb]\centering
	\includegraphics[height=120pt]{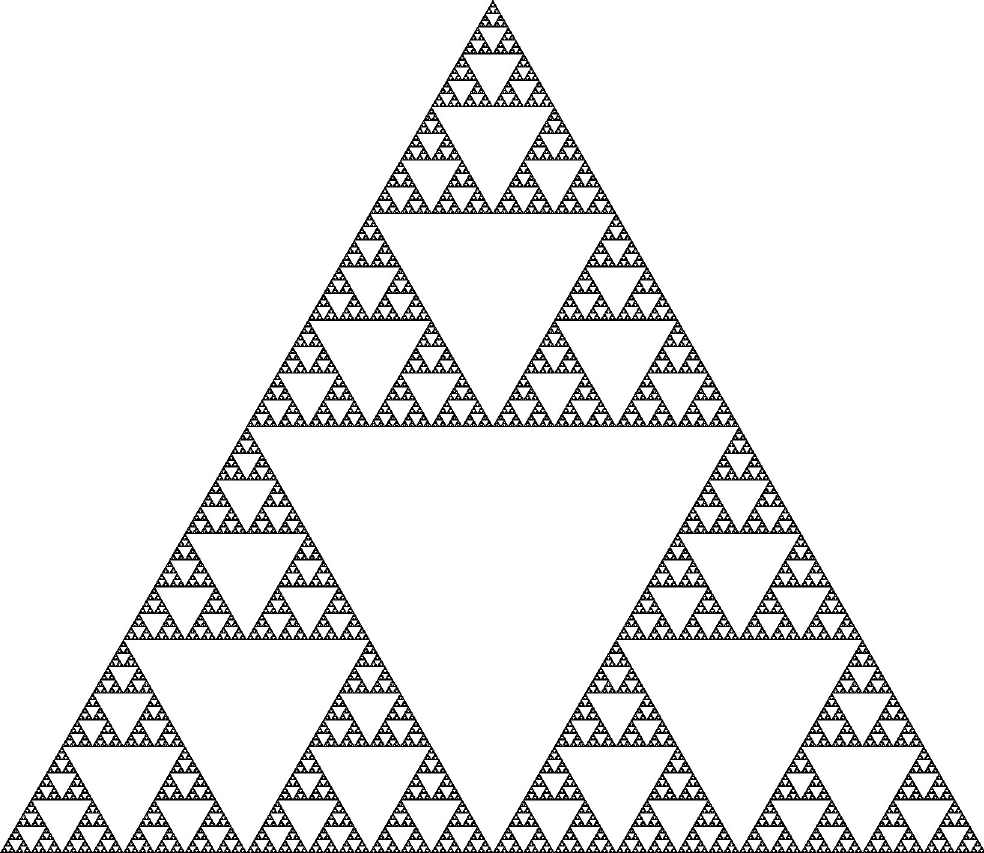}\hspace*{50pt}
	\includegraphics[height=120pt]{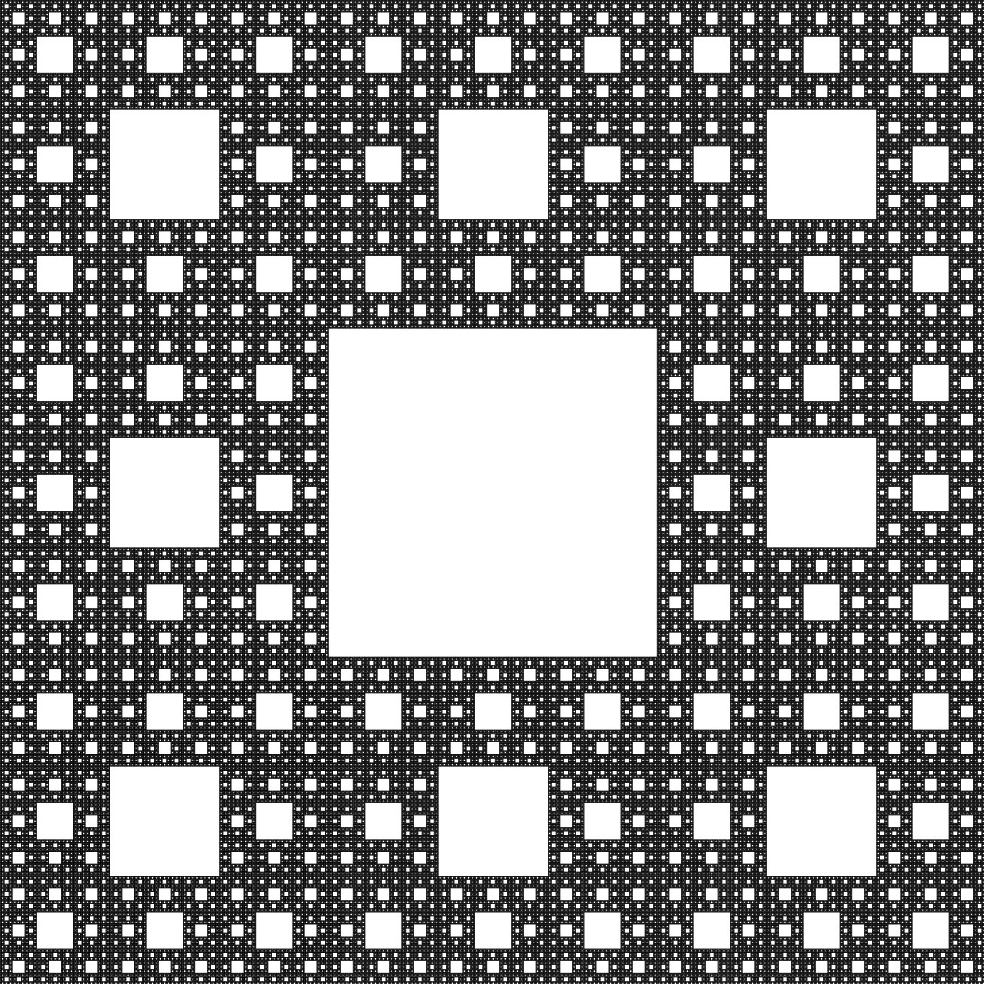}
	\caption{The Sierpi\'{n}ski gasket (left) and the Sierpi\'{n}ski carpet (right)}\label{fig.typicalfractals}
\end{figure}

The first missing piece is contraction properties of $(\mathcal{E}_{p},\mathcal{F}_{p},\Gamma_{p}\langle \,\cdot\, \rangle)$.
Every $p$-energy form $(\mathcal{E}_{p},\mathcal{F}_{p})$ constructed in the previous studies is known to satisfy the following \emph{unit contractivity:}
\begin{equation}\label{intro.unit}
    u^{+} \wedge 1 \in \mathcal{F}_{p} \quad \text{and} \quad \mathcal{E}_{p}(u^{+} \wedge 1) \le \mathcal{E}_{p}(u) \quad \text{for any $u \in \mathcal{F}_{p}$.}
\end{equation}
In the case of $p = 2$, by using some helpful expressions of $\mathcal{E}_{2}$, e.g., \cite[Lemma 1.3.4 and (3.2.12)]{FOT}, \eqref{intro.unit} can be improved to the following \emph{normal contractivity} (see, e.g., \cite[Theorem I.4.12]{MR}): if $n \in \mathbb{N}$ and $T \colon \mathbb{R}^{n} \to \mathbb{R}$ satisfy $\abs{T(x)} \le \sum_{k = 1}^{n}\abs{x_{k}}$ and $\abs{T(x) - T(y)} \le \sum_{k = 1}^{n}\abs{x_{k} - y_{k}}$ for any $x = (x_{1},\dots,x_{n}), y = (y_{1},\dots,y_{n}) \in \mathbb{R}^{n}$, then for any $\bm{u} = (u_{1},\dots,u_{n}) \in \mathcal{F}_{2}^{n}$ we have
\begin{equation}\label{intro.normal}
    T(\bm{u}) \in \mathcal{F}_{2} \quad \text{and} \quad \mathcal{E}_{2}(T(\bm{u}))^{\frac{1}{2}} \le \sum_{k = 1}^{n}\mathcal{E}_{2}(u_{k})^{\frac{1}{2}}. 
\end{equation}
It is natural to expect that $(\mathcal{E}_{p},\mathcal{F}_{p})$ for $p \in (1,\infty) \setminus \{2\}$ also has a similar property to \eqref{intro.normal} since $\mathcal{E}_{p}(u)$ is an analogue of $\int\abs{\nabla u}^{p}\,dx$; nevertheless, it is not clear whether \eqref{intro.unit} can be improved in such a way without going back to the constructions of $(\mathcal{E}_{p},\mathcal{F}_{p})$ in the previous studies.
Not only \eqref{intro.normal} but also other useful inequalities like the following \emph{strong subadditivity} and \emph{$p$-Clarkson's inequality}, were not mentioned in \cite{CGQ22,HPS04,Kig23,MS+,Shi24}: 
\begin{enumerate}[label=\textup{(\arabic*)},align=left,leftmargin=*,topsep=2pt,parsep=0pt,itemsep=2pt]
	\item [\textup{(Strong subadditivity)}] For any $u,v \in \mathcal{F}_{p}$, we have $u \vee v, u \wedge v \in \mathcal{F}_{p}$ and 
		\begin{equation}\label{sadd.intro}
			\mathcal{E}_{p}(u \vee v) + \mathcal{E}_{p}(u \wedge v) \le \mathcal{E}_{p}(u) + \mathcal{E}_{p}(v). 
		\end{equation}
	\item [\textup{($p$-Clarkson's inequality)}] For any $u,v \in \mathcal{F}_{p}$, 
		\begin{equation}\label{Cp.intro}
        \begin{cases}
            \mathcal{E}_{p}(u+v) + \mathcal{E}_{p}(u-v) \ge 2\bigl(\mathcal{E}_{p}(u)^{\frac{1}{p - 1}} + \mathcal{E}_{p}(v)^{\frac{1}{p - 1}}\bigr)^{p - 1} \, &\text{if $p \in (1,2]$,} \\
            \mathcal{E}_{p}(u+v) + \mathcal{E}_{p}(u-v) \le 2\bigl(\mathcal{E}_{p}(u)^{\frac{1}{p - 1}} + \mathcal{E}_{p}(v)^{\frac{1}{p - 1}}\bigr)^{p - 1} \, &\text{if $p \in (2,\infty)$.}
        \end{cases} \tag*{\textup{(Cla)$_p$}}
    \end{equation}
\end{enumerate}  
These inequalities play significant roles in the \emph{nonlinear potential theory} with respect to $(\mathcal{E}_{p},\mathcal{F}_{p})$.
For example, \eqref{sadd.intro} will be important to consider the $p$-capacity associated with $(\mathcal{E}_{p},\mathcal{F}_{p})$; see \cite[(H3)]{BV05}. 
Also, we will frequently use \ref{Cp.intro} in this paper; see Theorem \ref{thm.intro-EpCp} below for one of the most important consequences of \ref{Cp.intro}.
Since it is not known, unlike the case of $p = 2$, whether such desirable inequalities as \eqref{intro.normal}, \eqref{sadd.intro} and \ref{Cp.intro} are implied by the unit contractivity \eqref{intro.unit}, one needs to go back to the constructions of $(\mathcal{E}_{p},\mathcal{F}_{p})$ in the preceding works if one wishes to show them. 
The situation is similar for $p$-energy measures. 
While it is natural to expect that contraction properties of $(\mathcal{E}_{p},\mathcal{F}_{p})$ are inherited by the associated $p$-energy measures, in order to show them for $p$-energy measures, we need to recall how $p$-energy measures are constructed, partially because no canonical way to define $p$-energy measures for a given $p$-energy form $(\mathcal{E}_{p},\mathcal{F}_{p})$ is known (see \cite[Problem 10.4]{MS+}).

To overcome this situation, in this paper we develop a general theory of $p$-energy forms on the basis of the \emph{generalized $p$-contraction property}, which is arguably the strongest possible form of contraction properties of $p$-energy forms and defined as follows.
Throughout the rest of this section, we fix $p \in (1,\infty)$, a measure space $(X,\mathcal{B},m)$, and the pair $(\mathcal{E}_{p},\mathcal{F}_{p})$ of a linear subspace $\mathcal{F}_{p}$ of $L^{0}(X,m)$\footnote{We set $L^{0}(X,m) \coloneqq \{ \text{the $m$-equivalence class of $f$} \mid \text{$f \colon X \to \mathbb{R}$, $f$ is $\mathcal{B}$-measurable} \}$; see \eqref{L0-dfn}.} and a functional $\mathcal{E}_{p} \colon \mathcal{F}_{p} \to [0,\infty)$ which is \emph{$p$-homogeneous}, i.e., satisfies $\mathcal{E}_{p}(au) = \abs{a}^{p} \mathcal{E}_{p}(u)$ for any $u \in \mathcal{F}_{p}$ and any $a \in \mathbb{R}$.
The pair $(\mathcal{E}_{p},\mathcal{F}_{p})$ is said to be a \emph{$p$-energy form} on $(X,m)$ if and only if $\mathcal{E}_{p}^{1/p}$ is a seminorm on $\mathcal{F}_{p}$. 
\begin{defn}[Generalized $p$-contraction property; Definition \ref{defn.GC}] \label{defn.GC.intro}
    We say that $(\mathcal{E}_{p},\mathcal{F}_{p})$ satisfies the \emph{generalized $p$-contraction property}, \ref{intro.GC} for short, if and only if the following holds: if $n_{1},n_{2} \in \mathbb{N}$, $q_{1} \in (0,p]$, $q_{2} \in [p,\infty]$ and $T = (T_{1},\dots,T_{n_{2}}) \colon \mathbb{R}^{n_{1}} \to \mathbb{R}^{n_{2}}$ satisfy $T(0) = 0$ and $\norm{T(x)-T(y)}_{\ell^{q_{2}}} \le \norm{x - y}_{\ell^{q_{1}}}$ for any $x, y \in \mathbb{R}^{n_{1}}$, then for any $\bm{u} = (u_{1},\dots,u_{n_{1}}) \in \mathcal{F}_{p}^{n_{1}}$ we have
    \begin{equation}\label{intro.GC}
        T(\bm{u}) \in \mathcal{F}_{p}^{n_{2}} \quad \text{and} \quad
        \norm{\bigl(\mathcal{E}_{p}(T_{l}(\bm{u}))^{\frac{1}{p}}\bigr)_{l = 1}^{n_{2}}}_{\ell^{q_{2}}} \le \norm{\bigl(\mathcal{E}_{p}(u_{k})^{\frac{1}{p}}\bigr)_{k = 1}^{n_{1}}}_{\ell^{q_{1}}}. \tag*{\textup{(GC)$_{p}$}}
    \end{equation}
\end{defn}

Note that the particular case of \ref{intro.GC} for $(p,n_{1},n_{2},q_{1},q_{2}) = (2,n,1,1,p)$ is nothing but the normal contractivity \eqref{intro.normal}.
As recorded in the following proposition, \ref{intro.GC} is actually a generalization of many useful inequalities like \eqref{intro.normal}, \eqref{sadd.intro} and \ref{Cp.intro}. 
\begin{prop}[Proposition \ref{prop.GC-list}]\label{prop.intro-GClist}
	Let $\varphi \in \contfunc(\mathbb{R})$ satisfy $\varphi(0) = 0$ and $\abs{\varphi(t) - \varphi(s)} \le \abs{t - s}$ for any $s,t \in \mathbb{R}$.
    Assume that $(\mathcal{E}_{p},\mathcal{F}_{p})$ satisfies \ref{intro.GC}. Then the following hold. 
    \begin{enumerate}[label=\textup{(\alph*)},align=left,leftmargin=*,topsep=2pt,parsep=0pt,itemsep=2pt]
        \item\textup{(Triangle inequality and convexity)} $\mathcal{E}_{p}^{1/p}$ is a seminorm on $\mathcal{F}_{p}$, and $\mathcal{E}_{p}$ is convex on $\mathcal{F}_{p}$, i.e., for any $\lambda \in (0,1)$ and any $f,g \in \mathcal{F}_{p}$,
        	\begin{equation*}
        		\mathcal{E}_{p}(\lambda f + (1 - \lambda)g) \le \lambda\mathcal{E}_{p}(f) + (1 - \lambda)\mathcal{E}_{p}(g). 
        	\end{equation*}
        \item\textup{(Lipschitz contractivity)} $\varphi(u) \in \mathcal{F}_{p}$ and $\mathcal{E}_{p}(\varphi(u)) \le \mathcal{E}_{p}(u)$ for any $u \in \mathcal{F}_{p}$. 
        \item\textup{(Strong subadditivity)} Assume that $\varphi$ is non-decreasing.         
        Then for any $f,g \in \mathcal{F}_{p}$, 
        \begin{equation*}
            \mathcal{E}_{p}\bigl(f - \varphi(f - g)\bigr) + \mathcal{E}_{p}\bigl(g + \varphi(f - g)\bigr) \le \mathcal{E}_{p}(f) + \mathcal{E}_{p}(g). 
        \end{equation*}
        In particular, \eqref{sadd.intro} holds. 
        \item\textup{(Leibniz rule)} For any $f,g \in \mathcal{F}_{p} \cap L^{\infty}(X,m)$, we have
        \begin{equation*}
            f \cdot g \in \mathcal{F}_{p} \quad \text{and} \quad \mathcal{E}_{p}(f \cdot g)^{\frac{1}{p}} \le \norm{g}_{L^{\infty}(X,m)}\mathcal{E}_{p}(f)^{\frac{1}{p}} + \norm{f}_{L^{\infty}(X,m)}\mathcal{E}_{p}(g)^{\frac{1}{p}}. 
        \end{equation*} 
        \item\textup{($p$-Clarkson's inequality)} Let $f,g \in \mathcal{F}_{p}$. If $p \in (1,2]$, then  
        \[
        2\bigl(\mathcal{E}_{p}(f) + \mathcal{E}_{p}(g)\bigr)
        \ge \mathcal{E}_{p}(f+g) + \mathcal{E}_{p}(f-g)
        \ge 2\bigl(\mathcal{E}_{p}(f)^{\frac{1}{p - 1}} + \mathcal{E}_{p}(g)^{\frac{1}{p - 1}}\bigr)^{p - 1}.
        \]
        If $p \in [2,\infty)$, then 
        \[
        2\bigl(\mathcal{E}_{p}(f) + \mathcal{E}_{p}(g)\bigr)
        \le \mathcal{E}_{p}(f+g) + \mathcal{E}_{p}(f-g)
        \le 2\bigl(\mathcal{E}_{p}(f)^{\frac{1}{p - 1}} + \mathcal{E}_{p}(g)^{\frac{1}{p - 1}}\bigr)^{p - 1}. 
        \]
        In particular, \ref{Cp.intro} holds. 
    \end{enumerate}
\end{prop}

Since the generalized $p$-contraction property is introduced as arguably the strongest possible formulation of the contraction property of $(\mathcal{E}_{p},\mathcal{F}_{p})$, it is highly non-trivial whether $p$-energy forms constructed in the previous studies satisfy it. 
In Section \ref{sec.constr}, we see that the existing constructions of $p$-energy forms in the previous studies do yield ones satisfying \ref{intro.GC}. 
(See also \cite{KS.lim} for another approach, which is based on Korevaar--Schoen $p$-energy forms, to obtain $p$-energy forms satisfying \ref{intro.GC}.) 

In the rest of this section, we assume that $(\mathcal{E}_{p},\mathcal{F}_{p})$ is a $p$-energy form on $(X,m)$.
The other missing piece in the previous studies on $p$-energy forms is their differentiability, which should be useful to study \emph{$p$-harmonic functions} with respect to $\mathcal{E}_{p}$.
(See \cite[Problem 7.7]{KM23} and \cite[Conjecture 10.8]{MS+} for some motivations to investigate $p$-harmonic functions on fractals.)
In \cite{CGQ22,HPS04,Shi24}, $p$-harmonic functions are defined as functions minimizing $\mathcal{E}_{p}$ under prescribed boundary values.
However, it is still unclear how to give an equivalent definition of $p$-harmonic function in a weak sense due to the lack of a ``two-variable version'' $\mathcal{E}_{p}(u; \varphi)$ \cite[Problem 2 in Section 6.3]{Kig23}.
We shall recall the Euclidean case to explain the importance of this object.
Let $D \in \mathbb{N}$ and let $U$ be an open subset of $\mathbb{R}^{D}$.
A function $u \in W^{1,p}(\mathbb{R}^{D})$ is said to be $p$-harmonic on $U$ in the weak sense if and only if
\begin{equation}\label{RD.p-harm.weak}
    \int_{\mathbb{R}^{D}}\abs{\nabla u(x)}^{p - 2}\langle \nabla u(x), \nabla\varphi(x) \rangle_{\mathbb{R}^{D}}\,dx = 0 \quad \text{for every $\varphi \in \contfunc_{c}^{\infty}(U)$,}
\end{equation}
where $\langle \,\cdot\,,\,\cdot\, \rangle_{\mathbb{R}^{D}}$ denotes the inner product of $\mathbb{R}^{D}$.
It is well known that \eqref{RD.p-harm.weak} is equivalent to the variational equality
\begin{equation}\label{RD.p-harm.EL}
    \int_{\mathbb{R}^{D}}\abs{\nabla u(x)}^{p}\,dx = \inf\biggl\{ \int_{\mathbb{R}^{D}}\abs{\nabla v(x)}^{p}\,dx \biggm| \text{$v \in W^{1,p}(\mathbb{R}^{D})$, $v - u \in W_{0}^{1,p}(U)$} \biggr\}.
\end{equation}
The issue in considering an analogue of \eqref{RD.p-harm.weak} for $\mathcal{E}_{p}$ is that we do not have a satisfactory counterpart, $\mathcal{E}_{p}(u;\varphi)$, of $\int\abs{\nabla u}^{p - 2}\langle \nabla u, \nabla\varphi \rangle\,dx$ associated with $\mathcal{E}_{p}$.
As mentioned in \cite[(2.1)]{SW04}, the ideal definition of $\mathcal{E}_{p}(u;\varphi)$\footnote{Strichartz and Wong \cite{SW04} proposed an approach based on the \emph{subderivative} instead of \eqref{intro.diffble}, i.e., they defined $\mathcal{E}_{p}(u;\varphi)$ as the interval $\bigl[\mathcal{E}_{p}^{-}(u;\varphi), \mathcal{E}_{p}^{+}(u;\varphi)\bigr]$, where $\mathcal{E}_{p}^{\pm}(u;\varphi) \coloneqq \frac{d^{\pm}}{dt}\mathcal{E}_{p}(u + t\varphi)\bigr|_{t = 0}$.} is
\begin{equation}\label{intro.diffble}
    \mathcal{E}_{p}(u;\varphi) \coloneqq \frac{1}{p}\left.\frac{d}{dt}\mathcal{E}_{p}(u + t\varphi)\right|_{t = 0},
\end{equation}
but the existence of this derivative is unclear\footnote{The case of $p = 2$ is special because of the parallelogram law.
Indeed, $\mathcal{E}_{2}$ is known to be a quadratic form and hence $\mathcal{E}_{2}(u,v) \coloneqq \frac{1}{4}(\mathcal{E}_{2}(u + v) - \mathcal{E}_{2}(u - v))$ is a symmetric form satisfying \eqref{intro.diffble}.}
because the constructions of $\mathcal{E}_{p}$ in the previous studies include many steps such as the operation of taking a subsequential scaling limit of discrete $p$-energy forms.
Similarly, in respect of $p$-energy measures, no suitable way is known to define a ``two-variable version'' $\Gamma_{p}\langle u; \varphi \rangle$ which plays the role of $\abs{\nabla u}^{p - 2}\langle \nabla u, \nabla \varphi \rangle\,dx$. 
The ideal definition of $\Gamma_{p}\langle u; \varphi \rangle$ is similar to \eqref{intro.diffble}, i.e., for any Borel subset $A$ of $X$, 
\begin{equation}\label{intro.em-diffble}
	\Gamma_{p}\langle u; \varphi \rangle(A) \coloneqq \frac{1}{p}\left.\frac{d}{dt}\Gamma_{p}\langle u + t\varphi \rangle(A)\right|_{t = 0}. 
\end{equation}
Such a signed measure was discussed in \cite[Section 5]{BV05}, but the existence of the derivative in \eqref{intro.em-diffble} (in some uniform manner) was an assumption in \cite{BV05}; see \cite[(H4) and the beginning of Section 5]{BV05} for details. 
Similarly, in \cite{Cap07}, the (scale-invariant) elliptic Harnack inequality for $p$-harmonic functions on \emph{metric fractals} (\cite[Definition 2.3]{Cap07}) was proved under some assumptions including the existence of $\Gamma_{p}\langle u; \varphi \rangle$, which was called the \emph{measure-valued $p$-Lagrangian} and denoted by $\mathcal{L}^{(p)}(u,\varphi)$ in \cite{Cap07}. 
However, for situations where no explicit expression of the $p$-energy measure $\Gamma_{p}\langle u \rangle$ is available unlike the case of the Euclidean spaces, there is no proof of the existence of the derivative in \eqref{intro.em-diffble} in the literature. 
(The $p$-energy form on the Sierpi\'{n}ski gasket constructed in \cite{HPS04} is discussed in \cite[Section 5]{Cap07} as a concrete examples and it is stated in \cite[p.~1315]{Cap07} that ``we can define the corresponding Lagrangian $\mathcal{L}^{(p)}(u,v)$'', but we have been unable to find in the literature a rigorous proof of the existence of the derivatives in \cite[p.~1315]{Cap07} defining $\mathcal{E}_{g}(u,v)$ and in \cite[p.~1303, (L5)]{Cap07} defining $\mathcal{L}^{(p)}(u,v)$ for the $p$-energy form on the Sierpi\'{n}ski gasket obtained in \cite{HPS04}.)

As another main contribution of this paper, we make a key observation that $p$-Clarkson's inequality \ref{Cp.intro} implies the desired differentiability of $\mathcal{E}_{p}$. 
In addition to this result, we record basic properties of $\mathcal{E}_{p}(u; \varphi)$ given by \eqref{intro.diffble} in the following theorem. 
\begin{thm}[Proposition \ref{prop.diffble} and Theorem \ref{thm.p-form}]\label{thm.intro-EpCp}
	Assume that $(\mathcal{E}_{p},\mathcal{F}_{p})$ satisfies \ref{Cp.intro}.  
	Then the function $\mathbb{R} \ni t \mapsto \mathcal{E}_{p}(f + tg) \in [0, \infty)$ is differentiable for any $f,g \in \mathcal{F}_{p}$, and for any $c \in (0,\infty)$, 
	\begin{equation*}
    	\lim_{\delta \downarrow 0}\sup_{f,g \in \mathcal{F}_{p};\, \mathcal{E}_{p}(f) \leq c/(p-2)^{+},\, \mathcal{E}_{p}(g) \leq 1} \abs{\frac{\mathcal{E}_{p}(f + \delta g) - \mathcal{E}_{p}(f)}{\delta} - \frac{d}{dt}\mathcal{E}_{p}(f + tg)\biggr|_{t = 0}} = 0,  
    \end{equation*}
    where $c/0 \coloneqq \infty$. 
    Moreover, define $\mathcal{E}_{p}(\,\cdot\,;\,\cdot\,) \colon \mathcal{F}_{p} \times \mathcal{F}_{p} \to \mathbb{R}$ by $\mathcal{E}_{p}(f;g) \coloneqq \frac{1}{p}\frac{d}{dt}\mathcal{E}_{p}(f + tg)\bigr|_{t = 0}$, and let $a \in \mathbb{R}$, $f, f_{1}, f_{2}, g \in \mathcal{F}_{p}$ and $h \in \mathcal{E}_{p}^{-1}(0)$. 
    Then the following hold. 
    \begin{enumerate}[label=\textup{(\alph*)},align=left,leftmargin=*,topsep=2pt,parsep=0pt,itemsep=2pt]
    	\item $\mathcal{E}_{p}(f; f) = \mathcal{E}_{p}(f)$ and $\mathcal{E}_{p}(af; g) = \sgn(a)\abs{a}^{p - 1}\mathcal{E}_{p}(f; g)$.
    	\item The map $\mathcal{E}_{p}(f; \,\cdot\,) \colon \mathcal{F}_{p} \to \mathbb{R}$ is linear.
    	\item $\mathcal{E}_{p}(f; h) = 0$ and $\mathcal{E}_{p}(f + h; g) = \mathcal{E}_{p}(f;g)$. 
    	\item $\mathbb{R} \ni t \mapsto \mathcal{E}_{p}(f + tg; g) \in \mathbb{R}$ is strictly increasing if and only if $\mathcal{E}_{p}(g) > 0$.
    	\item $\abs{\mathcal{E}_{p}(f; g)} \le \mathcal{E}_{p}(f)^{\frac{p - 1}{p}}\mathcal{E}_{p}(g)^{\frac{1}{p}}$. 
    	\item\label{it:intro.localHol} $\abs{\mathcal{E}_{p}(f_1; g) - \mathcal{E}_{p}(f_2; g)} \le C_{p}\bigl(\mathcal{E}_{p}(f_1) \vee \mathcal{E}_{p}(f_2)\bigr)^{\frac{p - 1 - \alpha_{p}}{p}}\mathcal{E}_{p}(f_1 - f_2)^{\frac{\alpha_{p}}{p}} \mathcal{E}_{p}(g)^{\frac{1}{p}}$, where $\alpha_{p} \coloneqq \frac{1}{p} \wedge \frac{p - 1}{p}$ and $C_{p} \in (0,\infty)$ is a constant determined solely and explicitly by $p$.
    \end{enumerate}
\end{thm}

We also establish a similar result for $p$-energy measures as follows, which is the first rigorous result on the existence of the derivative in \eqref{intro.em-diffble} for $p$-energy measures on fractals. 
(Recall that the existence of $p$-energy measures in a general setting not assuming the self-similarity of the space and the $p$-energy form is unknown; see \cite[Problem 10.4]{MS+}.
Added in revision: this existence has recently been proved by Sasaya \cite{Sas26}; see also Example \ref{ex.em}-\ref{EX.EM-fractal} for some more details.)
\begin{thm}[Propositions \ref{prop.c-diff-em}, \ref{prop.em-holder} and Theorem \ref{thm.em-basic}]\label{thm.intro-emCp}
	Let $\SigmaAlgEM$ be a $\sigma$-algebra in $X$, and assume that $\{ \Gamma_{p}\langle u \rangle \}_{u \in \mathcal{F}_{p}}$ is a family of measures on $(X,\SigmaAlgEM)$ such that $\Gamma_{p}\langle f \rangle(X) \leq \mathcal{E}_{p}(f)$ for any $f \in \mathcal{F}_{p}$ and such that $(\Gamma_{p}\langle \,\cdot\, \rangle(A),\mathcal{F}_{p})$ is a $p$-energy form on $(X,m)$ satisfying \ref{Cp.intro} for any $A \in \SigmaAlgEM$. 
	Then $\mathbb{R} \ni t \mapsto \Gamma_{p}\langle f + tg \rangle(A) \in [0, \infty)$ is differentiable for any $f,g \in \mathcal{F}_{p}$ and any $A \in \SigmaAlgEM$, and for any $c \in (0,\infty)$, 
    \begin{equation*}
    	\lim_{\delta \downarrow 0}\sup_{A \in \SigmaAlgEM,\, f,g \in \mathcal{F}_{p};\, \mathcal{E}_{p}(f) \leq c/(p-2)^{+},\, \mathcal{E}_{p}(g) \leq 1} \abs{\frac{\Gamma_{p}\langle f + \delta g \rangle(A) - \Gamma_{p}\langle f \rangle(A)}{\delta} - \frac{d}{dt}\Gamma_{p}\langle f + tg \rangle(A)\biggr|_{t = 0}} = 0. 
    \end{equation*}
    Moreover, the set function $\Gamma_{p}\langle f; g \rangle \colon \SigmaAlgEM \to \mathbb{R}$ defined by $\Gamma_{p}\langle f; g \rangle(A) \coloneqq \frac{1}{p}\frac{d}{dt}\Gamma_{p}\langle f + tg \rangle(A)\bigr|_{t = 0}$ is a signed measure on $(X,\SigmaAlgEM)$ for any $f,g \in \mathcal{F}_{p}$, and the following hold for any $A \in \SigmaAlgEM$, any $a \in \mathbb{R}$ and any $f, f_{1}, f_{2}, g, h \in \mathcal{F}_{p}$ with $\Gamma_{p}\langle h \rangle(A) = 0$: 
    \begin{enumerate}[label=\textup{(\alph*)},align=left,leftmargin=*,topsep=2pt,parsep=0pt,itemsep=2pt]
    	\item $\Gamma_{p}\langle f; f \rangle(A) = \Gamma_{p}\langle f \rangle$ and $\Gamma_{p}\langle af; g \rangle(A) = \sgn(a)\abs{a}^{p - 1}\Gamma_{p}\langle f; g \rangle(A)$.
    	\item The map $\Gamma_{p}\langle f; \,\cdot\,\rangle(A) \colon \mathcal{F}_{p} \to \mathbb{R}$ is linear.
    	\item $\Gamma_{p}\langle f; h \rangle(A) = 0$ and $\Gamma_{p}\langle f + h; g \rangle(A) = \Gamma_{p}\langle f;g \rangle(A)$.  
    	\item $\mathbb{R} \ni t \mapsto \Gamma_{p}\langle f + tg; g \rangle(A) \in \mathbb{R}$ is strictly increasing if and only if $\Gamma_{p}\langle g \rangle(A) > 0$.
    	\item For any $\SigmaAlgEM$-measurable functions $\varphi,\psi \colon X \to [0,\infty]$,
		\begin{equation*}
			\int_{X}\varphi\psi\,d\abs{\Gamma_{p}\langle f; g \rangle}
			\le \left(\int_{X}\varphi^{\frac{p}{p - 1}}\,d\Gamma_{p}\langle f \rangle\right)^{\frac{p - 1}{p}}\left(\int_{X}\psi^{p}\,d\Gamma_{p}\langle g \rangle\right)^{\frac{1}{p}}.
		\end{equation*}
    	\item Let $\alpha_{p} = \frac{1}{p} \wedge \frac{p - 1}{p}$ and $C_{p}$ be the same constants as in Theorem \ref{thm.intro-EpCp}-\ref{it:intro.localHol}. Then 
    		\begin{align*}
    			&\abs{\Gamma_{p}\langle f_1; g \rangle(A) - \Gamma_{p}\langle f_2; g \rangle(A)} \\ 
    			&\le C_{p}\bigl(\Gamma_{p}\langle f_1 \rangle(A) \vee \Gamma_{p}\langle f_2 \rangle(A)\bigr)^{\frac{p - 1 - \alpha_{p}}{p}}\Gamma_{p}\langle f_1 - f_2 \rangle(A)^{\frac{\alpha_{p}}{p}} \Gamma_{p}\langle g \rangle(A)^{\frac{1}{p}}. 
    		\end{align*}
    \end{enumerate}
\end{thm}

In the second part of this paper (Sections \ref{sec.p-harm} and \ref{sec.compatible}), we aim at developing a general theory of $p$-energy forms taking \ref{intro.GC} into account and focusing on a ``low-dimensional'' setting.
Namely, we introduce the notion of \emph{$p$-resistance form} as defined in Definition \ref{defn.RFp.intro} below and establish fundamental properties of this class of $p$-energy forms, as a natural extension of the theory of resistance forms for $p=2$ introduced in \cite{Kig95} and developed further in \cite{Kig01,Kig12} by Kigami.
In the rest of this section, we consider the situation where $(\mathcal{B},m)$ is the pair of $2^{X} = \{ A \mid A \subseteq X \}$ and the counting measure on $X$, so that $L^{0}(X,m) = \mathbb{R}^{X}$; see also Remark \ref{rmk:wo-measure}. 
\begin{defn}[$p$-Resistance form; Definition \ref{defn.RFp}] \label{defn.RFp.intro}
	We say that $(\mathcal{E}_{p}, \mathcal{F}_{p})$ is a \emph{$p$-resistance form} on $X$ if and only if the following conditions hold:
    \begin{enumerate}[label=\textup{(RF\arabic*)$_p$},align=left,leftmargin=*,topsep=2pt,parsep=0pt,itemsep=2pt]
    \item $\mathcal{F}_{p}$ is a linear subspace of $\mathbb{R}^{X}$ containing $\indicator{X}$ and $\mathcal{E}_{p}(\,\cdot\,)^{1/p}$ is a seminorm on $\mathcal{F}_{p}$ satisfying $\{ u \in \mathcal{F}_{p} \mid \mathcal{E}_{p}(u) = 0 \} = \mathbb{R}\indicator{X}$.
    \item The quotient normed space $(\mathcal{F}_{p}/\mathbb{R}\indicator{X}, \mathcal{E}_{p}(\,\cdot\,)^{1/p})$ is a Banach space.
    \item For any $x,y \in X$ with $x \neq y$, there exists $u \in \mathcal{F}_{p}$ such that $u(x) \neq u(y)$.
    \item For any $x, y \in X$,
    \begin{equation*}
        R_{\mathcal{E}_{p}}(x,y) \coloneqq \sup\biggl\{ \frac{\abs{u(x) - u(y)}^{p}}{\mathcal{E}_{p}(u)} \biggm| u \in \mathcal{F}_{p} \setminus \mathbb{R}\indicator{X} \biggr\} < \infty.
    \end{equation*}
    \item $(\mathcal{E}_{p},\mathcal{F}_{p})$ satisfies the generalized $p$-contraction property \ref{intro.GC}.
    \end{enumerate}
\end{defn}

We verify that the $p$-energy forms on \emph{$p$-conductively homogeneous} compact metric spaces $(K,d)$ (Definition \ref{defn.pCH}) constructed by Kigami in \cite[Theorem 3.21]{Kig23}, where $p$ is assumed to be strictly greater than the Ahlfors regular conformal dimension of $(K,d)$ (Definition \ref{defn.AR}-\ref{it:ARCdim}), are $p$-resistance forms. 
In addition, we prove that the $p$-energy forms on post-critically finite self-similar sets constructed by Cao--Gu--Qiu in \cite[Proposition 5.3]{CGQ22} are also $p$-resistance forms for any $p \in (1,\infty)$ under the condition (\textbf{R}) in \cite[p.~18]{CGQ22}.  
See Section \ref{sec.constr} for details of the frameworks treated in \cite{CGQ22,Kig23}.
Similar to the case of $p = 2$, developing a general theory of $p$-resistance forms allows us to investigate $p$-energy forms provided by these broad frameworks in a synthetic manner.

It is immediate that if $(\mathcal{E}_{p},\mathcal{F}_{p})$ is a $p$-resistance form on $X$, then $R_{\mathcal{E}_{p}}(\,\cdot\,,\,\cdot\,)^{1/p}$ is a metric on $X$ and any function in $\mathcal{F}_{p}$ is a Lipschitz function on $X$ with respect to this metric. 
In the theory of resistance forms ($p = 2$), it is well known that $R_{\mathcal{E}_{2}}(\,\cdot\,,\,\cdot\,)$ is a metric, which is called the \emph{resistance metric} of the resistance form $(\mathcal{E}_{2},\mathcal{F}_{2})$; see \cite[Theorem 2.3.4]{Kig01} for a proof.
In view of this fact for $p = 2$, it is natural to seek the largest exponent $q$ such that $R_{\mathcal{E}_{p}}(\,\cdot\,,\,\cdot\,)^{q}$ is a metric.   
The following theorem gives the answer. 
\begin{thm}[Corollary \ref{cor.tri}]\label{thm.intro-pRM}
	If $(\mathcal{E}_{p},\mathcal{F}_{p})$ is a $p$-resistance form on $X$, then $R_{\mathcal{E}_{p}}(\,\cdot\,,\,\cdot\,)^{\frac{1}{p - 1}}$ is a metric on $X$. 
\end{thm}
The power $1/(p - 1)$ in Theorem \ref{thm.intro-pRM} is sharp; see Example \ref{ex:pmet.sharp}. 
Let us call $R_{\mathcal{E}_{p}}(\,\cdot\,,\,\cdot\,)^{\frac{1}{p - 1}}$ the \emph{$p$-resistance metric} of $(\mathcal{E}_{p},\mathcal{F}_{p})$. 
Theorem \ref{thm.intro-pRM} was proved in \cite{ACFP19,Her10} for the canonical $p$-energy forms (i.e., those given by \eqref{eq:p-RF-finite-graph}) on finite weighted graphs $(V,L)$ and in \cite{Shi21} for this class of forms on infinite graphs. 
Theorem \ref{thm.intro-pRM} establishes the same result for the first time for $p$-energy forms which are not of the form \eqref{eq:p-RF-finite-graph} and for ones on continuous spaces. 
Added in revision: in a very recent work \cite{PS25+}, Puchert and Schmidt have given an intriguing proof of Theorem \ref{thm.intro-pRM}. Their proof relies on a duality (a clever use of conjugate convex functions), which seems to capture the essence of why the power $\frac{1}{p-1}$ naturally appears. Moreover, they have introduced the notion of \emph{nonlinear resistance form}, which generalizes that of $p$-resistance form as observed in \cite[Proposition 5.11]{PS25+}, and shown a natural counterpart of Theorem \ref{thm.intro-pRM} for nonlinear resistance forms.

We also investigate $p$-harmonic functions with respect to $p$-resistance forms, which should be considered as part of \emph{nonlinear potential theory} under the condition that each point has a positive $p$-capacity. 
Let us explain some basic results in this introduction. 
The following definition is a natural analogue of \eqref{RD.p-harm.weak} (or of \eqref{RD.p-harm.EL}). 
\begin{defn}[$\mathcal{E}_{p}$-Harmonic function; see Definition \ref{dfn:part-harmonic}]
    Let $(\mathcal{E}_{p},\mathcal{F}_{p})$ be a $p$-resistance form on $X$ and let $B$ be a non-empty subset of $X$.
    A function $h \in \mathcal{F}_{p}$ is said to be \emph{$\mathcal{E}_{p}$-harmonic} on $X \setminus B$ if and only if 
    \[
    \mathcal{E}_{p}(h; \varphi) = 0 \quad \text{for any $\varphi \in \mathcal{F}_{p}$ with $\varphi|_{B} = 0$,}
    \]
    or equivalently (see Proposition \ref{prop.equiv} for this equivalence), 
    \[
    \mathcal{E}_{p}(h) = \inf\{ \mathcal{E}_{p}(u) \mid \text{$u \in \mathcal{F}_{p}$, $u|_{B} = h|_{B}$} \}. 
    \]
\end{defn}

A standard argument in variational analysis ensures the existence and  uniqueness of $\mathcal{E}_{p}$-harmonic functions with given boundary values. 
\begin{prop}[Part of Theorem \ref{thm.RF-exist}]\label{prop.intro-unique}
	Let $(\mathcal{E}_{p},\mathcal{F}_{p})$ be a $p$-resistance form on $X$ and let $B$ be a non-empty subset of $X$. 
	Define $\mathcal{F}_{p}|_{B} \coloneqq \{ u|_{B} \mid u \in \mathcal{F}_{p} \}$. 
	Then for any $u \in \mathcal{F}_{p}|_{B}$, there exists a unique function $h_{B}^{\mathcal{E}_{p}}[u] \in \mathcal{F}_{p}$ satisfying $h_{B}^{\mathcal{E}_{p}}[u]\bigr|_{B} = u$ and $\mathcal{E}_{p}(h_{B}^{\mathcal{E}_{p}}[u]) = \inf\{ \mathcal{E}_{p}(v) \mid \text{$v \in \mathcal{F}_{p}$, $v|_{B} = u$} \}$. 
\end{prop}

Using the (nonlinear) operator $h_{B}^{\mathcal{E}_{p}}[\,\cdot\,] \colon \mathcal{F}_{p}|_{B} \to \mathcal{F}_{p}$ given in Proposition \ref{prop.intro-unique}, we can introduce a new $p$-resistance form on the boundary set $B$, which is called the \emph{trace} of $(\mathcal{E}_{p},\mathcal{F}_{p})$ on $B$. 
This notion is at the core of our theory of $p$-resistance forms, and turns out to be a powerful tool especially when we work on post-critically finite self-similar sets; see Subsection \ref{sec.pcf} for example. 
Here we just record fundamental results on traces in the following theorem. 
\begin{thm}[Trace of $p$-resistance form; part of Theorem \ref{thm.RF-exist}]\label{thm.intro-trace}
	Let $(\mathcal{E}_{p},\mathcal{F}_{p})$ be a $p$-resistance form on $X$ and let $B$ be a non-empty subset of $X$. 
	Define $\mathcal{E}_{p}|_{B} \colon \mathcal{F}_{p}|_{B} \to [0,\infty)$ by $\mathcal{E}_{p}|_{B}(u) \coloneqq \mathcal{E}_{p}(h_{B}^{\mathcal{E}_{p}}[u])$ for $u \in \mathcal{F}_{p}|_{B}$. 
    Then $(\mathcal{E}_{p}|_{B}, \mathcal{F}_{p}|_{B})$ is a $p$-resistance form on $B$. 
    Furthermore, $R_{\mathcal{E}_{p}|_{B}} = R_{\mathcal{E}_{p}}|_{B \times B}$ and 
    \[
    \mathcal{E}_{p}|_{B}(u; v) = \mathcal{E}_{p}\bigl(h_{B}^{\mathcal{E}_{p}}[u]; h_{B}^{\mathcal{E}_{p}}[v]\bigr) \quad \text{for any $u,v \in \mathcal{F}_{p}|_{B}$.}
    \]
\end{thm}

Now let us state results on behavior of $\mathcal{E}_{p}$-harmonic functions. 
We start with \emph{comparison principles} for $\mathcal{E}_{p}$-harmonic functions, namely monotonicity properties of $h_{B}^{\mathcal{E}_{p}}[u]$ with respect to the boundary value $u$. 
Because of the \emph{nonlinearity} of the operator $h_{B}^{\mathcal{E}_{p}}$, a \emph{maximum principle} does not imply a comparison principle unlike the case of $p = 2$. 
Fortunately, by virtue of Proposition \ref{prop.intro-unique} and the strong subadditivity \eqref{sadd.intro}, we can prove the following \emph{weak comparison principle} for $\mathcal{E}_{p}$-harmonic functions (Proposition \ref{prop.cp1}): 
\begin{equation}\label{e:intro-cp}
	\text{If $\emptyset \neq B \subseteq X$ and $u, v \in \mathcal{F}_{p}|_{B}$ satisfy $u \le v$ on $B$, then $h_{B}^{\mathcal{E}_{p}}[u] \le h_{B}^{\mathcal{E}_{p}}[v]$ on $X$.}
\end{equation}
We also show a localized version of \eqref{e:intro-cp} under suitable assumptions (Proposition \ref{prop.cp2}).
Furthermore, by employing the approach in \cite{Cap07}, we show the following \emph{(scale-invariant) elliptic Harnack inequality} for non-negative $\mathcal{E}_{p}$-harmonic functions under some extra assumptions including the existence of nice $p$-energy measures (see Theorem \ref{thm.EHI} for the precise statement): there exists a constant $C \in (0,\infty)$ such that for any $(x,s) \in X \times (0,\infty)$ and any non-negative $h \in \mathcal{F}_{p}$ that is $\mathcal{E}_{p}$-harmonic on $B_{\pmetric_{p}}(x,2s)$, where $\pmetric_{p} \coloneqq R_{\mathcal{E}_{p}}^{1/(p - 1)}$,  
\begin{equation}\label{e:intro-EHI}
	\sup_{B_{\pmetric_{p}}(x,s)}h \le C\inf_{B_{\pmetric_{p}}(x,s)}h, 
\end{equation}
which is well known to imply a local H\"{o}lder continuity of $h$. 
Regarding continuity estimates for $\mathcal{E}_{p}$-harmonic functions, we also obtain the following sharp H\"{o}lder regularity estimate, which in fact implies Theorem \ref{thm.intro-pRM} as an easy corollary. 
\begin{thm}[Theorem \ref{t:lip-harm}]\label{thm.intro-holder}
	Let $(\mathcal{E}_{p},\mathcal{F}_{p})$ be a $p$-resistance form on $X$ and let $B$ be a non-empty subset of $X$. 
	Define $B^{\mathcal{F}_{p}} \coloneqq \bigcap_{u \in \mathcal{F}_{p}; u|_{B} = 0}u^{-1}(0)$ and, for $x \in X \setminus B^{\mathcal{F}_{p}}$, 
	\begin{equation*}
		\pmetric_{p}(x,B) \coloneqq \left(\sup\biggl\{ \frac{\abs{u(x)}^{p}}{\mathcal{E}_{p}(u)} \biggm| \text{$u \in \mathcal{F}_{p}$, $u|_{B} = 0$, $u(x) \neq 0$} \biggr\}\right)^{\frac{1}{p - 1}}. 
	\end{equation*}
	Let $x \in X \setminus B^{\mathcal{F}_{p}}$ and $y \in X$. Then
    \begin{equation*}
        h_{B \cup \{ x \}}^{\mathcal{E}_{p}}\bigl[\indicator{B}^{B \cup \{ x \}}\bigr](y)
        \le \frac{\pmetric_{p}(x,y)}{\pmetric_{p}(x,B)}.
    \end{equation*}
	Moreover, for any $h \in \mathcal{F}_{p}$ that is $\mathcal{E}_{p}$-harmonic on $X \setminus B$ and satisfies $\sup_{B}\abs{h} < \infty$,
    \begin{equation*}
        \abs{h(x) - h(y)} \le \frac{\pmetric_{p}(x,y)}{\pmetric_{p}(x,B)}\sup_{x',y' \in B}\abs{h(x') - h(y')}. 
    \end{equation*}
\end{thm}

Next let us move to applications of our general theory of $p$-resistance forms.
In their forthcoming papers \cite{KS.scp,KS.sing}, the authors will heavily use this theory to make some essential progress in the setting of post-critically finite self-similar structures; see \cite{KS.survey} for a survey of these results described in the setting of the Sierpi\'{n}ski gasket.
In Section \ref{sec.p-walk} of this paper, we shall give another application to strict inequalities for the $p$-walk dimensions of two classes of self-similar fractals, the generalized Sierpi\'{n}ski carpets and the $D$-dimensional level-$l$ Sierpi\'{n}ski gasket (see Figure \ref{fig.fractals}).  
Let $K$ be a generalized Sierpi\'{n}ski carpet or the $D$-dimensional level-$l$ Sierpi\'{n}ski gasket, equip $K$ with the Euclidean metric $d$, let $p \in (1,\infty)$, and assume in the former case that $p$ is strictly greater than the Ahlfors regular conformal dimension of $(K,d)$. 
Then by Theorem \ref{thm.KSgood-ss} in the former case and by Theorem \ref{thm.ANFsymform} in the latter case, we can construct a canonical $p$-resistance form $(\mathcal{E}_{p},\mathcal{F}_{p})$ on $K$. 
To be more precise, let $\{ F_{i} \}_{i \in S}$, with $S$ a suitable non-empty finite set, be the family of contractive similitudes defining $K$, i.e., such that $K = \bigcup_{i \in S}F_{i}(K)$.
Then there exists a $p$-resistance form $(\mathcal{E}_{p},\mathcal{F}_{p})$ on $K$ which satisfies $\mathcal{F}_{p} \subseteq \contfunc(K)$ and the following \emph{self-similarity} for some $\sigma_{p} \in (1,\infty)$(, which we call the \emph{weight} of $(\mathcal{E}_{p},\mathcal{F}_{p})$): 
\begin{equation}\label{intro.ss}
    \mathcal{E}_{p}(u) = \sigma_{p}\sum_{i \in S}\mathcal{E}_{p}(u \circ F_{i}), \quad u \in \mathcal{F}_{p}. 
\end{equation}  
Letting $r_{\ast} \in (0,1)$ denote the common contraction ratio of the similitudes $\{ F_{i} \}_{i \in S}$, we define the \emph{$p$-walk dimension} $\pwalk$ of $K$ by
\[
\pwalk \coloneqq \frac{\log{\bigl((\#S)\sigma_{p}\bigr)}}{\log(r_{\ast}^{-1})},
\]
which coincides with the walk dimension of $K$ if $p = 2$.
As shown in \cite[Theorem 7.1]{MS+}, the value $\pwalk$ shows up as a space-scaling exponent in the following manner:
\[
\mathcal{E}_{p}(u) \asymp \limsup_{r \downarrow 0}\int_{K}\fint_{\abs{x - y} < r}\frac{\abs{u(x) - u(y)}^{p}}{r^{\pwalk}}\,\mu(dy)\,\mu(dx), \quad u \in \mathcal{F}_{p},
\] 
where $\mu$ denotes the $\log(\#S)/\log(r_{\ast}^{-1})$-dimensional Hausdorff measure on $(K,d)$.
In the case of $p = 2$, the strict inequality $d_{\mathrm{w},2} > 2$ has been verified for various self-similar fractals, and has been shown to imply a number of anomalous features of the diffusion associated with $(\mathcal{E}_{2},\mathcal{F}_{2})$; see, e.g., \cite{Kaj23} and the references therein for further details. 
Compared with the case of $p = 2$, the class of self-similar fractals for which $\pwalk > p$ has been proved in \cite[Theorem 2.27]{Shi24} is limited to the \emph{planar} generalized Sierpi\'{n}ski carpets due to the lack of counterparts of many useful tools available in the case of $p = 2$.
As an application of the differentiability in \eqref{intro.diffble}, in Section \ref{sec.p-walk}, we show $\pwalk > p$ for \emph{any} generalized Sierpi\'{n}ski carpet and for the $D$-dimensional level-$l$ Sierpi\'{n}ski gasket with \emph{any} $D,l \in \mathbb{N} \setminus \{1\}$. 
The proof for the former follows closely the argument in \cite{Kaj23}, whereas for the latter we need a different argument from that in \cite{Kaj23}. 

\begin{figure}\centering
	\includegraphics[height=90pt]{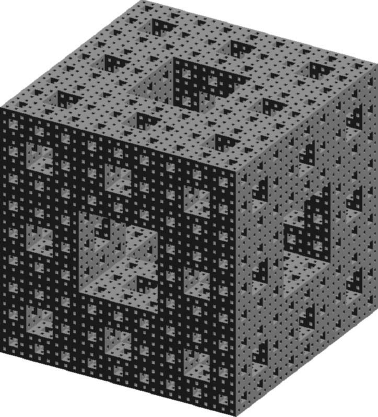}\hspace*{10pt}
	\includegraphics[height=90pt]{fig_SG2_reduced.pdf}\hspace*{10pt}
    \includegraphics[height=90pt]{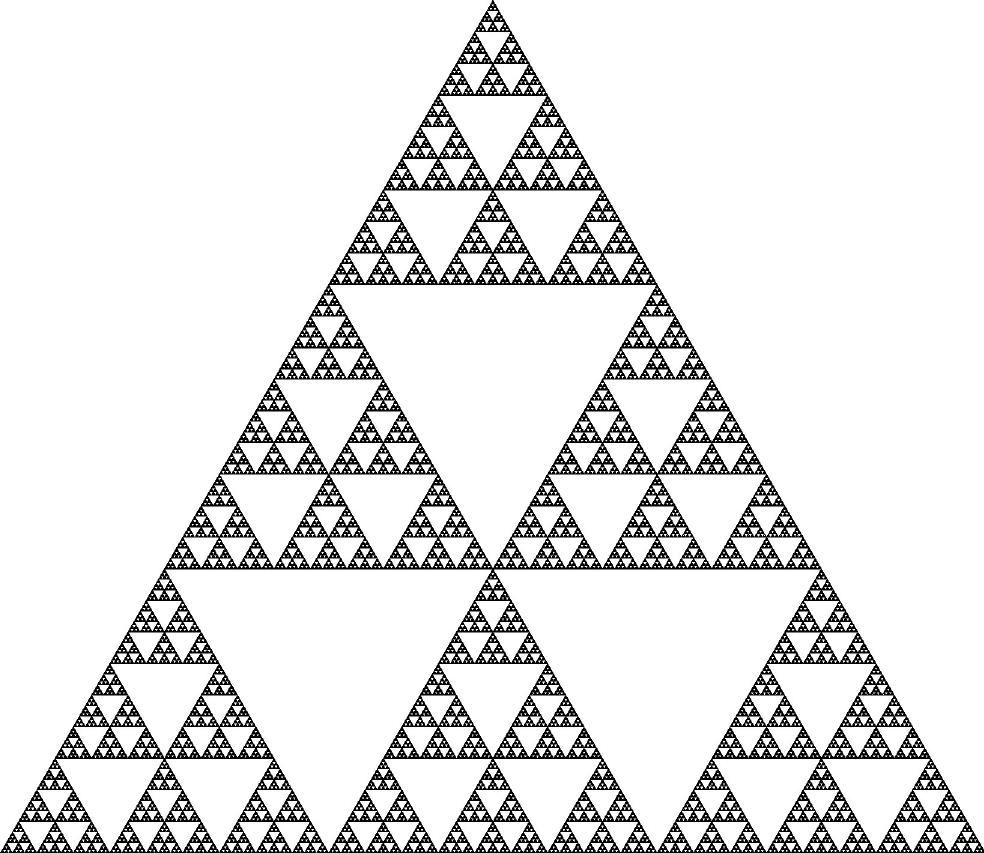}\hspace*{10pt}
    \includegraphics[height=90pt]{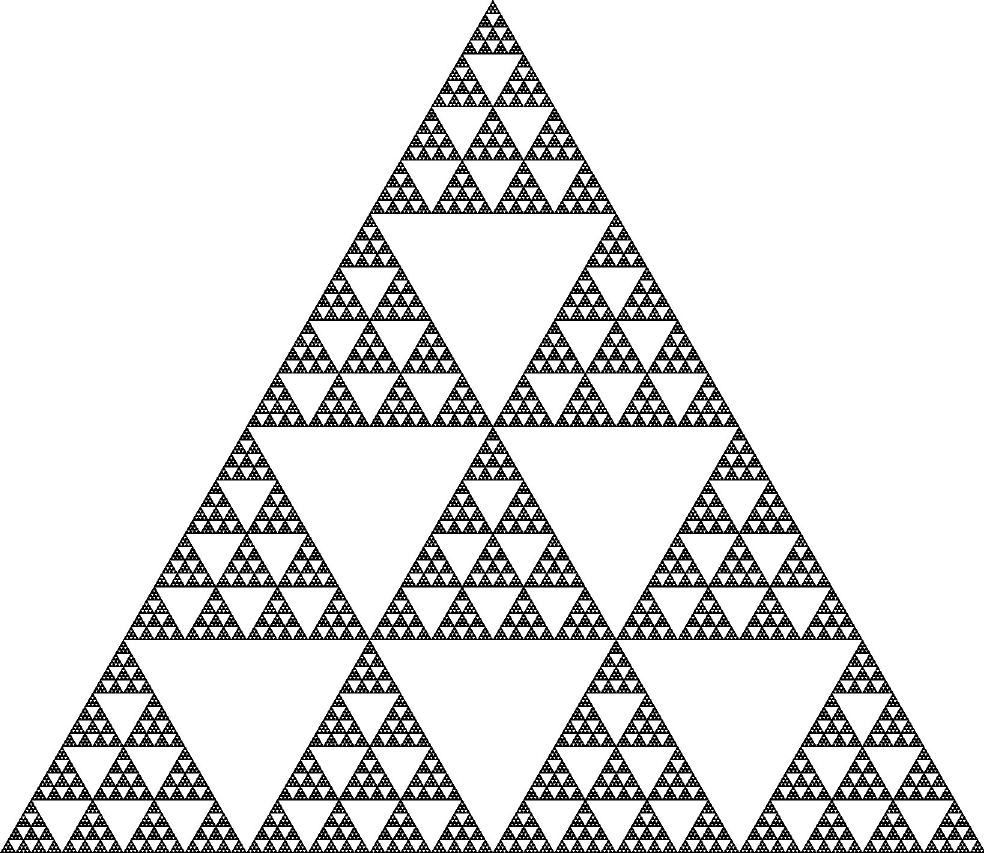}
	\caption{From the left, a non-planar generalized Sierpi\'{n}ski carpet (Menger Sponge) and the $2$-dimensional level-$l$ Sierpi\'{n}ski gaskets ($l = 2,3,4$)}\label{fig.fractals}
\end{figure}

We would also like to mention a geometric role of $\sigma_{p}$ appearing in \eqref{intro.ss}. 
As done in \cite{Kig20,Kig23}, the constant $\sigma_{p}$ is obtained by seeking the behavior of \emph{conductance constants} (\cite[Definition 2.17]{Kig23}) on approximating graphs of $K$; see Theorem \ref{t:pCH} for details.
A remarkable fact is that the behavior of $\sigma_{p}$ as a function of $p$ is deeply related to the \emph{Ahlfors regular conformal dimension} $\dim_{\mathrm{ARC}}(K,d)$ of $(K,d)$ (see Definition \ref{defn.AR}-\ref{it:ARCdim} for its definition); indeed, $\sigma_{p} > 1$ if and only if $p > \dim_{\mathrm{ARC}}(K,d)$ (see, e.g., \cite[Theorem 4.7.6]{Kig20}). 
Therefore, knowing properties of the function $p \mapsto \sigma_{p}$ is very important to understand the Ahlfors regular conformal dimension and related geometric information. 
Nevertheless, we do not know anything other than the following: 
\begin{enumerate}[label=\textup{(\arabic*)},align=left,leftmargin=*,topsep=2pt,parsep=0pt,itemsep=2pt]
	\item [\textup{(Continuity; \cite[Proposition 4.7.5]{Kig20})}] $\sigma_{p}$ is continuous in $p$. 
	\item [\textup{(Simple monotonicity; \cite[Proposition 4.7.5]{Kig20})}] $\sigma_{p}$ is non-decreasing in $p$.
	\item [\textup{(H\"{o}lder-type monotonicity; \cite[Lemma 4.7.4]{Kig20})}] $\pwalk/p$ is non-increasing in $p$.
	\item [\textup{(Relation with $\dim_{\mathrm{ARC}}$; \cite[Theorem 4.7.6]{Kig20})}] $\sigma_{p} > 1$ if and only if $p > \dim_{\mathrm{ARC}}(K,d)$.
\end{enumerate}
As yet another application of our theory of $p$-resistance forms, we prove in Theorems \ref{t:KS-mono} and \ref{t:KS-mono.pcf} the following new monotonicity behavior of $\sigma_{p}$ (in suitably general settings including any generalized Sierpi\'{n}ski carpet and the $D$-dimensional level-$l$ Sierpi\'{n}ski gasket with any $D,l \in \mathbb{N} \setminus \{1\}$)\footnote{It is essentially known to experts that $\dim_{\mathrm{ARC}}(K,d) = 1$ for the $D$-dimensional level-$l$ Sierpi\'{n}ski gasket $K$ equipped with the Euclidean metric $d$. In Theorem \ref{thm.dARC-ANF}, we give a new proof of this fact, based on the existence of self-similar $p$-resistance forms proved in Theorem \ref{thm.eigenform-ANF} as an extension of \cite[Theorem 6.3]{CGQ22}, for a large class of post-critically finite self-similar sets with good geometric symmetry; see Subsection \ref{sec.confdimANF} for details and relevant results in the literature.}: 
\begin{equation}\label{intro.newmono}
	\text{$(\dim_{\mathrm{ARC}}(K,d),\infty) \ni p \mapsto \sigma_{p}^{1/(p - 1)} \in (0,\infty)$ is non-decreasing,} 
\end{equation}
which is good evidence that properties of $p \mapsto \sigma_{p}^{1/(p - 1)}$ are also important to deepen our understanding of $(\mathcal{E}_{p},\mathcal{F}_{p})$ and, possibly, of $\dim_{\mathrm{ARC}}(K,d)$.  

Let us conclude this introduction by mentioning a significant difference between our theory and some recent results \cite{BBR23+,Kuw23+} on $p$-energy forms based on strongly local regular symmetric Dirichlet forms. (Similar $p$-energy forms were considered earlier in \cite[Remark 6.1]{HRT13}.) 
In the settings of \cite{BBR23+,Kuw23+}, the associated $p$-energy measure $\Gamma_{p}^{\mathrm{DF}}\langle u \rangle$ can be explicitly defined by using the ``density'' which plays the role of ``$\abs{\nabla u}$'' and is independent of $p$ (see Example \ref{ex.em}-\ref{exam.DF-em}), whereas it is almost impossible to find a priori such a density on fractals. 
Meanwhile, we can naturally define the \emph{self-similar $p$-energy measure} $\Gamma_{p}\langle u \rangle$ of $u$ by using \eqref{intro.ss}; see Section \ref{sec.ss} for details. (See also \cite{KS.lim} for $p$-energy measures associated with Korevaar--Schoen $p$-energy forms.)
In \cite{KS.sing}, the authors will show that $\Gamma_{p}\langle u_p \rangle$ and $\Gamma_{q}\langle u_q \rangle$ are mutually singular for any $p,q \in (1,\infty)$ with $p \neq q$ and any $(u_{p},u_{q}) \in \mathcal{F}_{p} \times \mathcal{F}_{q}$ for a certain class of post-critically finite self-similar sets including the $D$-dimensional level-$l$ Sierpi\'{n}ski gasket with any $D,l \in \mathbb{N} \setminus \{1\}$, by proving that $(1,\infty) \ni p \mapsto \sigma_{p}^{1/(p - 1)}$ is \emph{strictly} increasing. 
This phenomenon on the singularity of energy measures \emph{never} happens if we consider the energy measures $\Gamma_{p}^{\mathrm{DF}}\langle \,\cdot\, \rangle, \Gamma_{q}^{\mathrm{DF}}\langle \,\cdot\, \rangle$ that naturally show up in the settings of \cite{BBR23+,Kuw23+}. 
This point also motivates us to develop a general theory of $p$-energy forms in an abstract setting in order to deal with fractals. 

This paper is organized as follows.  
In Section \ref{sec.GC}, we collect basic results on the generalized $p$-contraction property \ref{intro.GC}. 
In Section \ref{sec.diffble}, we prove the differentiability of $p$-energy forms satisfying $p$-Clarkson's inequality (Theorem \ref{thm.intro-EpCp}). 
Moreover, we see that the (Fr\'{e}chet) derivative in \eqref{intro.diffble} gives a homeomorphism between $\mathcal{F}_{p}/\mathcal{E}_{p}^{-1}(0)$ and its dual. 
We also discuss regular and strong local properties of $p$-energy forms there.
In Section \ref{sec.pEM}, under the assumption of the existence of $p$-energy measures, we discuss their fundamental properties (Theorem \ref{thm.intro-emCp} for example). 
We also formulate a chain rule for $p$-energy measures and observe some consequences of it.  
In Section \ref{sec.ss}, we recall standard notions on self-similar structures, discuss the self-similarity of $p$-energy forms and see that we can associate self-similar $p$-energy measures to a given self-similar $p$-energy form.
Section \ref{sec.p-harm} is devoted to the study of fundamental nonlinear potential theory for $p$-resistance forms, most of which are mentioned in the introduction (see Theorems \ref{thm.intro-pRM}, \ref{thm.intro-trace}, \ref{thm.intro-holder}, Proposition \ref{prop.intro-unique}, \eqref{e:intro-cp} and \eqref{e:intro-EHI}). 
We further investigate the theory of $p$-resistance forms in the self-similar case in Section \ref{sec.compatible}. 
In particular, we establish a Poincar\'{e}-type inequality in terms of self-similar $p$-energy measures under some geometric assumptions on the $p$-resistance metric. 
In Section \ref{sec.constr}, the generalized $p$-contraction property \ref{intro.GC} is verified for the $p$-energy/$p$-resistance forms constructed in \cite{CGQ22,Kig23}. 
More precisely, in Subsections \ref{sec.Kig} and \ref{sec.Kigss}, we recall the notion of $p$-conductively homogeneous compact metric space and the construction of $p$-energy forms due to \cite{Kig23} and prove \ref{intro.GC} for them.
In Subsection \ref{sec.pcf}, we focus on the case of post-critically finite self-similar structures and show that the \emph{eigenforms} constructed in \cite{CGQ22} are indeed $p$-resistance forms. 
In Subsection \ref{sec.ANF}, we prove the existence of eigenforms for a large class of post-critically finite self-similar sets with good geometric symmetry (Theorem \ref{thm.eigenform-ANF}), extending \cite[Theorem 6.3]{CGQ22} by following the framework of \cite[Theorem 3.8.10]{Kig01}. 
In Section \ref{sec.p-walk}, we prove $\pwalk > p$ for the generalized Sierpi\'{n}ski carpets and the $D$-dimensional level-$l$ Sierpi\'{n}ski gasket by using properties of $p$-harmonic functions established in Section \ref{sec.p-harm}. 
In Appendix \ref{sec.bilinear}, we show that \hyperref[intro.GC]{(GC)$_{2}$} holds for any symmetric Dirichlet form, the ($2$-)energy measures associated with any regular symmetric Dirichlet form, and their densities, and that, under the strong locality, \ref{GC} holds for the $p$-energy form defined as the integral of the $\frac{p}{2}$-th power of those densities, which has recently been studied by Kuwae \cite{Kuw23+}.
Lastly, in Appendix \ref{sec:pcfcollect} we collect some miscellaneous results related to self-similar $p$-resistance forms on post-critically finite self-similar structures. 

\begin{notation}\label{notation.intro}
    Throughout this paper, we use the following notation and conventions.
    \begin{enumerate}[label=\textup{(\arabic*)},align=left,leftmargin=*,widest=99,topsep=2pt,parsep=0pt,itemsep=2pt]\addtocounter{enumi}{-1}
        \item\label{it:N-nonzero} $\mathbb{N} \coloneqq \{ n \in \mathbb{Z} \mid n > 0\}$, i.e., $0\not\in\mathbb{N}$.  
		\item\label{it:ineq-up-to-const} For $[0,\infty]$-valued quantities $A$ and $B$, we write $A \lesssim B$ to mean that there exists an implicit constant $C \in (0,\infty)$ depending on some unimportant parameters such that $A \leq CB$. We write $A \asymp B$ if $A \lesssim B$ and $B \lesssim A$.
	    \item\label{it:cardinality} For a set $A$, we let $\#A \in \mathbb{N} \cup \{ 0,\infty \}$ denote the cardinality of $A$.  
	    \item\label{it:supinf} We set $\sup\emptyset \coloneqq 0$, $\inf\emptyset \coloneqq \infty$, $a/0 \coloneqq \infty$ for $a \in (0,\infty]$ and $0^{0} \coloneqq 1$. We write $a \vee b \coloneqq \max\{ a, b \}$, $a \wedge b \coloneqq \min\{ a, b \}$ and $a^{+} \coloneqq a \vee 0$ for $a, b \in [-\infty,\infty]$, and we use the same notation also for $[-\infty,\infty]$-valued functions and equivalence classes of them. All numerical functions in this paper are assumed to be $[-\infty,\infty]$-valued.
	    \item\label{it:sgn} We define $\sgn\colon\mathbb{R}\to\mathbb{R}$ by $\sgn(a) \coloneqq \abs{a}^{-1}a$ for $a \in \mathbb{R} \setminus \{ 0 \}$ and $\sgn(0) \coloneqq 0$. 
		\item\label{it:ellp-norm} Let $n\in\mathbb{N}$. For $x=(x_{k})_{k=1}^{n}\in\mathbb{R}^{n}$, we set $\norm{x}_{\ell^{p}_{n}} \coloneqq \norm{x}_{\ell^{p}} \coloneqq (\sum_{k=1}^{n}\lvert x_{k}\rvert^{p})^{1/p}$ for $p\in(0,\infty)$, $\norm{x}_{\ell^{\infty}_{n}} \coloneqq \norm{x}_{\ell^{\infty}} \coloneqq \max_{1\leq k\leq n}\abs{x_{k}}$ and $\abs{x} \coloneqq \norm{x}_{\ell^{2}}$. For $\Phi\colon\mathbb{R}^{n}\to\mathbb{R}$ which is differentiable on $\mathbb{R}^{n}$ and for $k\in\{1,\ldots,n\}$, its first-order partial derivative in the $k$-th coordinate is denoted by $\partial_{k}\Phi$ and its gradient is denoted by $\nabla\Phi \coloneqq (\partial_{k}\Phi)_{k = 1}^{n}$. 
		\item\label{it:id-indicator} Let $X$ be a non-empty set. We define $\id_{X}\colon X \to X$ by $\id_{X}(x) \coloneqq x$, $\indicator{A}=\indicator{A}^{X} \in \mathbb{R}^{X}$ for $A \subseteq X$ by $\indicator{A}(x)\coloneqq\indicator{A}^{X}(x)\coloneqq \begin{cases} 1 \quad &\text{if $x \in A$,} \\ 0 \quad &\text{if $x \not\in A$,} \end{cases}$ and set $\norm{u}_{\sup} \coloneqq \norm{u}_{\sup,X} \coloneqq \sup_{x \in X}\abs{u(x)}$ for $u \colon X \to [-\infty,\infty]$ and $\osc_{X}[u] \coloneqq \sup_{x,y\in X}\abs{u(x) - u(y)}$ for $u \colon X \to \mathbb{R}$. We also let $\mathbb{R}\indicator{X} \coloneqq \{a\indicator{X} \mid a \in \mathbb{R}\}$ denote the set of $\mathbb{R}$-valued constant functions on $X$. 
		\item\label{it:Borel-contfunc} Let $X$ be a topological space. The Borel $\sigma$-algebra of $X$ is denoted by $\mathcal{B}(X)$, the closure of $A \subseteq X$ in $X$ by $\closure{A}^{X}$, and we say that $A \subseteq X$ is \emph{relatively compact}\index{relatively compact} in $X$ if and only if $\closure{A}^{X}$ is compact. We set $\contfunc(X) \coloneqq \{ u \in \mathbb{R}^{X} \mid \text{$u$ is continuous} \}$, $\supp_{X}[u] \coloneqq \closure{X \setminus u^{-1}(0)}^{X}$ for $u \in \contfunc(X)$, $\contfunc_{b}(X) \coloneqq \{ u \in \contfunc(X) \mid \norm{u}_{\sup} < \infty \}$, and $\contfunc_{c}(X) \coloneqq \{ u \in \contfunc(X) \mid \text{$\supp_{X}[u]$ is compact} \}$. 
		\item\label{it:support} Let $X$ be a topological space having a countable open base. For a measure $m$ on a $\sigma$-algebra $\mathcal{B}$ in $X$ including $\mathcal{B}(X)$, we let $\supp_{X}[m]$ denote the support of $m$ in $X$, i.e., the smallest closed subset $F$ of $X$ such that $m(X \setminus F) = 0$, and set $\supp_{m}[f] \coloneqq \supp_{X}[ \abs{f}\,dm ]$ for a $\mathcal{B}$-measurable function $f \colon X \to [-\infty,\infty]$ or an $m$-equivalence class $f$ of such functions. 
        \item\label{it:metric-sp} Let $(X,d)$ be a metric space. We set $B_{d}(x,r) \coloneqq \{ y \in X \mid d(x,y) < r \}$ for $(x,r) \in X \times (0,\infty)$, and $\diam(A,d) \coloneqq \sup_{x,y \in A}d(x,y)$ and $\dist_{d}(A,B) \coloneqq \inf\{ d(x,y) \mid x \in A, y \in B \}$ for subsets $A,B$ of $X$.  
        \item\label{it:measure-sp} Let $(X,\mathcal{B},m)$ be a measure space. We set $\fint_{A}f\,dm \coloneqq \frac{1}{m(A)}\int_{A}f\,dm$ for $f \in L^{1}(X,m)$ and $A \in \mathcal{B}$ with $m(A) \in (0,\infty)$, and set $m|_{A} \coloneqq m|_{\mathcal{B}|_{A}}$ for $A \in \mathcal{B}$, where $\mathcal{B}|_{A} \coloneqq \{ B \cap A \mid B \in \mathcal{B} \}$. For a measure $\mu$ on $(X,\mathcal{B})$, we write $\mu \ll m$ to mean that $\mu$ is absolutely continuous with respect to $m$. 
    \end{enumerate}
\end{notation}

\section{The generalized \texorpdfstring{$p$}{p}-contraction property}\label{sec.GC}
In this section, we will introduce the generalized $p$-contraction property and establish basic results on this property. 
Throughout this section, we fix $p \in (1,\infty)$, a measure space $(X,\mathcal{B},m)$, a linear subspace $\mathcal{F}$ of $L^{0}(X,m) \coloneqq L^{0}(X,\mathcal{B},m)$, where
\begin{equation}\label{L0-dfn}
L^{0}(X,\mathcal{B},m) \coloneqq \{ \text{the $m$-equivalence class of $f$} \mid \text{$f \colon X \to \mathbb{R}$, $f$ is $\mathcal{B}$-measurable} \},
\end{equation}
and a functional $\mathcal{E} \colon \mathcal{F} \to [0, \infty)$ which is \emph{$p$-homogeneous}, i.e., satisfies $\mathcal{E}(au) = \abs{a}^{p}\mathcal{E}(u)$ for any $(a,u) \in \mathbb{R}\times \mathcal{F}$. 
\begin{rmk}\label{rmk:wo-measure} 
	Note that the pair $(\mathcal{B},m)$ is arbitrary. 
	For example, $(\mathcal{B},m)$ could be the pair of $2^{X} = \{ A \mid A \subseteq X \}$ and the counting measure on $X$, in which case $L^{0}(X,\mathcal{B},m) = \mathbb{R}^{X}$.
	We will make this choice of $(\mathcal{B},m)$ later in Section \ref{sec.p-harm}.
\end{rmk}

Let us introduce the generalized $p$-contraction property as arguably the strongest possible form of contraction properties of $p$-energy forms in the next Definition \ref{defn.GC}. We refer to the discussion preceding Definition \ref{defn.GC.intro} in Section \ref{sec:intro} for motivations for introducing this property. Some examples of $p$-energy forms satisfying it are presented later in Example \ref{ex.Rn}. 

\begin{defn}[Generalized $p$-contraction property]\label{defn.GC}
    The pair $(\mathcal{E},\mathcal{F})$ is said to satisfy the \emph{generalized $p$-contraction property}\index{generalized $p$-contraction property (for $p$-energy form)}, \ref{GC} for short, if and only if the following hold: if $n_{1},n_{2} \in \mathbb{N}$, $q_{1} \in (0,p]$, $q_{2} \in [p,\infty]$ and $T = (T_{1},\dots,T_{n_{2}}) \colon \mathbb{R}^{n_{1}} \to \mathbb{R}^{n_{2}}$ satisfy
    \begin{equation}\label{GC-cond}
        T(0) = 0 \quad \text{and} \quad \norm{T(x)-T(y)}_{\ell^{q_{2}}}
        \le \norm{x - y}_{\ell^{q_{1}}} \quad \text{for any $x, y \in \mathbb{R}^{n_{1}}$,}
    \end{equation}
    then for any $\bm{u} = (u_{1},\dots,u_{n_{1}}) \in \mathcal{F}^{n_{1}}$ we have
    \begin{equation}\label{GC}
        T(\bm{u}) \in \mathcal{F}^{n_{2}} \quad \text{and} \quad
        \norm{\bigl(\mathcal{E}(T_{l}(\bm{u}))^{1/p}\bigr)_{l = 1}^{n_{2}}}_{\ell^{q_{2}}} \le \norm{\bigl(\mathcal{E}(u_{k})^{1/p}\bigr)_{k = 1}^{n_{1}}}_{\ell^{q_{1}}}. \tag*{\textup{(GC)$_{p}$}}
    \end{equation}
\end{defn} 

As observed in the following proposition, \ref{GC} includes various known inequalities. 
\begin{prop}\label{prop.GC-list}
    Assume that $(\mathcal{E},\mathcal{F})$ satisfies \ref{GC}.
    Let $\varphi \in \contfunc(\mathbb{R})$ satisfy $\varphi(0) = 0$ and $\abs{\varphi(t) - \varphi(s)} \le \abs{t - s}$ for any $s,t \in \mathbb{R}$.\footnote{Note that any such $\varphi$ satisfies $\abs{\varphi \circ f} \leq \abs{f}$ on $X$ for any $f \colon X \to \mathbb{R}$ and hence $\varphi \circ f \in L^{p}(X,m)$ for any $f \in L^{p}(X,m)$.} 
    \begin{enumerate}[label=\textup{(\alph*)},align=left,leftmargin=*,topsep=2pt,parsep=0pt,itemsep=2pt]
        \item\label{GC.tri} $T(x,y) \coloneqq x + y$, $x,y \in \mathbb{R}$, satisfies \eqref{GC-cond} in Definition \ref{defn.GC} with $(q_1,q_2,n_1,n_2) = (1,p,2,1)$. In particular, $\mathcal{E}^{1/p}$ is a seminorm on $\mathcal{F}$, and $\mathcal{E}$ is convex on $\mathcal{F}$, i.e., for any $\lambda \in (0,1)$ and any $f,g \in \mathcal{F}$, 
        	\begin{equation}\label{e:convex}
        		\mathcal{E}(\lambda f + (1 - \lambda)g) \leq \lambda\mathcal{E}(f) + (1 - \lambda)\mathcal{E}(g). 
        	\end{equation}
        \item\label{GC.lip} $T \coloneqq \varphi$ satisfies \eqref{GC-cond} with $(q_1,q_2,n_1,n_2) = (1,p,1,1)$.
        In particular, 
        \begin{equation}\label{lipcont}
        	\text{for any $\varphi$ as assumed above, $\varphi(u) \in \mathcal{F}$ and $\mathcal{E}(\varphi(u)) \le \mathcal{E}(u)$ for any $u \in \mathcal{F}$.}
        \end{equation} 
        \item\label{GC.markov} Assume that $\varphi$ is non-decreasing. Define $T = (T_1,T_2) \colon \mathbb{R}^{2} \to \mathbb{R}^{2}$ by
        \[
        T_{1}(x_1,x_2) = x_1 - \varphi(x_1 - x_2) \quad \text{and} \quad T_{2}(x_1,x_2) = x_2 + \varphi(x_1 - x_2), \quad (x_1,x_2) \in \mathbb{R}^{2}.
        \]
        Then $T$ satisfies \eqref{GC-cond} with $(q_1,q_2,n_1,n_2) = (p,p,2,2)$.
        In particular,  
        \begin{equation}\label{markov-nonDF}
            \mathcal{E}\bigl(f - \varphi(f - g)\bigr) + \mathcal{E}\bigl(g + \varphi(f - g)\bigr) \le \mathcal{E}(f) + \mathcal{E}(g) \quad \text{for any $f,g \in \mathcal{F}$.}
        \end{equation}
        In particular, by considering the case of $\varphi(x) = x^{+}$, we have the following \emph{strong subadditivity}\index{strong subadditivity}: for any $f,g \in \mathcal{F}$, $f \vee g, f \wedge g \in \mathcal{F}$ and
        \begin{equation}\label{sadd}
            \mathcal{E}(f \vee g) + \mathcal{E}(f \wedge g) \le \mathcal{E}(f) + \mathcal{E}(g).
        \end{equation} 
        \item\label{GC.leibniz} Let $a_{1},a_{2} \in (0,\infty)$ and define $T^{a_{1},a_{2}} \colon \mathbb{R}^{2} \to \mathbb{R}$ by
        \[
        T^{a_{1},a_{2}}(x_{1},x_{2}) \coloneqq \bigl( (-a_{1}) \vee ((a_{2}^{-1}x_{1}) \wedge a_{1}) \bigr) \cdot \bigl( (-a_{2}) \vee ( (a_{1}^{-1}x_{2}) \wedge a_{2} ) \bigr), \quad (x_{1},x_{2}) \in \mathbb{R}^{2}.
        \]
        Then $T^{a_1,a_2}$ satisfies \eqref{GC-cond} with $(q_1,q_2,n_1,n_2) = (1,p,2,1)$.
        In particular, for any $f,g \in \mathcal{F} \cap L^{\infty}(X,m)$ we have
        \begin{equation}\label{leibniz}
            f \cdot g \in \mathcal{F} \quad \text{and} \quad \mathcal{E}(f \cdot g)^{1/p} \le \norm{g}_{L^{\infty}(X,m)}\mathcal{E}(f)^{1/p} + \norm{f}_{L^{\infty}(X,m)}\mathcal{E}(g)^{1/p}. 
        \end{equation}  
        \item \label{GC.Cpsmall} Assume that $p \in (1,2]$. Define $T = (T_1,T_2) \colon \mathbb{R}^{2} \to \mathbb{R}^{2}$ by
        \[
        T_{1}(x_1,x_2) = 2^{-(p - 1)/p}(x_1 + x_2) \quad \text{and} \quad T_{2}(x_1,x_2) = 2^{-(p - 1)/p}(x_1 - x_2), \quad (x_1,x_2) \in \mathbb{R}^{2}.
        \]
        Then $T$ satisfies \eqref{GC-cond} with $(q_1,q_2,n_1,n_2) = (p,p/(p-1),2,2)$.
        In particular, $(\mathcal{E}, \mathcal{F})$ satisfies the following $p$-Clarkson's inequality:
        \begin{equation}\label{Cp.small}
            \mathcal{E}(f+g) + \mathcal{E}(f-g)
            \ge 2\bigl(\mathcal{E}(f)^{1/(p - 1)} + \mathcal{E}(g)^{1/(p - 1)}\bigr)^{p - 1} \quad \text{for any $f,g \in \mathcal{F}$.} 
        \end{equation}
        \item \label{GC.Cplarge} Assume that $p \in [2,\infty)$. Define $T = (T_1,T_2) \colon \mathbb{R}^{2} \to \mathbb{R}^{2}$ by
        \[
        T_{1}(x_1,x_2) = 2^{-1/p}(x_1 + x_2) \quad \text{and} \quad T_{2}(x_1,x_2) = 2^{-1/p}(x_1 - x_2), \quad (x_1,x_2) \in \mathbb{R}^{2}.
        \]
        Then $T$ satisfies \eqref{GC-cond} with $(q_1,q_2,n_1,n_2) = (p/(p - 1),p,2,2)$.
        In particular, $(\mathcal{E}, \mathcal{F})$ satisfies the following $p$-Clarkson's inequality:
        \begin{equation}\label{Cp.large}
            \mathcal{E}(f+g) + \mathcal{E}(f-g)
            \le 2\bigl(\mathcal{E}(f)^{1/(p - 1)} + \mathcal{E}(g)^{1/(p - 1)}\bigr)^{p - 1} \quad \text{for any $f,g \in \mathcal{F}$.} 
        \end{equation}
    \end{enumerate} 
\end{prop}
\begin{rmk}\label{rmk.GC}
	\begin{enumerate}[label=\textup{(\arabic*)},align=left,leftmargin=*,topsep=2pt,parsep=0pt,itemsep=2pt]
		\item\label{it:Rmk.GC-NDF} The property \eqref{markov-nonDF} in Proposition \ref{prop.GC-list}-\ref{GC.markov} above is inspired by the \emph{nonlinear Dirichlet form theory}\index{nonlinear Dirichlet form} due to Cipriani and Grillo \cite{CG03}. See \cite[Theorem 4.7]{Cla23} and the reference therein for further background. It is easy to see that the inequality \eqref{markov-nonDF} is equivalent to the following property: if $\varphi_{\#} \in \contfunc(\mathbb{R})$ satisfies $\varphi_{\#}(0) = 0$ and $\abs{\varphi_{\#}(t) - \varphi_{\#}(s)} \le \abs{t - s}$ for any $s,t \in \mathbb{R}$, then for any $f,g \in \mathcal{F}$, 
			\begin{equation}\label{markov-nonDF.another}
				\mathcal{E}(f + \varphi_{\#}(g)) + \mathcal{E}(f - \varphi_{\#}(g)) 
				\le \mathcal{E}(f + g) + \mathcal{E}(f - g). 
			\end{equation}
			Indeed, it suffices to consider the transformation $\varphi_{\#}(t) = t - \varphi(2t)$, where $\varphi$ is the same as in \eqref{markov-nonDF}; see \cite[Proof of Theorem 1]{Puc25+} for the details of this argument.
		\item\label{it:Rmk.GC-Cp} There are two versions of $p$-Clarkson's inequality\index{$p$-Clarkson's inequality (for $p$-energy form)}, one of which is stronger than the other. The inequalities \eqref{Cp.small} and \eqref{Cp.large} above are the stronger one for $p \in (1,2]$ and for $p \in [2,\infty)$, respectively; see Remark \ref{rmk:Cp-weak} below for the weaker one. 
	\end{enumerate}
\end{rmk}
\begin{proof}[Proof of Proposition \ref{prop.GC-list}]
	\ref{GC.tri}: 
	It is obvious that $T(x,y) \coloneqq x + y$ satisfies \eqref{GC-cond} with $(q_1,q_2,n_1,n_2) = (1,p,2,1)$ and hence the triangle inequality for $\mathcal{E}^{1/p}$ holds. 
	Since $[0,\infty) \ni x \mapsto x^{p}$ is convex, for any $\lambda \in (0,1)$ and any $f,g \in \mathcal{F}$, 
	\begin{align*}
		\mathcal{E}(\lambda f + (1 - \lambda)g) 
		\le \bigl(\lambda\mathcal{E}(f)^{1/p} + (1 - \lambda)\mathcal{E}(g)^{1/p}\bigr)^{p}
		\leq \lambda\mathcal{E}(f) + (1 - \lambda)\mathcal{E}(g), 
	\end{align*}
	where we used the triangle inequality for $\mathcal{E}^{1/p}$ in the first inequality.
	
	\ref{GC.lip}: This is obvious. 

    \ref{GC.markov}:
    Let $x = (x_{1},x_{2}), y = (y_{1},y_{2}) \in \mathbb{R}^{2}$.
    For ease of notation, set $z_{i} \coloneqq x_{i} - y_{i}$ and $A \coloneqq \varphi(x_{1} - x_{2}) - \varphi(y_{1} - y_{2})$.
    Then $\norm{T(x) - T(y)}_{\ell^p} \le \norm{x - y}_{\ell^p}$ ie equivalent to
    \begin{equation}\label{nonDF-pre}
        \abs{z_{1} - A}^{p} + \abs{z_{2} + A}^{p} \le \abs{z_{1}}^{p} + \abs{z_{2}}^{p},
    \end{equation}
    so we will show \eqref{nonDF-pre}.
    Note that $\abs{A} \le \abs{z_{1} - z_{2}}$ since $\varphi$ is $1$-Lipschitz.
    The desired estimate \eqref{nonDF-pre} is evident when $z_{1} = z_{2}$, so we consider the case of $z_{1} \neq z_{2}$.
    Assume that $z_{1} > z_{2}$ because the remaining case $z_{1} < z_{2}$ is similar.
    Then $(x_{1} - x_{2}) - (y_{1} - y_{2}) = z_{1} - z_{2} > 0$ and thus $0 \le A \le z_{1} - z_{2}$.
    Set $\psi_{p}(t) \coloneqq \abs{t}^{p} \, (t \in \mathbb{R})$ for brevity. 
    If $0 \le A < \frac{z_{1} - z_{2}}{2}$, then $z_{2} \le z_{2} + A < z_{1} - A \le z_{1}$ and we see that 
    \begin{align*}
        \abs{z_{1} - A}^{p} + \abs{z_{2} + A}^{p} - \abs{z_{1}}^{p} - \abs{z_{2}}^{p}
        &= \int_{z_{2}}^{z_{2} + A}\psi_{p}'(t)\,dt - \int_{z_{1} - A}^{z_{1}}\psi_{p}'(t)\,dt \\
        &\le A\bigl(\psi_{p}'(z_{2} + A) - \psi_{p}'(z_{1} - A)\bigr)
        \le 0.
    \end{align*}
	If $A \ge \frac{z_{1} - z_{2}}{2}$, then $z_{2} \le z_{1} - A \le z_{2} + A \le z_{1}$ and thus
    \begin{align*}
        \abs{z_{1} - A}^{p} + \abs{z_{2} + A}^{p} - \abs{z_{1}}^{p} - \abs{z_{2}}^{p}
        &= \int_{z_{2}}^{z_{1} - A}\psi_{p}'(t)\,dt - \int_{z_{2} + A}^{z_{1}}\psi_{p}'(t)\,dt \\
        &\le (z_{1} - z_{2} - A)\bigl(\psi_{p}'(z_{1} - A) - \psi_{p}'(z_{2} + A)\bigr)
        \le 0, 
    \end{align*}
    which proves \eqref{nonDF-pre}. 
    The case of $\varphi(x) = x^{+}$ immediately implies \eqref{sadd}.

    \ref{GC.leibniz}:
    For any $a_1,a_2 \in (0,\infty)$ and $(x_1,x_2), (y_1,y_2) \in \mathbb{R}^{2}$, we see that
    \begin{align*}
        &\abs{T^{a_1,a_2}(x_1,x_2) - T^{a_1,a_2}(y_1,y_2)} \\
        &\le \abs{(-a_1) \vee ((a_{2}^{-1}x_1) \wedge a_1)}\abs{\bigl((-a_2) \vee ((a_{1}^{-1}x_2) \wedge a_2)\bigr) - \bigl((-a_2) \vee ((a_{1}^{-1}y_2) \wedge a_2)\bigr)} \\
        &\quad + \abs{(-a_2) \vee ((a_{1}^{-1}y_2) \wedge a_2)}\abs{\bigl((-a_1) \vee ((a_{2}^{-1}x_1) \wedge a_1)\bigr) - \bigl((-a_1) \vee ((a_{2}^{-1}y_1) \wedge a_1)\bigr)} \\
        &\le a_{1}\abs{a_{1}^{-1}x_{2} - a_{1}^{-1}y_{2}} + a_{2}\abs{a_{2}^{-1}x_{1} - a_{2}^{-1}y_{1}}
        = \abs{x_1 - y_1} + \abs{x_2 - y_2},
    \end{align*}
    whence $T^{a_1,a_2}$ satisfies \eqref{GC-cond} with $(q_1,q_2,n_1,n_2) = (1,p,2,1)$.
    We get \eqref{leibniz} by applying \ref{GC} with $u_1 = \norm{g}_{L^{\infty}(X,m)}f$, $u_2 = \norm{f}_{L^{\infty}(X,m)}g$, $a_1 = \norm{f}_{L^{\infty}(X,m)}$ and $a_2 = \norm{g}_{L^{\infty}(X,m)}$.
    
    \ref{GC.Cpsmall},\ref{GC.Cplarge}: 
    These statements follow from $p$-Clarkson's inequality for the $\ell^p$-norm; see, e.g., \cite[Theorem 2]{Cla36} for the fact that $L^p$-norms satisfy $p$-Clarkson's inequality. 
\end{proof}

The following corollary is easily implied by Proposition \ref{prop.GC-list}-\ref{GC.lip},\ref{GC.leibniz}.
\begin{cor}\label{cor.lip-useful}
    Assume that $(\mathcal{E},\mathcal{F})$ satisfies \ref{GC}.
    \begin{enumerate}[label=\textup{(\alph*)},align=left,leftmargin=*,topsep=2pt,parsep=0pt,itemsep=2pt]
        \item \label{GC.compos} Let $u \in \mathcal{F} \cap L^{\infty}(X,m)$ and let $\Phi \in \contfunc^{1}(\mathbb{R})$ satisfy $\Phi(0) = 0$. Then
        \begin{equation}\label{compos}
    		\Phi(u) \in \mathcal{F} \quad \text{and} \quad \mathcal{E}(\Phi(u)) \le \sup\bigl\{ \abs{\Phi'(t)}^{p} \bigm| t \in \mathbb{R}, \abs{t} \le \norm{u}_{L^{\infty}(X,m)} \bigr\}\mathcal{E}(u).
    	\end{equation}
        \item \label{GC.divide} Let $\delta, M \in (0,\infty)$ and let $f, g \in \mathcal{F}$ satisfy $f \ge 0$, $g \ge 0$, $f \le M$ and $(f + g)|_{\{ f \neq 0 \}} \ge \delta$.
    	Then there exists $C \in (0,\infty)$ depending only on $p, \delta, M$ such that
    	\begin{equation}\label{negative-power}
    		\frac{f}{f + g} \in \mathcal{F} \quad \text{and} \quad \mathcal{E}\biggl(\frac{f}{f + g}\biggr) \le C\bigl(\mathcal{E}(f) + \mathcal{E}(g)\bigr).
    	\end{equation}
        \item\label{it:GC.difffunc} Let $n \in \mathbb{N}$, $q \in [1,p]$, $\bm{u} = (u_{1},\ldots,u_{n}) \in \mathcal{F}^{n}$ and $v \in L^{0}(X,m)$.
		If there exist $m$-versions of $\bm{u},v$ such that $\abs{v(x)} \le \norm{\bm{u}(x)}_{\ell^{q}}$ and $\abs{v(x) - v(y)} \le \norm{\bm{u}(x) - \bm{u}(y)}_{\ell^{q}}$ for any $x,y \in X$, then $v \in \mathcal{F}$ and $\mathcal{E}(v) \le \norm{\bigl(\mathcal{E}(u_{k})^{1/p}\bigr)_{k = 1}^{n}}_{\ell^{q}}$. 
    \end{enumerate}
\end{cor}
\begin{proof}
    \ref{GC.compos}: This is immediate from Proposition \ref{prop.GC-list}-\ref{GC.lip}.
    
    \ref{GC.divide}: We follow \cite[Proposition 6.26-(ii)]{MS.long}\footnote{The article \cite{MS.long} is a detailed version of \cite{MS+}. Several statements and proofs have been removed from the latter, so we still refer to \cite{MS.long} for those omitted results and arguments.}. Let $\varphi \in \contfunc(\mathbb{R})$ be a Lipschitz map such that $\varphi(x) = \frac{1}{x}$ for $x \ge \delta$ and $\sup_{x \neq y \in \mathbb{R}}\frac{\abs{\varphi(x) - \varphi(y)}}{\abs{x - y}} \le C'$ for some constant $C'$ depending only on $\delta$. Since $f \cdot \varphi(f + g) = \frac{f}{f + g}$, we get \eqref{negative-power} by using \eqref{lipcont} and \eqref{leibniz} in Proposition \ref{prop.GC-list}-\ref{GC.lip},\ref{GC.leibniz}. 
    
    \ref{it:GC.difffunc}:  
    The proof below is similar to \cite[Proof of Corollary I.4.13]{MR}.  
    Fix $m$-versions of $\bm{u},v$ satisfying $\abs{v(x)} \le \norm{\bm{u}(x)}_{\ell^{q}}$ and $\abs{v(x) - v(y)} \le \norm{\bm{u}(x) - \bm{u}(y)}_{\ell^{q}}$ for any $x,y \in X$. 
    We define $T_{0} \colon \bm{u}(X) \cup \{ 0 \} \to \mathbb{R}$ by setting $T_{0}(0) \coloneqq 0$ and $T_{0}(\bm{z}) \coloneqq v(x)$ for each $\bm{z} \in \bm{u}(X)$, where $x \in X$ satisfies $\bm{z} = \bm{u}(x)$. 
    This map $T_{0}$ is well-defined since $v(x) = 0$ for any $x \in X$ with $\bm{u}(x) = 0$ and $\abs{v(x) - v(y)} \le \norm{\bm{u}(x) - \bm{u}(y)}_{\ell^{q}} = 0$ for any $x,y \in X$ with $\bm{u}(x) = \bm{u}(y) \in \bm{u}(X)$. 
    In addition, we easily see that $ \abs{T_{0}(\bm{z}_{1}) - T_{0}(\bm{z}_{2})} \le \norm{\bm{z}_{1} - \bm{z}_{2}}_{\ell^{q}}$ for any $\bm{z}_{1},\bm{z}_{2} \in \bm{u}(X) \cup \{ 0 \}$, i.e., $T_{0} \colon (\bm{u}(X) \cup \{ 0 \},\norm{\,\cdot\,}_{\ell^{q}}) \to \mathbb{R}$ is $1$-Lipschitz. 
    Noting that $(\mathbb{R}^{n},\norm{\,\cdot\,}_{\ell^{q}})$ is a metric space by $q \ge 1$, we obtain a $1$-Lipschitz map $T \colon (\mathbb{R}^{n},\norm{\,\cdot\,}_{\ell^{q}}) \to \mathbb{R}$ satisfying $T(\bm{z}) = T_{0}(\bm{z})$ for any $\bm{z} \in \bm{u}(X) \cup \{ 0 \}$ by applying the McShane--Whitney extension lemma (see, e.g., \cite[p.~99]{HKST} for this extension lemma).
    Since $T$ satisfies \eqref{GC-cond} with $(q_{1},q_{2},n_{1},n_{2}) = (q,p,n,1)$ and $T(\bm{u}) = v$, the assertions follow from \ref{GC}. 
\end{proof}

There is a refinement of $p$-Clarkson's inequality on $L^p$-spaces, known as \emph{$p$-Hanner's inequality} in the literature and introduced in \cite{Han56}.
We note that \ref{GC} implies $p$-Hanner's inequality as shown in the following proposition, which is not used in the rest of this paper but recorded here for potential future applications. 

\begin{prop}[$p$-Hanner's inequality]\label{p:ICp} \index{$p$-Hanner's inequality}
    Define $\psi_{p} \colon (0,\infty) \to (0,\infty)$ by 
    \begin{equation}\label{phi_p}
        \psi_{p}(s) \coloneqq (1 + s)^{p - 1} + \sgn(1 - s)\abs{1 - s}^{p - 1}, \quad s > 0.
    \end{equation}
    \begin{enumerate}[label=\textup{(\alph*)},align=left,leftmargin=*,topsep=2pt,parsep=0pt,itemsep=2pt]
    	\item\label{it:ICp.small} Assume that $p \in (1,2]$. 
    	For $s \in (0,\infty)$, define $T^{s} = (T_{1}^{s},T_{2}^{s}) \colon \mathbb{R}^{2} \to \mathbb{R}^{2}$ by
    	\[
    	T_{1}^{s}(x_1,x_2) \coloneqq 2^{-1}\psi_{p}(s)^{1/p}(x_1 + x_2), \quad T_{2}^{s}(x_1,x_2) \coloneqq 2^{-1}\psi_{p}(s^{-1})^{1/p}(x_1 - x_2).
    	\] 
    	Then $T^{s}$ satisfies \eqref{GC-cond} in Definition \ref{defn.GC} with $(q_1,q_2,n_1,n_2) = (p,p,2,2)$ for any $s \in (0,\infty)$.
        If $(\mathcal{E}, \mathcal{F})$ satisfies \ref{GC}, then
        \begin{equation}\label{Hanner.small}
            \abs{\mathcal{E}(f)^{1/p} + \mathcal{E}(g)^{1/p}}^{p} + \abs{\mathcal{E}(f)^{1/p} - \mathcal{E}(g)^{1/p}}^{p} \le \mathcal{E}(f+g) + \mathcal{E}(f-g) \quad \text{for any $f,g \in \mathcal{F}$}.
        \end{equation}
        \item\label{it:ICp-Cp.small} Assume that $p \in (1,2]$. If $\mathcal{E}$ satisfies \eqref{Hanner.small}, then \eqref{Cp.small} in Proposition \ref{prop.GC-list}-\ref{GC.Cpsmall} holds. 
        \item\label{it:ICp.large} Assume that $p \in [2,\infty)$. 
    	For $s \in (0,\infty)$, define $T^{s} = (T_{1}^{s},T_{2}^{s}) \colon \mathbb{R}^{2} \to \mathbb{R}^{2}$ by
    	\[
    	T_{1}^{s}(x_1,x_2) \coloneqq \psi_{p}(s)^{-1/p}x_1 + \psi_{p}(s^{-1})^{-1/p}x_2, \quad T_{2}^{s}(x_1,x_2) \coloneqq \psi_{p}(s)^{-1/p}x_1 - \psi_{p}(s^{-1})^{-1/p}x_2.
    	\] 
    	Then $T^{s}$ satisfies \eqref{GC-cond} with $(q_1,q_2,n_1,n_2) = (p,p,2,2)$ for any $s \in (0,\infty)$.
        If $(\mathcal{E}, \mathcal{F})$ satisfies \ref{GC}, then
        \begin{equation}\label{Hanner.large}
            \mathcal{E}(f+g) + \mathcal{E}(f-g)
            \le \abs{\mathcal{E}(f)^{1/p} + \mathcal{E}(g)^{1/p}}^{p} + \abs{\mathcal{E}(f)^{1/p} - \mathcal{E}(g)^{1/p}}^{p} \quad \text{for any $f,g \in \mathcal{F}$}.
        \end{equation}
        \item\label{it:ICp-Cp.large} Assume that $p \in [2,\infty)$. If $\mathcal{E}$ satisfies \eqref{Hanner.large}, then \eqref{Cp.large} in Proposition \ref{prop.GC-list}-\ref{GC.Cplarge} holds. 
    \end{enumerate} 
\end{prop}
\begin{proof}
	The proof is heavily inspired by the arguments in \cite[p.~472]{BCL94}, where $p$-Hanner's inequality for $L^p$-norms is shown by using the following characterization of sums of $p$-th powers: for any $x,y \in \mathbb{R}$, 
	\begin{equation}\label{e:BCLlemma}
		\abs{x + y}^{p} + \abs{x - y}^{p} = 
		\begin{cases}
			\sup_{s > 0}\bigl\{ \psi_{p}(s)\abs{x}^{p} + \psi_{p}(s^{-1})\abs{y}^{p} \bigr\} \quad &\text{if $p \in (1,2]$,} \\
			\inf_{s > 0}\bigl\{ \psi_{p}(s)\abs{x}^{p} + \psi_{p}(s^{-1})\abs{y}^{p} \bigr\} \quad &\text{if $p \in [2,\infty)$;}
		\end{cases}
	\end{equation}
	see \cite[Lemma 4]{BCL94} for a proof of \eqref{e:BCLlemma}.
	
	\ref{it:ICp.small}: 
	By considering $x + y, x - y$ in \eqref{e:BCLlemma} instead of $x,y$, we have that for any $s \in (0,\infty)$, 
	\[
	2^{-p}\psi_{p}(s)\abs{x + y}^{p} + 2^{-p}\psi_{p}(s^{-1})\abs{x - y}^{p} \le \abs{x}^{p} + \abs{y}^{p}, 
	\]
	which means that $T^{s}$ satisfies \eqref{GC-cond} with $(q_1,q_2,n_1,n_2) = (p,p,2,2)$. 
	Since $s \in (0,\infty)$ is arbitrary, for any $f,g \in \mathcal{F}$,
	\begin{equation}\label{ICp.small}
    	\sup_{s > 0}\bigl\{\psi_{p}(s)\mathcal{E}(f) +\psi_{p}(s^{-1})\mathcal{E}(g)\bigr\} \le \mathcal{E}(f+g) + \mathcal{E}(f-g). 
    \end{equation}
    Using \eqref{e:BCLlemma} again, we get \eqref{Hanner.small}. 
	
	\ref{it:ICp-Cp.small}: 
	Let $f, g \in \mathcal{F}$ with $\mathcal{E}(f) \wedge \mathcal{E}(g) > 0$, set $a \coloneqq \mathcal{E}(f)^{1/(p - 1)}$ and $b \coloneqq \mathcal{E}(g)^{1/(p - 1)}$.
    Then, 
    \[
    \sup_{s > 0}\bigl\{\psi_{p}(s)\mathcal{E}(f) + \psi_{p}(s^{-1})\mathcal{E}(g)\bigr\}
    \ge \psi_{p}(b/a)a^{p - 1} + \psi_{p}(a/b)b^{p - 1}
    = 2(a + b)^{p - 1},
    \]
    which together with \eqref{ICp.small} yields \eqref{Cp.small}.
    
    The remaining statements, \ref{it:ICp.large} and \ref{it:ICp-Cp.large}, can be shown similarly by using \eqref{e:BCLlemma}.
\end{proof} 

The property \ref{GC} is stable under taking suitable limits and some algebraic operations like summations. 
To state precise results, we recall the following definition on convergences of functionals.
\begin{defn}[{\cite[Definition 4.1 and Proposition 8.1]{Dal}}]
    Let $\mathcal{X}$ be a topological space, let $F \colon \mathcal{X} \to [-\infty,\infty]$ and let $\{ F_{n} \colon \mathcal{X} \to [-\infty,\infty]\}_{n \in \mathbb{N}}$.
    \begin{enumerate}[label=\textup{(\arabic*)},align=left,leftmargin=*,topsep=2pt,parsep=0pt,itemsep=2pt]
        \item The sequence $\{ F_{n} \}_{n \in \mathbb{N}}$ is said to \emph{converge pointwise}\index{converge pointwise}\index{pointwise convergence} to $F$ if and only if $\lim_{n \to \infty}F_{n}(x) = F(x)$ for any $x \in \mathcal{X}$.
        \item Assume that $\mathcal{X}$ is a first-countable topological space.
        The sequence $\{ F_{n} \}_{n \in \mathbb{N}}$ is said to \emph{$\Gamma$-converge}\index{$\Gamma$-converge}\index{$\Gamma$-convergence} to $F$ (with respect to the topology of $\mathcal{X}$) if and only if the following conditions hold for any $x \in \mathcal{X}$:
        \begin{enumerate}[label=\textup{(\roman*)},align=left,leftmargin=*,topsep=2pt,parsep=0pt,itemsep=2pt]
            \item If $x_{n} \to x$ in $\mathcal{X}$, then $F(x) \le \liminf_{n \to \infty}F_{n}(x_{n})$.
            \item There exists a sequence $\{ x_n \}_{n \in \mathbb{N}}$ in $\mathcal{X}$ such that
            \begin{equation}\label{eq.limsup}
                \text{$x_n \to x$ in $\mathcal{X}$} \quad  \text{and} \quad  \text{$\displaystyle\limsup_{n \to \infty}F_{n}(x_{n}) \le F(x)$.}
            \end{equation}
        \end{enumerate}
        A sequence $\{ x_{n} \}_{n \in \mathbb{N}}$ satisfying \eqref{eq.limsup} is called a \emph{recovery sequence of $\{ F_{n} \}_{n \in \mathbb{N}}$ at $x$}.\index{recovery sequence}
    \end{enumerate}
\end{defn}

We also need the following reverse Minkowski inequality. 
See, e.g., \cite[Theorem 2.12]{AF} for a proof of this standard result.
\begin{prop}[Reverse Minkowski inequality]\label{prop.reverse} \index{reverse Minkowski inequality}
    Let $(Y,\mathcal{A},\mu)$ be a measure space\footnote{In the book \cite{AF}, the reverse Minkowski inequality is stated and proved only for the $L^{r}$-space over non-empty open subsets of the Euclidean space equipped with the Lebesgue measure, but the same proof works for any measure space.} and let $r \in (0,1]$.
    Then for any $\mathcal{A}$-measurable functions $f,g \colon Y \to [0,\infty]$,
    \begin{equation}\label{reverse}
        \biggl(\int_{Y}f^{r}\,d\mu\biggr)^{1/r} + \biggl(\int_{Y}g^{r}\,d\mu\biggr)^{1/r} \le \biggl(\int_{Y}(f + g)^{r}\,d\mu\biggr)^{1/r}.
    \end{equation}
\end{prop}

In the following definition, we introduce the set of $p$-homogeneous functionals on $\mathcal{F}$ which satisfies a contraction property or \ref{GC}. 
\begin{defn}\label{d:GC-cone}
    Recall that $\mathcal{F}$ is a linear subspace of $L^{0}(X,m)$. 
    Let $n_{1},n_{2} \in \mathbb{N}$, $q_{1} \in (0,p]$, $q_{2} \in [p,\infty]$ and $T = (T_{1},\dots,T_{n_{2}})\colon \mathbb{R}^{n_{1}} \to \mathbb{R}^{n_{2}}$ satisfy \eqref{GC-cond}. Assume that $T(\bm{u}) \in \mathcal{F}^{n_2}$ for any $\bm{u} = (u_1,\dots,u_{n_1}) \in \mathcal{F}^{n_1}$.
    Define 
    \begin{equation*}\label{Eachcont-cone}
        \mathcal{U}_{p,T}(\mathcal{F}) \coloneqq \mathcal{U}_{p,T} \coloneqq \biggl\{ \mathcal{E}' \colon \mathcal{F} \to [0,\infty) \biggm| 
        \begin{minipage}{218pt}
        	$\mathcal{E}'$ is $p$-homogeneous and for any $\bm{u} \in \mathcal{F}^{n_1}$, 
        	$\norm{\bigl(\mathcal{E}(T_{l}(\bm{u}))^{1/p}\bigr)_{l = 1}^{n_{2}}}_{\ell^{q_{2}}} \le \norm{\bigl(\mathcal{E}(u_{k})^{1/p}\bigr)_{k = 1}^{n_{1}}}_{\ell^{q_{1}}}$  
        \end{minipage}
        \biggr\}
    \end{equation*}
    and 
    \begin{equation*}\label{GC-cone}
        \mathcal{U}_{p}^{\mathrm{GC}}(\mathcal{F}) \coloneqq \mathcal{U}_{p}^{\mathrm{GC}} \coloneqq \{ \mathcal{E}' \colon \mathcal{F} \to [0,\infty) \mid \text{$\mathcal{E}'$ is $p$-homogeneous, $(\mathcal{E}',\mathcal{F})$ satisfies \ref{GC}} \}.
    \end{equation*}
\end{defn}

Now we can state the desired \emph{stability} of contraction properties under linear combination, pointwise convergence and $\Gamma$-convergence. 
\begin{prop}\label{prop.cone-gen}
	Let $n_{1},n_{2} \in \mathbb{N}$, $q_{1} \in (0,p]$, $q_{2} \in [p,\infty]$ and $T = (T_{1},\dots,T_{n_{2}})\colon \mathbb{R}^{n_{1}} \to \mathbb{R}^{n_{2}}$ satisfy \eqref{GC-cond} in Definition \ref{defn.GC}. Assume that $T(\bm{u}) \in \mathcal{F}^{n_2}$ for any $\bm{u} = (u_1,\dots,u_{n_1}) \in \mathcal{F}^{n_1}$. 
    \begin{enumerate}[label=\textup{(\alph*)},align=left,leftmargin=*,topsep=2pt,parsep=0pt,itemsep=2pt]
        \item \label{GC.cone} $a_{1}\mathcal{E}^{(1)} + a_{2}\mathcal{E}^{(2)} \in \mathcal{U}_{p,T}$ for any $\mathcal{E}^{(1)}, \mathcal{E}^{(2)} \in \mathcal{U}_{p,T}$ and any $a_{1},a_{2} \in [0,\infty)$. In particular, $a_{1}\mathcal{E}^{(1)} + a_{2}\mathcal{E}^{(2)} \in \mathcal{U}_{p}^{\mathrm{GC}}$ for any $\mathcal{E}^{(1)}, \mathcal{E}^{(2)} \in \mathcal{U}_{p}^{\mathrm{GC}}$ and any $a_{1},a_{2} \in [0,\infty)$. 
        \item \label{GC.pwlimit} Let $\bigl\{ \mathcal{E}^{(n)} \bigr\}_{n \in \mathbb{N}} \subseteq \mathcal{U}_{p,T}$, let $\mathcal{E}^{(\infty)} \colon \mathcal{F} \to [0, \infty)$ and assume that $\{ \mathcal{E}^{(n)} \}_{n \in \mathbb{N}}$ converges pointwise to $\mathcal{E}^{(\infty)}$. Then $\mathcal{E}^{(\infty)} \in \mathcal{U}_{p,T}$. In particular, if $\mathcal{E}^{(n)} \in \mathcal{U}_{p}^{\mathrm{GC}}$ for all $n \in \mathbb{N}$ in addition, then $\mathcal{E}^{(\infty)} \in \mathcal{U}_{p}^{\mathrm{GC}}$.
        \item \label{GC.gammalimit} Assume that $\mathcal{F} \subseteq L^{p}(X,m)$, and equip $\mathcal{F}$ with the relative topology inherited from the norm topology of $L^{p}(X,m)$. Let $\bigl\{ \mathcal{E}^{(n)} \bigr\}_{n \in \mathbb{N}} \subseteq \mathcal{U}_{p,T}$, let $\mathcal{E}^{(\infty)} \colon \mathcal{F} \to [0, \infty)$ and assume that $\{ \mathcal{E}^{(n)} \}_{n \in \mathbb{N}}$ $\Gamma$-converges to $\mathcal{E}^{(\infty)}$. Then $\mathcal{E}^{(\infty)} \in \mathcal{U}_{p,T}$. In particular, if $\mathcal{E}^{(n)} \in \mathcal{U}_{p}^{\mathrm{GC}}$ for all $n \in \mathbb{N}$ in addition, then $\mathcal{E}^{(\infty)} \in \mathcal{U}_{p}^{\mathrm{GC}}$.
    \end{enumerate}
\end{prop}
\begin{proof}
    The statement \ref{GC.pwlimit} is trivial, so we will show \ref{GC.cone} and \ref{GC.gammalimit}. 

    \ref{GC.cone}:
    Let $\mathcal{E}^{(1)}, \mathcal{E}^{(2)} \in \mathcal{U}_{p,T}$.
    Then $a\mathcal{E}^{(1)} \in \mathcal{U}_{p,T}$ is evident for any $a \in [0,\infty)$.
    Set $E(f) \coloneqq \mathcal{E}^{(1)}(f) + \mathcal{E}^{(2)}(f)$, $f \in \mathcal{F}$, and let $\bm{u} = (u_{1},\dots,u_{n_{1}}) \in \mathcal{F}^{n_{1}}$.
    It suffices to prove $\norm{\bigl(E(T_{l}(\bm{u}))^{1/p}\bigr)_{l = 1}^{n_{2}}}_{\ell^{q_{2}}} \le \norm{\bigl(E(u_{k})^{1/p}\bigr)_{k = 1}^{n_{1}}}_{\ell^{q_{1}}}$.
    For simplicity, we consider the case of $q_{2} < \infty$. (The case of $q_{2} = \infty$ is similar.) 
    Then we have 
    \begin{align}\label{GC.sum}
        &\sum_{l = 1}^{n_{2}}E\bigl(T_{l}(\bm{u})\bigr)^{q_{2}/p} 
            = \sum_{l = 1}^{n_{2}}\Bigl[\mathcal{E}^{(1)}\bigl(T_{l}(\bm{u})\bigr) + \mathcal{E}^{(2)}\bigl(T_{l}(\bm{u})\bigr)\Bigr]^{q_{2}/p} \nonumber \\
        &\le \left(\sum_{i \in \{ 1, 2 \}}\left[\sum_{l = 1}^{n_{2}}\mathcal{E}^{(i)}\bigl(T_{l}(\bm{u})\bigr)^{q_{2}/p}\right]^{p/q_{2}}\right)^{q_{2}/p} \quad \text{(by the triangle inequality for $\norm{\,\cdot\,}_{\ell^{q_{2}/p}}$)} \nonumber \\ 
        &\le \left(\left[\sum_{k = 1}^{n_{1}}\mathcal{E}^{(1)}(u_{k})^{q_{1}/p}\right]^{p/q_{1}} + \left[\sum_{k = 1}^{n_{1}}\mathcal{E}^{(2)}(u_{k})^{q_{1}/p}\right]^{p/q_{1}}\right)^{q_{2}/p} \quad \text{(by  $\mathcal{E}^{(1)}, \mathcal{E}^{(2)} \in \mathcal{U}_{p,T}$)} \nonumber \\
        &\overset{\eqref{reverse}}{\le} \left(\sum_{k = 1}^{n_{1}}\Bigl[\mathcal{E}^{(1)}(u_{k}) + \mathcal{E}^{(2)}(u_{k})\Bigr]^{q_{1}/p}\right)^{\frac{p}{q_{1}} \cdot \frac{q_{2}}{p}}
        = \left(\sum_{k = 1}^{n_{1}}E(u_{k})^{q_{1}/p}\right)^{q_{2}/q_{1}},
    \end{align}
    which implies $E \in \mathcal{U}_{p,T}$.

    \ref{GC.gammalimit}:
    Let $\bm{u} = (u_{1},\dots,u_{n_{1}}) \in \mathcal{F}^{n_{1}}$, and let $\{ \bm{u}_{n} = (u_{1,n},\dots,u_{n_{1},n}) \}_{n \in \mathbb{N}} \subseteq \mathcal{F}^{n_{1}}$ be a recovery sequence of $\{ \mathcal{E}^{(n)} \}_{n \in \mathbb{N}}$ at $\bm{u}$.
    We first show tnat $\norm{T_{l}(\bm{u}) - T_{l}(\bm{u}_{n})}_{L^{p}(X,m)} \to 0$ as $n \to \infty$.
    Indeed, for any $\bm{v} = (v_{1},\dots,v_{n_{1}})$ and any $\bm{z} = (z_{1},\dots,z_{n_{1}}) \in L^{p}(X,m)^{n_{1}}$, we see that
    \begin{align}\label{T.Lp}
    	\max_{l \in \{ 1,\dots,n_{2} \}}\norm{T_{l}(\bm{v}) - T_{l}(\bm{z})}_{L^{p}(X,m)}^{p}
        &\overset{\eqref{GC-cond}}{\le} \int_{X}\norm{\bm{v}(x) - \bm{z}(x)}_{\ell^{q_{1}}}^{p}\,m(dx) \nonumber \\
        &= \int_{X}\left(\sum_{k = 1}^{n_{1}}\abs{v_{k}(x) - z_{k}(x)}^{p \cdot \frac{q_{1}}{p}}\right)^{p/q_{1}}\,m(dx) \nonumber \\
        &\le n_{1}^{(p - q_{1})/q_{1}}\sum_{k = 1}^{n_{1}}\norm{v_{k} - z_{k}}_{L^{p}(X,m)}^{p}, 
    \end{align}
    where we used H\"{o}lder's inequality in the last line. 
    Since $\max_{k}\norm{u_{k} - u_{k,n}}_{L^{p}(X,m)} \to 0$ as $n \to \infty$, \eqref{T.Lp} implies the desired convergence $\norm{T_{l}(\bm{u}) - T_{l}(\bm{u}_{n})}_{L^{p}(X,m)} \to 0$.
    
    We are now ready to prove $\mathcal{E}^{(\infty)} \in \mathcal{U}_{p,T}$.
    It is easy to see that $\mathcal{E}^{(\infty)}$ is $p$-homogeneous (see, e.g., \cite[Proposition 11.6]{Dal} for a proof of this fact). 
    Assume that $q_{2} < \infty$ since the case $q_{2} = \infty$ is similar. 
    Then,
    \begin{align*}
        \sum_{l = 1}^{n_{2}}\mathcal{E}^{(\infty)}\bigl(T_{l}(\bm{u})\bigr)^{q_{2}/p}
        &\le \sum_{l = 1}^{n_{2}}\liminf_{n \to \infty}\mathcal{E}^{(n)}\bigl(T_{l}(\bm{u}_{n})\bigr)^{q_{2}/p}
        \le \liminf_{n \to \infty}\sum_{l = 1}^{n_{2}}\mathcal{E}^{(n)}\bigl(T_{l}(\bm{u}_{n})\bigr)^{q_{2}/p} \nonumber\\
        &\le \liminf_{n \to \infty}\left(\sum_{k = 1}^{n_{1}}\mathcal{E}^{(n)}(u_{k,n})^{q_{1}/p}\right)^{\frac{p}{q_{1}} \cdot \frac{q_{2}}{p}}
        = \left(\sum_{k = 1}^{n_{1}}\mathcal{E}^{(\infty)}(u_{k})^{q_{1}/p}\right)^{\frac{p}{q_{1}} \cdot \frac{q_{2}}{p}},
    \end{align*}
    which proves $\mathcal{E}^{(\infty)} \in \mathcal{U}_{p,T}$.
\end{proof}

\section{Differentiability of \texorpdfstring{$p$}{p}-energy forms and related results}\label{sec.diffble}
In this section, we show the existence of the derivative in \eqref{intro.diffble} for any $p$-energy form $(\mathcal{E},\mathcal{F})$ satisfying $p$-Clarkson's inequality, \eqref{Cp.small} or \eqref{Cp.large} in Proposition \ref{prop.GC-list}, and establish fundamental properties of the ``two-variable version'' of $\mathcal{E}$ defined by \eqref{intro.diffble}, i.e.,
\[
\mathcal{E}(f;g) \coloneqq \frac{1}{p}\left.\frac{d}{dt}\mathcal{E}(f + tg)\right|_{t = 0}.
\]

Throughout this section, we fix $p \in (1,\infty)$, a measure space $(X,\mathcal{B},m)$, and a \emph{$p$-energy form} $(\mathcal{E},\mathcal{F})$ on $(X,m)$ in the following sense: 
\begin{defn}[$p$-Energy form]\label{defn.p-form} 
    Let $\mathcal{F}$ be a linear subspace of $L^{0}(X,m)$ and let $\mathcal{E} \colon \mathcal{F} \to [0,\infty)$.
    The pair $(\mathcal{E},\mathcal{F})$ is said to be a \emph{$p$-energy form}\index{$p$-energy form} on $(X,m)$ if and only if $\mathcal{E}^{1/p}$ is a seminorm on $\mathcal{F}$.
\end{defn}
Note that the same argument as in the proof of Proposition \ref{prop.GC-list}-\ref{GC.tri} implies that $\mathcal{E}$ is convex on $\mathcal{F}$ (see \eqref{e:convex}). 

\subsection{\texorpdfstring{$p$}{p}-Clarkson's inequality and differentiability}
In this section, we mainly deal with $p$-energy forms satisfying $p$-Clarkson's inequality in the following sense. 
\begin{defn}[$p$-Clarkson's inequality]\label{d:Cp} 
    The pair $(\mathcal{E},\mathcal{F})$ is said to satisfy \emph{$p$-Clarkson's inequality}\index{$p$-Clarkson's inequality (for $p$-energy form)}, \ref{Cp} for short, if and only if for any $f,g \in \mathcal{F}$,
    \begin{equation}\label{Cp}
        \begin{cases}
            \mathcal{E}(f+g) + \mathcal{E}(f-g) \ge 2\bigl(\mathcal{E}(f)^{\frac{1}{p-1}} + \mathcal{E}(g)^{\frac{1}{p-1}}\bigr)^{p - 1} \quad &\text{if $p \in (1,2]$,} \\ 
            \mathcal{E}(f+g) + \mathcal{E}(f-g) \le 2\bigl(\mathcal{E}(f)^{\frac{1}{p-1}} + \mathcal{E}(g)^{\frac{1}{p-1}}\bigr)^{p - 1} \quad &\text{if $p \in [2,\infty)$.}
        \end{cases} \tag*{\textup{(Cla)$_p$}}
    \end{equation}
\end{defn}
\begin{rmk}\label{rmk:Cp-weak} 
	The following weaker version of $p$-Clarkson's inequality is also well known: for any $f,g \in \mathcal{F}$, 
	\begin{equation}\label{Cp-weak}
        \begin{cases}
            \mathcal{E}(f+g) + \mathcal{E}(f-g) \le 2\bigl(\mathcal{E}(f) + \mathcal{E}(g)\bigr) \quad &\text{if $p \in (1,2]$,} \\ 
            \mathcal{E}(f+g) + \mathcal{E}(f-g) \ge 2\bigl(\mathcal{E}(f) + \mathcal{E}(g)\bigr) \quad &\text{if $p \in [2,\infty)$.}
        \end{cases} \tag*{\textup{(Cla)$^{\prime}_p$}}
    \end{equation}
    Since, for any $a,b \in [0,\infty)$, H\"{o}lder's inequality yields $\bigl(a^{\frac{1}{p-1}}+b^{\frac{1}{p-1}}\bigr)^{p-1} \geq 2^{p-2}(a+b)$ if $p \in (1,2]$ and $\bigl(a^{\frac{1}{p-1}}+b^{\frac{1}{p-1}}\bigr)^{p-1} \leq 2^{p-2}(a+b)$ if $p \in [2,\infty)$, \emph{\ref{Cp} with $\frac{f+g}{2},\frac{f-g}{2}$ in place of $f,g$ implies \ref{Cp-weak}}. 
    In this paper, we will use this implication without further notice. 
\end{rmk} 

To state a consequence of \ref{Cp} on the convexity of $\mathcal{E}^{1/p}$, let us recall the notion of uniform convexity. 
See, e.g., \cite[Definition 1]{Cla36}.
(The notion of uniform convexity is usually defined for normed spaces in the literature. We present the definition for seminormed spaces because we are mainly interested in $(\mathcal{F},\mathcal{E}^{1/p})$.) 
\begin{defn}[Uniformly convex seminormed spaces]
	Let $(\mathcal{X}, \abs{\,\cdot\,}_{\mathcal{X}})$ be a seminormed space.
	We say that $(\mathcal{X}, \abs{\,\cdot\,}_{\mathcal{X}})$ is \emph{uniformly convex}\index{uniformly convex} if and only if for any $\varepsilon \in (0,\infty)$ there exists $\delta \in (0,\infty)$ with the property that $\abs{f + g}_{\mathcal{X}} \le 2(1 - \delta)$ for any $f, g \in \mathcal{X}$ satisfying $\abs{f}_{\mathcal{X}} = \abs{g}_{\mathcal{X}} = 1$ and $\abs{f - g}_{\mathcal{X}} > \varepsilon$.
\end{defn}

It is well known that \ref{Cp} implies the uniform convexity as follows. 
\begin{prop}\label{p:uc}
	Assume that $(\mathcal{E},\mathcal{F})$ satisfies \ref{Cp}. Then $(\mathcal{F},\mathcal{E}^{1/p})$ is uniformly convex.
\end{prop}
\begin{proof}
	This can be proved by the same argument as in \cite[Proof of Corollary of Theorem 2]{Cla36}, where $L^p$-spaces are shown to be uniformly convex when $p \in (1,\infty)$. 
\end{proof}

Moreover, \ref{Cp} provides the following quantitative estimate for the central difference, which plays a central role in this section.
\begin{prop}\label{prop.diffble}
	Assume that $(\mathcal{E},\mathcal{F})$ satisfies \ref{Cp}.
    Then for any $f,g \in \mathcal{F}$,
    \begin{equation}\label{c-diff}
    	\mathcal{E}(f + g) + \mathcal{E}(f - g) - 2\mathcal{E}(f) \le
        2\bigl(1 \vee (p-1)\bigr)\Bigl[\mathcal{E}(f)^{\frac{1}{p - 1}} + \mathcal{E}(g)^{\frac{1}{p - 1}}\Bigr]^{(p - 2)^{+}}\mathcal{E}(g)^{1 \wedge \frac{1}{p - 1}},
    \end{equation} 
    and the function $\mathbb{R} \ni t \mapsto \mathcal{E}(f + tg) \in [0, \infty)$ is differentiable. Moreover, for any $c \in (0,\infty)$, 
    \begin{equation}\label{e:frechet.diffble}
    	\lim_{\delta \downarrow 0}\sup_{f \in \mathcal{F};\, \mathcal{E}(f) \leq c/(p-2)^{+}} \sup_{g \in \mathcal{F};\, \mathcal{E}(g) \leq 1}\abs{\frac{\mathcal{E}(f + \delta g) - \mathcal{E}(f)}{\delta} - \left.\frac{d}{dt}\mathcal{E}(f + tg)\right|_{t = 0}} = 0. 
    \end{equation}
\end{prop}
\begin{proof}
    Let $f,g \in \mathcal{F}$. If $p \in (1,2]$, then \eqref{c-diff} is immediate from \ref{Cp-weak}.
	If $p \in (2,\infty)$, then setting $a \coloneqq \mathcal{E}(f)^{1/(p - 1)}$ and $b \coloneqq \mathcal{E}(g)^{1/(p - 1)}$, we see from \ref{Cp} that 
    \[
    \mathcal{E}(f + g) + \mathcal{E}(f - g) - 2\mathcal{E}(f)
    \le 2((a + b)^{p - 1} - a^{p - 1})
    = 2(p - 1)\int_{a}^{a + b}s^{p - 2}\,ds
    \le 2(p - 1)(a + b)^{p - 2}b, 
    \]
	proving \eqref{c-diff}.
	For the rest of the proof, we first note that by the convexity of $\mathcal{E}$, 
    \begin{equation} \label{eq:cdiff.right.left}
    \text{the limits} \quad \lim_{\delta \downarrow 0}\frac{\mathcal{E}(f + \delta g) - \mathcal{E}(f)}{\delta} 
    \quad \text{and} \quad 
    \lim_{\delta \downarrow 0}\frac{\mathcal{E}(f - \delta g) - \mathcal{E}(f)}{-\delta} 
    \quad \text{exist in $\mathbb{R}$,}
	\end{equation}
	and for any $\delta \in (0,\infty)$,
    \begin{align}\label{d:cdiff}
	D_{\delta}(f;g)
		\coloneqq \mathcal{E}(f + \delta g) + \mathcal{E}(f - \delta g) - 2\mathcal{E}(f) 
		&\geq 0, \\
	\abs{\frac{\mathcal{E}(f + \delta g) - \mathcal{E}(f)}{\delta} - \lim_{s \downarrow 0}\frac{\mathcal{E}(f + s g) - \mathcal{E}(f)}{s}}
		&\leq \frac{D_{\delta}(f;g)}{\delta}. 
	\label{e:frechet.diffble.cdiff}
    \end{align}
    On the other hand, we see from \eqref{c-diff} that for any $\delta \in (0,\infty)$,
    \begin{equation} \label{e:c-diff.unif}
        \frac{D_{\delta}(f;g)}{\delta} \leq
        \begin{cases}
            2\delta^{p - 1}\mathcal{E}(g) \quad &\text{if $p \in (1,2]$,} \\
            2(p - 1)\delta^{\frac{1}{p - 1}}\Bigl[\mathcal{E}(f)^{\frac{1}{p - 1}} + \delta^{\frac{p}{p - 1}} \mathcal{E}(g)^{\frac{1}{p - 1}}\Bigr]^{p - 2}\mathcal{E}(g)^{\frac{1}{p - 1}} \quad &\text{if $p \in (2,\infty)$.}
        \end{cases}
    \end{equation}
    By \eqref{d:cdiff} and \eqref{e:c-diff.unif}, the limits in \eqref{eq:cdiff.right.left} coincide, so that the function $t \mapsto \mathcal{E}(f + tg)$ is differentiable at $0$ and thereby at any $s \in \mathbb{R}$ by replacing $f$ with $f + sg$,
	and then we obtain \eqref{e:frechet.diffble} by combining this differentiability at $0$ with \eqref{e:frechet.diffble.cdiff} and \eqref{e:c-diff.unif}. 
\end{proof}

Proposition \ref{prop.diffble}, especially \eqref{e:frechet.diffble}, implies the Fr\'{e}chet differentiability of $\mathcal{E}$ on $\mathcal{F}/\mathcal{E}^{-1}(0)$.
We record this fact and basic properties of these derivatives in the following theorem.
\begin{thm}\label{thm.p-form}
    Assume that $(\mathcal{E},\mathcal{F})$ satisfies \ref{Cp}.
    Then $\mathcal{E} \colon \mathcal{F}/\mathcal{E}^{-1}(0) \to [0,\infty)$ is Fr\'{e}chet differentiable on the quotient normed space $\mathcal{F}/\mathcal{E}^{-1}(0)$.
    In particular, for any $f,g \in \mathcal{F}$,
    \begin{equation}\label{exist-deriva}
        \text{the derivative} \quad \mathcal{E}(f; g) \coloneqq \frac{1}{p}\left.\frac{d}{dt}\mathcal{E}(f + tg)\right|_{t = 0} \in \mathbb{R} \quad \text{exists,}
    \end{equation}
    the map $\mathcal{E}(f; \,\cdot\,) \colon \mathcal{F} \to \mathbb{R}$ is linear, $\mathcal{E}(f; f) = \mathcal{E}(f)$, and $\mathcal{E}(f; h) = 0$ for any $h \in \mathcal{E}^{-1}(0)$.
    Moreover, for any $f, f_1, f_2, g \in \mathcal{F}$, any $h \in \mathcal{E}^{-1}(0)$ and any $a \in \mathbb{R}$, the following hold:
    \begin{gather}
        \text{$\mathbb{R} \ni t \mapsto \mathcal{E}(f + tg; g) \in \mathbb{R}$ is strictly increasing if and only if $\mathcal{E}(g) > 0$}. \label{form.mono} \\
        \mathcal{E}(af; g) = \sgn(a)\abs{a}^{p - 1}\mathcal{E}(f; g), \qquad \mathcal{E}(f + h; g) = \mathcal{E}(f;g). \label{form.basic} \\
        \abs{\mathcal{E}(f; g)} \le \mathcal{E}(f)^{(p - 1)/p}\mathcal{E}(g)^{1/p}. \label{bdd.form} \\
        \abs{\mathcal{E}(f_1; g) - \mathcal{E}(f_2; g)} \le C_{p}\bigl(\mathcal{E}(f_1) \vee \mathcal{E}(f_2)\bigr)^{(p - 1 - \alpha_{p})/p}\mathcal{E}(f_1 - f_2)^{\alpha_{p}/p} \mathcal{E}(g)^{1/p}, \label{ncont}
    \end{gather}
    where $\alpha_{p} \coloneqq \frac{1}{p} \wedge \frac{p - 1}{p}$ and $C_{p} \in (0,\infty)$ is a constant determined solely and explicitly by $p$.
\end{thm}
\begin{rmk}\label{rmk:holexponent}
	The H\"{o}lder continuity exponent $\alpha_{p}$ appearing in \eqref{ncont} is not optimal because this exponent can be improved to $(p - 1) \wedge 1$ in the case of $\mathcal{E}(f; g) = \int_{\mathbb{R}^{n}}\abs{\nabla f}^{p - 2}\langle \nabla f, \nabla g \rangle\,dx$. 
	However, whether such an improved H\"{o}lder continuity holds is unclear even for concrete $p$-energy forms constructed in the previous works \cite{CGQ22,Kig23,MS+,Shi24}. 
	We can see the optimal H\"{o}lder continuity (\eqref{ncont} with $(p - 1) \wedge 1$ in place of $\alpha_{p}$) for $p$-energy forms constructed in \cite{KS.lim}, where a direct construction of $p$-energy forms based on the Korevaar--Schoen type $p$-energy forms is presented. 
	
	Note added in proof.
	After this paper was accepted for publication, the second-named author has found a relevant paper \cite{XR91} by Xu and Roach, and has realized that \eqref{ncont} with $\alpha_{p}$ replaced by $(p-1)\wedge\frac{1}{p-1}$ can be deduced from \cite[Theorem 2 and (3.1)$'$]{XR91} under the assumption of \ref{Cp}, and that so can \eqref{ncont} with $\alpha_{p}$ replaced by $(p-1)\wedge 1$ under the assumption of $p$-Hanner's inequality \eqref{Hanner.small}/\eqref{Hanner.large} in Proposition \ref{p:ICp}.
	Similarly, \eqref{e:conti.inverse} in Theorem \ref{thm.conti-inverse} below with the exponent $(p-1)\wedge 1$ replaced by $(p(p-1))\wedge\frac{p}{p-1}$ can be deduced from \cite[Theorem 1 and (2.16)]{XR91} under the assumption of \ref{Cp}, and so can \eqref{e:conti.inverse} with $(p-1)\wedge 1$ replaced by $p\wedge\frac{p}{p-1}$ under the assumption of $p$-Hanner's inequality \eqref{Hanner.small}/\eqref{Hanner.large}.
	The details of these observations are planned to be presented in a future article.
\end{rmk} 
\begin{proof}[Proof of Theorem \ref{thm.p-form}]
    The existence of $\mathcal{E}(f; g)$ in \eqref{exist-deriva} is already proved in Proposition \ref{prop.diffble}.   
    The properties $\mathcal{E}(f; ag) = a\mathcal{E}(f; g)$, $\mathcal{E}(af; g) = \sgn(a)\abs{a}^{p - 1}\mathcal{E}(f; g)$ and $\mathcal{E}(f; f) = \mathcal{E}(f)$ are obvious from the definition.
    The equalities $\mathcal{E}(f + h; g) = \mathcal{E}(f + g)$ and $\mathcal{E}(f; h) = 0$ for any $h \in \mathcal{E}^{-1}(0)$ follow from the triangle inequality for $\mathcal{E}^{1/p}$, so \eqref{form.basic} holds.
    The property \eqref{form.mono} is a consequence of the uniform convexity of $(\mathcal{F},\mathcal{E}^{1/p})$ in Proposition \ref{p:uc} and the differentiability \eqref{exist-deriva} of $\mathcal{E}$. 
    
    To show that $\mathcal{E}(f; \,\cdot\,)$ is linear, it suffices to prove $\mathcal{E}(f; g_{1} + g_{2}) = \mathcal{E}(f; g_{1}) + \mathcal{E}(f; g_{2})$ for any $g_{1}, g_{2} \in \mathcal{F}$.
    For any $t > 0$, the convexity of $\mathcal{E}$ implies that
    \begin{align}\label{bdd1}
        \frac{\mathcal{E}\bigl(f + t(g_1 + g_2)\bigr) - \mathcal{E}(f)}{t}
        &= \frac{\mathcal{E}\bigl(\frac{1}{2}(f + 2tg_{1}) + \frac{1}{2}(f + 2tg_{2})\bigr) - \mathcal{E}(f)}{t} \nonumber \\
        &\le \frac{\mathcal{E}(f + 2tg_{1}) - \mathcal{E}(f)}{2t} + \frac{\mathcal{E}(f + 2tg_{2}) - \mathcal{E}(f)}{2t}.
    \end{align}
    Passing to the limit as $t \downarrow 0$, we get $\mathcal{E}(f; g_{1} + g_{2}) \le \mathcal{E}(f; g_{1}) + \mathcal{E}(f; g_{2})$.
    We obtain the converse inequality by noting that     
    \[
    \frac{\mathcal{E}(f - tg) - \mathcal{E}(f)}{t} \to -\left.\frac{d}{dt}\mathcal{E}(f + tg)\right|_{t = 0} = -p\mathcal{E}(f; g) \quad \text{as $t \downarrow 0$,}
    \]
    and by applying \eqref{bdd1} with $-g_{1}, -g_{2}$ in place of $g_{1},g_{2}$ respectively. 
    
    The H\"{o}lder-type estimate \eqref{bdd.form} follows from the following elementary estimate: 
    \begin{equation}\label{e:p-power}
    	\abs{a^q - b^q}
    	= \abs{\int_{a \wedge b}^{a \vee b}qt^{q - 1}\,dt}
    	\le q(a^{q - 1} \vee b^{q - 1})\abs{a - b} \quad \text{for $q \in (0,\infty)$, $a, b \in [0,\infty)$}.
    \end{equation}
    Indeed, by \eqref{e:p-power} and the triangle inequality for $\mathcal{E}^{1/p}$, for any $t > 0$,
    \begin{equation}\label{e:preHolder}
    \begin{split}
    	\abs{\frac{\mathcal{E}(f + tg) - \mathcal{E}(f)}{t}}
    	&\le p\bigl(\mathcal{E}(f + tg)^{1/p} \vee \mathcal{E}(f)^{1/p}\bigr)^{p - 1}\mathcal{E}(g)^{1/p} \\
    	&\le p\bigl(\mathcal{E}(f)^{1/p} + t \mathcal{E}(g)^{1/p}\bigr)^{p - 1}\mathcal{E}(g)^{1/p}.
    \end{split}
    \end{equation}
    We obtain \eqref{bdd.form} by letting $t \downarrow 0$ in \eqref{e:preHolder}. We conclude that $\mathcal{E}(f;\,\cdot\,)$ is the Fr\'{e}chet derivative of $\mathcal{E}$ at $f$ by \eqref{e:frechet.diffble}, the linearity of $\mathcal{E}(f;\,\cdot\,)$ and \eqref{bdd.form}. 

    In the rest of this proof, we prove \eqref{ncont}.
    Our proof is partially inspired by an argument due to \v{S}mulian in \cite{Smu40}. 
    In this proof, $C_{p,i}$, $i \in \{ 1,\dots,5 \}$, is a constant depending only on $p$.
    We first show an analogue of \eqref{c-diff} for $\mathcal{E}^{1/p}$.
    Using \eqref{e:p-power}, we can show that there exists $c_{\ast} \in (0,2^{-p^{3}})$ depending only on $p$ such that 
    \begin{equation}\label{perturb}
        \sup\biggl\{ \frac{\abs{\mathcal{E}(f) - \mathcal{E}(f + \delta g)}}{\mathcal{E}(f)} \biggm| 
		\text{$f,g \in \mathcal{F}$, $\delta \in (0,\infty)$, $\delta < c_{\ast}\mathcal{E}(f)^{1/p}$, $\mathcal{E}(g) = 1$}
        \biggr\} 
        \le \frac{1}{10}. 
    \end{equation}
    Define $\psi \colon \mathbb{R} \to \mathbb{R}$ by $\psi(t) \coloneqq \abs{t}^{1/p}$, and fix $f,g \in \mathcal{F}$ and $\delta \in (0,\infty)$ with $\delta < c_{\ast}\mathcal{E}(f)^{1/p}$ and $\mathcal{E}(g) = 1$.  
    Then there exist $\theta_{1},\theta_{2},\theta \in [0,1]$ such that
    \begin{align}\label{1/p-cdiff}
        0
        &\le \psi(\mathcal{E}(f + \delta g)) + \psi(\mathcal{E}(f - \delta g)) - 2\psi(\mathcal{E}(f)) \nonumber \\
        &= \psi'(A_{1,\delta})\bigl[\mathcal{E}(f + \delta g) - \mathcal{E}(f)\bigr]
        - \psi'(A_{2,\delta})\bigl[\mathcal{E}(f) - \mathcal{E}(f - \delta g)\bigr] \nonumber \\
        &= \psi'(A_{1}(\delta))D_{\delta}(f; g) - \bigl(\psi'(A_{1,\delta}) - \psi'(A_{2,\delta})\bigr)\bigl[\mathcal{E}(f) - \mathcal{E}(f - \delta g)\bigr] \nonumber \\
        &= \psi'(A_{1,\delta})D_{\delta}(f; g) - \psi''\bigl(A_{1,\delta} + \theta(A_{2,\delta} - A_{1,\delta})\bigr)(A_{2,\delta} - A_{1,\delta})\bigl[\mathcal{E}(f) - \mathcal{E}(f - \delta g)\bigr], 
    \end{align}
    where $D_{\delta}(f;g)$ is the same as in \eqref{d:cdiff} in the proof of Proposition \ref{prop.diffble} and 
    \[
    A_{1,\delta} \coloneqq \mathcal{E}(f) + \theta_{1}\bigl[\mathcal{E}(f + \delta g) - \mathcal{E}(f)\bigr], \quad
    A_{2,\delta} \coloneqq \mathcal{E}(f - \delta g) + \theta_{2}\bigl[\mathcal{E}(f) - \mathcal{E}(f - \delta g)\bigr].
    \]
    By \eqref{perturb}, we note that $\abs{A_{1,\delta}} \wedge \abs{A_{1,\delta} + \theta(A_{2,\delta} - A_{1,\delta})} \ge \frac{1}{2}\mathcal{E}(f)$, which together with \eqref{1/p-cdiff} and the estimate \eqref{c-diff} in Proposition \ref{prop.diffble} implies that for any $(\delta,f) \in (0,\infty) \times \mathcal{F}$ with $0 < \delta < c_{\ast}\mathcal{E}(f)^{1/p}$, 
    \begin{align*}
        0
        &\le \psi(\mathcal{E}(f + \delta g)) + \psi(\mathcal{E}(f - \delta g)) - 2\psi(\mathcal{E}(f)) \nonumber \\
        &\le C_{p,1}\Bigl(\mathcal{E}(f)^{\frac{1}{p} - 1 + \frac{(p - 2)^{+}}{p - 1}}\delta^{p \wedge \frac{p}{p - 1}} + \mathcal{E}(f)^{\frac{1}{p} - 2 + \frac{2(p - 1)}{p}}\delta^{2}\Bigr) \nonumber \\
        &\le C_{p,1}\delta \cdot \delta^{(p - 1) \wedge \frac{1}{p - 1}}\Bigl(\mathcal{E}(f)^{\frac{1}{p} - 1 + \frac{(p - 2)^{+}}{p - 1}} + \mathcal{E}(f)^{\frac{1}{p} - 2 + \frac{2(p - 1)}{p}}\Bigr). 
    \end{align*}
    In particular, if $\mathcal{E}(f) = 1$, then
    \begin{equation}\label{ncont.1}
        \mathcal{E}(f + \delta g)^{1/p} + \mathcal{E}(f - \delta g)^{1/p} \le 2 + C_{p,1}\delta^{(p - 1) \wedge \frac{1}{p-1}}\delta \quad \text{for any $\delta \in (0,c_{\ast})$}. 
    \end{equation}
    Next let $f_{1},f_{2} \in \mathcal{F}$.
    Then, by \eqref{bdd.form} and \eqref{e:p-power},
	\begin{align}\label{ncont.2}
		\abs{\mathcal{E}(f_2; f_1) - \mathcal{E}(f_1)}
		&\le \abs{\mathcal{E}(f_2; f_1) - \mathcal{E}(f_2)} + \abs{\mathcal{E}(f_2) - \mathcal{E}(f_1)} \nonumber \\
		&\le \Bigl(\mathcal{E}(f_2)^{(p - 1)/p} + p\Bigl(\mathcal{E}(f_{2})^{(p - 1)/p} \vee \mathcal{E}(f_1)^{(p - 1)/p}\Bigr)\Bigr)\mathcal{E}(f_1 - f_2)^{1/p} \nonumber \\
        &\le C_{p,2}\Bigl(\mathcal{E}(f_1)^{(p - 1)/p} \vee \mathcal{E}(f_2)^{(p - 1)/p}\Bigr)\mathcal{E}(f_1 - f_2)^{1/p}.
	\end{align}
    Now, for any $f_1,f_2,g \in \mathcal{F}$ with $\mathcal{E}(f_1) = \mathcal{E}(g) = 1$ and any $\delta \in (0,c_{\ast})$, we see that
	\begin{align*}
		&\mathcal{E}(f_1; \delta g) - \mathcal{E}(f_2; \delta g) \\
		&= \mathcal{E}(f_1; f_1 + \delta g) + \mathcal{E}(f_2; f_1 - \delta g) - \mathcal{E}(f_1) - \mathcal{E}(f_2; f_1) \\
		&\overset{\eqref{bdd.form}}{\le} \Bigl(\mathcal{E}(f_1)^{(p - 1)/p} \vee \mathcal{E}(f_2)^{(p - 1)/p}\Bigr)\Bigl(\mathcal{E}(f_1 + \delta g)^{1/p} + \mathcal{E}(f_1 - \delta g)^{1/p}\Bigr) - \mathcal{E}(f_1) - \mathcal{E}(f_2; f_1) \\
		&\overset{\eqref{e:p-power},\eqref{ncont.1}}{\le} \Bigl(1 + C_{p,3}\mathcal{E}(f_1 - f_{2})^{1/p}\Bigr)\Bigl(2 + C_{p,1}\delta^{(p - 1) \wedge \frac{1}{p-1}}\delta\Bigr) - \mathcal{E}(f_1) - \mathcal{E}(f_2; f_1). 
	\end{align*}
	Similarly, we can show
	\begin{align*}
		&\mathcal{E}(f_1; \delta g) - \mathcal{E}(f_2; \delta g) \\
		&= -\mathcal{E}(f_1; f_1 - \delta g) - \mathcal{E}(f_2; f_1 + \delta g) + \mathcal{E}(f_1) + \mathcal{E}(f_2; f_1) \\
		&\ge -\Bigl(1 + C_{p,3}\mathcal{E}(f_1 - f_{2})^{1/p}\Bigr)\Bigl(2 + C_{p,1}\delta^{(p - 1) \wedge \frac{1}{p-1}}\delta\Bigr) + \mathcal{E}(f_1) + \mathcal{E}(f_2; f_1).
	\end{align*}
	From these estimates, we have
	\begin{align*}
		&\abs{\mathcal{E}(f_1; g) - \mathcal{E}(f_2; g)}
		= \frac{\abs{\mathcal{E}(f_1; \delta g) - \mathcal{E}(f_2; \delta g)}}{\delta} \\
		&\le \Bigl(1 + C_{p,3}\mathcal{E}(f_1 - f_{2})^{1/p}\Bigr)\Bigl(2\delta^{-1} + C_{p,1}\delta^{(p - 1) \wedge \frac{1}{p-1}}\Bigr) - \delta^{-1}\mathcal{E}(f_1) - \delta^{-1}\mathcal{E}(f_2; f_1) \\
        &= \Bigl(1 + C_{p,3}\mathcal{E}(f_1 - f_{2})^{1/p}\Bigr)\Bigl(2\delta^{-1} + C_{p,1}\delta^{(p - 1) \wedge \frac{1}{p-1}}\Bigr) - 2\delta^{-1}\mathcal{E}(f_1) + \delta^{-1}\bigl(\mathcal{E}(f_1) - \mathcal{E}(f_2; f_1)\bigr) \\
        &\overset{\eqref{ncont.2}}{\le} \Bigl(1 + C_{p,3}\mathcal{E}(f_1 - f_{2})^{1/p}\Bigr)\Bigl(2\delta^{-1} + C_{p,1}\delta^{(p - 1) \wedge \frac{1}{p-1}}\Bigr) - 2\delta^{-1} + C_{p,2}\delta^{-1}\mathcal{E}(f_1 - f_{2})^{1/p} \\
        &\le C_{p,4}\Bigl(\delta^{(p - 1) \wedge \frac{1}{p-1}} + \delta^{-1}\mathcal{E}(f_1 - f_{2})^{1/p}\Bigr).
	\end{align*}
    If $\mathcal{E}(f_1 - f_2) < c_{\ast}^{-p^{2}/((p - 1) \vee 1)}$, then, by choosing $\delta = \mathcal{E}(f_1 - f_{2})^{((p - 1) \vee 1)/p^{2}}$, we obtain
    \begin{equation}\label{ncont.3}
        \abs{\mathcal{E}(f_1; g) - \mathcal{E}(f_2; g)}
        \le C_{p,5}\mathcal{E}(f_1 - f_{2})^{((p - 1) \wedge 1)/p^{2}}.
    \end{equation}
	The same is clearly true if $\mathcal{E}(f_1 - f_2) \ge c_{\ast}^{-p^{2}/((p - 1) \vee 1)}$ since $\mathcal{E}(f_2) \le 2^{p - 1}\bigl(1 + \mathcal{E}(f_1 - f_2)\bigr)$.
    Finally, for any $f_1, f_2, g \in \mathcal{F}$ with $\mathcal{E}(f_1) \wedge \mathcal{E}(g) > 0$, we have
    \begin{align*}
        \abs{\mathcal{E}(f_1; g) - \mathcal{E}(f_2; g)}
        &= \mathcal{E}(f_1)^{(p - 1)/p}\mathcal{E}(g)^{1/p}
        \abs{\mathcal{E}\biggl(\frac{f_1}{\mathcal{E}(f_{1})^{1/p}}; \frac{g}{\mathcal{E}(g)^{1/p}}\biggr) - \mathcal{E}\biggl(\frac{f_2}{\mathcal{E}(f_{1})^{1/p}}; \frac{g}{\mathcal{E}(g)^{1/p}}\biggr)} \\
        &\overset{\eqref{ncont.3}}{\le} C_{p,5}\mathcal{E}(f_1)^{(p - 1)/p}\mathcal{E}(g)^{1/p}\mathcal{E}\biggl(\frac{f_1}{\mathcal{E}(f_{1})^{1/p}} - \frac{f_{2}}{\mathcal{E}(f_{1})^{1/p}}\biggr)^{((p - 1) \wedge 1)/p^{2}} \\
        &\overset{\eqref{ncont.3}}{\le} C_{p,5}\bigl(\mathcal{E}(f_1) \vee \mathcal{E}(f_2)\bigr)^{(p - 1 - \alpha_{p})/p}\mathcal{E}(g)^{1/p}\mathcal{E}(f_1 - f_{2})^{\alpha_{p}/p}.
    \end{align*}
    The same estimate is clearly true if $\mathcal{E}(f_{2}) \wedge \mathcal{E}(g) > 0$.
    Since \eqref{ncont} is obvious when $g \in \mathcal{E}^{-1}(0)$ or $\mathcal{E}(f_1) \vee \mathcal{E}(f_2) = 0$, we obtain \eqref{ncont}. 
\end{proof}

The following theorem gives a quantitative continuity for the inverse map of $f \mapsto \mathcal{E}(f; \,\cdot\,)$. 
\begin{thm}\label{thm.conti-inverse}
	Assume that $(\mathcal{E},\mathcal{F})$ satisfies \ref{Cp}. 
	Then for any $f,g \in \mathcal{F}$,  
	\begin{equation}\label{e:conti.inverse}
		\mathcal{E}(f - g) \le C_{p}' \bigl[\mathcal{E}(f) \vee \mathcal{E}(g)\bigr]^{\frac{1 + (p-1)(2-p)^{+}}{p}} \biggl(\sup_{\varphi \in \mathcal{F}; \mathcal{E}(\varphi) \le 1}\abs{\mathcal{E}(f; \varphi) - \mathcal{E}(g; \varphi)}\biggr)^{(p-1) \wedge 1}, 
	\end{equation}
	where $C_{p}' \in (0,\infty)$ is a constant determined solely and explicitly by $p$.
\end{thm}
\begin{proof}
	For ease of notation, for any linear functional $\Phi \colon \mathcal{F} \to \mathbb{R}$, we set $\norm{\Phi}_{\mathcal{F},\ast} \coloneqq \sup_{u \in \mathcal{F}; \mathcal{E}(u) \le 1}\abs{\Phi(u)}$. 
	Clearly, $\norm{\Phi_1 + \Phi_2}_{\mathcal{F},\ast} \le \norm{\Phi_1}_{\mathcal{F},\ast} + \norm{\Phi_2}_{\mathcal{F},\ast}$ for any linear functionals $\Phi_{1},\Phi_{2} \colon \mathcal{F} \to \mathbb{R}$. 
	Note that $\norm{\mathcal{E}(f; \,\cdot\,)}_{\mathcal{F},\ast} = \mathcal{E}(f)^{(p - 1)/p}$ for any $f \in \mathcal{F}$ by the H\"{o}lder-type estimate \eqref{bdd.form} in Theorem \ref{thm.p-form}.  
	In particular, for any $f,g \in \mathcal{F}$, 
	\begin{equation*}
		\abs{\mathcal{E}(f)^{\frac{p - 1}{p}} - \mathcal{E}(g)^{\frac{p - 1}{p}}}
		 = \abs{\norm{\mathcal{E}(f; \,\cdot\,)}_{\mathcal{F},\ast} - \norm{\mathcal{E}(g; \,\cdot\,)}_{\mathcal{F},\ast}}
		 \le \norm{\mathcal{E}(f; \,\cdot\,) - \mathcal{E}(g; \,\cdot\,)}_{\mathcal{F},\ast}, 
	\end{equation*}
	which together with \eqref{e:p-power} in the proof of Theorem \ref{thm.p-form} with $q = (p - 1)/p$ implies that 
	\begin{equation}\label{e:inverse.preconti}
		\abs{\mathcal{E}(f) - \mathcal{E}(g)} 
		\le \frac{p}{p - 1}\bigl(\mathcal{E}(f)^{1/p} \vee \mathcal{E}(g)^{1/p}\bigr)\norm{\mathcal{E}(f; \,\cdot\,) - \mathcal{E}(g; \,\cdot\,)}_{\mathcal{F},\ast}. 
	\end{equation}
	Let us define $\psi \colon \mathbb{R} \to \mathbb{R}$ by $\psi(t) \coloneqq \frac{1}{p}\mathcal{E}(f + t(g - f))$. 
	Then $\psi \in \contfunc^{1}(\mathbb{R})$ by \eqref{e:frechet.diffble} and \eqref{ncont}; indeed, \eqref{e:frechet.diffble} in Proposition \ref{prop.diffble} implies that $\psi'(t) = \mathcal{E}(f + t(g - f); g - f)$, which is continuous by \eqref{ncont} in Theorem \ref{thm.p-form}.  
	Now we see that 
	\begin{align*}
		&\abs{\psi'(0)}
		= \abs{\mathcal{E}(f; g - f)}
		\le \abs{\mathcal{E}(f; g) - \mathcal{E}(g)} + \abs{\mathcal{E}(g) - \mathcal{E}(f)} \\
		&\overset{\eqref{e:inverse.preconti}}{\le} \norm{\mathcal{E}(f; \,\cdot\,) - \mathcal{E}(g; \,\cdot\,)}_{\mathcal{F},\ast}\mathcal{E}(g)^{1/p} + \frac{p}{p - 1}\bigl(\mathcal{E}(f)^{1/p} \vee \mathcal{E}(g)^{1/p}\bigr)\norm{\mathcal{E}(f; \,\cdot\,) - \mathcal{E}(g; \,\cdot\,)}_{\mathcal{F},\ast} \\
		&\le \left(1 + \frac{p}{p - 1}\right)\bigl(\mathcal{E}(f)^{1/p} \vee \mathcal{E}(g)^{1/p}\bigr)\norm{\mathcal{E}(f; \,\cdot\,) - \mathcal{E}(g; \,\cdot\,)}_{\mathcal{F},\ast}. 
	\end{align*}
	Similarly, 
	\begin{equation*}
		\abs{\psi'(1)}
		= \abs{\mathcal{E}(g; g - f)} 
		\le \left(1 + \frac{p}{p - 1}\right)\bigl(\mathcal{E}(f)^{1/p} \vee \mathcal{E}(g)^{1/p}\bigr)\norm{\mathcal{E}(f; \,\cdot\,) - \mathcal{E}(g; \,\cdot\,)}_{\mathcal{F},\ast}. 
	\end{equation*}
	Since $\psi$ is $\contfunc^{1}$-convex, we obtain  
	\begin{align*}
		\abs{-\mathcal{E}(f) + \mathcal{E}\left(\frac{f + g}{2}\right)}
		&= p\abs{\psi(1/2) - \psi(0)} 
		\le \frac{p}{2}\bigl(\abs{\psi'(0)} \vee \abs{\psi'(1)}\bigr) \\
		&\le c_{p}\bigl(\mathcal{E}(f)^{1/p} \vee \mathcal{E}(g)^{1/p}\bigr)\norm{\mathcal{E}(f; \,\cdot\,) - \mathcal{E}(g; \,\cdot\,)}_{\mathcal{F},\ast}, 
	\end{align*}
	where we put $c_{p} \coloneqq \frac{p}{2}\bigl(1 + \frac{p}{p - 1}\bigr)$. 
	Similarly, 
	\[
	\abs{-\mathcal{E}(g) + \mathcal{E}\left(\frac{f + g}{2}\right)}
	= p\abs{\psi(1/2) - \psi(1)} 
	\le c_{p}\bigl(\mathcal{E}(f)^{1/p} \vee \mathcal{E}(g)^{1/p}\bigr)\norm{\mathcal{E}(f; \,\cdot\,) - \mathcal{E}(g; \,\cdot\,)}_{\mathcal{F},\ast}. 
	\]
	Therefore, it follows that 
	\begin{equation}\label{e:inverse.pre-sum}
		\mathcal{E}\left(\frac{f + g}{2}\right) 
		\ge \Bigl(\mathcal{E}(f) \vee \mathcal{E}(g) - c_{p}\bigl(\mathcal{E}(f)^{1/p} \vee \mathcal{E}(g)^{1/p}\bigr)\norm{\mathcal{E}(f; \,\cdot\,) - \mathcal{E}(g; \,\cdot\,)}_{\mathcal{F},\ast}\Bigr)^{+}. 
	\end{equation}
	Next we derive an estimate on $\mathcal{E}(\frac{f - g}{2})$ by using \ref{Cp} and \eqref{e:inverse.pre-sum}. 
	Set $a \coloneqq \mathcal{E}(f) \vee \mathcal{E}(g)$ for simplicity. 
	If $p \in [2,\infty)$, then  
	\begin{align*}
		\mathcal{E}\left(\frac{f - g}{2}\right)
		&\overset{\ref{Cp}}{\le} 2^{1 - p}\bigl(\mathcal{E}(f)^{1/(p - 1)} + \mathcal{E}(g)^{1/(p - 1)}\bigr)^{p - 1} - \mathcal{E}\left(\frac{f + g}{2}\right) \\
		&\overset{\eqref{e:inverse.pre-sum}}{\le} a - \Bigl(a - c_{p}a^{1/p}\norm{\mathcal{E}(f; \,\cdot\,) - \mathcal{E}(g; \,\cdot\,)}_{\mathcal{F},\ast}\Bigr)^{+} \\
		&\le c_{p}a^{1/p}\norm{\mathcal{E}(f; \,\cdot\,) - \mathcal{E}(g; \,\cdot\,)}_{\mathcal{F},\ast}. 
	\end{align*}
	In the rest of the proof, we assume that $p \in (1,2]$.
	We see that 
	\begin{align}\label{e:inverse.small-pre}
		\mathcal{E}\left(\frac{f - g}{2}\right)^{1/(p - 1)}
		&\overset{\ref{Cp}}{\le} \left(\frac{\mathcal{E}(f) + \mathcal{E}(g)}{2}\right)^{1/(p - 1)} - \mathcal{E}\left(\frac{f + g}{2}\right)^{1/(p - 1)} \nonumber \\
		&\overset{\eqref{e:inverse.pre-sum}}{\le} a^{1/(p - 1)} - \biggl[\Bigl(a - c_{p}a^{1/p}\norm{\mathcal{E}(f; \,\cdot\,) - \mathcal{E}(g; \,\cdot\,)}_{\mathcal{F},\ast}\Bigr)^{+}\biggr]^{1/(p - 1)}. 
	\end{align}
	In the case of $a \le c_{p}a^{1/p}\norm{\mathcal{E}(f; \,\cdot\,) - \mathcal{E}(g; \,\cdot\,)}_{\mathcal{F},\ast}$, we have 
	\[
	\mathcal{E}\left(\frac{f - g}{2}\right) 
	\le a 
	= a^{(2 - p) + (p - 1)} 
	\le c_{p}^{p-1}a^{2 - p + \frac{p-1}{p}}\norm{\mathcal{E}(f; \,\cdot\,) - \mathcal{E}(g; \,\cdot\,)}_{\mathcal{F},\ast}^{p-1}. 
	\]
	Let us consider the remaining case $a > c_{p}a^{1/p}\norm{\mathcal{E}(f; \,\cdot\,) - \mathcal{E}(g; \,\cdot\,)}_{\mathcal{F},\ast}$. 
	Then we have from \eqref{e:p-power} in the proof of Theorem \ref{thm.p-form} with $q = 1/(p-1)$ that 
	\begin{align*}
		\mathcal{E}\left(\frac{f - g}{2}\right)^{1/(p - 1)}  
		&= a^{1/(p - 1)} - \Bigl(a - c_{p}a^{1/p}\norm{\mathcal{E}(f; \,\cdot\,) - \mathcal{E}(g; \,\cdot\,)}_{\mathcal{F},\ast}\Bigr)^{1/(p - 1)} \\
		&\le \frac{c_{p}}{p - 1}a^{\frac{2 - p}{p - 1} + \frac{1}{p}}\norm{\mathcal{E}(f; \,\cdot\,) - \mathcal{E}(g; \,\cdot\,)}_{\mathcal{F},\ast}. 
	\end{align*}
	Hence we obtain the desired estimate \eqref{e:conti.inverse}. 
\end{proof}

The following proposition states a kind of monotonicity of the ``$p$-Laplacian'' in the pointwise values of its argument. 
This result will play important roles in Subsection \ref{sec.sharp} later and in the subsequent works \cite{KS.scp,KS.sing}.
\begin{prop}\label{prop.mono}
	Assume that $(\mathcal{E},\mathcal{F})$ satisfies \ref{Cp} and the strong subadditivity \eqref{sadd} in Proposition \ref{prop.GC-list}-\ref{GC.markov}. 
	Let $u_{1},u_{2},v \in \mathcal{F}$ satisfy $((u_{2} - u_{1}) \wedge v)(x) = 0$ for $m$-a.e.\ $x \in X$. 
	Then $\mathcal{E}(u_{1}; v) \ge \mathcal{E}(u_{2}; v)$. 
\end{prop}
\begin{proof}
    Let $t > 0$.
    Define $f,g \in \mathcal{F}$ by $f \coloneqq u_{1} + tv$ and $g \coloneqq u_{2}$.
    Then we easily see that $f \vee g = u_{2} + tv$ and $f \wedge g = u_{1}$.
    By \eqref{sadd}, we have $\mathcal{E}(u_{2} + tv) + \mathcal{E}(u_{1}) \le \mathcal{E}(u_{1} + tv) + \mathcal{E}(u_{2})$, which implies that
    \[
    \frac{\mathcal{E}(u_{2} + tv) - \mathcal{E}(u_{2})}{t} \le \frac{\mathcal{E}(u_{1} + tv) - \mathcal{E}(u_{1})}{t}.
    \]
    Letting $t \downarrow 0$, we get $\mathcal{E}(u_{2}; v) \le \mathcal{E}(u_{1}; v)$.
\end{proof}

We conclude this subsection by viewing typical examples of $p$-energy forms.
\begin{example}\label{ex.Rn}
	\begin{enumerate}[label=\textup{(\arabic*)},align=left,leftmargin=*,topsep=2pt,parsep=0pt,itemsep=2pt]
		\item\label{Ep-Rn} Let $D \in \mathbb{N}$, let $\Omega$ be an open subset of $\mathbb{R}^{D}$, let $\mathcal{B} \coloneqq \mathcal{B}(\Omega)$, let $m$ be the $D$-dimensional Lebesgue measure on $\Omega$ and let $\mathcal{F} = W^{1,p}(\Omega)$ be the usual $(1,p)$-Sobolev space on $\Omega$ (see, e.g., \cite[p.~60]{AF} for details on Sobolev spaces in the Euclidean setting).
        Define $\mathcal{E}(f) \coloneqq \norm{\nabla f}_{L^{p}(\Omega,m)}^{p}$, $f \in \mathcal{F}$, where the gradient operator $\nabla$ is to be interpreted in the sense of distributions.
		Then, by following a similar argument as in the proof of Theorem \ref{thm.GCp-kuwae}, one can show that $(\mathcal{E},\mathcal{F})$ is a $p$-energy form on $(\Omega,m)$ satisfying \ref{GC}.
		In this case, we have
		\[
		\mathcal{E}(f; g) = \int_{\Omega}\abs{\nabla f(x)}^{p - 2}\langle \nabla f(x), \nabla g(x) \rangle_{\mathbb{R}^{D}}\,dx, \quad f,g \in \mathcal{F},
		\] 
        where $\langle\,\cdot\,,\,\cdot\,\rangle_{\mathbb{R}^{D}}$ denotes the inner product on $\mathbb{R}^{D}$. 
		\item\label{Ep-KZ} In the recent work \cite{Kig23, MS+}, a $p$-energy form $(\mathcal{E},\mathcal{F})$ on a compact metrizable space is constructed via discrete approximations under some analytic and geometric assumptions. 
        See \cite{CGQ22,HPS04} for constructions of $p$-energy forms on post-critically finite self-similar sets.
        The construction in \cite{CGQ22} can be seen as a generalization of that in \cite{HPS04}.  
        The $p$-energy forms constructed in these works are regarded as analogues of the $p$-energy form on $\mathbb{R}^{D}$ as in \ref{Ep-Rn} above.
        As will be seen in more detail later in Section \ref{sec.constr}, we can prove that $p$-energy forms constructed in \cite{CGQ22,Kig23,MS+} satisfy \ref{GC}, whereas even \ref{Cp} is not mentioned in \cite{CGQ22,Kig23}. 
        Furthermore, very recently, Kuwae \cite{Kuw23+} introduced a $p$-energy form $(\mathscr{E}^{p},H^{1,p})$ based on a strongly local Dirichlet form $(\mathscr{E},D(\mathscr{E}))$ on $L^{2}(X,m)$. It is shown that $(\mathscr{E}^{p},H^{1,p})$ satisfies \ref{Cp} in \cite[Theorem 1.7]{Kuw23+}. We can also verify \ref{GC} for $(\mathscr{E}^{p},H^{1,p})$ by using some good estimates due to the bilinearity (Theorem \ref{thm.GCp-kuwae}). See Appendix \ref{sec.bilinear} for details. 
        \item\label{Ep-N} There are various ways to define $(1,p)$-Sobolev spaces in the field of analysis on metric spaces. See \cite[Chapter 10]{HKST} for a summary of several known approaches in the literature.
        In those definitions, roughly speaking, we find a counterpart of $\abs{\nabla u}$, e.g., the minimal $p$-weak upper gradient\index{minimal $p$-weak upper gradient} $g_{u} \ge 0$ (see, e.g., \cite[Chapter 6]{HKST} for details on this notion), and consider a $p$-energy form $(\widetilde{\mathcal{E}},\mathcal{F})$ on $(X,m)$ given by $\widetilde{\mathcal{E}}(u) \coloneqq \int_{X}g_{u}^{p}\,dm$ and $\mathcal{F} \coloneqq \{ u \in L^{p}(X,m) \mid g_{u} \in L^{p}(X,m) \}$.
        Unfortunately, this $p$-energy form may not satisfy \ref{Cp} because the map $u \mapsto g_{u}$ is not linear in general (see, e.g., \cite[(6.3.18)]{HKST} for this fact).
        However, in a suitable setting, we can construct a functional which is equivalent to $\widetilde{\mathcal{E}}$ and satisfies \ref{Cp}; see the $p$-energy form denoted by $(\mathcal{F}_{p},W^{1,p})$ in \cite[Theorem 40]{ACD15}.
        Moreover, we can  verify \ref{GC} for $(\mathcal{F}_{p},W^{1,p})$ since the pair $(\mathcal{F}_{\delta_{k},p},W^{1,p})$ defined in \cite[(7.3)]{ACD15} satisfies \ref{GC} and $\mathcal{F}_{p}$ is defined as a $\Gamma$-limit point of $\mathcal{F}_{\delta_{k},p}$ as $k \to \infty$. (We will use a similar argument later in the proof of Theorem \ref{t:Kig-good}.) 
	\end{enumerate}
\end{example}

\subsection{\texorpdfstring{$p$}{p}-Clarkson's inequality and approximations in \texorpdfstring{$p$}{p}-energy forms}
In this subsection, in addition to the setting specified at the beginning of this section, by considering $\mathcal{F} \cap L^{p}(X,m)$ instead of $\mathcal{F}$ if necessary, we also assume for simplicity that $\mathcal{F} \subseteq L^{p}(X,m)$.  

We introduce a family of natural norms on $\mathcal{F}$ in the following definition. 
\begin{defn}[$(\mathcal{E},\alpha)$-norm]\label{d:e1norm}
	Let $\alpha \in (0,\infty)$. 
	We define the norm $\norm{\,\cdot\,}_{\mathcal{E},\alpha}$ on $\mathcal{F}$ by
	\begin{equation}\label{e:defn-e1norm}
		\norm{f}_{\mathcal{E},\alpha} \coloneqq \Bigl(\mathcal{E}(f) + \alpha\norm{f}_{L^{p}(X,m)}^{p}\Bigr)^{1/p}, \quad f \in \mathcal{F}
	\end{equation} 
	We call $\norm{\,\cdot\,}_{\mathcal{E},\alpha}$ the \emph{$(\mathcal{E},\alpha)$-norm}\index{$(\mathcal{E},\alpha)$-norm} on $\mathcal{F}$. 
\end{defn}
Clearly, for any $\alpha,\alpha' \in (0,\infty)$, $\norm{\,\cdot\,}_{\mathcal{E},\alpha}$ and $\norm{\,\cdot\,}_{\mathcal{E},\alpha'}$ are equivalent to each other. 
 
The following proposition states on the convexity of $\norm{\,\cdot\,}_{\mathcal{E},\alpha}$. 
\begin{prop}\label{prop.E1-unifconvex}
	Let $\alpha \in (0,\infty)$ and assume that $(\mathcal{E},\mathcal{F})$ satisfies \ref{Cp}. 
	Then $(\norm{\,\cdot\,}_{\mathcal{E},\alpha}^{p},\mathcal{F})$ is a $p$-energy form on $(X,m)$ satisfying \ref{Cp}, and $(\mathcal{F},\norm{\,\cdot\,}_{\mathcal{E},\alpha})$ is uniformly convex.
	Moreover, if $(\mathcal{F},\norm{\,\cdot\,}_{\mathcal{E},\alpha})$ is a Banach space in addition, then it is reflexive. 
\end{prop}
\begin{proof}
	We have \ref{Cp} for the $p$-energy form $(\norm{\,\cdot\,}_{\mathcal{E},\alpha}^{p},\mathcal{F})$ on $(X,m)$ by applying \eqref{GC.sum} to $T \colon \mathbb{R}^{2} \to \mathbb{R}$ given in Proposition \ref{prop.GC-list}-\ref{GC.Cpsmall},\ref{GC.Cplarge}. 
	The uniform convexity of $\norm{\,\cdot\,}_{\mathcal{E},\alpha}$ follows from \cite[Theorem 1]{Cla36}, where a certain product of uniformly convex spaces is shown to be uniformly convex. (Alternatively, this follows also from Proposition \ref{prop.cone-gen}-\ref{GC.cone}.)
	
	Assume that $(\mathcal{F},\norm{\,\cdot\,}_{\mathcal{E},\alpha})$ is a Banach space. 
	Then, being a uniformly convex Banach space, $(\mathcal{F},\norm{\,\cdot\,}_{\mathcal{E},\alpha})$ is reflexive by the Milman--Pettis theorem\index{Milman--Pettis theorem} (see, e.g., \cite[Theorem 2 in Section V.2]{Yos} for a proof of it). 
\end{proof}

We will frequently use the following Mazur's lemma\index{Mazur's lemma}, which is an elementary fact in the theory of Banach spaces. 
\begin{lem}[Mazur's lemma; see, e.g., {\cite[Theorem 2 in Section V.1]{Yos}}]\label{lem.mazur} 
	Let $(v_{n})_{n \in \mathbb{N}}$ be a sequence in a normed space $V$ converging weakly to some element $v \in V$.
	Then there exist $N_{k} \in \mathbb{N}$ with $N_{k} \geq k$ and $\{ \lambda_{k,l} \}_{k \leq l \leq N_{k}} \subseteq [0,1]$ with $\sum_{l = k}^{N_{k}}\lambda_{k,l} = 1$ for each $k \in \mathbb{N}$ such that $\lim_{k \to \infty} \sum_{l = k}^{N_{k}} \lambda_{k, l}v_{l} = v$ in norm in $V$.
\end{lem}

We also prepare the following two lemmas.
\begin{lem}\label{lem.Lpdense} 
	Assume that $(\mathcal{E},\mathcal{F})$ satisfies \ref{Cp} and that $\mathcal{F}$ equipped with $\norm{\,\cdot\,}_{\mathcal{E},1}$ is a Banach space. 
	For $v \in L^{\frac{p}{p - 1}}(X,m)$, we define a bounded linear map $\Psi_{v} \colon L^{p}(X,\measure) \to \mathbb{R}$ by $\Psi_{v}(u) \coloneqq \int_{X}uv\,d\measure$. 
	Then $\{ \Psi_{v}|_{\mathcal{F}} \mid v \in L^{\frac{p}{p - 1}}(X,m) \}$ is dense in $\mathcal{F}^{\ast}$, and the map $L^{\frac{p}{p - 1}}(X,m) \ni v \mapsto \Psi_{v}|_{\mathcal{F}} \in \mathcal{F}^{\ast}$ is a bounded linear map with operator norm at most $1$. 
\end{lem}
\begin{proof}
	Set $M \coloneqq \{ \Psi_{v}|_{\mathcal{F}} \mid v \in L^{\frac{p}{p - 1}}(X,\measure) \}$. 
	Then $M \subseteq \mathcal{F}^{\ast}$ since $\norm{u}_{L^{p}(X,m)} \le \norm{u}_{\mathcal{E},1}$ for any $u \in \mathcal{F}$. 
	Suppose that $\closure{M}^{\mathcal{F}^{\ast}} \neq \mathcal{F}^{\ast}$. 
	Let $\varphi \in \mathcal{F}^{\ast} \setminus \closure{M}^{\mathcal{F}^{\ast}}$. 
	By the Hahn--Banach theorem, there exists $\Phi \in \mathcal{F}^{\ast\ast}$ such that $\Phi(\varphi) \neq 0$ and $\Phi|_{\closure{M}^{\mathcal{F}^{\ast}}} = 0$. 
	Since $\mathcal{F}$ is reflexive by Proposition \ref{prop.E1-unifconvex}, there exists $u \in \mathcal{F}$ such that $\Phi(\psi) = \psi(u)$ for any $\psi \in \mathcal{F}^{\ast}$. 
	Then for any $\psi \in M$ we have $\psi(u) = \Phi(\psi) = 0$, which implies that $u = 0$. 
	This contradicts $\varphi(u) = \Phi(\varphi) \neq 0$ and hence we obtain $\closure{M}^{\mathcal{F}^{\ast}} = \mathcal{F}^{\ast}$. 
	The map $L^{\frac{p}{p - 1}}(X,m) \ni v \mapsto \Psi_{v}|_{\mathcal{F}} \in \mathcal{F}^{\ast}$ is obviously linear, and is easily seen to have operator norm at most $1$ by H\"{o}lder's inequality and the fact that $\norm{u}_{L^{p}(X,m)} \le \norm{u}_{\mathcal{E},1}$ for any $u \in \mathcal{F}$.  
\end{proof}

\begin{cor}\label{cor:separable}
	Assume that $(\mathcal{E},\mathcal{F})$ satisfies \ref{Cp} and that $\mathcal{F}$ equipped with $\norm{\,\cdot\,}_{\mathcal{E},1}$ is a Banach space. 
	If $L^{p}(X,m)$ is separable, then $\mathcal{F}$ and $\mathcal{F}^{\ast}$ are separable. 
\end{cor}
\begin{proof}
	We first recall a standard fact that a Banach space is separable if its dual is separable (see, e.g., \cite[Lemma in Section V.2]{Yos} for a proof).
	Since $L^{\frac{p}{p-1}}(X,m)$ is separable by the separability of $L^{\frac{p}{p-1}}(X,m)^{\ast} = L^{p}(X,m)$, it follows from Lemma \ref{lem.Lpdense} that $\mathcal{F}^{\ast}$ is separable, which in turn implies that $\mathcal{F}$ is separable.  
\end{proof}

\begin{lem}\label{lem.Lpreduce} 
	Assume that $(\mathcal{E},\mathcal{F})$ satisfies \ref{Cp} and that $\mathcal{F}$ equipped with $\norm{\,\cdot\,}_{\mathcal{E},1}$ is a Banach space. 
	If $\{ u_{n} \}_{n \in \mathbb{N}} \subseteq \mathcal{F}$ converges in norm in $L^{p}(X,m)$ to $u \in L^{p}(X,m)$ and $\sup_{n \in \mathbb{N}}\mathcal{E}(u_{n}) < \infty$, then $u \in \mathcal{F}$ and $\{ u_{n} \}_{n \in \mathbb{N}}$ converges weakly in $(\mathcal{F},\norm{\,\cdot\,}_{\mathcal{E},1})$ to $u$.  
\end{lem}
\begin{proof}
	Since $\mathcal{F}$ is reflexive and $\sup_{n \in \mathbb{N}}\norm{u_{n}}_{\mathcal{E},1} < \infty$, some subsequence of $\{ u_{n} \}_{n \in \mathbb{N}}$ converges weakly in $(\mathcal{F},\norm{\,\cdot\,}_{\mathcal{E},1})$ to some $f \in \mathcal{F}$ by the Banach--Alaoglu theorem (see, e.g., \cite[Theorem 1 in Section V.2]{Yos}) and hence weakly in $L^{p}(X,m)$ to both $u$ and $f$ by the continuity of the inclusion map of $\mathcal{F}$ into $L^{p}(X,m)$, and thus $u = f \in \mathcal{F}$. 
	For any $\varphi \in \mathcal{F}^{\ast}$ and any $\varepsilon > 0$, by Lemma \ref{lem.Lpdense}, there exists $v \in L^{\frac{p}{p - 1}}(X,m)$ such that $\norm{\varphi - \Psi_{v}|_{\mathcal{F}}}_{\mathcal{F}^{\ast}} < \varepsilon$. 
	Then we easily see that 
	\begin{align*}
		\abs{\varphi(u) - \varphi(u_{n})}
		&\le \abs{\varphi(u) - \Psi_{v}(u)} + \abs{\Psi_{v}(u) - \Psi_{v}(u_{n})} + \abs{\varphi(u_{n}) - \Psi_{v}(u_{n})} \\
		&\le \varepsilon\biggl(\norm{u}_{\mathcal{E},1} + \sup_{n \in \mathbb{N}}\norm{u_{n}}_{\mathcal{E},1}\biggr) + \abs{\Psi_{v}(u) - \Psi_{v}(u_{n})}, 
	\end{align*}
	whence $\limsup_{n \to \infty}\abs{\varphi(u) - \varphi(u_{n})} \le \varepsilon\bigl(\norm{u}_{\mathcal{E},1} + \sup_{n \in \mathbb{N}}\norm{u_{n}}_{\mathcal{E},1}\bigr)$. 
	Since $\varepsilon > 0$ is arbitrary, we obtain $\lim_{n \to \infty}\varphi(u_{n}) = \varphi(u)$. 
	This completes the proof. 
\end{proof}

We collect some useful results on convergence in $\mathcal{E}$ in the following proposition.  
\begin{prop}\label{prop.useful-E1}
	Assume that $(\mathcal{E},\mathcal{F})$ satisfies \ref{Cp} and that $(\mathcal{F},\norm{\,\cdot\,}_{\mathcal{E},1})$ is a Banach space.	
	\begin{enumerate}[label=\textup{(\alph*)},align=left,leftmargin=*,topsep=2pt,parsep=0pt,itemsep=2pt]
		\item\label{it:E1-lsc} If $\{ u_{n} \}_{n \in \mathbb{N}} \subseteq L^{p}(X,m)$ converges in norm in $L^{p}(X,m)$ to $u \in L^{p}(X,m)$, then $\mathcal{E}(u) \le \liminf_{n \to \infty}\mathcal{E}(u_{n})$, where we set $\mathcal{E}(f) \coloneqq \infty$ for $f \in L^{p}(X,m) \setminus \mathcal{F}$.  
		\item\label{it:E1-conv} If $\{ u_{n} \}_{n \in \mathbb{N}} \subseteq \mathcal{F}$ converges in norm in $L^{p}(X,m)$ to $u \in \mathcal{F}$ and $\limsup_{n \to \infty}\mathcal{E}(u_{n}) \leq \mathcal{E}(u)$, then $\lim_{n \to \infty}\norm{u - u_{n}}_{\mathcal{E},1} = 0$. 
	\end{enumerate}
\end{prop}
\begin{proof}
	\ref{it:E1-lsc}: 
	If $\liminf_{n \to \infty}\mathcal{E}(u_{n}) = \infty$, then the desired statement clearly holds.
    So, we assume that $\liminf_{n \to \infty}\mathcal{E}(u_{n}) < \infty$.
    Pick a subsequence $\{ u_{n_{k}} \}_{k \in \mathbb{N}}$ such that $\lim_{k \to \infty}\mathcal{E}(u_{n_{k}}) = \liminf_{n \to \infty}\mathcal{E}(u_{n})$.
    Then $\{ u_{n_{k}} \}_{k \in \mathbb{N}}$ is a bounded sequence in $(\mathcal{F},\norm{\,\cdot\,}_{\mathcal{E},1})$ converging in norm in $L^{p}(X,m)$ to $u$ and hence Lemma \ref{lem.Lpreduce} implies that $u \in \mathcal{F}$ and that $\{ u_{n_{k}} \}_{k \in \mathbb{N}}$ converges weakly in $\mathcal{F}$ to $u$. 
    Since $\norm{\,\cdot\,}_{\mathcal{E},1}$ is lower semicontinuous with respect to the weak topology of $\mathcal{F}$, we have from $\lim_{k \to \infty}\norm{u_{n_{k}}}_{L^{p}(X,m)} = \norm{u}_{L^{p}(X,m)}$ that $\mathcal{E}(u)^{1/p} \le \liminf_{n \to \infty}\mathcal{E}(u_{n})^{1/p}$. 
	
	\ref{it:E1-conv}: 
	Note that $\lim_{n \to \infty}\mathcal{E}(u_{n}) = \mathcal{E}(u)$ by \ref{it:E1-lsc} of the present proposition.
    If $u \in \mathcal{E}^{-1}(0)$, then $\mathcal{E}(u - u_{n}) = \mathcal{E}(u_{n}) \to \mathcal{E}(u) = 0$.
	It suffices to consider the case of $\mathcal{E}(u) = 1$. 
    Since $u + u_{n}$ converges in $L^p(X,m)$ to $2u$ as $n \to \infty$, by \ref{it:E1-lsc} of the present proposition we have
    \begin{align*}
    2 = \mathcal{E}(2u)^{1/p}
    \le \liminf_{n \to \infty}\mathcal{E}\bigl(u + u_{n}\bigr)^{1/p}
    &\le \limsup_{n \to \infty}\mathcal{E}\bigl(u + u_{n}\bigr)^{1/p} \\
    &\le \lim_{n \to \infty}\mathcal{E}(u_{n})^{1/p} + \mathcal{E}(u)^{1/p} = 2,
    \end{align*}
    i.e., $\lim_{n \to \infty}\mathcal{E}(u + u_{n}) = 2^{p}$.
    By \ref{Cp}, if $p \le 2$, then
    \begin{align*}
    	\lim_{n \to \infty}\mathcal{E}(u - u_{n})^{1/(p - 1)}
    	&\le 2\Bigl(\mathcal{E}(u) + \lim_{n \to \infty}\mathcal{E}(u_{n})\Bigr)^{1/(p - 1)} - \lim_{n \to \infty}\mathcal{E}(u + u_{n})^{1/(p - 1)} \\
    	&= 2 \cdot 2^{1/(p - 1)} - 2^{p/(p - 1)} = 0. 
    \end{align*}
    If $p \ge 2$, then
    \begin{align*}
    	\lim_{n \to \infty}\mathcal{E}(u - u_{n})
    	\le 2^{p - 1}\Bigl(\mathcal{E}(u) + \lim_{n \to \infty}\mathcal{E}(u_{n})\Bigr) - \lim_{n \to \infty}\mathcal{E}(u + u_{n}) 
    	= 2^{p - 1} \cdot 2 - 2^{p} = 0. 
    \end{align*}
    Since $\{ u_{n} \}_{n \in \mathbb{N}}$ converges in norm in $L^{p}(X,m)$ to $u$, we obtain the desired convergence. 
\end{proof}

The following convergences in $\mathcal{E}$ are also useful.
These are analogues of \cite[Theorem 1.4.2-(iii),(iv),(v)]{FOT}. 
\begin{cor}\label{cor.approx-measure}
	Assume that $(\mathcal{E},\mathcal{F})$ satisfies \eqref{lipcont} in Proposition \ref{prop.GC-list}-\ref{GC.lip} and \ref{Cp} and that $(\mathcal{F},\norm{\,\cdot\,}_{\mathcal{E},1})$ is a Banach space.
    \begin{enumerate}[label=\textup{(\alph*)},align=left,leftmargin=*,topsep=2pt,parsep=0pt,itemsep=2pt]
      \item\label{it:approx-meas.1} Let $\{ \varphi_{n} \}_{n \in \mathbb{N}} \subseteq \contfunc(\mathbb{R})$ satisfy $\lim_{n' \to \infty}\varphi_{n'}(t) = t$, $\varphi_{n}(0) = 0$ and $\abs{\varphi_{n}(t) - \varphi_{n}(s)} \le \abs{t - s}$ for any $n \in \mathbb{N}$ and any $s,t \in \mathbb{R}$. Then $\{ \varphi_{n}(u) \}_{n \in \mathbb{N}} \subseteq \mathcal{F}$ and $\lim_{n \to \infty}\mathcal{E}(u - \varphi_{n}(u)) = 0$ for any $u \in \mathcal{F}$.
      \item\label{it:approx-meas.2} Let $u \in \mathcal{F}$, $\{ u_{n} \}_{n \in \mathbb{N}} \subseteq \mathcal{F}$ and $\varphi \in \contfunc(\mathbb{R})$ satisfy $\lim_{n \to \infty}\norm{u - u_{n}}_{\mathcal{E},1} = 0$, $\varphi(0) = 0$, $\abs{\varphi(t) - \varphi(s)} \le \abs{t - s}$ for any $s,t \in \mathbb{R}$, and $\varphi(u) = u$.
        Then $\{ \varphi(u_{n}) \}_{n \in \mathbb{N}} \subseteq \mathcal{F}$ and $\lim_{n \to \infty}\mathcal{E}(u - \varphi(u_{n})) = 0$.
    \end{enumerate}
\end{cor}
\begin{rmk}\label{rmk:meas-form.approx}
	Let us make the same remark as \cite[Remark 2.21]{KS.survey} for the reader's convenience. 
	Typical choices of $\{\varphi_{n}\}_{n\in\mathbb{N}}\subset\contfunc(\mathbb{R})$ in Corollary \ref{cor.approx-measure}-\ref{it:approx-meas.1} are $\varphi_{n}(t)=(-n)\vee(t\wedge n)$ and $\varphi_{n}(t)=t-(-\frac{1}{n})\vee(t\wedge\frac{1}{n})$.
	A typical use of Corollary \ref{cor.approx-measure}-\ref{it:approx-meas.2} is to obtain a sequence of $I$-valued functions converging to $u$ in $(\mathcal{F},\norm{\,\cdot\,}_{\mathcal{E},1})$ when $I\subset\mathbb{R}$ is a closed interval with $0 \in I$ and $u\in\mathcal{F}$ is $I$-valued, by considering $\varphi\in\contfunc(\mathbb{R})$ given by $\varphi(t)\coloneqq(\inf I)\vee(t\wedge\sup I)$.
\end{rmk}
\begin{proof}[Proof of Corollary \ref{cor.approx-measure}]
    \ref{it:approx-meas.1}: 
    It is immediate from the dominated convergence theorem that $\{ \varphi_{n}(u) \}_{n \in \mathbb{N}}$ converges in norm in $L^{p}(X,m)$ to $u$.
    Moreover, by \eqref{lipcont} we have $\varphi_{n}(u) \in \mathcal{F}$ and $\mathcal{E}(\varphi_{n}(u)) \leq \mathcal{E}(u)$ for any $n \in \mathbb{N}$, and in particular $\limsup_{n\to\infty}\mathcal{E}(\varphi_{n}(u)) \leq \mathcal{E}(u)$.
    Thus $\lim_{n \to \infty}\mathcal{E}(u - \varphi_{n}(u)) = 0$ by Proposition \ref{prop.useful-E1}-\ref{it:E1-conv}. 

    \ref{it:approx-meas.2}:  
	By \eqref{lipcont} we have $\varphi(u_{n}) \in \mathcal{F}$ and $\mathcal{E}(\varphi(u_{n})) \leq \mathcal{E}(u_{n})$ for any $n \in \mathbb{N}$, and therefore $\limsup_{n\to\infty}\mathcal{E}(\varphi(u_{n})) \leq \lim_{n\to\infty}\mathcal{E}(u_{n}) = \mathcal{E}(u)$ by the triangle inequality for $\mathcal{E}^{1/p}$.
    Also, $\{ \varphi(u_{n}) \}_{n \in \mathbb{N}}$ converges in norm in $L^{p}(X,m)$ to $\varphi(u) = u$ since $\abs{\varphi(u) - \varphi(u_{n})} \leq \abs{u - u_{n}}$ on $X$. 
    Thus $\lim_{n \to \infty}\mathcal{E}(u - \varphi(u_{n})) = 0$ by Proposition \ref{prop.useful-E1}-\ref{it:E1-conv}. 
\end{proof}

\subsection{Fr\'{e}chet derivative as a homeomorphism to the dual space}
In many practical situations, the quotient normed space $\mathcal{F}/\mathcal{E}^{-1}(0)$ (equipped with the norm $\mathcal{E}^{1/p}$) becomes a Banach space (see Subsection \ref{sec.trace}).
To state some basic properties of this Banach space, we recall the notion of \emph{uniformly smoothness}.
\begin{defn}[Uniformly smooth normed space]
    Let $(\mathcal{X}, \norm{\,\cdot\,})$ be a normed space.
    The normed space $\mathcal{X}$ is said to be \emph{uniformly smooth}\index{uniformly smooth} if and only if 
    \[
    \lim_{\tau \to 0}\tau^{-1}\sup\biggl\{ \frac{\norm{u + v} + \norm{u - v}}{2} - 1 \biggm| \norm{u} = 1, \norm{v} = \tau \biggr\} = 0.
    \]
\end{defn}

The following duality between uniform convexity and uniform smoothness is well known.
(See also \cite[Lemma 5]{BCL94} for a quantitative version of this theorem.)
\begin{thm}[Day's duality theorem\index{Day's duality theorem}; see, e.g., {\cite[Proposition 1.e.2]{LT}}]\label{thm.Day}
	Let $\mathcal{X}$ be a Banach space. 
	Then $\mathcal{X}$ is uniformly convex if and only if its dual space $\mathcal{X}^{\ast}$ is uniformly smooth. 
\end{thm}

We also recall the notion of duality mapping and fundamental results on it in the following proposition (see, e.g., \cite[Definition 2.1, Lemmas 2.1 and 2.2]{Miya} for details and proofs). 
\begin{prop}[Duality mapping]\label{prop.dualitymap}
	Let $\mathcal{X}$ be a Banach space and let $\mathcal{X}^{\ast}$ be the dual space of $\mathcal{X}$. 
	Let $\norm{\,\cdot\,}_{W}$ be the norm of $W$ for each $W \in \{ \mathcal{X},\mathcal{X}^{\ast} \}$. 
	For $(x,f) \in \mathcal{X} \times \mathcal{X}^{\ast}$, we set $\langle x,f \rangle \coloneqq f(x)$. 
	For $x \in \mathcal{X}$, define $F \colon \mathcal{X} \to 2^{\mathcal{X}^{\ast}}$ by 
	\[
	F(x) \coloneqq \bigl\{ f \in \mathcal{X}^{\ast} \bigm| \langle x,f \rangle = \norm{x}_{\mathcal{X}}^{2} = \norm{f}_{\mathcal{X}^{\ast}}^{2} \bigr\}, 
	\]
	which is called the \emph{duality mapping}\index{duality mapping} of $\mathcal{X}$. 
	Then the following properties hold: 
	\begin{enumerate}[label=\textup{(\alph*)},align=left,leftmargin=*,topsep=2pt,parsep=0pt,itemsep=2pt]
		\item \label{dm.nonempty} $F(x) \neq \emptyset$ for any $x \in \mathcal{X}$. 
		\item \label{dm.sur} If $\mathcal{X}$ is reflexive, then $\bigcup_{x \in \mathcal{X}}F(x) = \mathcal{X}^{\ast}$. 
		\item \label{dm.inj} If $\mathcal{X}$ is \emph{strictly convex}\index{strictly convex}, i.e., $\norm{\lambda x + (1 - \lambda)y}_{\mathcal{X}} < 1$ for any $\lambda \in (0,1)$ and any $x,y \in \mathcal{X}$ satisfying $\norm{x}_{\mathcal{X}}=\norm{y}_{\mathcal{X}}=1$ and $x \not= y$, then $\#(F(x)) = 1$ for any $x \in \mathcal{X}$. 
	\end{enumerate}
\end{prop}

Now we can state a result on the dual space of $\mathcal{F}/\mathcal{E}^{-1}(0)$.  
\begin{thm}\label{thm.banach}
    Assume that $(\mathcal{E},\mathcal{F})$ satisfies \ref{Cp} and that $\mathcal{F}/\mathcal{E}^{-1}(0)$ is a Banach space.
    \begin{enumerate}[label=\textup{(\alph*)},align=left,leftmargin=*,topsep=2pt,parsep=0pt,itemsep=2pt]
        \item \label{conv-smooth} The Banach space $\mathcal{F}/\mathcal{E}^{-1}(0)$ is  uniformly convex and uniformly smooth. In particular, it is reflexive and its dual Banach space $\bigl(\mathcal{F}/\mathcal{E}^{-1}(0)\bigr)^{\ast}$ is also uniformly convex and uniformly smooth.
        \item \label{dual} The map $f \mapsto \mathcal{E}(f;\cdot)$ is a homeomorphism from $\mathcal{F}/\mathcal{E}^{-1}(0)$ to $\bigl(\mathcal{F}/\mathcal{E}^{-1}(0)\bigr)^{\ast}$. In particular, $\bigl(\mathcal{F}/\mathcal{E}^{-1}(0)\bigr)^{\ast} = \{ \mathcal{E}(f;\,\cdot\,) \mid f \in \mathcal{F} \}$.
    \end{enumerate}
\end{thm}
\begin{proof}
    For ease of notation, set $\mathcal{X} \coloneqq \mathcal{F}/\mathcal{E}^{-1}(0)$ and $\norm{u}_{\mathcal{X}} \coloneqq \mathcal{E}(u)^{1/p}$ for any $u \in \mathcal{X}$. 

    \ref{conv-smooth}:
    The uniform convexity of $\mathcal{X}$ is immediate from Proposition \ref{p:uc}, whence $\mathcal{X}$ is reflexive by the Milman--Pettis theorem\index{Milman--Pettis theorem} (see, e.g., \cite[Theorem 2 in Section V.2]{Yos}).
    Also, we easily see from \eqref{ncont.1} in the proof of Theorem \ref{thm.p-form} that $\mathcal{X}$ is uniformly smooth. 
    The same properties for $\mathcal{X}^{\ast}$ follow from Theorem \ref{thm.Day}.

    \ref{dual}:
    Let $u \in \mathcal{X}$ and define $\mathcal{A}(u) \coloneqq \mathcal{E}(u)^{\frac{2}{p} - 1}\mathcal{E}(u; \,\cdot\,) \in \mathcal{X}^{\ast}$. (We define $\mathcal{A}(u) = 0$ if $\mathcal{E}(u) = 0$.)
    We will show that $\mathcal{A} \colon \mathcal{X} \to \mathcal{X}^{\ast}$ is a bijection.
    By the H\"{o}lder-type estimate \eqref{bdd.form} in Theorem \ref{thm.p-form}, we have
    \[
    \norm{\mathcal{A}(u)}_{\mathcal{X}^{\ast}} = \mathcal{E}(u)^{\frac{2}{p} - 1}\norm{\mathcal{E}(u; \,\cdot\,)}_{\mathcal{X}^{\ast}} = \mathcal{E}(u)^{\frac{2}{p} - 1 + \frac{p-1}{p}} = \norm{u}_{\mathcal{X}}.
    \]
    Then $\langle u, \mathcal{A}(u) \rangle = \mathcal{E}(u)^{\frac{2}{p}} = \norm{u}_{\mathcal{X}}^{2} = \norm{\mathcal{A}(u)}_{\mathcal{X}^{\ast}}^{2}$ and hence
    \[
    \mathcal{A}(u) \in \{ f \in \mathcal{X}^{\ast} \mid \langle u, f \rangle = \norm{u}_{\mathcal{X}}^{2} = \norm{f}_{\mathcal{X}^{\ast}}^{2} \} = F(u),
    \]
    where $F \colon \mathcal{X} \to \mathcal{X}^{\ast}$ is the duality mapping. 
    We see from Proposition \ref{prop.dualitymap} and \ref{conv-smooth} of the present theorem that $\mathcal{A} \colon \mathcal{X} \to \mathcal{X}^{\ast}$ is a surjection.
    Note that the mapping $F^{-1} \colon \mathcal{X}^{\ast} \to \mathcal{X}^{\ast\ast} = \mathcal{X}$ defined by $F^{-1}(f) = \{ u \in \mathcal{X} \mid \langle u, f \rangle = \norm{u}_{\mathcal{X}}^{2} = \norm{f}_{\mathcal{X}^{\ast}}^{2} \}$ for $f \in \mathcal{X}^{\ast}$ is the duality mapping from $\mathcal{X}^{\ast}$ to $\mathcal{X}$.
    By Proposition \ref{prop.dualitymap} and \ref{conv-smooth} again, we conclude that $\mathcal{A}$ is injective. 
    The map $f \mapsto \mathcal{E}(f;\cdot)$ and its inverse are continuous by \eqref{ncont} in Theorem \ref{thm.p-form} and \eqref{e:conti.inverse} in Theorem \ref{thm.conti-inverse}, respectively. 
\end{proof}

We also present a similar statement for $(\mathcal{F},\norm{\,\cdot\,}_{\mathcal{E},\alpha})$.  
\begin{cor}\label{cor.banach-E1}
    Let $\alpha \in (0,\infty)$. 
    Assume that $\mathcal{F} \subseteq L^{p}(X,m)$, that $(\mathcal{E},\mathcal{F})$ satisfies \ref{Cp} and that $\mathcal{X}_{\alpha} \coloneqq (\mathcal{F},\norm{\,\cdot\,}_{\mathcal{E},\alpha})$ is a Banach space.
    \begin{enumerate}[label=\textup{(\alph*)},align=left,leftmargin=*,topsep=2pt,parsep=0pt,itemsep=2pt]
        \item \label{conv-smooth-E1} The Banach space $\mathcal{X}_{\alpha}$ is  uniformly convex and uniformly smooth. In particular, it is reflexive and its dual space $\mathcal{X}_{\alpha}^{\ast}$ is also uniformly convex and uniformly smooth.
        \item \label{dual-E1} For each $f \in \mathcal{F}$, define a linear map $\Psi_{p,\alpha}^{f} \colon \mathcal{F} \to \mathbb{R}$ by  
        \begin{equation}
        	\Psi_{p,\alpha}^{f}(g) \coloneqq \mathcal{E}(f; g) + \alpha\int_{X}\sgn(f)\abs{f}^{p - 1}g\,dm, \quad g \in \mathcal{F}. 
        \end{equation}
		Then the map $f \mapsto \Psi_{p,\alpha}^{f}$ is a homeomorphism from $\mathcal{X}_{\alpha}$ to $\mathcal{X}_{\alpha}^{\ast}$. In particular, $\mathcal{X}_{\alpha}^{\ast} = \{ \Psi_{p,\alpha}^{f} \mid f \in \mathcal{F} \}$.
    \end{enumerate}
\end{cor}
\begin{proof}
	Recalling the definition \eqref{d:e1norm} of $\norm{\,\cdot\,}_{\mathcal{E},\alpha}$ in Definition \ref{e:defn-e1norm}, we define $\mathcal{E}_{\alpha} \colon \mathcal{F} \to [0,\infty)$ by $\mathcal{E}_{\alpha}(u) \coloneqq \norm{u}_{\mathcal{E},\alpha}^{p}$.
	Then $(\mathcal{E}_{\alpha},\mathcal{F})$ is a $p$-energy form on $(X,m)$ satisfying \ref{Cp} by Proposition \ref{prop.E1-unifconvex}, and by the differentiability \eqref{exist-deriva} of $\mathcal{E}$ in Theorem \ref{thm.p-form} and the dominated convergence theorem we have
	\[
	\mathcal{E}_{\alpha}(f; g) \coloneqq \frac{1}{p}\left.\frac{d}{dt}\mathcal{E}_{\alpha}(f + tg)\right|_{t = 0} = \Psi_{p,\alpha}^{f}(g) \quad \textrm{for any $f,g \in \mathcal{F}$.}
	\]
	We therefore obtain the desired result by applying Theorem \ref{thm.banach} to $(\mathcal{E}_{\alpha},\mathcal{F})$. 
\end{proof}

\subsection{Regularity and strong locality}\label{sec.regularlocal}
In this subsection, in addition to the setting specified at the beginning of this section, we make the same topological assumptions as \cite[(1.1.7)]{FOT}, i.e.,
\begin{gather}
	\text{$X$ is a locally compact separable metrizable topological space,} \label{a:loccpt} \\
	\text{$m$ is a (positive) Radon measure on $X$ with $\supp_{X}[m] = X$} \label{a:fullRadon}
\end{gather}
(it is implicit in \eqref{a:fullRadon} that the $\sigma$-algebra $\mathcal{B}$ which $X$ is equipped with is assumed to be the Borel $\sigma$-algebra $\mathcal{B}(X)$ of $X$).   
Here, as usual, by a \emph{(positive) Radon measure}\index{Radon measure} on $X$ we mean a Borel measure on $X$ which is finite on any compact subset of $X$. 
Under this setting, the map from $\contfunc(X)$ to $L^{0}(X,m) = L^{0}(X,\mathcal{B}(X),m)$ defined by taking $u \in \contfunc(X)$ to its $m$-equivalence class is injective and hence gives a canonical embedding of $\contfunc(X)$ into $L^{0}(X,m)$ as a subalgebra, and we will consider $\contfunc(X)$ as a subset of $L^{0}(X,m)$ through this embedding without further notice.  

The following definitions are analogues of the notions in the theory of regular symmetric Dirichlet forms (see, e.g., \cite[p.~6]{FOT}). 
\begin{defn}[Core]\label{d:core}
	\begin{enumerate}[label=\textup{(\arabic*)},align=left,leftmargin=*,topsep=2pt,parsep=0pt,itemsep=2pt]
		\item\label{it:d-core} A subset $\mathscr{C}$ of $\mathcal{F} \cap \contfunc_{c}(X)$ is said to be a \emph{core}\index{core} of $(\mathcal{E},\mathcal{F})$ if and only if $\mathscr{C}$ is dense both in $(\mathcal{F},\norm{\,\cdot\,}_{\mathcal{E},1})$ and in $(\contfunc_{c}(X),\norm{\,\cdot\,}_{\sup})$. 
		\item\label{it:d-special} A core $\mathscr{C}$ of $(\mathcal{E},\mathcal{F})$ is said to be \emph{special}\index{special (core)} if and only if $\mathscr{C}$ is a linear subspace of $\mathcal{F} \cap \contfunc_{c}(X)$, $\mathscr{C}$ is a dense subalgebra of $(\contfunc_{c}(X),\norm{\,\cdot\,}_{\sup})$, and for any compact subset $K$ of $X$ and any relatively compact open subset $G$ of $X$ with $K \subseteq G$, there exists $\varphi \in \mathscr{C}$ such that $\varphi \ge 0$ on $X$, $\varphi = 1$ on $K$ and $\varphi = 0$ on $X \setminus G$. 
	\end{enumerate}
\end{defn}

\begin{defn}[Regularity]\label{d:regularity.p-form}
	We say that $(\mathcal{E},\mathcal{F})$ is \emph{regular}\index{regular ($p$-energy form)} if and only if there exists a core $\mathscr{C}$ of $(\mathcal{E},\mathcal{F})$. 
\end{defn}

We can show the following result on regular $p$-energy forms, which is an analogue of \cite[Exercise 1.4.1]{FOT}.
\begin{prop}\label{prop.canonicalcore}
	Assume that $(\mathcal{E},\mathcal{F})$ is regular and that $\mathcal{F}$ has the following properties: 
	\begin{equation}\label{e:dom-unitcontractive}
		u^{+} \wedge 1 \in \mathcal{F} \quad \text{for any $u \in \mathcal{F}$,}
	\end{equation}
	\begin{equation}\label{e:dom-bddleibniz}
		uv \in \mathcal{F} \quad \text{for any $u, v \in \mathcal{F} \cap \contfunc_{b}(X)$.}
	\end{equation}
	Then $\mathcal{F} \cap \contfunc_{c}(X)$ is a special core of $(\mathcal{E},\mathcal{F})$. 
\end{prop}
\begin{proof}
	It is clear that $\mathcal{F} \cap \contfunc_{c}(X)$ is a core of $(\mathcal{E},\mathcal{F})$.  
	By \eqref{e:dom-bddleibniz}, $\mathcal{F} \cap \contfunc_{c}(X)$ is a subalgebra of $\contfunc_{c}(X)$. 
	Let $K$ be a compact subset of $X$ and $G$ be a relatively compact open subset $G$ of $X$ with $K \subseteq G$. 
	By Urysohn's lemma, there exists $\varphi_{0} \in \contfunc_{c}(X)$ such that $\varphi_{0} = 2$ on $K$ and $\varphi_{0} = 0$ on $X \setminus G$. 
	Let $\varepsilon \in (0,1/2)$.
	Since $\mathcal{F} \cap \contfunc_{c}(X)$ is a core of $(\mathcal{E},\mathcal{F})$, there exists $\widetilde{\varphi} \in \mathcal{F} \cap \contfunc_{c}(X)$ such that $\norm{\varphi_{0} - \widetilde{\varphi}}_{\sup} < \varepsilon$. 
	Choose $\psi \in \mathcal{F} \cap \contfunc_{c}(X)$ satisfying $\psi = 1$ on $\supp_{X}[\,\widetilde{\varphi}\,]$ and $\psi \ge 0$ on $X$, which exists by the regularity of $(\mathcal{E},\mathcal{F})$, the local compactness of $X$ and \eqref{e:dom-unitcontractive}.
	Now we define $\varphi \in \contfunc_{c}(X)$ by $\varphi \coloneqq (\widetilde{\varphi} - \varepsilon\psi)^{+} \wedge 1$. 
	(Note that $\supp_{X}[\varphi]$ is compact since $\closure{G}^{X}$ is compact.)
	Then $\varphi \in \mathcal{F} \cap \contfunc_{c}(X)$ by \eqref{e:dom-unitcontractive}. 
	Clearly, $\varphi \ge 0$ on $X$, $\varphi = 1$ on $K$ and $\varphi = 0$ on $X \setminus G$, so the proof is completed. 
\end{proof}

The proposition above ensures when there exist \emph{cutoff} functions in $\mathcal{F}$. \index{cutoff function}
We also introduce the following condition stating the existence of cutoff functions in a weaker sense. 
\begin{defn}\label{defn.CF-measure}
	We say that a $p$-energy form $(\mathcal{E},\mathcal{F})$ on $(X,m)$ satisfies the property \hyperref[defn.CF-measure]{\textup{(CF)$_{m}$}} if and only if, for any open subset $U$ of $X$ and any compact subset $K$ of $U$, there exists $\varphi \in \mathcal{F} \cap L^{\infty}(X,m)$ such that $\varphi(x) = 1$ for $m$-a.e.\ $x \in K$ and $\varphi(x) = 0$ for $m$-a.e.\ $x \in X \setminus U$.
\end{defn}

We could consider variants of \hyperref[defn.CF-measure]{\textup{(CF)$_{m}$}} such as one requiring $\varphi \in \mathcal{F} \cap \contfunc(X)$ in addition, but we do not discuss those in this paper.
Note that \hyperref[defn.CF-measure]{\textup{(CF)$_{m}$}} holds if $(\mathcal{E},\mathcal{F})$ admits a special core. 

Next we introduce two formulations of the notion of strong locality for $(\mathcal{E},\mathcal{F})$. 
\begin{defn}[Strong locality\index{strong locality (of $p$-energy form)}]\label{defn.Epsl}
	\begin{enumerate}[label=\textup{(\arabic*)},align=left,leftmargin=*,topsep=2pt,parsep=0pt,itemsep=2pt]
		\item\label{it:SL1} We say that $(\mathcal{E},\mathcal{F})$ has the strong local property\index{strong local property (of $p$-energy form)} \hyperref[it:SL1]{\textup{(SL1)}} if and only if, for any $f_{1},f_{2},g \in \mathcal{F}$ with either $\supp_{m}[f_{1} - a_{1}\indicator{X}]$ or $\supp_{m}[f_{2} - a_{2}\indicator{X}]$ compact and $\supp_{m}[f_{1} - a_{1}\indicator{X}] \cap \supp_{m}[f_{2} - a_{2}\indicator{X}] = \emptyset$ for some $a_{1},a_{2} \in \mathbb{R}$, 
			\begin{equation}\label{e:defn.sl1}
				\mathcal{E}(f_{1} + f_{2} + g) + \mathcal{E}(g) = \mathcal{E}(f_{1} + g) + \mathcal{E}(f_{2} + g). 
			\end{equation}
		\item\label{it:SL2} Assume that $(\mathcal{E},\mathcal{F})$ satisfies \ref{Cp}. We say that $(\mathcal{E},\mathcal{F})$ has the strong local property \hyperref[it:SL2]{\textup{(SL2)}} if and only if, for any $f_{1},f_{2},g \in \mathcal{F}$ with either $\supp_{m}[f_{1} - f_{2} - a\indicator{X}]$ or $\supp_{m}[g - b\indicator{X}]$ compact and $\supp_{m}[f_{1} - f_{2} - a\indicator{X}] \cap \supp_{m}[g - b\indicator{X}] = \emptyset$ for some $a,b \in \mathbb{R}$, 
			\begin{equation}\label{e:defn.sl2}
				\mathcal{E}(f_{1}; g) = \mathcal{E}(f_{2}; g).  
			\end{equation}
	\end{enumerate}
\end{defn}

In the following propositions, we collect basic results on \hyperref[it:SL1]{\textup{(SL1)}} and \hyperref[it:SL2]{\textup{(SL2)}}. 
\begin{prop}\label{prop.sl-baisc}
	Assume that $(\mathcal{E},\mathcal{F})$ satisfies \ref{Cp}. 
	\begin{enumerate}[label=\textup{(\alph*)},align=left,leftmargin=*,topsep=2pt,parsep=0pt,itemsep=2pt]
		\item\label{it:SL1-conseq} If $(\mathcal{E},\mathcal{F})$ satisfies \hyperref[it:SL1]{\textup{(SL1)}}, then for any $f_{1},f_{2},g \in \mathcal{F}$ with either $\supp_{m}[f_{1} - a_{1}\indicator{X}]$ or $\supp_{m}[f_{2} - a_{2}\indicator{X}]$ compact and $\supp_{m}[f_{1} - a_{1}\indicator{X}] \cap \supp_{m}[f_{2} - a_{2}\indicator{X}] = \emptyset$ for some $a_{1},a_{2} \in \mathbb{R}$,
			\begin{equation}\label{e:sl1-conseq}
				\mathcal{E}(f_{1} + f_{2}; g) = \mathcal{E}(f_{1}; g) + \mathcal{E}(f_{2}; g). 
			\end{equation}
		\item\label{it:SL2-conseq} If $(\mathcal{E},\mathcal{F})$ satisfies \hyperref[it:SL2]{\textup{(SL2)}}, then for any $f_{1},f_{2},g \in \mathcal{F}$ with either $\supp_{m}[f_{1} - f_{2} - a\indicator{X}]$ or $\supp_{m}[g - b\indicator{X}]$ compact and $\supp_{m}[f_{1} - f_{2} - a\indicator{X}] \cap \supp_{m}[g - b\indicator{X}] = \emptyset$ for some $a,b \in \mathbb{R}$, 
			\begin{equation}\label{e:sl2-conseq}
				\mathcal{E}(g; f_{1}) = \mathcal{E}(g; f_{2}).  
			\end{equation}
	\end{enumerate}
\end{prop}
\begin{proof}
	\ref{it:SL1-conseq}: 
	Note that \eqref{e:defn.sl1} in the definition of \hyperref[it:SL1]{\textup{(SL1)}} (Definition \ref{defn.Epsl}-\ref{it:SL1}) with $g = 0$ implies that $\mathcal{E}(f_1 + f_2) = \mathcal{E}(f_1)+ \mathcal{E}(f_2)$. 
	For any $t \in (0,\infty)$, we have from \eqref{e:defn.sl1} that 
	\[
	\frac{\mathcal{E}(f_1 + f_2 + tg) - \mathcal{E}(f_1 + f_2)}{t} + t^{p - 1}\mathcal{E}(g) = \frac{\mathcal{E}(f_1 + tg) - \mathcal{E}(f_1)}{t} + \frac{\mathcal{E}(f_2 + tg) - \mathcal{E}(f_2)}{t}. 
	\]
	We obtain \eqref{e:sl1-conseq} by letting $t \downarrow 0$ in this equality. 
	
	\ref{it:SL2-conseq}: 
	Since $\mathcal{E}(g; \,\cdot\,)$ is linear by Theorem \ref{thm.p-form}, it suffices to prove $\mathcal{E}(g; f_1 - f_2) = 0$, which follows from \eqref{e:defn.sl2} in the definition of \hyperref[it:SL2]{\textup{(SL2)}} (Definition \ref{defn.Epsl}-\ref{it:SL2}) with $g,0,f_1 - f_2$ in place of $f_1,f_2,g$. 
\end{proof}
 
\begin{prop}\label{prop.sl-other}
	Assume that $(\mathcal{E},\mathcal{F})$ satisfies \ref{Cp}. 
	\begin{enumerate}[label=\textup{(\alph*)},align=left,leftmargin=*,topsep=2pt,parsep=0pt,itemsep=2pt]
		\item\label{it:SL1-SL2} If $(\mathcal{E},\mathcal{F})$ satisfies \hyperref[it:SL1]{\textup{(SL1)}}, then $(\mathcal{E},\mathcal{F})$ also satisfies \hyperref[it:SL2]{\textup{(SL2)}}.
		\item\label{it:SL2-SL1} Assume that $(\mathcal{E},\mathcal{F})$ satisfies \hyperref[it:SL2]{\textup{(SL2)}} and the following three conditions:
		\begin{gather}
			\text{$uv \in \mathcal{F}$ for any $u,v \in \mathcal{F} \cap L^{\infty}(X,m)$.} \label{e:sl-leibniz}\\
			\hspace*{-18pt}\text{For any $u \in \mathcal{F}$, $\{(-n) \vee (u \wedge n)\}_{n \in \mathbb{N}} \subseteq \mathcal{F}$ and $\lim_{n \to \infty}\mathcal{E}\bigl(u - (-n) \vee (u \wedge n)\bigr) = 0$.} \label{e:sl-bddapprox}\\
			\text{$(\mathcal{E},\mathcal{F})$ satisfies \hyperref[defn.CF-measure]{\textup{(CF)$_{m}$}}.} \label{e:sl-special}
		\end{gather}
		Then $(\mathcal{E},\mathcal{F})$ satisfies \hyperref[it:SL1]{\textup{(SL1)}}. 
	\end{enumerate}
\end{prop}
\begin{proof}
	\ref{it:SL1-SL2}: 
	Let $f_{1},f_{2},g \in \mathcal{F}$, $a,b \in \mathbb{R}$ and $t \in \mathbb{R} \setminus \{0\}$, and assume that either $\supp_{m}[f_{1} - f_{2} - a\indicator{X}]$ or $\supp_{m}[g - b\indicator{X}]$ is compact and that $\supp_{m}[f_{1} - f_{2} - a\indicator{X}] \cap \supp_{m}[g - b\indicator{X}] = \emptyset$.  
	By \eqref{e:defn.sl1} in the definition of \hyperref[it:SL1]{\textup{(SL1)}} with $f_{2} - f_{1},tg,f_{1}$ in place of $f_{1},f_{2},g$ we have
	\[
	\mathcal{E}\bigl((f_{2} - f_{1}) + tg + f_{1}\bigr) + \mathcal{E}(f_{1}) 
	= \mathcal{E}\bigl((f_{2} - f_{1}) + f_{1}\bigr) + \mathcal{E}(tg + f_{1}), 
	\]
	whence 
	\[
	\mathcal{E}(f_{1};g) = \frac{1}{p} \lim_{t \to 0}\frac{\mathcal{E}(f_{1} + tg) - \mathcal{E}(f_{1})}{t} = \frac{1}{p} \lim_{t \to 0}\frac{\mathcal{E}(f_{2} + tg) - \mathcal{E}(f_{2})}{t} = \mathcal{E}(f_{2};g),
	\]
	proving \hyperref[it:SL2]{\textup{(SL2)}}.  
	
	\ref{it:SL2-SL1}:
	We first consider the case $g \in \mathcal{F} \cap L^{\infty}(X,m)$.
	Let $f_{1},f_{2} \in \mathcal{F}$ and $a_{1},a_{2} \in \mathbb{R}$, and assume that $\supp_{m}[f_{1} - a_{1}\indicator{X}]$ is compact and that $\supp_{m}[f_{1} - a_{1}\indicator{X}] \cap \supp_{m}[f_{2} - a_{2}\indicator{X}] = \emptyset$. 
	Let $U$ be an open neighborhood of $\supp_{m}[f_{1} - a_{1}\indicator{X}]$ such that $U \subseteq X \setminus \supp_{m}[f_{2} - a_{2}\indicator{X}]$. 
	By \eqref{e:sl-special} and the local compactness of $X$, there exists $\varphi \in \mathcal{F} \cap L^{\infty}(X,m)$ such that $\varphi(x) = 1$ for $m$-a.e.\ $x \in U$, $\supp_{m}[\varphi]$ is compact and $\supp_{m}[\varphi] \cap \supp_{m}[f_{2} - a_{2}\indicator{X}] = \emptyset$.   
	Note that $\varphi g \in \mathcal{F}$ by \eqref{e:sl-leibniz}. 
	Then we see from \hyperref[it:SL2]{\textup{(SL2)}} that 
	\begin{align}\label{e:SL2decomp}
		\mathcal{E}(f_{1} + f_{2} + g) + \mathcal{E}(g)
		&= \mathcal{E}(f_{1} + f_{2} + g; f_{1}) + \mathcal{E}(f_{1} + f_{2} + g; f_{2}) + \mathcal{E}(f_{1} + f_{2} + g; g) + \mathcal{E}(g) \nonumber \\
		&\overset{\hyperref[it:SL2]{\textup{(SL2)}}}{=} \mathcal{E}(f_{1} + g; f_{1}) + \mathcal{E}(f_{2} + g; f_{2}) + \mathcal{E}(f_{1} + f_{2} + g; g) + \mathcal{E}(g) \nonumber \\
		&= \mathcal{E}(f_{1} + g; f_{1}) + \mathcal{E}(f_{2} + g; f_{2})  \nonumber \\
		&\quad + \mathcal{E}(f_{1} + f_{2} + g; (1 - \varphi)g) + \mathcal{E}(f_{1} + f_{2} + g; \varphi g) + \mathcal{E}(g). 
	\end{align}
	Since $\supp_{m}[\varphi g]$ and $\supp_{m}[f_{1} - a_{1}\indicator{X}]$ are compact, $\supp_{m}[f_{1} - a_{1}\indicator{X}] \cap \supp_{m}[(1 - \varphi)g] = \emptyset$ and $\supp_{m}[f_{2} - a_{2}\indicator{X}] \cap \supp_{m}[\varphi g] = \emptyset$, we have the following equalities by \hyperref[it:SL2]{\textup{(SL2)}}: 
	\begin{align*}
		\mathcal{E}(f_{1} + f_{2} + g; (1 - \varphi)g) &= \mathcal{E}(f_{2} + g; (1 - \varphi)g), \\
		\mathcal{E}(f_{1} + f_{2} + g; \varphi g) &= \mathcal{E}(f_{1} + g; \varphi g), \\
		\mathcal{E}(g)
		= \mathcal{E}(g; (1 - \varphi)g) + \mathcal{E}(g; \varphi g)
		&= \mathcal{E}(f_{1} + g; (1 - \varphi)g) + \mathcal{E}(f_{2} + g; \varphi g). 
	\end{align*}
	By combining these equalities and \eqref{e:SL2decomp}, we obtain 
	\begin{align*}
		\mathcal{E}(f_{1} + f_{2} + g) + \mathcal{E}(g)
		&= \mathcal{E}(f_{1} + g; f_{1}) + \mathcal{E}(f_{2} + g; f_{2}) + \mathcal{E}(f_{1} + g; g) + \mathcal{E}(f_{2} + g; g) \\
		&= \mathcal{E}(f_{1} + g) + \mathcal{E}(f_{2} + g). 
	\end{align*}
	The proof for the case where $\supp_{m}[f_{2} - a_{2}\indicator{X}]$ instead of $\supp_{m}[f_{1} - a_{1}\indicator{X}]$ is compact is similar, so \hyperref[it:SL1]{\textup{(SL1)}} holds if $g \in \mathcal{F} \cap L^{\infty}(X,m)$. 
	
	Lastly, we prove \hyperref[it:SL1]{\textup{(SL1)}} without assuming the boundedness of $g$.  
	Let $g \in \mathcal{F}$, $n \in \mathbb{N}$ and set $g_{n} \coloneqq (-n) \vee (g \wedge n)$. 
	Then $g_{n} \in \mathcal{F}$ by \eqref{e:sl-bddapprox}, and the statement proved in the previous paragraph yields that 
	\[
	\mathcal{E}(f_{1} + f_{2} + g_{n}) + \mathcal{E}(g_{n})
	= \mathcal{E}(f_{1} + g_{n}) + \mathcal{E}(f_{2} + g_{n}).
	\]
	Thanks to \eqref{e:sl-bddapprox} and the triangle inequality for $\mathcal{E}^{1/p}$, we obtain the desired equality \eqref{e:defn.sl2} in the definition of \hyperref[it:SL1]{\textup{(SL1)}} by letting $n \to \infty$ in the last equality. 
\end{proof}

\section{\texorpdfstring{$p$}{p}-Energy measures and their basic properties}\label{sec.pEM}
In this section, we discuss $p$-energy measures dominated by a $p$-energy form.
Similar to the case of $p$-energy forms, we introduce the two-variable version of $p$-energy measures and prove their basic properties.

As in the previous section, throughout this section we fix $p \in (1,\infty)$, a measure space $(X,\mathcal{B},m)$ and a $p$-energy form $(\mathcal{E}, \mathcal{F})$ on $(X,m)$. 

\subsection{\texorpdfstring{$p$}{p}-Energy measures and \texorpdfstring{$p$}{p}-Clarkson's inequality}
The following definition specifies the class of families of measures which we call $p$-energy measures and consider in this section. 

\begin{defn}[$p$-Energy measures dominated by a $p$-energy form]\label{defn.em-Cp}
	Let $\SigmaAlgEM$ be a $\sigma$-algebra in $X$,\footnote{While we typically take $\SigmaAlgEM = \mathcal{B} = \mathcal{B}(X)$ for a prescribed topology on $X$, we allow $\SigmaAlgEM \not= \mathcal{B}$ here. This formulation is suitable in the setting of a $p$-resistance form on $X$ considered in Section \ref{sec.p-harm} and later, where we choose $(\mathcal{B},m)$ to be the pair of $2^{X}$ and the counting measure on $X$ as mentioned in Remark \ref{rmk:wo-measure} but may take $\SigmaAlgEM = \mathcal{B}(X)$ for the topology on $X$ induced by the associated $p$-resistance metric.} and let $\{ \Gamma\langle f \rangle \}_{f \in \mathcal{F}}$ be a family of measures on $(X,\SigmaAlgEM)$. 
	We say that $\{ \Gamma\langle f \rangle \}_{f \in \mathcal{F}}$ is a \emph{family of $p$-energy measures\index{$p$-energy measure}\index{family of $p$-energy measures dominated by $p$-energy form} on $(X,\SigmaAlgEM)$ dominated by $(\mathcal{E},\mathcal{F})$} if and only if the following hold: 
	\begin{enumerate}[label=\textup{(EM\arabic*)$_p$},align=left,leftmargin=*,topsep=2pt,parsep=0pt,itemsep=2pt]
    	\item\label{EM1} $\Gamma\langle f \rangle(X) \le \mathcal{E}(f)$ for any $f \in \mathcal{F}$.
    	\item\label{EM2} $\Gamma\langle \,\cdot\, \rangle(A)^{1/p}$ is a seminorm on $\mathcal{F}$ for any $A \in \SigmaAlgEM$.
	\end{enumerate}
	We then see that $(\Gamma\langle \,\cdot\, \rangle(A),\mathcal{F})$ is a $p$-energy form on $(X,m)$ for each $A \in \SigmaAlgEM$ by \ref{EM2}. 
	
	We say that $\{ \Gamma\langle f \rangle \}_{f \in \mathcal{F}}$ satisfies \emph{$p$-Clarkson's inequality}\index{$p$-Clarkson's inequality (for $p$-energy measures)}, \ref{Cp-em} for short, if and only if $(\Gamma\langle \,\cdot\, \rangle(A),\mathcal{F})$ satisfies \ref{Cp} for any $A \in \SigmaAlgEM$, i.e.,  for any $f, g \in \mathcal{F}$, 
	\begin{align*}\label{Cp-em}
    \mspace{-4mu}\begin{cases}
        \Gamma\langle f + g \rangle(A) + \Gamma\langle f - g \rangle(A) \ge 2\bigl(\Gamma\langle f \rangle(A)^{\frac{1}{p - 1}} + \Gamma\langle g \rangle(A)^{\frac{1}{p - 1}}\bigr)^{p - 1} &\textrm{if $p \in (1,2]$,} \\
        \Gamma\langle f + g \rangle(A) + \Gamma\langle f - g \rangle(A) \le 2\bigl(\Gamma\langle f \rangle(A)^{\frac{1}{p - 1}} + \Gamma\langle g \rangle(A)^{\frac{1}{p - 1}}\bigr)^{p - 1} &\textrm{if $p \in [2,\infty)$.} \tag*{\textup{(Cla)$^{\textup{EM}}_{p}$}}
    \end{cases}
	\end{align*}
	We also say that $\{ \Gamma\langle f \rangle \}_{f \in \mathcal{F}}$ satisfies the \emph{generalized $p$-contraction property}\index{generalized $p$-contraction property (for $p$-energy measures)}, \hypertarget{GC-em}{\textup{(GC)$^{\textup{EM}}_{p}$}} for short, if and only if $(\Gamma\langle \,\cdot\, \rangle(A),\mathcal{F})$ satisfies \ref{GC} for any $A \in \SigmaAlgEM$. 
\end{defn}

\begin{example}\label{ex.em}
	\begin{enumerate}[label=\textup{(\arabic*)},align=left,leftmargin=*,topsep=2pt,parsep=0pt,itemsep=2pt]
		\item\label{EM-Rn} Consider the same setting as in Example \ref{ex.Rn}-\ref{Ep-Rn}.
		Then the family $\{ \Gamma\langle f \rangle \}_{f \in W^{1,p}(\Omega)}$ of Borel measures on $\Omega$ given by 
		\[
		\Gamma\langle f \rangle(A) \coloneqq \int_{A}\abs{\nabla f(x)}^{p}\,dx \quad \text{for $f \in W^{1,p}(\Omega)$ and $A \in \mathcal{B}(\Omega)$,}
		\]
		is easily seen to be a family of $p$-energy measures on $(\Omega,\mathcal{B}(\Omega))$ dominated by the $p$-energy form $(\mathcal{E},W^{1,p}(\Omega))$ given by $\mathcal{E}(f) \coloneqq \int_{\Omega}\abs{\nabla f(x)}^{p}\,dx$. 
		Similar to Example \ref{ex.Rn}-\ref{Ep-Rn}, one can show \hyperlink{GC-em}{\textup{(GC)$^{\textup{EM}}_{p}$}} for $\{ \Gamma\langle f \rangle \}_{f \in W^{1,p}(\Omega)}$ by following an argument in the proof of Theorem \ref{thm.GCp-kuwae}. 
		Recall that $\mathcal{E}(f; g) = \int_{\Omega}\abs{\nabla f(x)}^{p - 2}\langle \nabla f(x), \nabla g(x) \rangle_{\mathbb{R}^{D}}\,dx$. 
		Then we can see that, by the Leibniz and the chain rule for $\nabla$, for any $u,\varphi \in W^{1,p}(\Omega) \cap C^{1}(\Omega)$,  
		\begin{equation}\label{pEM-Euclid}
			\int_{\Omega}\varphi\,d\Gamma\langle u \rangle
			= \mathcal{E}(u; u\varphi) - \left(\frac{p - 1}{p}\right)^{p - 1}\mathcal{E}\bigl(\abs{u}^{\frac{p}{p - 1}}; \varphi\bigr). 
		\end{equation}
		\item\label{EX.EM-fractal} Although $p$-energy forms have been constructed on compact metric spaces \cite{Kig23, MS+}, we do not know how to construct the associated $p$-energy measures because of the lack of the density ``$\abs{\nabla u(x)}^{p}$''.
		(As described in \ref{exam.DF-em} below, the theory of Dirichlet forms gives $2$-energy measures $\{ \mu_{\langle u \rangle} \}_{u \in \mathcal{F}_{2}}$ associated with a given nice Dirichlet form $(\mathcal{E}_{2},\mathcal{F}_{2})$.
		On a large class of self-similar sets, however, it is known that $\mu_{\langle u \rangle}$ is singular with respect to the natural Hausdorff measure on the underlying fractal \cite{Hin05,KM20}.)
		In the case of self-similar sets, under suitable assumptions, \emph{self-similar $p$-energy forms} are constructed in \cite{CGQ22,Kig23,MS+,Shi24}, and we can introduce $p$-energy measures satisfying \ref{EM1}, \ref{EM2} and \hyperlink{GC-em}{\textup{(GC)$^{\textup{EM}}_{p}$}} by using the self-similarity of $p$-energy forms. See Section \ref{sec.ss} for details. 
		
		In \cite{KS.lim}, under the assumption called the \emph{weak monotonicity condition}, the authors construct a good $p$-energy form $\mathcal{E}_{p}^{\mathrm{KS}}$, which is called a \emph{Korevaar--Shoen $p$-energy form}, on a locally compact separable metric space $(X,d)$ equipped with a $\sigma$-finite Borel measure $m$ with full topological support. As an advantage of $\mathcal{E}_{p}^{\mathrm{KS}}$, the right-hand side of \eqref{pEM-Euclid} with $\mathcal{E}_{p}^{\mathrm{KS}}$ in place of $\mathcal{E}$ can be extended to a bounded positive linear functional in $\varphi \in \contfunc_{c}(X)$ and the $p$-energy measure $\Gamma_{p}^{\mathrm{KS}}\langle u \rangle$ associated with $\mathcal{E}_{p}^{\mathrm{KS}}$ is constructed as the unique Radon measure corresponding to this functional through the Riesz--Markov--Kakutani representation theorem. A notable fact is that this approach does not rely on the self-similarity of the underlying space or of the $p$-energy form. In \cite[Sections 3 and 4]{KS.lim}, basic properties of $\Gamma_{p}^{\mathrm{KS}}\langle \,\cdot\, \rangle$ like \ref{EM1}, \ref{EM2} and \hyperlink{GC-em}{\textup{(GC)$^{\textup{EM}}_{p}$}} are also shown. 
		
		Added in revision: in a very recent result \cite{Sas26}, the $p$-energy measures associated to general strongly local $p$-energy forms satisfying \ref{Cp} and certain natural conditions have been constructed without using the self-similarity. 
		\item\label{exam.DF-em} The case of $p = 2$ is very special thanks to the theory of symmetric Dirichlet forms. If $(\mathscr{E}, D(\mathscr{E}))$ is a strongly local regular symmetric Dirichlet form on $L^{2}(X, m)$, where $X$ and $m$ are as specified in \eqref{a:loccpt} and \eqref{a:fullRadon} at the beginning of Subsection \ref{sec.regularlocal}, then $\mathscr{E}(u) \coloneqq \mathscr{E}(u,u)$ is a $2$-energy form on $(X,m)$ and satisfies \hyperref[GC]{\textup{(GC)$_2$}} (see Proposition \ref{prop.GC.DF}).
        In addition, the Dirichlet form theory provides us with a Borel measure $\mu_{\langle u \rangle}$ on $X$, called the \emph{$\mathscr{E}$-energy measure} of $u \in D(\mathscr{E})$ associated with $(\mathscr{E}, D(\mathscr{E}))$, through the following formula\footnote{To be precise, the definition of $\mu_{\langle u \rangle}$ through \eqref{d:2-EM} is valid only for $u \in D(\mathscr{E}) \cap L^{\infty}(X,m)$, but we can still define $\mu_{\langle u \rangle}$ for any $u \in D(\mathscr{E})$ by setting $\mu_{\langle u \rangle}(A) \coloneqq \lim_{n\to\infty}\mu_{\langle (-n) \vee (u \wedge n) \rangle}(A)$ for each $A \in \mathcal{B}(X)$.}:
        \begin{equation}\label{d:2-EM}
            \int_{X}\varphi\,d\mu_{\langle u \rangle} = \mathscr{E}(u,u\varphi) - \frac{1}{2}\mathscr{E}(u^{2}, \varphi) \quad \text{for any $\varphi \in D(\mathscr{E}) \cap \contfunc_{c}(X)$}
        \end{equation}
        (recall \eqref{pEM-Euclid}, and see \cite[Section 3.2]{FOT} for details on energy measures associated with regular symmetric Dirichlet forms).
        We easily see that $\{ \mu_{\langle u \rangle}\}_{u \in D(\mathscr{E})}$ satisfies \hyperref[EM1]{(\textup{EM1})$_{2}$} and the parallelogram law, which implies \hyperref[EM2]{(\textup{EM2})$_{2}$} and \hyperref[Cp-em]{\textup{(Cla)$^{\textup{EM}}_{2}$}}. 
        We can also verify \hyperlink{GC-em}{\textup{(GC)$^{\textup{EM}}_{2}$}} for $\{ \mu_{\langle u \rangle}\}_{u \in D(\mathscr{E})}$ (Proposition \ref{prop.GCkuwae}). 
        As discussed in \cite{Kuw23+}, under the additional assumption of a suitable closability in $L^{p}(X,m)$ formulated as \eqref{assum.closability} in Definition \ref{defn.kuwae-form}, we can introduce a family of $p$-energy measures on $(X,\mathcal{B}(X))$ satisfying \ref{EM1}, \ref{EM2} and \hyperlink{GC-em}{\textup{(GC)$^{\textup{EM}}_{p}$}} by setting $\Gamma\langle u \rangle(A) \coloneqq \int_{A}\Gamma_{\mu}(u)^{\frac{p}{2}}\,d\mu$, where $\mu$ is an $\mathscr{E}$-dominant measure (i.e., $\mu_{\langle u \rangle} \ll \mu$ for any $u \in D(\mathscr{E})$) and $\Gamma_{\mu}(u) \coloneqq d\mu_{\langle u \rangle}/d\mu$; see Theorem \ref{thm.GCp-kuwae} for the details of this family of $p$-energy measures. 
        \item Let $(X,d)$ be a separable metric space and $m$ a Borel measure on $X$ such that $m(X) > 0$ and $m(B_{d}(x,r)) < \infty$ for some $r \in (0,\infty)$ for any $x \in X$. 
		Let $g_{u}$ be the minimal $p$-weak upper gradient of $u \in N^{1,p}(X,m)$, where $N^{1,p}(X,m) \coloneqq \{ u \in L^{p}(X,m) \mid g_{u} \in L^{p}(X,m) \}$ is the Newton--Sobolev space\index{Newton--Sobolev space} (see \cite[Section 7.1]{HKST} for details on the Newton--Sobolev spaces). Then $\Gamma\langle u \rangle(A) \coloneqq \int_{A}g_{u}^{p}\,dm$ defines $p$-energy measures satisfying \ref{EM1} and \ref{EM2}.
        Indeed, we have \ref{EM2} by the sublinearity of the map $u \mapsto g_{u}$ \cite[(6.3.18)]{HKST}.
        However, \ref{Cp-em} for these measures is unclear because the map $u \mapsto g_{u}$ is not linear in general. 
	\end{enumerate}
\end{example}

In the rest of this subsection, we assume that $\SigmaAlgEM$ is a $\sigma$-algebra in $X$ and that $\{ \Gamma\langle f \rangle \}_{f \in \mathcal{F}}$ is a family of $p$-energy measures on $(X,\SigmaAlgEM)$ dominated by $(\mathcal{E},\mathcal{F})$. 
The same argument as in the proof of Proposition \ref{prop.diffble} yields the following result.
\begin{prop}\label{prop.c-diff-em}
	Assume that $\{ \Gamma\langle f \rangle \}_{f \in \mathcal{F}}$ satisfies \ref{Cp-em}.
	Then for any $f,g \in \mathcal{F}$ and any $A \in \SigmaAlgEM$, 
    \begin{align}\label{c-diff-em}
    	&\Gamma\langle f + g \rangle(A) + \Gamma\langle f - g \rangle(A) - 2\Gamma\langle f \rangle(A) \nonumber \\
    	&\le 2\bigl(1 \vee (p-1)\bigr)\Bigl[\Gamma\langle f \rangle(A)^{\frac{1}{p - 1}} + \Gamma\langle g \rangle(A)^{\frac{1}{p - 1}}\Bigr]^{(p - 2)^{+}}\Gamma\langle f \rangle(A)^{1 \wedge \frac{1}{p-1}},
    \end{align} 
    and the function $\mathbb{R} \ni t \mapsto \Gamma\langle f + tg \rangle(A) \in [0,\infty)$ is differentiable. Moreover, for any $c \in (0,\infty)$,
	\begin{equation}\label{e:frechet.diffble.em}
    	\lim_{\delta \downarrow 0}\sup_{A \in \SigmaAlgEM,\, f,g \in \mathcal{F};\, \mathcal{E}(f) \leq c/(p-2)^{+},\, \mathcal{E}(g) \leq 1} \abs{\frac{\Gamma\langle f + \delta g \rangle(A) - \Gamma\langle f \rangle(A)}{\delta} - \frac{d}{dt}\Gamma\langle f + tg \rangle(A)\biggr|_{t = 0}} = 0. 
    \end{equation}
\end{prop}

\begin{defn}\label{defn.emss}
    Assume that $\{ \Gamma\langle f \rangle \}_{f \in \mathcal{F}}$ satisfies \ref{Cp-em}.
    For each $f,g \in \mathcal{F}$, we define $\Gamma\langle f; g \rangle \colon \SigmaAlgEM \to \mathbb{R}$ by
    \begin{equation}\label{e:defn.emss}
    	\Gamma\langle f; g \rangle(A) \coloneqq \frac{1}{p}\left.\frac{d}{dt}\Gamma\langle f + tg \rangle(A)\right|_{t = 0}, \qquad A \in \SigmaAlgEM,
    \end{equation}
    which exists by Proposition \ref{prop.c-diff-em}. 
\end{defn}

The following properties of $\Gamma\langle f; g \rangle$ can be shown in a similar way as Theorem \ref{thm.p-form}.
\begin{thm}\label{thm.em-basic}
    Assume that $\{ \Gamma\langle f \rangle \}_{f \in \mathcal{F}}$ satisfies \ref{Cp-em}.
    Let $A \in \SigmaAlgEM$. 
    Then $\Gamma\langle f;\,\cdot\, \rangle(A)$ is the Fr\'{e}chet derivative of $\frac{1}{p}\Gamma\langle \,\cdot\, \rangle(A) \colon \mathcal{F}/\mathcal{E}^{-1}(0) \to [0,\infty)$ at $f \in \mathcal{F}$.
    In particular, the map $\Gamma\langle f; \,\cdot\, \rangle(A) \colon \mathcal{F} \to \mathbb{R}$ is linear, $\Gamma\langle f; f\rangle(A) = \Gamma\langle f \rangle(A)$, and $\Gamma\langle f; h \rangle(A) = 0$ for any $h \in \mathcal{F}$ with $\Gamma\langle h \rangle(A) = 0$.
    Moreover, for any $f, f_1, f_2, g, h \in \mathcal{F}$ with $\Gamma\langle h \rangle(A) = 0$ and any $a \in \mathbb{R}$, the following hold:
    \begin{gather}
        \text{$\mathbb{R} \ni t \mapsto \Gamma\langle f + tg; g\rangle(A) \in \mathbb{R}$ is strictly increasing if and only if $\Gamma\langle g \rangle(A) > 0$.} \label{em-form.mono} \\
        \Gamma\langle af; g\rangle(A) = \sgn(a)\abs{a}^{p - 1}\Gamma\langle f; g \rangle(A), \qquad \Gamma\langle f + h; g \rangle(A) = \Gamma\langle f;g \rangle(A). \label{em-form.basic} \\
        \abs{\Gamma\langle f; g \rangle(A)} \le \Gamma\langle f \rangle(A)^{(p - 1)/p}\Gamma\langle g \rangle(A)^{1/p}. \label{em-bdd.form} \\
        \abs{\Gamma\langle f_1; g \rangle(A) - \Gamma\langle f_2; g \rangle(A)} \le C_{p}\bigl(\Gamma\langle f_1 \rangle(A) \vee \Gamma\langle f_2 \rangle(A)\bigr)^{\frac{p - 1 - \alpha_{p}}{p}}\Gamma\langle f_1 - f_2 \rangle(A)^{\frac{\alpha_{p}}{p}} \Gamma\langle g \rangle(A)^{\frac{1}{p}}, \label{em-ncont}
    \end{gather}
    where $\alpha_{p}, C_{p}$ are the same as in Theorem \ref{thm.p-form}.
\end{thm}

The set function $\Gamma\langle f; g \rangle$ is a signed measure as shown in the following theorem. 
\begin{thm}\label{thm.signed}
	Assume that $\{ \Gamma\langle f \rangle \}_{f \in \mathcal{F}}$ satisfies \ref{Cp-em}.
	Then for any $f, g \in \mathcal{F}$, the set function $\Gamma\langle f; g \rangle$ is a signed measure on $(X, \SigmaAlgEM)$.
    Moreover, for any $\SigmaAlgEM$-measurable function $\varphi \colon X \to [0,\infty)$ with $\norm{\varphi}_{\sup} < \infty$, $\int_{X}\varphi\,d\Gamma \langle\,\cdot\, \rangle \colon \mathcal{F}/\mathcal{E}^{-1}(0) \to [0,\infty)$ is Fr\'{e}chet differentiable, 
    \begin{equation}\label{emint.frechet}
        \int_{X}\varphi\,d\Gamma\langle f; g \rangle = \frac{1}{p}\left.\frac{d}{dt}\int_{X}\varphi\,d\Gamma\langle f + tg \rangle\right|_{t = 0}
    \end{equation}
    for any $f,g \in \mathcal{F}$, and $\int_{X}\varphi\,d\Gamma \langle \,\cdot\, \rangle, \int_{X}\varphi\,d\Gamma \langle \,\cdot\,;\,\cdot\, \rangle$ have the same properties as those of $\Gamma\langle \,\cdot\, \rangle(A), \Gamma\langle \,\cdot\,;\,\cdot\, \rangle(A)$ in Theorem \ref{thm.em-basic}.
\end{thm}
\begin{proof}
	The equalities $\Gamma\langle f; g \rangle(\emptyset) = 0$ and $\abs{\Gamma\langle f; g \rangle(X)} = \abs{\mathcal{E}(f;g)} < \infty$ are clear from the definition.
	We will show the countable additivity of $\Gamma\langle f; g \rangle$ .
    The finite additivity of $\Gamma\langle f; g \rangle$ is obvious.
    Let $\{ A_{n} \}_{n \in \mathbb{N}} \subseteq \SigmaAlgEM$ be a family of disjoint measurable sets.
    Set $B_{N} \coloneqq \bigcup_{n = N +1}^{\infty}A_{n}$ for each $N \in \mathbb{N}$.
    Then we see that
    \begin{align*}
        \abs{\Gamma\langle f;g \rangle\left(\bigcup_{n \in \mathbb{N}}A_{n}\right) - \sum_{n = 1}^{N}\Gamma\langle f;g \rangle(A_{n})}
        &= \abs{\Gamma\langle f;g \rangle(B_{N})} \\
        &\overset{\eqref{em-bdd.form}}{\le} \Gamma\langle f \rangle(B_{N})^{(p - 1)/p}\Gamma\langle g \rangle(B_{N})^{1/p}
        \xrightarrow[N \to \infty]{} 0,
    \end{align*}
    which shows that $\Gamma\langle f;g \rangle$ is a signed measure on $(X,\SigmaAlgEM)$.

    The other properties except for \eqref{emint.frechet} can be proved by following the arguments in the proof of Theorem \ref{thm.p-form}, so we shall prove \eqref{emint.frechet}.
    By the finite additivity of $\int_{X}\varphi\,d\Gamma\langle f; g \rangle$ and $\frac{1}{p}\left.\frac{d}{dt}\int_{X}\varphi\,d\Gamma\langle f + tg \rangle\right|_{t = 0}$ in $\varphi$, we can assume that $\varphi \ge 0$.
    Let $s_{n} = \sum_{k = 1}^{l_{n}}a_{n,k}\indicator{A_{n,k}}$ with $a_{n,k} \ge 0$ and $A_{n,k} \in \SigmaAlgEM$ be a non-decreasing sequence of simple functions such that $\lim_{n\to\infty}s_{n}(x) = \varphi(x)$ for any $x \in X$. 
    Then we immediately have \eqref{emint.frechet} with $\varphi = s_{n}$. 
    Since $\lim_{n \to \infty}\int_{X}s_{n}\,d\Gamma\langle f; g \rangle = \int_{X}\varphi\,d\Gamma\langle f; g \rangle$ by the dominated convergence theorem, it suffices to prove 
    \begin{equation}\label{intderiv.lim}
    	\lim_{n \to \infty}\left.\frac{d}{dt}\int_{X}s_{n}\,d\Gamma\langle f + tg \rangle\right|_{t = 0} 
    	= \left.\frac{d}{dt}\int_{X}\varphi\,d\Gamma\langle f + tg \rangle\right|_{t = 0}.
    \end{equation}
    Since \eqref{e:preHolder} in the proof of Theorem \ref{thm.p-form} with $\int_{X}\varphi\,d\Gamma\langle \,\cdot\, \rangle$ in place of $\mathcal{E}$ holds by the fact that $(\int_{X}\varphi\,d\Gamma\langle \,\cdot\, \rangle, \mathcal{F})$ is a $p$-energy form on $(X,m)$, we know that for any $\SigmaAlgEM$-measurable function $\psi \colon X \to [0,\infty)$ with $\norm{\psi}_{\sup} < \infty$, 
    \begin{equation}\label{e:pemintHolder}
    	\abs{\left.\frac{d}{dt}\int_{X}\psi\,d\Gamma\langle f + tg \rangle\right|_{t = 0}} 
    	\le \left(\int_{X}\psi\,d\Gamma\langle f \rangle\right)^{(p - 1)/p}\left(\int_{X}\psi\,d\Gamma\langle g \rangle\right)^{1/p}.
    \end{equation} 
    By \eqref{e:pemintHolder} with $\psi = \varphi - s_{n}$ and the dominated convergence theorem, we obtain \eqref{intderiv.lim}.  
\end{proof}

\begin{rmk}
	As mentioned in the introduction, a signed measure corresponding to $\Gamma\langle f; g \rangle$ is discussed in \cite[Section 5]{BV05} under some non-trivial assumptions, which have not been verified for fractals like the Sierpi\'{n}ski gasket and the Sierpi\'{n}ski carpet in the literature. 
\end{rmk}

The following proposition gives a natural H\"{o}lder-type inequality for the total variation measure $\abs{\Gamma\langle f;g \rangle}$ of $\Gamma\langle f;g \rangle$. 
\begin{prop}\label{prop.em-holder}
    Assume that $\{ \Gamma\langle f \rangle \}_{f \in \mathcal{F}}$ satisfies \ref{Cp-em}.
    Then for any $f,g \in \mathcal{F}$ and any $\SigmaAlgEM$-measurable functions $\varphi,\psi \colon X \to [0,\infty]$,
	\begin{equation}\label{em.holder}
		\int_{X}\varphi\psi\,d\abs{\Gamma\langle f; g \rangle}
		\le \left(\int_{X}\varphi^{\frac{p}{p - 1}}\,d\Gamma\langle f \rangle\right)^{(p - 1)/p}\left(\int_{X}\psi^{p}\,d\Gamma\langle g \rangle\right)^{1/p}.
	\end{equation}
\end{prop}
\begin{proof}
	Let $X = \mathcal{P} \sqcup \mathcal{N}$ be the Hahn decomposition with respect to $\Gamma\langle f;g \rangle$, i.e., $\mathcal{P},\mathcal{N} \in \SigmaAlgEM$, $\Gamma\langle f;g \rangle(A \cap \mathcal{P}) \ge 0$ and $\Gamma\langle f;g \rangle(A \cap \mathcal{N}) \le 0$ for any $A \in \SigmaAlgEM$.   
	Then the total variation measure $\abs{\Gamma\langle f;g \rangle}$ is given by
	\[
	\abs{\Gamma\langle f; g \rangle}(A) = \Gamma\langle f; g \rangle(\mathcal{P} \cap A) - \Gamma\langle f; g \rangle(\mathcal{N} \cap A), \quad A \in \SigmaAlgEM. 
	\]
	Therefore, by the H\"{o}lder-type estimate \eqref{em-bdd.form} in Theorem \ref{thm.em-basic},
	\begin{align}\label{em.TV}
		\abs{\Gamma\langle f; g \rangle}(A) 
		&\overset{\eqref{em-bdd.form}}{\le} \Gamma\langle f \rangle(\mathcal{P} \cap A)^{(p - 1)/p}\Gamma\langle g \rangle(\mathcal{P} \cap A)^{1/p} + \Gamma\langle f \rangle(\mathcal{N} \cap A)^{(p - 1)/p}\Gamma\langle g \rangle(\mathcal{N} \cap A)^{1/p} \nonumber \\
		&\le \bigl(\Gamma\langle f \rangle(\mathcal{P} \cap A) + \Gamma\langle f \rangle(\mathcal{N} \cap A)\bigr)^{(p - 1)/p}\bigl(\Gamma\langle g \rangle(\mathcal{P} \cap A) + \Gamma\langle g \rangle(\mathcal{N} \cap A)\bigr)^{1/p} \nonumber \\
		&= \Gamma\langle f \rangle(A)^{(p - 1)/p}\Gamma\langle g \rangle(A)^{1/p},
	\end{align}
	where we used H\"{o}lder's inequality in the second inequality.

    Now we prove \eqref{em.holder}.
    First, we consider the case where $\varphi$ and $\psi$ are given by
    \[
    \varphi = \sum_{k = 1}^{N_{1}}\widetilde{a}_{k}\indicator{A_{k}}, \quad \psi = \sum_{k = 1}^{N_{2}}\widetilde{b}_{k}\indicator{B_{k}}, \text{ where $\widetilde{a}_{k}, \widetilde{b}_{k} \in [0,\infty)$ and $A_{k}, B_{k} \in \SigmaAlgEM$.}
    \]
    Then we can assume that there exist $N \in \mathbb{N}$, $\{ a_{k} \}_{k = 1}^{N}, \{ b_{k} \}_{k = 1}^{N} \subseteq [0,\infty)$ and a disjoint family of measurable sets $\{ E_{k} \}_{k = 1}^{N} \subseteq \SigmaAlgEM$ such that $\varphi = \sum_{k = 1}^{N}a_{k}\indicator{E_{k}}$ and $\psi = \sum_{k = 1}^{N}b_{k}\indicator{E_{k}}$. 
    Since $\varphi \psi = \sum_{k = 1}^{N}a_{k}b_{k}\indicator{E_{k}}$, combining \eqref{em.TV} and H\"{o}lder's inequality yields
    \begin{align}\label{em-holder.simple}
        \int_{X}\varphi\psi\,d\abs{\Gamma\langle f; g \rangle}
        &= \sum_{k = 1}^{N}a_{k}b_{k}\abs{\Gamma\langle f; g \rangle(E_{k})} \nonumber \\ 
        &\le \Biggl(\sum_{k = 1}^{N}a_{k}^{p/(p - 1)}\Gamma\langle f \rangle(E_{k})\Biggr)^{(p - 1)/p}\Biggl(\sum_{k = 1}^{N}b_{k}^{p}\,\Gamma\langle g \rangle(E_{k})\Biggr)^{1/p} \nonumber \\
        &= \biggl(\int_{X}\varphi^{p/(p - 1)}\,d\Gamma\langle f \rangle\biggr)^{(p - 1)/p}\biggl(\int_{X}\psi^{p}\,d\Gamma\langle g \rangle\biggr)^{1/p}. 
    \end{align} 
    Next, assume that $\varphi$ and $\psi$ are $[0,\infty]$-valued $\SigmaAlgEM$-measurable functions, and for each $w \in \{ \varphi,\psi \}$ let $\{ s_{n, w} \}_{n \in \mathbb{N}}$ be a non-decreasing sequence of non-negative $\SigmaAlgEM$-measurable simple functions such that $\lim_{n\to\infty}s_{n, w}(x) = w(x)$ for any $x \in X$. 
    Then by \eqref{em-holder.simple} we have \eqref{em.holder} with $s_{n, \varphi},s_{n, \psi}$ in place of $\varphi,\psi$ for any $n \in \mathbb{N}$, and letting $n\to\infty$ yields \eqref{em.holder} by the monotone convergence theorem.  
\end{proof}

In the following proposition, we show that integrals of non-negative bounded $\SigmaAlgEM$-measurable functions with respect to $p$-energy measures satisfying \hyperlink{GC-em}{\textup{(GC)$^{\textup{EM}}_{p}$}} give $p$-energy forms on $(X,m)$ that satisfy \ref{GC}. 
\begin{prop}\label{prop:EMGC.intver}
	Assume that $\{ \Gamma\langle f \rangle \}_{f \in \mathcal{F}}$ satisfies \hyperlink{GC-em}{\textup{(GC)$^{\textup{EM}}_{p}$}}. 
	Then for any $\SigmaAlgEM$-measurable function $\varphi \colon X \to [0,\infty)$ with $\norm{\varphi}_{\sup} < \infty$, $(\int_{X}\varphi\,d\Gamma\langle \,\cdot\, \rangle,\mathcal{F})$ is a $p$-energy form on $(X,m)$ satisfying \ref{GC}. 
\end{prop}
\begin{proof}
	Let $n_{1},n_{2} \in \mathbb{N}$, $q_{1} \in (0,p]$, $q_{2} \in [p,\infty]$ and $T = (T_{1},\dots,T_{n_{2}}) \colon \mathbb{R}^{n_{1}} \to \mathbb{R}^{n_{2}}$ satisfy \eqref{GC-cond}, and let $\bm{u} = (u_{1},\dots,u_{n_{1}}) \in \mathcal{F}^{n_{1}}$. 
	Similar to \eqref{GC.sum}, by using the triangle inequality for the $\ell^{q_{2}/p}$-norm and the reverse Minkowski inequality (Proposition \ref{prop.reverse}) for the $\ell^{q_{1}/p}$-norm, we see that for any non-negative $\SigmaAlgEM$-measurable simple function $\varphi$ on $X$, 
	\begin{equation}\label{e:em.int.GC}
		\norm{\left(\left(\int_{X}\varphi\,d\Gamma\langle T_{l}(\bm{u}) \rangle\right)^{1/p}\right)_{l = 1}^{n_{2}}}_{\ell^{q_{2}}} 
    	\le \norm{\left(\left(\int_{X}\varphi\,d\Gamma\langle u_{k} \rangle\right)^{1/p}\right)_{k = 1}^{n_{1}}}_{\ell^{q_{1}}}. 
	\end{equation}
	We can extend \eqref{e:em.int.GC} to any $\SigmaAlgEM$-measurable function $\varphi \colon X \to [0,\infty]$ by taking a non-decreasing sequence of non-negative $\SigmaAlgEM$-measurable simple functions converging pointwise to $\varphi$ and applying the monotone convergence theorem, which completes the proof. 
\end{proof}

The following Fatou-type property is useful. 
\begin{prop}\label{prop.em-int-fatou}
	Assume that $\mathcal{F} \subseteq L^p(X,m)$. 
	Let $\varphi \colon X \to [0,\infty)$ be $\SigmaAlgEM$-measurable and satisfy $\norm{\varphi}_{\sup} < \infty$. 
	If $\{ u_{n} \}_{n \in \mathbb{N}} \subseteq \mathcal{F}$ converges weakly in $(\mathcal{F},\norm{\,\cdot\,}_{\mathcal{E},1})$ to $u \in \mathcal{F}$, then 
	\begin{equation}\label{e:em-int-fatou}
		\int_{X}\varphi\,d\Gamma\langle u \rangle \le \liminf_{n \to \infty}\int_{X}\varphi\,d\Gamma\langle u_{n} \rangle.
	\end{equation}
\end{prop}
\begin{proof}
    Let $\{ u_{n_{k}} \}_{k}$ be a subsequence with $\lim_{k \to \infty}\int_{X}\varphi\,d\Gamma\langle u_{n_{k}} \rangle = \liminf_{n \to \infty}\int_{X}\varphi\,d\Gamma\langle u_{n} \rangle$. 
    By Mazur's lemma (Lemma \ref{lem.mazur}), there exist $N(l) \in \mathbb{N}$ and $\{ \alpha_{l,k} \}_{k = l}^{N(l)} \subseteq [0,1]$ such that $N(l) > l$, $\sum_{k = l}^{N(l)}\alpha_{l,k} = 1$ and $v_{l} \coloneqq \sum_{k = l}^{N(l)}\alpha_{l,k}u_{n_{k}}$ converges to $u$ in $\mathcal{F}$ as $l \to \infty$. 
    We see from the triangle inequality for $\left(\int_{X}\varphi\,d\Gamma\langle \,\cdot\, \rangle\right)^{1/p}$ that 
    \[
    \left(\int_{X}\varphi\,d\Gamma\langle v_{l} \rangle\right)^{1/p} 
    \le \sum_{k = l}^{N(l)}\alpha_{l,k}\left(\int_{X}\varphi\,d\Gamma\langle u_{n_{k}} \rangle\right)^{1/p},  
    \]
    which implies \eqref{e:em-int-fatou} by letting $l \to \infty$.
\end{proof}

\subsection{Extensions of \texorpdfstring{$p$}{p}-energy measures}
Throughout this subsection, we fix a linear subspace $\core$ of $\mathcal{F}$ and assume that $\SigmaAlgEM$ is a $\sigma$-algebra in $X$ and that $\{ \Gamma\langle f \rangle \}_{f \in \core}$ is a family of $p$-energy measures on $(X,\SigmaAlgEM)$ dominated by $(\mathcal{E},\core)$. 
In the following proposition, we extend $\{ \Gamma\langle f \rangle \}_{f \in \core}$ to $f \in \ExtD$, where $\ExtD$ is a linear subspace of $\mathcal{F}$ defined as 
\begin{equation}\label{e:defn.ExtD}
	\ExtD \coloneqq \Bigl\{u \in \mathcal{F} \Bigm| \text{$\lim_{n\to\infty}\mathcal{E}(u-u_{n})=0$ for some $\{u_{n}\}_{n\in\mathbb{N}} \subseteq \core$}\Bigr\}; 
\end{equation} 
note that, if $\mathcal{F} \subseteq L^{p}(X,m)$ and $\mathcal{F}$ is equipped with the norm $\norm{\,\cdot\,}_{\mathcal{E},1}$, then $\overline{\core}^{\mathcal{F}} \subseteq \ExtD$, where the inclusion can be strict in general. 

\begin{prop}\label{prop.em-ext}
	For any $u \in \ExtD$, there exists a unique measure $\Gamma\langle u \rangle$ on $(X,\SigmaAlgEM)$ such that for any $\{ u_{n} \}_{n \in \mathbb{N}} \subseteq \core$ with $\lim_{n \to \infty}\mathcal{E}(u - u_{n}) = 0$ and any $\SigmaAlgEM$-measurable function $\varphi \colon X \to [0,\infty)$ with $\norm{\varphi}_{\sup} < \infty$, 
	\begin{equation}\label{e:em.extension}
		\int_{X}\varphi\,d\Gamma\langle u \rangle
		= \lim_{n \to \infty}\int_{X}\varphi\,d\Gamma\langle u_{n} \rangle, 
	\end{equation}
	and $\Gamma\langle u \rangle$ further satisfies $\Gamma\langle u \rangle(X) \le \mathcal{E}(u)$. 
	Moreover, for each such $\varphi$, $(\int_{X}\varphi\,d\Gamma\langle \,\cdot\, \rangle,\ExtD)$ is a $p$-energy form on $(X,m)$.
\end{prop}
\begin{proof}
	By \ref{EM2} in Definition \ref{defn.em-Cp} and the monotone convergence theorem, for any $\SigmaAlgEM$-measurable function $\varphi \colon X \to [0,\infty]$ and any $u,v \in \core$, 
	\begin{equation}\label{e:intem.triangle}
		\left(\int_{X}\varphi\,d\Gamma\langle u + v \rangle\right)^{1/p} \le \left(\int_{X}\varphi\,d\Gamma\langle u \rangle\right)^{1/p} + \left(\int_{X}\varphi\,d\Gamma\langle v \rangle\right)^{1/p}. 
	\end{equation}
    In the rest of this proof, let $\varphi \colon X \to [0,\infty)$ be $\SigmaAlgEM$-measurable and satisfy $\norm{\varphi}_{\sup} < \infty$. 
    Let $u \in \ExtD$ and $\{ u_{n} \}_{n \in \mathbb{N}} \subseteq \core$ satisfy $\lim_{n \to \infty}\mathcal{E}(u - u_{n}) = 0$. 
    By \eqref{e:intem.triangle}, $\bigl\{ \int_{X}\varphi\,d\Gamma\langle u_{n} \rangle \bigr\}_{n \in \mathbb{N}}$ is a Cauchy sequence in $[0,\infty)$ and $\lim_{n \to \infty}\int_{X}\varphi\,d\Gamma\langle u_{n} \rangle \eqqcolon I_{u}(\varphi)$ is independent of the choice of $\{ u_{n} \}_{n}$. 
    In addition, we have that 
    \begin{equation}\label{e:emint.unif}
    	\abs{\left(\int_{X}\varphi\,d\Gamma\langle u_{n} \rangle\right)^{1/p} - I_{u}(\varphi)^{1/p}}
    	\le \norm{\varphi}_{\sup}^{1/p}\mathcal{E}(u_{n} - u)^{1/p}, 
    \end{equation}
    that $0 \le I_{u}(\varphi) \le \norm{\varphi}_{\sup}\mathcal{E}(u)$ and that $I_{n}$ is linear in the sense that $I_{u}\bigl(\sum_{k = 1}^{N}a_{k}\varphi_{k}\bigr) = \sum_{k = 1}^{N}a_{k}I_{u}(\varphi_{k})$ for any $N \in \mathbb{N}$, $(a_{k})_{k = 1}^{N} \subseteq [0,\infty)$ and $\SigmaAlgEM$-measurable functions $\varphi_{k} \colon X \to [0,\infty)$ with $\norm{\varphi_{k}}_{\sup} < \infty$, $k \in \{ 1,\dots, N \}$.  
    Now we define $\Gamma\langle u \rangle(A) \coloneqq I_{u}(\indicator{A}) \in [0,\infty)$ for $A \in \SigmaAlgEM$, and show that $\Gamma\langle u \rangle$ is a finite measure on $(X,\SigmaAlgEM)$.
    Clearly, $\Gamma\langle u \rangle$ is finitely additive and $\Gamma\langle u \rangle(X) \le \mathcal{E}(u) < \infty$.  
    Let us show the countable additivity of $\Gamma\langle u \rangle$. 
	By \eqref{e:emint.unif}, for any $\varepsilon > 0$ there exists $N_{0} \in \mathbb{N}$ such that $\sup_{A \in \SigmaAlgEM}\abs{\Gamma\langle u \rangle(A)^{1/p} - \Gamma\langle u_{n} \rangle(A)^{1/p}} < \varepsilon$ for any $n \ge N_{0}$. 
    Let $\{ A_{k} \}_{k \in \mathbb{N}} \subseteq \SigmaAlgEM$ be a sequence of disjoint measurable sets, and set $B_{N} \coloneqq \bigcup_{k = N +1}^{\infty}A_{k}$ for each $N \in \mathbb{N}$.
    Then we see that for any $N \in \mathbb{N}$ and any $n \ge N_{0}$, 
    \[
    \abs{\Gamma\langle u \rangle\left(\bigcup_{k \in \mathbb{N}}A_{k}\right) - \sum_{k = 1}^{N}\Gamma\langle u \rangle(A_{k})}^{1/p}
    = \Gamma\langle u \rangle(B_{N})^{1/p} 
        \le \varepsilon + \Gamma\langle u_{n} \rangle(B_{N})^{1/p}, 
    \]
    whence $\lim_{N \to \infty} \abs{\Gamma\langle u \rangle\left(\bigcup_{k \in \mathbb{N}}A_{k}\right) - \sum_{k = 1}^{N}\Gamma\langle u \rangle(A_{k})} = 0$, proving the desired countable additivity. 
    
    Note that $I_{u + v}(\varphi)^{1/p} \le I_{u}(\varphi)^{1/p} + I_{v}(\varphi)^{1/p}$ for any $u,v \in \ExtD$ by \eqref{e:intem.triangle} and the definition of $I_{\bullet}(\varphi)$. 
    This together with the monotone convergence theorem implies the triangle inequality for $\left(\int_{X}\varphi\,d\Gamma\langle \,\cdot\, \rangle\right)^{1/p}$ on $\ExtD$; in particular, $(\int_{X}\varphi\,d\Gamma\langle \,\cdot\, \rangle,\ExtD)$ is a $p$-energy form on $(X,m)$. 
    Next we show \eqref{e:em.extension}.   
    Let $\{ u_{n} \}_{n \in \mathbb{N}} \subseteq \core$ be a sequence satisfying $\lim_{n \to \infty}\mathcal{E}(u - u_{n}) = 0$. 
    By the triangle inequality for $(\int_{X}\varphi\,d\Gamma\langle \,\cdot\, \rangle, \ExtD)$, 
    \[
    \abs{\left(\int_{X}\varphi\,d\Gamma\langle u \rangle\right)^{1/p} - \left(\int_{X}\varphi\,d\Gamma\langle u_{n} \rangle\right)^{1/p}}
    \le \left(\int_{X}\varphi\,d\Gamma\langle u - u_{n} \rangle\right)^{1/p}
    \le \norm{\varphi}_{\sup}^{1/p}\mathcal{E}(u - u_{n})^{1/p}, 
    \]
    which together with \eqref{e:emint.unif} implies \eqref{e:em.extension}; indeed,  
    \begin{align*}
    	&\abs{I_{u}(\varphi)^{1/p} - \left(\int_{X}\varphi\,d\Gamma\langle u \rangle\right)^{1/p}} \\
    	&\le \abs{I_{u}(\varphi)^{1/p} - \left(\int_{X}\varphi\,d\Gamma\langle u_{n} \rangle\right)^{1/p}} + \abs{\left(\int_{X}\varphi\,d\Gamma\langle u_{n} \rangle\right)^{1/p} - \left(\int_{X}\varphi\,d\Gamma\langle u \rangle\right)^{1/p}} \\
    	&\le 2\norm{\varphi}_{\sup}^{1/p}\mathcal{E}(u - u_{n})^{1/p} 
    	\xrightarrow[n \to \infty]{} 0.  \qedhere 
    \end{align*}
\end{proof}

If, in addition, $\{ \Gamma\langle f \rangle \}_{f \in \core}$ satisfies \ref{Cp-em}, then we can easily see that $\{ \Gamma\langle f \rangle \}_{f \in \ExtD}$ also satisfies \ref{Cp-em}.
We record this fact in the following proposition. 
\begin{prop}\label{prop.em-Cp-ext}
	If $\{ \Gamma\langle f \rangle \}_{f \in \core}$ satisfies \ref{Cp-em}, then so does $\{ \Gamma\langle f \rangle \}_{f \in \ExtD}$.  
\end{prop}
\begin{proof}
	This is immediate from \eqref{e:em.extension}.  
\end{proof}

If $\mathcal{F} \subseteq L^p(X,m)$ and $\mathcal{F}$ equipped with $\norm{\,\cdot\,}_{\mathcal{E},1}$ is a Banach space, then $\closure{\core}^{\mathcal{F}} \subseteq \ExtD$ as remarked after the definition of $\ExtD$ in \eqref{e:defn.ExtD}, and \hyperlink{GC-em}{\textup{(GC)$^{\textup{EM}}_{p}$}} for $\{ \Gamma\langle f \rangle \}_{f \in \core}$ also extends to $\{ \Gamma\langle f \rangle \}_{f \in \closure{\core}^{\mathcal{F}}}$ as follows. 
\begin{prop}\label{prop.em-GC-ext}
	Assume that $\mathcal{F} \subseteq L^{p}(X,m)$, that $\mathcal{F}$ equipped with $\norm{\,\cdot\,}_{\mathcal{E},1}$ is a Banach space, that $(\mathcal{E},\core)$ satisfies \ref{GC} and that $\{ \Gamma\langle f \rangle \}_{f \in \core}$ satisfies \hyperlink{GC-em}{\textup{(GC)$^{\textup{EM}}_{p}$}}. 
	Then for any $\SigmaAlgEM$-measurable function $\varphi \colon X \to [0,\infty)$ with $\norm{\varphi}_{\sup} < \infty$, $\bigl(\int_{X}\varphi\,d\Gamma\langle \,\cdot\, \rangle, \closure{\core}^{\mathcal{F}}\bigr)$ is a $p$-energy form on $(X,m)$ satisfying \ref{GC}. 
\end{prop}
\begin{proof}
    Since $(\mathcal{E},\core)$ satisfies \ref{Cp} by \ref{GC} for $(\mathcal{E},\core)$ and Proposition \ref{prop.GC-list}-\ref{GC.Cpsmall},\ref{GC.Cplarge}, so does $\bigl(\mathcal{E},\closure{\core}^{\mathcal{F}}\bigr)$, which together with the completeness of $\bigl(\closure{\core}^{\mathcal{F}},\norm{\,\cdot\,}_{\mathcal{E},1}\bigr)$ guarantees that Lemma \ref{lem.Lpreduce} is applicable to $\bigl(\mathcal{E},\closure{\core}^{\mathcal{F}}\bigr)$.  
	
	Let $n_{1},n_{2} \in \mathbb{N}$, $q_{1} \in (0,p]$, $q_{2} \in [p,\infty]$ and $T = (T_{1},\dots,T_{n_{2}}) \colon \mathbb{R}^{n_{1}} \to \mathbb{R}^{n_{2}}$ satisfy \eqref{GC-cond} in Definition \ref{defn.GC}. 
    Let $\bm{u} = (u_{1},\dots,u_{n_{1}}) \in \bigl(\closure{\core}^{\mathcal{F}}\bigr)^{n_{1}}$.
    For each $k \in \{ 1,\dots,n_{1} \}$, choose $\{ u_{k,n} \}_{n \in \mathbb{N}} \subseteq \core$ so that $\lim_{n \to \infty}\norm{u_{k} - u_{k,n}}_{\mathcal{F},1} = 0$, set $\bm{u}_{n} \coloneqq (u_{1,n},\dots,u_{n_{1},n})$, and let $l \in \{ 1,\dots,n_{2} \}$. 
    Then by \ref{GC} for $(\mathcal{E},\core)$ and \eqref{GC-cond} we have $\{ T_{l}(\bm{u}_{n}) \}_{n \in \mathbb{N}} \subseteq \core$, $\sup_{n \in \mathbb{N}}\mathcal{E}(T_{l}(\bm{u}_{n})) < \infty$ and $\lim_{n \to \infty}\norm{T_{l}(\bm{u}_{n}) - T_{l}(\bm{u})}_{L^{p}(X,m)} = 0$, and therefore Lemma \ref{lem.Lpreduce} applied to $\bigl(\mathcal{E},\closure{\core}^{\mathcal{F}}\bigr)$ implies that $T_{l}(\bm{u}) \in \closure{\core}^{\mathcal{F}}$ and that $\{ T_{l}(\bm{u}_{n}) \}_{n \in \mathbb{N}}$ converges weakly in $\bigl(\closure{\core}^{\mathcal{F}},\norm{\,\cdot\,}_{\mathcal{E},1}\bigr)$ to $T_{l}(\bm{u})$.  
    If $q_{2} < \infty$, then we see from Proposition \ref{prop.em-int-fatou}, which is applicable to $\{ \Gamma\langle f \rangle \}_{f \in \closure{\core}^{\mathcal{F}}}$ by Proposition \ref{prop.em-ext}, and \ref{GC} for $(\int_{X}\varphi\,d\Gamma\langle \,\cdot\, \rangle,\core)$, which is implied by \hyperlink{GC-em}{\textup{(GC)$^{\textup{EM}}_{p}$}} for $\{ \Gamma\langle f \rangle \}_{f \in \core}$ and Proposition \ref{prop:EMGC.intver}, that 
    \begin{align}\label{eq:em-GC-ext}
    	\norm{\Biggl(\biggl(\int_{X}\varphi\,d\Gamma\langle T_{l}(\bm{u}) \rangle\biggr)^{1/p}\Biggr)_{l = 1}^{n_{2}}}_{\ell^{q_{2}}} 
    	&\overset{\eqref{e:em-int-fatou}}{\leq} \Biggl(\sum_{l = 1}^{n_{2}}\liminf_{n \to \infty}\biggl(\int_{X}\varphi\,d\Gamma\langle T_{l}(\bm{u}_{n}) \rangle\biggr)^{q_{2}/p}\Biggr)^{1/q_{2}} \nonumber \\
    	&\leq \liminf_{n \to \infty}\Biggl(\sum_{l = 1}^{n_{2}}\biggl(\int_{X}\varphi\,d\Gamma\langle T_{l}(\bm{u}_{n}) \rangle\biggr)^{q_{2}/p}\Biggr)^{1/q_{2}} \nonumber \\
    	&\overset{\ref{GC}}{\leq} \liminf_{n \to \infty}\Biggl(\sum_{k = 1}^{n_{1}}\biggl(\int_{X}\varphi\,d\Gamma\langle u_{k,n} \rangle\biggr)^{q_{1}/p}\Biggr)^{1/q_{1}} \nonumber \\
    	&\overset{\eqref{e:em.extension}}{=} \norm{\Biggl(\biggl(\int_{X}\varphi\,d\Gamma\langle u_{k} \rangle\biggr)^{1/p}\Biggr)_{k = 1}^{n_{1}}}_{\ell^{q_{1}}}. 
    \end{align} 
    The case of $q_{2} = \infty$ is similar, so $\bigl(\int_{X}\varphi\,d\Gamma\langle \,\cdot\, \rangle, \closure{\core}^{\mathcal{F}}\bigr)$ satisfies \ref{GC}.
\end{proof}

\subsection{Chain rule and strong locality of \texorpdfstring{$p$}{p}-energy measures}
The purpose of this subsection is to formulate chain rules for $p$-energy measures (see Definition \ref{defn.chainrule} below) and deduce some important consequences of them including the strong locality of $p$-energy measures. 
In this subsection, we assume that $X$ is a topological space having a countable open base, that $m$ is a Borel measure on $X$ with $\supp_{X}[m] = X$ (so that $\contfunc(X)$ is canonically considered as a subset of $L^{0}(X,m)$ as noted after \eqref{a:fullRadon}), that $\core$ is a linear subspace of $\mathcal{F} \cap \contfunc(X)$, that $\{ \Gamma\langle f \rangle \}_{f \in \core}$ is a family of $p$-energy measures on $(X,\mathcal{B}(X))$ dominated by $(\mathcal{E},\core)$, and that $\mathcal{F} \subseteq L^{p}(X,m)$.
We equip $\mathcal{F}$ with the norm $\norm{\,\cdot\,}_{\mathcal{E},1}$. 

\begin{defn}[Chain rules\index{chain rule} for $p$-energy measures]\label{defn.chainrule}
	\begin{enumerate}[label=\textup{(\arabic*)},align=left,leftmargin=*,topsep=2pt,parsep=0pt,itemsep=2pt]
		\item\label{it:Cha1} We say that $\{ \Gamma\langle f \rangle \}_{f \in \core}$ satisfies the chain rule \hyperref[it:Cha1]{\textup{(Cha1)}} if and only if for any $u \in \core$ and any $\Phi \in C^{1}(\mathbb{R})$ with $\Phi(0) = 0$, we have $\Phi(u) \in \core$ and 
		\begin{equation}\label{e:Cha1}
			d\Gamma\langle \Phi(u) \rangle
        	= \abs{\Phi'(u)}^{p}\,d\Gamma\langle u \rangle. 
		\end{equation}
		\item\label{it:Cha2} Assume that $\{ \Gamma\langle f \rangle \}_{f \in \core}$ satisfies \ref{Cp-em}. We say that $\{ \Gamma\langle f \rangle \}_{f \in \core}$ satisfies the chain rule \hyperref[it:Cha2]{\textup{(Cha2)}} if and only if for any $n \in \mathbb{N}$, $u \in \core$, $\bm{v} = (v_{1},\dots,v_{n}) \in \core^{n}$, $\Phi \in C^{1}(\mathbb{R})$ and $\Psi \in C^{1}(\mathbb{R}^{n})$ with $\Phi(0) = \Psi(\bm{0}) = 0$, we have $\Phi(u),\Psi(\bm{v}) \in \core$ and 
		\begin{equation}\label{e:Cha2}
			d\Gamma\langle \Phi(u); \Psi(\bm{v}) \rangle
        	= \sum_{k = 1}^{n}\sgn\bigl(\Phi'(u)\bigr)\abs{\Phi'(u)}^{p - 1}\partial_{k}\Psi(\bm{v})\,d\Gamma\langle u; v_{k} \rangle. 
		\end{equation}
	\end{enumerate}
\end{defn}

\begin{thm}\label{thm.chain-basic}
	Assume that $\{ \Gamma\langle f \rangle \}_{f \in \core}$ satisfies \ref{Cp-em} and \hyperref[it:Cha2]{\textup{(Cha2)}}. 
	\begin{enumerate}[label=\textup{(\alph*)},align=left,leftmargin=*,topsep=2pt,parsep=0pt,itemsep=2pt]
		\item\label{it:Cha2-Cha1} $\{ \Gamma\langle f \rangle \}_{f \in \core}$ satisfies \hyperref[it:Cha1]{\textup{(Cha1)}}. 
		\item\label{it:em.leibniz} \textup{(Leibniz rule)\index{Leibniz rule}} For any $u,v,w \in \core$, we have $vw \in \core$ and 
   		\begin{equation}\label{eq:em-leib}
        	d\Gamma\langle u; vw \rangle = v\,d\Gamma\langle u; w \rangle + w\,d\Gamma\langle u; v \rangle.
    	\end{equation}
    	\item\label{it:em.chain-GC} \textup{(\hyperlink{GC-em}{\textup{(GC)$^{\textup{EM}}_{p}$}} for $\mathbb{R}$-valued $T$)} 
		Let $n \in \mathbb{N}$, $q \in (0,p]$ and $T \colon \mathbb{R}^{n} \to \mathbb{R}$ satisfy $T(0) = 0$ and $\abs{T(x)-T(y)} \le \norm{x-y}_{\ell^{q}}$ for any $x, y \in \mathbb{R}^{n}$, and let $\varphi \colon X \to [0,\infty]$ be Borel measurable. 
			\begin{enumerate}[label=\textup{(\arabic*)},align=left,leftmargin=*,topsep=2pt,parsep=0pt,itemsep=2pt]
			\item\label{it:em.chain-GC-C1} If $T \in C^{1}(\mathbb{R}^{n})$, then for any $\bm{u} = (u_{1},\dots,u_{n}) \in \core^{n}$, \textup{($T(\bm{u}) \in \core$ by \hyperref[it:Cha2]{\textup{(Cha2)}}, and)} 
			\begin{equation}\label{e:chain-GC}
				\biggl(\int_{X}\varphi\,d\Gamma\langle T(\bm{u}) \rangle\biggr)^{1/p} 
				\le \norm{\biggl(\int_{X}\varphi\,d\Gamma\langle u_{k} \rangle\biggr)^{1/p}}_{\ell^{q}}. 
			\end{equation} 
			\item\label{it:em.chain-GC-gen} If $\Gamma\langle u \rangle(X) = \mathcal{E}(u)$ for any $u \in \core$, and if $\mathcal{F}$ equipped with $\norm{\,\cdot\,}_{\mathcal{E},1}$ is a Banach space, then for any $\bm{u} = (u_{1},\dots,u_{n}) \in \bigl(\closure{\core}^{\mathcal{F}}\bigr)^{n}$, $T(\bm{u}) \in \closure{\core}^{\mathcal{F}}$ and \eqref{e:chain-GC} holds. 
			\end{enumerate}
    	\item\label{it:em.chain-Markov} Let $\psi \colon \mathbb{R} \to \mathbb{R}$ satisfy $\psi(0) = 0$ and $0 \leq \psi(t) - \psi(s) \leq t - s$ for any $s,t \in \mathbb{R}$ with $s \leq t$, and let $\varphi \colon X \to [0,\infty]$ be Borel measurable. 
			\begin{enumerate}[label=\textup{(\arabic*)},align=left,leftmargin=*,topsep=2pt,parsep=0pt,itemsep=2pt]
			\item\label{it:em.chain-Markov-C1} If $\psi \in C^{1}(\mathbb{R})$, then for any $u,v \in \core$, \textup{($\psi(u - v) \in \core$ by \hyperref[it:Cha1]{\textup{(Cha1)}} from \ref{it:Cha2-Cha1}, and)} 
			\begin{equation}\label{e:chain-Markov}
				\int_{X}\varphi\,d\Gamma\langle u - \psi(u - v) \rangle + \int_{X}\varphi\,d\Gamma\langle v + \psi(u - v) \rangle
				\leq \int_{X}\varphi\,d\Gamma\langle u \rangle + \int_{X}\varphi\,d\Gamma\langle v \rangle.
			\end{equation} 
			\item\label{it:em.chain-Markov-gen} If $\Gamma\langle u \rangle(X) = \mathcal{E}(u)$ for any $u \in \core$, and if $\mathcal{F}$ equipped with $\norm{\,\cdot\,}_{\mathcal{E},1}$ is a Banach space, then for any $u,v \in \closure{\core}^{\mathcal{F}}$, $\psi(u - v) \in \closure{\core}^{\mathcal{F}}$ and \eqref{e:chain-Markov} holds. 
			\end{enumerate}
	\end{enumerate}
\end{thm}
\begin{proof}
	\ref{it:Cha2-Cha1},\ref{it:em.leibniz}: 
	These are immediate from \hyperref[it:Cha2]{\textup{(Cha2)}}.  
	
	\ref{it:em.chain-GC}-\ref{it:em.chain-GC-C1}: 
	Assume $T \in C^{1}(\mathbb{R}^{n})$, and let $\bm{u} = (u_{1},\dots,u_{n}) \in \core^{n}$. 
	It suffices to prove that 
	\begin{equation}\label{e:chain-GC.setwise}
		\Gamma\langle T(\bm{u}) \rangle(A)^{1/p} \le \norm{\bigl(\Gamma\langle u_{k} \rangle(A)^{1/p}\bigr)_{k = 1}^{n}}_{\ell^{q}} \quad \text{for any $A \in \mathcal{B}(X)$;} 
	\end{equation}
	indeed, it is routine to extend \eqref{e:chain-GC.setwise} to \eqref{e:chain-GC} (see the proof of Proposition \ref{prop:EMGC.intver}).
	To show \eqref{e:chain-GC.setwise}, we first construct a good $\mu$-version of $\Upsilon\langle v_{1}; v_{2}\rangle \coloneqq \frac{d\Gamma\langle v_{1}; v_{2} \rangle}{d\mu}$ for each $v_{1},v_{2} \in \{ T(\bm{u}), u_{1},\dots, u_{n} \}$, where $\mu \coloneqq \Gamma\langle T(\bm{u}) \rangle + \sum_{k=1}^{n}\Gamma\langle u_{k} \rangle$.  
	Let $\{ A_{k} \}_{k \in \mathbb{N}}$ be a countable open base for the topology of $X$. 
	Set $A_{k}^{0} \coloneqq X \setminus A_{k}$ and $A_{k}^{1} \coloneqq A_{k}$ for each $k \in \mathbb{N}$, and define 
	\begin{equation}\label{e:defn.cylinder}
		\mathcal{A}_{k} \coloneqq \Biggl\{ \bigcup_{\alpha \in \mathcal{I}}A_{k}^{\alpha} \Biggm| \mathcal{I} \subseteq \{ 0,1 \}^{k} \Biggr\}, \quad k \in \mathbb{N}, 
	\end{equation} 
	where $A_{k}^{\alpha} \coloneqq \bigcap_{i = 1}^{k}A_{k}^{\alpha_{i}}$ for $\alpha = (\alpha_{i})_{i = 1}^{k} \in \{ 0,1 \}^{k}$.
	Note that $\bigcup_{\alpha \in \mathcal{I}}A_{k}^{\alpha} = \emptyset$ if $\mathcal{I} = \emptyset$. 
	Then $\{ \mathcal{A}_{k} \}_{k \in \mathbb{N}}$ is a non-decreasing sequence of $\sigma$-algebras on $X$ with $\bigcup_{k \in \mathbb{N}}\mathcal{A}_{k}$ generating $\mathcal{B}(X)$.  
	Note that $\bigcup_{\alpha \in \{ 0,1 \}^{k}}A_{k}^{\alpha} = X$ and that $A_{k}^{\alpha} \cap A_{k}^{\beta} = \emptyset$ for $\alpha,\beta \in \{ 0,1 \}^{k}$ with $\alpha \neq \beta$. 
	For $v_{1},v_{2} \in \{ T(\bm{u}), u_{1},\dots, u_{n} \}$, $k \in \mathbb{N}$, $\alpha \in \{ 0,1 \}^{k}$, define $\Upsilon_{k}\langle v_{1}; v_{2}\rangle \colon X \to [0,\infty)$ by, for $x \in A_{k}^{\alpha}$, 
	\begin{equation}\label{e:gooddensity}
		\Upsilon_{k}\langle v_{1}; v_{2}\rangle(x) 
		\coloneqq \mu(A_{k}^{\alpha})^{-1}\Gamma\langle v_{1}; v_{2} \rangle(A_{k}^{\alpha}). 
	\end{equation} 
	Then $\mathbb{E}_{\mu}[\Upsilon\langle v_{1}; v_{2}\rangle \mid \mathcal{A}_{k}] = \Upsilon_{k}\langle v_{1}; v_{2}\rangle$ $\mu$-a.e.\ on $X$ and hence $\lim_{k \to \infty}\Upsilon_{k}\langle v_{1}; v_{2}\rangle = \Upsilon\langle v_{1}; v_{2}\rangle$ $\mu$-a.e.\ on $X$ by the martingale convergence theorem (see, e.g., \cite[Theorem 10.5.1]{Dud}) and the fact that $\bigcup_{k \in \mathbb{N}}\mathcal{A}_{k}$ generates $\mathcal{B}(X)$. 
	From this convergence together with \eqref{e:gooddensity} and the H\"{o}lder-type estimate \eqref{em-bdd.form} in Theorem \ref{thm.em-basic}, we obtain 
	\begin{equation}\label{e:holder.density}
		\abs{\frac{d\Gamma\langle v_{1}; v_{2} \rangle}{d\mu}} 
		\le \biggl(\frac{d\Gamma\langle v_{1} \rangle}{d\mu}\biggr)^{\frac{p-1}{p}}\biggl(\frac{d\Gamma\langle v_{2} \rangle}{d\mu}\biggr)^{\frac{1}{p}} \quad \text{$\mu$-a.e.\ on $X$.}
	\end{equation} 
	
	Now we prove \eqref{e:chain-GC.setwise} on the basis of \hyperref[it:Cha2]{\textup{(Cha2)}} and \eqref{e:holder.density}. 
	Recalling that we have assumed $T \in \contfunc^{1}(\mathbb{R}^{n})$, we see from the assumption on $T$ that for any $x, y = (y_{1},\ldots,y_{n}) \in \mathbb{R}^{n}$, 
	\begin{equation}\label{e:GC-operatornorm.special}
		\abs{\sum_{k = 1}^{n}\partial_{k}T(x)y_{k}}
		= \lim_{\varepsilon \downarrow 0}\varepsilon^{-1}\abs{T(x) - T(x + \varepsilon y)}
		\le \norm{y}_{\ell^{q}}. 
	\end{equation}
	Then for any $A \in \mathcal{B}(X)$, from \hyperref[it:Cha2]{\textup{(Cha2)}}, \eqref{e:holder.density}, \eqref{e:GC-operatornorm.special}, H\"{o}lder's inequality, and the triangle inequality for the $L^{p/q}(A,\mu|_{A})$-norm, we obtain 
	\begin{align*}
		\Gamma\langle T(\bm{u}) \rangle(A) 
		&\overset{\hyperref[it:Cha2]{\textup{(Cha2)}}}{=} \int_{A}\sum_{k=1}^{n}\partial_{k}T(\bm{u}(x))\frac{\Gamma\langle T(\bm{u}); u_{k} \rangle}{d\mu}(x)\,\mu(dx) \\ 
		&\overset{\eqref{e:holder.density}}{\le} \int_{A}\sum_{k=1}^{n}\abs{\partial_{k}T(\bm{u}(x))}\biggl(\frac{d\Gamma\langle T(\bm{u}) \rangle}{d\mu}(x)\biggr)^{\frac{p-1}{p}}\biggl(\frac{d\Gamma\langle u_{k} \rangle}{d\mu}(x)\biggr)^{\frac{1}{p}}\,\mu(dx) \\
		&\overset{\eqref{e:GC-operatornorm.special}}{\le} \int_{A}\biggl(\frac{d\Gamma\langle T(\bm{u}) \rangle}{d\mu}(x)\biggr)^{\frac{p-1}{p}}
		\norm{\biggl(\sgn(\partial_{k}T(\bm{u}(x)))\biggl(\frac{d\Gamma\langle u_{k} \rangle}{d\mu}(x)\biggr)^{\frac{1}{p}}\biggr)_{k=1}^{n}}_{\ell^{q}}\,\mu(dx) \\
		&\le \biggl(\int_{A}\frac{d\Gamma\langle T(\bm{u}) \rangle}{d\mu}\,d\mu\biggr)^{\frac{p-1}{p}}\Biggl(\int_{A}\Biggl[\sum_{k=1}^{n}\biggl(\frac{d\Gamma\langle u_{k} \rangle}{d\mu}\biggr)^{\frac{q}{p}}\Biggr]^{\frac{p}{q}}\,d\mu\Biggr)^{\frac{1}{p}} \\
		&\le \Gamma\langle T(\bm{u}) \rangle(A)^{\frac{p-1}{p}}\Biggl[\sum_{k = 1}^{n}\biggl(\int_{A}\frac{d\Gamma\langle u_{k} \rangle}{d\mu}\,d\mu\biggr)^{\frac{q}{p}}\Biggr]^{\frac{1}{q}} \\
		&= \Gamma\langle T(\bm{u}) \rangle(A)^{\frac{p-1}{p}}\norm{\bigl(\Gamma\langle u_{k} \rangle(A)^{1/p}\bigr)_{k = 1}^{n}}_{\ell^{q}},
	\end{align*}
	proving \eqref{e:chain-GC.setwise} and thereby \ref{it:em.chain-GC}-\ref{it:em.chain-GC-C1}. 
	
	\ref{it:em.chain-GC}-\ref{it:em.chain-GC-gen}: 
	Recall that $\{ \Gamma\langle f \rangle \}_{f \in \closure{\core}^{\mathcal{F}}}$ is uniquely defined through \eqref{e:em.extension} by Proposition \ref{prop.em-ext}, and note that the equality $\Gamma\langle u \rangle(X) = \mathcal{E}(u)$ extends from $u \in \core$ to $u \in \closure{\core}^{\mathcal{F}}$ by \eqref{e:em.extension} and the triangle inequality for $\mathcal{E}^{1/p}$, and hence that $\{ \Gamma\langle f \rangle \}_{f \in \closure{\core}^{\mathcal{F}}}$ and $\bigl(\mathcal{E},\closure{\core}^{\mathcal{F}}\bigr)$ satisfy \ref{Cp-em} by Proposition \ref{prop.em-Cp-ext}. 
	In particular, in view of the completeness of $\bigl(\closure{\core}^{\mathcal{F}},\norm{\,\cdot\,}_{\mathcal{E},1}\bigr)$, Lemma \ref{lem.Lpreduce} is applicable to $\bigl(\mathcal{E},\closure{\core}^{\mathcal{F}}\bigr)$.  
	Now, to see $T(\bm{u}) \in \closure{\core}^{\mathcal{F}}$ and \eqref{e:chain-GC.setwise} for $\bm{u} = (u_{1},\ldots,u_{n}) \in \core^{n}$, we will follow an argument in \cite[Proof of Theorem 1.8]{Kuw23+}. Define $j \colon \mathbb{R}^{n} \to \mathbb{R}$ by $j(x) \coloneqq \exp{\bigl(-\frac{1}{1 - \norm{x}^{2}}\bigr)}$ for $\norm{x} \le 1$ and $j(x) \coloneqq 0$ for $\norm{x} > 1$, and set $j_{l}(x) \coloneqq l^{n}j(lx)$ for each $l \in \mathbb{N}$. 
	We further define $T_{l}(x) \coloneqq \int_{\mathbb{R}^{n}}(j_{l}(x - y) - j_{l}(y))T(y)\,dy = \int_{\mathbb{R}^{n}}j_{l}(y)(T(x - y) - T(y))\,dy$ so that $T_{l} \in \contfunc^{\infty}(\mathbb{R}^{n})$, $T_{l}(0) = 0$ and $\lim_{l \to \infty}T_{l}(x) = T(x)$ for any $x \in \mathbb{R}^{n}$. 
	Moreover, for any $x,y \in \mathbb{R}^{n}$, 
	\begin{equation}\label{e:chain-GC.mollifier}
		\abs{T_{l}(x) - T_{l}(y)}
		= \abs{\int_{\mathbb{R}^{n}}j_{l}(z)(T(x - z) - T(y - z))\,dz} 
		\le \norm{x - y}_{\ell^{q}}. 
	\end{equation}
	Therefore, letting $\bm{u} = (u_{1},\ldots,u_{n}) \in \core^{n}$, by \ref{it:em.chain-GC}-\ref{it:em.chain-GC-C1} we have \eqref{e:chain-GC.setwise} with $T_{l}$ in place of $T$, which together with $\mathcal{E}(T_{l}(\bm{u})) = \Gamma\langle T_{l}(\bm{u}) \rangle(X)$ implies that $\sup_{l \in \mathbb{N}}\mathcal{E}(T_{l}(\bm{u})) < \infty$. 
	Since $\{ T_{l}(\bm{u}) \}_{l \in \mathbb{N}}$ converges in $L^{p}(X,m)$ to $T(\bm{u})$ as $l \to \infty$ by $T_{l}(0) = 0$, \eqref{e:chain-GC.mollifier} and the dominated convergence theorem, we conclude from Lemma \ref{lem.Lpreduce} that $T(\bm{u}) \in \closure{\core}^{\mathcal{F}}$ and that $\{ T_{l}(\bm{u}) \}_{l \in \mathbb{N}}$ converges weakly in $\bigl(\closure{\core}^{\mathcal{F}},\norm{\,\cdot\,}_{\mathcal{E},1}\bigr)$ to $T(\bm{u})$. 
	Now we obtain \eqref{e:chain-GC.setwise} by combining Proposition \ref{prop.em-int-fatou} applied to $\{ \Gamma\langle f \rangle \}_{f \in \closure{\core}^{\mathcal{F}}}$ and \eqref{e:chain-GC.setwise} with $T_{l}$ in place of $T$. 
	
	Lastly, let $\bm{u} = (u_{1},\ldots,u_{n}) \in \bigl(\closure{\core}^{\mathcal{F}}\bigr)^{n}$, and choose $\bigl\{ \bm{u}^{(l)} = \bigl(u^{(l)}_{1},\ldots,u^{(l)}_{n}\bigr) \bigr\}_{l \in \mathbb{N}} \subseteq \core^{n}$ so that $\bigl\{ u^{(l)}_{k} \bigr\}_{l \in \mathbb{N}}$ converges in norm in $\mathcal{F}$ to $u_{k}$ for any $k \in \{1,\ldots,n\}$. 
	Then by the result of the previous paragraph we have $\{ T(\bm{u}^{(l)}) \}_{l \in \mathbb{N}} \subseteq \closure{\core}^{\mathcal{F}}$ and \eqref{e:chain-GC.setwise} with $\bm{u}^{(l)}$ in place of $\bm{u}$, which together with $\mathcal{E}(T(\bm{u}^{(l)})) = \Gamma\langle T(\bm{u}^{(l)}) \rangle(X)$ and the assumption on $T$ implies that $\{ T(\bm{u}^{(l)}) \}_{l \in \mathbb{N}}$ is a bounded sequence in $\bigl(\closure{\core}^{\mathcal{F}},\norm{\,\cdot\,}_{\mathcal{E},1}\bigr)$ converging in norm in $L^{p}(X,m)$ to $T(\bm{u})$. 
	Thus $T(\bm{u}) \in \closure{\core}^{\mathcal{F}}$ and $\{ T(\bm{u}^{(l)}) \}_{l \in \mathbb{N}}$ converges weakly in $\bigl(\closure{\core}^{\mathcal{F}},\norm{\,\cdot\,}_{\mathcal{E},1}\bigr)$ to $T(\bm{u})$ by Lemma \ref{lem.Lpreduce}, and hence combining Proposition \ref{prop.em-int-fatou} applied to $\{ \Gamma\langle f \rangle \}_{f \in \closure{\core}^{\mathcal{F}}}$ and \eqref{e:chain-GC.setwise} with $\bm{u}^{(l)}$ in place of $\bm{u}$ yields \eqref{e:chain-GC.setwise} for $\bm{u} = (u_{1},\ldots,u_{n}) \in \bigl(\closure{\core}^{\mathcal{F}}\bigr)^{n}$, proving \ref{it:em.chain-GC}-\ref{it:em.chain-GC-gen}.
	
	\ref{it:em.chain-Markov}-\ref{it:em.chain-Markov-C1}: 
	Assume $\psi \in C^{1}(\mathbb{R})$, and let $u,v \in \core$. 
	Again, in view of the proof of Proposition \ref{prop:EMGC.intver} it suffices to show that 
	\begin{equation}\label{e:chain-Markov.setwise}
		\Gamma\langle u - \psi(u - v) \rangle(A) + \Gamma\langle v + \psi(u - v) \rangle(A)
		\leq \Gamma\langle u \rangle(A) + \Gamma\langle v \rangle(A) \quad \text{for any $A \in \mathcal{B}(X)$.}
	\end{equation} 
	Indeed, since $\int_{X} \varphi\,d\Gamma\langle f; \,\cdot\, \rangle$ is linear for any $f \in \core$ by Theorem \ref{thm.signed} if $\norm{\varphi}_{\sup} < \infty$, we see from \hyperref[it:Cha2]{\textup{(Cha2)}}, Proposition \ref{prop.em-holder}, $0 \leq \psi' \leq 1$ on $\mathbb{R}$, and H\"{o}lder's inequality that for any $A \in \mathcal{B}(X)$, 
	\begin{align*}
		&\Gamma\langle u - \psi(u - v) \rangle(A) + \Gamma\langle v + \psi(u - v) \rangle(A) \\
		&\overset{\hyperref[it:Cha2]{\textup{(Cha2)}}}{=} \Gamma\langle u - \psi(u - v); u \rangle(A) - \int_{A} \psi'(u - v)\,d\Gamma\langle u - \psi(u - v); u - v \rangle \\
		&\qquad + \Gamma\langle v + \psi(u - v); v \rangle(A) + \int_{A} \psi'(u - v)\,d\Gamma\langle v + \psi(u - v); u - v \rangle \\
		&= \int_{A} (1 - \psi'(u-v))\,d\Gamma\langle u - \psi(u - v); u \rangle + \int_{A} \psi'(u - v)\,d\Gamma\langle u - \psi(u - v); v \rangle \\
		&\qquad + \int_{A} (1 - \psi'(u-v))\,d\Gamma\langle v + \psi(u - v); v \rangle + \int_{A} \psi'(u - v)\,d\Gamma\langle v + \psi(u - v); u \rangle \\
		&\overset{\eqref{em.holder}}{\leq} \biggl(\int_{A} (1 - \psi'(u-v))\,d\Gamma\langle u - \psi(u - v) \rangle\biggr)^{\frac{p-1}{p}} \biggl(\int_{A} (1 - \psi'(u-v))\,d\Gamma\langle u \rangle\biggr)^{\frac{1}{p}} \\
		&\qquad + \biggl(\int_{A} \psi'(u - v)\,d\Gamma\langle u - \psi(u - v) \rangle\biggr)^{\frac{p-1}{p}} \biggl(\int_{A} \psi'(u - v)\,d\Gamma\langle v \rangle\biggr)^{\frac{1}{p}} \\
		&\qquad + \biggl(\int_{A} (1 - \psi'(u-v))\,d\Gamma\langle v + \psi(u - v) \rangle\biggr)^{\frac{p-1}{p}} \biggl(\int_{A} (1 - \psi'(u-v))\,d\Gamma\langle v \rangle\biggr)^{\frac{1}{p}} \\
		&\qquad + \biggl(\int_{A} \psi'(u - v)\,d\Gamma\langle v + \psi(u - v) \rangle\biggr)^{\frac{p-1}{p}} \biggl(\int_{A} \psi'(u - v)\,d\Gamma\langle u \rangle\biggr)^{\frac{1}{p}} \\
		&\overset{\text{H\"{o}lder}}{\leq} \bigl(\Gamma\langle u - \psi(u - v) \rangle(A) + \Gamma\langle v + \psi(u - v) \rangle(A)\bigr)^{\frac{p-1}{p}} \bigl(\Gamma\langle u \rangle(A) + \Gamma\langle v \rangle(A)\bigr)^{\frac{1}{p}},
	\end{align*}
	proving \eqref{e:chain-Markov.setwise} and thereby \ref{it:em.chain-Markov}-\ref{it:em.chain-Markov-C1}.  
	
	\ref{it:em.chain-Markov}-\ref{it:em.chain-Markov-gen}: 
	This is proved by following closely the above proof of \ref{it:em.chain-GC}-\ref{it:em.chain-GC-gen} on the basis of \ref{it:em.chain-Markov}-\ref{it:em.chain-Markov-C1} and arguing as in \eqref{eq:em-GC-ext} upon applying Proposition \ref{prop.em-int-fatou} to conclude \eqref{e:chain-Markov.setwise}. 
\end{proof}
 
We also have the following representation formula (see also \cite[Theorem 4.1]{Cap03}).
\begin{prop}[Representation formula\index{representation formula (for $p$-energy measures)}]\label{prop.em-express}
	Assume that $\{ \Gamma\langle f \rangle \}_{f \in \core}$ satisfies \ref{Cp-em} and \hyperref[it:Cha2]{\textup{(Cha2)}} and that $\Gamma\langle f \rangle(X) = \mathcal{E}(f)$ for any $f \in \core$. 
    Then for any $u, \varphi \in \core$,
    \begin{equation}\label{e:em.functional}
    	\int_{X}\varphi\,d\Gamma\langle u \rangle
    	=
    	\mathcal{E}(u; u\varphi) - \biggl(\frac{p - 1}{p}\biggr)^{p - 1}\mathcal{E}\bigl(\abs{u}^{\frac{p}{p - 1}}; \varphi\bigr).
    \end{equation}
\end{prop}
\begin{proof}
	Note that $(\mathcal{E},\core)$ satisfies \ref{Cp-em} by $\Gamma\langle f \rangle(X) = \mathcal{E}(f)$. 
    Define $\Phi \in \contfunc^{1}(\mathbb{R})$ by $\Phi(x) \coloneqq \abs{x}^{\frac{p}{p-1}}$.
    Note that $\Phi'(x) = \frac{p}{p - 1}\sgn(x)\abs{x}^{\frac{1}{p - 1}}$.
    By the Leibniz rule \eqref{eq:em-leib} in Theorem \ref{thm.chain-basic}-\ref{it:em.leibniz} and \hyperref[it:Cha2]{\textup{(Cha2)}}, we see that      
    \begin{align}\label{eq:em-express-proof}
        &\mathcal{E}(u; u\varphi) - \biggl(\frac{p - 1}{p}\biggr)^{p - 1}\mathcal{E}(\Phi(u); \varphi) \nonumber \\
        &= \int_{X}u\,d\Gamma\langle u; \varphi \rangle + \int_{X}\varphi\,d\Gamma\langle u \rangle - \biggl(\frac{p - 1}{p}\biggr)^{p - 1}\int_{X}\sgn\bigl(\Phi'(u)\bigr)\abs{\Phi'(u)}^{p - 1}\,d\Gamma\langle u; \varphi \rangle \nonumber \\
        &= \int_{X}u\,d\Gamma\langle u; \varphi \rangle + \int_{X}\varphi\,d\Gamma\langle u \rangle - \biggl(\frac{p - 1}{p}\biggr)^{p - 1}\biggl(\frac{p}{p - 1}\biggr)^{p - 1}\int_{X}\sgn(u)\abs{u}\,d\Gamma\langle u; \varphi \rangle \nonumber \\
        &= \int_{X}\varphi\,d\Gamma\langle u \rangle, 
    \end{align}
    proving \eqref{e:em.functional}. 
\end{proof}
 
We have the following theorem as a consequence of \hyperref[it:Cha1]{\textup{(Cha1)}}. 
\begin{thm}[Image density property\index{image density property}]\label{thm.EIDP}
	Assume that $(\mathcal{E},\core)$ satisfies \eqref{lipcont} in Proposition \ref{prop.GC-list}-\ref{GC.lip} and \ref{Cp}, that $(\mathcal{F},\norm{\,\cdot\,}_{\mathcal{E},1})$ is a Banach space, and that $\{ \Gamma\langle f \rangle \}_{f \in \core}$ satisfies \hyperref[it:Cha1]{\textup{(Cha1)}}. 
    Then for any $u \in \core$, the Borel measure $\Gamma\langle u \rangle \circ u^{-1}$ on $\mathbb{R}$ defined by $(\Gamma\langle u \rangle \circ u^{-1})(A) \coloneqq \Gamma\langle u \rangle(u^{-1}(A))$, $A \in \mathcal{B}(\mathbb{R})$, is absolutely continuous with respect to the Lebesgue measure on $\mathbb{R}$. 
\end{thm}
\begin{proof}
	This is proved, on the basis of \eqref{e:Cha1} in the definition of \hyperref[it:Cha1]{\textup{(Cha1)}}, in exactly the same way as \cite[Proposition 7.6]{Shi24}, which is a simple adaptation of \cite[Theorem 4.3.8]{CF}, but we present the details because in \cite{Shi24} the underlying topological space $X$ is assumed to be a generalized Sierpi\'{n}ski carpet. 
    It suffices to prove that $(\Gamma\langle u \rangle \circ u^{-1})(F) = 0$ for any $u \in \core$ and any compact subset $F$ of $\mathbb{R}$ such that $\mathscr{L}^{1}(F) = 0$, where $\mathscr{L}^{1}$ denotes the $1$-dimensional Lebesgue measure on $\mathbb{R}$.
    Let $\{ \varphi_{n} \}_{n \in \mathbb{N}} \subseteq \contfunc_{c}(\mathbb{R})$ satisfy $\abs{\varphi_{n}} \le 1$, $\lim_{n \to \infty}\varphi_{n}(x) = \indicator{F}(x)$ for any $x \in \mathbb{R}$ and
    \[
    \int_{0}^{\infty}\varphi_{n}(t)\,dt = \int_{-\infty}^{0}\varphi_{n}(t)\,dt = 0 \quad \text{for any $n \in \mathbb{N}$.}
    \]
    We define $\Phi_{n}(x) \coloneqq \int_{0}^{x}\varphi_{n}(t)\,dt$, $x \in \mathbb{R}$, and $u_{n} \coloneqq \Phi_{n} \circ u$ for any $n \in \mathbb{N}$.
    Then we easily see that $\Phi_{n} \in \contfunc^{1}(\mathbb{R}) \cap \contfunc_{c}(\mathbb{R})$, $\Phi_{n}(0) = 0$, and $\Phi_{n}' = \varphi_{n}$ for any $n \in \mathbb{N}$.
    Also, $\{ u_{n} \}_{n \in \mathbb{N}}$ converges in norm in $L^{p}(X,m)$ to $0$ by the dominated convergence theorem, and by \eqref{lipcont} for $(\mathcal{E},\core)$ we have $\{ u_{n} \}_{n \in \mathbb{N}} \subseteq \core$ and $\sup_{n \in \mathbb{N}}\mathcal{E}(u_{n}) < \infty$. 
    Since \ref{Cp} for $(\mathcal{E},\core)$ yields \ref{Cp} for $\bigl(\mathcal{E},\closure{\core}^{\mathcal{F}}\bigr)$ and $\bigl(\closure{\core}^{\mathcal{F}},\norm{\,\cdot\,}_{\mathcal{E},1}\bigr)$ is complete, Lemma \ref{lem.Lpreduce} is applicable to $\bigl(\mathcal{E},\closure{\core}^{\mathcal{F}}\bigr)$ and implies that $\{ u_{n} \}_{n \in \mathbb{N}}$ converges weakly in $\bigl(\closure{\core}^{\mathcal{F}},\norm{\,\cdot\,}_{\mathcal{E},1}\bigr)$ to $0$.  
    By Mazur's lemma (Lemma \ref{lem.mazur}), there exist $N(l) \in \mathbb{N}$ and $\{ a_{l,k} \}_{k = l}^{N(l)} \subseteq [0,1]$ such that $N(l) > l$, $\sum_{k = l}^{N(l)}a_{l,k} = 1$ and $\sum_{k = l}^{N(l)}a_{l,k}u_{n_{k}}$ converges in norm in $\mathcal{F}$ to $0$ as $l \to \infty$. 
    Let us define $\Psi_{l} \in \contfunc^{1}(\mathbb{R})$ by $\Psi_{l} \coloneqq \sum_{k = l}^{N(l)}a_{l,k}\Phi_{n_k}$. 
    Then $\Psi_{l}(0) = 0$ and $\lim_{l \to \infty}\Psi_{l}'(x) = \indicator{F}(x)$ for any $x \in \mathbb{R}$.
    Furthermore, by Fatou's lemma, \eqref{e:Cha1} from \hyperref[it:Cha1]{\textup{(Cha1)}}, and \ref{EM1} in Definition \ref{defn.em-Cp}, 
    \begin{align*}
        (\Gamma\langle u \rangle \circ u^{-1})(F)
        &= \int_{\mathbb{R}}\lim_{l \to \infty}\abs{\Psi_{l}'(t)}^{p}\,(\Gamma\langle u \rangle \circ u^{-1})(dt) \\
        &\le \liminf_{l \to \infty}\int_{X}\abs{\Psi_{l}'(u(x))}^{p}\,\Gamma\langle u \rangle(dx) \\
        &= \liminf_{l \to \infty}\Gamma\langle \Psi_{l}(u) \rangle(X)
        \le \liminf_{l \to \infty}\mathcal{E}(\Psi_{l}(u)) = 0, 
    \end{align*}
    which completes the proof. 
\end{proof}

The following theorem gives arguably the strongest possible forms of the strong locality of $p$-energy measures.
\begin{thm}[Strong locality\index{strong locality (of $p$-energy measures)} of $p$-energy measures]\label{thm.slocal}
    Assume that $(\mathcal{E},\core)$ satisfies \eqref{lipcont} in Proposition \ref{prop.GC-list}-\ref{GC.lip} and \ref{Cp}, that $(\mathcal{F},\norm{\,\cdot\,}_{\mathcal{E},1})$ is a Banach space, and that $\{ \Gamma\langle f \rangle \}_{f \in \core}$ satisfies \hyperref[it:Cha1]{\textup{(Cha1)}}. 
    Let $u,u_{1},u_{2},v \in \core$, $a,a_{1},a_{2},b \in \mathbb{R}$ and $A \in \mathcal{B}(X)$. 
    \begin{enumerate}[label=\textup{(\alph*)},align=left,leftmargin=*,topsep=2pt,parsep=0pt,itemsep=2pt]
        \item\label{it:SLbasic.em}  If $A \subseteq u^{-1}(a)$, then $\Gamma\langle u \rangle(A) = 0$.
        \item\label{it:SL0.em} If $A \subseteq (u - v)^{-1}(a)$, then $\Gamma\langle u \rangle(A) = \Gamma\langle v \rangle(A)$.
        \item\label{it:SL1.em} If $A \subseteq u_{1}^{-1}(a_{1}) \cup u_{2}^{-1}(a_{2})$, then
			\begin{equation}\label{e:pem-sl1}
				\Gamma\langle u_{1} + u_{2} + v \rangle(A) + \Gamma\langle v \rangle(A) = \Gamma\langle u_{1} + v \rangle(A) + \Gamma\langle u_{2} + v \rangle(A).
			\end{equation}
			If $\{ \Gamma\langle f \rangle \}_{f \in \core}$ satisfies \ref{Cp-em} and $A \subseteq u_{1}^{-1}(a_{1}) \cup u_{2}^{-1}(a_{2})$, then
			\begin{equation}\label{e:pem-sl1-cor}
				\Gamma\langle u_{1} + u_{2}; v \rangle(A) = \Gamma\langle u_{1}; v \rangle(A) + \Gamma\langle u_{2}; v \rangle(A). 
			\end{equation}
        \item\label{it:SL2.em} If $\{ \Gamma\langle f \rangle \}_{f \in \core}$ satisfies \ref{Cp-em} and $A \subseteq (u_1 - u_2)^{-1}(a) \cup v^{-1}(b)$, then
			\begin{equation}\label{e:pem-sl2}
				\Gamma\langle u_{1}; v \rangle(A) = \Gamma\langle u_{2}; v \rangle(A)
				\quad\text{and}\quad
				\Gamma\langle v; u_{1} \rangle(A) = \Gamma\langle v; u_{2} \rangle(A).
			\end{equation}
    \end{enumerate}
\end{thm}
\begin{proof}
	\ref{it:SLbasic.em}: 
	This is immediate from Theorem \ref{thm.EIDP}. 
	
	\ref{it:SL0.em}: 
	This follows from \ref{it:SLbasic.em} and the triangle inequality for $\Gamma\langle \,\cdot\,\rangle(A)^{1/p}$. 
	
	\ref{it:SL1.em}: 
	Set $A_{i} \coloneqq A \cap u_{i}^{-1}(a_{i})$, $i \in \{ 1,2 \}$. 
    We see from \ref{it:SL0.em} that 
    \begin{align*}
        &\Gamma\langle u_{1} + u_{2} + v \rangle(A) + \Gamma\langle v \rangle(A) \\
        &= \Gamma\langle u_{2} + v \rangle(A_{1}) + \Gamma\langle u_{1} + v \rangle(A_{2}) + \Gamma\langle v \rangle(A) \\
        &= \Gamma\langle u_{2} + v \rangle(A_{1}) + \Gamma\langle u_{1} + v \rangle(A_{2}) + \Gamma\langle u_{1} + v \rangle(A_{1}) + \Gamma\langle u_{2} + v \rangle(A_{2}) \\
        &= \Gamma\langle u_{1} + v \rangle(A) + \Gamma\langle u_{2} + v \rangle(A), 
    \end{align*}
    which proves \eqref{e:pem-sl1}.
    Note that $\Gamma\langle u_{1} + u_{2} \rangle(A) = \Gamma\langle u_{1} \rangle(A) + \Gamma\langle u_{2} \rangle(A)$ by \eqref{e:pem-sl1} in the case $v = 0$. 
    Next assume that $\{ \Gamma\langle f \rangle \}_{f \in \core}$ satisfies \ref{Cp-em}.
    By using this equality and applying \eqref{e:pem-sl1} with $v$ replaced by $tv$ for $t \in (0,\infty)$, we have 
    \begin{align*}
    	&\frac{\Gamma\langle u_{1} + u_{2} + tv \rangle(A) - \Gamma\langle u_{1} + u_{2} \rangle(A)}{t} + t^{p - 1}\Gamma\langle v \rangle(A) \\
    	&= \frac{\Gamma\langle u_{1} + tv \rangle(A) - \Gamma\langle u_{1} \rangle(A)}{t} + \frac{\Gamma\langle u_{2} + tv \rangle(A) - \Gamma\langle u_{2} \rangle(A)}{t}, 
    \end{align*}
    which implies \eqref{e:pem-sl1-cor} by letting $t \downarrow 0$. 
    
    \ref{it:SL2.em}: 
	The proof will be very similar to that of Proposition \ref{prop.sl-other}-\ref{it:SL1-SL2}.
	By applying \eqref{e:pem-sl1} with $u_{2} - u_{1}, tv, u_{1}$ for $t \in (0,\infty)$ in place of $u_{1},u_{2},v$, we have 
	\[
	\frac{\Gamma\langle u_{1} + tv \rangle(A) - \Gamma\langle u_{1} \rangle(A)}{t} 
	= \frac{\Gamma\langle u_{2} + tv \rangle(A) - \Gamma\langle u_{2} \rangle(A)}{t}, 
	\]
	which implies the former equality in \eqref{e:pem-sl2} by letting $t \downarrow 0$. 
	This equality in turn with $v,0,u_{1}-u_{2}$ in place of $u_{1},u_{2},v$ yields the latter equality in \eqref{e:pem-sl2} by the linearity of $\Gamma\langle v; \,\cdot\, \rangle(A)$.
\end{proof}

We conclude this section by showing an upper bound for $p$-energy measures of products of functions. Recall \eqref{leibniz} in Definition \ref{prop.GC-list}-\ref{GC.leibniz} for a similar inequality under \ref{GC}.
\begin{prop}\label{prop:leibniz-Cha}
	Assume that $\{ \Gamma\langle f \rangle \}_{f \in \core}$ satisfies \ref{Cp-em} and the Leibniz rule as in Theorem \ref{thm.chain-basic}-\ref{it:em.leibniz}.
	Then for any $u,v \in \core$ and any $A \in \mathcal{B}(X)$, \textup{($uv \in \mathcal{D}$ by the Leibniz rule, and)}
	\begin{equation}\label{eq:leibniz-Cha}
		\Gamma\langle uv \rangle(A)^{1/p} \le \norm{u}_{\sup,A}\Gamma\langle v \rangle(A)^{1/p} + \norm{v}_{\sup,A}\Gamma\langle u \rangle(A)^{1/p}.
	\end{equation}
\end{prop}
\begin{proof}
	By the Leibniz rule \eqref{eq:em-leib} and the H\"{o}lder-type estimate \eqref{em.holder} in Proposition \ref{prop.em-holder}, 
	\begin{align*}
		&\Gamma\langle uv \rangle(A) 
		\overset{\eqref{eq:em-leib}}{=} \int_{A}u\,d\Gamma\langle uv; v \rangle + \int_{A}v\,d\Gamma\langle uv; u \rangle \\
		&\overset{\eqref{em.holder}}{\le} \left(\int_{A}\,d\Gamma\langle uv \rangle\right)^{(p-1)/p}\left(\int_{A}\abs{u}^{p}\,d\Gamma\langle v \rangle\right)^{1/p} + \left(\int_{A}\,d\Gamma\langle uv \rangle\right)^{(p-1)/p}\left(\int_{A}\abs{v}^{p}\,d\Gamma\langle u \rangle\right)^{1/p} \\
		&\le \Gamma\langle uv \rangle(A)^{(p-1)/p}\Bigl[\norm{u}_{\sup,A}\Gamma\langle v \rangle(A)^{1/p} + \norm{v}_{\sup,A}\Gamma\langle u \rangle(A)^{1/p}\Bigr], 
	\end{align*}
	which shows \eqref{eq:leibniz-Cha}. 
\end{proof}

\section{\texorpdfstring{$p$}{p}-Energy measures associated with self-similar \texorpdfstring{$p$}{p}-energy forms}\label{sec.ss}
In this section, we focus on the self-similar case.
We will introduce the self-similarity of $p$-energy forms and construct $p$-energy measures with respect to self-similar $p$-energy forms.
Some fundamental properties of $p$-energy measures will be shown. 
 
\subsection{Self-similar structure and related notions}
We first recall standard notation and terminology on self-similar structures (see \cite[Chapter 1]{Kig01} for example).
Throughout this section, we fix a compact metrizable space $K$, a finite set $S$ with $\#S\geq 2$ and a continuous injective map $F_{i}\colon K\to K$ for each $i\in S$. We set $\mathcal{L} \coloneqq (K,S,\{F_{i}\}_{i\in S})$.
\begin{defn}\label{d:shift}
\begin{enumerate}[label=\textup{(\arabic*)},align=left,leftmargin=*,topsep=5pt,parsep=0pt,itemsep=2pt]
\item\label{it:word} Let $W_{0} \coloneqq \{\emptyset\}$, where $\emptyset$ is an element
	called the \emph{empty word}\index{empty word}, let
	$W_{n} \coloneqq S^{n}=\{w_{1}\dots w_{n}\mid w_{i}\in S\textrm{ for }i\in\{1,\dots,n\}\}$
	for $n\in\mathbb{N}$ and let $W_{\ast} \coloneqq \bigcup_{n\in\mathbb{N}\cup\{0\}}W_{n}$.
	For $w\in W_{\ast}$, the unique $n\in\mathbb{N}\cup\{0\}$ with $w\in W_{n}$
	is denoted by $\lvert w\rvert$ and called the \emph{length of $w$}.\index{length (of a word)} 
	For $w,v \in W_{\ast}$, $w = w_{1} \dots w_{n_{1}}$, $v = v_{1} \dots v_{n_{2}}$, we define $wv \in W_{\ast}$ by $wv \coloneqq w_{1} \dots w_{n_{1}}v_{1} \dots v_{n_{2}} \, (w\emptyset \coloneqq w, \emptyset v \coloneqq v)$.
\item\label{it:shift} We set
	$\Sigma \coloneqq S^{\mathbb{N}}=\{\omega_{1}\omega_{2}\omega_{3}\ldots\mid \omega_{i}\in S\textrm{ for }i\in\mathbb{N}\}$,
	which is always equipped with the product topology of the discrete topology on $S$,
	and define the \emph{shift map}\index{shift map} $\sigma\colon\Sigma\to\Sigma$ by
	$\sigma(\omega_{1}\omega_{2}\omega_{3}\dots)\coloneqq\omega_{2}\omega_{3}\omega_{4}\dots$.
	For $i\in S$ we define $\sigma_{i}\colon\Sigma\to\Sigma$  by
	$\sigma_{i}(\omega_{1}\omega_{2}\omega_{3}\dots) \coloneqq i\omega_{1}\omega_{2}\omega_{3}\dots$.
	For $\omega=\omega_{1}\omega_{2}\omega_{3}\ldots\in\Sigma$ and
	$n\in\mathbb{N}\cup\{0\}$, we write $[\omega]_{n} \coloneqq \omega_{1}\dots\omega_{n}\in W_{n}$.
\item\label{it:FwKw} For $w=w_{1}\dots w_{n}\in W_{\ast}$, we set
	$F_{w} \coloneqq F_{w_{1}}\circ\dots\circ F_{w_{n}}$ ($F_{\emptyset} \coloneqq \id_{K}$),
	$K_{w} \coloneqq F_{w}(K)$, $\sigma_{w} \coloneqq \sigma_{w_{1}}\circ\dots\circ \sigma_{w_{n}}$
	($\sigma_{\emptyset} \coloneqq \id_{\Sigma}$) and $\Sigma_{w} \coloneqq \sigma_{w}(\Sigma)$. 
\item\label{it:partition} A finite subset $\Lambda$ of $W_{\ast}$ is called a \emph{partition}\index{partition (of $\Sigma = S^{\mathbb{N}}$)} of $\Sigma$ if and only if $\Sigma_{w} \cap \Sigma_{v} = \emptyset$ for any $w,v \in \Lambda$ with $w \neq v$ and $\Sigma = \bigcup_{w \in \Lambda}\Sigma_{w}$.
	For partitions $\Lambda_{1},\Lambda_{2}$ of $\Sigma$, we say that $\Lambda_{1}$ is a \emph{refinement}\index{refinement (for partitions of $\Sigma = S^{\mathbb{N}}$)} of $\Lambda_{2}$, and write $\Lambda_{1}\leq\Lambda_{2}$, if and only if for each $w^{1}\in\Lambda_{1}$ there exist $w^{2}\in\Lambda_{2}$ and $\tau \in W_{\ast}$ such that $w^{1} = w^{2}\tau$.
\end{enumerate}
\end{defn}
\begin{defn}\label{d:sss}
$\mathcal{L}=(K,S,\{F_{i}\}_{i\in S})$ is called a \emph{self-similar structure}\index{self-similar structure}
if and only if there exists a continuous surjective map $\chi\colon\Sigma\to K$ such that
$F_{i}\circ\chi=\chi\circ\sigma_{i}$ for any $i\in S$.
Note that such $\chi$, if it exists, is unique and satisfies
$\{\chi(\omega)\}=\bigcap_{n\in\mathbb{N}}K_{[\omega]_{n}}$ for any $\omega\in\Sigma$.
\end{defn}

In the following definition, we recall the definition of \emph{post-critically finite self-similar structures}\index{post-critically finite (p.-c.f.)} introduced by Kigami in \cite{Kig93}, which is mainly dealt with in Subsection \ref{sec.pcf}. 
\begin{defn}\label{d:V0Vstar}
Let $\mathcal{L} = (K,S,\{ F_{i} \}_{i \in S})$ be a self-similar structure.
\begin{enumerate}[label=\textup{(\arabic*)},align=left,leftmargin=*,topsep=5pt,parsep=0pt,itemsep=2pt]
\item\label{it:C-P} We define the \emph{critical set}\index{critical set} $\mathcal{C}_{\mathcal{L}}$ and the
	\emph{post-critical set}\index{post-critical set} $\mathcal{P}_{\mathcal{L}}$ of $\mathcal{L}$ by
	\begin{equation}\label{e:C-P}\textstyle
		\mathcal{C}_{\mathcal{L}}\coloneqq\chi^{-1}\bigl(\bigcup_{i,j\in S,\,i\not=j}K_{i}\cap K_{j}\bigr)
		\qquad\textrm{and}\qquad
		\mathcal{P}_{\mathcal{L}}\coloneqq\bigcup_{n\in\mathbb{N}}\sigma^{n}(\mathcal{C}_{\mathcal{L}}).
	\end{equation}
	$\mathcal{L}$ is called \emph{post-critically finite}, or \emph{p.-c.f.}\ for short,
	if and only if $\mathcal{P}_{\mathcal{L}}$ is a finite set.
\item\label{it:V0Vstar} We set $V_{0}\coloneqq\chi(\mathcal{P}_{\mathcal{L}})$, $V_{n}\coloneqq\bigcup_{w\in W_{n}}F_{w}(V_{0})$
	for $n\in\mathbb{N}$ and $V_{\ast}\coloneqq\bigcup_{n\in \mathbb{N} \cup \{ 0 \}}V_{n}$.
\end{enumerate}
\end{defn}
The set $V_{0}$ should be considered as the \emph{``boundary"} of the self-similar set $K$;
indeed, by \cite[Proposition 1.3.5-(2)]{Kig01}, we have 
\begin{equation}\label{V0bdry}
	\text{$K_{w}\cap K_{v}=F_{w}(V_{0})\cap F_{v}(V_{0})$ for any $w,v\in W_{\ast}$ with
$\Sigma_{w}\cap\Sigma_{v}=\emptyset$.}	
\end{equation}
According to \cite[Lemma 1.3.11]{Kig01}, $V_{n-1}\subseteq V_{n}$ for any $n\in\mathbb{N}$,
and $V_{\ast}$ is dense in $K$ if $V_{0} \neq \emptyset$. 

The family of cells $\{ K_{w} \}_{w \in W_{\ast}}$ describes the local topology of a self-similar structure. 
Indeed, $\{ K_{n,x} \}_{n \ge 0}$, where $K_{n,x} \coloneqq \bigcup_{w \in W_{n}; x \in K_{w}}K_{w}$, forms a fundamental system of neighborhoods of $x \in K$ \cite[Proposition 1.3.6]{Kig01}. 
Moreover, the proof of \cite[Proposition 1.3.6]{Kig01} implies that any metric $d$ on $K$ giving the original topology of $K$ satisfies 
\begin{equation}\label{ss-diam}
    \lim_{n \to \infty}\max_{w \in W_{n}}\diam(K_{w},d) = 0.
\end{equation}

Let us recall the notion of self-similar measure.
\begin{defn}[Self-similar measure]\label{dfn:ssmeas}
    Let $\mathcal{L} = (K,S,\{ F_{i} \}_{i \in S})$ be a self-similar structure and let $(\theta_{i})_{i \in S} \in (0,1)^{S}$ satisfy $\sum_{i \in S}\theta_{i} = 1$.
    A Borel probability measure $m$ on $K$ is said to be a \emph{self-similar measure}\index{self-similar measure} on $\mathcal{L}$ with weight $(\theta_{i})_{i \in S}$ if and only if the following equality (of Borel measures on $K$) holds:
    \begin{equation}\label{ea:ss-meas}
        m = \sum_{i \in S}\theta_{i}(m \circ F_{i}^{-1}), 
    \end{equation}
    where $m \circ F_{i}^{-1}$ denotes the image measure of $m$ by $F_{i}$, i.e., $(m \circ F_{i}^{-1})(A) \coloneqq m(F_{i}^{-1}(A))$ for $A \in \mathcal{B}(K)$. 
\end{defn}

\begin{rmk}\label{rmk:pullback}
	Let $\mathcal{L} = (K,S,\{ F_{i} \}_{i \in S})$ be a self-similar structure, $m$ a self-similar measure on $\mathcal{L}$, and $w \in W_{\ast}$. 
	We then easily see from \eqref{ea:ss-meas} that $u \circ F_{w} = v \circ F_{w}$ $m$-a.e.\ on $K$ for any Borel measurable functions $u,v \colon K \to [-\infty,\infty]$ satisfying $u = v$ $m$-a.e.\ on $K$, thereby that we can define a map $F_{w}^{\ast} \colon L^{0}(K,m) \to L^{0}(K,m)$ by setting $F_{w}^{\ast}u \coloneqq u \circ F_{w}$, and further that $F_{w}^{\ast} \colon L^{p}(K,m) \to L^{p}(K,m)$ is a bounded linear operator for any $p \in [1,\infty]$. 
\end{rmk}

Let us describe a standard way to construct a self-similar measure on a self-similar structure $\mathcal{L} = (K,S,\{ F_{i} \}_{i \in S})$. 
Let $(\theta_{i})_{i \in S} \in (0,1)^{S}$ satisfy $\sum_{i \in S}\theta_{i} = 1$, and let $\nu$ be the Bernoulli measure on $\Sigma$ with weight $(\theta_{i})_{i \in S}$, i.e., the unique Borel measure on $\Sigma$ such that 
\begin{equation}\label{eq:Bernoulli.meas}
\nu(\Sigma_{w}) = \theta_{w} \quad \textrm{for all $w \in W_{\ast}$,}
\end{equation}
where $\theta_{w} \coloneqq \theta_{w_1}\theta_{w_2} \cdots \theta_{w_n}$ for $w = w_{1}w_{2}\ldots w_{n} \in W_{\ast}$ ($\theta_{\emptyset} \coloneqq 1$).
Then the Borel measure $m$ on $K$ given by 
\begin{equation}\label{eq:ssmeas.canonical}
	m \coloneqq \nu \circ \chi^{-1}
\end{equation}
turns out to be a self-similar measure on $\mathcal{L}$ with weight $(\theta_{i})_{i \in S}$; see \cite[Section 1.4]{Kig01} for further details on the self-similar measures obtained in this way. 

The uniqueness of a self-similar measure with given weight appears to be delicate. The following proposition provides a simple sufficient condition for a self-similar measure on $\mathcal{L}$ with weight $(\theta_{i})_{i \in S}$ to assign $K_{w}$ the same mass as in \eqref{eq:Bernoulli.meas} for any $w \in W_{\ast}$ and thereby to be unique. 
\begin{prop}[{\cite[Theorem 1.2.7]{Kig09}}]\label{p:ss-meas}
    Let $\mathcal{L} = (K,S,\{ F_{i} \}_{i \in S})$ be a self-similar structure satisfying $K \neq \closure{V_{0}}^{K}$.
    Then for any $(\theta_{i})_{i \in S} \in (0,1)^{S}$ with $\sum_{i \in S}\theta_{i} = 1$, a self-similar measure $m$ on $\mathcal{L}$ with weight $(\theta_{i})_{i \in S}$ is unique and satisfies $m(K_{w}) = \theta_{w}$ and $m(F_{w}(\closure{V_{0}}^{K})) = 0$ for any $w \in W_{\ast}$.
\end{prop}

\subsection{Self-similar \texorpdfstring{$p$}{p}-energy forms and \texorpdfstring{$p$}{p}-energy measures}
In this subsection, we introduce the notion of self-similar $p$-energy form and define the $p$-energy measures associated with a given self-similar $p$-energy form.  
Throughout this subsection, we fix a self-similar structure $\mathcal{L} = (K,S,\{ F_{i} \}_{i \in S})$ with $K$ connected, a $\sigma$-algebra $\mathcal{B}$ in $K$ including $\mathcal{B}(K)$, a measure $m$ on $(K,\mathcal{B})$ with $\supp_{K}[m] = K$, $p \in (1,\infty)$, and a $p$-energy form $(\mathcal{E},\mathcal{F})$ on $(K,m)$. 

\begin{defn}[Self-similar $p$-energy form\index{self-similar $p$-energy form}]\label{defn.ssform}
    Let $\bm{\rweight} = (\rweight_{i})_{i \in S} \in (0, \infty)^{S}$.
    A $p$-energy form $(\mathcal{E}, \mathcal{F})$ on $(K,m)$ is said to be \emph{self-similar on $(\mathcal{L},m)$ with weight $\bm{\rweight}$} if and only if the following hold:
    \begin{gather}
		\mathcal{F} \cap \contfunc(K) = \{ f \in \contfunc(K) \mid \text{$f \circ F_{i} \in \mathcal{F}$ for any $i \in S$} \}, \label{SSE1}\\
		\mathcal{E}(f) = \sum_{i \in S}\rweight_{i}\mathcal{E}(f \circ F_{i}) \quad \text{for any $f \in \mathcal{F} \cap \contfunc(K)$.} \label{SSE2}
	\end{gather}
    Note that for any partition $\Lambda$ of $\Sigma$, \eqref{SSE2} implies
    \begin{equation}\label{ss.partition}
        \mathcal{E}(f) = \sum_{w \in \Lambda}\rweight_{w}\mathcal{E}(f \circ F_{w}), \quad f \in \mathcal{F} \cap \contfunc(K),
    \end{equation}
    where $\rweight_{w} \coloneqq \rweight_{w_1} \cdots \rweight_{w_n}$ for $w = w_{1} \dots w_{n} \in W_{\ast}$ ($\rweight_{\emptyset} \coloneqq 1$).
    Indeed, \eqref{ss.partition} follows from an induction with respect to $\max_{w \in \Lambda}\abs{w}$.
\end{defn}

In the rest of this subsection, we assume that $(\mathcal{E}, \mathcal{F})$ is a self-similar $p$-energy form on $\mathcal{L}$ with weight $\bm{\rweight} = (\rweight_{i})_{i \in S}$.
We can see that the two-variable version $\mathcal{E}(f;g)$ also has the following self-similarity.
Recall $p$-Clarkson's inequality \ref{Cp} from Definition \ref{d:Cp} and that it implies the existence \eqref{exist-deriva} of the derivative $\mathcal{E}(f; g) \coloneqq \frac{1}{p}\left.\frac{d}{dt}\mathcal{E}(f + tg)\right|_{t = 0} \in \mathbb{R}$.

\begin{prop}\label{prop.ss-pform}
    Assume that $(\mathcal{E},\mathcal{F} \cap \contfunc(K))$ satisfies \ref{Cp}.
    Then
    \begin{equation}\label{ss-pform}
        \mathcal{E}(f; g) = \sum_{i \in S}\rweight_{i}\mathcal{E}(f \circ F_{i}; g \circ F_{i}) \quad \text{for any $f, g \in \mathcal{F} \cap \contfunc(K)$.}
    \end{equation}
\end{prop}
\begin{proof}
    For any $f, g \in \mathcal{F} \cap \contfunc(K)$ and any $t > 0$, we have
    \[
    \frac{\mathcal{E}(f + tg) - \mathcal{E}(f)}{t}
    = \sum_{i \in S}\rweight_{i}\frac{\mathcal{E}\bigl(f \circ F_{i} + t(g \circ F_{i})\bigr) - \mathcal{E}(f \circ F_{i})}{t}.
    \]
    Letting $t \downarrow 0$ yields \eqref{ss-pform}.
\end{proof}

Next we see that $p$-energy measures are naturally introduced by virtue of the self-similarity of $(\mathcal{E},\mathcal{F})$ (see also \cite{Hin05,MS+}).
For $f \in \mathcal{F} \cap \contfunc(K)$, we define a finite measure $\mathfrak{m}_{\mathcal{E}}^{(n)}\langle f \rangle$ on $W_{n} = S^{n}$ by putting $\mathfrak{m}_{\mathcal{E}}^{(n)}\langle f \rangle(\{w\}) \coloneqq \rweight_{w}\mathcal{E}(f \circ F_{w})$ for each $w \in W_{n}$.
Then $\{ \mathfrak{m}_{\mathcal{E}}^{(n)}\langle f \rangle \}_{n \ge 0}$ satisfies the consistency condition by \eqref{ss.partition}, and hence Kolmogorov's extension theorem (see, e.g., \cite[Theorem 12.1.2]{Dud}) guarantees that there exists a unique Borel measure $\mathfrak{m}_{\mathcal{E}}\langle f \rangle$ on $\Sigma = S^{\mathbb{N}}$ such that $\mathfrak{m}_{\mathcal{E}}\langle f \rangle(\Sigma_{w}) = \rweight_{w}\mathcal{E}(f \circ F_{w})$ for any $w \in W_{\ast}$.
In particular, $\mathfrak{m}_{\mathcal{E}}\langle f \rangle(\Sigma) = \mathcal{E}(f)$.
Basic properties of $\mathfrak{m}_{\mathcal{E}}\langle \,\cdot\, \rangle$ are collected in the following proposition.
Recall the generalized $p$-contraction property \ref{GC} from Definition \ref{defn.GC}.
\begin{prop}\label{prop.sspem-pre}
	\begin{enumerate}[label=\textup{(\alph*)},align=left,leftmargin=*,topsep=2pt,parsep=0pt,itemsep=2pt]
		\item\label{it:sspem-pre.GC} Assume that $(\mathcal{E},\mathcal{F} \cap \contfunc(K))$ satisfies \ref{GC}. Then for any $A \in \mathcal{B}(\Sigma)$, $(\mathfrak{m}_{\mathcal{E}}\langle \,\cdot\, \rangle(A),\mathcal{F} \cap \contfunc(K))$ is a $p$-energy form on $(K,m)$ satisfying \ref{GC}. 
		\item\label{it:sspem-pre.signed} Assume that $(\mathcal{E},\mathcal{F} \cap \contfunc(K))$ satisfies \ref{Cp}. Then for any $A \in \mathcal{B}(\Sigma)$, $(\mathfrak{m}_{\mathcal{E}}\langle \,\cdot\, \rangle(A),\mathcal{F} \cap \contfunc(K))$ is a $p$-energy form on $(K,m)$ satisfying \ref{Cp}, and in particular, for any $f,g \in \mathcal{F} \cap \contfunc(K)$, the following derivative exists in $\mathbb{R}$: 
    	\begin{equation}\label{e:defn.emss.pre}
    		\mathfrak{m}_{\mathcal{E}}\langle f; g \rangle(A) \coloneqq \frac{1}{p}\left.\frac{d}{dt}\mathfrak{m}_{\mathcal{E}}\langle f + tg \rangle(A)\right|_{t = 0}. 
    	\end{equation}
    	Moreover, $\mathfrak{m}_{\mathcal{E}}\langle f; g \rangle$ is a signed Borel measure on $\Sigma$. 
	\end{enumerate}
\end{prop}
\begin{proof}
	\ref{it:sspem-pre.GC}: 
	Let $n_{1},n_{2} \in \mathbb{N}$, $q_{1} \in (0,p]$, $q_{2} \in [p,\infty]$ and $T = (T_{1},\dots,T_{n_{2}}) \colon \mathbb{R}^{n_1} \to \mathbb{R}^{n_2}$ satisfy \eqref{GC-cond} in Definition \ref{defn.GC}, and let $\bm{u} = (u_{1},\dots,u_{n_{1}}) \in \bigl(\mathcal{F} \cap \contfunc(K)\bigr)^{n_{1}}$. We are to show that  
    \begin{equation}\label{ss-em.GCpre}
        \norm{\bigl( \mathfrak{m}_{\mathcal{E}}\langle T_{l}(\bm{u}) \rangle(A)^{1/p}\bigr)_{l = 1}^{n_{2}}}_{\ell^{q_{2}}} \le \norm{\bigl(\mathfrak{m}_{\mathcal{E}}\langle u_{k} \rangle(A)^{1/p}\bigr)_{k = 1}^{n_{1}}}_{\ell^{q_{1}}}, \quad A \in \mathcal{B}(\Sigma). 
    \end{equation}
    If $A = \Sigma_{w}$ for some $w \in W_{\ast}$, then \eqref{ss-em.GCpre} is clearly true by \ref{GC} for $(\mathcal{E},\mathcal{F})$. 
    By using the reverse Minkowski inequality on $\ell^{q_{1}/p}$ and the Minkowski inequality on $\ell^{q_{2}/p}$ in the same way as \eqref{GC.sum} in the proof of Proposition \ref{prop.cone-gen}, we obtain \eqref{ss-em.GCpre} for any $A$ belonging to the algebra in $\Sigma$ generated by $\{ \Sigma_{w} \}_{w \in W_{\ast}}$. 
    Hence the monotone class theorem (see, e.g., \cite[Theorem 4.4.2]{Dud}) implies that \eqref{ss-em.GCpre} holds for any $A \in \mathcal{B}(\Sigma)$. 
    
    \ref{it:sspem-pre.signed}: 
    Note that a special case of \eqref{ss-em.GCpre} proves \ref{Cp} for $(\mathfrak{m}_{\mathcal{E}}\langle \,\cdot\, \rangle(A),\mathcal{F} \cap \contfunc(K))$; see also Proposition \ref{prop.GC-list}-\ref{GC.Cpsmall},\ref{GC.Cplarge}. 
    Then the derivative in \eqref{e:defn.emss.pre} exists by Proposition \ref{prop.diffble} and \eqref{ss-em.GCpre}.
    In addition, $\mathfrak{m}_{\mathcal{E}}\langle f; g \rangle$ turns out to be a signed Borel measure on $\Sigma$ by Theorem \ref{thm.signed}. 
    (Even if $(\mathcal{E},\mathcal{F})$ does not satisfy \ref{GC}, the above proof of \ref{it:sspem-pre.GC} together with the triangle inequality for $\mathcal{E}^{1/p}$ shows \eqref{ss-em.GCpre} with $(n_{1},n_{2},q_{1},q_{2}) = (2,1,p,p)$ and $T_{1}(x,y) = x + y$, namely the triangle inequality on $\mathcal{F} \cap \contfunc(K)$ for $\mathfrak{m}_{\mathcal{E}}\langle \,\cdot\, \rangle(A)^{1/p}$.) 
\end{proof}

We now define a family $\{ \Gamma_{\mathcal{E}}\langle f \rangle \}_{f \in \mathcal{F} \cap \contfunc(K)}$ of finite Borel measures on $K$, which we call the \emph{self-similar $p$-energy measures} \index{self-similar $p$-energy measure} associated with $(\mathcal{E}, \mathcal{F})$, by 
\begin{equation}\label{e:defn.em.one}
   \Gamma_{\mathcal{E}}\langle f \rangle(A) \coloneqq (\mathfrak{m}_{\mathcal{E}}\langle f \rangle \circ \chi^{-1})(A) \coloneqq \mathfrak{m}_{\mathcal{E}}\langle f \rangle(\chi^{-1}(A)), \quad A \in \mathcal{B}(K)
\end{equation}
for $f \in \mathcal{F} \cap \contfunc(K)$, where $\chi \colon \Sigma \to K$ is the same map as in Definition \ref{d:sss}.
The following proposition states basic properties and the self-similarity of $\{ \Gamma_{\mathcal{E}}\langle f \rangle \}_{f \in \mathcal{F} \cap \contfunc(K)}$. 
\begin{prop}\label{prop.ss-pform-em} 
    \begin{enumerate}[label=\textup{(\alph*)},align=left,leftmargin=*,topsep=2pt,parsep=0pt,itemsep=2pt]
        \item \label{ssem} $\{ \Gamma_{\mathcal{E}}\langle f \rangle \}_{f \in \mathcal{F} \cap \contfunc(K)}$ satisfies $\Gamma_{\mathcal{E}}\langle f \rangle(K) = \mathcal{E}(f)$ for any $f \in \mathcal{F} \cap \contfunc(K)$, in particular \ref{EM1}, and \ref{EM2}. 
        \item \label{it:ssem-eachcell} For any $f \in \mathcal{F} \cap \contfunc(K)$, any $w \in W_{\ast}$ and any $n \in \mathbb{N} \cup \{ 0 \}$, 
        \begin{equation}\label{e:eachcell}
        	\rweight_{w}\mathcal{E}(f \circ F_w) \le \Gamma_{\mathcal{E}}\langle f \rangle(K_{w}) \le \sum_{v \in W_{n}; K_{v} \cap K_{w} \neq \emptyset}\rweight_{v}\mathcal{E}(f \circ F_{v}).
        \end{equation}
        \item \label{ssem-GC} Assume that $(\mathcal{E},\mathcal{F} \cap \contfunc(K))$ satisfies \ref{GC}, let $\varphi \colon K \to [0,\infty]$ be Borel measurable, and let $n_{1},n_{2} \in \mathbb{N}$, $q_{1} \in (0,p]$, $q_{2} \in [p,\infty]$ and $T = (T_{1},\dots,T_{n_{2}}) \colon \mathbb{R}^{n_1} \to \mathbb{R}^{n_2}$ satisfy \eqref{GC-cond} in Definition \ref{defn.GC}.  
        	Then for any $\bm{u} = (u_{1},\dots,u_{n_{1}}) \in \bigl(\mathcal{F} \cap \contfunc(K)\bigr)^{n_{1}}$, 
        \begin{equation}\label{ss-em.GCint}
            \norm{\left(\left(\int_{K}\varphi\,d\Gamma_{\mathcal{E}}\langle T_{l}(\bm{u}) \rangle\right)^{1/p}\right)_{l = 1}^{n_{2}}}_{\ell^{q_{2}}} \le \norm{\left(\left(\int_{K}\varphi\,d\Gamma_{\mathcal{E}}\langle u_{k} \rangle\right)^{1/p}\right)_{k = 1}^{n_{1}}}_{\ell^{q_{1}}}.
        \end{equation}
        In particular Proposition \ref{prop.GC-list} with $(\int_{K}\varphi\,d\Gamma_{\mathcal{E}}\langle \,\cdot\, \rangle, \mathcal{F} \cap \contfunc(K))$ in place of $(\mathcal{E},\mathcal{F})$ holds provided $\norm{\varphi}_{\sup} < \infty$.  
        \item \label{ssem-ss} The following equality (of Borel measures on $K$) holds:
        \begin{equation}\label{ss-em}
            \Gamma_{\mathcal{E}}\langle f \rangle = \sum_{i \in S}\rweight_{i}(\Gamma_{\mathcal{E}}\langle f \circ F_{i} \rangle \circ F_{i}^{-1}) \quad \text{for any $f \in \mathcal{F} \cap \contfunc(K)$.}
        \end{equation}
        \item \label{ssem-Cp} Assume that $(\mathcal{E},\mathcal{F} \cap \contfunc(K))$ satisfies \ref{Cp}. Then $\{ \Gamma_{\mathcal{E}}\langle f \rangle \}_{f \in \mathcal{F} \cap \contfunc(K)}$ also satisfies \ref{Cp-em}, and the following equality (of signed Borel measures on $K$) holds:
        \begin{equation}\label{ss-emform}
            \Gamma_{\mathcal{E}}\langle f; g \rangle = \sum_{i \in S}\rweight_{i}(\Gamma_{\mathcal{E}}\langle f \circ F_{i}; g \circ F_{i} \rangle \circ F_{i}^{-1}) \quad \text{for any $f,g \in \mathcal{F} \cap \contfunc(K)$.}
        \end{equation}
        \item \label{it:ssem-two-chi} Assume that $(\mathcal{E},\mathcal{F} \cap \contfunc(K))$ satisfies \ref{Cp}. Then $\mathfrak{m}_{\mathcal{E}}\langle f;g \rangle \circ \chi^{-1} = \Gamma_{\mathcal{E}}\langle f;g \rangle$ for any $f,g \in \mathcal{F} \cap \contfunc(K)$. 
    \end{enumerate}
\end{prop}
\begin{proof}
    \ref{ssem}: 
    We easily have $\Gamma_{\mathcal{E}}(K) = \mathfrak{m}_{\mathcal{E}}\langle f \rangle(\chi^{-1}(K)) = \mathfrak{m}_{\mathcal{E}}\langle f \rangle(\Sigma) = \mathcal{E}(f)$. 
    The proof of \ref{EM2} is included in the proof of \ref{ssem-GC} below. 
    
    \ref{it:ssem-eachcell}: 
    This statement is the same as \cite[Lemma 9.15]{MS.long}, which is easily proved by noting that $\Sigma_{w} \subseteq \chi^{-1}(K_{w}) \subseteq \bigcup_{v \in W_{n}; K_{v} \cap K_{w} \neq \emptyset}\Sigma_{v}$. 

    \ref{ssem-GC}:
    Assume that $(\mathcal{E},\mathcal{F})$ satisfies \ref{GC}.
    Let us fix $T = (T_{1},\dots,T_{n_{2}}) \colon \mathbb{R}^{n_1} \to \mathbb{R}^{n_2}$ satisfying \eqref{GC-cond} and $\bm{u} = (u_{1},\dots,u_{n_{1}}) \in \bigl(\mathcal{F} \cap \contfunc(K)\bigr)^{n_{1}}$. 
    For any $B \in \mathcal{B}(K)$, by \ref{GC} for $(\mathfrak{m}_{\mathcal{E}}\langle \,\cdot\, \rangle(\chi^{-1}(B)),\mathcal{F} \cap \contfunc(K))$ (see Proposition \ref{prop.sspem-pre}-\ref{it:sspem-pre.GC}), we obtain 
    \begin{equation}\label{ss-em.GC}
        \norm{\bigl( \Gamma_{\mathcal{E}}\langle T_{l}(\bm{u}) \rangle(B)^{1/p}\bigr)_{l = 1}^{n_{2}}}_{\ell^{q_{2}}} \le \norm{\bigl( \Gamma_{\mathcal{E}}\langle u_{k} \rangle(B)^{1/p}\bigr)_{k = 1}^{n_{1}}}_{\ell^{q_{1}}}.
    \end{equation}
  	Again by the same argument as \eqref{GC.sum} in the proof of Proposition \ref{prop.cone-gen}, we see that \eqref{ss-em.GCint} holds for any non-negative Borel measurable simple function $\varphi$ on $K$. We get the desired extension, \eqref{ss-em.GCint} for any Borel measurable function $\varphi \colon K \to [0,\infty]$, by the monotone convergence theorem. 
  	
  	\ref{ssem-ss}:
  	The proof is very similar to \cite[Proof of Theorem 7.5]{Shi24}. 
  	Let $k \in \mathbb{N}$, $w = w_{1} \dots w_{k} \in W_{k}$ and $n \in \mathbb{N}$. 
  	We see that 
  	\begin{align*}
      	\sum_{i \in S}\rweight_{i}\mathfrak{m}_{\mathcal{E}}\langle f \circ F_{i} \rangle(\sigma_{i}^{-1}(\Sigma_{w}))
      	&= \rweight_{w_{1}}\mathfrak{m}_{\mathcal{E}}\langle f \circ F_{w_{1}} \rangle(\sigma_{w_{1}}^{-1}(\Sigma_{w})) 
      	= \rweight_{w_{1}}\mathfrak{m}_{\mathcal{E}}\langle f \circ F_{w_{1}} \rangle(\Sigma_{w_{2} \dots w_{k}}) \\ 
      	&= \rweight_{w_{1}}\rweight_{w_{2} \dots w_{k}}\mathcal{E}((f \circ F_{w_1}) \circ F_{w_{2} \dots w_{k}})
      	= \mathfrak{m}_{\mathcal{E}}\langle f \rangle(\Sigma_{w})
  	\end{align*}
  	Since $w \in W_{\ast}$ is arbitrary, by Dynkin's $\pi$-$\lambda$ theorem, we deduce that
  	\[
  	\mathfrak{m}_{\mathcal{E}}\langle f \rangle(A)
  	= \sum_{i \in S}\rweight_{i}(\mathfrak{m}_{\mathcal{E}}\langle f \circ F_{i} \rangle \circ \sigma_{i}^{-1})(A), \quad A \in \mathcal{B}(\Sigma). 
  	\]
  	We obtain \eqref{ss-em} by $\chi \circ \sigma_{i} = F_{i} \circ \chi$.  
    
    \ref{ssem-Cp}:
    Assume that $(\mathcal{E},\mathcal{F})$ satisfies \ref{Cp}.
    Then $\{ \Gamma_{\mathcal{E}}\langle f \rangle \}_{f \in \mathcal{F} \cap \contfunc(K)}$ satisfies \ref{Cp-em} by \eqref{ss-em.GC} (see also Proposition \ref{prop.GC-list}-\ref{GC.Cpsmall},\ref{GC.Cplarge}).
    Now we obtain \eqref{ss-emform} by letting $t \downarrow 0$ in
    \[
    \Gamma_{\mathcal{E}}\langle f + tg \rangle(A) = \sum_{i \in S}\rweight_{i}\Gamma_{\mathcal{E}}\langle f \circ F_{i} +  t(g \circ F_{i}) \rangle\bigl(F_{i}^{-1}(A)\bigr). 
    \]
    
    \ref{it:ssem-two-chi}: 
    This is immediate from \eqref{e:defn.em.one}, the definition \eqref{e:defn.emss} of $\Gamma_{\mathcal{E}}\langle f;g \rangle$ and that \eqref{e:defn.emss.pre} of $\mathfrak{m}_{\mathcal{E}}\langle f; g \rangle$. 
\end{proof}

We next prove the chain rules \hyperref[it:Cha1]{\textup{(Cha1)}} and \hyperref[it:Cha2]{\textup{(Cha2)}} (recall Definition \ref{defn.chainrule}) for $\Gamma_{\mathcal{E}}\langle \,\cdot\,\rangle$. 
Such chain rules have been obtained also in \cite{BV05}, but we provide here self-contained proofs because our present framework is different from that of \cite{BV05} and our version \hyperref[it:Cha2]{\textup{(Cha2)}} is stronger than the chain rule proved in \cite{BV05}. 
\begin{thm}\label{thm.em-chain-one}
    Assume that $\mathbb{R}\indicator{K} \subseteq \mathcal{E}^{-1}(0)$ and that $(\mathcal{E},\mathcal{F} \cap \contfunc(K))$ satisfies \eqref{lipcont} in Proposition \ref{prop.GC-list}-\ref{GC.lip}. 
    Then $\{ \Gamma_{\mathcal{E}}\langle f \rangle \}_{f \in \mathcal{F} \cap \contfunc(K)}$ satisfies \hyperref[it:Cha1]{\textup{(Cha1)}}, i.e., for any $u \in \mathcal{F} \cap \contfunc(K)$ and any $\Phi \in C^{1}(\mathbb{R})$, we have $\Phi(u) \in \mathcal{F} \cap \contfunc(K)$ and
    \begin{equation}\label{e:chain.one}
        d\Gamma_{\mathcal{E}}\langle \Phi(u) \rangle
        = \abs{\Phi'(u)}^{p}\,d\Gamma_{\mathcal{E}}\langle u \rangle.
    \end{equation}
\end{thm}
\begin{proof}
	First, let us observe a few consequences of \eqref{lipcont}. 
	The proof of Corollary \ref{cor.lip-useful}-\ref{GC.compos} works even if we assume \eqref{lipcont} instead of \ref{GC}, so we have \eqref{compos}.
    We then obtain $\Phi(u) \in \mathcal{F} \cap \contfunc(K)$ by \eqref{compos} and $\mathbb{R}\indicator{K} \subseteq \mathcal{E}^{-1}(0)$. 
    Also, by \eqref{compos}, the identity $ab = \frac{1}{4}\bigl[(a+b)^2 - (a-b)^2\bigr]$ for $a,b \in \mathbb{R}$, and the triangle inequality for $\mathcal{E}^{1/p}$, there exists a constant $c_{p} \in (0,\infty)$ depending only on $p$ such that for any $u,v \in \mathcal{F} \cap \contfunc(K)$, 
    \begin{equation}\label{e:weak.leibniz}
    	uv \in \mathcal{F} \cap \contfunc(K) \quad \text{and} \quad \mathcal{E}(uv) \le c_{p}\bigl( \norm{v}_{\sup}^{p} \mathcal{E}(u) + \norm{u}_{\sup}^{p} \mathcal{E}(v)\bigr); 
    \end{equation} 
    indeed, \eqref{e:weak.leibniz} for $u,v \in \mathcal{F} \cap \contfunc(K)$ with $\norm{u}_{\sup} = \norm{v}_{\sup} = 1$ is easily verified, and this special case applied to $\norm{u}_{\sup}^{-1}u,\norm{v}_{\sup}^{-1}v$ yields \eqref{e:weak.leibniz} for general $u,v \in \mathcal{F} \cap \contfunc(K)$. 
    
    Next we will prove that
    \begin{equation}\label{stp-chain.Cha1}
        \lim_{l \to \infty}\abs{\rweight_{w}\mathcal{E}\bigl(\Phi(u \circ F_{w})\bigr) - \mathcal{S}_{l}^{(1)}(w)} = 0 \quad \text{for any $w \in W_{\ast}$,}
    \end{equation}
    where for $w \in W_{\ast}$ and $l \in \mathbb{N} \cup \{ 0 \}$ we set, with an arbitrarily fixed $x_{0} \in K$,  
    \begin{equation}\label{stp-chain.Cha1-S1lw}
    \mathcal{S}_{l}^{(1)}(w) \coloneqq \sum_{\tau \in W_{l}}\rweight_{w\tau}\mathcal{E}\bigl(\Phi'\bigl(u(F_{w\tau}(x_0))\bigr) \cdot (u \circ F_{w\tau})\bigr).
    \end{equation}
    We need some preparations to prove \eqref{stp-chain.Cha1}. 
    Note that, for any $z \in W_{\ast}$ and any $x \in K$, 
    \begin{align*}
        &\Phi\bigl(u(F_{z}(x))\bigr) - \Phi\bigl(u(F_{z}(x_0))\bigr) \\
        &= \bigl[u(F_{z}(x)) - u(F_{z}(x_0))\bigr]\Biggl(\Phi'\bigl(u(F_{z}(x_0))\bigr) \\
        &\hspace*{30pt}+ \int_{0}^{1}\Bigl[\Phi'\bigl(u(F_{z}(x_0)) + t\bigl(u(F_{z}(x)) - u(F_{z}(x_0))\bigr)\bigr) - \Phi'(u(F_{z}(x_0)))\Bigr]\,dt\Biggr).
    \end{align*}
    In particular,
    \begin{equation}\label{e:diff.expression}
        \Phi(u \circ F_{z}) - \widehat{u}_{z}
        = \Phi\bigl(u(F_{z}(x_0)\bigr) - \Phi'\bigl(u(F_{z}(x_0))\bigr)u(F_{z}(x_0))
        + D_{z}I_{z},
    \end{equation}
    where $\widehat{u}_{z}, D_{z}, I_{z} \in \contfunc(K)$ are given by 
    \begin{align*}
        \widehat{u}_{z}(x) &\coloneqq \Phi'\bigl(u(F_{z}(x_0))\bigr) \cdot (u \circ F_{z})(x), \\
        D_{z}(x) &\coloneqq u(F_{z}(x)) -  u(F_{z}(x_0)), \\
        I_{z}(x) &\coloneqq \int_{0}^{1}\Bigl[\Phi'\bigl(u(F_{z}(x_0)) + tD_{z}(x)\bigr) - \Phi'\bigl(u(F_{z}(x_0))\bigr)\Bigr]\,dt, \quad x \in K.
    \end{align*}  
    Note that $\widehat{u}_{z} \in \mathcal{F}$ by the self-similarity \eqref{SSE1} of $\mathcal{F}$. 
    By \eqref{compos}, we have that $I_{z} \in \mathcal{F}$ and that there exists a constant $C_{u,\Phi} \in (0,\infty)$ depending only on $p,\norm{u}_{\sup},\norm{\Phi'}_{\sup,[-2\norm{u}_{\sup},2\norm{u}_{\sup}]}$ such that $\mathcal{E}(I_{z}) \le C_{u,\Phi}\mathcal{E}(u \circ F_{z})$ and $\mathcal{E}\bigl(\Phi(u \circ F_{z})\bigr) \le C_{u,\Phi}\mathcal{E}(u \circ F_{z})$.
    Therefore, for any $w \in W_{\ast}$ and any $l \in \mathbb{N} \cup \{ 0 \}$,
    \begin{align}\label{e:chain.diff.piecewise}
        &\sum_{\tau \in W_{l}}\rweight_{w\tau}\mathcal{E}\bigl(\Phi(u \circ F_{w\tau}) - \widehat{u}_{w\tau}\bigr) 
            \overset{\eqref{e:diff.expression}}{=} \sum_{\tau \in W_{l}}\rweight_{w\tau}\mathcal{E}(D_{w\tau}I_{w\tau}) \nonumber \\
        &\overset{\eqref{e:weak.leibniz}}{\le} c_{p}\sum_{\tau \in W_{l}}\rweight_{w\tau}\bigl( \norm{I_{w\tau}}_{\sup}^{p} \mathcal{E}(D_{w\tau}) + \norm{D_{w\tau}}_{\sup}^{p} \mathcal{E}(I_{w\tau}) \bigr) \nonumber \\
        &\le c_{p}\biggl(\max_{\tau \in W_{l}}\norm{I_{w\tau}}_{\sup} + \max_{\tau \in W_{l}}\norm{D_{w\tau}}_{\sup}\biggr)^{p}\sum_{\tau \in W_{l}}\rweight_{w\tau}\Bigl(\mathcal{E}(D_{w\tau}) + C_{u, \Phi}\mathcal{E}(u \circ F_{w\tau})\Bigr) \nonumber \\
        &\le c_{p}(1 + C_{u,\Phi})\mathcal{E}(u)\biggl(\max_{\tau \in W_{l}}\norm{I_{w\tau}}_{\sup} + \max_{\tau \in W_{l}}\norm{D_{w\tau}}_{\sup}\biggr)^{p}.
    \end{align}
    Since $u$ and $\Phi'$ are uniformly continuous on $K$ and $\lim_{l \to \infty}\max_{\tau \in W_{l}}\diam(K_{w\tau},d) = 0$ by \eqref{ss-diam}, both $\max_{\tau \in W_{l}}\norm{I_{w\tau}}_{\sup}$ and $\max_{\tau \in W_{l}}\norm{D_{w\tau}}_{\sup}$ converge to $0$ as $l \to \infty$, and hence by \eqref{e:chain.diff.piecewise} we get 
    \begin{equation}\label{Rsum.u}
        \lim_{l \to \infty}\sum_{\tau \in W_{l}}\rweight_{w\tau}\mathcal{E}\bigl(\Phi(u \circ F_{w\tau}) - \widehat{u}_{w\tau}\bigr) = 0. 
    \end{equation} 
    Now, it follows from \eqref{SSE2}, \eqref{stp-chain.Cha1-S1lw}, \eqref{e:preHolder} with $t=1$, H\"{o}lder's inequality and \eqref{Rsum.u} that for any $w \in W_{\ast}$ and any $l \in \mathbb{N} \cup \{ 0 \}$,
    \begin{align*}
        &\abs{\rweight_{w}\mathcal{E}\bigl(\Phi(u \circ F_{w})\bigr) - \mathcal{S}_{l}^{(1)}(w)}
            \overset{\eqref{SSE2},\eqref{stp-chain.Cha1-S1lw}}{=} \abs{\sum_{\tau \in W_{l}} \rweight_{w\tau} \bigl( \mathcal{E}\bigl(\Phi(u \circ F_{w\tau})\bigr) - \mathcal{E}(\widehat{u}_{w\tau}) \bigr) } \\
        &\overset{\eqref{e:preHolder}}{\leq} 2^{p-1} p \sum_{\tau \in W_{l}} \rweight_{w\tau} \Bigl( \mathcal{E}\bigl(\Phi(u \circ F_{w\tau})\bigr)^{\frac{p-1}{p}}\mathcal{E}\bigl( \Phi(u \circ F_{w\tau}) - \widehat{u}_{w\tau} \bigr)^{\frac{1}{p}} + \mathcal{E}\bigl( \Phi(u \circ F_{w\tau}) - \widehat{u}_{w\tau} \bigr) \Bigr) \\
        &\overset{\textrm{H\"{o}lder}}{\leq} 2^{p-1} p \biggl(\sum_{\tau \in W_{l}} \rweight_{w\tau} \mathcal{E}\bigl(\Phi(u \circ F_{w\tau})\bigr) \biggr)^{(p-1)/p} \biggl(\sum_{\tau \in W_{l}} \rweight_{w\tau} \mathcal{E}\bigl( \Phi(u \circ F_{w\tau}) - \widehat{u}_{w\tau} \bigr) \biggr)^{1/p} \\
        &\qquad\qquad + 2^{p-1} p \sum_{\tau \in W_{l}} \rweight_{w\tau} \mathcal{E}\bigl( \Phi(u \circ F_{w\tau}) - \widehat{u}_{w\tau} \bigr) \\
        &\overset{\eqref{SSE2}}{=} 2^{p-1} p \Bigl(\rweight_{w}\mathcal{E}\bigl(\Phi(u \circ F_{w})\bigr)\Bigr)^{(p-1)/p} \biggl(\sum_{\tau \in W_{l}} \rweight_{w\tau} \mathcal{E}\bigl( \Phi(u \circ F_{w\tau}) - \widehat{u}_{w\tau} \bigr) \biggr)^{1/p} \\
        &\qquad\qquad + 2^{p-1} p \sum_{\tau \in W_{l}} \rweight_{w\tau} \mathcal{E}\bigl( \Phi(u \circ F_{w\tau}) - \widehat{u}_{w\tau} \bigr) \\
        &\xrightarrow[l \to \infty]{\eqref{Rsum.u}}0,
    \end{align*}
    proving \eqref{stp-chain.Cha1}. 

	By the uniform continuity of $\Phi'$ and the fact that $\mathfrak{m}_{\mathcal{E}}\langle f \rangle(\Sigma_{w}) = \rweight_{w}\mathcal{E}(f \circ F_{w})$ for any $f \in \mathcal{F} \cap \contfunc(K)$ and any $w \in W_{\ast}$, we easily observe that 
	\[
    \lim_{l \to \infty}\abs{\sum_{k = 1}^{n}\int_{\Sigma_{w}}\abs{\Phi'(u \circ \chi)}^{p}\,d\mathfrak{m}_{\mathcal{E}}\langle u \rangle - \mathcal{S}_{l}^{(1)}(w)} = 0.   
    \]
    Hence, by \eqref{stp-chain.Cha1} and Dynkin's $\pi$-$\lambda$ theorem, 
	\begin{equation}\label{chain-symbol.Cha1}
        d\mathfrak{m}_{\mathcal{E}}\langle \Phi(u) \rangle
        = \sum_{k = 1}^{n}\abs{\Phi'(u \circ \chi)}^{p}\,d\mathfrak{m}_{\mathcal{E}}\langle u \rangle.
    \end{equation}
    Then we obtain the desired equality \eqref{e:chain.one} by \eqref{chain-symbol.Cha1} and Proposition \ref{prop.ss-pform-em}-\ref{it:ssem-two-chi}. 
\end{proof}

To prove \hyperref[it:Cha2]{\textup{(Cha2)}}, in addition to \ref{Cp}, we need to assume the closedness of $(\mathcal{E},\mathcal{F} \cap L^{p}(K,m))$ in $L^{p}(K,m)$.
Recall the definition \eqref{e:defn-e1norm} of the norm $\norm{\,\cdot\,}_{\mathcal{E},1}$, which we here define on $\mathcal{F} \cap L^{p}(K,m)$ without assuming that $\mathcal{F} \subseteq L^{p}(K,m)$. 
\begin{thm}[Chain rule\index{chain rule}]\label{thm.em-chain}
    Assume that $\mathbb{R}\indicator{K} \subseteq \mathcal{E}^{-1}(0)$, that $(\mathcal{E},\mathcal{F} \cap \contfunc(K))$ satisfies \eqref{lipcont} in Proposition \ref{prop.GC-list}-\ref{GC.lip} and \ref{Cp}, and that $(\mathcal{F} \cap L^{p}(K,m),\norm{\,\cdot\,}_{\mathcal{E},1})$ is a Banach space. 
    Then $\{ \Gamma_{\mathcal{E}}\langle f \rangle \}_{f \in \mathcal{F} \cap \contfunc(K)}$ satisfies \hyperref[it:Cha2]{\textup{(Cha2)}}, i.e., for any $n \in \mathbb{N}$, $u \in \mathcal{F} \cap \contfunc(K)$, $\bm{v} = (v_{1},\ldots,v_{n}) \in (\mathcal{F} \cap \contfunc(K))^{n}$, $\Phi \in C^{1}(\mathbb{R})$ and $\Psi \in C^{1}(\mathbb{R}^{n})$, we have $\Phi(u), \Psi(\bm{v}) \in \mathcal{F} \cap \contfunc(K)$ and
    \begin{equation}\label{chain}
        d\Gamma_{\mathcal{E}}\langle \Phi(u); \Psi(\bm{v}) \rangle
        = \sum_{k = 1}^{n}\sgn\bigl(\Phi'(u)\bigr)\abs{\Phi'(u)}^{p - 1}\partial_{k}\Psi(\bm{v})\,d\Gamma_{\mathcal{E}}\langle u; v_{k} \rangle.
    \end{equation}
\end{thm}
\begin{proof} 
	Let $n \in \mathbb{N}$, $u \in \mathcal{F} \cap \contfunc(K)$, $\bm{v} = (v_{1},\ldots,v_{n}) \in (\mathcal{F} \cap \contfunc(K))^{n}$, $\Phi \in C^{1}(\mathbb{R})$ and $\Psi \in C^{1}(\mathbb{R}^{n})$, so that $\Phi(u) \in \mathcal{F} \cap \contfunc(K)$ as observed at the beginning of the proof of Theorem \ref{thm.em-chain-one}. 
	\emph{We fix these $n,u,\bm{v} = (v_{1},\ldots,v_{n}),\Phi,\Psi$ throughout this proof}, and first, under the additional assumption that $\Psi(\bm{v}) \in \mathcal{F} \cap \contfunc(K)$, we will prove that 
    \begin{equation}\label{stp-chain}
        \lim_{l \to \infty}\abs{\rweight_{w}\mathcal{E}\bigl(\Phi(u \circ F_{w}); \Psi(\bm{v} \circ F_{w})\bigr) - \mathcal{S}_{l}^{(2)}(w)} = 0 \quad \text{for any $w \in W_{\ast}$,}
    \end{equation}
    where for $w \in W_{\ast}$ and $l \in \mathbb{N} \cup \{ 0 \}$ we set, with an arbitrarily fixed $x_{0} \in K$,  
    \[
    \mathcal{S}_{l}^{(2)}(w) \coloneqq \sum_{\tau \in W_{l}}\rweight_{w\tau}\mathcal{E}\biggl(\Phi'(u \circ F_{w\tau}(x_0)) \cdot (u \circ F_{w\tau}); \sum_{k = 1}^{n}\partial_{k}\Psi(v \circ F_{w\tau}(x_0)) \cdot (v_{k} \circ F_{w\tau})\biggr). 
    \]
    To prove \eqref{stp-chain}, we observe that $\abs{\rweight_{w}\mathcal{E}\bigl(\Phi(u \circ F_{w}); \Psi(v \circ F_{w})\bigr) - \mathcal{S}_{l}^{(2)}(w)} \le A_{1,l} + A_{2,l}$, where 
    \begin{align*}
    	\widehat{u}_{z}(x) &\coloneqq \Phi'\bigl(u(F_{z}(x_0))\bigr) \cdot (u \circ F_{z})(x), \\
    	\widehat{v}_{z}(x) &\coloneqq \sum_{k = 1}^{n}\partial_{k}\Psi\bigl(\bm{v}(F_{z}(x_0))\bigr) \cdot (v_{k} \circ F_{z})(x) \quad \text{for } z \in W_{\ast}, x \in K, \\
        A_{1,l} &\coloneqq \sum_{\tau \in W_{l}}\rweight_{w\tau} \abs{\mathcal{E}\bigl(\Phi(u \circ F_{w\tau}); \Psi(\bm{v} \circ F_{w\tau})\bigr) - \mathcal{E}\bigl(\Phi(u \circ F_{w\tau}); \widehat{v}_{w\tau}\bigr)}, \\
        A_{2,l} &\coloneqq \sum_{\tau \in W_{l}}\rweight_{w\tau}\abs{\mathcal{E}\bigl(\Phi(u \circ F_{w\tau}); \widehat{v}_{w\tau}\bigr) - \mathcal{E}\bigl(\widehat{u}_{w\tau}; \widehat{v}_{w\tau}\bigr)}.
    \end{align*}
    Similar to \eqref{Rsum.u}, we can show that 
    \begin{equation}\label{Rsum.v}
        \lim_{l \to \infty}\sum_{\tau \in W_{l}}\rweight_{w\tau}\mathcal{E}\bigl(\Psi(\bm{v} \circ F_{w\tau}) - \widehat{v}_{w\tau}\bigr) = 0. 
    \end{equation} 
    By the H\"{o}lder-type estimate \eqref{bdd.form} and the continuity estimate \eqref{ncont} from Theorem \ref{thm.p-form} and H\"{o}lder's inequality, we have 
    \[
    A_{1,l} \lesssim \mathcal{E}(u \circ F_{w})^{(p - 1)/p}\left(\sum_{\tau \in W_{l}}\rweight_{w\tau}\mathcal{E}(\Psi(\bm{v} \circ F_{w\tau}) - \widehat{v}_{w\tau})\right)^{1/p},
    \]
    and
    \begin{align*}
        A_{2,l}
        &\lesssim \sum_{\tau \in W_{l}}\rweight_{w\tau}\mathcal{E}(u \circ F_{w\tau})^{(p - 1 - \alpha_{p})/p} \mathcal{E}\bigl(\Phi(u \circ F_{w\tau}) - \widehat{u}_{w\tau}\bigr)^{\alpha_{p}/p}\mathcal{E}\bigl(\widehat{v}_{w\tau}\bigr)^{1/p} \\
        &\le \mathcal{E}(u \circ F_{w})^{(p - 1 - \alpha_{p})/p}\left(\sum_{\tau \in W_{l}}\rweight_{w\tau}\mathcal{E}(\Phi(u \circ F_{w\tau}) - \widehat{u}_{w\tau})\right)^{\alpha_{p}/p}\left(\sum_{\tau \in W_{l}}\rweight_{w\tau}\mathcal{E}\bigl(\widehat{v}_{w\tau}\bigr)\right)^{1/p} \\
        &\lesssim \mathcal{E}(u \circ F_{w})^{(p - 1 - \alpha_{p})/p}\left(\sum_{\tau \in W_{l}}\rweight_{w\tau}\mathcal{E}(\Phi(u \circ F_{w\tau}) - \widehat{u}_{w\tau})\right)^{\alpha_{p}/p}\max_{k \in \{ 1,\dots,n \}}\mathcal{E}(v_{k} \circ F_{w})^{1/p}.
    \end{align*}
    Combining these estimates with \eqref{Rsum.u}, which was shown in the proof of Theorem \ref{thm.em-chain-one}, and \eqref{Rsum.v}, we obtain $\lim_{l \to \infty}A_{i,l} = 0$ and thus \eqref{stp-chain} holds.

	Continuing to assume that $\Psi(\bm{v}) \in \mathcal{F} \cap \contfunc(K)$, by the uniform continuities of $\Phi',\partial \Psi_{k}$ and the fact that $\mathfrak{m}_{\mathcal{E}}\langle f; g \rangle(\Sigma_{w}) = \rweight_{w}\mathcal{E}(f \circ F_{w}; g \circ F_{w})$ for any $f,g \in \mathcal{F} \cap \contfunc(K)$ and any $w \in W_{\ast}$, we easily observe that 
	\[
    \lim_{l \to \infty}\abs{\sum_{k = 1}^{n}\int_{\Sigma_{w}}\sgn\bigl(\Phi'(u \circ \chi)\bigr)\abs{\Phi'(u \circ \chi)}^{p - 1}\partial_{k}\Psi(\bm{v} \circ \chi)\,d\mathfrak{m}_{\mathcal{E}}\langle u; v_{k} \rangle - \mathcal{S}_{l}^{(2)}(w)} = 0.   
    \]
    Hence, by \eqref{stp-chain} and Dynkin's $\pi$-$\lambda$ theorem,  
	\begin{equation}\label{chain-symbol}
        d\mathfrak{m}_{\mathcal{E}}\langle \Phi(u); \Psi(\bm{v}) \rangle
        = \sum_{k = 1}^{n}\sgn\bigl(\Phi'(u \circ \chi)\bigr)\abs{\Phi'(u \circ \chi)}^{p - 1}\partial_{k}\Psi(\bm{v} \circ \chi)\,d\mathfrak{m}_{\mathcal{E}}\langle u; v_{k} \rangle.
    \end{equation}
    Then by \eqref{chain-symbol} and Proposition \ref{prop.ss-pform-em}-\ref{it:ssem-two-chi}, we obtain the desired equality \eqref{chain} under the additional assumption that $\Psi(\bm{v}) \in \mathcal{F} \cap \contfunc(K)$. 
	We stress here that \emph{the arguments in this and the last paragraphs do NOT require the assumption that $(\mathcal{F} \cap L^{p}(K,m),\norm{\,\cdot\,}_{\mathcal{E},1})$ is a Banach space}. 
	(Note also that $u \in \mathcal{F} \cap \contfunc(K)$ and $\Phi \in C^{1}(\mathbb{R})$ are arbitrary here, and hence can be chosen to be $u = \Psi(\bm{v})$ and $\Phi = \id_{\mathbb{R}}$ as long as $\Psi(\bm{v}) \in \mathcal{F} \cap \contfunc(K)$.)  
    
    Thus it remains to prove that $\Psi(\bm{v}) \in \mathcal{F} \cap \contfunc(K)$. 
    We can assume that $\Psi(0) = 0$ since $\mathbb{R}\indicator{K} \subseteq \mathcal{E}^{-1}(0)$. 
	Define $Q(\bm{v}) \subseteq \mathbb{R}^{n}$ by  
	\[
    Q(\bm{v}) \coloneqq \Bigl[-\norm{v_{1}}_{\sup},\norm{v_{1}}_{\sup}\Bigr] \times \Bigl[-\norm{v_{2}}_{\sup},\norm{v_{2}}_{\sup}\Bigr] \times \cdots \times \Bigl[-\norm{v_{n}}_{\sup},\norm{v_{n}}_{\sup}\Bigr]. 
    \]
    Then there exists a sequence $\{ \Psi_{l} \}_{l \in \mathbb{N}}$ of polynomials in $n$ variables with real coefficients such that $\Psi_{l}(0) = 0$, $\norm{\Psi - \Psi_{l}}_{\sup, Q(\bm{v})} \to 0$ and $\norm{\partial_{k}\Psi - \partial_{k}\Psi_{l}}_{\sup, Q(\bm{v})} \to 0$ for each $k \in \{ 1,\dots,n \}$ as $l \to \infty$ (see \cite[Section II.4.3]{CH} for a proof of the existence of such polynomials). 
    Let $l \in \mathbb{N}$. 
    By the mean value theorem, for any $x = (x_{1},\ldots,x_{n}), y = (y_{1},\ldots,y_{n}) \in Q(\bm{v})$, 
    \begin{equation}\label{e:approx.poly.diff}
    	\abs{\Psi_{l}(x) - \Psi_{l}(y)} \le \sum_{k = 1}^{n}\norm{\partial_{k}\Psi_{l}}_{\sup, Q(\bm{v})}\abs{x_{k} - y_{k}}. 
    \end{equation} 
	Noting that $\Psi_{l}(\bm{v}) \in \mathcal{F} \cap \contfunc(K)$ by $\Psi_{l}$ being a polynomial and \eqref{e:weak.leibniz} and hence that \eqref{chain} with $\Psi_{l}$ in place of $\Psi$ holds by the result of the previous paragraph, we see from Propositions \ref{prop.ss-pform-em}-\ref{ssem} and \ref{prop.em-holder} that 
	\begin{align*}
		\mathcal{E}(\Psi_{l}(\bm{v})) 
		&= \Gamma_{\mathcal{E}}\langle \Psi_{l}(\bm{v}) \rangle(K) \qquad \text{(by Proposition \hyperref[ssem]{\ref{prop.ss-pform-em}}-\ref{ssem})} \\
		&= \int_{K}\sum_{k=1}^{n}\partial_{k}\Psi_{l}(\bm{v}(x))\,\Gamma_{\mathcal{E}}\langle \Psi_{l}(\bm{v}); v_{k} \rangle(dx) \qquad \text{(by \eqref{chain} with $\Psi_{l}$ in place of $\Psi$)} \\ 
		&\leq \sum_{k=1}^{n} \norm{\partial_{k}\Psi_{l}}_{\sup, Q(\bm{v})} \Gamma_{\mathcal{E}}\langle \Psi_{l}(\bm{v}) \rangle(K)^{\frac{p-1}{p}} \Gamma_{\mathcal{E}}\langle v_{k} \rangle(K)^{\frac{1}{p}} \qquad \text{(by Proposition \ref{prop.em-holder})} \\
		&= \mathcal{E}(\Psi_{l}(\bm{v}))^{\frac{p-1}{p}}\sum_{k = 1}^{n}\norm{\partial_{k}\Psi_{l}}_{\sup, Q(\bm{v})}\mathcal{E}(v_{k})^{1/p} \qquad \text{(by Proposition \ref{prop.ss-pform-em}-\ref{ssem})}, 
	\end{align*}
	which implies that $\sup_{l \in \mathbb{N}}\mathcal{E}(\Psi_{l}(\bm{v})) < \infty$. 
    Also, by $\Psi_{l}(0) = 0$, \eqref{e:approx.poly.diff} and the dominated convergence theorem, $\{ \Psi_{l}(\bm{v}) \}_{l \in \mathbb{N}}$ converges in $L^{p}(K,m)$ to $\Psi(\bm{v})$ as $l \to \infty$. 
    Now we conclude from Lemma \ref{lem.Lpreduce} applied to $\bigl(\mathcal{E},\closure{\mathcal{F} \cap \contfunc(K)}^{\mathcal{F} \cap L^{p}(K,m)}\bigr)$, which clearly satisfies \ref{Cp}, that $\Psi(\bm{v}) \in \closure{\mathcal{F} \cap \contfunc(K)}^{\mathcal{F} \cap L^{p}(K,m)} \cap \contfunc(K) = \mathcal{F} \cap \contfunc(K)$, completing the proof. 
\end{proof}

In the following corollaries, we recall useful consequences of the chain rule in Theorem \ref{thm.em-chain}, which are immediate from Proposition \ref{prop.em-express} (or more precisely, \eqref{eq:em-express-proof} in its proof), Theorems \ref{thm.EIDP} and \ref{thm.slocal}. 
\begin{cor}\label{cor:ssEMrepresent} 
	Assume that $\mathbb{R}\indicator{K} \subseteq \mathcal{E}^{-1}(0)$ and that $(\mathcal{E},\mathcal{F} \cap \contfunc(K))$ satisfies \eqref{lipcont} in Proposition \ref{prop.GC-list}-\ref{GC.lip} and \ref{Cp}. 
    Then for any $u, \varphi \in \mathcal{F} \cap \contfunc(K)$,
    \begin{equation}\label{e:ssem.functional}
    	\int_{K}\varphi\,d\Gamma_{\mathcal{E}}\langle u \rangle
    	= \mathcal{E}(u; u\varphi) - \biggl(\frac{p - 1}{p}\biggr)^{p - 1}\mathcal{E}\bigl(\abs{u}^{\frac{p}{p - 1}}; \varphi\bigr).
    \end{equation} 
\end{cor}

\begin{proof}
For any $u,\varphi \in \mathcal{F} \cap \contfunc(K)$, since $u\varphi \in \mathcal{F} \cap \contfunc(K)$ by \eqref{e:weak.leibniz}, we have \eqref{chain} with either of $(u,u\varphi)$ and $\bigl(\abs{u}^{\frac{p}{p - 1}},\varphi\bigr)$ in place of $(\Phi(u),\Psi(\bm{v}))$ by the second paragraph of the above proof of Theorem \ref{thm.em-chain}, and therefore \eqref{e:ssem.functional} follows from Proposition \ref{prop.ss-pform-em}-\ref{ssem} and the argument in \eqref{eq:em-express-proof}. 
\end{proof}

\begin{cor}\label{cor.EIDP.ss}
	Assume that $\mathbb{R}\indicator{K} \subseteq \mathcal{E}^{-1}(0)$, that $(\mathcal{E},\mathcal{F} \cap \contfunc(K))$ satisfies \eqref{lipcont} in Proposition \ref{prop.GC-list}-\ref{GC.lip} and \ref{Cp}, and that $(\mathcal{F} \cap L^p(K,m),\norm{\,\cdot\,}_{\mathcal{E},1})$ is a Banach space. 
	Then, for any $u \in \mathcal{F} \cap \contfunc(K)$, the Borel measure $\Gamma_{\mathcal{E}}\langle u \rangle \circ u^{-1}$ on $\mathbb{R}$ defined by $(\Gamma_{\mathcal{E}}\langle u \rangle \circ u^{-1})(A) \coloneqq \Gamma_{\mathcal{E}}\langle u \rangle(u^{-1}(A))$, $A \in \mathcal{B}(\mathbb{R})$, is absolutely continuous with respect to the Lebesgue measure on $\mathbb{R}$.
\end{cor}

\begin{cor}\label{cor.slocal.ss}
	Assume that $\mathbb{R}\indicator{K} \subseteq \mathcal{E}^{-1}(0)$, that $(\mathcal{E},\mathcal{F} \cap \contfunc(K))$ satisfies \eqref{lipcont} in Proposition \ref{prop.GC-list}-\ref{GC.lip} and \ref{Cp}, and that $(\mathcal{F} \cap L^p(K,m),\norm{\,\cdot\,}_{\mathcal{E},1})$ is a Banach space. 
	Let $u,u_{1},u_{2},v \in \mathcal{F} \cap \contfunc(K)$, $a,a_{1},a_{2},b \in \mathbb{R}$ and $A \in \mathcal{B}(K)$. 
    \begin{enumerate}[label=\textup{(\alph*)},align=left,leftmargin=*,topsep=2pt,parsep=0pt,itemsep=2pt]
        \item\label{it:SLbasic.em.ss}  If $A \subseteq u^{-1}(a)$, then $\Gamma_{\mathcal{E}}\langle u \rangle(A) = 0$.
        \item\label{it:SL0.em.ss} If $A \subseteq (u - v)^{-1}(a)$, then $\Gamma_{\mathcal{E}}\langle u \rangle(A) = \Gamma_{\mathcal{E}}\langle v \rangle(A)$.
        \item\label{it:SL1.em.ss} If $A \subseteq u_{1}^{-1}(a_{1}) \cup u_{2}^{-1}(a_{2})$, then
			\begin{align}\label{e:pem-sl1.ss}
			\Gamma_{\mathcal{E}}\langle u_1 + u_2 + v \rangle(A) + \Gamma_{\mathcal{E}}\langle v \rangle(A) &= \Gamma_{\mathcal{E}}\langle u_1 + v \rangle(A) + \Gamma_{\mathcal{E}}\langle u_2 + v \rangle(A), \\
			\Gamma_{\mathcal{E}}\langle u_{1} + u_{2}; v \rangle(A) &= \Gamma_{\mathcal{E}}\langle u_{1}; v \rangle(A) + \Gamma_{\mathcal{E}}\langle u_{2}; v \rangle(A). 
			\label{e:pem-sl1-cor.ss}
			\end{align}
        \item\label{it:SL2.em.ss} If $A \subseteq (u_1 - u_2)^{-1}(a) \cup v^{-1}(b)$, then
			\begin{equation}\label{e:pem-sl2.ss}
				\Gamma_{\mathcal{E}}\langle u_1; v \rangle(A) = \Gamma_{\mathcal{E}}\langle u_2; v \rangle(A)
				\quad\text{and}\quad
				\Gamma_{\mathcal{E}}\langle v; u_1 \rangle(A) = \Gamma_{\mathcal{E}}\langle v; u_2 \rangle(A).
			\end{equation}
    \end{enumerate}
\end{cor}

\subsection{Extensions of self-similar \texorpdfstring{$p$}{p}-energy measures}\label{subsec:ext-ssem}
In this subsection, we fix a self-similar structure $\mathcal{L} = (K,S,\{ F_{i} \}_{i \in S})$ with $K$ connected, a self-similar measure $m$ on $\mathcal{L}$, $p \in (1,\infty)$, and a self-similar $p$-energy form $(\mathcal{E},\mathcal{F})$ on $(\mathcal{L},m)$ with weight $(\rweight_{i})_{i \in S} \in (0,\infty)^{S}$, and further assume that $\mathcal{F} \subseteq L^{p}(K,m)$.  
In this setting, we first discuss the extension of self-similar $p$-energy measures to $\closure{\mathcal{F} \cap \contfunc(K)}^{\mathcal{F}} \eqqcolon \CoreClosure$. 
Recall the feature noted in Remark \ref{rmk:pullback} of $m$ as a self-similar measure on $\mathcal{L}$. 

\begin{lem}\label{lem.ssform-ext}
	Assume that $(\mathcal{F},\norm{\,\cdot\,}_{\mathcal{E},1})$ is a Banach space.  
	Let $u \in \mathcal{F}$ and $\{ u_{n} \}_{n \in \mathbb{N}} \subseteq \mathcal{F} \cap \contfunc(K)$. 
	If $\{ u_{n} \}_{n \in \mathbb{N}}$ converges in $\mathcal{F}$ to $u$, then $\{ u_{n} \circ F_{w} \}_{n \in \mathbb{N}}$ converges in $\mathcal{F}$ to $u \circ F_{w}$ for any $w \in W_{\ast}$. 
	In particular, 
	\begin{gather}
		u \circ F_{w} \in \CoreClosure \quad \text{for any $u \in \CoreClosure$ and any $w \in W_{\ast}$.} \label{e:ssdom-ext}\\
		\mathcal{E}(u) = \sum_{i \in S}\rweight_{i}\mathcal{E}(u \circ F_{i}) \quad \text{for any $u \in \CoreClosure$.} \label{e:ssform-ext}
	\end{gather}
\end{lem}
\begin{proof}
	Let $\{ u_{n} \}_{n \in \mathbb{N}}$ satisfy $\lim_{n \to \infty}\norm{u - u_{n}}_{\mathcal{E},1} = 0$.
	Then we easily see from the self-similarity of $m$ that $\{ u_{n} \circ F_{w} \}_{n \in \mathbb{N}}$ converges in $L^{p}(K,m)$ to $u \circ F_{w}$ for any $w \in W_{\ast}$. 
	Since $\mathcal{E}(u_{n} \circ F_{w} - u_{k} \circ F_{w}) \le \rweight_{w}^{-1}\mathcal{E}(u_{n} - u_{k})$ for any $n,k \in \mathbb{N}$ by the self-similarity \eqref{SSE2} of $\mathcal{E}$, $\{ u_{n} \circ F_{w} \}_{n \in \mathbb{N}}$ is a Cauchy sequence in $\mathcal{F}$. 
	Therefore, it has to converge to $u \circ F_{w}$ in $\mathcal{F}$, which shows \eqref{e:ssdom-ext}. 
	By letting $n \to \infty$ in \eqref{SSE2} for $u_{n}$, we obtain \eqref{e:ssform-ext}. 
\end{proof}

Now that we have obtained the identity \eqref{e:ssform-ext}, in a similar way using Kolmogorov's extension theorem as in the previous subsection, for each $u \in \CoreClosure$ we get a unique Borel measure $\mathfrak{m}_{\mathcal{E}}\langle u \rangle$ on $\Sigma$ such that $\mathfrak{m}_{\mathcal{E}}\langle u \rangle(\Sigma_{w}) = \rweight_{w}\mathcal{E}(u \circ F_{w})$ for any $w \in W_{\ast}$. 
The following lemma states the triangle inequality for $\mathfrak{m}_{\mathcal{E}}\langle \,\cdot\, \rangle(A)^{1/p}$ on $\CoreClosure$.
\begin{lem}\label{lem.prepEMconti}
	Assume that $(\mathcal{F},\norm{\,\cdot\,}_{\mathcal{E},1})$ is a Banach space.  
	Then for any $u,v \in \CoreClosure$ and any $A \in \mathcal{B}(\Sigma)$,
	\[
	\mathfrak{m}_{\mathcal{E}}\langle u + v \rangle(A)^{1/p} \le \mathfrak{m}_{\mathcal{E}}\langle u \rangle(A)^{1/p} + \mathfrak{m}_{\mathcal{E}}\langle v \rangle(A)^{1/p}. 
	\]
\end{lem}
\begin{proof}
	This follows from the triangle inequality for $\mathcal{E}^{1/p}$ and the argument in the proof of Proposition \ref{prop.sspem-pre}-\ref{it:sspem-pre.GC}.  
\end{proof}

Now we identify the $p$-energy measures $\{ \Gamma_{\mathcal{E}}\langle u \rangle \}_{u \in \CoreClosure}$, obtained by applying Proposition \ref{prop.em-ext} to the measures defined in \eqref{e:defn.em.one}, as $\{ \mathfrak{m}_{\mathcal{E}}\langle u \rangle \circ \chi^{-1} \}_{u \in \CoreClosure}$. 
\begin{prop}\label{prop.sspEM-ext}
	Assume that $(\mathcal{F},\norm{\,\cdot\,}_{\mathcal{E},1})$ is a Banach space. 
	Then for any $u \in \CoreClosure$, 
	\begin{equation}\label{e:pEMext.coincidence}
		\Gamma_{\mathcal{E}}\langle u \rangle = \mathfrak{m}_{\mathcal{E}}\langle u \rangle \circ \chi^{-1} \quad \text{(as Borel measures on $K$).} 
	\end{equation}
\end{prop}
\begin{proof}
	The equality \eqref{e:pEMext.coincidence} for $u \in \mathcal{F} \cap \contfunc(K)$ is obvious from the definition of $\Gamma_{\mathcal{E}}\langle u \rangle$ in \eqref{e:defn.em.one}. 
	Then the desired assertion immediately follows from \eqref{e:em.extension} in Proposition \ref{prop.em-ext}, Lemma \ref{lem.prepEMconti} and $\sup_{A \in \mathcal{B}(\Sigma)}\mathfrak{m}_{\mathcal{E}}\langle u \rangle(A) \le \mathcal{E}(u)$. 
\end{proof}

We conclude this subsection by seeing that self-similar $p$-energy measures can be extended to \emph{functions belonging locally to $\CoreClosure$} in Definition \ref{defn.Flocal-ss} below. 
To this end, we need the following lemma. 
\begin{lem}[Weak locality\index{weak locality (self-similar $p$-energy measures)} of self-similar $p$-energy measures; {\cite[Lemma 9.6]{MS.long}}]\label{lem.pEMwloc}
	Assume that $(\mathcal{F},\norm{\,\cdot\,}_{\mathcal{E},1})$ is a Banach space. 
	Let $U$ be an open subset of $K$.
	If $u,v \in \CoreClosure$ satisfy $u = v$ $\measure$-a.e.\ on $U$, then $\Gamma_{\mathcal{E}}\langle u \rangle|_{U} = \Gamma_{\mathcal{E}}\langle v \rangle|_{U}$.
\end{lem}
\begin{proof}
	The proof is exactly the same as \cite[Lemma 9.6]{MS.long}, but we recall the details here for the reader's convenience. 
	By the triangle inequality for $\Gamma_{\mathcal{E}}\langle \,\cdot\, \rangle$ from Proposition \ref{prop.em-ext} and the inner regularity of $\Gamma_{\mathcal{E}}\langle u - v \rangle$ (see, e.g., \cite[Theorem 7.1.3]{Dud} for the fact that any finite Borel measure on $K$ is inner regular), it suffices to show $\Gamma_{\mathcal{E}}\langle u - v \rangle(A) = 0$ for any compact subset $A$ of $U$. 
	Let $d$ be a metric on $K$ giving the original topology of $K$. 
	By \eqref{ss-diam}, we can choose $\delta \in (0, \dist_{\metric}(A, K \setminus U))$ and $N \in \mathbb{N}$ so that $\max_{w \in W_{n}}\diam(K_{w},d) < \delta$ for any $n \ge N$.
	For $n \in \mathbb{N}$, define $C_{n} \coloneqq \{ w \in W_{n} \mid \Sigma_{w} \cap \chi^{-1}(A) \neq \emptyset \}$.
	Since $(u - v) \circ F_{w} = 0$ $\measure$-a.e.\ on $K$ for any $n \ge N$ and any $w \in C_{n}$, we have
	\begin{align*}
		\mathfrak{m}_{\mathcal{E}}\langle u - v \rangle(\Sigma_{C_{n}})
		= \sum_{w \in C_{n}}\rweight_{w}\mathcal{E}((u - v) \circ F_w)
		= 0.
	\end{align*}
	Now, recalling \eqref{e:pEMext.coincidence} and noting that $\{ \Sigma_{C_{n}} \}_{n \in \mathbb{N}}$ is non-increasing and satisfies $\bigcap_{n \in \mathbb{N}}\Sigma_{C_{n}} = \chi^{-1}(A)$ (see \cite[Proof of Lemma 4.1]{Hin05} or \cite[Proof of Proposition 9.3]{MS.long} for a proof of this fact), we obtain $\Gamma_{\mathcal{E}}\langle u - v \rangle(A) = \mathfrak{m}_{\mathcal{E}}\langle u - v \rangle(\chi^{-1}(A)) = \lim_{n\to\infty} \mathfrak{m}_{\mathcal{E}}\langle u - v \rangle(\Sigma_{C_{n}}) = 0$.
\end{proof}

\begin{defn}\label{defn.Flocal-ss}
    Let $U$ be a non-empty open subset of $K$. 
    \begin{enumerate}[label=\textup{(\arabic*)},align=left,leftmargin=*,topsep=2pt,parsep=0pt,itemsep=2pt]
    	\item\label{it:Flocal-ss} For each linear subspace $\mathcal{D}$ of $\mathcal{F}$, we define a linear subspace $\mathcal{D}_{\mathrm{loc}}(U)$ of $L^{0}(U,m|_{U})$ by 
    	\begin{equation}\label{e:defn.Floc}
    		\mathcal{D}_{\mathrm{loc}}(U) \coloneqq
        	\biggl\{ f \in L^{0}(U,m|_{U}) \biggm|
        	\begin{minipage}{220pt}
            	$f = f^{\#}$ $m$-a.e.\ on $V$ for some $f^{\#} \in \mathcal{D}$ for each relatively compact open subset $V$ of $U$ 
        	\end{minipage}
        	\biggr\}.
    	\end{equation}
    	\item\label{it:Flocal-ss-sspem} Assume that $(\mathcal{F},\norm{\,\cdot\,}_{\mathcal{E},1})$ is a Banach space. In this setting, for each $f \in (\CoreClosure)_{\mathrm{loc}}(U) \eqqcolon \CoreClosure_{\mathrm{loc}}(U)$, we further define a Radon measure $\Gamma_{\mathcal{E}}\langle f \rangle$ on $U$ as follows.
			We first define $\Gamma_{\mathcal{E}}\langle f \rangle(E) \coloneqq \Gamma_{\mathcal{E}}\langle f^{\#} \rangle(E)$ for each relatively compact Borel subset $E$ of $U$, with $A \subseteq U$ and $f^{\#} \in \CoreClosure$ as in \eqref{e:defn.Floc} chosen so that $E \subseteq A$; this definition of $\Gamma_{\mathcal{E}}\langle f \rangle(E)$ is independent of a particular choice of such $A$ and $f^{\#}$ by Lemma \ref{lem.pEMwloc}.
			We then define $\Gamma_{\mathcal{E}}\langle f \rangle(E) \coloneqq \lim_{n \to \infty}\Gamma_{\mathcal{E}}\langle f \rangle(E \cap A_{n})$ for each $E \in \mathcal{B}|_{U}$, where $\{ A_{n} \}_{n \in \mathbb{N}}$ is a non-decreasing sequence of relatively compact open subsets of $U$ such that $\bigcup_{n \in \mathbb{N}}A_{n} = U$; it is clear that this definition of $\Gamma_{\mathcal{E}}\langle f \rangle(E)$ is independent of a particular choice of $\{ A_{n} \}_{n \in \mathbb{N}}$, coincides with the previous one when $E$ is relatively compact in $U$, and gives a Radon measure on $U$. 
    \end{enumerate}
\end{defn}

\subsection{Self-similar \texorpdfstring{$p$}{p}-energy form as a fixed point}
This subsection is devoted to presenting a standard method to construct a self-similar $p$-energy form. 
The main result of this subsection (Theorem \ref{thm.ssenergy-fix}) is essentially the same as the fixed point theorem in \cite[Theorem 1.5]{Kig00}, but we present the details to show a useful version of this fixed point theorem where a fixed point is explicitly given as a limit.  

In this subsection, we fix a self-similar structure $\mathcal{L} = (K,S,\{ F_{i} \}_{i \in S})$ with $K$ connected, a self-similar measure $m$ on $\mathcal{L}$, $p \in (1,\infty)$, and a linear subspace $\mathcal{F}$ of $L^{p}(K,m)$ with the following property:  
\begin{equation}\label{eq.SSdomain-whole}
	u \circ F_{w} \in \mathcal{F} \quad \text{for any $u \in \mathcal{F}$ and any $w \in W_{\ast}$}
\end{equation}
(recall Remark \ref{rmk:pullback}). 
We define 
\[
\mathfrak{E}_{p}(\mathcal{F}) \coloneqq \{ \mathcal{E} \colon \mathcal{F} \to [0,\infty) \mid \text{$(\mathcal{E},\mathcal{F})$ is a $p$-energy form on $(K,m)$} \}. 
\]

\begin{defn}\label{defn.ssenergyoperator}
	Let $\bm{\rweight} = (\rweight_{i})_{i \in S}$. 
	For $n \in \mathbb{N} \cup \{ 0 \}$, we define $\mathcal{S}_{\bm{\rweight},n} \colon \mathfrak{E}_{p}(\mathcal{F}) \to \mathfrak{E}_{p}(\mathcal{F})$ by 
    \begin{equation}\label{e:defn.renorm}
        \mathcal{S}_{\bm{\rweight},n}(E)(u) \coloneqq \sum_{w \in W_{n}}\rweight_{w}E(u \circ F_{w}) \quad \text{for $E \in \mathfrak{E}_{p}(\mathcal{F})$ and $u \in \mathcal{F}$.}
    \end{equation}
    (Note that the triangle inequality for $\mathcal{S}_{\bm{\rweight},n}(E)^{1/p}$ can be shown easily.) 
    Set $\mathcal{S}_{\bm{\rweight}} \coloneqq \mathcal{S}_{\bm{\rweight},1}$ and $\mathcal{S}_{\bm{\rweight},0} \coloneqq \id_{\mathfrak{E}_{p}(\mathcal{F})}$ for simplicity. 
    Clearly, $\mathcal{S}_{\bm{\rweight},n} = \mathcal{S}_{\bm{\rweight}}^{n} \coloneqq \underbrace{\mathcal{S}_{\bm{\rweight}} \circ \mathcal{S}_{\bm{\rweight}} \circ \cdots \circ \mathcal{S}_{\bm{\rweight}}}_{n}$.  
\end{defn}

The desired self-similar $p$-energy form with weight $\bm{\rweight}$ will be constructed as a non-trivial fixed point of $\mathcal{S}_{\bm{\rweight}}$. 
The following theorem, which can be regarded as a version of \cite[Theorem 1.5]{Kig00} in a specific situation, describes when we can find such a fixed point and how it is obtained. 
\begin{thm}\label{thm.ssenergy-fix}
	Let $\bm{\rweight} = (\rweight_{i})_{i \in S}$ and let $\mathcal{E}^{0} \in \mathfrak{E}_{p}(\mathcal{F})$. 
	Assume that the quotient normed space $\mathcal{F}/(\mathcal{E}^{0})^{-1}(0)$ (equipped with the norm $\mathcal{E}^{0}(\,\cdot\,)^{1/p}$) is separable and that there exists a constant $C \in [1,\infty)$ such that   
	\begin{equation}\label{e:assum.PSS}
		C^{-1}\mathcal{E}^{0}(u) \le \mathcal{S}_{\bm{\rweight},n}(\mathcal{E}^{0})(u) \le C\mathcal{E}^{0}(u) \quad \text{for any $u \in \mathcal{F}$ and any $n \in \mathbb{N}$.} 
	\end{equation}
	Then there exists $\{ n_{k} \}_{k \in \mathbb{N}} \subseteq \mathbb{N}$ with $n_{k} < n_{k + 1}$ for any $k \in \mathbb{N}$ such that the following limit exists in $[0,\infty)$ for any $u \in \mathcal{F}$: 
	\begin{equation}\label{e:fixpt.explicit}
		\mathcal{E}(u) \coloneqq \lim_{k \to \infty}\frac{1}{n_{k}}\sum_{j = 0}^{n_{k} - 1}\mathcal{S}_{\bm{\rweight},j}(\mathcal{E}^{0})(u).   
	\end{equation}
	Furthermore, $(\mathcal{E},\mathcal{F})$ is a $p$-energy form on $(K,m)$ satisfying 
	\begin{equation}\label{e:comparable.ss}
	C^{-1}\mathcal{E}^{0}(u) \le \mathcal{E}(u) \le C\mathcal{E}^{0}(u) \quad \text{for any $u \in \mathcal{F}$,} 
	\end{equation}
	where $C$ is the constant in \eqref{e:assum.PSS}, and 
	\begin{equation}\label{e:fixed.ss}
		\mathcal{E}(u) = \sum_{w \in W_{n}}\rweight_{w}\mathcal{E}(u \circ F_{w}) \quad \text{for any $u \in \mathcal{F}$ and any $n \in \mathbb{N} \cup \{ 0 \}$}. 
	\end{equation} 
\end{thm}
\begin{proof} 
	Set $\mathcal{E}^{n} \coloneqq \frac{1}{n}\sum_{j = 0}^{n - 1}\mathcal{S}_{\bm{\rweight},j}(\mathcal{E}^{0})$ for $n \in \mathbb{N}$ for ease of notation. 
	Then it is clear that $\mathcal{E}^{n} \in \mathfrak{E}_{p}(\mathcal{F})$. 
	Let $\mathscr{C}$ be a countable dense subset of $\mathcal{F}/(\mathcal{E}^{0})^{-1}(0)$. 
	Since $\{ \mathcal{E}^{n}(u) \}_{n \in \mathbb{N}}$ is bounded in $[0,\infty)$ for any $u \in \mathcal{F}$ by \eqref{e:assum.PSS}, by a standard diagonal procedure, there exists $\{ n_{k} \}_{k \in \mathbb{N}} \subseteq \mathbb{N}$ with $n_{k} < n_{k + 1}$ for any $k \in \mathbb{N}$ such that $\{ \mathcal{E}^{n_{k}}(u') \}_{k \in \mathbb{N}}$ is convergent in $[0,\infty)$ for any $u' \in \mathscr{C}$. 
	Let $u \in \mathcal{F}$, $\varepsilon > 0$ and $u_{\ast} \in \mathscr{C}$ satisfy $\mathcal{E}^{0}(u - u_{\ast})^{1/p} < \varepsilon$.
	Then for any $k,l \in \mathbb{N}$, by the triangle inequality for $\mathcal{E}^{n}(\,\cdot\,)^{1/p}$ and \eqref{e:assum.PSS}, 
	\begin{align*}
		&\abs{\mathcal{E}^{n_k}(u)^{1/p} - \mathcal{E}^{n_l}(u)^{1/p}} \\ 
		&\le \abs{\mathcal{E}^{n_k}(u)^{1/p} - \mathcal{E}^{n_k}(u_{\ast})^{1/p}}  + \abs{\mathcal{E}^{n_k}(u_{\ast})^{1/p} - \mathcal{E}^{n_l}(u_{\ast})^{1/p}}  + \abs{\mathcal{E}^{n_l}(u)^{1/p} - \mathcal{E}^{n_l}(u_{\ast})^{1/p}} \\
		&\le 2C^{1/p}\varepsilon + \abs{\mathcal{E}^{n_k}(u)^{1/p} - \mathcal{E}^{n_l}(u)^{1/p}}, 
	\end{align*}
	whence $\limsup_{k \wedge l \to \infty}\abs{\mathcal{E}^{n_k}(u)^{1/p} - \mathcal{E}^{n_l}(u)^{1/p}} \le 2C^{1/p}\varepsilon$. 
	Therefore $\{ \mathcal{E}^{n_k}(u) \}_{k \in \mathbb{N}}$ is convergent in $[0,\infty)$ for any $u \in \mathcal{F}$, so the limit in \eqref{e:fixpt.explicit} exists. 
	It is clear that $(\mathcal{E},\mathcal{F})$ is a $p$-energy form on $(K,m)$ satisfying \eqref{e:comparable.ss}. 
	
	Let us show \eqref{e:fixed.ss}. 
	For any $n \in \mathbb{N}$ and any $u \in \mathcal{F}$, we easily see that 
	\begin{equation}\label{e:fix.algebraic}
		\frac{1}{n}\mathcal{E}^{0}(u) + \mathcal{S}_{\bm{\rweight}}(\mathcal{E}^{n})(u)
		= \frac{1}{n}\mathcal{E}^{0}(u) + \frac{1}{n}\sum_{l = 0}^{n - 1}\mathcal{S}_{\bm{\rweight},l + 1}(\mathcal{E}^{0})(u)
		= \mathcal{E}^{n}(u) + \frac{1}{n}\mathcal{S}_{\bm{\rweight},n}(\mathcal{E}^{0})(u). 
	\end{equation}
	Since $\lim_{k \to \infty}\mathcal{S}_{\rweight}(\mathcal{E}^{n_{k}})(u) = \mathcal{S}_{\rweight}(\mathcal{E})(u)$ and $\lim_{k \to \infty}n_{k}^{-1}\mathcal{S}_{\bm{\rweight},n_{k}}(\mathcal{E}^{0})(u) = 0$ by \eqref{e:assum.PSS}, we obtain $\mathcal{S}_{\bm{\rweight}}(\mathcal{E}) = \mathcal{E}$ by letting $n \to \infty$ along $\{ n_{k} \}_{k \in \mathbb{N}}$ in \eqref{e:fix.algebraic}. 
	Hence \eqref{e:fixed.ss} holds. 	
\end{proof}

By virtue of the explicit representation \eqref{e:fixed.ss}, the resulting $p$-energy form $(\mathcal{E},\mathcal{F})$ inherits some nice properties of $(\mathcal{E}^{0},\mathcal{F})$. 
In the following proposition, we see that \ref{GC} and the invariance under good transformations are examples of such properties. 
\begin{prop}\label{prop.ssenergy-GCinv}
	Assume the same conditions as in Theorem \ref{thm.ssenergy-fix} and let $\mathcal{E}$ be given by \eqref{e:fixpt.explicit}. 
	\begin{enumerate}[label=\textup{(\alph*)},align=left,leftmargin=*,topsep=2pt,parsep=0pt,itemsep=2pt]
        \item \label{it:ssenergy.GC} If $(\mathcal{E}^{0},\mathcal{F})$ satisfies \ref{GC}, then $(\mathcal{E},\mathcal{F})$ also satisfies \ref{GC}. 
        \item \label{it:ssenergy.inv} Let $\mathscr{T}$ be a family of Borel measurable maps from $K$ to $K$. Assume that $u \circ T \in \mathcal{F}$ and $\mathcal{E}^{0}(u \circ T) = \mathcal{E}^{0}(u)$ for any $u \in \mathcal{F}$ and any $T \in \mathscr{T}$. Furthermore, we assume that for any $T \in \mathscr{T}$, there exists a bijection $\tau_{T} \colon W_{\ast} \to W_{\ast}$ such that
        \begin{equation}\label{e:trans.bi-level}
        	\text{$\tau_{T}|_{W_{n}}$ is a bijection from $W_{n}$ to itself for each $n \in \mathbb{N} \cup \{ 0 \}$,}
        \end{equation}
        \begin{equation}\label{e:trans.closed}
        	T(K_{w}) \subseteq K_{\tau_{T}(w)} \quad \text{and} \quad F_{\tau_{T}(w)}^{-1} \circ T \circ F_{w} \in \mathscr{T} \quad \text{for any $w \in W_{\ast}$,}
        \end{equation}
        and  
        \begin{equation}\label{e:trans.weightinv}
        	\rweight_{w} = \rweight_{\tau_{T}(w)} \quad \text{for any $w \in W_{\ast}$.}
        \end{equation}
        Then $\mathcal{E}(u \circ T) = \mathcal{E}(u)$ for any $u \in \mathcal{F}$ and any $T \in \mathscr{T}$. 
    \end{enumerate}
\end{prop}
\begin{proof}
	\ref{it:ssenergy.GC}: 
	Let $n_{1},n_{2} \in \mathbb{N}$, $q_{1} \in (0,p]$, $q_{2} \in [p,\infty]$ and $T = (T_{1},\dots,T_{n_{2}})\colon \mathbb{R}^{n_{1}} \to \mathbb{R}^{n_{2}}$ satisfy \eqref{GC-cond} in Definition \ref{defn.GC}. 
	Let $\bm{u} = (u_{1},\dots,u_{n_{1}}) \in \mathcal{F}$. 
	Then $T_{l}(u_{k} \circ F_{w}) = T_{l}(u_{k}) \circ F_{w} \in \mathcal{F}$ for any $k \in \{ 1,\dots,n_{1} \}$ and any $w \in W_{\ast}$ by \ref{GC} for $(\mathcal{E}^{0},\mathcal{F})$ and Lemma \ref{e:assum.PSS}. 
	If $q_{2} < \infty$, then by the same argument as \eqref{GC.sum} in the proof of Proposition \ref{prop.cone-gen},
	 \begin{align*}
        &\sum_{l = 1}^{n_{2}}\mathcal{S}_{\rweight}(\mathcal{E}^{0})\bigl(T_{l}(\bm{u})\bigr)^{q_{2}/p} 
	    = \sum_{l = 1}^{n_{2}}\Biggl[\sum_{i \in S}\rweight_{i}\mathcal{E}^{0}(T_{l}(\bm{u}) \circ F_{i})\Biggr]^{q_{2}/p} \nonumber \\
        &\le \left(\sum_{i \in S}\rweight_{i}\left[\sum_{l = 1}^{n_{2}}\mathcal{E}^{0}\bigl(T_{l}(\bm{u}) \circ F_{i}\bigr)^{q_{2}/p}\right]^{p/q_{2}}\right)^{q_{2}/p} \quad \text{(by the triangle inequality for $\norm{\,\cdot\,}_{\ell^{q_{2}/p}}$)} \nonumber \\ 
        &\overset{\ref{GC}}{\le} \left(\sum_{i \in S}\rweight_{i}\left[\sum_{k = 1}^{n_{1}}\mathcal{E}^{0}(u_{k} \circ F_{i})^{q_{1}/p}\right]^{p/q_{1}}\right)^{q_{2}/p} \nonumber \\
        &\overset{\eqref{reverse}}{\le} \left(\sum_{k = 1}^{n_{1}}\Biggl[\sum_{i \in S}\rweight_{i}\mathcal{E}^{0}(u_{k} \circ F_{i})\Biggr]^{q_{1}/p}\right)^{\frac{p}{q_{1}} \cdot \frac{q_{2}}{p}}
        = \left(\sum_{k = 1}^{n_{1}}\mathcal{S}_{\rweight}(\mathcal{E}^{0})(u_{k})^{q_{1}/p}\right)^{q_{2}/q_{1}},
    \end{align*}
	whence $\norm{\bigl(\mathcal{S}_{\rweight}(\mathcal{E}^{0})(T_{l}(\bm{u}))^{1/p}\bigr)_{l = 1}^{n_{2}}}_{\ell^{q_{2}}} \le \norm{\bigl(\mathcal{S}_{\rweight}(\mathcal{E}^{0})(u_{k})^{1/p}\bigr)_{k = 1}^{n_{1}}}_{\ell^{q_{1}}}$. 
	The case of $q_{2} = \infty$ is similar, so $(\mathcal{S}_{\rweight}(\mathcal{E}^{0}),\mathcal{F})$ satisfies \ref{GC}. 
	Similarly, one can easily show that $(\mathcal{S}_{\rweight,n}(\mathcal{E}^{0}),\mathcal{F})$ satisfies \ref{GC} for any $n \in \mathbb{N}$. 
	Hence \ref{GC} for $(\mathcal{E},\mathcal{F})$ holds by \eqref{e:fixed.ss} and Proposition \ref{prop.cone-gen}-\ref{GC.pwlimit}. 
	
	\ref{it:ssenergy.inv}: 
	By \eqref{e:fixed.ss}, it suffices to prove $\mathcal{S}_{\bm{\rweight},n}(\mathcal{E}^{0})(u \circ T) = \mathcal{S}_{\bm{\rweight},n}(\mathcal{E}^{0})(u)$ for any $u \in \mathcal{F}$, any $n \in \mathbb{N} \cup \{ 0 \}$ and any $T \in \mathscr{T}$. 
	We immediately see that 
	\begin{align*}
		\mathcal{S}_{\bm{\rweight},n}(\mathcal{E}^{0})(u \circ T)
		&= \sum_{w \in W_{n}}\rweight_{w}\mathcal{E}^{0}((u \circ T) \circ F_{w}) \\
		&= \sum_{w \in W_{n}}\rweight_{w}\mathcal{E}^{0}\bigl((u \circ F_{\tau_{T}(w)}) \circ F_{\tau_{T}(w)}^{-1} \circ T \circ F_{w}\bigr) \\
		&\overset{\eqref{e:trans.closed}}{=} \sum_{w \in W_{n}}\rweight_{w}\mathcal{E}^{0}(u \circ F_{\tau_{T}(w)}) 
		\overset{\eqref{e:trans.bi-level},\eqref{e:trans.weightinv}}{=} \mathcal{S}_{\bm{\rweight},n}(\mathcal{E}^{0})(u), 
	\end{align*}
	which completes the proof. 
\end{proof}

Also, $(\mathcal{E},\mathcal{F})$ in Theorem \ref{thm.ssenergy-fix} turns out to have the strong local property \hyperref[it:SL1]{\textup{(SL1)}} (recall Definition \ref{defn.Epsl}) under a mild condition. 
\begin{prop}\label{prop.ssenergy-sl}
	Assume the same conditions as in Theorem \ref{thm.ssenergy-fix} and let $\mathcal{E}$ be given by \eqref{e:fixpt.explicit}. 
	If $\mathbb{R}\indicator{K} \subseteq \{ u \in \mathcal{F} \mid \mathcal{E}^{0}(u) = 0 \}$, then $\mathbb{R}\indicator{K} \subseteq \{ u \in \mathcal{F} \mid \mathcal{E}(u) = 0 \}$ and $(\mathcal{E},\mathcal{F})$ satisfies the strong local property \hyperref[it:SL1]{\textup{(SL1)}}. 
\end{prop}
\begin{proof}
	It is immediate from \eqref{e:comparable.ss} that $\mathbb{R}\indicator{K} \subseteq \{ u \in \mathcal{F} \mid \mathcal{E}(u) = 0 \}$. 
	We will show \hyperref[it:SL1]{\textup{(SL1)}} for $(\mathcal{E},\mathcal{F})$. 
	Let $u_{1},u_{2},v \in \mathcal{F}$ and $a_{1},a_{2} \in \mathbb{R}$.
	Set $A_{i} \coloneqq \supp_{m}[u_{i} - a_{i}\indicator{K}]$ for $i \in \{ 1,2 \}$ and assume that $A_{1} \cap A_{2} = \emptyset$. 
	By \eqref{ss-diam}, there exists $n \in \mathbb{N}$ such that $(\bigcup_{w \in W_{n}[A_{1}]}K_{w}) \cap (\bigcup_{w \in W_{n}[A_{2}]}K_{w}) = \emptyset$, where $W_{n}[A_{i}] \coloneqq \{ w \in W_{n} \mid K_{w} \cap A_{i} \neq \emptyset \}$. 
	Note that $u_{i} \circ F_{w} = a_{i}\indicator{K}$ for $w \in W_{n} \setminus W_{n}[A_{i}]$. 
	This together with $\mathcal{E}(\indicator{K}) = 0$ and \eqref{e:fixed.ss} yields that 
	\begin{align*}
		&\quad \mathcal{E}(u_{1} + u_{2} + v) \\
		&\begin{multlined}
			= \sum_{w \in W_{n}[A_{1}]}\rweight_{w}\mathcal{E}(u_{1} \circ F_{w} + v \circ F_{w})
			+ \sum_{w \in W_{n}[A_{2}]}\rweight_{w}\mathcal{E}(u_{2} \circ F_{w} + v \circ F_{w}) \\
			+ \sum_{w \in W_{n} \setminus (W_{n}[A_{1}] \cup W_{n}[A_{2}])}\rweight_{w}\mathcal{E}(v \circ F_{w})
		\end{multlined}\\
		&\begin{multlined}
			= \mathcal{E}(u_{1} + v) + \mathcal{E}(u_{2} + v) - \sum_{w \in W_{n} \setminus W_{n}[A_{1}]}\rweight_{w}\mathcal{E}(v \circ F_{w}) - \sum_{w \in W_{n} \setminus W_{n}[A_{2}]}\rweight_{w}\mathcal{E}(v \circ F_{w}) \\
			+ \sum_{w \in W_{n} \setminus (W_{n}[A_{1}] \cup W_{n}[A_{2}])}\rweight_{w}\mathcal{E}(v \circ F_{w})
		\end{multlined}\\
		&= \mathcal{E}(u_{1} + v) + \mathcal{E}(u_{2} + v) - \mathcal{E}(v), 
	\end{align*}
	which shows \hyperref[it:SL1]{\textup{(SL1)}}. 
\end{proof}

\section{\texorpdfstring{$p$}{p}-Resistance forms and nonlinear potential theory}\label{sec.p-harm}
In this section, we will introduce the notion of $p$-resistance form as a special class of $p$-energy forms, and investigate harmonic functions with respect to a $p$-resistance form.
In particular, we prove fundamental results on taking the operation of traces of $p$-resistance forms, weak comparison principle and H\"{o}lder continuity estimates for harmonic functions. 
We also show the elliptic Harnack inequality for non-negative harmonic functions under some assumptions, and introduce the notion of $p$-resistance metric with respect to a given $p$-resistance form. 

Throughout this section, we fix $p \in (1, \infty)$, a non-empty set $X$, a linear subspace $\mathcal{F}$ of $\mathbb{R}^{X}$ and $\mathcal{E} \colon \mathcal{F} \to [0,\infty)$.
(This setting corresponds to choosing as $(\mathcal{B},m)$ the pair of $2^{X}$ and the counting measure on $X$ in the previous sections; recall Remark \ref{rmk:wo-measure}.) 

\subsection{Basics of \texorpdfstring{$p$}{p}-resistance forms}
The following definition is an $L^{p}$-analogue of the notion of \emph{resistance form} introduced by Kigami in \cite{Kig95}; see \cite{Kig01,Kig03,Kig12} for details on resistance forms. 
Recall that $\mathbb{R}\indicator{X} \coloneqq \{a\indicator{X} \mid a \in \mathbb{R}\}$ denotes the set of $\mathbb{R}$-valued constant functions on $X$.
\begin{defn}[$p$-Resistance form\index{$p$-resistance form}]\label{defn.RFp}
    The pair $(\mathcal{E}, \mathcal{F})$ of $\mathcal{F} \subseteq \mathbb{R}^{X}$ and $\mathcal{E} \colon \mathcal{F} \to [0,\infty)$ is said to be a \emph{$p$-resistance form} on $X$ if and only if it satisfies the following conditions \ref{RF1}--\ref{RF5}:
    \begin{enumerate}[label=\textup{(RF\arabic*)$_{p}$},align=left,leftmargin=*,topsep=2pt,parsep=0pt,itemsep=2pt]
    \item\label{RF1} $\mathcal{F}$ is a linear subspace of $\mathbb{R}^{X}$ including $\mathbb{R}\indicator{X}$ and $\mathcal{E}(\,\cdot\,)^{1/p}$ is a seminorm on $\mathcal{F}$ satisfying $\{ u \in \mathcal{F} \mid \mathcal{E}(u) = 0 \} = \mathbb{R}\indicator{X}$.
    \item\label{RF2} The quotient normed space $(\mathcal{F}/\mathbb{R}\indicator{X}, \mathcal{E}^{1/p})$ is a Banach space.
    \item\label{RF3} For any $x,y \in X$ with $x \neq y$, there exists $u \in \mathcal{F}$ such that $u(x) \neq u(y)$.
    \item\label{RF4} For any $x, y \in X$,
    \begin{equation}\label{R-def}
        R_{\mathcal{E}}(x,y) \coloneqq R_{(\mathcal{E},\mathcal{F})}(x,y) \coloneqq \sup\biggl\{ \frac{\abs{u(x) - u(y)}^{p}}{\mathcal{E}(u)} \biggm| u \in \mathcal{F} \setminus \mathbb{R}\indicator{X} \biggr\} < \infty.
    \end{equation}
    \item\label{RF5} $(\mathcal{E},\mathcal{F})$ satisfies \ref{GC}.
    \end{enumerate}
\end{defn}
\begin{rmk}
	\begin{enumerate}[label=\textup{(\arabic*)},align=left,leftmargin=*,topsep=2pt,parsep=0pt,itemsep=2pt]
		\item The notion of $2$-resistance form coincides with the original notion of resistance form as defined in \cite[Definition 2.3.1]{Kig01}, although the requirement \textup{(RF5)} there is weaker than \hyperref[RF5]{\textup{(RF5)$_{2}$}}.
			Indeed, assume first that $(\mathcal{E},\mathcal{F})$ is a $2$-resistance form on $X$. Then it satisfies \textup{(RF1)}, \textup{(RF2)} and \textup{(RF4)} in \cite[Definition 2.3.1]{Kig01} by \hyperref[RF1]{\textup{(RF1)$_{2}$}}, \hyperref[RF2]{\textup{(RF2)$_{2}$}} and \hyperref[RF4]{\textup{(RF4)$_{2}$}}, \textup{(RF5)} there by Proposition \ref{prop.GC-list}-\ref{GC.lip}, and \textup{(RF3)} there by Proposition \ref{prop.trace-dom} below.
			Moreover, since $(\mathcal{E},\mathcal{F})$ satisfies \hyperref[Cp]{\textup{(Cla)$_{2}$}} in Definition \ref{d:Cp} by Proposition \ref{prop.GC-list}-\ref{GC.Cpsmall},\ref{GC.Cplarge} and \hyperref[Cp-weak]{\textup{(Cla)$^{\prime}_{2}$}} by Remark \ref{rmk:Cp-weak}, we have $\mathcal{E}(u+v) + \mathcal{E}(u-v) = 2(\mathcal{E}(u) + \mathcal{E}(v))$ for any $u,v \in \mathcal{F}$, so that $\mathcal{F} \times \mathcal{F} \ni (u,v) \mapsto \frac{1}{4}\bigl( \mathcal{E}(u+v) - \mathcal{E}(u-v) \bigr)$ is easily seen to be a symmetric bilinear form with associated quadratic form $\mathcal{E}$ and hence a resistance form on $X$ in the sense of \cite[Definition 2.3.1]{Kig01}.
			
			Conversely, assume that $(\mathcal{E},\mathcal{F})$ is a resistance form on $X$ in the sense of \cite[Definition 2.3.1]{Kig01}, so that by \textup{(RF1)}--\textup{(RF4)} there it satisfies \hyperref[RF1]{\textup{(RF1)$_{2}$}}--\hyperref[RF4]{\textup{(RF4)$_{2}$}}.
			To see \hyperref[RF5]{\textup{(RF5)$_{2}$}}, let $n_{1},n_{2} \in \mathbb{N}$, $q_{1} \in (0,2]$, $q_{2} \in [2,\infty]$ and $T = (T_{1},\dots,T_{n_{2}}) \colon \mathbb{R}^{n_{1}} \to \mathbb{R}^{n_{2}}$ satisfy \eqref{GC-cond} in Definition \ref{defn.GC}, let $\bm{u} = (u_{1},\dots,u_{n_1}) \in \mathcal{F}^{n_{1}}$ and $\varepsilon \in (0,\infty)$.
			Then by a characterization of $(\mathcal{E},\mathcal{F})$ as the supremum over certain resistance forms on the non-empty finite subsets of $X$ in \cite[Corollary 2.37]{Kaj.ln} (see also \cite[Theorems 2.3.6 2.3.7 and Lemma 2.3.8]{Kig01} and Theorem \ref{thm.inductive} below), we easily see that $T(\bm{u}) \in \mathcal{F}^{n_{2}}$ and that there exists a non-empty finite subset $V$ of $X$ such that 
			\begin{equation}\label{eq:RF-GC-proof}
			\norm{\bigl(\mathcal{E}(T_{l}(\bm{u}),T_{l}(\bm{u}))^{1/2}\bigr)_{l = 1}^{n_{2}}}_{\ell^{q_{2}}} - \varepsilon
			\le \norm{\bigl(\mathcal{E}|_{V}(T_{l}(\bm{u})|_{V},T_{l}(\bm{u})|_{V})^{1/2}\bigr)_{l = 1}^{n_{2}}}_{\ell^{q_{2}}}, 
			\end{equation}
			where $\mathcal{E}|_{V}$ denotes the trace of $\mathcal{E}$ on $V$ (see \cite[Lemma 2.3.5]{Kig01} or Theorem \ref{thm.RF-exist} below for the definition of $\mathcal{E}|_{V}$). 
			By \cite[Lemma 2.3.5 and Proposition 2.1.3]{Kig01}, there exists $(L_{xy})_{x,y \in V} \in [0,\infty)^{V \times V}$ such that $L_{xy} = L_{yx}$ for any $x,y \in V$ and 
			\[
			\mathcal{E}|_{V}(u,u) = \frac{1}{2}\sum_{x,y \in V}L_{xy}\abs{u(x) - u(y)}^{2} \quad \textrm{for any $u \in \mathbb{R}^{V}$.} 
			\]
			It follows from this expression and Proposition \ref{prop.cone-gen}-\ref{GC.cone} that $\mathbb{R}^{V} \ni u \mapsto \mathcal{E}|_{V}(u,u)$ satisfies \hyperref[GC]{(GC)$_{2}$}, and therefore the right-hand side of \eqref{eq:RF-GC-proof} is bounded from above by
			\[ 
				\norm{\bigl(\mathcal{E}|_{V}(u_{k}|_{V},u_{k}|_{V})^{1/2}\bigr)_{k = 1}^{n_{1}}}_{\ell^{q_{1}}} 
				\le \norm{\bigl(\mathcal{E}(u_{k},u_{k})^{1/2}\bigr)_{k = 1}^{n_{1}}}_{\ell^{q_{1}}},
			\]
			where we used \cite[(2.53) in Corollary 2.37]{Kaj.ln} (see also \eqref{inductive.ene} below) in the last inequality. 
			Since $\varepsilon \in (0,\infty)$ is arbitrary, we obtain \hyperref[RF5]{(RF5)$_{2}$} for the functional $\mathcal{F} \ni u \mapsto \mathcal{E}(u,u)$, which is thereby a $2$-resistance form on $X$.  
		\item Let $V$ be a non-empty finite set. In this case, any $p$-resistance form $(\mathcal{E},\mathcal{F})$ on $V$ satisfies $\mathcal{F} = \mathbb{R}^{V}$ by Proposition \ref{prop.trace-dom} below, and in view of this observation we say simply that $\mathcal{E}$ is a $p$-resistance form on $V$.  
	\end{enumerate}
\end{rmk}
\begin{example}\label{ex.pRF}
	\begin{enumerate}[label=\textup{(\arabic*)},align=left,leftmargin=*,topsep=2pt,parsep=0pt,itemsep=2pt]
		\item\label{RF-Rn} Consider the same setting as in Example \ref{ex.Rn}-\ref{Ep-Rn} and assume that $\Omega$ is a bounded domain satisfying the strong local Lipschitz condition (see \cite[Paragraph 4.9]{AF}). Then the $p$-energy form $(\int_{\Omega}\abs{\nabla f}^{p}\,dx,W^{1,p}(\Omega))$ is a $p$-resistance form on $\Omega$ if and only if $p > D$. Indeed, \ref{RF1} is clear from the definition (we use the boundedness of $\Omega$ to ensure $\mathbb{R}\indicator{\Omega} \subseteq L^{p}(\Omega)$), \ref{RF5}, namely \ref{GC}, holds as observed already in Example \ref{ex.Rn}-\ref{Ep-Rn} (see also Theorem \ref{thm.GCp-kuwae}), and \ref{RF2} and \ref{RF3} follow from the well-known fact that $W^{1,p}(\Omega)$ is a Banach space including $\contfunc_{c}^{\infty}(\Omega) \coloneqq \contfunc^{\infty}(\Omega) \cap \contfunc_{c}(\Omega)$ for any $p \in (1,\infty)$ (see \cite[Theorem 3.3 and Corollary 3.4]{AF} for proofs). If $p > D$, then we can use the Morrey-type inequality \cite[Lemma 4.28]{AF} to verify \ref{RF4}. Conversely, the supremum in \eqref{R-def} is not finite when $p \le D$. To see it, we can assume that $x = 0 \in \Omega$. Let $\delta \in (0,\infty)$ be small enough so that $\closure{B(0,\delta)} \subseteq \Omega$ and $y \not\in \closure{B(0,\delta)}$. For all large $n \in \mathbb{N}$ so that $n^{-1} < \delta$, define $u_{n} \in \contfunc(\Omega)$ by $u_{n}(0) \coloneqq 1$ and
			\begin{align*}
				u_{n}(z) \coloneqq \left(\frac{\log{\abs{z}^{-1}} - \log{\delta^{-1}}}{\log{n} - \log{\delta^{-1}}}\right)^{+} \wedge 1, \quad z \in \Omega \setminus \{ 0 \}. 
			\end{align*}
			Then we easily see that $u_{n}|_{B(0,1/n)} = 1$, $u_{n}|_{\Omega \setminus \closure{B(0,\delta)}} = 0$ and $u_{n} \in W^{1,p}(\Omega)$ with 
			\begin{align*}
				\int_{\Omega}\abs{\nabla u_{n}}^{p}\,dz 
				&\le \abs{\frac{1}{\log{(n\delta)}}}^{p}\int_{B(0,\delta) \setminus B(0,n^{-1})}\abs{z}^{-p}\,dz 
				= \abs{S_{D - 1}}\abs{\frac{1}{\log{(n\delta)}}}^{p}\int_{\frac{1}{n}}^{\delta}r^{- p + D - 1}\,dr \\
				&= 
				\begin{cases}
					\abs{S_{D - 1}}\abs{\log{(n\delta)}}^{-(p - 1)} \quad &\text{if $p = D$,} \\
					\frac{\abs{S_{D - 1}}}{D - p}\abs{\log{(n\delta)}}^{-p}\Bigl(\delta^{D - p} - n^{-(D - p)}\Bigr) \quad &\text{if $p < D$,}
				\end{cases}
			\end{align*}
			where $\abs{S_{D - 1}}$ is the volume of the $(D - 1)$-dimensional unit sphere. 
			In particular, $\frac{\abs{u_{n}(x) - u_{n}(y)}^{p}}{\norm{\abs{\nabla u_{n}}}_{L^{p}(\Omega)}^{p}} \to \infty$ as $n \to \infty$, so \ref{RF4} does not hold. 
		\item\label{RF-kigami} The construction of a regular $p$-energy form on a $p$-conductively homogeneous compact metric space $(K,d)$ in \cite[Theorem 3.21]{Kig23} requires the assumption $p > \dim_{\mathrm{ARC}}(K,d)$, where $\dim_{\mathrm{ARC}}(K,d)$ is the Ahlfors regular conformal dimension of $(K,d)$. (See Definition \ref{defn.AR}-\ref{it:ARCdim} for the definition of $\dim_{\mathrm{ARC}}(K,d)$. The same condition $p > \dim_{\mathrm{ARC}}(K,d)$ is also assumed in \cite{Shi24}.) This condition $p > \dim_{\mathrm{ARC}}(K,d)$ plays the same role as $p > D$ in \ref{RF-Rn} above (see also \cite[Theorem 1.1]{CCK23+}). In Theorem \ref{t:Kig-good}, we will see that $p$-energy forms constructed in \cite[Theorem 3.21]{Kig23} are indeed $p$-resistance forms. We also show that $p$-energy forms on p.-c.f.\ self-similar sets in \cite[Theorem 5.1]{CGQ22} under the condition (\textbf{R}) in \cite[p.~18]{CGQ22} are $p$-resistance forms in Theorem \ref{thm.CGQ}. 
		\item\label{RF-graph} Here we recall typical $p$-resistance forms on finite sets given in \cite[Example 2.2-(1)]{KS.survey} because these examples are important to construct self-similar $p$-resistance forms on p.-c.f.\ self-similar structures in Subsection \ref{sec.pcf}. 
		Let $V$ be a non-empty finite set. Note that in this case $\mathcal{E}$ is a $p$-resistance form on $V$ if and only if $\mathcal{E} \colon \mathbb{R}^{V}\to[0,\infty)$ satisfies \ref{RF1} and \ref{RF5}; indeed, \ref{RF3} is obvious for $\mathcal{F} = \mathbb{R}^{V}$, and \ref{RF2} and \ref{RF4} are easily implied by \ref{RF1} and $\dim(\mathcal{F}/\mathbb{R}\indicator{V})<\infty$. Now, consider any functional $\mathcal{E}\colon\mathbb{R}^{V}\to[0,\infty)$ of the form
		\begin{equation}\label{eq:p-RF-finite-graph}
			\mathcal{E}(u) = \frac{1}{2}\sum_{x,y\in V}L_{xy}\abs{u(x)-u(y)}^{p}
		\end{equation}
		for some $L=(L_{xy})_{x,y\in V}\in[0,\infty)^{V\times V}$ such that $L_{xy}=L_{yx}$ for any $x,y\in V$. It is obvious that $\mathcal{E}$ satisfies \ref{RF1} if and only if the graph $(V,E_{L})$ is connected, where $E_{L} \coloneqq \{\{x,y\}\mid\textrm{$x,y\in V$, $x\not=y$, $L_{xy}>0$}\}$. It is also easy to see from Proposition \ref{prop.cone-gen}-\ref{GC.cone} that $\mathcal{E}$ satisfies \ref{RF5}. It thus follows that $\mathcal{E}$ is a $p$-resistance form on $V$ if and only if $(V,E_{L})$ is connected.
	
		Note that, while any $2$-resistance form on $V$ is of the form \eqref{eq:p-RF-finite-graph}
		with $p=2$, \emph{the counterpart of this fact for $p\not=2$ is NOT true} unless $\#V\leq 2$.
	\end{enumerate}
\end{example}

In the rest of this section, we assume that $(\mathcal{E},\mathcal{F})$ is a $p$-resistance form on $X$.
Then the following proposition is immediate from the definition \eqref{R-def} of $R_{\mathcal{E}}$ and Theorem \ref{thm.banach}. 
\begin{prop}\label{prop.R-conseq}
    \begin{enumerate}[label=\textup{(\arabic*)},align=left,leftmargin=*,topsep=2pt,parsep=0pt,itemsep=2pt]
        \item\label{it:RF.basic-ineq} For any $u \in \mathcal{F}$ and any $x,y \in X$,
        \begin{equation}\label{R-basic}
            \abs{u(x) - u(y)}^{p} \le R_{\mathcal{E}}(x,y)\mathcal{E}(u).
        \end{equation}
        \item $R_{\mathcal{E}}^{1/p}$ is a metric on $X$.
        \item\label{it:RF.unif-conv} $(\mathcal{F}/\mathbb{R}\indicator{X}, \mathcal{E}^{1/p})$ is a uniformly convex Banach space, and thus it is reflexive.
    \end{enumerate}
\end{prop}
In particular, the metric $R_{\mathcal{E}}^{1/p}$ induces a topology on $X$. 
Throughout the rest of this section, we consider $X$ as a topological space with the topology induced by $R_{\mathcal{E}}^{1/p}$. 
Note that then $\mathcal{F} \subseteq \contfunc(X)$ by \eqref{R-basic}.  

We introduce the notion of regularity of $p$-resistance forms as follows.
\begin{defn}[Regularity of $p$-resistance form]\label{dfn:regRF}
	Assume that $X$ is locally compact. We say that $(\mathcal{E}, \mathcal{F})$ is \emph{regular}\index{regular ($p$-resistance form)} if and only if $\mathcal{F} \cap \contfunc_{c}(X)$ is dense in $(\contfunc_{c}(X),\norm{\,\cdot\,}_{\sup})$. 
\end{defn}

The regularity ensures the existence of cutoff functions.
\begin{prop}\label{prop.regular}
    Assume that $X$ is locally compact and that $(\mathcal{E},\mathcal{F})$ is regular.
    Then for any open subset $U$ of $X$ and any compact subset $K$ of $U$, there exists $\psi \in \mathcal{F} \cap \contfunc_{c}(X)$ such that $0 \le \psi \le 1$, $\psi = 1$ on an open neighborhood of $K$ and $\supp_{X}[\psi] \subseteq U$. 
    In particular, $\mathcal{F} \cap \contfunc_{c}(X)$ is a special core. 
\end{prop}
\begin{proof}
    Since $X$ is locally compact, we can pick an open subset $\Omega$ of $X$ such that $K \subseteq \Omega$, $\closure{\Omega}^{X} \subseteq U$ and $\closure{\Omega}^{X}$ is compact.
    By Urysohn's lemma, there exists $\psi_{0} \in \contfunc_{c}(X)$ such that $0 \le \psi_{0} \le 1$, $\psi_{0} = 1$ on $\Omega$ and $\supp_{X}[\psi_{0}] \subseteq U$. 
    Since $(\mathcal{E},\mathcal{F})$ is regular, for any $\varepsilon \in (0,1/2)$ there exists $\psi_{\varepsilon} \in \mathcal{F} \cap \contfunc_{c}(X)$ such that $\norm{\psi_{0} - \psi_{\varepsilon}}_{\mathrm{sup}} < \varepsilon$, and then the function $\psi \coloneqq \bigl[(1 - 2\varepsilon)^{-1}(\psi_{\varepsilon} - \varepsilon)^{+}\bigr] \wedge 1$ belongs to $\mathcal{F} \cap \contfunc_{c}(X)$ by \ref{RF1} and Proposition \ref{prop.GC-list}-\ref{GC.lip} and has the desired properties. 
\end{proof}

We need the following notation to define traces of $(\mathcal{E},\mathcal{F})$ on subsets of $X$ later. 
\begin{defn}\label{dfn:restrdom}
    Let $B$ be a non-empty subset of $X$.
    We define a linear subspace $\mathcal{F}|_{B}$ of $\mathbb{R}^{B}$ by $\mathcal{F}|_{B} \coloneqq \bigl\{ u|_{B} \bigm| u \in \mathcal{F} \bigr\}$.
\end{defn}

The following proposition is useful to discuss boundary conditions on finite sets.
\begin{prop}[{\cite[Proposition 3.2]{Kig12}}]\label{prop.trace-dom}
    Let $B$ be a non-empty finite subset of $X$. Then $\mathcal{F}|_{B} = \mathbb{R}^{B}$.
\end{prop}
\begin{proof}
    This is a special case of \cite[Proposition 3.2]{Kig12}, but we give a detailed proof here for the reader's convenience.
    By virtue of \ref{RF1}, it suffices to show that $\indicator{x}^{B} \in \mathcal{F}|_{B}$ for any $x \in B$ under the assumption that $\#B \geq 2$.  
    Let $x \in B$.
    For each $y \in B \setminus \{ x \}$, by \ref{RF3} and \ref{RF1}, there exists $u_{y} \in \mathcal{F}$ satisfying $u_{y}(x) = 1$ and $u_{y}(y) = 0$.
    Let $f \coloneqq \sum_{y \in B \setminus \{ x \}}(u_{y}^{+} \wedge 1)$ and $g \coloneqq \sum_{y \in B \setminus \{ x \}}\bigl((1 - u_{y})^{+} \wedge 1\bigr)$.
    Then $f, g \in \mathcal{F}$ by \ref{RF1} and \ref{RF5}.
    Since $f(x) = \#B - 1$, $f|_{B \setminus \{ x \}} \le \#B - 2$, $g(x) = 0$ and $g|_{B \setminus \{ x \}} \ge 1$, the function $h \in \mathcal{F}$ given by
    \[
    h \coloneqq \bigl(f - (\#B - 2)(g^{+} \wedge 1)\bigr)^{+} \wedge 1
    \]
    satisfies $h|_{B} = \indicator{x}^{B}$ and hence $\indicator{x}^{B} \in \mathcal{F}|_{B}$.
\end{proof}

The next definition is introduced to deal with Dirichlet-type boundary conditions.
\begin{defn}\label{defn.RFp-general}
    For each subset $B$ of $X$, we define
    \[
    \mathcal{F}^{0}(B) \coloneqq \{ u \in \mathcal{F} \mid \text{$u(x) = 0$ for any $x \in X \setminus B$} \}, \quad B^{\mathcal{F}} \coloneqq \bigcap_{u \in \mathcal{F}^{0}(X \setminus B)}u^{-1}(0).
    \]
\end{defn}

See \cite[Chapters 2, 5 and 6]{Kig12} for basic properties of $B^{\mathcal{F}}$.
Here we recall just the following results, which will be used later.
\begin{prop}[{\cite[Theorem 2.4, Lemma 2.5 and Theorem 6.3]{Kig12}}]\label{p:genbord}\strut
    \begin{enumerate}[label=\textup{(\alph*)},align=left,leftmargin=*,topsep=2pt,parsep=0pt,itemsep=2pt]
        \item \label{fcn-top} $\mathbb{C}_{\mathcal{F}} \coloneqq \bigl\{ B \bigm| \textrm{$B \subseteq X$, $B^{\mathcal{F}} = B$} \bigr\}$ satisfies the axiom of closed sets and it defines a topology on $X$. Moreover, $\{ x \} \in \mathbb{C}_{\mathcal{F}}$ for any $x \in X$.
        \item \label{fcn-sep} For any subset $B$ of $X$ and any $x \in X \setminus B^{\mathcal{F}}$, there exists $u \in \mathcal{F}^{0}(X \setminus B)$ such that $u(x) = 1$ and $0 \le u(y) \le 1$ for any $y \in X$.
        \item \label{fcn-same} Assume that $X$ is locally compact and that $(\mathcal{E},\mathcal{F})$ is regular. Then $B^{\mathcal{F}} = B$ for any closed subset $B$ of $X$.
    \end{enumerate}
\end{prop}
\begin{proof}
    The statements \ref{fcn-top} and \ref{fcn-sep} are included in \cite[Theorem 2.4 and Lemma 2.5]{Kig12}.
    The argument showing $\textup{(R1)} \Rightarrow \textup{(R2)}$ in \cite[Proof of Theorem 6.3]{Kig12} proves \ref{fcn-same}.
\end{proof}

For a non-empty subset $B$ of $X$ and $x \in X \setminus B^{\mathcal{F}}$, we define
\begin{equation}\label{R-def.gen}
    R_{\mathcal{E}}(x,B) \coloneqq R_{(\mathcal{E},\mathcal{F})}(x,B) \coloneqq \sup\biggl\{ \frac{\abs{u(x)}^{p}}{\mathcal{E}(u)} \biggm| \textrm{$u \in \mathcal{F}^{0}(X \setminus B)$, $u(x) \not= 0$} \biggr\} \in (0,\infty). 
\end{equation}
Note that $R_{\mathcal{E}}(x,\{ y \}) = R_{\mathcal{E}}(x,y)$ for any $x,y \in X$ with $x \not= y$ by Proposition \ref{p:genbord}-\ref{fcn-top}. 

\subsection{Harmonic functions and traces of \texorpdfstring{$p$}{p}-resistance forms}\label{sec.trace}
In this subsection, we consider harmonic functions with respect to $p$-resistance forms and traces of $p$-resistance forms to subsets of the original domains.  

The following proposition states that the variational and distributional formulations of harmonic functions coincide for $p$-resistance forms.
\begin{prop}\label{prop.equiv}
    Let $h \in \mathcal{F}$ and $B \subseteq X$.
    Then the following conditions are equivalent:
    \begin{enumerate}[label=\textup{(\arabic*)},align=left,leftmargin=*,topsep=2pt,parsep=0pt,itemsep=2pt]
        \item \label{harm.1} $\mathcal{E}(h) = \inf\{ \mathcal{E}(u) \mid \text{$u \in \mathcal{F}$, $u|_{B} = h|_{B}$} \}$.
        \item \label{harm.2} $\mathcal{E}(h; \varphi) = 0$ for any $\varphi \in \mathcal{F}^{0}(X \setminus B)$.
    \end{enumerate}
\end{prop}
\begin{proof}
    For $\varphi \in \mathcal{F}$, define $E_{\varphi} \colon \mathbb{R} \to \mathbb{R}$ by $E_{\varphi}(t) \coloneqq \mathcal{E}(h + t\varphi)$, so that $E_{\varphi}$ is differentiable by Proposition \ref{prop.diffble}.
	If $\mathcal{E}(h) = \inf\{ \mathcal{E}(u) \mid \text{$u \in \mathcal{F}$, $u|_{B} = h|_{B}$} \}$ and $\varphi \in \mathcal{F}^{0}(X \setminus B)$, then $E_{\varphi}$ takes its minimum at $t = 0$, and thus $\mathcal{E}(h;\varphi) = \frac{1}{p} \frac{d}{dt}E_{\varphi}(t)\bigr|_{t = 0} = 0$, proving \ref{harm.1} $\Rightarrow$ \ref{harm.2}. 

	Conversely, assume that $\mathcal{E}(h; \varphi) = 0$ for any $\varphi \in \mathcal{F}^{0}(X \setminus B)$, and let $u \in \mathcal{F}$ satisfy $u|_{B} = h|_{B}$. 
	Then by $u - h \in \mathcal{F}^{0}(X \setminus B)$ we have $\frac{d}{dt}\bigr|_{t=0}E_{u-h}(t) = p \mathcal{E}(h; u - h) = 0$, which together with the convexity of $E_{u-h}$ implies that $\mathcal{E}(u) = E_{u-h}(1) \geq E_{u-h}(0) = \mathcal{E}(h)$, proving \ref{harm.2} $\Rightarrow$ \ref{harm.1}. 
\end{proof}

\begin{defn}[$\mathcal{E}$-(sub,super)harmonic function\index{$\mathcal{E}$-harmonic function}\index{$\mathcal{E}$-subharmonic function}\index{$\mathcal{E}$-superharmonic function}]\label{dfn:part-harmonic}
    Let $B \subseteq X$ and $h \in \mathcal{F}$.  
    We say that $h$ is \emph{$\mathcal{E}$-subharmonic} on $X \setminus B$ if and only if 
    \begin{equation}\label{e:defn.subharm}
		\mathcal{E}(h; \varphi) \le 0 \quad \text{for any $\varphi \in \mathcal{F}^{0}(X \setminus B)$ with $\varphi \ge 0$.} 
	\end{equation}
    We say that $h$ is \emph{$\mathcal{E}$-superharmonic} on $X \setminus B$ if and only if $-h$ is $\mathcal{E}$-subharmonic on $X \setminus B$, and say that $h$ is \emph{$\mathcal{E}$-harmonic} on $X \setminus B$ if and only if $h$ is both $\mathcal{E}$-subharmonic and $\mathcal{E}$-superharmonic on $X \setminus B$, i.e., $h$ satisfies either (and hence both) of \ref{harm.1} and \ref{harm.2} in Proposition \ref{prop.equiv}. 
    We set $\mathcal{H}_{\mathcal{E},B} \coloneqq \{ h \in \mathcal{F} \mid \text{$h$ is $\mathcal{E}$-harmonic on $X \setminus B$} \}$.
\end{defn}

$\mathcal{E}$-harmonic functions with any boundary values $u \in \mathcal{F}|_{B}$ on a given non-empty subset $B$ of $X$ uniquely exist, and their energies under $\mathcal{E}$ define a new $p$-resistance form $(\mathcal{E}|_{B}, \mathcal{F}|_{B})$ on the boundary set $B$, as follows.
This new $p$-resistance form $(\mathcal{E}|_{B}, \mathcal{F}|_{B})$ on $B$ is called the \emph{trace}\index{trace (of $p$-resistance form)} of $(\mathcal{E},\mathcal{F})$ on $B$. 
\begin{thm}\label{thm.RF-exist}
    Let $B \subseteq X$ be non-empty, and define $\mathcal{E}|_{B} \colon \mathcal{F}|_{B} \to [0,\infty)$ by
    \begin{equation} \label{eq:dfn-trace}
        \mathcal{E}|_{B}(u) \coloneqq \inf\{ \mathcal{E}(v) \mid v \in \mathcal{F}, v|_{B} = u \}, \quad u \in \mathcal{F}|_{B}.
    \end{equation}
    Then $(\mathcal{E}|_{B}, \mathcal{F}|_{B})$ is a $p$-resistance form on $B$ and $R_{\mathcal{E}|_{B}} = R_{\mathcal{E}}|_{B \times B}$.
    Moreover, for any $u \in \mathcal{F}|_{B}$ there exists a unique $h_{B}^{\mathcal{E}}[u] \in \mathcal{F}$ such that $h_{B}^{\mathcal{E}}[u]\bigr|_{B} = u$ and $\mathcal{E}\bigl(h_{B}^{\mathcal{E}}[u]\bigr) = \mathcal{E}|_{B}(u)$, so that $h_{B}^{\mathcal{E}}(\mathcal{F}|_{B}) = \mathcal{H}_{\mathcal{E},B}$, and
    \begin{align}
        h_{B}^{\mathcal{E}}[au + b\indicator{B}] &= ah_{B}^{\mathcal{E}}[u] + b\indicator{X} \quad \text{for any $u \in \mathcal{F}|_{B}$ and any $a,b \in \mathbb{R}$,} \label{harm-const} \\
        \mathcal{E}|_{B}(u; v) &= \mathcal{E}\bigl(h_{B}^{\mathcal{E}}[u]; h_{B}^{\mathcal{E}}[v]\bigr) \quad \text{for any $u,v \in \mathcal{F}|_{B}$,} \label{diff-compat} \\
        \mathcal{E}|_{B}(f|_{B}; g|_{B}) &= \mathcal{E}(f; g) \quad \text{for any $f \in \mathcal{H}_{\mathcal{E},B}$ and any $g \in \mathcal{F}$,} \label{harm-compat}
    \end{align}
    where $\mathcal{E}|_{B}(u; v) \coloneqq \frac{1}{p}\left.\frac{d}{dt}\mathcal{E}|_{B}(u + tv)\right|_{t = 0}$ for $u,v \in \mathcal{F}|_{B}$ \textup{(recall \eqref{exist-deriva} in Theorem \ref{thm.p-form} for this definition)}.
\end{thm}
\begin{rmk}\label{rmk.nonlin}
    The map $h_{B}^{\mathcal{E}}[\,\cdot\,]$ does \emph{not} satisfy either $h_{B}^{\mathcal{E}}[u + v] \le h_{B}^{\mathcal{E}}[u] + h_{B}^{\mathcal{E}}[u]$ for any $u,v \in \mathcal{F}|_{B}$ or $h_{B}^{\mathcal{E}}[u + v] \ge h_{B}^{\mathcal{E}}[u] + h_{B}^{\mathcal{E}}[u]$ for any $u,v \in \mathcal{F}|_{B}$ in general, unless $p = 2$ or $\#B \le 2$.
\end{rmk}
\begin{proof}[Proof of Theorem \ref{thm.RF-exist}]
    We first show the desired existence of $h_{B}^{\mathcal{E}}[u]$ for any $u \in \mathcal{F}|_{B}$.
    Let us fix $y_{\ast} \in B$ and let $\alpha \coloneqq \inf\bigl\{ \mathcal{E}(v) \bigm| \text{$v \in \mathcal{F}$ with $v|_{B} = u$} \bigr\} \in [0, \infty)$.
    Then there exists $\{ v_{n} \}_{n \in \mathbb{N}}$ such that $v_{n} \in \mathcal{F}$, $v_{n}|_{B} = u$ and $\mathcal{E}(v_{n}) \le \alpha + n^{-1}$ for any $n \in \mathbb{N}$.
    Note that $\frac{v_{k} + v_{l}}{2} \in \mathcal{F}$ also satisfies $\bigl(\frac{v_{k} + v_{l}}{2}\bigr)\bigr|_{B} = u$ for any $k, l \in \mathbb{N}$.
    In the case of $p \in (1,2]$, we see that
    \begin{align}\label{cauchy1}
        \mathcal{E}(v_k - v_l)^{1/(p - 1)}
        &\overset{\eqref{Cp.small}}{\le} 2\bigl(\mathcal{E}(v_k) + \mathcal{E}(v_l)\bigr)^{1/(p - 1)} - \mathcal{E}(v_k + v_l)^{1/(p - 1)} \nonumber \\
        &\le 2\bigl(2\alpha + k^{-1} + l^{-1}\bigr)^{1/(p - 1)} - 2^{p/(p - 1)}\alpha^{1/(p - 1)} \nonumber \\
        &\xrightarrow[k \wedge l \to \infty]{}
        2(2\alpha)^{1/(p - 1)} - 2^{p/(p - 1)}\alpha^{1/(p - 1)} = 0.
    \end{align}
    Similarly, in the case of $p \in [2,\infty)$, we have
    \begin{align}\label{cauchy2}
        \mathcal{E}(v_k - v_l)
        &\overset{\eqref{Cp.large}}{\le} 2\bigl(\mathcal{E}(v_k)^{1/(p - 1)} + \mathcal{E}(v_l)^{1/(p - 1)}\bigr)^{p - 1} - \mathcal{E}(v_k + v_l) \nonumber \\
        &\le 2\bigl((\alpha + k^{-1})^{1/(p - 1)} + (\alpha + l^{-1})^{1/(p - 1)}\bigr)^{p - 1} - 2^{p}\alpha \nonumber \\
        &\xrightarrow[k \wedge l \to \infty]{}
        2\bigl(2\alpha^{1/(p - 1)}\bigr)^{p - 1} - 2^{p}\alpha = 0.
    \end{align}
    Consequently, $\{ v_{n} \}_{n \in \mathbb{N}}$ is a Cauchy sequence in $(\mathcal{F}/\mathbb{R}\indicator{X}, \mathcal{E}^{1/p})$.
    By \ref{RF2}, there exists $h \in \mathcal{F}$ such that $h(y_{\ast}) = u(y_{\ast})$ and $\lim_{n \to \infty}\mathcal{E}(h - v_{n}) = 0$.
    For any $y \in B$, by \ref{RF4},
    \[
    \abs{h(y) - u(y)}^{p}
    = \abs{h(y) - v_{n}(y)}^{p}
    = \abs{(h - v_{n})(y) - (h - v_{n})(y_{\ast})}^{p}
    \le R_{\mathcal{E}}(y, y_{\ast})\mathcal{E}(h - v_{n}) \to 0 
    \]
    as $n\to\infty$, and hence $h|_{B} = u$.
    In particular, $h$ is a minimizer of $\alpha$.
    Assume that $g \in \mathcal{F}$ also satisfies $g|_{B} = u$ and $\mathcal{E}(g) = \alpha$.
    Then a similar estimate to \eqref{cauchy1} or to \eqref{cauchy2} imply that $\mathcal{E}(h - g) = 0$.
    Since $h - g \in \mathcal{F}^{0}(X \setminus B)$ and $B \neq \emptyset$, we have $h = g \eqqcolon h_{B}^{\mathcal{E}}[u]$ by \ref{RF1}.
    The property \eqref{harm-const} immediately follows from \ref{RF1} for $(\mathcal{E},\mathcal{F})$.

    Next we prove that $(\mathcal{E}|_{B},\mathcal{F}|_{B})$ is a $p$-resistance form on $B$.
    It is clear that $\mathcal{E}|_{B}(au) = \abs{a}^{p}\mathcal{E}|_{B}(u)$ for any $u \in \mathcal{F}|_{B}$.
	Let us show the triangle inequality for $\mathcal{E}|_{B}(\,\cdot\,)^{1/p}$,
    Since $\bigl.(h_{B}^{\mathcal{E}}[u] + h_{B}^{\mathcal{E}}[v])\bigr|_{B} = u + v$ for any $u, v \in \mathcal{F}|_{B}$, we see that
    \begin{align*}
    	\mathcal{E}|_{B}(u + v)^{1/p}
    	&= \mathcal{E}\bigl(h_{B}^{\mathcal{E}}[u + v]\bigr)^{1/p} 
    	\le \mathcal{E}\bigl(h_{B}^{\mathcal{E}}[u] + h_{B}^{\mathcal{E}}[v]\bigr)^{1/p} \\
    	&\le \mathcal{E}\bigl(h_{B}^{\mathcal{E}}[u]\bigr)^{1/p} + \mathcal{E}\bigl(h_{B}^{\mathcal{E}}[v]\bigr)^{1/p}
    	= \mathcal{E}|_{B}(u)^{1/p} + \mathcal{E}|_{B}(v)^{1/p}.
    \end{align*}
    By \eqref{harm-const}, we easily see that $\mathcal{F}|_{B}$ contains $\mathbb{R}\indicator{B}$.
    If $u \in \mathcal{F}|_{B}$ satisfies $\mathcal{E}|_{B}(u) = 0$, then $h_{B}^{\mathcal{E}}[u] \in \mathbb{R}\indicator{X}$ and hence $\bigl.h_{B}^{\mathcal{E}}[u]\bigr|_{B} = u \in \mathbb{R}\indicator{B}$.
    Thus \ref{RF1} for $(\mathcal{E}|_{B}, \mathcal{F}|_{B})$ holds.
    To prove \ref{RF2} for $(\mathcal{E}|_{B}, \mathcal{F}|_{B})$, let $\{ u_{n} \} \subseteq \mathcal{F}|_{B}$ satisfy $\lim_{n \wedge m \to \infty}\mathcal{E}|_{B}(u_{n} - u_{m}) = 0$.
    Then, by the triangle inequality for $\mathcal{E}|_{B}(\,\cdot\,)^{1/p}$, we easily see that $\{ \mathcal{E}|_{B}(u_{n}) \}_{n \in \mathbb{N}}$ is a Cauchy sequence in $[0,\infty)$.
    By \ref{Cp} for $(\mathcal{E}, \mathcal{F})$ and a similar argument to \eqref{cauchy1} (or to \eqref{cauchy2}), we have $\lim_{n \wedge m \to \infty}\mathcal{E}\bigl(h_{B}^{\mathcal{E}}[u_{n}] - h_{B}^{\mathcal{E}}[u_{m}]\bigr) = 0$.
    Hence there exists $h \in \mathcal{F}$ such that $\lim_{n \to \infty}\mathcal{E}(h - h_{B}^{\mathcal{E}}[u_{n}]) = 0$ by \ref{RF2} for $(\mathcal{E}, \mathcal{F})$.
    Then $\mathcal{E}|_{B}(h|_{B} - u_{n}) \le \mathcal{E}(h - h_{B}^{\mathcal{E}}[u_{n}]) \to 0$, which proves the completeness of $\bigl(\mathcal{F}|_{B}/\mathbb{R}\indicator{B}, \mathcal{E}|_{B}(\,\cdot\,)^{1/p}\bigr)$.
    The condition \ref{RF3} for $\mathcal{F}|_{B}$ is clear from that of $\mathcal{F}$.
    The inequality $R_{\mathcal{E}|_{B}} \le R_{\mathcal{E}}|_{B \times B}$ (and hence \ref{RF4} for $(\mathcal{E}|_{B}, \mathcal{F}|_{B})$) follows from the following estimate:
    \[
    \frac{\abs{u(x) - u(y)}^{p}}{\mathcal{E}|_{B}(u)} = \frac{\abs{h_{B}^{\mathcal{E}}[u](x) - h_{B}^{\mathcal{E}}[u](y)}^{p}}{\mathcal{E}(h_{B}^{\mathcal{E}}[u])} \le R_{\mathcal{E}}(x,y) \quad \text{for any $x,y \in B$, $u \in \mathcal{F}|_{B}$.}
    \]
    To show the converse inequality $R_{\mathcal{E}|_{B}} \ge R_{\mathcal{E}}|_{B \times B}$, let $x, y \in B$ and let $u \in \mathcal{F} \setminus \mathbb{R}\indicator{X}$ be such that $u(x) \neq u(y)$. 
    Then $u|_{B} \in \mathcal{F}|_{B} \setminus \mathbb{R}\indicator{B}$ and $\mathcal{E}(u) \ge \mathcal{E}|_{B}(u|_{B}) > 0$.
    Therefore,
    \[
    \frac{\abs{u(x) - u(y)}^{p}}{\mathcal{E}(u)} \le \frac{\abs{u|_{B}(x) - u|_{B}(y)}^{p}}{\mathcal{E}|_{B}(u|_{B})} \le R_{\mathcal{E}|_{B}}(x,y).
    \]
    The same estimate is clear if $u(x) = u(y)$, so taking the supremum over $u \in \mathcal{F} \setminus \mathbb{R}\indicator{X}$ yields $R_{\mathcal{E}}(x,y) \le R_{\mathcal{E}|_{B}}(x,y)$.
    Lastly, we prove \ref{RF5} for $(\mathcal{E}|_{B},\mathcal{F}|_{B})$.
    Let $n_{1},n_{2} \in \mathbb{N}$, $q_{1} \in (0,p]$, $q_{2} \in [p,\infty]$ and $T = (T_{1},\dots,T_{n_{2}}) \colon \mathbb{R}^{n_{1}} \to \mathbb{R}^{n_{2}}$ satisfy \eqref{GC-cond} in Definition \ref{defn.GC}, and let $\bm{u} = (u_{1},\dots,u_{n_{1}}) \in \bigl(\mathcal{F}|_{B}\bigr)^{n_{1}}$. 
    Note that $T_{l}(\bm{u}) =  T_{l}\bigl.\bigl(h_{B}^{\mathcal{E}}[u_{1}],\dots,h_{B}^{\mathcal{E}}[u_{n_{1}}]\bigr)\bigr|_{B} \in \mathcal{F}|_{B}$.
    Therefore, if $q_{2} < \infty$, then
    \begin{align*}
        \left(\sum_{l = 1}^{n_{2}}\mathcal{E}|_{B}\bigl(T_{l}(\bm{u})\bigr)^{q_{2}/p}\right)^{1/q_{2}}
        &\le \left(\sum_{l = 1}^{n_{2}}\mathcal{E}\Bigl(T_{l}\bigl.\bigl(h_{B}^{\mathcal{E}}[u_{1}],\dots,h_{B}^{\mathcal{E}}[u_{n_{1}}]\bigr)\Bigr)^{q_{2}/p}\right)^{1/q_{2}} \\
        &\le \left(\sum_{k = 1}^{n_{1}}\mathcal{E}\bigl(h_{B}^{\mathcal{E}}[u_{k}]\bigr)^{q_{1}/p}\right)^{1/q_{1}}
        = \left(\sum_{k = 1}^{n_{1}}\mathcal{E}|_{B}(u_{k})^{q_{1}/p}\right)^{1/q_{1}}.
    \end{align*}
    The case $q_{2} = \infty$ is similar, so $(\mathcal{E}|_{B},\mathcal{F}|_{B})$ satisfies \ref{GC}.

    We conclude the proof by showing \eqref{diff-compat} and \eqref{harm-compat}.
    By Proposition \ref{prop.diffble}, we know that
    \[
    \lim_{t \downarrow 0}\frac{\mathcal{E}|_{B}(u \pm tv) - \mathcal{E}|_{B}(u)}{\pm t} = \left.\frac{d}{dt}\mathcal{E}|_{B}(u + tv)\right|_{t = 0},
    \]
    and
    \[
    \lim_{t \downarrow 0}\frac{\mathcal{E}\bigl(h_{B}^{\mathcal{E}}[u] \pm th_{B}^{\mathcal{E}}[v]\bigr) - \mathcal{E}\bigl(h_{B}^{\mathcal{E}}[u]\bigr)}{\pm t} = p\mathcal{E}\bigl(h_{B}^{\mathcal{E}}[u]; h_{B}^{\mathcal{E}}[v]\bigr).
    \]
    For any $t > 0$, we have
    \begin{align*}
        \frac{\mathcal{E}\bigl(h_{B}^{\mathcal{E}}[u] - th_{B}^{\mathcal{E}}[v]\bigr) - \mathcal{E}\bigl(h_{B}^{\mathcal{E}}[u]\bigr)}{-t}
        &\le \frac{\mathcal{E}|_{B}(u - tv) - \mathcal{E}|_{B}(u)}{-t} \\
        &\hspace*{-30pt}\le \frac{\mathcal{E}|_{B}(u + tv) - \mathcal{E}|_{B}(u)}{t}
        \le \frac{\mathcal{E}\bigl(h_{B}^{\mathcal{E}}[u] + th_{B}^{\mathcal{E}}[v]\bigr) - \mathcal{E}\bigl(h_{B}^{\mathcal{E}}[u]\bigr)}{t},
    \end{align*}
    and hence we obtain \eqref{diff-compat} by letting $t \downarrow 0$.
    If $f \in \mathcal{H}_{\mathcal{E},B}$, i.e., $h_{B}^{\mathcal{E}}[f|_{B}] = f$, then $\mathcal{E}(f; g) = \mathcal{E}(f; h_{B}^{\mathcal{E}}[g]) = \mathcal{E}|_{B}(f|_{B}; g|_{B})$ since $g - h_{B}^{\mathcal{E}}[g|_{B}] \in \mathcal{F}^{0}(X \setminus B)$ for any $g \in \mathcal{F}$.
    This proves \eqref{harm-compat}.
\end{proof}

The following proposition states a compatibility of the operation of taking traces.
\begin{prop}\label{prop.trace-comp}
    Let $A,B$ be subsets of $X$ such that $\emptyset \neq A \subseteq B$.
	Then $(\mathcal{E}|_{B}|_{A}, \mathcal{F}|_{B}|_{A}) = (\mathcal{E}|_{A}, \mathcal{F}|_{A})$ and $h_{B}^{\mathcal{E}} \circ h_{A}^{\mathcal{E}|_{B}} = h_{A}^{\mathcal{E}}$. 
    In particular, $h_{A}^{\mathcal{E}|_{B}}[u] = h_{A}^{\mathcal{E}}[u]\bigr|_{B}$ for any $u \in \mathcal{F}|_{A}$.
\end{prop}
\begin{proof}
    Clearly, we have $\mathcal{F}|_{B}|_{A} = \mathcal{F}|_{A}$.
	For any $u \in \mathcal{F}|_{A}$, we see that
	\begin{align*}
		\mathcal{E}|_{A}(u)
		= \mathcal{E}\bigl(h_{A}^{\mathcal{E}}[u]\bigr)
		&\ge \min\bigl\{ \mathcal{E}(v) \bigm| \text{$v \in \mathcal{F}$ such that $v|_{B} = \bigl.h_{A}^{\mathcal{E}}[u]\bigr|_{B}$} \bigr\} \\
		&= \mathcal{E}|_{B}\bigl(\bigl.h_{A}^{\mathcal{E}}[u]\bigr|_{B}\bigr) \\
		&\ge \min\bigl\{ \mathcal{E}|_{B}(w) \bigm| \text{$w \in \mathcal{F}|_{B}$ such that $w|_{A} = \bigl.h_{A}^{\mathcal{E}}[u]\bigr|_{A} = u$} \bigr\} \\
		&= \mathcal{E}|_{B}|_{A}(u)
		= \mathcal{E}|_{B}\bigl(h_{A}^{\mathcal{E}|_{B}}[u]\bigr)
		= \mathcal{E}\Bigl(h_{B}^{\mathcal{E}}\bigl[h_{A}^{\mathcal{E}|_{B}}[u]\bigr]\Bigr) \\
		&\ge \min\Bigl\{ \mathcal{E}(v) \Bigm| \text{$v \in \mathcal{F}$ such that $v|_{A} = \bigl(h_{B}^{\mathcal{E}} \circ \bigl.h_{A}^{\mathcal{E}|_{B}}\bigr)[u]\bigr|_{A} = u$} \Bigr\} = \mathcal{E}|_{A}(u),
	\end{align*}
	which implies $\mathcal{E}|_{A}(u) = \mathcal{E}|_{B}|_{A}(u)$ and $\mathcal{E}\bigl(h_{A}^{\mathcal{E}}[u]\bigr) = \mathcal{E}\bigl((h_{B}^{\mathcal{E}} \circ h_{A}^{\mathcal{E}|_{B}})[u]\bigr)$.
	Since the restrictions to $A$ of both functions $h_{A}^{\mathcal{E}}[u]$ and $(h_{B}^{\mathcal{E}} \circ h_{A}^{\mathcal{E}|_{B}})[u]$ are $u$, the uniqueness in Theorem \ref{thm.RF-exist} implies $h_{A}^{\mathcal{E}}[u] = \bigl(h_{B}^{\mathcal{E}} \circ h_{A}^{\mathcal{E}|_{B}} \bigr)[u]$.
    Taking their restrictions to $B$ yields $h_{A}^{\mathcal{E}|_{B}}[u] = h_{A}^{\mathcal{E}}[u]\bigr|_{B}$.
\end{proof}

The following theorem, which is a straightforward generalization of its counterpart for resistance forms given in \cite[Corollary 2.37]{Kaj.ln}, presents an expression of $(\mathcal{E},\mathcal{F})$ as the ``inductive limit'' of its traces $\{ \mathcal{E}|_{V} \}_{V \subseteq X, \, 1 \le \#V < \infty}$ on the non-empty finite subsets of $X$.
This expression can be applied to get some useful results on convergence of the seminorm $\mathcal{E}^{1/p}$.
\begin{thm}\label{thm.inductive}
    It holds that
    \begin{gather}
        \mathcal{F} = \biggl\{ u \in \mathbb{R}^{X} \biggm| \sup_{V \subseteq X; \, 1 \le \#V < \infty}\mathcal{E}|_{V}(u|_{V}) < \infty \biggr\}, \label{inductive.dom} \\
        \mathcal{E}(u) = \sup_{V \subseteq X; \, 1 \le \#V < \infty}\mathcal{E}|_{V}(u|_{V}) \quad \text{for any $u \in \mathcal{F}$.} \label{inductive.ene}
    \end{gather}
\end{thm}
\begin{proof}
    Define $(\mathcal{E}_{\ast},\mathcal{F}_{\ast})$ by
    \[
    \mathcal{E}_{\ast}(u) \coloneqq \sup_{V \subseteq X;  1 \le \#V < \infty}\mathcal{E}|_{V}(u|_{V}), \quad u \in \mathbb{R}^{X},
    \]
    and $\mathcal{F}_{\ast} \coloneqq \{ u \in \mathbb{R}^{X} \mid \mathcal{E}_{\ast}(u) < \infty \}$.
    Then $\mathcal{E}_{\ast}^{1/p}$ is clearly a seminorm on $\mathcal{F}_{\ast}$ and $\{ u \in \mathcal{F}_{\ast} \mid \mathcal{E}_{\ast}(u) = 0 \} = \mathbb{R}\indicator{X}$.
    We first show that, for any $V \subseteq X$ with $1 \le \#V < \infty$ and any $u \in \mathbb{R}^{V}$,
    \begin{equation}\label{sametrace}
        h_{V}^{\mathcal{E}}[u] \in \mathcal{F}_{\ast} \quad \text{and} \quad \mathcal{E}|_{V}(u) = \min\{ \mathcal{E}_{\ast}(v) \mid v \in \mathcal{F}, v|_{V} = u \} = \mathcal{E}_{\ast}\bigl(h_{V}^{\mathcal{E}}[u]\bigr),
    \end{equation}
    both of which are obtained by seeing that, for any $U \subseteq X$ with $1 \le \#U < \infty$,
    \[
    \mathcal{E}|_{U}\bigl(h_{V}^{\mathcal{E}}[u]\bigr|_{U}\bigr)
    \le \mathcal{E}\bigl(h_{V}^{\mathcal{E}}[u]\bigr)
    = \mathcal{E}|_{V}(u).
    \]
    Indeed, taking the supremum over $U$, we get $\mathcal{E}_{\ast}\bigl(h_{V}^{\mathcal{E}}[u]\bigr) \le \mathcal{E}|_{V}(u)$ and hence \eqref{sametrace} holds.
    (The converse $\mathcal{E}|_{V}(u) \le \mathcal{E}_{\ast}\bigl(h_{V}^{\mathcal{E}}[u]\bigr)$ is clear from the definition.)
    We also note that $\mathcal{E}_{\ast}$ satisfies \ref{Cp} since $(\mathcal{E}|_{Y},\mathcal{F}|_{Y})$ is a $p$-resistance form for each $Y \subseteq X$ and $\mathcal{E}|_{V}(u|_{V}) \le \mathcal{E}|_{U}(u|_{U})$ for any $U,V \subseteq X$ with $\emptyset \neq V \subseteq U$ and $u \in \mathbb{R}^{U}$. 

    The inclusion $\mathcal{F} \subseteq \mathcal{F}_{\ast}$ and the estimate $\mathcal{E}_{\ast} \le \mathcal{E}$ (on $\mathcal{F}$) easily follow from the following estimate:
    \[
    \mathcal{E}|_{V}(u|_{V})
    = \mathcal{E}\bigl(h_{V}^{\mathcal{E}}[u|_{V}]\bigr)
    \le \mathcal{E}(u) \quad \text{for any $u \in \mathcal{F}$ and any $V \subseteq X$ with $1 \le \#V < \infty$}.
    \] 
    To show $\mathcal{F}_{\ast} \subseteq \mathcal{F}$ and $\mathcal{E} \le \mathcal{E}_{\ast}$, let $u \in \mathcal{F}_{\ast}$, let us choose a subset $V_{n} \subseteq X$ for each $n \in \mathbb{N}$ such that $1 \le \#V_{n} < \infty$ and $\mathcal{E}|_{V_{n}}(u|_{V_{n}}) \ge \mathcal{E}_{\ast}(u) - n^{-1}$, and set $u_{n} \coloneqq h_{V_{n}}^{\mathcal{E}}[u|_{V_{n}}]$.
    Then
    \[
    \mathcal{E}_{\ast}(u) - n^{-1}
    \le \mathcal{E}|_{V_{n}}(u|_{V_{n}})
    \overset{\eqref{sametrace}}{=} \mathcal{E}_{\ast}(u_{n})
    \overset{\eqref{sametrace}}{\le} \mathcal{E}_{\ast}(u),
    \]
    which implies that $\lim_{n \to \infty}\mathcal{E}_{\ast}(u_{n}) = \lim_{n \to \infty}\mathcal{E}(u_{n}) = \mathcal{E}_{\ast}(u)$.
    Using \ref{Cp} for $\mathcal{E}_{\ast}$ and $\mathcal{E}_{\ast}\bigl(\frac{u + u_{n}}{2}\bigr) \ge \mathcal{E}_{\ast}(u_{n})$, we easily obtain $\lim_{n \to \infty}\mathcal{E}_{\ast}(u - u_{n}) = 0$ similarly as \eqref{cauchy1} or \eqref{cauchy2}.
    We next show that $\{ u_{n} \}_{n \in \mathbb{N}}$ is a Cauchy sequence in $(\mathcal{F}/\mathbb{R}\indicator{X}, \mathcal{E}^{1/p})$.
    From \ref{Cp} for $\mathcal{E}$, $\lim_{n \to \infty}\mathcal{E}(u_{n}) = \lim_{n \to \infty}\mathcal{E}_{\ast}(u_{n}) = \mathcal{E}_{\ast}(u)$ and
    \[
    \mathcal{E}(u_{k} + u_{l}) \ge \mathcal{E}\bigl(h_{V_{k} \cup V_{l}}^{\mathcal{E}}[(u_{k} + u_{l})|_{V_{k} \cup V_{l}}]\bigr)
    \ge 2^{p}\mathcal{E}|_{V_{k} \cup V_{l}}(u|_{V_{k} \cup V_{l}})
    \overset{\eqref{sametrace}}{=} 2^{p}\mathcal{E}_{\ast}(u_{k + l}),
    \]
    we can obtain $\lim_{k \wedge l \to \infty}\mathcal{E}(u_{k} - u_{l}) = 0$ similarly as \eqref{cauchy1} or \eqref{cauchy2} in the proof of Theorem \ref{thm.RF-exist}. 
    Hence, by \ref{RF1} for $(\mathcal{E},\mathcal{F})$, there exists $v \in \mathcal{F}$ such that $\lim_{n \to \infty}\mathcal{E}(v - u_{n}) = 0$.
    By $\mathcal{E}_{\ast} \le \mathcal{E}$ on $\mathcal{F}$, we conclude that $\lim_{n \to \infty}\mathcal{E}_{\ast}(v - u_{n}) = 0$, which together with the triangle inequality for $\mathcal{E}_{\ast}^{1/p}$ and $\lim_{n \to \infty}\mathcal{E}_{\ast}(u - u_{n}) = 0$ implies that $\mathcal{E}_{\ast}(u - v) = 0$ and thus $u - v \in \mathbb{R}\indicator{X}$.
    In particular, $u = (u - v) + v \in \mathcal{F}_{\ast}$ and $\mathcal{E}(u) = \lim_{n \to \infty}\mathcal{E}(u_{n}) = \mathcal{E}_{\ast}(u)$, completing the proof.
\end{proof}

\begin{cor}\label{cor.approx1}
    Let $u \in \mathbb{R}^{X}$ and $\{ u_{n} \}_{n \in \mathbb{N}} \subseteq \mathcal{F}$.
    \begin{enumerate}[label=\textup{(\alph*)},align=left,leftmargin=*,topsep=2pt,parsep=0pt,itemsep=2pt]
        \item \label{approx1-1} Assume that $\lim_{n \to \infty}(u_{n}(x) - u_{n}(y)) = u(x) - u(y)$ for any $x,y \in X$.
        Then $\mathcal{E}(u) \le \liminf_{n \to \infty}\mathcal{E}(u_{n})$, where we set $\mathcal{E}(f) \coloneqq \infty$ for $f \in \mathbb{R}^{X} \setminus \mathcal{F}$. In particular, if $\liminf_{n \to \infty}\mathcal{E}(u_{n}) < \infty$, then $u \in \mathcal{F}$.
        \item \label{approx1-2} Assume that $u \in \mathcal{F}$. Then $\lim_{n \to \infty}\mathcal{E}(u - u_{n}) = 0$ if and only if $\limsup_{n \to \infty}\mathcal{E}(u_{n}) \le \mathcal{E}(u)$ and $\lim_{n \to \infty}(u_{n}(x) - u_{n}(y)) = u(x) - u(y)$ for any $x,y \in X$.
    \end{enumerate}
\end{cor}
\begin{proof}
    \ref{approx1-1}: If $\liminf_{n \to \infty}\mathcal{E}(u_{n}) = \infty$, then the desired statement clearly holds.
    Assume that $\liminf_{n \to \infty}\mathcal{E}(u_{n}) < \infty$, choose a strictly increasing sequence $\{ n_{k} \}_{k \in \mathbb{N}} \subseteq \mathbb{N}$ so that $\lim_{k\to\infty}\mathcal{E}(u_{n_{k}}) = \liminf_{n\to\infty}\mathcal{E}(u_{n})$, and let $V$ be a non-empty finite subset of $X$.
    Then since $\mathcal{E}|_{V}(\,\cdot\,)^{1/p}$ is a seminorm on the finite-dimensional vector space $\mathbb{R}^{V}$ with $\mathcal{E}|_{V}(\indicator{V})^{1/p} = 0$ by \ref{RF1}, and since $\lim_{k \to \infty}(u_{n_{k}}(x) - u_{n_{k}}(y)) = u(x) - u(y)$ for any $x,y \in V$, it follows that
    \[
    \mathcal{E}|_{V}(u|_{V}) = \lim_{k \to \infty}\mathcal{E}|_{V}(u_{n_{k}}|_{V}) \leq \lim_{k \to \infty}\mathcal{E}(u_{n_{k}}) = \liminf_{n\to\infty}\mathcal{E}(u_{n}),
    \]
    where the inequality holds by the definition \eqref{eq:dfn-trace} of $\mathcal{E}|_{V}$ in Theorem \ref{thm.RF-exist}.
    We thus obtain $\sup_{V \subseteq X; \, 1 \le \#V < \infty}\mathcal{E}|_{V}(u|_{V}) \leq \liminf_{n\to\infty}\mathcal{E}(u_{n}) < \infty$, hence $u \in \mathcal{F}$ by \eqref{inductive.dom} and $\mathcal{E}(u) \leq \liminf_{n \to \infty}\mathcal{E}(u_{n})$ by \eqref{inductive.ene}, proving \ref{approx1-1}.
    
    \ref{approx1-2}: Assume $\limsup_{n \to \infty}\mathcal{E}(u_{n}) \leq \mathcal{E}(u)$ and that $\lim_{n \to \infty}(u_{n}(x) - u_{n}(y)) = u(x) - u(y)$ for any $x,y \in X$.
    Then $\lim_{n \to \infty}\mathcal{E}(u_{n}) = \mathcal{E}(u)$ by \ref{approx1-1} of the present corollary, and we also have $\lim_{n \to \infty}\mathcal{E}\bigl(\frac{u + u_{n}}{2}\bigr) = \mathcal{E}(u)$ since $\{ \frac{u + u_{n}}{2} \}_{n \in \mathbb{N}}$ satisfies the same conditions as $\{ u_{n} \}_{n \in \mathbb{N}}$ by the triangle inequality for $\mathcal{E}^{1/p}$.
    Similar to \eqref{cauchy1} or \eqref{cauchy2} in the proof of Theorem \ref{thm.RF-exist}, we now conclude by $p$-Clarkson's inequality \ref{Cp} for $\mathcal{E}$ from Proposition \ref{prop.GC-list}-\ref{GC.Cpsmall},\ref{GC.Cplarge} that $\lim_{n \to \infty}\mathcal{E}(u - u_{n}) = 0$.
    The converse part of \ref{approx1-2} is immediate from the triangle inequality for $\mathcal{E}^{1/p}$ and the estimate \eqref{R-basic} in Proposition \ref{prop.R-conseq}-\ref{it:RF.basic-ineq}.
\end{proof}

\begin{cor}\label{cor.approx2}
    \begin{enumerate}[label=\textup{(\alph*)},align=left,leftmargin=*,topsep=2pt,parsep=0pt,itemsep=2pt]
        \item \label{approx2-1} Let $\{ \varphi_{n} \}_{n \in \mathbb{N}} \subseteq \contfunc(\mathbb{R})$ satisfy $\lim_{n \to \infty}\varphi_{n}(t) = t$ for any $t \in \mathbb{R}$ and $\abs{\varphi_{n}(t) - \varphi_{n}(s)} \le \abs{t - s}$ for any $n \in \mathbb{N}$ and any $s,t \in \mathbb{R}$.
        Then $\{ \varphi_{n}(u) \}_{n \in \mathbb{N}} \subseteq \mathcal{F}$ and $\lim_{n \to \infty}\mathcal{E}(u - \varphi_{n}(u)) = 0$ for any $u \in \mathcal{F}$.
        \item \label{approx2-2} Let $u \in \mathcal{F}$, $\{ u_{n} \}_{n \in \mathbb{N}} \subseteq \mathcal{F}$ and $\varphi \in \contfunc(\mathbb{R})$ satisfy $\lim_{n \to \infty}\mathcal{E}(u - u_{n}) = 0$, $\lim_{n \to \infty}u_{n}(x) = u(x)$ for some $x \in X$, $\abs{\varphi(t) - \varphi(s)} \le \abs{t - s}$ for any $s,t \in \mathbb{R}$, and $\varphi(u) = u$.
        Then $\{ \varphi(u_{n}) \}_{n \in \mathbb{N}} \subseteq \mathcal{F}$ and $\lim_{n \to \infty}\mathcal{E}(u - \varphi(u_{n})) = 0$.
    \end{enumerate}
\end{cor}
\begin{proof}
    \ref{approx2-1}:
	This is immediate from \ref{RF5} and Corollary \ref{cor.approx1}-\ref{approx1-2}.
	
	\ref{approx2-2}:
    For any $y \in X$, by \eqref{R-basic} in Proposition \ref{prop.R-conseq}-\ref{it:RF.basic-ineq} we have
    \[
    \abs{u(y) - u_{n}(y)}
    \le R_{\mathcal{E}}(x,y)^{1/p}\mathcal{E}(u - u_{n})^{1/p} + \abs{u(x) - u_{n}(x)}
    \xrightarrow[n \to \infty]{}
    0,
    \]
    and hence $\lim_{n \to \infty}\varphi(u_{n}(y)) = \varphi(u(y)) = u(y)$. 
    By \ref{RF5} and the triangle inequality for $\mathcal{E}^{1/p}$ we also have $\{ \varphi(u_{n}) \}_{n \in \mathbb{N}} \subseteq \mathcal{F}$ and $\limsup_{n \to \infty}\mathcal{E}(\varphi(u_{n})) \leq \lim_{n \to \infty}\mathcal{E}(u_{n}) = \mathcal{E}(u)$. 
	Thus $\lim_{n \to \infty}\mathcal{E}(u - \varphi(u_{n})) = 0$ by Corollary \ref{cor.approx1}-\ref{approx1-2}.  
\end{proof}

In the following proposition, we record a useful variant of Theorem \ref{thm.inductive}.  
\begin{prop}\label{prop.harmapprox}
    Let $\{ \mathcal{V}_{n} \}_{n \in \mathbb{N} \cup \{ 0 \}}$ be a non-decreasing sequence of non-empty finite subsets of $X$, and set $\mathcal{V}_{\ast} \coloneqq \bigcup_{n \in \mathbb{N} \cup \{ 0 \}}\mathcal{V}_{n}$. 
    If 
	\begin{equation}\label{e:lineariso.Vast}
    	\text{the map $\mathcal{F} \ni u \mapsto  u\vert_{\mathcal{V}_{\ast}} \in \mathcal{F}\vert_{\mathcal{V}_{\ast}}$ is injective (and hence a linear isomorphism),} 
    \end{equation}
    then \textup{(note that $\{ \mathcal{E}\vert_{\mathcal{V}_{n}}(u\vert_{\mathcal{V}_{n}}) \}_{n \in \mathbb{N} \cup \{ 0 \}}$ is non-decreasing since $\{ \mathcal{V}_{n} \}_{n \in \mathbb{N} \cup \{ 0 \}}$ is non-decreasing),} 
    \begin{align}
    \mathcal{F}|_{\mathcal{V}_{\ast}} 
    &= \Bigl\{ u \in \mathbb{R}^{\mathcal{V}_{\ast}} \Bigm\vert \lim_{n\to\infty}\mathcal{E}\vert_{\mathcal{V}_{n}}(u\vert_{\mathcal{V}_{n}})<\infty \Bigr\}, \label{monolim-dom}\\ 
		\mathcal{E}(u; v) &= \lim_{n \to \infty}\mathcal{E}|_{\mathcal{V}_{n}}(u|_{\mathcal{V}_{n}}; v|_{\mathcal{V}_{n}}) \quad \text{for any $u,v \in \mathcal{F}$,} \label{compat-conv} \\
		\lim_{n \to \infty}&\mathcal{E}\bigl(u - h_{\mathcal{V}_{n}}^{\mathcal{E}}[u|_{\mathcal{V}_{n}}]\bigr) = 0 \quad \text{for any $u \in \mathcal{F}$.}\label{harmapprox} 
    \end{align}
	In particular, if $\closure{\mathcal{V}_{\ast}}^{X} = X$, then \eqref{e:lineariso.Vast}, \eqref{monolim-dom}, \eqref{compat-conv} and \eqref{harmapprox} hold, and 
	\begin{equation}\label{monolim-dom-dense}
    	\mathcal{F} = \Bigl\{ u \in \contfunc(X) \Bigm\vert \lim_{n\to\infty}\mathcal{E}\vert_{\mathcal{V}_{n}}(u\vert_{\mathcal{V}_{n}})<\infty \Bigr\}. 
	\end{equation}
\end{prop}
\begin{proof}
    Assume \eqref{e:lineariso.Vast}.  
	By Theorem \ref{thm.inductive}, we have $\mathcal{F}\vert_{\mathcal{V}_{\ast}} \subseteq \bigl\{ u \in \mathbb{R}^{\mathcal{V}_{\ast}} \bigm\vert \lim_{n\to\infty}\mathcal{E}\vert_{\mathcal{V}_{n}}(u\vert_{\mathcal{V}_{n}})<\infty \bigr\}$ and $\mathcal{E}(u) \ge \lim_{n \to \infty}\mathcal{E}\vert_{\mathcal{V}_{n}}(u\vert_{\mathcal{V}_{n}})$ for any $u \in \mathcal{F}$. 
    To show the converse, let $u \in \mathbb{R}^{\mathcal{V}_{\ast}}$ satisfy $\lim_{n \to \infty}\mathcal{E}\vert_{\mathcal{V}_{n}}(u\vert_{\mathcal{V}_{n}}) < \infty$, set $u_{n} \coloneqq h_{\mathcal{V}_{n}}^{\mathcal{E}}(u\vert_{\mathcal{V}_{n}}) \in \mathcal{F}$ for each $n \in \mathbb{N} \cup \{ 0 \}$ and fix $x_{0} \in \mathcal{V}_{0}$. 
    We can assume that $u(x_{0}) = 0$ by considering $u - u(x_{0})$ instead of $u$.
    A similar estimate to \eqref{cauchy1} or \eqref{cauchy2} for $\mathcal{E}$ and \ref{RF2} together imply that $\lim_{n \to \infty}\mathcal{E}(v - u_{n}) = 0$ for some $v \in \mathcal{F}$ with $v(x_0) = 0$. 
    Since $\abs{v(x) - u(x)}^{p} \le R_{\mathcal{E}}(x,x_{0})\mathcal{E}(v - u_{n})$ for any $x \in \mathcal{V}_{\ast}$ and any $n \in \mathbb{N}$ with $x \in \mathcal{V}_{n}$ by \eqref{R-basic}, we get $u = v\vert_{\mathcal{V}_{\ast}} \in \mathcal{F}\vert_{\mathcal{V}_{\ast}}$, proving $\mathcal{F}\vert_{\mathcal{V}_{\ast}} \supseteq \bigl\{ u \in \mathbb{R}^{\mathcal{V}_{\ast}} \bigm\vert \lim_{n\to\infty}\mathcal{E}\vert_{\mathcal{V}_{n}}(u\vert_{\mathcal{V}_{n}})<\infty \bigr\}$ and thereby \eqref{monolim-dom}. 
    We then have \eqref{harmapprox} by \eqref{e:lineariso.Vast} and $\lim_{n \to \infty}\mathcal{E}(v - u_{n}) = 0$, and obtain \eqref{compat-conv} from \eqref{harmapprox}, the H\"{o}lder-type estimate \eqref{bdd.form}, the continuity estimate \eqref{ncont} in Theorem \ref{thm.p-form}, and the expression \eqref{diff-compat} for $\mathcal{E}|_{\mathcal{V}_{n}}(\,\cdot\,;\,\cdot\,)$ in Theorem \ref{thm.RF-exist}.
	
	Lastly, if $\closure{\mathcal{V}_{\ast}}^{X} = X$, then since $\mathcal{F} \subseteq \contfunc(X)$ by the estimate \eqref{R-basic} in Proposition \ref{prop.R-conseq}-\ref{it:RF.basic-ineq}, we have \eqref{e:lineariso.Vast}, hence \eqref{monolim-dom}--\eqref{harmapprox} hold, and \eqref{monolim-dom-dense} follows from \eqref{monolim-dom} and $\mathcal{F} \subseteq \contfunc(X)$.  
\end{proof}

Based on Proposition \ref{prop.harmapprox}, standard machinery for constructing the ``inductive limit'' of $p$-energy forms on p.-c.f.\ self-similar structures can be stated in Theorems \ref{thm.Epcountable} and \ref{thm.Rpcompletion} below, which are extensions of the counterpart for resistance forms given in \cite[Lemma 2.24, Theorem 2.25 and Corollary 2.43]{Kaj.ln} to $p$-resistance forms. 
This approach will be used in Subsection \ref{sec.pcf}, where the construction of $p$-energy forms due to \cite{CGQ22} is reviewed. See also \cite[Sections 2.2, 2.3 and 3.3]{Kig01} for the details in the case of $p = 2$. 
\begin{defn}[Compatible sequence\index{compatible sequence} of $p$-resistance forms on finite sets]\label{defn.compatseq} 
	Let $\mathcal{V}_{n}$ be a non-empty finite set and let $\mathcal{E}^{(n)}$ be a $p$-resistance form on $\mathcal{V}_{n}$ for each $n \in \mathbb{N} \cup \{ 0 \}$. 
	We say that the sequence $\mathcal{S} \coloneqq \{ (\mathcal{V}_{n}, \mathcal{E}^{(n)}) \}_{n \in \mathbb{N} \cup \{ 0 \}}$ is a \emph{compatible sequence of $p$-resistance forms} if and only if $\mathcal{V}_{n} \subseteq \mathcal{V}_{n + 1}$ and $\mathcal{E}^{(n + 1)}|_{\mathcal{V}_{n}} = \mathcal{E}^{(n)}$ for any $n \in \mathbb{N} \cup \{ 0 \}$. 
\end{defn}

\begin{thm}\label{thm.Epcountable}
	Let $\mathcal{S} = \{ (\mathcal{V}_{n}, \mathcal{E}^{(n)}) \}_{n \in \mathbb{N} \cup \{ 0 \}}$ be a compatible sequence of $p$-resistance forms. 
	We define 
	$\mathcal{V}_{\ast} \coloneqq \bigcup_{n \in \mathbb{N} \cup \{ 0 \}}\mathcal{V}_{n}$, 
	\begin{gather}
	    \mathcal{F}_{\mathcal{S}} \coloneqq \Bigl\{ u \in \mathbb{R}^{\mathcal{V}_{\ast}} \Bigm| \lim_{n \to \infty}\mathcal{E}^{(n)}(u|_{\mathcal{V}_{n}}) < \infty \Bigr\},  \quad \text{and} \label{e:defn.compat.dom} \\
        \mathcal{E}_{\mathcal{S}}(u) \coloneqq \lim_{n \to \infty}\mathcal{E}^{(n)}(u|_{\mathcal{V}_{n}}), \quad u \in \mathcal{F}_{\mathcal{S}}.  \label{e:defn.compat.form}
    \end{gather}
	Then $(\mathcal{E}_{\mathcal{S}},\mathcal{F}_{\mathcal{S}})$ is a $p$-resistance form on $\mathcal{V}_{\ast}$ and $\mathcal{E}_{\mathcal{S}}|_{\mathcal{V}_{n}} = \mathcal{E}^{(n)}$ for any $n \in \mathbb{N} \cup \{ 0 \}$. 
\end{thm}
\begin{proof}
	Noting that $\{ \mathcal{E}^{(n)}(u|_{\mathcal{V}_{n}}) \}_{n \in \mathbb{N} \cup \{ 0 \}}$ is non-decreasing for any $u \in \mathbb{R}^{\mathcal{V}_{\ast}}$, we easily obtain \ref{RF1} for $(\mathcal{E}_{\mathcal{S}},\mathcal{F}_{\mathcal{S}})$. 
	To see \ref{RF5} for $(\mathcal{E}_{\mathcal{S}},\mathcal{F}_{\mathcal{S}})$, let $n_{1},n_{2} \in \mathbb{N}$, $q_{1} \in (0,p]$, $q_{2} \in [p,\infty]$ and $T = (T_{1},\dots,T_{n_{2}}) \colon \mathbb{R}^{n_{1}} \to \mathbb{R}^{n_{2}}$ satisfy \eqref{GC-cond} in Definition \ref{defn.GC}, and let $\bm{u} = (u_{1},\dots,u_{n_{1}}) \in \mathcal{F}_{\mathcal{S}}^{n_{1}}$.
    Then, for any $l \in \{ 1,\dots,n_{2} \}$, \ref{GC} for $\mathcal{E}^{(n)}$ implies that 
    \begin{align*}
    \mathcal{E}^{(n)}(T_{l}(\bm{u})|_{\mathcal{V}_{n}})^{1/p}
    &\le \norm{\bigl(\mathcal{E}^{(n)}(T_{l}(\bm{u}|_{\mathcal{V}_{n}}))^{1/p}\bigr)_{l = 1}^{n_{2}}}_{\ell^{q_{2}}} \\
    &\le \norm{\bigl(\mathcal{E}^{(n)}(u_{k}|_{\mathcal{V}_{n}})^{1/p}\bigr)_{k = 1}^{n_{1}}}_{\ell^{q_{1}}}
    \le \norm{\bigl(\mathcal{E}_{\mathcal{S}}(u_{k})^{1/p}\bigr)_{k = 1}^{n_{1}}}_{\ell^{q_{1}}} < \infty.
    \end{align*} 
    By letting $n \to \infty$, we obtain \ref{GC} for $(\mathcal{E}_{\mathcal{S}},\mathcal{F}_{\mathcal{S}})$, i.e., \ref{RF5} for $(\mathcal{E}_{\mathcal{S}},\mathcal{F}_{\mathcal{S}})$ holds. 
	Before proving \ref{RF2}--\ref{RF4} for $(\mathcal{E}_{\mathcal{S}},\mathcal{F}_{\mathcal{S}})$, we shall show the following claim: 
	\begin{equation}\label{e:harmextVstar}
		\begin{minipage}{400pt}
			For any $n \in \mathbb{N} \cup \{ 0 \}$ and any $u \in \mathbb{R}^{\mathcal{V}_{n}}$, there exists a unique $h_{\mathcal{V}_{n}}^{\mathcal{S}}[u] \in \mathcal{F}_{\mathcal{S}}$ such that $h_{\mathcal{V}_{n}}^{\mathcal{S}}[u]\bigr|_{\mathcal{V}_{n}} = u$ and $\mathcal{E}_{\mathcal{S}}\bigl(h_{\mathcal{V}_{n}}^{\mathcal{S}}[u]\bigr) = \min\{ \mathcal{E}_{\mathcal{S}}(v) \mid v \in \mathcal{F}_{\mathcal{S}}, v|_{\mathcal{V}_{n}} = u \} = \mathcal{E}^{(n)}(u)$.  
		\end{minipage}
	\end{equation}
	To prove \eqref{e:harmextVstar}, by \ref{RF1} and \ref{RF5} for $(\mathcal{E}_{\mathcal{S}},\mathcal{F}_{\mathcal{S}})$, we first note that $\#\{ v \in \mathcal{F}_{\mathcal{S}} \mid \mathcal{E}_{\mathcal{S}}(v) = \alpha \} \le 1$,   where $\alpha \coloneqq \min\{ \mathcal{E}_{\mathcal{S}}(v) \mid v \in \mathcal{F}_{\mathcal{S}}, v|_{\mathcal{V}_{n}} = u \}$. 
	(Recall the arguments in \eqref{cauchy1} and \eqref{cauchy2} in the proof of Theorem \ref{thm.RF-exist}.)
	Hence it suffices to show the existence of the minimizer realizing $\alpha$. 
	For any $k_{2} \ge k_{1} \ge n$, we have $h_{\mathcal{V}_{n}}^{\mathcal{E}^{(k_{2})}}[u]\bigr|_{\mathcal{V}_{k_{1}}} = h_{\mathcal{V}_{n}}^{\mathcal{E}^{(k_{1})}}[u]$ by $\mathcal{E}^{(k_{2})}|_{\mathcal{V}_{k_{1}}} = \mathcal{E}^{(k_{1})}$ and Proposition \ref{prop.trace-comp}, which implies that $u_{\ast}(x) \coloneqq h_{\mathcal{V}_{n}}^{\mathcal{E}^{(k)}}[u](x)$ for $x \in \mathcal{V}_{k}$ with $k \ge n$ is well-defined. 
	Clearly, $u_{\ast}|_{\mathcal{V}_{n}} = u$. 
	For any $k \ge n$, we have $\mathcal{E}^{(k)}(u_{\ast}|_{\mathcal{V}_{k}}) = \mathcal{E}^{(k + 1)}(u_{\ast}|_{\mathcal{V}_{k + 1}})$ by Proposition \ref{prop.trace-comp} again, whence $u_{\ast} \in \mathcal{F}_{\mathcal{S}}$ and $\mathcal{E}_{\mathcal{S}}(u_{\ast}) = \mathcal{E}^{(n)}(u)$. 
	Since $\mathcal{E}^{(n)}(u) \le \mathcal{E}_{\mathcal{S}}(v)$ for any $v \in \mathcal{F}_{\mathcal{S}}$ with $v|_{\mathcal{V}_{n}} = u$, we also get $\mathcal{E}_{\mathcal{S}}(u_{\ast}) = \alpha$, so $h_{\mathcal{V}_{n}}^{\mathcal{S}}[u] \coloneqq u_{\ast}$ is the desired function. 
	
	Now let us go back to the proof of \ref{RF2}--\ref{RF4}. 
	\begin{enumerate}[label=\textup{(RF\arabic*)$_{p}$:},align=left,leftmargin=*,topsep=2pt,parsep=0pt,itemsep=2pt]
		\item[\ref{RF3}:] This is immediate since $\mathcal{F}_{\mathcal{S}}|_{\mathcal{V}_{n}} = \mathbb{R}^{V_{n}}$ for any $n \in \mathbb{N} \cup \{ 0 \}$ by \eqref{e:harmextVstar}. 
		\item[\ref{RF4}:] Let $x,y \in \mathcal{V}_{\ast}$ with $x \neq y$ and let $n \in \mathbb{N} \cup \{ 0 \}$ satisfy $x,y \in \mathcal{V}_{n}$. 
    	Let $u \coloneqq h_{\{ x,y \}}^{\mathcal{E}^{(n)}}\bigl[\indicator{x}^{\{ x,y \}}\bigr] \in \mathbb{R}^{\mathcal{V}_{n}}$. 
    	Then for any $v \in \mathcal{F}_{\mathcal{S}}$ with $v|_{\{ x,y \}} = \indicator{x}^{\{ x,y \}}$, 
    	\begin{align*}
    		\mathcal{E}_{\mathcal{S}}(v) 
    		\overset{\eqref{e:harmextVstar}}{\ge} \mathcal{E}^{(n)}(v|_{\mathcal{V}_{n}}) 
    		\ge R_{\mathcal{E}^{(n)}}(x,y)^{-1} = \mathcal{E}^{(n)}(u) 
    		\overset{\eqref{e:harmextVstar}}{=} \mathcal{E}_{\mathcal{S}}\bigl(h_{\mathcal{V}_{n}}^{\mathcal{S}}[u]\bigr). 
    	\end{align*} 
   		Therefore, we have 
   		\begin{equation}\label{e:pRMinductive}
   			R_{\mathcal{E}_{\mathcal{S}}}(x,y) 
   			= \mathcal{E}_{\mathcal{S}}\bigl(h_{\mathcal{V}_{n}}^{\mathcal{S}}[u]\bigr)^{-1} 
   			= R_{\mathcal{E}^{(n)}}(x,y) < \infty. 
   		\end{equation}
   		\item[\ref{RF2}:] Fix $x_{\ast} \in \mathcal{V}_{\ast}$, and let $\{ u_{k} \}_{k \in \mathbb{N}} \subseteq \mathcal{F}_{\mathcal{S}}$ satisfy $u_{k}(x_{\ast}) = 0$ for any $k \in \mathbb{N}$ and $\lim_{k \wedge l \to \infty}\mathcal{E}_{\mathcal{S}}(u_{k} - u_{l}) = 0$. 
    	From \ref{RF4}, $\{ u_{k}(x) \}_{k \in \mathbb{N}}$ is a Cauchy sequence in $\mathbb{R}$ for any $x \in \mathcal{V}_{\ast}$, so we can define $u \in \mathbb{R}^{V_{\ast}}$ by $u(x) \coloneqq \lim_{k \to \infty}u_{k}(x)$.
    	Let $\varepsilon \in (0,\infty)$. Then there exists $N_{0} \in \mathbb{N}$ such that $\sup_{k,l \ge N_{0}}\mathcal{E}_{\mathcal{S}}(u_{k} - u_{l}) \le \varepsilon$.
    	Since $\mathcal{E}^{(n)}(\,\cdot\,)^{1/p}$ is a norm on the finite-dimensional vector space $\mathbb{R}^{\mathcal{V}_{n}}/\mathbb{R}\indicator{\mathcal{V}_{n}}$, we obtain
    	\[
    	\mathcal{E}^{(n)}(u|_{\mathcal{V}_{n}} - u_{l}|_{\mathcal{V}_{n}}) 
    	\le \liminf_{k \to \infty}\mathcal{E}_{\mathcal{S}}(u_{k} - u_{l}) \le \varepsilon \quad \text{for any $l \ge N_{0}$ and any $n \in \mathbb{N} \cup \{ 0 \}$.}
    	\]
    	Letting $n \to \infty$ here, for any $l \ge N_{0}$ we obtain $u - u_{l} \in \mathcal{F}_{\mathcal{S}}$, therefore $u = (u - u_{l}) + u_{l} \in \mathcal{F}_{\mathcal{S}}$, also $\mathcal{E}_{\mathcal{S}}(u - u_{l}) \leq \varepsilon$, and thus $\lim_{l \to \infty}\mathcal{E}_{\mathcal{S}}(u - u_{l}) = 0$, which proves that $(\mathcal{F}_{\mathcal{S}}/\mathbb{R}\indicator{\mathcal{V}_{\ast}},\mathcal{E}_{\mathcal{S}}^{1/p})$ is a Banach space. 
	\end{enumerate}

	Now we know that $(\mathcal{E}_{\mathcal{S}},\mathcal{F}_{\mathcal{S}})$ is a $p$-resistance form on $\mathcal{V}_{\ast}$. 
	Then \eqref{e:harmextVstar} means that $h_{\mathcal{V}_{n}}^{\mathcal{S}} = h_{\mathcal{V}_{n}}^{\mathcal{E}_{\mathcal{S}}}[u]$ for any $u \in \mathbb{R}^{\mathcal{V}_{n}}$, whence $\mathcal{E}_{\mathcal{S}}|_{\mathcal{V}_{n}} = \mathcal{E}^{(n)}$ by \eqref{e:harmextVstar} again. 
\end{proof}

The following theorem yields a $p$-resistance form on the completion of $(X,R_{\mathcal{E}}^{1/p})$. 
\begin{thm}\label{thm.Rpcompletion} 
	Let $(Y,\theta)$ be the completion of the metric space $(X,R_{\mathcal{E}}^{1/p})$. 
	Define $\overline{\mathcal{F}} \subseteq \mathbb{R}^{Y}$ and $\overline{\mathcal{E}} \colon \overline{\mathcal{F}} \to [0,\infty)$ by  
	\begin{gather}
	    \overline{\mathcal{F}} \coloneqq \bigl\{ u \in \contfunc(Y) \bigm| u|_{X} \in \mathcal{F} \bigr\}, \label{e:defn.completion.dom} \\
        \overline{\mathcal{E}}(u) \coloneqq \mathcal{E}(u|_{X}), \quad u \in \overline{\mathcal{F}}.  \label{e:defn.completion.form}
    \end{gather}
    Then $(\overline{\mathcal{E}}, \overline{\mathcal{F}})$ is a $p$-resistance form on $Y$, $R_{\overline{\mathcal{E}}}^{1/p} = \theta$, and the map $\overline{\mathcal{F}} \ni u \mapsto u|_{X} \in \mathcal{F}$ is a linear isomorphism. 
\end{thm}
\begin{proof}
	Set $\overline{R}(x,y) \coloneqq \theta(x,y)^{p}$ for ease of notation, so that $\overline{R}\,\big|_{X \times X} = R_{\mathcal{E}}$. 
	For any $u \in \mathcal{F}$, we know that $u$ is uniformly continuous with respect to $\theta$ by the estimate \eqref{R-basic} for $(\mathcal{E},\mathcal{F})$, so there exists a unique $u_{\ast} \in \contfunc(Y)$ satisfying $u_{\ast}|_{X} = u$ and then $u_{\ast} \in \overline{\mathcal{F}}$. 
	This implies that the map $\overline{\mathcal{F}} \ni u \mapsto u|_{X} \in \mathcal{F}$ is a bijection and thus it is a linear isomorphism.
	Moreover, for any $u \in \overline{\mathcal{F}}$, the function $\eta_{u} \colon Y \times Y \to \mathbb{R}$ defined by $\eta_{u}(x,y) \coloneqq \abs{u(x) - u(y)}^{p} - \overline{R}(x,y)\overline{\mathcal{E}}(u)$ is $(-\infty,0]$-valued on $X \times X$ by \eqref{R-basic} for $(\mathcal{E},\mathcal{F})$ and hence also on $Y \times Y$ by its continuity, i.e.,
	\begin{equation}\label{e:R-basic.ext}
		\abs{u(x) - u(y)}^{p} \le \overline{R}(x,y)\overline{\mathcal{E}}(u) \quad \textrm{for any $x,y \in Y$.} 
	\end{equation}
	Now we show \ref{RF1}--\ref{RF5} for $(\overline{\mathcal{E}}, \overline{\mathcal{F}})$ and $R_{\overline{\mathcal{E}}} = \overline{R}$. 
	
	\ref{RF1}: 
	Clearly, $\overline{\mathcal{F}}$ is a linear subspace of $\mathbb{R}^{Y}$ including $\mathbb{R}\indicator{Y}$ and $\overline{\mathcal{E}}(\,\cdot\,)^{1/p}$ is a seminorm on $\overline{\mathcal{F}}$. By $\indicator{Y}|_{X} = \indicator{X}$ and \ref{RF1} for $(\mathcal{E},\mathcal{F})$, it holds that $\{ u \in \overline{\mathcal{F}} \mid \overline{\mathcal{E}}(u) = 0 \} = \mathbb{R}\indicator{Y}$. 
	
	\ref{RF2}: 
	This is immediate from \ref{RF2} for $(\mathcal{E},\mathcal{F})$ since $\overline{\mathcal{F}} \ni u \mapsto u|_{X} \in \mathcal{F}$ is a linear isomorphism. 
	
	\ref{RF4} and $R_{\overline{\mathcal{E}}} \leq \overline{R}$:
	This is immediate from \eqref{e:R-basic.ext} and the definition of $R_{\overline{\mathcal{E}}}$ in \eqref{R-def}.
	
	\ref{RF5}:
	This is immediate from \ref{RF5} for $(\mathcal{E},\mathcal{F})$.
	
	\ref{RF3} and $R_{\overline{\mathcal{E}}} = \overline{R}$:
	Let $x,y \in Y$ satisfy $x \neq y$, and choose $\{ x_{n} \}_{n \in \mathbb{N}}, \{ y_{n} \}_{n \in \mathbb{N}} \subseteq X$ so that $\lim_{n \to \infty}\overline{R}(x,x_{n}) = 0 = \lim_{n \to \infty}\overline{R}(y,y_{n})$. 
	We can assume that $x_{n} \neq y_{n}$ for any $n \in \mathbb{N}$. 
	For each $n \in \mathbb{N}$, let $u_{n} \in \overline{\mathcal{F}}$ be the unique function satisfying $u_{n}|_{X} = h_{\{ x_{n},y_{n} \}}^{\mathcal{E}}\bigl[\indicator{x_{n}}^{\{ x_{n}, y_{n} \}}\bigr]$, so that by \eqref{e:R-basic.ext} we have
	\begin{align}
	\overline{\mathcal{E}}(u_{n}) = \mathcal{E}(u_{n}|_{X}) = R_{\mathcal{E}}(x_{n},y_{n})^{-1} &= \overline{R}(x_{n},y_{n})^{-1} \xrightarrow[n \to \infty]{} \overline{R}(x,y)^{-1}, \label{eq:Rpcompletion-RF3-norm} \\
	\abs{1 - u_{n}(x)} = \abs{u_{n}(x_{n}) - u_{n}(x)} &\overset{\eqref{e:R-basic.ext}}{\leq} \overline{R}(x_{n},x)^{1/p} \overline{\mathcal{E}}(u_{n})^{1/p} \xrightarrow[n \to \infty]{} 0, \label{eq:Rpcompletion-RF3-x} \\
	\abs{u_{n}(y)} = \abs{u_{n}(y_{n}) - u_{n}(y)} &\overset{\eqref{e:R-basic.ext}}{\leq} \overline{R}(y_{n},y)^{1/p} \overline{\mathcal{E}}(u_{n})^{1/p} \xrightarrow[n \to \infty]{} 0. \label{eq:Rpcompletion-RF3-y} 
	\end{align}
	In particular, $\sup_{n \in \mathbb{N}}\overline{\mathcal{E}}(u_{n}) < \infty$ by \eqref{eq:Rpcompletion-RF3-norm}, and since $(\overline{\mathcal{F}}/\mathbb{R}\indicator{Y}, \overline{\mathcal{E}}^{1/p})$ is isometrically isomorphic to $(\mathcal{F}/\mathbb{R}\indicator{X}, \mathcal{E}^{1/p})$ and hence is a reflexive Banach space by Proposition \ref{prop.R-conseq}-\ref{it:RF.unif-conv}, some subsequence $\{ u_{n_{k}} \}_{k \in \mathbb{N}}$ of $\{ u_{n} \}_{n \in \mathbb{N}}$ converges weakly in $(\overline{\mathcal{F}}/\mathbb{R}\indicator{Y}, \overline{\mathcal{E}}^{1/p})$ to some $u \in \overline{\mathcal{F}}$ with $u(y) = 0$ by the Banach--Alaoglu theorem (see, e.g., \cite[Theorem 1 in Section V.2]{Yos}).
	Noting that $\overline{\mathcal{F}} \ni v \mapsto v(x) - v(y)$ defines a bounded linear functional on $(\overline{\mathcal{F}}/\mathbb{R}\indicator{Y}, \overline{\mathcal{E}}^{1/p})$ by \eqref{e:R-basic.ext}, we now see from the weak convergence of $\{ u_{n_{k}} \}_{k \in \mathbb{N}}$ to $u$ in $(\overline{\mathcal{F}}/\mathbb{R}\indicator{Y}, \overline{\mathcal{E}}^{1/p})$, together with $u(y) = 0$, \eqref{eq:Rpcompletion-RF3-x}, \eqref{eq:Rpcompletion-RF3-y}, the definition of $R_{\overline{\mathcal{E}}}(x,y)$ in \eqref{R-def}, \eqref{eq:Rpcompletion-RF3-norm} and \eqref{e:R-basic.ext}, that
	\begin{align*}
	u(x) = u(x) - u(y) &= \lim_{k \to \infty}(u_{n_{k}}(x) - u_{n_{k}}(y)) \overset{\eqref{eq:Rpcompletion-RF3-x},\eqref{eq:Rpcompletion-RF3-y}}{=} 1 \not= 0 = u(y), \\
	R_{\overline{\mathcal{E}}}(x,y)^{-1} \overset{\eqref{R-def}}{\leq} \overline{\mathcal{E}}(u) &\leq \lim_{k \to \infty}\overline{\mathcal{E}}(u_{n_{k}}) \overset{\eqref{eq:Rpcompletion-RF3-norm}}{=} \overline{R}(x,y)^{-1} \overset{\eqref{e:R-basic.ext},\eqref{R-def}}{\leq} R_{\overline{\mathcal{E}}}(x,y)^{-1},
	\end{align*}
	proving \ref{RF3} and $R_{\overline{\mathcal{E}}} = \overline{R}$.
\end{proof}

\begin{cor}\label{cor.Epext} 
	Let $\mathcal{S} = \{ (\mathcal{V}_{n}, \mathcal{E}^{(n)}) \}_{n \in \mathbb{N} \cup \{ 0 \}}$ be a compatible sequence of $p$-resistance forms and let $(K,d)$ be the completion of $(\mathcal{V}_{\ast},R_{\mathcal{E}_{\mathcal{S}}}^{1/p})$, where $(\mathcal{E}_{\mathcal{S}},\mathcal{F}_{\mathcal{S}})$ is the $p$-resistance form on $\mathcal{V}_{\ast} = \bigcup_{n \in \mathbb{N} \cup \{ 0 \}}\mathcal{V}_{n}$ given in Theorem \ref{thm.Epcountable}.
	Define $\overline{\mathcal{F}}_{\mathcal{S}} \subseteq \mathbb{R}^{K}$ and $\overline{\mathcal{E}}_{\mathcal{S}} \colon \overline{\mathcal{F}}_{\mathcal{S}} \to [0,\infty)$ by  
	\begin{gather}
	    \overline{\mathcal{F}}_{\mathcal{S}} \coloneqq \bigl\{ u \in \contfunc(K) \bigm| u|_{\mathcal{V}_{\ast}} \in \mathcal{F}_{\mathcal{S}} \bigr\} =  \Bigl\{ u \in \contfunc(K) \Bigm| \lim_{n \to \infty}\mathcal{E}^{(n)}(u|_{\mathcal{V}_{n}}) < \infty \Bigr\}, \label{e:defn.compatext.dom} \\
        \overline{\mathcal{E}}_{\mathcal{S}}(u) \coloneqq \mathcal{E}_{\mathcal{S}}(u|_{\mathcal{V}_{\ast}}) =  \lim_{n \to \infty}\mathcal{E}^{(n)}(u|_{\mathcal{V}_{n}}), \quad u \in \overline{\mathcal{F}}_{\mathcal{S}}.  \label{e:defn.compatext.form}
    \end{gather}
    Then $(\overline{\mathcal{E}}_{\mathcal{S}}, \overline{\mathcal{F}}_{\mathcal{S}})$ is a $p$-resistance form on $K$, $R_{\overline{\mathcal{E}}_{\mathcal{S}}}^{1/p} = d$, and the map $\overline{\mathcal{F}}_{\mathcal{S}} \ni u \mapsto u|_{\mathcal{V}_{\ast}} \in \mathcal{F}_{\mathcal{S}}$ is a linear isomorphism. 
    In particular, $\overline{\mathcal{E}}_{\mathcal{S}}\big|_{\mathcal{V}_{n}} = \mathcal{E}^{(n)}$ for any $n \in \mathbb{N} \cup \{ 0 \}$. 
\end{cor}
\begin{proof}
	We obtain the desired assertions by applying Theorem \ref{thm.Rpcompletion} with $\mathcal{V}_{\ast}$, $(\mathcal{E}_{\mathcal{S}},\mathcal{F}_{\mathcal{S}})$ in place of $X$, $(\mathcal{E},\mathcal{F})$. 
	Also, by $\mathcal{E}_{\mathcal{S}}|_{\mathcal{V}_{n}} = \mathcal{E}^{(n)}$ (see Theorem \ref{thm.Epcountable}) and the fact that $\overline{\mathcal{F}}_{\mathcal{S}} \ni u \mapsto u|_{\mathcal{V}_{\ast}} \in \mathcal{F}_{\mathcal{S}}$ is a linear isomorphism, we have $\overline{\mathcal{E}}_{\mathcal{S}}\big|_{\mathcal{V}_{n}} = \mathcal{E}^{(n)}$. 
\end{proof}

We conclude this subsection with a discussion of strong locality of $p$-resistance forms.
\begin{defn}[Strong locality\index{strong locality (of $p$-resistance form)} of $p$-resistance form]\label{defn.RF-sl}
	\begin{enumerate}[label=\textup{(\arabic*)},align=left,leftmargin=*,topsep=2pt,parsep=0pt,itemsep=2pt]
		\item\label{it:SL1s.pRF} We say that $(\mathcal{E},\mathcal{F})$ has the strong local property\index{strong local property (of $p$-resistance form)} \hyperref[it:SL1s.pRF]{\textup{(SL1)$_{\mathrm{s}}$}} if and only if, for any $u_{1},u_{2},v \in \mathcal{F}$ with either $\supp_{X}[u_{1} - a_{1}\indicator{X}]$ or $\supp_{X}[u_{2} - a_{2}\indicator{X}]$ compact and $(u_{1}(x) - a_{1})(u_{2}(x) - a_{2}) = 0$ for any $x \in X$ for some $a_{1},a_{2} \in \mathbb{R}$, 
			\begin{equation}\label{e:defn.sl1.pRF}
				\mathcal{E}(u_{1} + u_{2} + v) + \mathcal{E}(v) = \mathcal{E}(u_{1} + v) + \mathcal{E}(u_{2} + v). 
			\end{equation}
		\item\label{it:SL2s.pRF} We say that $(\mathcal{E}, \mathcal{F})$ has the strong local property \hyperref[it:SL2s.pRF]{\textup{(SL2)$_{\mathrm{s}}$}}, or $(\mathcal{E}, \mathcal{F})$ is \emph{strongly local}\index{strongly local ($p$-resistance form)}, if and only if, for any $u_{1}, u_{2}, v \in \mathcal{F}$ with either $\supp_{X}[u_{1} - u_{2} - a\indicator{X}]$ or $\supp_{X}[v - b\indicator{X}]$ compact and $(u_{1}(x) - u_{2}(x) - a)(v(x) - b) = 0$ for any $x \in X$ for some $a,b \in \mathbb{R}$, 
			\begin{equation}\label{e:defn.sl2.pRF}
				\mathcal{E}(u_{1}; v) = \mathcal{E}(u_{2}; v). 
			\end{equation}
		\item\label{it:SL1w.pRF} We say that $(\mathcal{E},\mathcal{F})$ has the strong local property \hyperref[it:SL1w.pRF]{\textup{(SL1)$_{\mathrm{w}}$}} if and only if \hyperref[it:SL1s.pRF]{\textup{(SL1)$_{\mathrm{s}}$}} with ``$(u_{1}(x) - a_{1})(u_{2}(x) - a_{2}) = 0$ for any $x \in X$'' replaced by ``$\supp_{X}[u_{1} - a_{1}\indicator{X}] \cap \supp_{X}[u_{2} - a_{2}\indicator{X}] = \emptyset$'' holds. 
		\item\label{it:SL2w.pRF} We say that $(\mathcal{E},\mathcal{F})$ has the strong local property \hyperref[it:SL2w.pRF]{\textup{(SL2)$_{\mathrm{w}}$}} if and only if \hyperref[it:SL2s.pRF]{\textup{(SL2)$_{\mathrm{s}}$}} with ``$(u_{1}(x) - u_{2}(x) - a)(v(x) - b) = 0$ for any $x \in X$'' replaced by ``$\supp_{X}[u_{1} - u_{2} - a\indicator{X}] \cap \supp_{X}[v - b\indicator{X}] = \emptyset$'' holds. 
	\end{enumerate}
\end{defn}

Note that \hyperref[it:SL1w.pRF]{\textup{(SL1)$_{\mathrm{w}}$}} and \hyperref[it:SL2w.pRF]{\textup{(SL2)$_{\mathrm{w}}$}} are exactly \hyperref[it:SL1]{\textup{(SL1)}} and \hyperref[it:SL2]{\textup{(SL2)}}, respectively, in Definition \ref{defn.Epsl} with ``$\supp_{m}$'' replaced by ``$\supp_{X}$''. 
In the following proposition, we discuss relations among the strong local properties \hyperref[it:SL1s.pRF]{\textup{(SL1)$_{\mathrm{s}}$}}, \hyperref[it:SL2s.pRF]{\textup{(SL2)$_{\mathrm{s}}$}}, \hyperref[it:SL1w.pRF]{\textup{(SL1)$_{\mathrm{w}}$}} and \hyperref[it:SL2w.pRF]{\textup{(SL2)$_{\mathrm{w}}$}} introduced in Definition \ref{defn.RF-sl}. 
\begin{prop}\label{prop.RF.SL}
	\begin{enumerate}[label=\textup{(\alph*)},align=left,leftmargin=*,topsep=2pt,parsep=0pt,itemsep=2pt]
		\item\label{it:RF.SL2w1w} If $X$ is locally compact and $(\mathcal{E},\mathcal{F})$ is regular and satisfies \hyperref[it:SL2w.pRF]{\textup{(SL2)$_{\mathrm{w}}$}}, then $(\mathcal{E},\mathcal{F})$ satisfies \hyperref[it:SL1w.pRF]{\textup{(SL1)$_{\mathrm{w}}$}}.  
		\item\label{it:RF.SL1w2w} If $(\mathcal{E},\mathcal{F})$ satisfies \hyperref[it:SL1w.pRF]{\textup{(SL1)$_{\mathrm{w}}$}}, then $(\mathcal{E},\mathcal{F})$ satisfies \hyperref[it:SL2w.pRF]{\textup{(SL2)$_{\mathrm{w}}$}}. 
		\item\label{it:RF.SL1sw} $(\mathcal{E},\mathcal{F})$ satisfies \hyperref[it:SL1s.pRF]{\textup{(SL1)$_{\mathrm{s}}$}} if and only if $(\mathcal{E},\mathcal{F})$ satisfies \hyperref[it:SL1w.pRF]{\textup{(SL1)$_{\mathrm{w}}$}}. 
		\item\label{it:RF.SL2sw} $(\mathcal{E},\mathcal{F})$ satisfies \hyperref[it:SL2s.pRF]{\textup{(SL2)$_{\mathrm{s}}$}} if and only if $(\mathcal{E},\mathcal{F})$ satisfies \hyperref[it:SL2w.pRF]{\textup{(SL2)$_{\mathrm{w}}$}}. 
	\end{enumerate}
	
	In particular, if $X$ is locally compact and $(\mathcal{E},\mathcal{F})$ is regular, then \hyperref[it:SL1s.pRF]{\textup{(SL1)$_{\mathrm{s}}$}}, \hyperref[it:SL2s.pRF]{\textup{(SL2)$_{\mathrm{s}}$}}, \hyperref[it:SL1w.pRF]{\textup{(SL1)$_{\mathrm{w}}$}} and \hyperref[it:SL2w.pRF]{\textup{(SL2)$_{\mathrm{w}}$}} are equivalent to each other. 
\end{prop} 
\begin{proof}
	The proofs of \ref{it:RF.SL2w1w} and \ref{it:RF.SL1w2w} are very similar to those of Proposition \ref{prop.sl-other}-\ref{it:SL2-SL1},\ref{it:SL1-SL2}.

	\ref{it:RF.SL2w1w}: Since $(\mathcal{E},\mathcal{F})$ satisfies \eqref{e:sl-leibniz}, \eqref{e:sl-bddapprox} and \eqref{e:sl-special} with $m$ the counting measure on $X$ by Proposition \ref{prop.GC-list}-\ref{GC.leibniz}, Corollary \ref{cor.approx2}-\ref{approx2-1} and Proposition \ref{prop.regular}, the implication from \hyperref[it:SL2w.pRF]{\textup{(SL2)$_{\mathrm{w}}$}} to \hyperref[it:SL1w.pRF]{\textup{(SL1)$_{\mathrm{w}}$}} is proved in exactly the same way as the proof of Proposition \ref{prop.sl-other}-\ref{it:SL2-SL1} (note that the separability of $X$ is used there only to define $\supp_{m}[\,\cdot\,]$). 
	
	\ref{it:RF.SL1w2w}: This is proved in exactly the same way as the proof of Proposition \ref{prop.sl-other}-\ref{it:SL1-SL2}. 
	
	\ref{it:RF.SL1sw}: The implication from \hyperref[it:SL1s.pRF]{\textup{(SL1)$_{\mathrm{s}}$}} to \hyperref[it:SL1w.pRF]{\textup{(SL1)$_{\mathrm{w}}$}} is obvious. 
	Conversely, assume \hyperref[it:SL1w.pRF]{\textup{(SL1)$_{\mathrm{w}}$}}, let $u_{1},u_{2},v \in \mathcal{F}$, $a_{1},a_{2} \in \mathbb{R}$ and assume that either $\supp_{X}[u_{1} - a_{1}\indicator{X}]$ or $\supp_{X}[u_{2} - a_{2}\indicator{X}]$ is compact and $(u_{1}(x) - a_{1})(u_{2}(x) - a_{2}) = 0$ for any $x \in X$. 
	For $n \in \mathbb{N}$, let $\varphi_{n} \in \contfunc(\mathbb{R})$ be given by $\varphi_{n}(t) \coloneqq t-(-\frac{1}{n})\vee(t\wedge\frac{1}{n})$ and set $u_{1,n} \coloneqq \varphi_{n}(u_{1} - a_{1}\indicator{X})$ and $u_{2,n} \coloneqq \varphi_{n}(u_{2} - a_{2}\indicator{X})$, so that $u_{i,n} \in \mathcal{F}$ and $\lim_{n \to \infty}\mathcal{E}(u_{i} - u_{i,n}) = 0$ for $i \in \{ 1,2 \}$ by Corollary \ref{cor.approx2}-\ref{approx2-1} and \ref{RF1}. 
	Then for each $n \in \mathbb{N}$, since $\supp_{X}[u_{1,n}] \cap \supp_{X}[u_{2,n}] = \emptyset$ and either $\supp_{X}[u_{1,n}]$ or $\supp_{X}[u_{2,n}]$ is compact by the assumptions on $u_{1},u_{2}$, it follows from \hyperref[it:SL1w.pRF]{\textup{(SL1)$_{\mathrm{w}}$}} that $\mathcal{E}(u_{1,n} + u_{2,n} + v) + \mathcal{E}(v) = \mathcal{E}(u_{1,n} + v) + \mathcal{E}(u_{2,n} + v)$, and we obtain $\mathcal{E}(u_{1} + u_{2} + v) + \mathcal{E}(v) = \mathcal{E}(u_{1} + v) + \mathcal{E}(u_{2} + v)$ by letting $n \to \infty$, proving \hyperref[it:SL1s.pRF]{\textup{(SL1)$_{\mathrm{s}}$}}. 
	
	\ref{it:RF.SL2sw}: The implication from \hyperref[it:SL2s.pRF]{\textup{(SL2)$_{\mathrm{s}}$}} to \hyperref[it:SL2w.pRF]{\textup{(SL2)$_{\mathrm{w}}$}} is obvious. 
	Conversely, assume \hyperref[it:SL2w.pRF]{\textup{(SL2)$_{\mathrm{w}}$}}, let $u_{1},u_{2},v \in \mathcal{F}$, $a,b \in \mathbb{R}$ and assume that either $\supp_{X}[u_{1} - u_{2} - a\indicator{X}]$ or $\supp_{X}[v - b\indicator{X}]$ is compact and $(u_{1}(x) - u_{2}(x) - a)(v(x) - b) = 0$ for any $x \in X$. 
	For $n \in \mathbb{N}$, set $v_{n} \coloneqq \varphi_{n}(v - b\indicator{X})$ , where $\varphi_{n}$ is the same as in the proof of \ref{it:RF.SL1sw}, so that $v_{n} \in \mathcal{F}$ and $\lim_{n \to \infty}\mathcal{E}(v - v_{n}) = 0$ by Corollary \ref{cor.approx2}-\ref{approx2-1} and \ref{RF1}. 
	Then for each $n \in \mathbb{N}$, since $\supp_{X}[u_{1} - u_{2} - a\indicator{X}] \cap \supp_{X}[v_{n}] = \emptyset$ and either $\supp_{X}[u_{1} - u_{2} - a\indicator{X}]$ or $\supp_{X}[v_{n}]$ is compact by the assumptions on $u_{1},u_{2},v$, it follows from \hyperref[it:SL2w.pRF]{\textup{(SL2)$_{\mathrm{w}}$}} that $\mathcal{E}(u_{1}; v_{n}) = \mathcal{E}(u_{2}; v_{n})$, and we obtain $\mathcal{E}(u_{1}; v) = \mathcal{E}(u_{2}; v)$ by letting $n \to \infty$, proving \hyperref[it:SL2s.pRF]{\textup{(SL2)$_{\mathrm{s}}$}}. 
\end{proof}

\subsection{Weak comparison principles}
In this subsection, we show some weak comparison principles in this context.
The first one is obtained as an application of the strong subadditivity.
\begin{prop}[Weak comparison principle I\index{weak comparison principle}]\label{prop.cp1}
	Let $B$ be a non-empty subset of $X$.
	Then, for any $u, v \in \mathcal{F}|_{B}$ satisfying $u(y) \le v(y)$ for any $y \in B$, it holds that
	\begin{equation}\label{e:comparison}
        h_{B}^{\mathcal{E}}[u](x) \le h_{B}^{\mathcal{E}}[v](x) \quad \text{for any $x \in X$.}
    \end{equation}
    In particular,
    \begin{equation}\label{mp}
	       \inf_{B}u \le h_{B}^{\mathcal{E}}[u](x) \le \sup_{B}u \quad \text{for any $x \in X$.}
    \end{equation}
\end{prop}
\begin{proof}
    Let $f \coloneqq h_{B}^{\mathcal{E}}[u]$ and $g \coloneqq h_{B}^{\mathcal{E}}[v]$.
    We will prove $f \wedge g = f$, which precisely means \eqref{e:comparison}.
    Since $(f \wedge g)|_{B} = u$ and $(f \vee g)|_{B} = v$, we have
    \begin{equation*}
        \mathcal{E}(f) \le \mathcal{E}(f \wedge g) \quad \text{and} \quad \mathcal{E}(g) \le \mathcal{E}(f \vee g).
    \end{equation*}
    By the strong subadditivity \eqref{sadd} in Proposition \ref{prop.GC-list}-\ref{GC.markov}, we obtain $\mathcal{E}(f \wedge g) = \mathcal{E}(f)$ (and $\mathcal{E}(f \vee g) = \mathcal{E}(g)$), which together with the uniqueness of $f = h_{B}^{\mathcal{E}}[u]$ in Theorem \ref{thm.RF-exist} yields $f \wedge g = f$ and proves \eqref{e:comparison}.
    Then \eqref{mp} follows from \eqref{e:comparison} and the fact that $h_{B}^{\mathcal{E}}[a\indicator{B}] = a\indicator{X}$ for any $a \in \mathbb{R}$ by \eqref{harm-const} in Theorem \ref{thm.RF-exist}.
\end{proof}

We can extend the weak comparison principle above to arbitrary open subsets if $X$ is locally compact and $(\mathcal{E},\mathcal{F})$ is regular and strongly local; see Proposition \ref{prop.cp2} below.
This version of weak comparison principle will be used to prove the \emph{strong comparison principle} on p.-c.f.\ self-similar structures in a forthcoming paper \cite{KS.scp}.  
We begin with some preparations. Recall \hyperref[it:SL2s.pRF]{\textup{(SL2)$_{\mathrm{s}}$}} in Definition \ref{defn.RF-sl}-\ref{it:SL2s.pRF} for the definition of $(\mathcal{E},\mathcal{F})$ being strongly local.
\begin{defn}\label{defn.RFlocal}
    Let $U$ be a non-empty open subset of $X$.
    \begin{enumerate}[label=\textup{(\arabic*)},align=left,leftmargin=*,topsep=2pt,parsep=0pt,itemsep=2pt]
        \item \label{it:RF.localDir} We define
        \begin{equation*}
            \mathcal{F}_{\mathrm{loc}}(U) \coloneqq
            \biggl\{ f \in \mathbb{R}^{U} \biggm|
            \begin{minipage}{193pt}
                $f|_{V} = f^{\#}|_{V}$ for some $f^{\#} \in \mathcal{F}$ for each relatively compact open subset $V$ of $U$ 
            \end{minipage}
            \biggr\}.
        \end{equation*}
        \item \label{it:RF.local.harm} Assume that $(\mathcal{E},\mathcal{F})$ is strongly local, and let $V$ be an open subset of $U$.
        A function $h \in \mathcal{F}_{\mathrm{loc}}(U)$ is said to be \emph{$\mathcal{E}$-harmonic}\index{$\mathcal{E}$-harmonic} on $V$ if and only if $\mathcal{E}(h^{\#}; \varphi) = 0$ for any $\varphi \in \mathcal{F}$ such that $\supp_{X}[\varphi]$ is a compact subset of $V$, where $h^{\#} \in \mathcal{F}$ is chosen so as to satisfy $h|_{\supp_{X}[\varphi]} = h^{\#}|_{\supp_{X}[\varphi]}$.
    \end{enumerate}
\end{defn}
\begin{rmk}\label{rmk.RFlocal}
	\begin{enumerate}[label=\textup{(\arabic*)},align=left,leftmargin=*,topsep=2pt,parsep=0pt,itemsep=2pt]
		\item\label{it:rmk.RF.localDir} If $X \eqqcolon K$ comes from a self-similar structure and the topology induced by $R_{\mathcal{E}}^{1/p}$ coincides with the original topology of $K$, then the definition of $\mathcal{F}_{\mathrm{loc}}(U)$ above agrees with the one introduced in \eqref{e:defn.Floc} by virtue of $\mathcal{F} \subseteq \contfunc(K)$. 
		\item\label{it:rmk.RF.local.harm} Assume the situation of Definition \ref{defn.RFlocal}-\ref{it:RF.local.harm}. Then the strong locality \hyperref[it:SL2s.pRF]{\textup{(SL2)$_{\mathrm{s}}$}} ensures that the value $\mathcal{E}(h^{\#}; \varphi)$ is independent of a particular choice of $h^{\#}$. 
			Moreover, \emph{if $V$ is relatively compact in $U$, $h \in \mathcal{F}_{\mathrm{loc}}(U)$ is $\mathcal{E}$-harmonic on $V$ and $h^{\#} \in \mathcal{F}$ satisfies $h|_{V} = h^{\#}|_{V}$, then $h^{\#} \in \mathcal{H}_{\mathcal{E},X \setminus V}$};
			indeed, for any $v \in \mathcal{F}^{0}(V)$, since $v_{n} \coloneqq v - \bigl(-\frac{1}{n}\bigr) \vee \bigl(v \wedge \frac{1}{n}\bigr)$ satisfies $\supp_{X}[v_{n}] \subseteq V$ for any $n \in \mathbb{N}$,
			we see from Corollary \ref{cor.approx2}-\ref{approx2-1} and the H\"{o}lder-type estimate \eqref{bdd.form} in Theorem \ref{thm.p-form} that $\mathcal{E}(h^{\#}; v) = \lim_{n\to\infty}\mathcal{E}(h^{\#}; v_{n}) = 0$.
	\end{enumerate} 
\end{rmk}

We need the following proposition to prove the desired weak comparison principle. 
\begin{prop}\label{prop.localcut}
    Assume that $X$ is locally compact and that $(\mathcal{E},\mathcal{F})$ is regular and strongly local.
    Let $U$ be a relatively compact non-empty open subset of $X$ and let $u \in \mathcal{F}$ satisfy $u(x) = 0$ for any $x \in \partial_{X}U$, where $\partial_{X}U \coloneqq \closure{U}^{X}\setminus U$.
    Then $u\indicator{U} \in \mathcal{F}$.
\end{prop}
\begin{proof}
    Define $\varphi_{n} \in \contfunc(\mathbb{R})$ by $\varphi_{n}(t) \coloneqq t - \bigl(-\frac{1}{n}\bigr) \vee \bigl(t \wedge \frac{1}{n}\bigr)$ and set $A_{n} \coloneqq U \cap \supp_{X}[\varphi_{n}(u)]$ for each $n \in \mathbb{N}$.
    Since $u|_{\partial_{X}U} = 0$, $A_{n} = \closure{U}^{X} \cap \supp_{X}[\varphi_{n}(u)]$ and thus $A_{n}$ is a compact subset of $U$.
    By Proposition \ref{prop.regular}, there exists $v_{n} \in \mathcal{F}$ such that $\indicator{A_{n}} \le v_{n} \le \indicator{U}$.
    Then we easily obtain $\varphi_{n}(u)\indicator{U} = \varphi_{n}(u)\indicator{A_{n}} = \varphi_{n}(u)v_{n}$, and hence by Corollary \ref{cor.approx2}-\ref{approx2-1} and Proposition \ref{prop.GC-list}-\ref{GC.leibniz} we have $\varphi_{n}(u)\indicator{U} \in \mathcal{F}$.
    Now, recalling that \hyperref[it:SL1w.pRF]{\textup{(SL1)$_{\mathrm{w}}$}} in Definition \ref{defn.RF-sl}-\ref{it:SL1w.pRF} holds by \hyperref[it:SL2s.pRF]{\textup{(SL2)$_{\mathrm{s}}$}} and Proposition \ref{prop.RF.SL}-\ref{it:RF.SL2sw},\ref{it:RF.SL2w1w}, we see from \hyperref[it:SL1w.pRF]{\textup{(SL1)$_{\mathrm{w}}$}} and Corollary \ref{cor.approx2}-\ref{approx2-1} that $\mathcal{E}(\varphi_{k}(u)\indicator{U} - \varphi_{l}(u)\indicator{U}) \leq \mathcal{E}(\varphi_{k}(u) - \varphi_{l}(u)) \xrightarrow[k \wedge l \to \infty]{} 0$,
	and thus by \ref{RF2} and the estimate \eqref{R-basic} in Proposition \ref{prop.R-conseq}-\ref{it:RF.basic-ineq}, $\{ \varphi_{n}(u)\indicator{U} \}_{n \in \mathbb{N}}$ converges in norm in $(\mathcal{F}/\mathbb{R}\indicator{X}, \mathcal{E}^{1/p})$ to its pointwise limit $u\indicator{U}$, whence $u\indicator{U} \in \mathcal{F}$.
\end{proof}

Now we can state the desired version of the weak comparison principle.
\begin{prop}[Weak comparison principle II\index{weak comparison principle}]\label{prop.cp2}
    Assume that $X$ is locally compact and that $(\mathcal{E},\mathcal{F})$ is regular and strongly local.
    Let $U$ be a relatively compact non-empty open subset of $X$ such that $U \not= X$.
	If $u,v \in \contfunc(\closure{U}^{X}) \cap \mathcal{F}_{\mathrm{loc}}(U)$ are $\mathcal{E}$-harmonic on $U$ and $u(x) \le v(x)$ for any $x \in \partial_{X}U$, then $u(x) \le v(x)$ for any $x \in \closure{U}^{X}$.
\end{prop}
\begin{proof}
	We first observe that $\partial_{X}O \neq \emptyset$ for any relatively compact non-empty open subset $O$ of $X$ such that $O \not= X$.  
	To this end, suppose that $\partial_{X}O = \emptyset$, and we will show $O = X$. 
	Since $O = \closure{O}^{X}$ is compact, we see from Proposition \ref{prop.regular} that $\indicator{O} \in \mathcal{F} \cap \contfunc_{c}(X)$,
	then $\mathcal{E}(\indicator{O}) = \mathcal{E}(\indicator{O}; \indicator{O}) = \mathcal{E}(0; \indicator{O}) = 0$ by the strong locality \hyperref[it:SL2s.pRF]{\textup{(SL2)$_{\mathrm{s}}$}},
	and hence $\indicator{O} \in \mathbb{R}\indicator{X}$ by \ref{RF1}, which together with $O \not= \emptyset$ shows $O = X$. 
	
	Let us go back to the proof. Let $\varepsilon \in (0,\infty)$.
    Since $u$ and $v$ are uniformly continuous on $\closure{U}^{X}$ and $\partial_{X}U \neq \emptyset$, there exists $\delta \in (0,\infty)$ such that 
    \[
    V \coloneqq \Bigl\{ x \in U \Bigm| \dist_{R_{\mathcal{E}}^{1/p}}(x, \partial_{X}U) > \delta \Bigr\} \neq \emptyset
    \]
    and $u(x) \le v(x) + \varepsilon$ for any $x \in \closure{U}^{X} \setminus V$. 
    Then $V$ is a relatively compact open subset of $U$ and hence there exist $u^{\#}, v^{\#} \in \mathcal{F}$ such that $u|_{V} = u^{\#}|_{V}$ and $v|_{V} = v^{\#}|_{V}$. 
    Define $v^{\#}_{\varepsilon} \coloneqq v^{\#} + \varepsilon \indicator{X} \in \mathcal{F}$, $f \coloneqq u^{\#} - (u^{\#} - v^{\#}_{\varepsilon})^{+}\indicator{X \setminus V}$ and $g \coloneqq v^{\#}_{\varepsilon} + (u^{\#} - v^{\#}_{\varepsilon})^{+}\indicator{X \setminus V}$.
    Then since $u^{\#}(x) = u(x) \leq v(x) + \varepsilon = v^{\#}_{\varepsilon}(x)$ for any $x \in \partial_{X}V \subseteq \closure{U}^{X} \setminus V$, we have $f,g \in \mathcal{F}$ by Propositions \ref{prop.GC-list}-\ref{GC.lip} and \ref{prop.localcut}.
    We also have $f,g \in \mathcal{H}_{\mathcal{E},X \setminus V}$ by Remark \ref{rmk.RFlocal}-\ref{it:rmk.RF.local.harm},
    $f(x) = (u^{\#} \wedge v^{\#}_{\varepsilon})(x) \leq (u^{\#} \vee v^{\#}_{\varepsilon})(x) = g(x)$ for any $x \in X \setminus V$, and thus Proposition \ref{prop.cp1} implies that $u(x) = u^{\#}(x) = f(x) \le g(x) = v^{\#}_{\varepsilon}(x) = v(x) + \varepsilon$ for any $x \in V$.
    Therefore, we conclude that $u(x) \le v(x) + \varepsilon$ for any $x \in \closure{U}^{X}$, which completes the proof since $\varepsilon \in (0,\infty)$ is arbitrary.
\end{proof}

\subsection{Sharp H\"{o}lder regularity of harmonic functions}\label{sec.sharp}
In this subsection, we present a sharp H\"{o}lder regularity estimate on $\mathcal{E}$-harmonic functions and prove that $R_{\mathcal{E}}^{1/(p - 1)}$ is a metric on $X$.

As an application of Proposition \ref{prop.mono}, we can show the following H\"{o}lder continuity estimate for $\mathcal{E}$-harmonic functions. 
Recall Definition \ref{defn.RFp-general} for $B^{\mathcal{F}}$ and \eqref{R-def.gen} for $R_{\mathcal{E}}(x,B)$.
\begin{thm}\label{t:lip-harm}
    Let $B$ be a non-empty subset of $X$, $x \in X \setminus B^{\mathcal{F}}$ and $y \in X$. Then 
    \begin{equation}\label{hB}
        h_{B \cup \{ x \}}^{\mathcal{E}}\bigl[\indicator{B}^{B \cup \{ x \}}\bigr](y)
        \le \frac{R_{\mathcal{E}}(x,y)^{1/(p - 1)}}{R_{\mathcal{E}}(x,B)^{1/(p - 1)}}.
    \end{equation}
    Moreover, for any $h \in \mathcal{H}_{\mathcal{E},B}$ with $\sup_{B}\abs{h} < \infty$, 
    \begin{equation}\label{lip-harm}
        \abs{h(x) - h(y)} \le \frac{R_{\mathcal{E}}(x,y)^{1/(p - 1)}}{R_{\mathcal{E}}(x,B)^{1/(p - 1)}}\osc_{B}[h]. 
    \end{equation}
\end{thm}
\begin{proof}
    Since \eqref{hB} and \eqref{lip-harm} are obvious if $x = y$, we may and do assume that $x \neq y$.
	
	To show \eqref{hB}, on one hand, we see that
    \begin{align}\label{hB1}
        -\mathcal{E}|_{B \cup \{ x \}}(\indicator{B}; \indicator{x})
        &= \mathcal{E}|_{B \cup \{ x \}}(\indicator{B}; \indicator{B \cup \{ x \}}) -\mathcal{E}|_{B \cup \{ x \}}(\indicator{B}; \indicator{x}) \nonumber \\
        &= \mathcal{E}|_{B \cup \{ x \}}(\indicator{B}; \indicator{B})
        = R_{\mathcal{E}}(x,B)^{-1}.
    \end{align}
    On the other hand,
    \begin{align}\label{hB2}
        &-\mathcal{E}|_{B \cup \{ x \}}(\indicator{B}; \indicator{x}) \nonumber \\
        &\quad = -\mathcal{E}\Bigl(h_{B \cup \{ x \}}^{\mathcal{E}}[\indicator{B}]; h_{B \cup \{ x, y \}}^{\mathcal{E}}[\indicator{x}]\Bigr) \quad \text{(by \eqref{harm-compat} in Theorem \ref{thm.RF-exist})} \nonumber \\
        &\quad = -\mathcal{E}|_{B \cup \{ x, y \}}\Bigl(\bigl.h_{B \cup \{ x \}}^{\mathcal{E}}[\indicator{B}]\bigr|_{B \cup \{ x, y \}}; \indicator{x}\Bigr) \quad \text{(by Proposition \ref{prop.trace-comp} and \eqref{harm-compat})} \nonumber \\
        &\quad \ge -\mathcal{E}|_{B \cup \{ x, y \}}\Bigl(\bigl.\bigl(h_{B \cup \{ x \}}^{\mathcal{E}}[\indicator{B}](y) \cdot h_{\{x,y\}}^{\mathcal{E}}[\indicator{y}]\bigr)\bigr|_{B \cup \{ x,y \}}; \indicator{x}\Bigr) \quad \text{(by Proposition \ref{prop.mono})} \nonumber \\
        &\quad = -h_{B \cup \{ x \}}^{\mathcal{E}}[\indicator{B}](y)^{p - 1}\mathcal{E}|_{B \cup \{ x, y \}}\Bigl(\bigl.h_{\{x,y\}}^{\mathcal{E}}[\indicator{y}]\bigr|_{B \cup \{ x,y \}}; \indicator{x}\Bigr) \nonumber \\
        &\quad = -h_{B \cup \{ x \}}^{\mathcal{E}}[\indicator{B}](y)^{p - 1}\mathcal{E}|_{\{ x, y \}}(\indicator{y}; \indicator{\{x,y\}} - \indicator{y}) \quad \text{(by Proposition \ref{prop.trace-comp} and \eqref{harm-compat})} \nonumber \\
		&\quad = h_{B \cup \{ x \}}^{\mathcal{E}}[\indicator{B}](y)^{p - 1}R_{\mathcal{E}}(x,y)^{-1}.
    \end{align} 
    We obtain \eqref{hB} by combining \eqref{hB1} and \eqref{hB2}.

    Next we prove \eqref{lip-harm}.
    Let $h \in \mathcal{H}_{\mathcal{E},B}$ satisfy $\sup_{B}\abs{h} < \infty$.
    Then we see that 
    \begin{align*}
        h - h(x)
        &\le h_{B \cup \{ x \}}^{\mathcal{E}}\Bigl[\bigl.(h - h(x))^{+}\bigr|_{B \cup \{ x \}}\Bigr] \quad \text{(by Propositions \ref{prop.cp1} and \ref{prop.trace-comp})} \nonumber \\
        &\le h_{B \cup \{ x \}}^{\mathcal{E}}\Bigl[\osc_{B}[h]\cdot\indicator{B}^{B \cup \{x\}}\Bigr] \quad \text{(by Proposition \ref{prop.cp1} and $(h - h(x))^{+}(x) = 0$)} \nonumber \\
        &= \osc_{B}[h] \cdot h_{B \cup \{ x \}}^{\mathcal{E}}\Bigl[\indicator{B}^{B \cup \{x\}}\Bigr].
    \end{align*}
    Similarly, we have
    \begin{equation*}
        h - h(x)
        \ge -h_{B \cup \{ x \}}^{\mathcal{E}}\Bigl[\bigl.(h - h(x))^{-}\bigr|_{B \cup \{ x \}}\Bigr]
        \ge -\osc_{B}[h] \cdot h_{B \cup \{ x \}}^{\mathcal{E}}\Bigl[\indicator{B}^{B \cup \{x\}}\Bigr].
    \end{equation*}
    Hence, by combining these estimates with \eqref{hB}, we get \eqref{lip-harm}. 
\end{proof}

Using Theorem \ref{t:lip-harm}, we can show the triangle inequality for $R_{\mathcal{E}}^{1/(p - 1)}$. 
\begin{cor}\label{cor.tri}
    $R_{\mathcal{E}}^{1/(p - 1)} \colon X \times X \to [0, \infty)$ is a metric on $X$.
\end{cor}
\begin{defn}[$p$-Resistance metric\index{$p$-resistance metric}]\label{defn:pResmet}
	We define $\pmetric_{p,\mathcal{E}} \coloneqq R_{\mathcal{E}}^{1/(p - 1)}$ and call $\pmetric_{p,\mathcal{E}}$ the \emph{$p$-resistance metric} of $(\mathcal{E},\mathcal{F})$. 
\end{defn}
\begin{proof}[Proof of Corollary \ref{cor.tri}]
    It suffices to prove the triangle inequality $R_{\mathcal{E}}(x,z)^{1/(p - 1)} \le R_{\mathcal{E}}(x,y)^{1/(p - 1)} + R_{\mathcal{E}}(y,z)^{1/(p - 1)}$ for any $x,y,z \in X$ with $\#\{ x,y,z \} = 3$.
    By \eqref{hB} with $B = \{ z \}$ we have $h_{\{ x, z \}}^{\mathcal{E}}\bigl[\indicator{x}^{\{ x, z \}}\bigr](y) \le \frac{R_{\mathcal{E}}(x,y)^{1/(p - 1)}}{R_{\mathcal{E}}(x,z)^{1/(p - 1)}}$, and interchanging the roles of $x$ and $z$ yields $h_{\{ x, z \}}^{\mathcal{E}}\bigl[\indicator{z}^{\{ x, z \}}\bigr](y) \le \frac{R_{\mathcal{E}}(y,z)^{1/(p - 1)}}{R_{\mathcal{E}}(x,z)^{1/(p - 1)}}$.
    Since $\indicator{X} = h_{\{ x, z \}}^{\mathcal{E}}\bigl[\indicator{x}^{\{ x, z \}}\bigr] + h_{\{ x, z \}}^{\mathcal{E}}\bigl[\indicator{z}^{\{ x, z \}}\bigr]$ by \eqref{harm-const} in Theorem \ref{thm.RF-exist}, we obtain
    \[
    1 \le \frac{R_{\mathcal{E}}(x,y)^{1/(p - 1)}}{R_{\mathcal{E}}(x,z)^{1/(p - 1)}} + \frac{R_{\mathcal{E}}(y,z)^{1/(p - 1)}}{R_{\mathcal{E}}(x,z)^{1/(p - 1)}},
    \]
    which proves the desired triangle inequality for $R_{\mathcal{E}}^{1/(p - 1)}$.
\end{proof}

\begin{example}\label{ex:pmet.sharp}
	Let $p \in (1,\infty)$ and $(\mathcal{E},\mathcal{F})$ be a $p$-resistance form on the unit open interval $(0,1)$ given by 
	\[
	\mathcal{F} \coloneqq W^{1,p}(0,1) \quad \text{and} \quad \mathcal{E}(u) \coloneqq \int_{0}^{1}\abs{\nabla u}^{p}\,dx 
	\]
	(recall Example \ref{ex.pRF}-\ref{RF-Rn}). 
	For any $x,y \in (0,1)$ with $0 < x < y < 1$, we easily see that $u \in W^{1,p}(0,1)$ defined by $u(t) \coloneqq (y - x)^{-1}(t - x)\indicator{[x,y]}(t)$, $t \in (0,1)$, is $\mathcal{E}$-harmonic on $(0,1) \setminus \{ x,y \}$. 
	Therefore we have $R_{\mathcal{E}}(x,y) = (y - x)^{p - 1}$ and the $p$-resistance metric $\widehat{R}_{p,\mathcal{E}}$ coincides with the Euclidean metric on $(0,1)$. 
	In particular, the H\"{o}lder regularity estimate \eqref{lip-harm} is sharp. 
	This example also shows that exponent $1/(p - 1)$ in the $p$-resistance metric is sharp, that is, $R_{\mathcal{E}}^{\alpha}$ is not a metric for $\alpha > 1/(p - 1)$ in general.   
\end{example}

\subsection{Elliptic Harnack inequality for non-negative superharmonic functions}\label{sec.EHI}
Throughout this subsection, we assume that $\{ \Gamma\langle u \rangle \}_{u \in \mathcal{F}}$ is a family of $p$-energy measures on $(X,\mathcal{B}(X))$ dominated by $(\mathcal{E},\mathcal{F})$ and satisfies \ref{Cp-em}.   
For ease of the notation, we set $\pmetric_{p} \coloneqq \pmetric_{p,\mathcal{E}} = R_{\mathcal{E}}^{1/(p - 1)}$. 

In this subsection, we establish the elliptic Harnack inequality for non-negative $\mathcal{E}$-superharmonic functions under some extra analytic conditions (Theorem \ref{thm.EHI}).
We mainly follow the argument in \cite{Cap07}, but we assume the two-point estimate \eqref{e:TPE.RF} instead of the Poincar\'{e} inequality \cite[(2.4)]{Cap07} (see also Remark \ref{rmk:TPE-PI}).
Let us start with proving the following log-Caccioppoli inequality under the assumption of the chain rule \hyperref[it:Cha2]{\textup{(Cha2)}} (recall Definition \ref{defn.chainrule}-\ref{it:Cha2}). 

\begin{lem}[Log-Caccioppoli type inequality\index{log-Caccioppoli inequality}]\label{lem.logC}
	Assume that $\{ \Gamma\langle u \rangle \}_{u \in \mathcal{F}}$ satisfies the chain rule \hyperref[it:Cha2]{\textup{(Cha2)}}.
	Then there exists $C \in (0,\infty)$ (depending only on $p$) such that for any $A,\varepsilon \in (0,\infty)$ with $A > 1$, any $(x,s) \in X \times (0,\infty)$ and any $u \in \mathcal{F}$ such that $u \ge 0$ on $X$, $u$ is $\mathcal{E}$-superharmonic on $B_{\pmetric_{p}}(x,As)$ and $\Gamma\langle u \rangle(X) = \mathcal{E}(u)$, it holds that 
	\begin{equation}\label{e:logC}
		\int_{B_{\pmetric_{p}}(x,s)}\,d\Gamma\langle \Phi_{\varepsilon}(u) \rangle \le C\inf\bigl\{ \mathcal{E}(\varphi) \bigm| \text{$\varphi \in \mathcal{F}$, $\varphi|_{B_{\pmetric_{p}}(x,s)} = 1$, $\supp_{X}[\varphi] \subseteq B_{\pmetric_{p}}(x,As)$} \bigr\},  
	\end{equation}
	where $\Phi_{\varepsilon} \in C^{1}(\mathbb{R})$ is any function satisfying $\Phi_{\varepsilon}(x) = \log{(x + \varepsilon)} - \log{\varepsilon}$ for any $x \in [0,\infty)$.
\end{lem}
\begin{proof}
	Let $\varphi_{0} \in \mathcal{F}$ satisfy $\varphi_{0}|_{B_{\pmetric_{p}}(x,s)} = 1$ and $\supp_{X}[\varphi_{0}] \subseteq B_{\pmetric_{p}}(x,As)$ and set $\varphi \coloneqq \varphi_{0}^{+}\wedge 1$.
	Set $u_{\varepsilon} \coloneqq u + \varepsilon$, so that $u_{\varepsilon}$ is $\mathcal{E}$-superharmonic on $B_{\pmetric_{p}}(x,As)$ and $\Gamma\langle u_{\varepsilon} \rangle(X) = \mathcal{E}(u_{\varepsilon})$ by \ref{RF1}, \eqref{form.basic}, \ref{EM1} and \ref{EM2}. 
	Note that $\varphi, u_{\varepsilon}^{1 - p}, \varphi^{p}, \varphi^{p}u_{\varepsilon}^{1 - p} \in \mathcal{F}$ by Proposition \ref{prop.GC-list}-\ref{GC.lip},\ref{GC.leibniz} (see also Corollary \ref{cor.lip-useful}-\ref{GC.compos}), and that for any $v \in \mathcal{F}$ we have $\mathcal{E}(u_{\varepsilon}; v) - \Gamma\langle u_{\varepsilon}; v \rangle(X) = 0$ since it is the derivative at $t = 0$ of the function $\mathbb{R} \ni t \mapsto \frac{1}{p}\bigl(\mathcal{E}(u_{\varepsilon} + tv) - \Gamma\langle u_{\varepsilon} + tv \rangle(X)\bigr)$, which takes its minimum $0$ at $t = 0$ by \ref{EM1} and $\Gamma\langle u_{\varepsilon} \rangle(X) = \mathcal{E}(u_{\varepsilon})$.
	We thus see that 
	\begin{align*} 
		\int_{X}\varphi^{p}\,d\Gamma\langle \Phi_{\varepsilon}(u) \rangle 
		&\overset{\hyperref[it:Cha2]{\textup{(Cha2)}}}{=} \frac{1}{1 - p}\int_{X}\varphi^{p}\,d\Gamma\langle u_{\varepsilon}; u_{\varepsilon}^{1 - p} \rangle \\
		&\overset{\hyperref[it:Cha2]{\textup{(Cha2)}}}{=} \frac{1}{1 - p}\biggl(\int_{X}\,d\Gamma\langle u_{\varepsilon}; \varphi^{p}u_{\varepsilon}^{1 - p} \rangle - \int_{X}u_{\varepsilon}^{1 - p}\,d\Gamma\langle u_{\varepsilon}; \varphi^{p} \rangle\biggr) \\
		&\overset{\textup{($\ast$)}}{=} \frac{1}{1 - p}\biggl(\mathcal{E}(u_{\varepsilon}; \varphi^{p}u_{\varepsilon}^{1 - p}) - \int_{X}u_{\varepsilon}^{1 - p}\,d\Gamma\langle u_{\varepsilon}; \varphi^{p} \rangle\biggr) \\  
		&\overset{\textup{($\ast\ast$)}}{\le} \frac{-1}{1 - p}\int_{X}u_{\varepsilon}^{1 - p}\,d\Gamma\langle u_{\varepsilon}; \varphi^{p} \rangle \\
		&\overset{\hyperref[it:Cha2]{\textup{(Cha2)}}}{=} \frac{p}{p - 1}\int_{X}\varphi^{p - 1}\,d\Gamma\langle \Phi_{\varepsilon}(u); \varphi \rangle \\
		&\overset{\eqref{em.holder}}{\le} \frac{p}{p - 1}\biggl(\frac{1}{2}\int_{X}\varphi^{p}\,d\Gamma\langle \Phi_{\varepsilon}(u) \rangle\biggr)^{\frac{p - 1}{p}}\biggl(2^{p - 1}\int_{X}\,d\Gamma\langle \varphi \rangle\biggr)^{\frac{1}{p}} \\
		&\le \frac{p}{p - 1}\biggl(\frac{p - 1}{2p}\int_{X}\varphi^{p}\,d\Gamma\langle \Phi_{\varepsilon}(u) \rangle + \frac{2^{p - 1}}{p}\int_{X}\,d\Gamma\langle \varphi \rangle\biggr), 
	\end{align*}
	where we used $\mathcal{E}(u_{\varepsilon}; \varphi^{p}u_{\varepsilon}^{1 - p}) - \Gamma\langle u_{\varepsilon}; \varphi^{p}u_{\varepsilon}^{1 - p} \rangle(X) = 0$ in \textup{($\ast$)}, the fact that $u_{\varepsilon}$ is $\mathcal{E}$-superharmonic on $B_{\pmetric_{p}}(x,As)$ in  \textup{($\ast\ast$)}, and Young's inequality in the last inequality.
	Hence we have $\int_{B_{\pmetric_{p}}(x,s)}\,d\Gamma\langle \Phi_{\varepsilon}(u) \rangle \leq \int_{X}\varphi^{p}\,d\Gamma\langle \Phi_{\varepsilon}(u) \rangle \leq (p-1)^{-1}2^{p}\Gamma\langle \varphi \rangle(X)$ and then, recalling that $\Gamma\langle \varphi \rangle(X) \leq \mathcal{E}(\varphi) \leq \mathcal{E}(\varphi_{0})$ by \ref{EM1} and \eqref{lipcont} in Proposition \ref{prop.GC-list}-\ref{GC.lip}, obtain \eqref{e:logC} by taking the infimum over $\varphi_{0}$.
\end{proof}
 
Now we can prove the desired elliptic Harnack inequality as in the following theorem. 
We will see later in Theorem \ref{thm:EHI.pcf} that Theorem \ref{thm.EHI} is applicable to p.-c.f.\ self-similar structures equipped with good self-similar $p$-resistance forms (see Subsection \ref{subsec:TPE} for the precise setting).  
\begin{thm}[Elliptic Harnack inequality\index{elliptic Harnack inequality}]\label{thm.EHI}
	Assume that there exist $\Upsilon \colon X \times (0,\infty) \to (0,\infty)$ and $A_{1}, A_{2},C \in (0,\infty)$ with $A_{1} \ge 1$ and $A_{2} > 1$ such that the following hold:
	\begin{enumerate}[label=\textup{(\roman*)},align=left,leftmargin=*,topsep=2pt,parsep=0pt,itemsep=2pt]
		\item\label{it:EHI.doubling} For any $(x,s) \in X \times (0,\infty)$, 
		\begin{equation}\label{e:VD.upsilon}
			\Upsilon(x,2s) \le C\Upsilon(x,s). 
		\end{equation}
		\item\label{it:EHI.TPE} For any $(x,s) \in X \times (0,\infty)$ and any $u \in \mathcal{F}$, 
		\begin{equation}\label{e:TPE.RF}
			\sup_{y,z \in B_{\pmetric_{p}}(x,s)}\abs{u(y) - u(z)}^{p} \le C\Upsilon(x,s)^{-1}\Gamma\langle u \rangle\bigl(B_{\pmetric_{p}}(x,A_{1}s)\bigr). 
		\end{equation}
		\item\label{it:EHI.capu} For any $(x,s) \in X \times (0,\infty)$ with $B_{\pmetric_{p}}(x,A_{2}s) \neq X$, 
		\begin{equation}\label{e:capu.RF}
			\inf\bigl\{ \mathcal{E}(\varphi) \bigm| \text{$\varphi \in \mathcal{F}$, $\varphi|_{B_{\pmetric_{p}}(x,s)} = 1$, $\supp_{X}[\varphi] \subseteq B_{\pmetric_{p}}(x,A_{2}s)$} \bigr\}
			\le C\Upsilon(x,s).  
		\end{equation}
		\item\label{it:EHI.Cha} $\{ \Gamma\langle u \rangle \}_{u \in \mathcal{F}}$ satisfies the chain rule \hyperref[it:Cha2]{\textup{(Cha2)}}.
	\end{enumerate}
	Then there exist $C_{\mathrm{H}} \in (0,\infty)$ and $\delta_{\mathrm{H}} \in (0,1)$ such that for any $(x,s) \in X \times (0,\infty)$ with $B_{\pmetric_{p}}(x,\delta_{\mathrm{H}}^{-1}s) \neq X$ and any $u \in \mathcal{F}$ such that $u \ge 0$ on $X$, $u$ is $\mathcal{E}$-superharmonic on $B_{\pmetric_{p}}(x,\delta_{\mathrm{H}}^{-1}s)$ and $\Gamma\langle u \rangle(X) = \mathcal{E}(u)$, it holds that 
	\begin{equation}\label{e:EHI}
		\sup_{B_{\pmetric_{p}}(x,s)}u \le C_{\mathrm{H}}\inf_{B_{\pmetric_{p}}(x,s)}u. 
	\end{equation}
\end{thm}
\begin{proof}
	Let $\varepsilon \in (0,\infty)$ and set $\delta_{\mathrm{H}} \coloneqq (A_{1}A_{2})^{-1}$. 
	Let $(x,s) \in X \times (0,\infty)$ satisfy $B_{\pmetric_{p}}(x,\delta_{\mathrm{H}}^{-1}s) \neq X$, and let $u \in \mathcal{F}$ be such that $u \ge 0$ on $X$, $u$ is $\mathcal{E}$-superharmonic on $B_{\pmetric_{p}}(x,\delta_{\mathrm{H}}^{-1}s)$ and $\Gamma\langle u \rangle(X) = \mathcal{E}(u)$. 
	Set $u_{\varepsilon} \coloneqq u + \varepsilon$, $M_{\varepsilon} \coloneqq \sup_{B_{\pmetric_{p}}(x,s)}u_{\varepsilon}$ and $m_{\varepsilon} \coloneqq \inf_{B_{\pmetric_{p}}(x,s)}u_{\varepsilon}$.  
	By \eqref{e:logC}--\eqref{e:capu.RF}, there exists $C_{0} \in (0,\infty)$ independent of $x,s,u,\varepsilon$ such that 
	\begin{align*}
		\sup_{B_{\pmetric_{p}}(x,s)}\log{u_{\varepsilon}} - \inf_{B_{\pmetric_{p}}(x,s)}\log{u_{\varepsilon}}
		\le C_{0}, 
	\end{align*}
	whence $\log{\bigl(\frac{M_{\varepsilon}}{m_{\varepsilon}}\bigr)} \le C_{0}$. 
	Thus $M_{\varepsilon}/m_{\varepsilon} \le e^{C_{0}}$, and we obtain \eqref{e:EHI} by letting $\varepsilon \downarrow 0$. 
\end{proof}

\section{Self-similar \texorpdfstring{$p$}{p}-resistance forms and \texorpdfstring{$p$}{p}-energy measures}\label{sec.compatible}
In this section, we investigate $p$-resistance forms by focusing on the self-similar case as in Section \ref{sec.ss}. 
Throughout this section, we fix $p \in (1,\infty)$ and a self-similar structure $\mathcal{L} = (K,S,\{ F_{i} \}_{i \in S})$ with $\#S \ge 2$ and $K$ connected. 

\subsection{Self-similar \texorpdfstring{$p$}{p}-resistance forms}
We first introduce the notion of \emph{self-similar $p$-resistance form}. \index{self-similar $p$-resistance form} 
\begin{defn}[Self-similar $p$-resistance form] \label{dfn.pRFss}
	Let $\bm{\rweight} = (\rweight_{i})_{i \in S} \in (0,\infty)^{S}$ and let $(\mathcal{E},\mathcal{F})$ be a $p$-resistance form on $K$. 
	We say that $(\mathcal{E},\mathcal{F})$ is a \emph{self-similar $p$-resistance form on $\mathcal{L}$ with weight $\bm{\rweight}$} if and only if $\mathcal{F} \subseteq \contfunc(K)$ and $(\mathcal{E},\mathcal{F})$ satisfies the self-similarity conditions \eqref{SSE1} and \eqref{SSE2} in Definition \ref{defn.ssform} (under the original topology of $K$ implicit in $\mathcal{L} = (K,S,\{ F_{i} \}_{i \in S})$ being a self-similar structure). 
\end{defn}
Throughout the rest of this section except Proposition \ref{prop.RFrenorm} and Theorem \ref{t:KS-mono.pcf}, we fix a self-similar $p$-resistance form $(\mathcal{E},\mathcal{F})$ on $\mathcal{L}$ with weight $\bm{\rweight} = (\rweight_{i})_{i \in S} \in (0,\infty)^{S}$. 
Note that the topology induced by the $p$-resistance metric $\pmetric_{p,\mathcal{E}}$ of $(\mathcal{E},\mathcal{F})$ may be different from the original topology of $K$ implicit in $\mathcal{L} = (K,S,\{ F_{i} \}_{i \in S})$ being a self-similar structure. 
Under the present setting, in referring to a topology of $K$ we always consider its original topology. 
Note also that then $\mathcal{F}$ is dense in $(\contfunc(K),\norm{\,\cdot\,}_{\sup})$ by the Stone--Weierstrass theorem (see, e.g., \cite[Theorem 2.4.11]{Dud}), which applies to $K$ and $\mathcal{F}$ by the compactness of $K$, \ref{RF1}, \ref{RF3} and the fact that $\mathcal{F}$ is an algebra by \eqref{leibniz} in Proposition \ref{prop.GC-list}-\ref{GC.leibniz}. 

The following properties of the $p$-resistance metric are elementary.
\begin{prop}\label{prop.pRMss} 
	\begin{enumerate}[label=\textup{(\arabic*)},align=left,leftmargin=*,topsep=4pt,itemsep=2pt]
		\item\label{pRMcontraction} For any $w \in W_{\ast}$ and any $x,y \in K$,
    	\begin{equation}\label{pRMss}
        	R_{\mathcal{E}}(F_{w}(x),F_{w}(y)) \le \rweight_{w}^{-1}R_{\mathcal{E}}(x,y). 
    	\end{equation}
    	\item\label{pRMcompatible}  If $\min_{i \in S}\rweight_{i} > 1$ and if either $\diam(K,\pmetric_{p,\mathcal{E}}) < \infty$ or $\mathcal{L}$ is a p.-c.f.\ self-similar structure, then $\pmetric_{p,\mathcal{E}}$ is compatible with the original topology of $K$, and in particular, $V_{\ast}$ is dense in $(K,\pmetric_{p,\mathcal{E}})$.  
	\end{enumerate}
\end{prop}
\begin{rmk}
	It is known that, if $p = 2$, $\min_{i \in S}\rweight_{i} > 1$ and $\mathcal{L}$ is a p.-c.f.\ self-similar structure, then there exists $c \in (0,\infty)$ such that for any $w \in W_{\ast}$ and any $x,y \in K$, 
	\begin{equation}\label{e:pRMcontraction}
        R_{\mathcal{E}}(F_{w}(x),F_{w}(y)) \ge c\rweight_{w}^{-1}R_{\mathcal{E}}(x,y);  
    \end{equation}
    see \cite[Theorem A.1]{Kig03}.
	We extend this result to the case of $p \in (1,\infty) \setminus \{2\}$ in Subsection \ref{sec:pcf-contraction}; see Theorem \ref{thm.pcf-contraction}. 
\end{rmk}
\begin{proof}[Proof of Proposition \ref{prop.pRMss}]
	\ref{pRMcontraction}: This is immediate from the self-similarity \eqref{SSE2} of $\mathcal{E}$. (See \cite[Lemma 3.3.5]{Kig01} for the case of $p = 2$.) 
	
	\ref{pRMcompatible}: This is proved by modifying the known arguments for the case of $p = 2$ in \cite{Kig01,Kig09}. Indeed, we can show that $\pmetric_{p,\mathcal{E}}$ is compatible with the original topology of $K$, by following \cite[Proof of Proposition B.1]{Kig09} if $\diam(K,\pmetric_{p,\mathcal{E}}) < \infty$, and by following \cite[Proof of Theorem 3.3.4]{Kig01} if $\mathcal{L}$ is a p.-c.f.\ self-similar structure (see also Lemma \ref{lem.identify} below). 
	Then $V_{\ast}$ is dense in $(K,\pmetric_{p,\mathcal{E}})$ since $\closure{V_{\ast}}^{K} = K$ by \cite[Lemma 1.3.11]{Kig01}.  
\end{proof}

The following proposition presents compatible sequences of $p$-resistance forms having a self-similarity. 
\begin{prop}\label{prop.compatible}
    Let $n \in \mathbb{N} \cup \{ 0 \}$, let $\Lambda$ be a partition of $\Sigma$ and set $V_{n,\Lambda} \coloneqq \bigcup_{w \in \Lambda}F_{w}(V_{n})$. 
    Then for any $u \in \mathcal{F}|_{V_{n,\Lambda}}$,
    \begin{align}\label{eq:compat-seq.partition}
        \mathcal{E}|_{V_{n,\Lambda}}(u) &= \sum_{w \in \Lambda}\rweight_{w}\mathcal{E}|_{V_{n}}(u \circ F_{w}), \\
        h_{V_{n,\Lambda}}^{\mathcal{E}}[u] \circ F_{w} &= h_{V_{n}}^{\mathcal{E}}[u \circ F_{w}] \quad \text{for any $w \in \Lambda$.} 
	\label{ss.harmext}
    \end{align}
    In particular, for any $m \in \mathbb{N} \cup \{ 0 \}$ and any $u \in \mathcal{F}|_{V_{n + m}}$,
    \begin{equation}\label{eq:compat-seq}
        \mathcal{E}|_{V_{n + m}}(u) = \sum_{w \in W_{m}}\rweight_{w}\mathcal{E}|_{V_{n}}(u \circ F_{w}).
    \end{equation}
\end{prop}
\begin{proof}
    Note that \eqref{eq:compat-seq} follows from \eqref{eq:compat-seq.partition} by choosing $\Lambda = W_{m}$ and that the sequence $\mathcal{S} \coloneqq \bigl\{ (V_{n,\Lambda},\mathcal{E}_{V_{n,\Lambda}}) \bigr\}_{n \in \mathbb{N} \cup \{ 0 \}}$ is a compatible sequence of $p$-resistance forms by Proposition \ref{prop.trace-comp}. 
    Let $u \in \mathcal{F}|_{V_{n,\Lambda}}$. 
    Then we see that
    \begin{align*}
        \mathcal{E}|_{V_{n,\Lambda}}(u)
        &= \min\bigl\{ \mathcal{E}(v) \bigm| \text{$v \in \mathcal{F}$ with $v|_{V_{n,\Lambda}} = u$} \bigr\} \\
        &\overset{\eqref{ss.partition}}{=} \min\Biggl\{ \sum_{w \in \Lambda}\rweight_{w}\mathcal{E}(v \circ F_{w}) \Biggm| \text{$v \in \mathcal{F}$ with $v|_{V_{n,\Lambda}} = u$} \Biggr\} \\
        &\ge \sum_{w \in \Lambda}\rweight_{w}\min\bigl\{ \mathcal{E}(v) \bigm| \text{$v \in \mathcal{F}$ with $v|_{V_{n}} = u \circ F_{w}$} \bigr\}
        = \sum_{w \in \Lambda}\rweight_{w}\mathcal{E}|_{V_{n}}(u \circ F_{w}).
    \end{align*}
    To prove the converse, define $v \in \contfunc(K)$ so that $v \circ F_{w} = h_{V_{n}}^{\mathcal{E}}[u \circ F_{w}]$ for any $w \in \Lambda$; note that such $v$ is well-defined by virtue of \eqref{V0bdry}. 
    Then $v|_{V_{n,\Lambda}} = u$ and $v \in \mathcal{F}_{\mathcal{S}}$ by the self-similarity \eqref{SSE1} of $\mathcal{F}$. 
    Since
    \begin{align*}
        \mathcal{E}|_{V_{n,\Lambda}}(u)
        \le \mathcal{E}(v)
        &\overset{\eqref{ss.partition}}{=} \sum_{w \in \Lambda}\rweight_{w}\mathcal{E}(v \circ F_{w})
        = \sum_{w \in \Lambda}\rweight_{w}\mathcal{E}\bigl(h_{V_{n}}^{\mathcal{E}}[u \circ F_{w}]\bigr)
        =
        \sum_{w \in \Lambda}\rweight_{w}\mathcal{E}|_{V_{n}}(u \circ F_{w}), 
    \end{align*}
    we have \eqref{eq:compat-seq.partition}. 
    Next we prove \eqref{ss.harmext}.
    We have $\mathcal{E}\bigl(h_{V_{n,\Lambda}}^{\mathcal{E}}[u] \circ F_{w}\bigr) \ge \mathcal{E}\bigl(h_{V_{n}}^{\mathcal{E}}[u \circ F_{w}]\bigr)$ for any $w \in \Lambda$.
    Since 
    \begin{align*}
        \mathcal{E}|_{V_{n,\Lambda}}(u)
        = \mathcal{E}\bigl(h_{V_{n,\Lambda}}^{\mathcal{E}}[u]\bigr)
        &= \sum_{w \in \Lambda}\rweight_{w}\mathcal{E}\bigl(h_{V_{n,\Lambda}}^{\mathcal{E}}[u] \circ F_{w}\bigr) \\
        &\ge \sum_{w \in \Lambda}\rweight_{w}\mathcal{E}\bigl(h_{V_{n}}^{\mathcal{E}}[u \circ F_{w}]\bigr)
        = \sum_{w \in \Lambda}\rweight_{w}\mathcal{E}|_{V_{n}}(u \circ F_{w})
        = \mathcal{E}|_{V_{n,\Lambda}}(u),
    \end{align*}
    we obtain $\mathcal{E}\bigl(h_{V_{n,\Lambda}}^{\mathcal{E}}[u] \circ F_{w}\bigr) = \mathcal{E}\bigl(h_{V_{n}}^{\mathcal{E}}[u \circ F_{w}]\bigr)$ for any $w \in \Lambda$.
    The uniqueness of $h_{V_{n}}^{\mathcal{E}}[u \circ F_{w}]$ in Theorem \ref{thm.RF-exist} implies $h_{V_{n,\Lambda}}^{\mathcal{E}}[u] \circ F_{w} = h_{V_{n}}^{\mathcal{E}}[u \circ F_{w}]$.
\end{proof}

The following corollary is an immediate consequence of Proposition \ref{prop.harmapprox}. 
\begin{cor}\label{cor:sspRF.compat} 
    Assume that $\mathcal{L} = (K,S,\{ F_{i} \}_{i \in S})$ is a p.-c.f.\ self-similar structure.
    Then 
    \begin{align}
    	\mathcal{F} &= \Bigl\{ u \in \contfunc(K) \Bigm| \lim_{n \to \infty}\mathcal{E}|_{V_{n}}(u|_{V_{n}}) < \infty \Bigr\}, \label{sspRF:monolim.domain}\\
    	\mathcal{E}(u; v) &= \lim_{n \to \infty}\mathcal{E}|_{V_{n}}(u|_{V_{n}}; v|_{V_{n}}) \quad \text{for any $u,v \in \mathcal{F}$,} \label{sspRF:monolim}\\
    	\lim_{n \to \infty}&\mathcal{E}\bigl(u - h_{V_{n}}^{\mathcal{E}}[u|_{V_{n}}]\bigr) = 0 \quad \text{for any $u \in \mathcal{F}$.} \label{sspRF:monoapprox} 
    \end{align}
\end{cor}
\begin{proof}
The condition \eqref{e:lineariso.Vast} in Proposition \ref{prop.harmapprox} holds by $\mathcal{F} \subseteq \contfunc(K)$ and $\closure{V_{\ast}}^{K} = K$. Therefore \eqref{sspRF:monolim.domain} follows from \eqref{monolim-dom}, $\mathcal{F} \subseteq \contfunc(K)$ and $\closure{V_{\ast}}^{K} = K$, \eqref{sspRF:monolim} from \eqref{compat-conv}, and \eqref{sspRF:monoapprox} from \eqref{harmapprox}. 
(Note that \eqref{monolim-dom-dense} may not be applicable to the present situation because the topology considered in \eqref{monolim-dom-dense} is that induced by $R_{\mathcal{E}}^{1/p}$ and may be different from the original topology of $K$.) 
\end{proof} 

The following proposition gives characterizations of $\mathcal{E}$-harmonic functions on $K \setminus V_{n}$.
\begin{prop}\label{prop.harmVn} 
    Let $n\in\mathbb{N} \cup \{ 0 \}$ and $h\in\contfunc(K)$.
    Then the following two conditions are equivalent to each other:
    \begin{enumerate}[label=\textup{(\arabic*)},align=left,leftmargin=*,topsep=4pt,itemsep=2pt]
        \item\label{harmVn} $h\in\mathcal{H}_{\mathcal{E},V_{n}}$. 
        \item\label{harmVn-each} $h \circ F_{w} \in \mathcal{H}_{\mathcal{E},V_{0}}$ for any $w \in W_{n}$.
    \end{enumerate}
    If in addition $\mathcal{L}$ is a p.-c.f.\ self-similar structure, then each of \ref{harmVn} and \ref{harmVn-each} above is equivalent also to the following condition:
    \begin{enumerate}[label=\textup{(\arabic*)},align=left,leftmargin=*,topsep=4pt,itemsep=2pt]\setcounter{enumi}{2}
        \item\label{harmVn-test} For any $m\in\mathbb{N}$ with $m>n$ and any $x\in V_{m} \setminus V_{n}$,
	       \begin{equation}\label{eq:harmVn-test}
	              \sum_{w \in W_{m};\,x\in F_{w}(V_{0})}\rweight_{w}\mathcal{E}|_{V_{0}}\Bigl(h \circ F_{w}\vert_{V_{0}};\indicator{F_{w}^{-1}(x)}^{V_{0}}\Bigr)=0.
	       \end{equation}
    \end{enumerate}
\end{prop}
\begin{proof}
    To see \ref{harmVn} $\Rightarrow$ \ref{harmVn-each}, let $w \in W_{n}$, $\varphi \in \mathcal{F}^{0}(K \setminus V_{0})$ and define $(F_{w})_{\ast}\varphi \colon K \to \mathbb{R}$ by 
    \[
    (F_{w})_{\ast}\varphi \coloneqq 
    \begin{cases}
    	 \varphi \circ F_{w}^{-1} &\text{on $K_{w}$,} \\
    	 0 &\text{on $K \setminus K_{w}$.}
    \end{cases} 
    \]
  	Then since $(F_{w})_{\ast}\varphi \in \contfunc(K)$ by $\varphi|_{V_{0}}=0$ and \eqref{V0bdry}, it follows from the self-similarity \eqref{SSE1} of $\mathcal{F}$ that $(F_{w})_{\ast}\varphi \in \mathcal{F}^{0}(K \setminus V_{n})$, and then from \ref{harmVn} and the self-similarity \eqref{SSE2} of $\mathcal{E}$ that $0 = \mathcal{E}(h; (F_{w})_{\ast}\varphi) = \rweight_{w}\mathcal{E}(h \circ F_{w}; \varphi)$, proving $h \circ F_{w} \in \mathcal{H}_{\mathcal{E},V_{0}}$, namely \ref{harmVn-each}. 
    The converse implication \ref{harmVn-each} $\Rightarrow$ \ref{harmVn} is obvious from \eqref{SSE2}.

    Next we prove the equivalence between \ref{harmVn} and \ref{harmVn-test} for a p.-c.f.\ self-similar structure $\mathcal{L}$.
    We first show \ref{harmVn} $\Rightarrow$ \ref{harmVn-test}.
    For any $m > n$ and any $\varphi \in \mathcal{F}^{0}(K \setminus V_{n})$, we note that $h_{V_{m}}^{\mathcal{E}}[\varphi|_{V_{m}}]\bigr|_{V_{n}} = 0$.
    Then, for any $h \in \mathcal{H}_{\mathcal{E},V_{n}}$, we have from \eqref{eq:compat-seq} that
    \begin{equation*}
        0 = \mathcal{E}|_{V_{m}}(h|_{V_{m}}; \varphi|_{V_{m}}) = \sum_{w \in W_{m}}\rweight_{w}\mathcal{E}|_{V_{0}}\bigl(h \circ F_{w}|_{V_{0}}; \varphi \circ F_{w}|_{V_{0}}\bigr) \quad \text{for any $\varphi \in \mathcal{F}^{0}(K \setminus V_{0})$.}
    \end{equation*}
    By choosing $\varphi \in \mathcal{F}^{0}(K \setminus V_{n})$ so that $\varphi|_{V_{m}} = \indicator{x}^{V_{m}}$ for $x \in V_{m} \setminus V_{n}$, we obtain \ref{harmVn-test}.
    We next assume that $h \in \contfunc(K)$ satisfies \eqref{eq:harmVn-test} and fix $\varphi \in \mathcal{F}^{0}(K \setminus V_{n})$ in order to show the converse implication \ref{harmVn-test} $\Rightarrow$ \ref{harmVn}.
    For $m > n$, we see from \eqref{eq:compat-seq}, $\varphi|_{V_{n}} = 0$ and \eqref{eq:harmVn-test} that
    \begin{align*}
        \mathcal{E}|_{V_{m}}(h|_{V_{m}}; \varphi|_{V_{m}})
        &= \sum_{w \in W_{m}}\rweight_{w}\mathcal{E}|_{V_{0}}\bigl(h \circ F_{w}|_{V_{0}}; \varphi \circ F_{w}|_{V_{0}}\bigr) \\
        &= \sum_{w \in W_{m}}\sum_{y \in V_{0}}\varphi(F_{w}(y))\rweight_{w}\mathcal{E}|_{V_{0}}\bigl(h \circ F_{w}|_{V_{0}}; \indicator{y}^{V_{0}}\bigr) \\
        &= \sum_{x \in V_{m} \setminus V_{n}}\varphi(x)\sum_{w \in W_{m}; x \in F_{w}(V_{0})}\rweight_{w}\mathcal{E}|_{V_{0}}\Bigl(h \circ F_{w}|_{V_{0}}; \indicator{F_{w}^{-1}(x)}^{V_{0}}\Bigr)
        = 0.
    \end{align*}
    Letting $m \to \infty$ here on the basis of \eqref{sspRF:monolim}, we obtain $\mathcal{E}(h; \varphi) = 0$, and hence $h \in \mathcal{H}_{\mathcal{E},V_n}$.
\end{proof}

Thanks to the self-similarity, we can get the following localized version of the weak comparison principle (recall Proposition \ref{prop.cp1}).
\begin{prop}[A localized weak comparison principle]\label{prop:localcp} 
    Let $n\in\mathbb{N} \cup \{ 0 \}$, $w \in W_{n}$, and let $u,v\in \mathcal{H}_{\mathcal{E},V_{n}}$ satisfy $u(x)\leq v(x)$ for any $x\in F_{w}(V_{0})$. Then $u(x) \leq v(x)$ for any $x \in K_{w}$.
\end{prop}
\begin{proof}
    Since $h \circ F_{w} \in \mathcal{H}_{\mathcal{E},V_{0}}$ by the implication from \ref{harmVn} to \ref{harmVn-each} in Proposition \ref{prop.harmVn}, the assertion follows by applying Proposition \ref{prop.cp1} to $h \circ F_{w}$.  
\end{proof}

Next we show the monotonicity in $p$ of the $\frac{1}{p-1}$-th power of the weight of a self-similar $p$-resistance form with constant weight on a p.-c.f.\ self-similar structure (Theorem \ref{t:KS-mono.pcf} below); see also Theorem \ref{t:KS-mono} for a similar result in another framework including the generalized Sierpi\'{n}ski carpets. 
The proof of Theorem \ref{t:KS-mono.pcf} requires the following basic result, which is immediate from \eqref{V0bdry} and Proposition \ref{prop.cone-gen}-\ref{GC.cone}. 
\begin{prop}\label{prop.RFrenorm}
	Assume that $\mathcal{L}$ is a p.-c.f.\ self-similar structure. 
	Let $k,n \in \mathbb{N} \cup \{ 0 \}$ and let $E$ be a $p$-resistance form on $V_{k}$. 
	Let $\bm{\rweight} = (\rweight_{i})_{i \in S} \in (0,\infty)^{S}$ and define $\Lambda_{\bm{\rweight}}^{n}(E) \colon \mathbb{R}^{V_{k + n}}\to [0,\infty)$ by 
	\begin{equation}\label{RFrenorm}
		\Lambda_{\bm{\rweight}}^{n}(E)(u) \coloneqq \sum_{w \in W_{n}}\rweight_{w}E(u \circ F_{w}|_{V_{k}}), \quad u \in \mathbb{R}^{V_{k + n}}.  
	\end{equation}
	Then $\Lambda_{\bm{\rweight}}^{n}(E)$ is a $p$-resistance form on $V_{k + n}$. 
\end{prop}

\begin{thm}\label{t:KS-mono.pcf}
    Assume that $\mathcal{L}$ is a p.-c.f.\ self-similar structure. 
	Let $p_{1},p_{2} \in (1,\infty)$ satisfy $p_{1} \leq p_{2}$, and for each $s \in \{ 1, 2 \}$, let $\rweight_{s} \in (1,\infty)$ and let $(\mathcal{E}_{s},\mathcal{F}_{s})$ be a self-similar $p_{s}$-resistance form on $\mathcal{L}$ with weight $(\rweight_{s})_{i \in S}$. 
    Then
    \begin{equation}\label{KS-mono.pcf}
        \rweight_{1}^{1/(p_{1}-1)} \leq \rweight_{2}^{1/(p_{2}-1)}. 
    \end{equation}
\end{thm}
\begin{proof} 
    Let $s \in \{ 1, 2 \}$, $n \in \mathbb{N} \cup \{0\}$, and let $E_{s,n}$ be the $p_{s}$-resistance form on $V_{n}$ given by
    \[
    E_{s,n}(u) \coloneqq \rweight_{s}^{n} \sum_{v \in W_{n}} \sum_{x,y \in V_{0}} \abs{u(F_{v}(x)) - u(F_{v}(y))}^{p_{s}}, \quad u \in \mathbb{R}^{V_{n}}, 
    \]
    so that $\Lambda_{(\rweight_{s})_{i \in S}}^{n}(E_{s,0}) = E_{s,n}$.
    Since both $E_{s,0}(\,\cdot\,)^{1/p_{s}}$ and $\mathcal{E}_{s}|_{V_{0}}(\,\cdot\,)^{1/p_{s}}$ are norms on the finite-dimensional vector space $\mathbb{R}^{V_{0}}/\mathbb{R}\indicator{V_{0}}$, there exists $C_{s} \in [1,\infty)$ such that 
    \begin{equation}\label{comp.level0}
        C_{s}^{-1}E_{s,0}(u) \le \mathcal{E}_{s}|_{V_{0}}(u) \le C_{s}E_{s,0}(u) \quad \text{for any $u \in \mathbb{R}^{V_{0}}$.}
    \end{equation}
    Since $\Lambda_{(\rweight_{s})_{i \in S}}^{n}(\mathcal{E}_{s}|_{V_{0}}) = \mathcal{E}_{s}|_{V_{n}}$ by \eqref{eq:compat-seq} in Proposition \ref{prop.compatible}, we see from \eqref{comp.level0} that
    \begin{equation}\label{comp.leveln}
        C_{s}^{-1}E_{s,n}(u) \le \mathcal{E}_{s}|_{V_{n}}(u) \le C_{s}E_{s,n}(u) \quad \text{for any $n \in \mathbb{N} \cup \{ 0 \}$ and any $u \in \mathbb{R}^{V_{n}}$.}
    \end{equation}
	
    Now we move to the proof of \eqref{KS-mono.pcf}.
    Let us fix $x_{0},y_{0} \in V_{0}$ with $x_{0} \neq y_{0}$ and set $B \coloneqq \{ x_{0},y_{0} \}$.
    Then we can find $w \in W_{\ast}$ so that $B \cap K_{w} = \emptyset$ and $h_{1,w} \coloneqq h_{1} \circ F_{w} \not\in \mathbb{R}\indicator{K}$, where $h_{1} \coloneqq h_{B}^{\mathcal{E}_{1}}[\indicator{x_{0}}]$. 
    (Supposing that $h_{1} \circ F_{w} \in \mathbb{R}\indicator{K}$ for any $w \in W_{\ast}$ with $B \cap K_{w} = \emptyset$, we would easily get a contradiction to $h_{1}(x_{0}) \not= h_{1}(y_{0})$ by combining the arcwise connectedness of $K$ implied by its connectedness and \cite[Theorem 1.6.2]{Kig01} with the continuity of $h_{1} \in \mathcal{F}_{1} \subseteq \contfunc(K)$.) 
    Noting that $c \coloneqq \inf_{x \in K_{w}}R_{\mathcal{E}_{1}}(x, B) \geq \mathcal{E}_{1}\bigl( h_{V_{\lvert w\rvert}}^{\mathcal{E}_{1}}[\indicator{B}] \bigr)^{-1} > 0$ by Proposition \ref{prop:localcp} and the definition of $R_{\mathcal{E}_{1}}(x, B)$ in \eqref{R-def.gen} and that $0 \leq h_{1} \leq 1$ by the weak comparison principle \eqref{mp}, for any $n \in \mathbb{N} \cup \{0\}$ we obtain 
    \begin{align}\label{diff.harm}
        &\mathcal{E}_{2}|_{V_{0}}(h_{1,w}|_{V_{0}}) \leq \mathcal{E}_{2}|_{V_{n}}(h_{1,w}|_{V_{n}}) \quad \text{(by Proposition \ref{prop.trace-comp} and  \eqref{eq:dfn-trace})} \nonumber \\
        &\overset{\eqref{comp.leveln}}{\leq} C_{2}E_{2,n}(h_{1,w}|_{V_{n}}) \nonumber \\
        &= C_{2}\rweight_{2}^{n}\sum_{v \in W_{n}}\sum_{x,y \in V_{0}}\abs{h_{1}(F_{wv}(x)) - h_{1}(F_{wv}(y))}^{p_{2} - p_{1}} \cdot \abs{h_{1,w}(F_{v}(x)) - h_{1,w}(F_{v}(y))}^{p_{1}} \nonumber \\
        &\overset{\eqref{lip-harm}}{\leq} C_{2}\rweight_{2}^{n}\sum_{v \in W_{n}}\sum_{x,y \in V_{0}}\biggl(\frac{R_{\mathcal{E}_{1}}(F_{wv}(x),F_{wv}(y))}{R_{\mathcal{E}_{1}}(F_{wv}(x),B)}\biggr)^{\frac{p_{2} - p_{1}}{p_{1} - 1}} \cdot \abs{h_{1,w}(F_{v}(x)) - h_{1,w}(F_{v}(y))}^{p_{1}} \nonumber \\
        &\overset{\eqref{pRMss}}{\leq} C_{2} \Bigl( c^{-1} \sup_{x,y \in K}R_{\mathcal{E}_{1}}(x,y) \Bigr)^{(p_{2} - p_{1})/(p_{1} - 1)} \bigl(\rweight_{2}\rweight_{1}^{-(p_{2} - 1)/(p_{1} - 1)}\bigr)^{n}E_{1,n}(h_{1,w}|_{V_{n}}) \nonumber \\
        &\overset{\eqref{comp.leveln}}{\leq} C_{1}C_{2} \Bigl( c^{-1} \sup_{x,y \in K}R_{\mathcal{E}_{1}}(x,y)\Bigr)^{(p_{2} - p_{1})/(p_{1} - 1)} \bigl(\rweight_{2}\rweight_{1}^{-(p_{2} - 1)/(p_{1} - 1)}\bigr)^{n}\mathcal{E}_{1}(h_{1,w}).
    \end{align}
    Since $\sup_{x,y \in K}R_{\mathcal{E}_{1}}(x,y) < \infty$ by Proposition \ref{prop.pRMss}-\ref{pRMcompatible} and $\mathcal{E}_{2}|_{V_{0}}(h_{1,w}|_{V_{0}}), \mathcal{E}_{1}(h_{1,w}) \in (0,\infty)$, we conclude by letting $n \to \infty$ in \eqref{diff.harm} that $\rweight_{2}\rweight_{1}^{-(p_{2} - 1)/(p_{1} - 1)} \geq 1$, proving \eqref{KS-mono.pcf}. 
\end{proof}

\subsection{Two-point estimate and capacity upper estimate}\label{subsec:TPE}
This subsection is devoted to proving the two-point estimate and the $(p,p)$-Poincar\'{e} inequality in terms of self-similar $p$-energy measures, and showing also the capacity upper estimate under the additional assumption that $\mathcal{L}$ is a p.-c.f.\ self-similar structure. 

Recall that we fix a self-similar $p$-resistance form $(\mathcal{E},\mathcal{F})$ on $\mathcal{L}$ with weight $\bm{\rweight} = (\rweight_{i})_{i \in S} \in (0,\infty)^{S}$. 
In this subsection, we further assume that $\min_{i \in S}\rweight_{i} > 1$ and that the $p$-resistance metric $\pmetric_{p} \coloneqq \pmetric_{p,\mathcal{E}}$ of $(\mathcal{E},\mathcal{F})$ is compatible with the original topology of $K$, which is, in view of Propositions \ref{prop.R-conseq}-\ref{it:RF.basic-ineq} and \ref{prop.pRMss}-\ref{pRMcompatible}, equivalent to assuming $\min_{i \in S}\rweight_{i} > 1$ and $\diam(K,\pmetric_{p}) < \infty$.  
(Also by Proposition \ref{prop.pRMss}-\ref{pRMcompatible}, the assumption of $\diam(K,\pmetric_{p}) < \infty$ can be dropped when $\mathcal{L}$ is a p.-c.f.\ self-similar structure.) 
We also let $\{ \Gamma_{\mathcal{E}}\langle u \rangle \}_{u \in \mathcal{F}}$ be the associated self-similar $p$-energy measures defined by \eqref{e:defn.em.one} in the paragraph before Proposition \ref{prop.ss-pform-em}.  
In the following definition, we introduce natural \emph{scales} $\{ \Lambda_{s} \}_{s \in (0,1]}$ with respect to $\pmetric_{p}$.   
See \cite{Kig09,Kig20} for further details on scales. 
\begin{defn}\label{defn.scale-Rp}
    \begin{enumerate}[label=\textup{(\arabic*)},align=left,leftmargin=*,topsep=2pt,parsep=0pt,itemsep=2pt]
        \item\label{it:scale.res} We define $\Lambda_{1}^{\pmetric_{p}} \coloneqq \{ \emptyset \}$,
        \begin{equation*}
            \Lambda_{s}^{\pmetric_{p}} \coloneqq \Bigl\{ w \Bigm| w = w_{1} \dots w_{n} \in W_{\ast} \setminus \{ \emptyset \}, (\rweight_{w_{1} \dots w_{n - 1}})^{-1/(p - 1)} > s \ge \rweight_{w}^{-1/(p - 1)} \Bigr\}
        \end{equation*}
        for each $s \in (0,1)$.
        (Note that $\{ \Lambda_{s}^{\pmetric_{p}} \}_{s \in (0,1]}$ is the scale associated with the weight function $g(w) \coloneqq \rweight_{w}^{-1/(p - 1)}$; see \cite[Definition 2.3.1]{Kig20}.)
        \item\label{it:scale-Rp-UMxs} For each $(s,x) \in (0,1] \times K$, we define $\Lambda_{s,0}^{\pmetric_{p}}(x) \coloneqq \{ w \in \Lambda_{s}^{\pmetric_{p}} \mid x \in K_{w} \}$ and $U_{0}^{\pmetric_{p}}(x,s) \coloneqq \bigcup_{w \in \Lambda_{s,0}^{\pmetric_{p}}(x)}K_{w}$. Inductively, for $M \in \mathbb{N}$, define $\Lambda_{s,M}^{\pmetric_{p}}(x) \coloneqq \{ w \in \Lambda_{s}^{\pmetric_{p}} \mid K_{w} \cap U_{M - 1}^{\pmetric_{p}}(x,s) \neq \emptyset \}$ and $U_{M}^{\pmetric_{p}}(x,s) \coloneqq \bigcup_{w \in \Lambda_{s,M}^{\pmetric_{p}}(x)}K_{w}$.
    \end{enumerate}
\end{defn}
It is easy to see that $\lim_{s \downarrow 0}\min\{ \abs{w} \mid w \in \Lambda_{s}^{\pmetric_{p}} \} = \infty$, that $\Lambda_{s}^{\pmetric_{p}}$ is a partition of $\Sigma$ for any $s \in (0,1]$, and that $\Lambda_{s_{1}}^{\pmetric_{p}} \le \Lambda_{s_{2}}^{\pmetric_{p}}$ for any $s_1, s_2 \in (0,1]$ with $s_{1} \le s_{2}$.
Moreover, for any $x \in K$ and any $M \in \mathbb{N} \cup \{ 0 \}$, $\bigl\{ U_{M}^{\pmetric_{p}}(x,s) \bigr\}_{s \in (0,1]}$ is non-decreasing in $s$ and forms a fundamental system of neighborhoods of $x$ in $K$ (see, e.g., \cite[Proposition 2.3.7]{Kig20}).

If $\bigl\{ U_{M_{\ast}}^{\pmetric_{p}}(x,s) \bigr\}_{(x,s) \in K \times (0,1]}$ is comparable to the metric balls with respect to $\pmetric_{p}$ (in the sense of \eqref{e:Rp-PI.adapted} below) for some $M_{\ast} \in \mathbb{N}$, then we have the following two-point estimate.
\begin{prop}[Two-point estimate\index{two-point estimate}]\label{prop.Rp-TPE}
	Assume that there exist $\alpha_{1},\alpha_{2} \in (0,\infty)$ such that for any $(x,s) \in K \times (0,1]$, 
	\begin{equation}\label{e:Rp-PI.adapted}
		B_{\pmetric_{p}}(x,\alpha_{1}s) \subseteq U_{M_{\ast}}^{\pmetric_{p}}(x,s) \subseteq B_{\pmetric_{p}}(x,\alpha_{2}s). 
	\end{equation} 
	Then there exist $C,A \in (0,\infty)$ with $A \geq 1$ such that for any $(x,s) \in K \times (0,\infty)$ and any $u \in \mathcal{F}_{\mathrm{loc}}\bigl(B_{\pmetric_{p}}(x,As)\bigr)$, 
    \begin{equation}\label{e:sspRF-TPE}
    	\sup_{y,z \in B_{\pmetric_{p}}(x,s)}\abs{u(y) - u(z)}^{p} \le Cs^{p - 1}\Gamma_{\mathcal{E}}\langle u \rangle\bigl(B_{\pmetric_{p}}(x,As)\bigr).
    \end{equation}
\end{prop}
\begin{proof}
	We can assume that $\alpha_{1} \le \alpha_{2}$ and $\alpha_{1} \le 1$ without loss of generality.  
	Throughout this proof, we fix $x \in K$ and set $A \coloneqq \alpha_{1}^{-1}(\alpha_{2} \vee \diam(K,\pmetric_{p}))$. 
	We first consider the case of $s \in (\alpha_{1},\infty)$. 
	Note that $B_{\pmetric_{p}}(x,As) = K$. 
	By the estimate \eqref{R-basic} in Proposition \ref{prop.R-conseq}-\ref{it:RF.basic-ineq} and Proposition \ref{prop.ss-pform-em}-\ref{ssem}, for any $y,z \in B_{\pmetric_{p}}(x,s)$ and any $u \in \mathcal{F}$, 
	\[
	\abs{u(y) - u(z)}^{p} \le \diam(K,\pmetric_{p})^{p - 1}\mathcal{E}(u) = C_{1}\alpha_{1}^{p-1}\Gamma_{\mathcal{E}}\langle u \rangle(K), 
	\]
	where $C_{1} \coloneqq \alpha_{1}^{-(p-1)}\diam(K,\pmetric_{p})^{p - 1}$. 
	This shows \eqref{e:sspRF-TPE} in the case of $s \in (\alpha_{1},\infty)$.  
    
    Next let $s \in (0,\alpha_{1}]$. 
    Let $U$ be a relatively compact open subset of $K$ such that $U \supseteq U_{M_{\ast}}^{\pmetric_{p}}(x,\alpha_{1}^{-1}s)$ and let $u^{\#} \in \mathcal{F}$ satisfy $u = u^{\#}$ on $U$. 
	For any $y,z \in B_{\pmetric_{p}}(x,s)$, there exists $\{ v(i) \}_{i = 1}^{2M_{\ast} + 1} \subseteq \Lambda_{\alpha_{1}^{-1}s,M_{\ast}}^{\pmetric_{p}}(x)$ such that $y \in K_{v(1)}$, $z \in K_{v(2M_{\ast} + 1)}$ and $K_{v(i)} \cap K_{v(i + 1)} \neq \emptyset$ for each $i \in \{ 1,2,\dots,2M_{\ast} \}$.
	Let us fix $x_{i} \in K_{v(i)} \cap K_{v(i + 1)}$ and $q_{i} \in V_{0}$ that satisfy $x_{i} = F_{v(i)}(q_{i})$.
	We note that, for any $y',z' \in K_{v(i)}$, 
	\begin{align*}
		\abs{u(y') - u(z')}^{p} 
		&= \abs{u(F_{v(i)}(F_{v(i)}^{-1}(y'))) - u(F_{v(i)}(F_{v(i)}^{-1}(z')))}^{p} \\
		&\le R_{\mathcal{E}}(F_{v(i)}^{-1}(y'),F_{v(i)}^{-1}(z'))\mathcal{E}(u^{\#} \circ F_{v(i)}) \\
		&\overset{\eqref{e:eachcell}}{\le} \diam(K,\pmetric_{p})^{p - 1}\rweight_{v(i)}^{-1}\Gamma_{\mathcal{E}}\langle u^{\#} \rangle(K_{v(i)})
		= \diam(K,\pmetric_{p})^{p - 1}\rweight_{v(i)}^{-1}\Gamma_{\mathcal{E}}\langle u \rangle(K_{v(i)}). 
	\end{align*}
	Hence 
	\begin{align*}\label{e:pRF-PI-Holder}
        &\abs{u(y) - u(z)}^{p} \nonumber \\
        &\le (2M_{\ast} + 1)^{p - 1}\Biggl(\abs{u(y) - u(x_{1})}^{p} + \sum_{i = 1}^{2M_{\ast} - 1}\abs{u(x_{i}) - u(x_{i + 1})}^{p} + \abs{u(x_{2M_{\ast}}) - u(z)}^{p}\Biggr) \nonumber \\ 
        &\overset{\eqref{R-basic}}{\le} \bigl((2M_{\ast} + 1)\diam(K,\pmetric_{p})\bigr)^{p - 1}\sum_{i = 1}^{2M_{\ast} + 1}\rweight_{v(i)}^{-1}\Gamma_{\mathcal{E}}\langle u \rangle(K_{v(i)}) \nonumber \\
        &\le C_{2}s^{p - 1}\Gamma_{\mathcal{E}}\langle u \rangle\Biggl(\bigcup_{i = 1}^{2M_{\ast} + 1}K_{v(i)}\Biggr)
        \le C_{2}s^{p - 1}\Gamma_{\mathcal{E}_{p}}\langle u \rangle(B_{\pmetric}(x,\alpha_{1}^{-1}\alpha_{2}s)), 
    \end{align*}
    where $C_{2} \coloneqq \bigl((2M_{\ast} + 1)\alpha_{1}^{-1}\diam(K,\pmetric_{p})\bigr)^{p - 1}$.  
    This proves \eqref{e:sspRF-TPE} for $s \in (0,\alpha_{1}]$.  
\end{proof}

From \eqref{e:sspRF-TPE}, we easily obtain the following $(p,p)$-Poincar\'{e} inequality.
\begin{prop}[$(p,p)$-Poincar\'{e} inequality\index{$(p,p)$-Poincar\'{e} inequality}]\label{prop:PI.RF}
	Assume that there exist $\alpha_{1},\alpha_{2} \in (0,\infty)$ such that \eqref{e:Rp-PI.adapted} in Proposition \ref{prop.Rp-TPE} holds for any $(x,s) \in K \times (0,1]$. 
	Let $\mu$ be a Radon measure on $K$ with $\supp_{K}[\mu] = K$. 
	Then there exist $C,A \in (0,\infty)$ with $A \ge 1$ such that for any $(x,s) \in K \times (0,\infty)$ and any $u \in \mathcal{F}_{\mathrm{loc}}\bigl(B_{\pmetric_{p}}(x,As)\bigr)$,
	\begin{equation}\label{e:ssRF.pp-PI}
		\fint_{B_{\pmetric_{p}}(x,s)}\abs{u - \fint_{B_{\pmetric_{p}}(x,s)}u\,d\mu}^{p}\,d\mu 
		\le Cs^{p-1}\Gamma_{\mathcal{E}}\langle u \rangle\bigl(B_{\pmetric_{p}}(x,As)\bigr). 
	\end{equation} 
\end{prop}
\begin{proof} 
	This is immediate from \eqref{e:sspRF-TPE} and the obvious inequality 
	\[
	\fint_{B_{\pmetric_{p}}(x,s)}\abs{u - \fint_{B_{\pmetric_{p}}(x,s)}u\,d\mu}^{p}\,d\mu 
	\le \sup_{y,z \in B_{\pmetric_{p}}(x,s)}\abs{u(y) - u(z)}^{p}. \qedhere 
	\]
\end{proof}
\begin{rmk}\label{rmk:TPE-PI} 
	In \cite[Theorem 2.4]{Cap07}, Capitanelli obtained an oscillation estimate like the two-point estimate \eqref{e:sspRF-TPE} from the $(p,p)$-Poincar\'{e} inequality \cite[(2.4)]{Cap07} under a suitable volume growth condition for the measure $\mu$. 
	This implication can be seen by a well-known telescopic sum argument (see, e.g., \cite[Proof of Lemma 5.17]{HK98} for such an argument).  
\end{rmk}

As shown in \cite[Lemma 6.7 and Proposition 6.9]{KS.lim}, if $\mathcal{L}$ is a p.-c.f.\ self-similar structure, then the condition \eqref{e:Rp-PI.adapted} and the capacity upper estimate hold. 
Furthermore by \cite[Lemma 6.8]{KS.lim}, there exists a self-similar measure on $\mathcal{L}$ which is Ahlfors regular with respect to $\pmetric_{p}$ (see Definition \ref{defn.AR}-\ref{it:AR}). 
We record these results in the following proposition. 
\begin{prop}\label{prop:Rp-geom} 
	Assume that $\mathcal{L}$ is a p.-c.f.\ self-similar structure.  
	\begin{enumerate}[label=\textup{(\alph*)},align=left,leftmargin=*,topsep=2pt,parsep=0pt,itemsep=2pt]
		\item\label{it:Rp.adapted} There exist $\alpha_{1},\alpha_{2} \in (0,\infty)$ such that for any $(s,x) \in (0,1] \times K$,
    		\begin{equation}\label{Rp-adapted}
        		B_{\pmetric_{p}}(x,\alpha_{1}s) \subseteq U_{1}^{\pmetric_{p}}(x,s) \subseteq B_{\pmetric_{p}}(x,\alpha_{2}s). 
    		\end{equation}
    		(Equivalently, $\pmetric_{p}$ is $1$-adapted to the weight function $g(w) \coloneqq \rweight_{w}^{-1/(p - 1)}$; see \cite[Definition 2.4.1]{Kig20}.)
    	\item\label{it:Rp.AR} Let $\whdim \in (0,\infty)$ be such that $\sum_{i \in S}\rweight_{i}^{-\whdim/(p - 1)} = 1$, and let $m$ be the self-similar measure on $\mathcal{L}$ with weight $\bigl(\rweight_{i}^{-\whdim/(p - 1)})_{i \in S}$. Then there exist $c_{1},c_{2} \in (0,\infty)$ such that for any $(x,s) \in K \times (0,2\diam(K,\pmetric_{p}))$,
    		\begin{equation}\label{Rp-AR}
        		c_{1}s^{\whdim} \le m\bigl(B_{\pmetric_{p}}(x,s)\bigr) \le c_{2}s^{\whdim}.
    		\end{equation}
    		In particular, $\pmetric_{p}$ is metric doubling. \index{metric doubling} \textup{(Recall that a metric space $(X,d)$ is said to be \emph{metric doubling} if and only if there exists $N \in \mathbb{N}$ such that any $(x,r) \in X \times (0,\infty)$ satisfies $B_{d}(x,r) \subseteq \bigcup_{i=1}^{N}B_{d}(x_i,r/2)$ for some $\{ x_i \}_{i = 1}^{N} \subseteq X$.)}  
    	\item\label{it:Rp.capu} There exists $C \in (0,\infty)$ such that for any $(x,s) \in K \times (0,\infty)$,
    		\begin{equation}\label{Rp-capu}
        		\inf\bigl\{ \mathcal{E}(u) \bigm| \text{$u \in \mathcal{F}$, $u|_{B_{\pmetric_{p}}(x,\alpha_{1}s)} = 1$, $\supp[u] \subseteq B_{\pmetric_{p}}(x,2\alpha_{2}s)$} \bigr\} \le Cs^{-(p - 1)},
    		\end{equation}
    		where $\alpha_{1},\alpha_{2}$ are the constants in \eqref{Rp-adapted}. \index{capacity upper estimate}
		\end{enumerate}
\end{prop}
\begin{proof}
	Although the proof is the same as \cite[Lemmas 6.7, 6.8 and Proposition 6.9]{KS.lim}, we recall the proof below for the reader's convenience. 
	Throughout this proof, we set $\Lambda_{s} \coloneqq \Lambda_{s}^{\pmetric_{p}}$ for ease of notation.  
	Note that $K \neq \closure{V_0}^{K}$ since $\#V_{0} < \infty$ and $K$ is connected.
	
	\ref{it:Rp.adapted}: 
    By \eqref{pRMss} in Proposition \ref{prop.pRMss}-\ref{pRMcontraction}, we have $\diam(K_{w},\pmetric_{p}) \le \rweight_{w}^{-1/(p - 1)}\diam(K,\pmetric_{p})$ for any $w \in W_{\ast}$, which implies the latter inclusion in \eqref{Rp-adapted} with $\alpha_{2} \in (2\diam(K,\pmetric_{p}),\infty)$ arbitrary.
    (In particular, $\diam(K_{w},\pmetric_{p}) < \alpha_{2}s$ for any $w \in \Lambda_{s}$.)
    We will show the former inclusion in \eqref{Rp-adapted}.
    It suffices to prove that there exists $\alpha_{1} \in (0,\infty)$ such that $\pmetric_{p}(x,y) \ge \alpha_{1}s$ for any $s \in(0,1]$, any $w,v \in \Lambda_{s}$ with $K_{w} \cap K_{v} = \emptyset$ and any $(x,y) \in K_{w} \times K_{v}$.
    Let $\psi_{q} \coloneqq h_{V_{0}}^{\mathcal{E}}\bigl[\indicator{q}^{V_{0}}\bigr]$ for any $q \in V_{0}$.
    Fix $w \in \Lambda_{s}$ and let $u_{w} \in \contfunc(K)$ be such that, for $\tau \in \Lambda_{s}$,
    \begin{align}\label{d:uw}
        u_{w} \circ F_{\tau} =
        \begin{cases}
            1 \quad &\text{if $\tau = w$,} \\
            \sum_{q \in V_{0}; F_{\tau}(q) \in F_{w}(V_{0})}\psi_{q} \quad &\text{if $\tau \neq w$ and $K_{\tau} \cap K_{w} \neq \emptyset$,} \\
            0 \quad &\text{if $K_{\tau} \cap K_{w} = \emptyset$.}
        \end{cases}
    \end{align}
    By the self-similarity \eqref{SSE1}--\eqref{SSE2} of $(\mathcal{E},\mathcal{F})$, we have $u_{w} \in \mathcal{F}$ and
    \begin{equation}\label{uw.ss}
        \mathcal{E}(u_{w})
        = \sum_{\tau \in \Lambda_{s}}\rweight_{\tau}\mathcal{E}(u_{w} \circ F_{\tau}) 
        = \sum_{\tau \in \Lambda_{s} \setminus \{ w \}; K_{\tau} \cap K_{w} \neq \emptyset}\rweight_{\tau}\mathcal{E}\left(\sum_{q \in V_{0}; F_{\tau}(q) \in F_{w}(V_{0})}\psi_{q}\right). 
    \end{equation}
    (Note that $\Lambda_{s}$ is a partition of $\Sigma$.)
    Set $\overline{\rweight} \coloneqq \max_{i \in S}\rweight_{i} \in (1,\infty)$ and $c_{1} \coloneqq \max_{q \in V_{0}}\mathcal{E}(\psi_{q}) \in (0,\infty)$.
    Then $\rweight_{\tau}^{-1} \ge \overline{\rweight}^{-1}s^{p - 1}$ for any $\tau \in \Lambda_{s}$.
    Since $\#\{ \tau \in \Lambda_{s} \mid K_{\tau} \cap K_{w} \neq \emptyset \} \le (\#\mathcal{C}_{\mathcal{L}})(\#V_{0})$ by \cite[Lemma 4.2.3]{Kig01}, \eqref{uw.ss} together with H\"{o}lder's inequality implies that
    \begin{equation}\label{uw.upper}
        \mathcal{E}(u_{w}) \le (\#\mathcal{C}_{\mathcal{L}})(\#V_{0})\overline{\rweight}s^{-p + 1}(\#V_{0})^{p - 1}c_{1} \eqqcolon (\alpha_{1}s)^{-(p - 1)}.
    \end{equation}
    For any $v \in \Lambda_{s}$ with $K_{w} \cap K_{v} = \emptyset$ and any $(x,y) \in K_{w} \times K_{v}$, we clearly have $u_{w}(x) = 1$ and $u_{w}(y) = 0$.
    Hence
    \[
    \pmetric_{p}(x,y) \ge \mathcal{E}(u)^{-1/(p - 1)} \ge \alpha_{1}s,
    \]
    which proves the desired result.
    
    \ref{it:Rp.AR}: 
    This is immediate from \eqref{Rp-adapted}, $\#\{ \tau \in \Lambda_{s} \mid K_{\tau} \cap K_{w} \neq \emptyset \} \le (\#\mathcal{C}_{\mathcal{L}})(\#V_{0})$ (\cite[Lemma 4.2.3]{Kig01}) and $m(K_{w}) = \rweight_{w}^{-\whdim/(p - 1)}$ (Proposition \ref{p:ss-meas}).
    
    \ref{it:Rp.capu}: 
    It suffices to consider the case of $s \in (0,1]$ since $B_{\pmetric_{p}}(x,2\alpha_{2}s) = K$ for any $(x,s) \in K \times (1,\infty)$ and $\mathcal{E}^{-1}(0) = \mathbb{R}\indicator{K}$ by \ref{RF1}. 
    Let $u_{w} \in \mathcal{F}$ be the same function as in the proof of \ref{it:Rp.adapted} of the present proposition for each $w \in \Lambda_{s}$.
    Then $\varphi \coloneqq \max_{w \in \Lambda_{s,1}(x)}u_{w}$ satisfies $\varphi|_{U_{1}^{\pmetric_{p}}(x,s)} = 1$.
    Since $\diam(K_{w},\pmetric_{p}) < \alpha_{2}s$, we see from \eqref{Rp-adapted} that $\supp[\varphi] \subseteq B_{\pmetric_{p}}(x,2\alpha_{2}s)$.
    By the strong subadditivity \eqref{sadd} for $(\mathcal{E},\mathcal{F})$, \eqref{uw.upper} and $\#\Lambda_{s,1}(x) \le (\#\mathcal{C}_{\mathcal{L}})(\#V_{0})$ (\cite[Lemma 4.2.3]{Kig01}), we have $\varphi \in \mathcal{F}$ and
    \[
    \mathcal{E}(\varphi) \le \sum_{w \in \Lambda_{s,1}(x)}\mathcal{E}(u_{w}) \le (\alpha_{1}s)^{-(p - 1)}(\#\mathcal{C}_{\mathcal{L}})(\#V_{0}) \eqqcolon Cs^{-(p - 1)}. 
    \qedhere\]
\end{proof}

Combining Propositions \ref{prop.Rp-TPE}, \ref{prop:Rp-geom} and Theorem \ref{thm.EHI}, we obtain the elliptic Harnack inequality for self-similar $p$-resistance forms on p.-c.f.\ self-similar structures.
\begin{thm}\label{thm:EHI.pcf}
	Assume that $\mathcal{L}$ is a p.-c.f.\ self-similar structure. 
	Then there exist $C_{\mathrm{H}} \in (0,\infty)$ and $\delta_{\mathrm{H}} \in (0,1)$ such that for any $(x,s) \in K \times (0,\infty)$ with $B_{\pmetric_{p}}(x,\delta_{\mathrm{H}}^{-1}s) \neq K$ and any $u \in \mathcal{F}$ such that $u \ge 0$ on $K$ and $u$ is $\mathcal{E}$-superharmonic on $B_{\pmetric_{p}}(x,\delta_{\mathrm{H}}^{-1}s)$, it holds that 
	\begin{equation}\label{e:EHI.pcf}
		\sup_{B_{\pmetric_{p}}(x,s)}u \le C_{\mathrm{H}}\inf_{B_{\pmetric_{p}}(x,s)}u. 
	\end{equation}
\end{thm}
\begin{proof}
	We have Theorem \ref{thm.EHI}-\ref{it:EHI.doubling},\ref{it:EHI.TPE},\ref{it:EHI.capu} with $\Upsilon(x,s) \coloneqq s^{-(p-1)}$ by Propositions \ref{prop.Rp-TPE} and \ref{prop:Rp-geom}, and Theorem \ref{thm.EHI}-\ref{it:EHI.Cha} holds by $\mathcal{F} \subseteq \contfunc(K)$ and Theorem \ref{thm.em-chain}. 
	Since $\Gamma_{\mathcal{E}}\langle u \rangle(K) = \mathcal{E}(u)$ for any $u \in \mathcal{F}$ by Proposition \ref{prop.ss-pform-em}-\ref{ssem}, the desired estimate \eqref{e:EHI.pcf} follows from Theorem \ref{thm.EHI}. 
\end{proof}

\begin{rmk}\label{rmk:EHI.pcf}
The results in this subsection, Propositions \ref{prop.Rp-TPE}, \ref{prop:PI.RF}, \ref{prop:Rp-geom} and Theorem \ref{thm:EHI.pcf}, are applicable to a large class of p.-c.f.\ self-similar structures. 
Indeed, their assumptions are all satisfied in the situation of Theorem \ref{thm.CGQ}, which summarizes the construction of regular self-similar $p$-resistance forms on p.-c.f.\ self-similar structures due to \cite{CGQ22}, and the assumptions of Theorem \ref{thm.CGQ} in turn hold for strongly symmetric p.-c.f.\ self-similar sets (see Framework \ref{frmwrk:ANF} and Definition \ref{defn.ANF}) as proved in Theorem \ref{thm.eigenform-ANF} below. 
\end{rmk} 

\section{Constructions of \texorpdfstring{$p$}{p}-energy forms satisfying the generalized \texorpdfstring{$p$}{p}-contraction property}\label{sec.constr}
In the preceding sections, we have established fundamental results on $p$-energy forms satisfying the generalized $p$-contraction property \ref{GC}, in particular $p$-Clarkson's inequality \ref{Cp}.
In this section, we would like to describe how to get a good $p$-energy form satisfying these properties in a few settings inspired by \cite{Kig23} and \cite{CGQ22}. 
(See also \cite{KS.lim} for another approach toward such a construction.)

\subsection{\texorpdfstring{$p$}{p}-Energy forms on \texorpdfstring{$p$}{p}-conductively homogeneous compact metric spaces}\label{sec.Kig}
In this subsection, we verify that $p$-energy forms on \emph{$p$-conductively homogeneous} compact metric spaces constructed in \cite{Kig23} satisfy \ref{GC}. 
We mainly follow the notation and terminology of \cite{Kig23} in this and the next subsections. 
We refer to \cite[Chapter 2]{Kig23} and \cite[Chapters 2 and 3]{Kig20} for further details. 

Throughout this subsection, we fix a locally finite, non-directed infinite tree $(T,E_{T})$ in the usual sense (see \cite[Definition 2.1]{Kig23} for example), and fix a \emph{root} $\phi \in T$ of $T$.
(Here $T$ is the set of vertices and $E_{T}$ is the set of edges.)
For any $w \in T \setminus \{ \phi \}$, we use $\overline{\phi w}$ to denote the unique simple path in $(T,E_{T})$ from $\phi$ to $w$. \index{rooted tree}
\begin{defn}[{\cite[Definition 2.2]{Kig23}}]\label{defn.tree-notation}
	\begin{enumerate}[label=\textup{(\arabic*)},align=left,leftmargin=*,topsep=2pt,parsep=0pt,itemsep=2pt]
		\item\label{it:tree-vertices} For $w \in T$, define $\pi \colon T \to T$ by
		\begin{equation*}\label{defn.pi}
			\pi(w) \coloneqq 
			\begin{cases}
				w_{n - 1} \quad &\text{if $w \neq \phi$ and $\overline{\phi w} = (w_{0}, \dots, w_{n})$,} \\
				\phi \quad &\text{if $w = \phi$.}
			\end{cases}
		\end{equation*}
		Set $S(w) \coloneqq \{ v \in T \mid \pi(v) = w \} \setminus \{ w \}$.
		Moreover, for $k \in \mathbb{N}$, we define $S^{k}(w)$ inductively as $S^{1}(w) \coloneqq S(w)$ and 
		\[
		S^{k + 1}(w) \coloneqq \bigcup_{v \in S(w)}S^{k}(v).
		\]
		For $A \subseteq T$, define $S^{k}(A) \coloneqq \bigcup_{w \in A}S^{k}(A)$.
		\item\label{it:tree-levels} For $w \in T$ and $n \in \mathbb{N} \cup \{ 0 \}$, define $\abs{w} \coloneqq \min\{ n \ge 0 \mid \pi^{n}(w) = \phi \}$ and $T_{n} \coloneqq \{ w \in T \mid \abs{w} = n \}$.
		\item\label{it:tree-bdry-sections} Define $\Sigma \coloneqq \{ (\omega_{n})_{n \ge 0} \mid \text{$\omega_{n} \in T_{n}$ and $\omega_{n} = \pi(\omega_{n + 1})$ for all $n \in \mathbb{N} \cup \{ 0 \}$} \}$. 
		For $\omega = (\omega_{n})_{n \ge 0} \in \Sigma$, we write $[\omega]_{n}$ for $\omega_{n} \in T_{n}$.
		Define $\Sigma_{w} \coloneqq \{ (\omega_{n})_{n \ge 0} \in \Sigma \mid \text{$\omega_{\abs{w}} = w$} \}$ for $w \in T$, and
		$\Sigma_{A} \coloneqq \bigcup_{w \in A}\Sigma_{w}$ for $A \subseteq T$.  
	\end{enumerate}
\end{defn}

Let us recall the definition of a partition parametrized by a rooted tree.
\begin{defn}[Partition parametrized by a tree\index{partition parametrized by a tree}; {\cite[Definition 2.2.1]{Kig20} and {\cite[Lemma 3.6]{Sas23}}}]\label{defn.partition}
	Let $K$ be a compact metrizable topological space without isolated points. 
	A family of non-empty compact subsets $\{ K_{w} \}_{w \in T}$ of $K$ is called a \emph{partition of $K$ parametrized by the rooted tree $(T, E_{T}, \phi)$} if and only if it satisfies the following conditions:
	\begin{enumerate}[label=\textup{(P\arabic*)},align=left,leftmargin=*,topsep=2pt,parsep=0pt,itemsep=2pt]
		\item $K_{\phi} = K$ and for any $w \in T$, $\#K_{w} \ge 2$ and $K_{w} = \bigcup_{v \in S(w)}K_{v}$.
		\item For any $w \in \Sigma$, $\bigcap_{n \ge 0}K_{[\omega]_{n}}$ is a single point.
	\end{enumerate} 
\end{defn}

In the rest of this subsection, we fix a compact metrizable topological space without isolated points $K$, a locally finite rooted tree $(T, E_{T}, \phi)$ satisfying $\#\{ v \in T \mid \{v,w\} \in E_{T} \} \ge 2$ for any $w \in T$, a partition $\{ K_{w} \}_{w \in T}$ parametrized by $(T,E_{T},\phi)$, a metric $d$ on $K$ with $\diam(K, d) = 1$, and a Borel probability measure $m$ on $K$.
Now we introduce a graph approximation $\{ (T_{n},E_{n}^{\ast}) \}_{n \in \mathbb{N} \cup \{ 0 \}}$ of $K$.
\begin{defn}[{\cite[Proposition 2.8 and Definition 2.5-(3)]{Kig23}}]\label{defn.h-networks}
	For $n \in \mathbb{N} \cup \{ 0 \}$ and $A \subseteq T_{n}$, define
	\begin{equation*}\label{defn.h-edge}
		E_{n}^{\ast} \coloneqq \bigl\{ \{ v, w \} \bigm| v, w \in T_{n}, v \neq w, K_{v} \cap K_{w} \neq \emptyset \bigr\}, 
	\end{equation*}
	and $E_{n}^{\ast}(A) = \bigl\{ \{ v, w \} \in E_{n}^{\ast} \bigm| v, w \in A \bigr\}$. 
	Let $d_{n}$ be the graph distance of $(T_{n}, E_{n}^{\ast})$.
	For $M \in \mathbb{N} \cup \{ 0 \}$ and $w \in T_{n}$, define
	\begin{equation*}\label{defn.h-nbd}
		\Gamma_{M}(w) \coloneqq \{ v \in T_{n} \mid d_{n}(v, w) \le M \} \quad \text{and} \quad U_{M}(x; n) \coloneqq \bigcup_{w \in T_{n}; x \in K_{w}}\bigcup_{v \in \Gamma_{M}(w)}K_{v}.
	\end{equation*} 
\end{defn}

To state geometric assumptions in \cite{Kig23}, we need the following definition. 
\begin{defn}[{\cite[Definitions 2.2.1 and 3.1.15]{Kig20}}]
	\begin{enumerate}[label=\textup{(\arabic*)},align=left,leftmargin=*,topsep=2pt,parsep=0pt,itemsep=2pt]
        \item The partition $\{ K_{w} \}_{w \in T}$ is said to be \emph{minimal}\index{minimal (partition)} if and only if $K_{w} \setminus \bigcup_{v \in T_{\abs{w}} \setminus \{ w \}} \neq \emptyset$ for any $w \in T$.
        \item The partition $\{ K_{w} \}_{w \in T}$ is said to be \emph{uniformly finite}\index{uniformly finite (partition)} if and only if $\sup_{w \in T}\#\Gamma_{1}(w) < \infty$. We set $L_{\ast} \coloneqq \sup_{w \in T}\#\Gamma_{1}(w)$. 
    \end{enumerate}
\end{defn}

We also recall the following standard notion on metric measure spaces; see, e.g., \cite{Hei,Kig20,MT} for further background.  
\begin{defn}\label{defn.AR}
    \begin{enumerate}[label=\textup{(\arabic*)},align=left,leftmargin=*,topsep=2pt,parsep=0pt,itemsep=2pt]
    	\item\label{it:defn.VD} The measure $m$ is said to be \emph{volume doubling}\index{volume doubling} with respect to the metric $d$ if and only if there exists $C_{\mathrm{D}} \in (0,\infty)$ such that
   		    \begin{equation}\label{VD-growth}
        		m(B_{d}(x,2r)) \le C_{\mathrm{D}}\,m(B_{d}(x,r)) \quad \text{for any $(x,r) \in K \times (0,\infty)$.}
    		\end{equation}
    		The constant $C_{\mathrm{D}}$ is called the doubling constant of $m$. 
    	\item\label{it:AR} Let $Q \in (0,\infty)$. The measure $m$ is said to be \emph{$Q$-Ahlfors regular}\index{Ahlfors regular} with respect to the metric $d$ if and only if there exists $C_{\mathrm{AR}} \in [1,\infty)$ such that
    		\begin{equation}\label{AR}
        		C_{\mathrm{AR}}^{-1}\,r^{Q} \le m(B_{d}(x,r)) \le C_{\mathrm{AR}}\,r^{Q} \quad \text{for any $(x,r) \in K \times (0,2\diam(K,d))$.}
    		\end{equation}
    		The measure $m$ is simply said to be \emph{Ahlfors regular} (with respect to $d$) if there exists $Q \in (0,\infty)$ such that $m$ is $Q$-Ahlfors regular. 
    		Also, the metric $d$ is said to be $Q$-Ahlfors regular if there exists a Borel measure $\mu$ on $K$ which is $Q$-Ahlfors regular with respect to $d$. 
    	\item\label{it:QS} A metric $\theta$ on $K$ is said to be \emph{quasisymmetric}\index{quasisymmetry} to $d$, $\theta \underset{\textrm{QS}}{\sim} d$ for short, if and only if there exists a homeomorphism $\eta \colon [0,\infty) \to [0,\infty)$ such that 
    	\begin{equation*}
    		\frac{\theta(x,b)}{\theta(x,a)} \le \eta\biggl(\frac{d(x,b)}{d(x,a)}\biggr) \quad \text{for any $x,a,b \in K$ with $x \neq a$.}
    	\end{equation*}
    	\item\label{it:ARCdim} The \emph{Ahlfors regular conformal dimension}\index{Ahlfors regular conformal dimension} of $(K,d)$ is the value $\dim_{\mathrm{ARC}}(K,d)$ defined as 
    	\begin{equation*}
    		\dim_{\mathrm{ARC}}(K,d) \coloneqq 
    		\inf\biggl\{ Q \in (0,\infty) \biggm|
    		\begin{minipage}{190pt}
    		there exists a metric $\theta$ on $K$ such that $\theta \underset{\textrm{QS}}{\sim} d$ and $\theta$ is $Q$-Ahlfors regular
    		\end{minipage}
   			\biggr\}.	
    	\end{equation*} 
    \end{enumerate}
\end{defn}
If $m$ is Ahlfors regular, then it is clearly volume doubling. 
It is well known that the existence of a $Q$-Ahlfors regular $m$ on $(K,d)$ implies that the Hausdorff dimension of $(K,d)$ is $Q$. 

Now we recall basic geometric conditions in \cite{Kig23}. 
The conditions \ref{BF1}, \ref{BF2} and \ref{BF3} below are important to follow the rest of this paper. 
\begin{assum}[{\cite[Assumption 2.15]{Kig23}}]\label{assum.BF}
	Let $(K, \mathcal{O})$ be a connected compact metrizable space, $\{ K_{w} \}_{w \in T}$ a partition parametrized by the rooted tree $(T, \phi)$, let $d$ be a metric on $K$ that is compatible with the topology $\mathcal{O}$ and $\diam(K, d) = 1$ and let $\measure$ be a Borel probability measure on $K$. 
	There exist $M_{\ast} \in \mathbb{N}$ and $r_{\ast} \in (0, 1)$ such that the following conditions \ref{BF1}--\ref{BF5} hold. 
	\begin{enumerate}[label=\textup{(\arabic*)},align=left,leftmargin=*,topsep=2pt,parsep=0pt,itemsep=2pt]
		\item\label{BF1} $K_{w}$ is connected for any $w \in T$, $\{ K_{w} \}_{w \in T}$ is minimal and uniformly finite, and $\inf_{m \ge 0}\min_{w \in T_{m}}\#S(w) \ge 2$.
		\item\label{BF2} There exist $c_{i} \in (0,\infty)$, $i \in \{ 1,\dots,5 \}$, such that the following conditions (2A)-(2C) are true.
		\begin{itemize}
			\item [(2A)]\label{BF2A} For any $w \in T$,
			\begin{equation}\label{BF.bi-Lip}
				c_{1}\,r_{\ast}^{\abs{w}} \le \diam(K_{w}, \metric) \le c_{2}\,r_{\ast}^{\abs{w}}.
			\end{equation}
			\item [(2B)]\label{BF2B} For any $n \in \mathbb{N}$ and any $x \in K$,
			\begin{equation}\label{BF.adapted}
				B_{d}(x, c_{3}\,r_{\ast}^{n}) \subseteq U_{M_{\ast}}(x; n) \subseteq B_{d}(x, c_{4}\,r_{\ast}^{n}).
			\end{equation}
			(In \cite{Kig20}, the metric $d$ is called \emph{$M_{\ast}$-adapted}\index{adapted (metric)} if the condition \eqref{BF.adapted} holds.)
			\item [(2C)]\label{BF2C} For any $n \in \mathbb{N}$ and $w \in T_{n}$, there exists $x_{w} \in K_{w}$ satisfying
			\begin{equation}\label{BF.thick}
				K_{w} \supseteq B_{d}(x_{w}, c_{5}\,r_{\ast}^{n}).
			\end{equation}
		\end{itemize}
		\item\label{BF3} There exist $m_{1} \in \mathbb{N}$, $\gamma_{1} \in (0, 1)$ and $\gamma \in (0, 1)$ such that
		\begin{equation}\label{BF.super-exp}
			\measure(K_{w}) \ge \gamma\,\measure(K_{\pi(w)}) \quad \text{for any $w \in T$,}
		\end{equation}
		and
		\begin{equation}\label{BF.sub-exp}
			\measure(K_{v}) \le \gamma_{1}\,\measure(K_{w}) \quad \text{for any $w \in T$ and any $v \in S^{m_{1}}(w)$.}
		\end{equation}
		Furthermore, $\measure$ is volume doubling with respect to $\metric$ and  
		\begin{equation}\label{BF-nooverlap}
			\measure(K_{w}) = \sum_{v \in S(w)}\measure(K_{v}) \quad \text{for any $w \in T$.}
		\end{equation}
		\item\label{BF4} There exists $M_{0} \ge M_{\ast}$ such that for any $w \in T$, any $k \ge 1$ and any $v \in S^{k}(w)$,
		\[
		\Gamma_{M_{\ast}}(v) \cap S^{k}(w) \subseteq \biggl\{ v' \in T_{\abs{v}} \biggm| 
		\begin{minipage}{240pt}
            there exist $l \le M_{0}$ and $(v_{0},\dots,v_{l}) \in S^{k}(w)^{l + 1}$ such that $(v_{j - 1},v_{j}) \in E_{\abs{v}}^{\ast}$ for any $j \in \{ 1,\dots,l \}$ 
        \end{minipage}
		\biggr\}. 
		\]
		\item\label{BF5} For any $w \in T$, $\pi(\Gamma_{M_{\ast} + 1}(w)) \subseteq \Gamma_{M_{\ast}}(\pi(w))$.
	\end{enumerate}
\end{assum}

We record a simple consequence of \eqref{BF-nooverlap} in the next proposition. 
\begin{prop}\label{prop.nointersection}
	Assume that the Borel probability measure $m$ satisfies \eqref{BF-nooverlap} in Assumption \ref{assum.BF}-\ref{BF3}. 
	Then $m(K_{v} \cap K_{w}) = 0$ for any $v,w \in T$ with $v \neq w$ and $\abs{v} = \abs{w}$. 
\end{prop}
\begin{proof}
	Let $n \in \mathbb{N} \cup \{ 0 \}$ and $v, w \in T_{n}$ satisfy $v \neq w$. 
	Enumerate $T_{n}$ as $\{ z(1), z(2), \dots, z(l_{n}) \}$ such that $z(1) = v$ and $z(2) = w$, where $l_{n} = \#T_{n}$. 
	Inductively, define $\widetilde{K}_{z(j)}$ by 
	\[
	\widetilde{K}_{z(1)} = K_{z(1)}
	\]
	and 
	\[
	\widetilde{K}_{z(j + 1)} = K_{z(j + 1)} \setminus \left(\bigcup_{i = 1}^{k}\widetilde{K}_{z(i)}\right). 
	\]
	Then $\bigl\{ \widetilde{K}_{z(j)} \bigr\}_{j = 1}^{l_{n}}$ is a disjoint family of Borel sets and $\bigcup_{j = 1}^{l_{n}}\widetilde{K}_{z(j)} = K$. 
	Therefore, 
	\[
	1 = m(K) = \sum_{j = 1}^{l_{n}}m\Bigl(\widetilde{K}_{z(j)}\Bigr). 
	\]
	On the other hand, \eqref{BF-nooverlap} implies that 
	\[
	1 = m(K_{\phi}) = \sum_{j = 1}^{l_{n}}m\bigl(K_{z(j)}\bigr).
	\]
	Therefore, we conclude that $m\bigl(K_{z(j)} \setminus \widetilde{K}_{z(j)}\bigr) = 0$ for any $j \in \{ 1, \dots, l_{n} \}$. 
	In particular, 
	\[
	0 = m\Bigl(K_{z(2)} \setminus \widetilde{K}_{z(2)}\Bigr) = m\Bigl(K_{w} \setminus \bigl(K_{w} \setminus (K_{v} \cap K_{w})\bigr)\Bigr) = m(K_{v} \cap K_{w}), 
	\]
	which completes the proof.
\end{proof}

Next we introduce conductance, neighbor disparity constants and the notion of $p$-conductive homogeneity in Definitions \ref{defn.con-const}, \ref{defn.nei-const} and \ref{defn.pCH}, following \cite[Sections 2.2, 2.3 and 3.3]{Kig23}.
We will state some definitions and statements below for any $p \in (0,\infty)$ or $p \in [1,\infty)$, but on each such occasion we will explicitly declare that we let $p \in (0,\infty)$ or $p \in [1,\infty)$. Our main interest lies in the case $p \in (1,\infty)$. 
\begin{defn}[{\cite[Definitions 2.17 and 3.4]{Kig23}}]\label{defn.con-const}
    Let $p \in (0,\infty)$, $n \in \mathbb{N} \cup \{ 0 \}$ and $A \subseteq T_{n}$.
    \begin{enumerate}[label=\textup{(\arabic*)},align=left,leftmargin=*,topsep=2pt,parsep=0pt,itemsep=2pt]
        \item\label{it:sub.discreteform} Define $\mathcal{E}_{p,A}^{n} \colon \mathbb{R}^{A} \to [0,\infty)$ by
        \[
        \mathcal{E}_{p,A}^{n}(f) \coloneqq \sum_{\{u,v\} \in E_{n}^{\ast}(A)}\abs{f(u) - f(v)}^{p}, \quad f \in \mathbb{R}^{A}.
        \]
        We write $\mathcal{E}_{p}^{n}(f)$ for $\mathcal{E}_{p,T_{n}}^{n}(f)$.   
        \item\label{it:concap} For $A_{0},A_{1} \subseteq A$, define $\mathrm{cap}_{p}^{n}(A_{0},A_{1};A)$ by
        \[
        \mathrm{cap}_{p}^{n}(A_{0},A_{1};A) \coloneqq \inf\bigl\{ \mathcal{E}_{p,A}^{n}(f) \bigm| f \in \mathbb{R}^{A}, \text{$f|_{A_{i}} = i$ for $i \in \{ 0,1 \}$} \bigr\}.
        \]
        \item\label{it:con-const} (Conductance constant\index{conductance constant}) For $A_{1},A_{2} \subseteq A$ and $k \in \mathbb{N} \cup \{ 0 \}$, define 
        \[
        \mathcal{E}_{p,k}(A_{1},A_{2},A) \coloneqq \mathrm{cap}_{p}^{n + k}\bigl(S^{k}(A_{1}), S^{k}(A_{2}); S^{k}(A)\bigr).
        \]
    	For $M \in \mathbb{N}$, define $\mathcal{E}_{M, p, k} \coloneqq \sup_{w \in T}\mathcal{E}_{p,k}(\{w\},T_{\abs{w}} \setminus \Gamma_{M}(w),T_{\abs{w}})$.
    \end{enumerate}
\end{defn}

Let us recall the notion of \emph{covering system}\index{covering system}, which will be used to define neighbor disparity constants and the notion of conductive homogeneity. 
\begin{defn}[{\cite[Definitions 2.26-(3) and 2.29]{Kig23}}]\label{defn.covering}
	Let $N_{T},N_{E} \in \mathbb{N}$. 
	\begin{enumerate}[label=\textup{(\arabic*)},align=left,leftmargin=*,topsep=2pt,parsep=0pt,itemsep=2pt]
	\item Let $n \in \mathbb{N} \cup \{ 0 \}$ and $A \subseteq T_{n}$. A collection $\{ G_{i} \}_{i = 1}^{k}$ with $G_{i} \subseteq T_{n}$ is called a \emph{covering of $(A,E_{n}^{\ast}(A))$ with covering numbers $(N_T,N_E)$} if and only if $A = \bigcup_{i = 1}^{k}G_{k}$, $\max_{x \in A}\#\{ i \mid x \in G_{i} \} \le N_T$ and for any $(u,v) \in E_{n}^{\ast}(A)$, there exists $l \le N_E$ and $\{ w(1),\dots,w(l + 1)\} \subseteq A$ such that $w(1) = u$, $w(l + 1) = v$ and $(w(i),w(i + 1)) \in \bigcup_{j = 1}^{k}E_{n}^{\ast}(G_{j})$ for any $i \in \{ 1,\dots,l \}$. 
	\item Let $\mathscr{J} \subseteq \bigcup_{n \in \mathbb{N} \cup \{ 0 \}}\{ A \mid A \subseteq T_{n} \}$. The collection $\mathscr{J}$ is called a \emph{covering system with covering number $(N_T,N_E)$} if and only if the following conditions are satisfied: 
		\begin{enumerate}[label=\textup{(\roman*)},align=left,leftmargin=*,topsep=2pt,parsep=0pt,itemsep=2pt]
       		\item $\sup_{A \in \mathscr{J}}\#A < \infty$. 
        	\item For any $w \in T$ and any $k \in \mathbb{N}$, there exists a finite subset $\mathscr{N} \subseteq \mathscr{J} \cap T_{\abs{w} + k}$ such that $\mathscr{N}$ is a covering of $\bigl(S^{k}(w),E_{\abs{w} + k}^{\ast}(S^{k}(w))\bigr)$ with covering numbers $(N_T,N_E)$. 
        	\item For any $G \in \mathscr{J}$ and any $k \in \mathbb{N} \cup \{ 0 \}$, if $G \subseteq T_{n}$, then there exists a finite subset $\mathscr{N} \subseteq \mathscr{J} \cap T_{n + k}$ such that $\mathscr{N}$ is a covering of $\bigl(S^{k}(G),E_{n + k}^{\ast}(S^{k}(G))\bigr)$ with covering numbers $(N_T,N_E)$.
    	\end{enumerate}
    	The collection $\mathscr{J}$ is simply said to be a \emph{covering system} if and only if there exist $(N_T, N_E) \in \mathbb{N}^{2}$ such that $\mathscr{J}$ is a covering system with covering number $(N_T,N_E)$. 
    \end{enumerate}
\end{defn}

\begin{defn}[{\cite[Definitions 2.26-(1),(2) and 2.29]{Kig23}}]\label{defn.nei-const}
    Let $p \in (0,\infty)$, $n \in \mathbb{N}$ and $A \subseteq T_{n}$.
    \begin{enumerate}[label=\textup{(\arabic*)},align=left,leftmargin=*,topsep=2pt,parsep=0pt,itemsep=2pt]
        \item\label{it:Pnkf} For $k \in \mathbb{N} \cup \{ 0 \}$ and $f \colon T_{n + k} \to \mathbb{R}$, define $P_{n,k}f \colon T_{n} \to \mathbb{R}$ by
        	\[
        	(P_{n,k}f)(w) \coloneqq \frac{1}{\sum_{v \in S^{k}(w)}\measure(K_{v})}\sum_{v \in S^{k}(w)}f(v)\measure(K_{v}), \quad w \in T_{n}.
        	\]
        (Note that $P_{n,k}f$ depends on the measure $\measure$.)
        \item\label{it:nei-const} (Neighbor disparity constant\index{neighbor disparity constant}) For $k \in \mathbb{N} \cup \{ 0 \}$, define
        \[
        \sigma_{p,k}(A) \coloneqq \sup_{f \colon S^{k}(A) \to \mathbb{R}}\frac{\mathcal{E}_{p,A}^{n}(P_{n,k}f)}{\mathcal{E}_{p,S^{k}(A)}^{n + k}(f)}. 
        \]
        \item\label{it:sigmaJpk}  Let $\mathscr{J} \subseteq \bigcup_{n \ge 0}\{ A \mid A \subseteq T_{n}\}$ be a covering system. Define
        \[
        \sigma_{p,k,n}^{\mathscr{J}} \coloneqq \max\{ \sigma_{p,k}(A) \mid A \in \mathscr{J}, A \subseteq T_{n} \} \quad \text{and} \quad \sigma_{p,k}^{\mathscr{J}} \coloneqq \sup_{n \in \mathbb{N} \cup \{ 0 \}}\sigma_{p,k,n}^{\mathscr{J}}.
        \]
    \end{enumerate}
\end{defn}

\begin{defn}[{\cite[Definition 3.4]{Kig23}}]\label{defn.pCH}
    Let $p \in [1,\infty)$.
    The compact metric space $K$ (with a partition $\{ K_{w} \}_{w \in T}$ and a measure $\measure$) is said to be \emph{$p$-conductively homogeneous}\index{$p$-conductively homogeneous} if and only if there exists a covering system $\mathscr{J}$ such that
    \begin{equation}\label{d:pCH}
        \sup_{k \in \mathbb{N} \cup \{ 0 \}}\sigma_{p,k}^{\mathscr{J}}\mathcal{E}_{M_{\ast},p,k} < \infty.
    \end{equation}
    When we would like to clarify which partition is considered, we also say that $K$ is $p$-conductively homogeneous with respect to $\{ K_{w} \}_{w \in T}$.   
\end{defn}

For our purposes, the following consequence of the $p$-conductive homogeneity is more important than its original definition. 
\begin{thm}[Part of {\cite[Theorem 3.30]{Kig23}}]\label{t:pCH}
    Let $p \in [1,\infty)$ and assume that Assumption \ref{assum.BF} holds.
    If $K$ is $p$-conductively homogeneous, then there exist $\alpha_0,\alpha_1 \in (0,\infty)$, $\sigma_{p} \in (0,\infty)$ and a covering system $\mathscr{J}$ such that for any $k \in \mathbb{N} \cup \{ 0 \}$, 
    \begin{equation}\label{pCH.1}
        \alpha_{0}\sigma_{p}^{-k} \le \mathcal{E}_{M_{\ast},p,k} \le \alpha_{1}\sigma_{p}^{-k} \quad \text{and} \quad \alpha_{0}\sigma_{p}^{k} \le \sigma^{\mathscr{J}}_{p,k} \le \alpha_{1}\sigma_{p}^{k}. 
    \end{equation}
    In particular, the constant $\sigma_{p}$ is determined by the following limit: 
    \begin{equation}\label{p-factor}
        \sigma_{p} = \lim_{k \to \infty}\bigl(\mathcal{E}_{M_{\ast},p,k}\bigr)^{-1/k}.
    \end{equation}
\end{thm}
\begin{rmk}\label{rmk.p-factor}
    The existence of the limit in \eqref{p-factor} is true without the $p$-conductive homogeneity.
    Indeed, if $(K,d,\{ K_{w} \}_{w \in T})$ satisfies the conditions Assumption \ref{assum.BF}-\ref{BF1},\ref{BF2},\ref{BF4},\ref{BF5}, then the sub-multiplicativity inequality for $\{ \mathcal{E}_{M_{\ast},p,k} \}_{k \in \mathbb{N} \cup \{ 0 \}}$ from \cite[Theorem 2.23]{Kig23} together with Fekete's lemma implies the existence of the limit in \eqref{p-factor} for \emph{any} $p \in (0,\infty)$.
    For convenience, we call $\sigma_{p}$ the \emph{$p$-scaling factor} of $(K,d,\{ K_{w} \}_{w \in T})$. 
\end{rmk}

We also recall the ``Sobolev space'' $\mathcal{W}^{p}$ introduced in \cite[Lemma 3.13]{Kig23}.
\begin{defn}\label{d:Kig-sob}
    Let $p \in [1,\infty)$.
    Assume that Assumption \ref{assum.BF}-\ref{BF1},\ref{BF2},\ref{BF4},\ref{BF5} hold and let $\sigma_{p}$ be the constant introduced in \eqref{p-factor} of Theorem \ref{t:pCH}.
    \begin{enumerate}[label=\textup{(\arabic*)},align=left,leftmargin=*,topsep=2pt,parsep=0pt,itemsep=2pt]
    	\item\label{it:Pnf} For $n \in \mathbb{N} \cup \{ 0 \}$, define $P_{n} \colon L^{1}(K,m) \to  \mathbb{R}^{T_{n}}$ by $P_{n}f(w) \coloneqq \fint_{K_{w}}f\,dm$, $w \in T_{n}$. 
        \item\label{it:NpWp} Define $\mathcal{N}_{p} \colon L^{p}(K,m) \to [0,\infty]$ and a linear subspace $\mathcal{W}^{p}$ of $L^{p}(K,m)$ by
        \begin{align*}
            \mathcal{N}_{p}(f) &\coloneqq \left(\sup_{n \in \mathbb{N} \cup \{ 0 \}}\sigma_{p}^{n}\mathcal{E}_{p}^{n}(P_{n}f)\right)^{1/p}, \quad f \in L^{p}(K,m), \\
            \mathcal{W}^{p} &\coloneqq \bigl\{ f \in L^{p}(K,\measure) \bigm| \mathcal{N}_{p}(f) < \infty \bigr\},
        \end{align*}
        and we equip $\mathcal{W}^{p}$ with the norm $\norm{\,\cdot\,}_{\mathcal{W}^{p}}$ defined by 
        \[
        \norm{f}_{\mathcal{W}^{p}} \coloneqq \Bigl(\norm{f}_{L^{p}(K,m)}^{p} + \mathcal{N}_{p}(f)^{p}\Bigr)^{1/p}, \quad f \in \mathcal{W}^{p}. 
        \]
        \item\label{it:defn.basecone} For a linear subspace $\mathcal{D}$ of $\mathcal{W}^{p}$, we define
        \[
        \mathcal{U}_{p}(\mathcal{D}) \coloneqq \biggl\{ \mathscr{E} \colon \mathcal{D} \to [0,\infty) \biggm|
        	\begin{minipage}{278pt}
			$\mathscr{E}^{1/p}$ is a seminorm on $\mathcal{D}$, there exist $\alpha_{0},\alpha_{1} \in (0,\infty)$ such that $\alpha_{0}\mathcal{N}_{p}(f) \le \mathscr{E}(f)^{1/p} \le \alpha_{1}\mathcal{N}_{p}(f)$ for any $f \in \mathcal{D}$
			\end{minipage}
        \biggr\}.
        \]
        For ease of notation, set $\mathcal{U}_{p} \coloneqq \mathcal{U}_{p}(\mathcal{W}^{p})$. 
        \item\label{it:defn.rescaled-discrete-p-energy} For $n \in \mathbb{N} \cup \{ 0 \}$ and $A \subseteq T_{n}$, we define $\widetilde{\mathcal{E}}_{p,A}^{n} \colon L^{p}(K,m) \to [0,\infty)$ by
        \begin{equation*}
            \widetilde{\mathcal{E}}_{p,A}^{n}(f) \coloneqq \sigma_{p}^{n}\mathcal{E}_{p,A}^{n}(P_{n}f), \quad f \in L^{p}(K,m).
        \end{equation*}
        We also set $\widetilde{\mathcal{E}}_{p}^{n}(f) \coloneqq \widetilde{\mathcal{E}}_{p,T_{n}}^{n}(f)$. 
    \end{enumerate}
\end{defn}

We have the following irreducibility property of $\mathcal{N}_{p}$ thanks to the connectedness of $K$ and Assumption \ref{assum.BF}-\ref{BF3}.  
\begin{prop}\label{prop.Np-zero}
	Let $p \in [1,\infty)$. 
	Assume that Assumption \ref{assum.BF} holds, and let $f \in L^{p}(K,\measure)$.
	Then $\mathcal{N}_{p}(f) = 0$ if and only if there exists $c \in \mathbb{R}$ such that $f(x) = c$ for $m$-a.e.\ $x \in K$. 
\end{prop}
\begin{proof}
	It is clear that $\mathcal{N}_{p}(f) = 0$ if $f$ is constant. 
	Assume that $\mathcal{N}_{p}(f) = 0$. 
	Note that $(T_{n},E_{n}^{\ast})$ is a connected graph for each $n \in \mathbb{N} \cup \{ 0 \}$ (\cite[Proposition 2.8]{Kig23}). 
	Therefore, $\mathcal{N}_{p}(f) = 0$ implies that for each $n \in \mathbb{N} \cup \{ 0 \}$ there exists $c_{n} \in \mathbb{R}$ such that $P_{n}f(w) = c_{n}$ for any $w \in T_{n}$. 
	By \eqref{BF-nooverlap} in Assumption \ref{assum.BF}-\ref{BF3}, we have $c_{n} = c_{n + 1}$ and hence there exists $c \in \mathbb{R}$ such that $c_{n} = c$ for any $n \in \mathbb{N} \cup \{ 0 \}$. 
	Now we let $\mathscr{L}_{f} \subseteq K$ denote the set of \emph{Lebesgue points of $f$}\index{Lebesgue point}, i.e., 
	\begin{equation}\label{e:defn.Lebpt}
		\mathscr{L}_{f} \coloneqq \biggl\{ x \in K \biggm| \lim_{r \downarrow 0}\fint_{B_{d}(x,r)}\abs{f(x) - f(y)}\,m(dy) = 0 \biggr\}. 
	\end{equation}
	Then, by the volume doubling property of $m$ and the Lebesgue differentiation theorem (see, e.g., \cite[Theorem 1.8]{Hei}), we have $\mathscr{L}_{f} \in \mathcal{B}(K)$ and $m(K \setminus \mathscr{L}_{f}) = 0$. 
	For any $x \in \mathscr{L}_{f}$ and any $n \in \mathbb{N} \cup \{ 0 \}$, by Proposition \ref{prop.nointersection} and Assumption \ref{assum.BF}-\ref{BF2},\ref{BF3}, 
	\begin{align*}
		\abs{f(x) - c}
		= \abs{f(x) - \fint_{U_{M_{\ast}}(x; n)}f\,dm}
		&\le \frac{1}{m(U_{M_{\ast}}(x; n))}\int_{B_{d}(x,c_{4}r_{\ast}^{n})}\abs{f(x) - f(y)}\,m(dy) \\
		&\le C\fint_{B_{d}(x,c_{4}r_{\ast}^{n})}\abs{f(x) - f(y)}\,m(dy), 
	\end{align*}
	where we used \eqref{BF.adapted} in Assumption \ref{assum.BF}-\ref{BF2} and the volume doubling property of $m$ in the last inequality. 
	Here $C \in (0,\infty)$ is a constant that is independent of $x,f$ and $n$. 
	By letting $n \to \infty$ in the estimate above, we obtain $f(x) = c$ for any $x \in \mathscr{L}_{f}$, which completes the proof. 
\end{proof}

As shown in \cite{Kig23,Shi24}, $\mathcal{W}^{p}$ is a nice Banach space embedded in $\contfunc(K)$ if $K$ is $p$-conductively homogeneous and $p > \dim_{\mathrm{ARC}}(K,d)$. 
More generally, we can show the following theorem.
\begin{thm}\label{thm.Wp}
	Let $p \in [1,\infty)$. 
    Assume that $(K,d,\{ K_{w} \}_{w \in T},m)$ satisfies Assumption \ref{assum.BF} and that $K$ is $p$-conductively homogeneous.
    Then $\mathcal{W}^{p}$ is a Banach space and $\mathcal{W}^{p}$ is dense in $L^{p}(K,m)$. 
    If $p \in (1,\infty)$, then $\mathcal{W}^{p}$ is reflexive and separable. 
    Moreover, if in addition $p > \dim_{\mathrm{ARC}}(K,d)$, then $\mathcal{W}^{p}$ is a dense linear subspace of $(\contfunc(K),\norm{\,\cdot\,}_{\sup})$. 
\end{thm}
\begin{rmk}
	By \cite[Theorem 4.6.9]{Kig20}, the condition $p > \dim_{\mathrm{ARC}}(K,d)$ is equivalent to $\sigma_{p} > 1$. 
\end{rmk}
\begin{proof}[Proof of Theorem \ref{thm.Wp}]
	Note that $\mathcal{W}^{p}$ is a Banach space by \cite[Lemma 3.24]{Kig23} and that $\mathcal{W}^{p}$ is dense in $L^{p}(K,m)$ by \cite[Lemma 3.28]{Kig23}.
    
    In the rest of this proof, we assume that $p \in (1,\infty)$.
    Let us show that $\mathcal{W}^{p}$ is reflexive.
    The upper estimate for $\widetilde{\mathcal{E}}_{p, A}^{k}(P_{k,l}f)$ from \cite[Lemma 2.27]{Kig23} and \eqref{pCH.1} in Theorem \ref{t:pCH} together imply that there exists a constant $C \in (0,\infty)$ such that for any $k,l \in \mathbb{N}$, any $A \subseteq T_{k}$ and any $f \in \mathbb{R}^{S^{l}(A)}$,
    \begin{equation}\label{e:wm}
        \widetilde{\mathcal{E}}_{p, A}^{k}(P_{k,l}f) \le C\widetilde{\mathcal{E}}_{p, S^{l}(A)}^{k + l}(f).
    \end{equation}
    The rest of the proof is very similar to \cite[Proof of Theorem 6.17(ii)]{MS+}, so we give only a sketch (see also \cite[Theorem 5.9]{Shi24} and the proof of Theorem \ref{t:Kig-good}-\ref{Kig.equiv} below). 
    Define $\norm{\,\cdot\,}_{p, n} \coloneqq \Bigl(\norm{\,\cdot\,}_{L^p(K,m)}^{p} + \widetilde{\mathcal{E}}_{p}^{n}(\,\cdot\,)\Bigr)^{1/p}$, which can be regarded as the $L^{p}$-norm on $K \sqcup E_{n}^{\ast}$.
    Also, we consider $\widetilde{\mathcal{E}}_{p}^{n}$ as a $[0,\infty]$-valued functional on $L^{p}(K,m)$.
	From \cite[Theorem 8.5 and Proposition 11.6]{Dal}, by extracting a subsequence of $\{ \widetilde{\mathcal{E}}_{p}^{n} \}_{n \in \mathbb{N}}$ if necessary, we can assume that $\{ \widetilde{\mathcal{E}}_{p}^{n} \}_{n \in \mathbb{N}}$ $\Gamma$-converges to some $p$-homogeneous functional $E_{p} \colon L^{p}(K,m) \to [0,\infty]$ as $n \to \infty$.
    Then $\{ \norm{\,\cdot\,}_{p, n} \}_{n \in \mathbb{N}}$ $\Gamma$-converges to $\trinorm{\,\cdot\,} \coloneqq \bigl(\norm{\,\cdot\,}_{L^{p}(K,m)}^{p} + E_{p}\bigr)^{1/p}$ as $n \to \infty$, and hence $(\trinorm{\,\cdot\,}^{p},\mathcal{W}^{p})$ is a $p$-energy form on $(K,m)$ satisfying \ref{Cp}.  
    By using \eqref{e:wm} and noting that $\lim_{k \to \infty}P_{n}f_{k}(w) = P_{n}f(w)$ for any $n \in \mathbb{N} \cup \{ 0 \}$, any $w \in T_{n}$ and any $f,f_{k} \in L^{p}(K,m)$ with $\lim_{k \to \infty}\norm{f - f_{k}}_{L^{p}(K,m)} = 0$, we can show that $\trinorm{\,\cdot\,}$ is a norm on $\mathcal{W}^{p}$ that is equivalent to $\norm{\,\cdot\,}_{\mathcal{W}^{p}}$.
    Thus, $\mathcal{W}^{p}$ is reflexive by Proposition \ref{p:uc} and the Milman--Pettis theorem.
    The separability of $\mathcal{W}^{p}$ immediately follows from Corollary \ref{cor:separable} (see also \cite[Proposition 4.1]{AHM23}).
    
    In the case of $p > \dim_{\mathrm{ARC}}(K,d)$, $\mathcal{W}^{p}$ can be identified with a subspace of $\contfunc(K)$ and is dense in $(\contfunc(K),\norm{\,\cdot\,}_{\sup})$ by \cite[Lemmas 3.15, 3.16 and 3.19]{Kig23}. 
\end{proof}

Let us introduce an important value, $p$-walk dimension, which will be a main topic in Section \ref{sec.p-walk}.
\begin{defn}[$p$-Walk dimension\index{$p$-walk dimension}]\label{d:values}
    Let $p \in (0,\infty)$.
    Assume that $(K,d,\{ K_{w} \}_{w \in T})$ satisfies Assumption \ref{assum.BF}-\ref{BF1},\ref{BF2},\ref{BF4},\ref{BF5}.
    Let $r_{\ast} \in (0,1)$ be the constant in \eqref{BF.adapted}, let $\sigma_{p}$ be the $p$-scaling factor of $(K,d,\{ K_w \}_{w \in T})$ (see \eqref{p-factor} and Remark \ref{rmk.p-factor}). 
    We define $\tau_{p} \in \mathbb{R}$ by 
    \begin{equation}\label{e:defn.taup}
    	\tau_{p} \coloneqq \frac{\log{\sigma_{p}}}{\log{r_{\ast}^{-1}}}. 
    \end{equation}
    If in addition $m$ is Ahlfors regular with respect to $d$, then we define $\pwalk \in \mathbb{R}$ by  
    \begin{equation} \label{eq:pwalk}
        \pwalk \coloneqq \hdim + \tau_{p}, 
    \end{equation}
    where $\hdim$ denotes the Hausdorff dimension of $(K,d)$. 
    We call $\pwalk$ the \emph{$p$-walk dimension} of $(K,d,\{ K_{w} \}_{w \in T})$.
\end{defn}

Now we prove the main result in this subsection, which is an improvement of \cite[Theorem 3.21]{Kig23}. 
\begin{thm}\label{t:Kig-good} 
	Let $p \in (1,\infty)$. 
    Assume that $(K,d,\{ K_{w} \}_{w \in T},m)$ satisfies Assumption \ref{assum.BF} and that $K$ is $p$-conductively homogeneous.
    Then there exist $\widehat{\mathcal{E}}_{p} \colon \mathcal{W}^{p} \to [0,\infty)$ and $c \in (0,\infty)$ such that the following hold:
    \begin{enumerate}[label=\textup{(\alph*)},align=left,leftmargin=*,topsep=2pt,parsep=0pt,itemsep=2pt]
        \item \label{Kig.equiv} $(\widehat{\mathcal{E}}_{p})^{1/p}$ is a seminorm on $\mathcal{W}^{p}$ and
        \begin{equation}\label{e:Kiggamma.equiv}
        	c\,\mathcal{N}_{p}(f) \le \widehat{\mathcal{E}}_{p}(f)^{1/p} \le \mathcal{N}_{p}(f) \quad \text{for any $f \in \mathcal{W}^{p}$.}
        \end{equation}
        \item \label{Kig.GC} $(\widehat{\mathcal{E}}_{p},\mathcal{W}^{p})$ is a $p$-energy form on $(K,m)$ satisfying \ref{GC}. 
        \item \label{Kig.inv} \textup{(Invariance)} Let $\mathsf{T} \colon (K,\mathcal{B}(K),m) \to (K,\mathcal{B}(K),m)$ be Borel measurable and satisfy $\widetilde{\mathcal{E}}_{p}^{n}(f \circ \mathsf{T}) = \widetilde{\mathcal{E}}_{p}^{n}(f)$ for any $n \in \mathbb{N}$ and any $f \in L^{p}(K,m)$. Then $f \circ \mathsf{T} \in \mathcal{W}^{p}$ and $\widehat{\mathcal{E}}_{p}(f \circ \mathsf{T}) = \widehat{\mathcal{E}}_{p}(f)$ for any $f \in \mathcal{W}^{p}$.
        \item \label{Kig.RF} If in addition $p > \dim_{\mathrm{ARC}}(K,d)$, then $(\widehat{\mathcal{E}}_{p},\mathcal{W}^{p})$ is a regular $p$-resistance form on $K$ and there exist $C \in [1,\infty)$ such that 
    	\begin{equation}\label{RM.comp}
        	C^{-1}d(x,y)^{\tau_{p}}
        	\le R_{\widehat{\mathcal{E}}_{p}}(x,y)
        	\le Cd(x,y)^{\tau_{p}} \quad \text{for any $x,y \in K$}.
    	\end{equation}
    \end{enumerate}
\end{thm}
\begin{proof}
	The most part of the proof will be very similar to that in \cite[Theorem 3.21]{Kig23}, but we present the details because we do not assume $p > \dim_{\mathrm{ARC}}(K,d)$ unlike \cite[Theorem 3.21]{Kig23}. 
    Let $\widehat{\mathcal{E}}_{p}$ be a subsequential $\Gamma$-limit of $\{ \widetilde{\mathcal{E}}_{p}^{n} \}_{n}$ with respect to the topology of $L^{p}(K,m)$ as in \cite[Proof of Theorem 3.21]{Kig23}, i.e., there exists a subsequence $\{ \widetilde{\mathcal{E}}_{p}^{n'} \}_{n'}$ $\Gamma$-converging to $\widehat{\mathcal{E}}_{p}$ with respect to $L^{p}(K,m)$ as $n' \to \infty$.
    (Note that such a subsequential $\Gamma$-limit exists by \cite[Theorem 8.5]{Dal}.)
    
    \ref{Kig.equiv}: 
    $\widehat{\mathcal{E}}_{p}$ is $p$-homogeneous by \cite[Proposition 11.6]{Dal}.  The triangle inequality for $\widehat{\mathcal{E}}_{p}(\,\cdot\,)^{1/p}$ will be included in the proof of \ref{Kig.GC}, so we shall prove \eqref{e:Kiggamma.equiv}.
    From the definition of the $\Gamma$-convergence, it is immediate that $\widehat{\mathcal{E}}_{p}(f) \le \liminf_{n \to \infty}\widetilde{\mathcal{E}}_{p}^{n}(f) \le \mathcal{N}_{p}(f)^{p}$.
    Let us show the former inequality in \eqref{e:Kiggamma.equiv}.
    Let $f \in \mathcal{W}^{p}$ and let $\{ f_{n'} \}_{n'}$ be a recovery sequence of $\{ \widetilde{\mathcal{E}}_{p}^{n} \}_{n'}$ at $f$, i.e., $\lim_{n' \to \infty}\norm{f - f_{n'}}_{L^{p}(K,m)} = 0$ and $\widehat{\mathcal{E}}_{p}(f) = \lim_{n' \to \infty}\widetilde{\mathcal{E}}_{p}^{n'}(f_{n'})$. 
    Since $\lim_{n' \to \infty}P_{k}f_{n'}(w) = P_{k}f(w)$ for any $k \in \mathbb{N}$ and any $w \in T_{k}$, by the estimate \eqref{e:wm} in the proof of Theorem \ref{thm.Wp}, 
    \[
    \widetilde{\mathcal{E}}_{p}^{k}(f)
    = \lim_{n' \to \infty}\widetilde{\mathcal{E}}_{p}^{k}(f_{n'})
    \le C\lim_{n' \to \infty}\widetilde{\mathcal{E}}_{p}^{n'}(f_{n'}) 
    = C\widehat{\mathcal{E}}_{p}(f), 
    \]
    where $C \in (0,\infty)$ is the constant in \eqref{e:wm}. 
    We obtain the desired estimate by taking the supremum over $k \in \mathbb{N} \cup \{ 0 \}$. 

    \ref{Kig.GC}: 
    Let $n_{1},n_{2} \in \mathbb{N}$, $q_{1} \in (0,p]$, $q_{2} \in [p,\infty]$ and $T = (T_{1},\dots,T_{n_{2}}) \colon \mathbb{R}^{n_{1}} \to \mathbb{R}^{n_{2}}$ satisfy \eqref{GC-cond} in Definition \ref{defn.GC}.  
    Define $Q_{n} \colon L^{1}(K,m) \to L^{1}(K,m)$ by 
    \begin{equation}\label{defn.Qn}
    	Q_{n}f \coloneqq \sum_{w \in T_{n}}P_{n}f(w)\indicator{K_{w}} \quad \text{for $f \in L^{1}(K,m)$.} 
    \end{equation}
    Note that $\norm{Q_{n}}_{L^p(K,m) \to L^p(K,m)} \le 1$ by \eqref{BF-nooverlap} in Assumption \ref{assum.BF}-\ref{BF3} and H\"{o}lder's inequality. 
    Let us show $\norm{f - Q_{n}f}_{L^{p}(K,m)} \to 0$ as $n \to \infty$ for any $f \in L^{p}(K,m)$. 
	Define the Hardy--Littlewood maximal operator $\mathscr{M} \colon L^{p}(K, \measure) \to L^{0}(K, \measure)$ by 
	\[
	\mathscr{M}f(x) = \sup_{r > 0}\fint_{B_{\metric}(x, r)}\abs{f(y)}\,\measure(dy), \quad x \in K.
	\]
	Since $m$ is volume doubling with respect to $d$ by Assumption \ref{assum.BF}-\ref{BF3}, there exists a constant $C \in (0,\infty)$ such that $\norm{\mathscr{M}f}_{L^{p}(K,m)} \le C\norm{f}_{L^{p}(K,m)}$ for any $f \in L^{p}(K, m)$ (see, e.g., \cite[Theorem 3.5.6]{HKST}).
	We also easily see that for any $f \in L^{p}(K,m)$ and any $x \in K$, 
	\begin{align*}
		\abs{Q_{n}f(x)} 
		\le \sum_{w \in T_{n}; x \in K_{w}}\abs{P_{n}f(w)} 
		&\le \sum_{w \in T_{n}; x \in K_{w}}\frac{m(B_{d}(x,2c_{2}r_{\ast}^{n}))}{m(K_{w})}\fint_{B_{d}(x,2c_{2}r_{\ast}^{n})}\abs{f}\,dm \\
		&\le \sum_{w \in T_{n}; x \in K_{w}}\frac{m(B_{d}(x,2c_{2}r_{\ast}^{n}))}{m(B_{d}(x_{w},c_{5}r_{\ast}^{n}))}\mathscr{M}f(x)
		\le C_{1}\mathscr{M}f(x), 
	\end{align*}
	where $x_{w} \in K_{w}$ and $c_{2}, c_{5}$ are the same as in Assumption \ref{assum.BF}-\ref{BF2} and we used the volume doubling property in the last inequality, and $C_{1} \in (0,\infty)$ is a constant depending only on $\sup_{w \in T}\#\Gamma_{1}(w)$, $c_{2},c_{5}$ and the doubling constant of $m$. 
	Let $f \in L^{p}(K,m)$ and let $\mathscr{L}_{f} \subseteq K$ denote the set of Lebesgue points of $f$ as in \eqref{e:defn.Lebpt}. 
	Then $\mathscr{L}_{f} \in \mathcal{B}(K)$ and $m(K \setminus \mathscr{L}_{f}) = 0$ by the Lebesgue differentiation theorem for a volume doubling metric measure space (see, e.g., \cite[(2.8)]{Hei}). 
	Since 
	\begin{align*}
		\abs{f(x) - Q_{n}f(x)}
		&\le \sum_{w \in T_{n}; x \in K_{w}}\fint_{K_{w}}\abs{f(x) - f(y)}\,m(dy) \\
		&\le C_{1}\fint_{B_{d}(x,2c_{2}r_{\ast}^{n})}\abs{f(x) - f(y)}\,m(dy), 
	\end{align*}
	we have $\abs{f(x) - Q_{n}f(x)} \to 0$ as $n \to \infty$ for any $x \in \mathscr{L}_{f}$. 
	Now the dominated convergence theorem implies $\norm{f - Q_{n}f}_{L^{p}(K,m)} \to 0$.
    
    Let $\bm{u} = (u_{1},\dots,u_{n_{1}}) \in (\mathcal{W}^{p})^{n_{1}}$ and choose a recovery sequence $\{ u_{k,n'} \}_{n'}$ of $\{ \widetilde{\mathcal{E}}_{p}^{n'} \}_{n'}$ at $u_{k}$ for each $k \in \{ 1,\dots,n_{1} \}$.
    For brevity, we write $\bm{u}_{n'} = (u_{1,n'},\dots,u_{n_{1},n'})$ and 
    \begin{align*}
        P_{n'}\bm{u}_{n'}(v)
        &= \bigl(P_{n'}u_{1,n'}(v),\dots,P_{n'}u_{n_1,n'}(v)\bigr) \in \mathbb{R}^{n_{1}}, \quad v \in T_{n'},  \\
        Q_{n'}\bm{u}_{n'}(v)
        &= \bigl(Q_{n'}u_{1,n'}(v),\dots,Q_{n'}u_{n_1,n'}(v)\bigr) \in \mathbb{R}^{n_{1}}, \quad v \in T_{n'}.
    \end{align*}
    Note that $\norm{u_{n'} - Q_{n'}u_{k,n'}}_{L^{p}(K,m)} \to 0$ as $n' \to \infty$ by the fact proved in the previous paragraph.
    By using $\norm{Q_{n}}_{L^p(K,m) \to L^p(K,m)} \le 1$ and the estimate \eqref{T.Lp}, one can show 
    \begin{equation}\label{Lp.approx}
        \norm{T_{l}(\bm{u}) - T_{l}(Q_{n'}\bm{u}_{n'})}_{L^{p}(K,m)} \xrightarrow[n' \to \infty]{} 0 \quad \text{for any $l \in \{ 1,\dots,n_{2} \}$};
    \end{equation}
    see \cite[Proof of Theorem 3.21 on p.~46]{Kig23} for the details.
    Also, by Proposition \ref{prop.nointersection}, we note that 
    \begin{equation}\label{commutative}
         P_{n'}\bigl(T_{l}(Q_{n'}\bm{u}_{n'})\bigr) = T_{l}(P_{n'}\bm{u}_{n'}) \in \mathbb{R}^{T_{n'}} \quad \text{for any $l \in \{ 1,\dots,n_{2} \}$}.
    \end{equation}
    With these preparations, we prove \ref{GC} for $(\widehat{\mathcal{E}}_{p},\mathcal{W}^{p})$.
    We consider the case of $q_{2} < \infty$ since the case of $q_{2} = \infty$ is similar.
    By \eqref{Lp.approx} and \eqref{commutative}, we see that
    \begin{align}\label{improve}
        \sum_{l = 1}^{n_{2}}\widehat{\mathcal{E}}_{p}\bigl(T_{l}(\bm{u})\bigr)^{q_{2}/p}
        &\overset{\eqref{Lp.approx}}{\le} \sum_{l = 1}^{n_{2}}\liminf_{n' \to \infty}\widetilde{\mathcal{E}}_{p}^{n'}\bigl(T_{l}(Q_{n'}\bm{u}_{n'})\bigr)^{q_{2}/p} \nonumber\\
        &\overset{\eqref{commutative}}{\le} \liminf_{n' \to \infty}\sum_{l = 1}^{n_{2}}\Biggl[\frac{\sigma_{p}^{n'}}{2}\sum_{(v,w) \in E_{n'}^{\ast}}\abs{T_{l}(P_{n'}\bm{u}_{n'}(v)) - T_{l}(P_{n'}\bm{u}_{n'}(w))}^{q_{2}\cdot\frac{p}{q_{2}}}\Biggr]^{q_{2}/p} \nonumber\\
        &\overset{\eqref{reverse}}{\le} \liminf_{n' \to \infty}\left(\frac{\sigma_{p}^{n'}}{2}\sum_{(v,w) \in E_{n'}^{\ast}}\norm{T(P_{n'}\bm{u}_{n'}(v)) - T(P_{n'}\bm{u}_{n'}(v))}_{\ell^{q_{2}}}^{p}\right)^{q_{2}/p} \nonumber\\ 
        &\overset{\eqref{GC-cond}}{\le} \liminf_{n' \to \infty}\left(\frac{\sigma_{p}^{n'}}{2}\sum_{(v,w) \in E_{n'}^{\ast}}\norm{P_{n'}\bm{u}(v) - P_{n'}\bm{u}(v)}_{\ell^{q_{1}}}^{p}\right)^{q_{2}/p} \nonumber\\
        &\le \liminf_{n' \to \infty}\left(\frac{\sigma_{p}^{n'}}{2}\sum_{(v,w) \in E_{n'}^{\ast}}\left[\sum_{k = 1}^{n_{1}}\abs{P_{n'}u_{k,n'}(v) - P_{n'}u_{k,n'}(w)}^{p \cdot \frac{q_{1}}{p}}\right]^{p/q_{1}}\right)^{q_{2}/p} \nonumber\\
        &\overset{\textup{($\ast$)}}{\le} \liminf_{n' \to \infty}\left(\sum_{k = 1}^{n_{1}}\Biggl[\frac{\sigma_{p}^{n'}}{2}\sum_{(v,w) \in E_{n'}^{\ast}}\abs{P_{n'}u_{k,n'}(v) - P_{n'}u_{k,n'}(w)}^{p}\Biggr]^{q_{1}/p}\right)^{\frac{p}{q_{1}} \cdot \frac{q_{2}}{p}} \nonumber\\
        &\le \left(\sum_{k = 1}^{n_{1}}\limsup_{n' \to \infty}\widetilde{\mathcal{E}}_{p}^{n'}(u_{k,n'})^{q_{1}/p}\right)^{\frac{p}{q_{1}} \cdot \frac{q_{2}}{p}}
        \le \left(\sum_{k = 1}^{n_{1}}\widehat{\mathcal{E}}_{p}(u_{k})^{q_{1}/p}\right)^{\frac{p}{q_{1}} \cdot \frac{q_{2}}{p}}, 
    \end{align}
    where we used the triangle inequality for the $\ell^{p/q_{1}}$-norm on $E_{n}^{\ast}$ in ($\ast$). 
    Hence $(\widehat{\mathcal{E}}_{p},\mathcal{W}^{p})$ satisfies \ref{GC}.

    \ref{Kig.inv}:
    This is clear from the definitions of $\mathcal{W}^{p}$ and of $\widehat{\mathcal{E}}_{p}$.
    
    \ref{Kig.RF}: 
    In the case of $p > \dim_{\mathrm{ARC}}(K,d)$, a combination of \ref{Kig.GC}, \cite[Lemmas 3.13, 3.16, 3.19 and Theorem 3.21]{Kig23} and Theorem \ref{thm.Wp} implies that $(\widehat{\mathcal{E}}_{p},\mathcal{W}^{p})$ is a regular $p$-resistance form on $K$. 
    Then the estimate \eqref{RM.comp} is exactly the same as  \cite[(3.21) in Lemma 3.34]{Kig23}, so we complete the proof. 
\end{proof}
\begin{rmk}\label{rmk.others}
    The construction of $\mathcal{E}_{p}^{\Gamma}$ in \cite[Theorem 6.22]{MS+} is very similar to that of $\widehat{\mathcal{E}}_{p}$ in the proof above although the setting and assumption on a partition in \cite{MS+} is slightly different from ours.
    Thanks to Proposition \ref{prop.nointersection}, the operators $M_{n}$ and $J_{n}$ defined in \cite[(6.8) and (6.9)]{MS+} correspond to $P_{n}$ and $Q_{n}$ respectively. 
    In particular, \eqref{Lp.approx} and \eqref{commutative} for $M_{n}$ and $J_{n}$ are also true. 
    Hence we can easily see that the $p$-energy form $(\mathcal{E}_{p}^{\Gamma},\mathcal{F}_{p})$ in \cite[Theorem 6.22]{MS+} also satisfies \ref{GC}. 
\end{rmk}

Before concluding this subsection, we deal with the capacity upper estimate and a Poincar\'{e}-type inequality under the additional assumption on the Ahlfors regularity of $m$. 
In addition to the density of $\mathcal{W}^{p}$ in $\contfunc(K)$, we can obtain the following capacity upper bound under the $p$-conductive homogeneity of $K$ if $p > \dim_{\mathrm{ARC}}(K,d)$ and $m$ is Ahlfors regular. 
\begin{prop}[Capacity upper estimate\index{capacity upper estimate}]\label{prop.Kig-capu}
	Let $p \in (1,\infty)$ and $\lambda \in (1,\infty)$. 
	Assume that Assumption \ref{assum.BF} holds, that $K$ is $p$-conductively homogeneous, that $p > \dim_{\mathrm{ARC}}(K,d)$ and that $m$ is Ahlfors regular. 
	Then there exists $C \in (0,\infty)$ such that for any $(x,r) \in K \times (0,1]$,
	\begin{equation}\label{e:Kig.capu}
		\inf\bigl\{ \mathcal{N}_{p}(u)^{p} \bigm| \text{$u \in \mathcal{W}^{p}$, $u|_{B_{d}(x,r)} = 1$, $\supp_{K}[u] \subseteq B_{d}(x,\lambda r)$} \bigr\} \le C\frac{m(B_{d}(x,r))}{r^{\pwalk}}. 
	\end{equation} 
\end{prop}
\begin{proof}  
    Let $r_{\ast} \in (0,1)$ and $M_{\ast} \in \mathbb{N}$ be the constants in Assumption \ref{assum.BF}.
    For $r \in (0,1]$, choose $n \in \mathbb{N}$ as the minimal positive integer such that $c_{2}(M_{\ast} + 1)r_{\ast}^{n} < (\lambda - 1)r$, where $c_{2}$ is the constant in \eqref{BF.bi-Lip}. 
    Let $x \in K$ and set $T_{n}(x,r) \coloneqq T_{n}[B_{d}(x,r)]$ for ease of notation.
    Then, by the metric doubling property of $(K,d)$, there exists $N \in \mathbb{N}$ which is independent of $x$ and $r$ such that $\#T_{n}(x,r) \le N$. 
    By \cite[Lemma 3.18]{Kig23} and its proof, where a partition of unity consisting of elements of $\mathcal{W}^{p}$ is constructed, for any $w \in T_{n}(x,r)$ there exists $h_{M_{\ast},w} \in \mathcal{W}^{p}$ such that $h_{M_{\ast},w}|_{K_{w}} = 1$, $\supp_{K}[h_{M_{\ast},w}] \subseteq \bigcup_{v \in \Gamma_{M_{\ast}}(w)}K_{v}$ and $\mathcal{N}_{p}(h_{M_{\ast},w})^{p} \lesssim \sigma_{p}^{n}$.
    Now we define $\psi_{x,r} \coloneqq \sum_{w \in T_{n}(x,r)}h_{M_{\ast},w} \in \mathcal{W}^{p}$.
    Then $\psi_{x,r}|_{B_{d}(x,r)} \ge 1$, $\supp_{K}[\psi_{x,r}] \subseteq B_{d}(x,\lambda r)$ and 
    \[
    \mathcal{N}_{p}(\psi_{x,r})^{p}
    \le N^{p - 1}\max_{w \in T_{n}(x,r)}\mathcal{N}_{p}(h_{M_{\ast},w})^{p} 
    \lesssim \sigma_{p}^{n} = r_{\ast}^{n(\hdim - \pwalk)} \lesssim r^{\hdim - \pwalk}.  
    \]
    Since $m$ is Ahlfors regular and $\mathcal{N}_{p}(\psi_{x,r} \wedge 1) \lesssim \mathcal{N}_{p}(\psi_{x,r})$ by \cite[Theorem 3.21]{Kig23}, we obtain \eqref{e:Kig.capu}.
\end{proof}

The following Poincar\'{e}-type inequality for cells can be shown by using a result on relations between neighbor disparity and Poincar\'e constants from \cite[Theorem 5.11]{Kig23}.
\begin{lem}\label{lem.Kig-prePI1}
	Let $p \in (1,\infty)$. 
	Assume that Assumption \ref{assum.BF} holds, that $K$ is $p$-conductively homogeneous, and that $m$ is Ahlfors regular. 
	Then there exists a constant $C \in (0,\infty)$ such that for any $f \in L^{p}(K,m)$ and any $w \in T$,
	\begin{equation}\label{e:Kig-prePI1}
		\int_{K_{w}}\abs{f(x) - \fint_{K_{w}}f\,dm}^{p}\,m(dx)
        \le Cr_{\ast}^{\abs{w}\pwalk}\liminf_{n \to \infty}\widetilde{\mathcal{E}}_{p,S^{n}(w)}^{n + \abs{w}}(f).
	\end{equation}
\end{lem}
\begin{proof}
	Set $k \coloneqq \abs{w}$. 
	Recall that $\lim_{n \to \infty}\norm{Q_{n}f - f}_{L^{p}(K,m)} = 0$ as shown in the proof of Theorem \ref{t:Kig-good}-\ref{Kig.GC}.
	Hence, for any $n \in \mathbb{N}$, we see that
	\begin{align}\label{e:lambda-limit}
		&\frac{1}{m(K_{w})}\sum_{v \in S^{n}(w)}\abs{P_{n + k}f(v) - P_{k}f(w)}^{p}m(K_{v}) \nonumber \\
		&= \frac{1}{m(K_{w})}\sum_{v \in S^{n}(w)}\int_{K_{v}}\abs{Q_{n + k}f(x) - P_{k}f(w)}^{p}\,m(dx) \nonumber \\
		&= \fint_{K_{w}}\abs{Q_{n + k}f(x) - P_{k}f(w)}^{p}\,m(dx)
		\xrightarrow[n \to \infty]{} \fint_{K_{w}}\abs{f(x) - P_{k}f(w)}^{p}\,m(dx), 
	\end{align}
	where we used Proposition \ref{prop.nointersection} in the second equality. 
	By \cite[(5.11) in Theorem 5.11]{Kig23} and \eqref{pCH.1} in Theorem \ref{t:pCH}, there exists $C \in (0,\infty)$ which is independent of $f$ and $n$ such that
	\begin{equation}\label{e:lambda}
		\frac{1}{m(K_{w})}\sum_{v \in S^{n}(w)}\abs{P_{n + k}f(v) - P_{k}f(w)}^{p}m(K_{v}) 
		\le Cr_{\ast}^{k(\pwalk - \hdim)}\widetilde{\mathcal{E}}_{p,S^{n}(w)}^{n + k}(f).
	\end{equation}
	We obtain \eqref{e:Kig-prePI1} by combining \eqref{e:lambda-limit}, \eqref{e:lambda}, \eqref{BF.thick} in Assumption \ref{assum.BF}-\ref{BF2}, and the Ahlfors regularity of $m$.
\end{proof}

To upgrade \eqref{e:Kig-prePI1} in Lemma \ref{lem.Kig-prePI1} to a Poincar\'{e} inequality for metric balls in $K$, we need the following standard fact. 
\begin{lem}[{\cite[Lemma 4.17]{BjBj}}]\label{lem:varcomp} 
	Let $q \in [1,\infty)$ and let $(Y,\mathcal{A},\mu)$ be a measure space. 
	For any $f \in L^{1}(Y,\mu)$ and any $E \in \mathcal{A}$ with $\mu(E) \in (0,\infty)$, 
	\begin{equation}\label{e:varcomp}
    	\fint_{E}\abs{f - \fint_{E}f\,d\mu}^{q}\,d\mu \le 2^{q}\inf_{a \in \mathbb{R}}\fint_{E}\abs{f - a}^{q}\,d\mu. 
    \end{equation}
\end{lem}

Now we prove a Poincar\'{e}-type inequality in terms of discrete $p$-energy forms.
\begin{prop}\label{prop.Kig-prePI2}
	Let $p \in (1,\infty)$. 
	Assume that Assumption \ref{assum.BF} holds, that $K$ is $p$-conductively homogeneous, and that $m$ is Ahlfors regular. 
    Then there exist $C,\alpha \in (0,\infty)$ such that for any $(x,r) \in K \times (0,1]$ and any $f \in L^{p}(K,m)$,
    \begin{equation}\label{e:Kig-prePI2}
    	\int_{B_{d}(x,r)}\abs{f - \fint_{B_{d}(x,r)}f\,dm}^{p}\,dm \le Cr^{\pwalk}\liminf_{k \to \infty}\widetilde{\mathcal{E}}_{p,T_{k}[B_{d}(x,\alpha r)]}^{k}(f). 
    \end{equation}
\end{prop}
\begin{proof}
    Throughout this proof, $M_{\ast} \in \mathbb{N}$ and $r_{\ast} \in (0,1)$ are the same constants as in Assumption \ref{assum.BF}.
    Let $(x,r) \in K \times (0,1]$. 
    We first consider the case of $r \in (c_{3}r_{\ast},1]$, where $c_{3}$ is the constant in \eqref{BF.adapted} in Assumption \ref{assum.BF}-\ref{BF2}. 
    By the Poincar\'e inequality \eqref{e:Kig-prePI1} for cells in Lemma \ref{lem.Kig-prePI1} with $w = \phi$, 
    \begin{align*}
    	\int_{B_{d}(x,r)}\abs{f - \fint_{B_{d}(x,r)}f\,dm}^{p}\,dm 
    	&\overset{\eqref{e:varcomp}}{\le} 2^{p}\int_{B_{d}(x,r)}\abs{f - \fint_{K}f\,dm}^{p}\,dm \\
    	&\le 2^{p}\int_{K}\abs{f - \fint_{K}f\,dm}^{p}\,dm 
    	\overset{\eqref{e:Kig-prePI1}}{\le} 2^{p}C\liminf_{n \to \infty}\widetilde{\mathcal{E}}_{p}^{n}(f), 
    \end{align*}
    where $C \in (0,\infty)$ is the constant in \eqref{e:Kig-prePI1}. 
    Since $\diam(K,d) = 1$, this shows \eqref{e:Kig-prePI2} for any $A \ge (c_{3}r_{\ast})^{-1}$. 
    Hence it suffices to consider the remaining case, i.e., the case of $r \in (0, c_{3}r_{\ast}]$. 
    Let $n \in \mathbb{N}$ satisfy $c_{3}r_{\ast}^{n} \ge r > c_{3}r_{\ast}^{n + 1}$. 
    Set $\Gamma_{M_{\ast}}(x; n) \coloneqq \{ v \in T \mid \text{$v \in \Gamma_{M_{\ast}}(w)$ for some $w \in T_{n}$ such that $x \in K_{w}$} \}$.
    Then we see that
    \begin{align}\label{e:Kig.var} 
        &\int_{U_{M_{\ast}}(x; n)}\abs{f(y) - \fint_{U_{M_{\ast}}(x; n)}f\,dm}^{p}\,m(dy) \nonumber \\ 
        &\le 2^{p - 1}\sum_{w \in \Gamma_{M_{\ast}}(x; n)}\left(\int_{K_{w}}\abs{f(y) - P_{n}f(w)}^{p}\,m(dy) + m(K_{w})\abs{P_{n}f(w) - \fint_{U_{M_{\ast}}(x; n)}f\,dm}^{p}\right)  \nonumber \\
        &\lesssim \sum_{w \in \Gamma_{M_{\ast}}(x; n)}\left(r^{\pwalk}\liminf_{k \to \infty}\widetilde{\mathcal{E}}_{p,S^{k}(w)}^{n + k}(f) + r^{\hdim}\abs{P_{n}f(w) - \fint_{U_{M_{\ast}}(x; n)}f\,dm}^{p}\right).
    \end{align}
    Note that, by Proposition \ref{prop.nointersection},  
    \begin{equation}\label{e:maxavediff}
    	P_{n}f(w) - \fint_{U_{M_{\ast}}(x; n)}f\,dm
    	= \frac{1}{m(U_{M_{\ast}}(x; n))}\sum_{v \in \Gamma_{M_{\ast}}(x; n)}(P_{n}f(w) - P_{n}f(v))m(K_{v}). 
    \end{equation}
    For any $w \in \Gamma_{M_{\ast}}(x; n)$, by choosing $w' \in \Gamma_{M_{\ast}}(x; n) \setminus \{ w \}$ so that $\abs{P_{n}f(w) - P_{n}f(w')} = \max_{v \in \Gamma_{M_{\ast}}(x; n)}\abs{P_{n}f(w) - P_{n}f(v)}$, we have from \eqref{e:maxavediff} and Proposition \ref{prop.nointersection} that 
    \[
    \abs{P_{n}f(w) - \fint_{U_{M_{\ast}}(x; n)}f\,dm} \le \abs{P_{n}f(w) - P_{n}f(w')}. 
    \]
    Hence, by H\"{o}lder's inequality, \eqref{pCH.1} in Theorem \ref{t:pCH}, and a weak monotonicity estimate for discrete $p$-energy forms from \cite[Lemma 2.27]{Kig23}, 
    \begin{align}\label{e:avediff.2}
        \abs{P_{n}f(w) - \fint_{U_{M_{\ast}}(x; n)}f\,dm}^{p}
        &\le (2M_{\ast} + 1)^{p - 1}\mathcal{E}_{p,\Gamma_{M_{\ast}}(x; n)}^{n}(P_{n}f) \nonumber \\
        &\lesssim r^{\pwalk - \hdim}\liminf_{k \to \infty}\widetilde{\mathcal{E}}_{p,S^{k}(\Gamma_{M_{\ast}}(x; n))}^{n + k}(f). 
    \end{align}
    Note that $\#\Gamma_{M_{\ast}}(x; n) \le L_{\ast}^{M_{\ast} + 2}$ by Assumption \ref{assum.BF}-\ref{BF1} and that $S^{k}(\Gamma_{M_{\ast}}(x; n)) \subseteq T_{n + k}[B_{d}(x,c_{4}r_{\ast}^{n})] \subseteq T_{n + k}[B_{d}(x,c_{3}^{-1}r_{\ast}^{-1}c_{4}r)]$ by Assumption \ref{assum.BF}-\ref{BF2}, where $c_{4}$ is the same as in \eqref{BF.adapted}. 
    Now we set $A \coloneqq (1 \vee c_{4})c_{3}^{-1}r_{\ast}^{-1}$. 
    Then, by \eqref{e:Kig.var} and \eqref{e:avediff.2}, 
    \begin{align*}
        &\int_{U_{M_{\ast}}(x; n)}\abs{f(y) - \fint_{U_{M_{\ast}}(x; n)}f\,dm}^{p}\,m(dy)  \\
        &\overset{\eqref{e:avediff.2}}{\lesssim} r^{\pwalk}\liminf_{k \to \infty}\sum_{w \in \Gamma_{M_{\ast}}(x; n)}\widetilde{\mathcal{E}}_{p,S^{k}(\Gamma_{M_{\ast}}(x; n))}^{n + k}(f)
        \le L_{\ast}^{M_{\ast} + 2}r^{\pwalk}\liminf_{k \to \infty}\widetilde{\mathcal{E}}_{p,T_{k}[B_{d}(x,Ar)]}^{k}(f). 
    \end{align*}
    Since 
    \begin{align*}
    	\int_{B_{d}(z,s)}\abs{f(y) - \fint_{B_{d}(x,r)}f\,dm}^{p}\,m(dy)
    	&\overset{\eqref{e:varcomp}}{\le} 2^{p}\int_{B_{d}(z,s)}\abs{f(y) - \fint_{U_{M_{\ast}}(x; n)}f\,dm}^{p}\,m(dy) \\
    	&\overset{\eqref{BF.adapted}}{\le} 2^{p}\int_{U_{M_{\ast}}(x; n)}\abs{f(y) - \fint_{U_{M_{\ast}}(x; n)}f\,dm}^{p}\,m(dy), 
    \end{align*}
    we obtain \eqref{e:Kig-prePI2}.
\end{proof}

\subsection{Self-similar \texorpdfstring{$p$}{p}-energy forms on \texorpdfstring{$p$}{p}-conductively homogeneous self-similar structures}\label{sec.Kigss}
In this subsection, we construct a self-similar $p$-resistance form on self-similar structures under suitable assumptions.
Our main result in this subsection, Theorem \ref{thm.KSgood-ss}, implies that self-similar $p$-energy forms constructed in \cite[Theorem 4.6]{Kig23} satisfy \ref{GC}.

We start with some preparations before constructing self-similar $p$-resistance forms.
In the following definition, we introduce a good partition parametrized by a rooted tree.  
\begin{defn}[{\cite[Definition 4.2]{Kig23}}]\label{d:scale}
    Let $\mathcal{L} = (K,S,\{ F_{i} \}_{i \in S})$ be a self-similar structure, let $r \in (0,1)$ and let $(j_{s})_{s \in S} \in \mathbb{N}^{S}$.
    Define
    \[
    j(w) \coloneqq \sum_{i = 1}^{n}j_{w_{i}} \quad \textrm{and} \quad g(w) \coloneqq r^{j(w)} \quad \textrm{for $w = w_{1} \ldots w_{n} \in W_{\ast}$ ($j(\emptyset) \coloneqq 0$).}
    \]
    Define $\widetilde{\pi}(w_{1} \ldots w_{n}) \coloneqq w_{1} \ldots w_{n-1}$ for $w = w_{1} \ldots w_{n} \in W_{\ast}$ ($\widetilde{\pi}(\emptyset) \coloneqq \emptyset$), $\Lambda_{1}^{g} \coloneqq \{ \emptyset \}$ and
    \[
    \Lambda_{r^{k}}^{g} \coloneqq \{ w = w_{1} \ldots w_{n} \in W_{\ast} \mid g(\widetilde{\pi}(w)) > r^{k} \ge g(w) \}, \quad k \in \mathbb{N}.
    \]
    Set $T_{k}^{(r)} \coloneqq \{ (k,w) \mid w \in \Lambda_{r^{k}}^{g} \}$, $T^{(r)} \coloneqq \bigcup_{k \in \mathbb{N} \cup \{ 0 \}}T_{k}^{(r)}$ and define $\iota \colon T^{(r)} \to W_{\ast}$ as $\iota(k,w) \coloneqq w$.
    Moreover, define $E_{T^{(r)}} \subseteq T^{(r)} \times T^{(r)}$ by
    \[
    E_{T^{(r)}} \coloneqq \Bigl\{ ((k,v), (k + 1,w)) \in T_{k}^{(r)} \times T_{k + 1}^{(r)} \Bigm| \textrm{$k \in \mathbb{N} \cup \{ 0 \}$, $v = w$ or $v = \widetilde{\pi}(w)$} \Bigr\}, 
    \]
    so that $(T^{(r)},E_{T^{(r)}})$ is a rooted tree (\cite[Proposition 4.3]{Kig23}). 
\end{defn}

In the rest of this subsection, we presume the following assumption on the geometry of our self-similar structure. 
\begin{assum}\label{assum.BFss} 
    Let $\mathcal{L} = (K,S,\{ F_{i} \}_{i \in S})$ be a self-similar structure with $\#S \ge 2$ and $K$ connected.
    There exist $r_{\ast} \in (0,1)$ and a metric $d$ on $K$ giving the original topology of $K$ with $\diam(K,d) = 1$ such that $(K,d,\{ K_{w} \}_{w \in T^{(r_{\ast})}},m)$ satisfies Assumption \ref{assum.BF}, where $K_{w} \coloneqq K_{\iota(w)}$ for $w \in T^{(r_{\ast})}$ and $m$ is the Borel measure on $K$ defined by \eqref{eq:ssmeas.canonical} with $(\theta_{i})_{i \in S} \coloneqq (r_{\ast}^{j_{s}\hdim})_{s \in S}$ for the unique $\hdim \in \mathbb{R}$ satisfying $\sum_{s \in S}r_{\ast}^{j_{s}\hdim} = 1$; note that $\hdim > 0$ by $\#S \ge 2$. 
\end{assum}

Under Assumption \ref{assum.BFss}, we have the $\hdim$-Ahlfors regularity of $m$ as follows.
\begin{prop}[{\cite[Proposition 4.5]{Kig23}}]\label{prop.ARss}
	The value $\hdim$ coincides with the Hausdorff dimension of $(K,d)$ and $m$ is $\hdim$-Ahlfors regular with respect to $d$. 
\end{prop}

To obtain a self-similar $p$-energy form on $\mathcal{L}$, we first discuss the self-similarity of $\mathcal{W}^{p}$ (recall \eqref{SSE1}).
The following lemma can be shown in exactly the same way as \cite[Theorem 4.6-(1)]{Kig23} although the condition $p > \dim_{\mathrm{ARC}}(K,d)$ is assumed in \cite[Theorem 4.6]{Kig23}. 
\begin{lem}\label{lem.PSS-easy.pre}
	For any $u \in L^{p}(K,m)$, any $k \in \mathbb{N} \cup \{ 0 \}$ and any $n \in \mathbb{N} \cup \{ 0 \}$ with $n \ge \max_{w \in W_{k}}j(w)$, 
	\begin{equation}\label{e:PSS-easy.pre}
		\sum_{w \in W_{k}}\mathcal{E}_{p}^{n - j(w)}(P_{n - j(w)}(u \circ F_{w})) \le \mathcal{E}_{p}^{n}(P_{n}u). 
	\end{equation}
	In particular, if in addition $K$ is $p$-conductively homogeneous (with respect to $\{ K_{w} \}_{w \in T^{(r_{\ast})}}$), then $u \circ F_{w} \in \mathcal{W}^{p}$ for any $u \in \mathcal{W}^{p}$ and any $w \in W_{\ast}$, and hence 
	\begin{equation}\label{e:PSS.easyside}
		\mathcal{W}^{p} \cap \contfunc(K) \subseteq \{ u \in \contfunc(K) \mid \textrm{$u \circ F_{i} \in \mathcal{W}^{p}$ for any $i \in S$} \}. 
	\end{equation}
\end{lem}

Similar to the case of $p = 2$ (see, e.g., \cite{Kig00,KZ92}), we will obtain a self-similar $p$-energy form on $(\mathcal{L},m)$ with weight $\bm{\sigma}_{p} \coloneqq (\sigma_{p}^{j_{s}})_{s \in S}$ as a fixed point obtained by applying Theorem \ref{thm.ssenergy-fix}. 
To this end, we need the converse inclusion of \eqref{e:PSS.easyside} and uniform estimates on $\mathcal{S}_{\bm{\sigma_{p}},n}(E)$ for any/some $E \in \mathcal{U}_{p}(\mathcal{W}^{p} \cap \contfunc(K))$; recall the definition of $\mathcal{S}_{\bm{\sigma}_{p},n}$ in Definition \ref{defn.ssenergyoperator}. 
These conditions are true if $K$ is $p$-conductively homogeneous and $p > \dim_{\mathrm{ARC}}(K,d)$ as described in the following proposition. (This result is essentially proved in \cite[Proof of Theorem 4.6]{Kig23}.) 

\begin{prop}\label{prop.PSS}
	Let $p \in (1,\infty)$ and assume that $K$ is $p$-conductively homogeneous (with respect to $\{ K_{w} \}_{w \in T^{(r_{\ast})}}$).
	If $p > \dim_{\mathrm{ARC}}(K,d)$, then 
	\begin{equation}\label{e:PSS.domain.lowdim}
		\mathcal{W}^{p} = \{ u \in \contfunc(K) \mid \textrm{$u \circ F_{i} \in \mathcal{W}^{p}$ for any $i \in S$} \}, 
	\end{equation}
	and there exists $C \in [1,\infty)$ such that for any $E \in \mathcal{U}_{p}$ \textup{(recall Definition \ref{d:Kig-sob}-\ref{it:defn.basecone})}, any $u \in \mathcal{W}^{p}$ and any $n \in \mathbb{N}$, 
	\begin{equation}\label{e:PSS.form}
		C^{-1}\mathcal{N}_{p}(u)^{p} \le \mathcal{S}_{\bm{\sigma_{p}},n}(E)(u) \le C\mathcal{N}_{p}(u)^{p}. 
	\end{equation}
\end{prop}
\begin{proof}
	The uniform estimate \eqref{e:PSS.form} is proved in \cite[(4.6) and (4.8)]{Kig23} 
	(the proof of \cite[(4.8)]{Kig23} requires the assumption $p > \dim_{\mathrm{ARC}}(K,d)$). 
	It suffices to prove
	\[
	\mathcal{W}^{p} \supseteq \{ u \in \contfunc(K) \mid \text{$u \circ F_{i} \in \mathcal{W}^{p}$ for any $i \in S$} \} \eqqcolon \mathcal{W}^{p}_{S} 
	\]
	(since the converse inclusion follows by combining \eqref{e:PSS.easyside} in Lemma \ref{lem.PSS-easy.pre} with $\mathcal{W}^{p} \subseteq \contfunc(K)$ from Theorem \ref{thm.Wp}).
    We note the following estimate obtained in \cite[p.~61, lines 8--9]{Kig23} under the assumption $p > \dim_{\mathrm{ARC}}(K,d)$: there exists a constant $C' \in (0,\infty)$ such that
    \begin{equation}\label{pre-domss}
        \widetilde{\mathcal{E}}_{p}^{n}(u)
        \le C'\sum_{w \in W_{n}}\sigma_{p}^{j(w)}\mathcal{N}_{p}(u \circ F_{w})^{p}
        = C'\mathcal{S}_{\bm{\sigma}_{p},n}(\mathcal{N}_{p}^{p})(u) \quad \textrm{for any $u \in \mathcal{W}^{p}_{S}$ and $n \in \mathbb{N}$.}
    \end{equation}
    Taking the supremum over $n \in \mathbb{N}$ in the left-hand side of \eqref{pre-domss}, we have $\mathcal{W}^{p}_{S} \subseteq \mathcal{W}^{p}$.
\end{proof}

Now we can obtain a desired self-similar $p$-energy form. 
The following theorem is a generalization of \cite[Theorem 4.6]{Kig23} taking into account the case of $p \le \dim_{\mathrm{ARC}}(K,d)$.
\begin{thm}\label{thm.KSgood-ss}
	Let $p \in (1,\infty)$. 
	Assume that Assumption \ref{assum.BFss} holds, that $K$ is $p$-conductively homogeneous (with respect to $\{ K_{w} \}_{w \in T^{(r_{\ast})}}$) and that the following \emph{pre-self-similarity conditions}\index{pre-self-similarity conditions} hold: 
	\begin{gather}
		\mathcal{W}^{p} \cap \contfunc(K) = \{ u \in \contfunc(K) \mid \text{$u \circ F_{i} \in \mathcal{W}^{p}$ for any $i \in S$} \}. \label{e:PSS.dom} \\
		\textrm{There exists $C \in [1,\infty)$ such that \eqref{e:PSS.form} holds for any $u \in \mathcal{W}^{p} \cap \contfunc(K)$, $n \in \mathbb{N}$.} \label{e:PSS.general} 
	\end{gather} 
	Let $\sigma_{p}$ be the constant introduced in \eqref{p-factor} of Theorem \ref{t:pCH}, set $\bm{\sigma}_{p} \coloneqq (\sigma_{p}^{j_s})_{s \in S}$, let $(\widehat{\mathcal{E}}_{p},\mathcal{W}^{p})$ be any $p$-energy form on $(K,m)$ given in Theorem \ref{t:Kig-good} and set $\mathcal{F}_{p} \coloneqq \closure{\mathcal{W}^{p} \cap \contfunc(K)}^{\mathcal{W}^{p}}$. 
	Then there exists $\{ n_{k} \}_{k \in \mathbb{N}} \subseteq \mathbb{N}$ with $n_{k} < n_{k + 1}$ for any $k \in \mathbb{N}$ such that the following limit exists in $[0,\infty)$ for any $u \in \mathcal{F}_{p}$: 
	\begin{equation}\label{e:defn.Kigss}
		\mathcal{E}_{p}(u) \coloneqq \lim_{k \to \infty}\frac{1}{n_{k}}\sum_{j = 0}^{n_{k} - 1}\mathcal{S}_{\bm{\sigma}_{p},j}(\widehat{\mathcal{E}}_{p})(u). 
	\end{equation}
    Moreover, the following properties hold: 
    \begin{enumerate}[label=\textup{(\alph*)},align=left,leftmargin=*,topsep=2pt,parsep=0pt,itemsep=2pt]
        \item \label{Kigss.equiv} $(\mathcal{E}_{p},\mathcal{F}_{p})$ is a self-similar $p$-energy form on $(\mathcal{L},m)$ with weight $\bm{\sigma}_{p}$, and there exist $\alpha_{0},\alpha_{1} \in (0,\infty)$ such that $\alpha_{0}\,\mathcal{N}_{p}(u)^{p} \le \mathcal{E}_{p}(u) \le \alpha_{1}\,\mathcal{N}_{p}(u)^{p}$ for any $u \in \mathcal{F}_{p}$.  
        \item\label{Kigss.GC} $(\mathcal{E}_{p},\mathcal{F}_{p})$ satisfies \ref{GC}.
        \item \label{Kigss.slocal} $(\mathcal{E}_{p},\mathcal{F}_{p})$ satisfies the strong local property \hyperref[it:SL1]{\textup{(SL1)}} \textup{(recall Definition \ref{defn.Epsl})}.
        \item \label{Kigss.RF} If in addition $p > \dim_{\mathrm{ARC}}(K,d)$, then $\mathcal{F}_{p} = \mathcal{W}^{p}$ and $(\mathcal{E}_{p},\mathcal{F}_{p})$ is a regular self-similar $p$-resistance form on $\mathcal{L}$ with weight $\bm{\sigma}_{p}$ and there exist $\alpha_{0},\alpha_{1} \in (0,\infty)$ such that
        \begin{equation}\label{RM-ss.comp}
            \alpha_{0}\,d(x,y)^{\tau_p} \le R_{\mathcal{E}_{p}}(x,y) \le \alpha_{1}\,d(x,y)^{\tau_p} \quad \text{for any $x,y \in K$.}
        \end{equation}
    \end{enumerate}
\end{thm}
\begin{rmk}\label{rmk.CGQ-KigforANF}
	\begin{enumerate}[label=\textup{(\arabic*)},align=left,leftmargin=*,topsep=2pt,parsep=0pt,itemsep=2pt]
		\item In the case of $p > \dim_{\mathrm{ARC}}(K,d)$, the pre-self-similarity conditions, \eqref{e:PSS.dom} and \eqref{e:PSS.general}, can be dropped by virtue of Proposition \ref{prop.PSS}. 
		\item On p.-c.f.\ self-similar structures, self-similar $p$-energy forms have been constructed also in \cite{CGQ22}; we show in Subsection \ref{sec.pcf} below that the self-similar $p$-energy forms considered in \cite{CGQ22} are all $p$-resistance forms (on $V_{\ast}$, and ones on $K$ if the weight $\bm{\rweight} = (\rweight_{i})_{i \in S}$ of the form satisfies $\min_{i \in S}\rweight_{i} > 1$). 
		Note that \emph{any} $p \in (1,\infty)$ is allowed in the framework of \cite{CGQ22} unlike that of \cite{Kig23} (see Theorem \ref{thm.KSgood-ss}-\ref{Kigss.RF} above). 
		While it is extremely hard to determine the value $\dim_{\mathrm{ARC}}(K,d)$ in general, for a p.-c.f.\ self-similar set $K$ typically $\dim_{\mathrm{ARC}}(K,d) = 1$; a result in this spirit can be found in \cite[Theorem 1.2]{CP14}. 
		In Appendix \ref{sec.confdimANF}, we prove that the Ahlfors regular conformal dimension of any \emph{strongly symmetric p.-c.f.\ self-similar set} (see Framework \ref{frmwrk:ANF} and Definition \ref{defn.ANF} below) is one when it is equipped with the $p$-resistance metric of a nice self-similar $p$-resistance form; the proof is based on the existence of self-similar $p$-resistance forms on strongly symmetric p.-c.f.\ self-similar sets proved in Theorem \ref{thm.eigenform-ANF} as an extension of \cite[Theorem 6.3]{CGQ22}. 
	\end{enumerate}
\end{rmk}
\begin{proof}[Proof of Theorem \ref{thm.KSgood-ss}]
	The existence of the limit in \eqref{e:defn.Kigss} and its properties \ref{Kigss.equiv}, \ref{Kigss.GC} and \ref{Kigss.slocal} are immediate from \eqref{e:PSS.dom}, \eqref{e:PSS.form}, Lemma \ref{lem.ssform-ext}, Theorem \ref{thm.ssenergy-fix}, Propositions \ref{prop.ssenergy-GCinv}-\ref{it:ssenergy.GC} and \ref{prop.ssenergy-sl}. 
	Let us verify \ref{Kigss.RF}. 
	Recall that $\mathcal{W}^{p} \subseteq \contfunc(K)$ by $p > \dim_{\mathrm{ARC}}(K,d)$ (Theorem \ref{thm.Wp}), whence $\mathcal{F}_{p} = \mathcal{W}^{p}$. 
    A similar argument as in the proof of Theorem \ref{t:Kig-good}-\ref{Kig.RF} shows that $(\mathcal{E}_{p},\mathcal{W}^{p})$ is a regular $p$-resistance form on $K$ satisfying \eqref{RM-ss.comp}. 
    This completes the proof. 
\end{proof}

Similar to Theorem \ref{t:KS-mono.pcf}, we can obtain the monotonicity of $\sigma_{p}^{1/(p - 1)}$ in $p > \dim_{\mathrm{ARC}}(K,d)$. \index{monotonicity of $\sigma_{p}^{1/(p - 1)}$}
Note that the following result is \emph{not} restricted to p.-c.f.\ self-similar structures.
\begin{thm}\label{t:KS-mono}
	Assume that Assumption \ref{assum.BFss} holds. 
    Let $p,q \in (\dim_{\mathrm{ARC}}(K,d),\infty)$ satisfy $p \le q$, and assume that $K$ is $s$-conductively homogeneous (with respect to $\{ K_{w} \}_{w \in T^{(r_{\ast})}}$) for each $s \in \{ p,q \}$. 
    Then
    \begin{equation}\label{KS-mono}
        \sigma_{p}^{1/(p - 1)} \le \sigma_{q}^{1/(q - 1)}.
    \end{equation}
\end{thm}
\begin{proof}
    The proof is very similar to that of Theorem \ref{t:KS-mono.pcf}.
    We see from Proposition \ref{prop.PSS} that \eqref{e:PSS.dom} and \eqref{e:PSS.form} with $s \in \{ p,q \}$ in place of $p$ hold.   
    Let $(\mathcal{E}_{s},\mathcal{W}^{s})$ be a self-similar $s$-resistance form on $\mathcal{L}$ given in Theorem \ref{thm.KSgood-ss} for each $s \in \{ p,q \}$.
    Fix two distinct points $x_{0}, y_{0} \in K$, set $B \coloneqq \{ x_{0}, y_{0} \}$ and define $h_{p} \coloneqq h_{B}^{\mathcal{E}_{p}}\bigl[\indicator{x_{0}}^{B}\bigr] \in \mathcal{W}^{p}$.
    Then $0 \le h_{p} \le 1$ by the weak comparison principle (Proposition \ref{prop.cp1}) and we can find $w \in W_{\ast}$ satisfying $K_{w} \cap B = \emptyset$ and $h_{p,w} \coloneqq h_{p} \circ F_{w} \not\in \mathbb{R}\indicator{K}$.
    Similar to \eqref{diff.harm}, by using \eqref{lip-harm} and \eqref{pRMss}, we can show that for any $\{ u,v \} \in E_{n}^{\ast}$,
    \begin{equation*}
        \abs{P_{n}h_{q,w}(u) - P_{n}h_{q,w}(v)}^{q - p} 
        \le Cr_{\ast}^{n(\pwalk - \hdim)\frac{q - p}{p - 1}},
    \end{equation*}
    where $C \in (0,\infty)$ is independent of $n$.
    Hence we have
    \[
    \mathcal{E}_{q}^{n}(h_{p,w})
    = \sum_{\{ u,v \} \in E_{n}^{\ast}}\abs{P_{n}h_{q,w}(u) - P_{n}h_{q,w}(v)}^{q}
    \le Cr_{\ast}^{n(\pwalk - \hdim)\frac{q - p}{p - 1}}\mathcal{E}_{p}^{n}(h_{p,w}),
    \]
    which implies that
    \begin{equation}\label{diff.harm-Kig}
        \Bigl(\sigma_{q}^{-1}\sigma_{p}^{(q - 1)/(p - 1)}\Bigr)^{n}\widetilde{\mathcal{E}}_{q}^{n}(h_{p,w})
        \le C\widetilde{\mathcal{E}}_{p}^{n}(h_{p,w})
        \le C\mathcal{N}_{p}(h_{p,w})^{p}.
    \end{equation}
    By \eqref{e:wm} in the proof of Theorem \ref{thm.Wp}, there exists $C_{q} \in (0,\infty)$ such that $\mathcal{N}_{q}(f)^{q} \le C_{q}\liminf_{n \to \infty}\widetilde{\mathcal{E}}_{q}^{n}(f)$ for any $f \in L^{q}(K,m)$.
    This together with \eqref{diff.harm-Kig} implies that 
    \[
    \mathcal{N}_{q}(h_{p,w})^{q}\limsup_{n \to \infty}\Bigl(\sigma_{q}^{-1}\sigma_{p}^{(q - 1)/(p - 1)}\Bigr)^{n}
    \le C'\mathcal{N}_{p}(h_{p,w})^{p}
    < \infty.
    \]
    Since $\mathcal{N}_{q}(h_{p,w}) > 0$, we obtain $\sigma_{q}^{-1}\sigma_{p}^{(q - 1)/(p - 1)} \le 1$, which means \eqref{KS-mono}.
\end{proof}

We conclude this subsection by applying Theorem \ref{thm.EHI} (elliptic Harnack inequality) to the $p$-energy form $(\mathcal{E}_{p},\mathcal{F}_{p})$ in Theorem \ref{thm.KSgood-ss} in the case of $p > \dim_{\mathrm{ARC}}(K,d)$.
We immediately obtain the following corollary by combining Propositions \ref{prop.ss-pform-em}, \ref{prop.Rp-TPE}, \ref{prop.Kig-capu}, \ref{prop.ARss}, \eqref{BF.adapted} in Assumption \ref{assum.BF}-\ref{BF2}, \eqref{RM-ss.comp} in Theorem \ref{thm.KSgood-ss}-\ref{Kigss.RF}, and Theorem \ref{thm.EHI}. 
Recall \eqref{e:defn.em.one} in the paragraph before Proposition \ref{prop.ss-pform-em} for the self-similar $p$-energy measures associated with a self-similar $p$-energy form.
\begin{cor}[Elliptic Harnack inequality\index{elliptic Harnack inequality} for self-similar $p$-resistance form]\label{cor.Kig-EHI}
	Let $p \in (1,\infty)$. 
	Assume that Assumption \ref{assum.BFss} holds, that $K$ is $p$-conductively homogeneous (with respect to $\{ K_{w} \}_{w \in T^{(r_{\ast})}}$) and that $p > \dim_{\mathrm{ARC}}(K,d)$.
	Then $(\mathcal{E}_{p},\mathcal{W}^{p})$ given in Theorem \ref{thm.KSgood-ss} and the self-similar $p$-energy measures $\{ \Gamma_{\mathcal{E}_{p}}\langle u \rangle \}_{u \in \mathcal{W}^{p}}$ associated with $(\mathcal{E}_{p},\mathcal{W}^{p})$ satisfy the assumptions, and thereby the property in the conclusion, of Theorem \ref{thm.EHI} with $K,m,\frac{m(B_{d}(x,s))}{s^{\pwalk}}$ in place of $X,\mu,\Upsilon(x,s)$. 
\end{cor}

\subsection{Self-similar \texorpdfstring{$p$}{p}-resistance forms on p.-c.f.\ self-similar structures}\label{sec.pcf}
In this subsection, under the condition (\textbf{R}) in \cite[p.~18]{CGQ22}, we see that the construction of $p$-energy forms on p.-c.f.\ self-similar structures constructed due to \cite{CGQ22} yields $p$-resistance forms.
The framework in \cite{CGQ22} focuses only on p.-c.f.\ self-similar structures, but allows any $p \in (1,\infty)$ throughout, and also the choice of the weights of self-similar $p$-resistance forms is flexible there so that non-arithmetic weights can arise unlike Theorem \ref{thm.KSgood-ss}; see Subsection \ref{sec:gap} for a proof that non-arithmetic weights do arise in the framework of Subsection \ref{sec.ANF} under a mild condition on the p.-c.f.\ self-similar structure $\mathcal{L}$. 

In the following definitions, we recall some classes of $p$-energy forms on finite sets considered in \cite{CGQ22}. 
\begin{defn}[{\cite[Definition 2.1]{CGQ22}}]\label{defn.CGQ-Mp}
	Let $A$ be a finite set with $\#A \ge 2$. 
	Let $E \colon \mathbb{R}^{A} \to [0,\infty)$ and consider the following conditions. 
	\begin{enumerate}[label=\textup{(\roman*)},align=left,leftmargin=*,topsep=2pt,parsep=0pt,itemsep=2pt]
		\item\label{it:Mp1} $E(tf + (1-t)g) \le tE(f) + (1 - t)E(g)$ for any $f,g \in \mathbb{R}^{A}$ and any $t \in [0,1]$. 
		\item\label{it:Mp2} $E(tf) = \abs{t}^{p}E(f)$ for any $f \in \mathbb{R}^{A}$ and any $t \in \mathbb{R}$. 
		\item\label{it:Mp3} $E(f + t\indicator{A}) = E(f)$ for any $f \in \mathbb{R}^{A}$ and any $t \in \mathbb{R}$. 
		\item\label{it:Mp4} $E(f^{+} \wedge 1) \le E(f)$ for any $f \in \mathbb{R}^{A}$. 
		\item\label{it:Mp5} $\{ f \in \mathbb{R}^{A} \mid E(f) = 0 \} = \mathbb{R}\indicator{A}$. 
	\end{enumerate}
	We define $\mathcal{M}_{p}(A)$ and $\widetilde{\mathcal{M}}_{p}(A)$ by 
	\begin{align}
		\mathcal{M}_{p}(A) 
		&\coloneqq \{ E \colon \mathbb{R}^{A} \to [0,\infty) \mid \text{$E$ satisfies \ref{it:Mp1}-\ref{it:Mp5}} \}, \label{e:defn.Mp} \\
		\widetilde{\mathcal{M}}_{p}(A) 
		&\coloneqq \{ E \colon \mathbb{R}^{A} \to [0,\infty) \mid \text{$E$ satisfies \ref{it:Mp1}-\ref{it:Mp4}} \}. \label{e:defn.tildeMp} 
	\end{align}
\end{defn}

\begin{defn}[{\cite[Definition 2.8]{CGQ22}}]\label{defn.CGQ-Qp}
	Let $A$ be a finite set with $\#A \ge 2$. 
	For $E_{1},E_{2} \in \widetilde{\mathcal{M}}_{p}(A)$, define a metric $d_{\widetilde{\mathcal{M}}_{p}(A)}$ on $\widetilde{\mathcal{M}}_{p}(A)$ by 
	\begin{align}\label{e:defn.tildeMpdist}
		d_{\widetilde{\mathcal{M}}_{p}(A)}(E_{1},E_{2}) 
		&\coloneqq \sup\Bigl\{ \abs{E_{1}(u) - E_{2}(u)} \Bigm| \textrm{$u \in \mathbb{R}^{A}$, $\osc_{A}[u] = 1$} \Bigr\} \nonumber \\
		&\,= \sup\Bigl\{ \abs{E_{1}(u) - E_{2}(u)} \Bigm| \textrm{$u \in \mathbb{R}^{A}$, $\osc_{A}[u] \leq 1$} \Bigr\}.
	\end{align}
	For ease of notation, we set $\abs{E}_{\widetilde{\mathcal{M}}_{p}(A)} \coloneqq d_{\widetilde{\mathcal{M}}_{p}(A)}(E,0)$ for $E \in \widetilde{\mathcal{M}}_{p}(A)$.  
	\begin{enumerate}[label=\textup{(\arabic*)},align=left,leftmargin=*,topsep=2pt,parsep=0pt,itemsep=2pt]
		\item\label{it:Sp} We define $\mathcal{S}_{p}(A) \subseteq \mathcal{M}_{p}(A)$ by 
		\begin{equation}\label{e:defn.Sp}
			\mathcal{S}_{p}(A) \coloneqq \biggl\{ E \in \mathcal{M}_{p}(A) \biggm| 
			\begin{minipage}{246pt}
				there exists $(c_{xy})_{x,y \in A} \subseteq [0,\infty)$ such that\\
				$E(f) = \sum_{x,y \in A}c_{xy}\abs{f(x) - f(y)}^{p}$ for any $f \in \mathbb{R}^{A}$
			\end{minipage}
			\biggr\}. 
		\end{equation}
		Note that any $E \in \mathcal{S}_{p}(A)$ is a $p$-resistance form on $A$ as observed in Example \ref{ex.pRF}-\ref{RF-graph}. 
		\item\label{it:Qp} We define $\mathcal{Q}_{p}'(A) \subseteq \mathcal{M}_{p}(A)$ by 
		\begin{equation}\label{e:defn.preQp}
			\mathcal{Q}_{p}'(A) \coloneqq \biggl\{ E \in \mathcal{M}_{p}(A) \biggm| 
			\begin{minipage}{182pt}
				there exist a finite set $B$ with $B \supseteq A$ and $\widetilde{E} \in \mathcal{S}_{p}(B)$ such that $\widetilde{E}\big|_{A} = E$ 
			\end{minipage}
			\biggr\}, 
		\end{equation}
		where $\widetilde{E}\big|_{A}$ denotes the trace of $\widetilde{E}$ on $A$ (recall \eqref{eq:dfn-trace} in Theorem \ref{thm.RF-exist}).
		We further define $\mathcal{Q}_{p}(A) \subseteq \mathcal{M}_{p}(A)$ as the closure of $\mathcal{Q}_{p}'(A)$ in $(\mathcal{M}_{p}(A),d_{\widetilde{\mathcal{M}}_{p}(A)})$, i.e., 
		\begin{equation}\label{e:defn.Qp}
			\mathcal{Q}_{p}(A) \coloneqq \biggl\{ E \in \mathcal{M}_{p}(A) \biggm| 
			\begin{minipage}{178pt}
				there exists $\{ E_{n} \}_{n \in \mathbb{N}} \subseteq \mathcal{Q}_{p}'(A)$ such that $\lim_{n \to \infty}d_{\widetilde{\mathcal{M}}_{p}(A)}(E, E_{n}) = 0$
			\end{minipage}
			\biggr\}. 
		\end{equation} 
	\end{enumerate}
\end{defn} 

Then any $E \in \mathcal{Q}_{p}(A)$ is a $p$-resistance form on $A$, as stated in the following proposition.
\begin{prop}\label{prop.RF-newQp}
	Let $A$ be a finite set with $\#A \ge 2$ and let $E \in \mathcal{Q}_{p}(A)$. 
	Then $E$ is a $p$-resistance form on $A$. 
\end{prop}
\begin{proof}
	Note that \ref{RF1}--\ref{RF4} for $E \in \mathcal{Q}_{p}(A)$ are clear, so we shall prove \ref{RF5}, i.e., \ref{GC}. 
	Choose $\{ E_{n} \}_{n \in \mathbb{N}} \subseteq \mathcal{Q}_{p}'(A)$ so that $\lim_{n \to \infty}d_{\widetilde{\mathcal{M}}_{p}(A)}(E, E_{n}) = 0$. 
	Then clearly $\lim_{n \to \infty}E_{n}(u) = E(u)$ for any $u \in \mathbb{R}^{A}$ (see also \cite[Lemma A.1]{CGQ22} for a refined version of this convergence), and since $E_{n}$ satisfies \ref{GC} by Theorem \ref{thm.RF-exist} for any $n \in \mathbb{N}$, we have \ref{GC} for $E$ by the stability of \ref{GC} under pointwise convergence from Proposition \ref{prop.cone-gen}-\ref{GC.pwlimit}. 
\end{proof}

Next we introduce renormalization operators playing central roles in the construction of $p$-energy forms on p.-c.f.\ self-similar structures. 
In the rest of this subsection, we always assume that $\mathcal{L} = (K,S,\{ F_{i} \}_{i \in S})$ is a p.-c.f.\ self-similar structure with $\#S \ge 2$ and $K$ connected.
\begin{defn}[Renormalization operator\index{renormalization operator}; {\cite[Definition 3.1]{CGQ22}}] \label{defn.renorm-op}
    Let $\bm{\rweight}_{p} = (\rweight_{p,i})_{i \in S} \in (0,\infty)^{S}$ and $k \in \mathbb{N} \cup \{ 0 \}$. 
    For a $p$-resistance form $E$ on $V_{k}$, define $p$-resistance forms
     $\Lambda_{\bm{\rweight}_{p}}(E) \colon \mathbb{R}^{V_{k + 1}} \to [0,\infty)$ and $\mathcal{R}_{\bm{\rweight}_{p}}(E) \colon \mathbb{R}^{V_{k}} \to [0,\infty)$ by\footnote{We use different symbols from \cite{CGQ22}.}
    \begin{equation}\label{defn.renorm}
    	\Lambda_{\bm{\rweight}_{p}}(E)(u) \coloneqq \sum_{i \in S}\rweight_{p,i}E(u \circ F_{i}), \quad u \in \mathbb{R}^{V_{k + 1}}, \quad \text{and} \quad \mathcal{R}_{\bm{\rweight}_{p}}(E) \coloneqq \Lambda_{\bm{\rweight}_{p}}(E)\bigr|_{V_{k}} 
    \end{equation}
    (recall Proposition \ref{prop.RFrenorm} and Theorem \ref{thm.RF-exist}).
    To be precise, $\Lambda_{\bm{\rweight}_{p}},\mathcal{R}_{\bm{\rweight}_{p}}$ depend on $k$, but we omit it for ease of the notation.
    By \cite[Lemma 3.2-(b)]{CGQ22}, we have $\Lambda_{\bm{\rweight}_{p}}^{n}(E)\bigr|_{V_{k}} = \mathcal{R}_{\bm{\rweight}_{p}}^{n}(E)$ for any $n \in \mathbb{N} \cup \{ 0 \}$, i.e.,
    \[
    \mathcal{R}_{\bm{\rweight}_{p}}^{n}(E)(u) = \inf\Biggl\{ \sum_{w \in W_{n}}\rweight_{p,w}\,E(v \circ F_{w}) \Biggm| v \in \mathbb{R}^{V_{n + k}}, v|_{V_{k}} = u \Biggr\}, \quad u \in \mathbb{R}^{V_{k}}.
    \] 
\end{defn}

The following theorem, which is an adaptation of \cite[Theorem 4.2]{CGQ22}, gives a necessary and sufficient condition for the existence of an eigenform with respect to $\mathcal{R}_{\bm{\rweight}_{p}}$. 
This theorem can be shown by combining \cite[Lemma 4.4 and Proof of Theorem 4.2]{CGQ22} with Proposition \ref{prop.RF-newQp}, so we omit the proof. 
\begin{thm}[Condition for the existence of an eigenform\index{eigenform}; cf.\ {\cite[Theorem 4.2]{CGQ22}}]\label{thm.eigenform} 
    Let $\bm{\rweight}_{p} = (\rweight_{p,i})_{i \in S} \in (0,\infty)^{S}$. 
    Let us consider the following condition \eqref{condA}: there exist $c \in (0,\infty)$ and a $p$-resistance form $E$ on $V_{0}$ such that
    \begin{equation}\label{condA}
    	\min_{x,y \in V_{0}; x \neq y}R_{\mathcal{R}_{\bm{\rweight}_{p}}^{n}(E)}(x,y) \ge c\max_{x,y \in V_{0}; x \neq y}R_{\mathcal{R}_{\bm{\rweight}_{p}}^{n}(E)}(x,y) \quad \text{for any $n \in \mathbb{N} \cup \{ 0 \}$}. \tag{\textbf{A}}
    \end{equation}
    \begin{enumerate}[label=\textup{(\alph*)},align=left,leftmargin=*,topsep=2pt,parsep=0pt,itemsep=2pt]
    	\item\label{it:CGQ.eigenvalue} Assume that \eqref{condA} holds. Then there exists a unique number $\lambda = \lambda(\bm{\rweight}_{p}) \in (0,\infty)$ such that the following hold: for any $E' \in \mathcal{M}_{p}(V_{0})$, there exists $C \in [1,\infty)$ such that 
    	\begin{equation}\label{e:charact.eigenvalue}
    		C^{-1}\lambda^{n}E'(u) \le \mathcal{R}_{\bm{\rweight}_{p}}^{n}(E')(u) \le C\lambda^{n}E'(u) \quad \text{for any $n \in \mathbb{N} \cup \{ 0 \}$ and any $u \in \mathbb{R}^{V_{0}}$.}
    	\end{equation}
    	\item\label{it:CGQ.newinitial} Assume that \eqref{condA} holds. Let $E_{0} \in \mathcal{S}_{p}(V_{0})$. For $n \in \mathbb{N}$, define $E_{n} \in \mathcal{Q}_{p}'(V_{0})$ by 
    	\begin{equation}\label{e:defn.CGGQ-En}
    		E_{n}(u) \coloneqq \inf\biggl\{ \frac{1}{n + 1}\sum_{j = 0}^{n}\lambda^{-j}\Lambda_{\bm{\rweight}_{p}}^{j}(E_{0})(v|_{V_{j}}) \biggm| v \in \mathbb{R}^{V_{n}}, v|_{V_{0}} = u \biggr\}, \quad u \in \mathbb{R}^{V_{0}},  
    	\end{equation}
    	where $\lambda$ is the number given in \ref{it:CGQ.eigenvalue}. 
    	Then there exists a subsequence $\{ E_{n_{k}} \}_{k \in \mathbb{N}}$ such that it converges in the topology induced by $d_{\widetilde{\mathcal{M}}_{p}}$. 
    	In particular, there exists $E_{\ast} \in \mathcal{Q}_{p}(V_{0})$ such that 
    	\begin{equation}\label{e:CGQlimit}
    		E_{\ast}(u) = \lim_{k \to \infty}E_{n_{k}}(u), \quad u \in \mathbb{R}^{V_{0}}. 
    	\end{equation}
    	\item\label{it:CGQ.eigenform} Assume that \eqref{condA} holds. Let $E_{0} \in \mathcal{S}_{p}(V_{0})$, let $E_{\ast} \in \mathcal{Q}_{p}(V_{0})$ be given by \eqref{e:CGQlimit} and let $\lambda$ be the number given in \ref{it:CGQ.eigenvalue}. Then $\{ \lambda^{-l}\mathcal{R}_{\bm{\rweight}_{p}}^{l}(E_{\ast})(u) \}_{l \in \mathbb{N} \cup \{ 0 \}}$ is non-decreasing for any $u \in \mathbb{R}^{V_{0}}$ and $\mathcal{R}_{\bm{\rweight}_{p}}(\mathcal{E}_{p}^{(0)}) = \lambda\mathcal{E}_{p}^{(0)}$, where $\mathcal{E}_{p}^{(0)} \in \mathcal{Q}_{p}(V_{0})$ is given by
    	\begin{equation}\label{e:eigenform.concrete}
    		\mathcal{E}_{p}^{(0)}(u) \coloneqq \lim_{l \to \infty}\lambda^{-l}\mathcal{R}_{\bm{\rweight}_{p}}^{l}(E_{\ast})(u), \quad u \in \mathbb{R}^{V_{0}}. 
   		\end{equation}
   		\item\label{it:CGQ.converse} Assume that there exist $\lambda \in (0,\infty)$ and a $p$-resistance form $E$ on $V_{0}$ such that $\mathcal{R}_{\bm{\rweight}_{p}}(E) = \lambda E$. Then \eqref{condA} holds.
    \end{enumerate}
\end{thm}

\begin{rmk}\label{rmk.condA}
	\begin{enumerate}[label=\textup{(\arabic*)},align=left,leftmargin=*,topsep=2pt,parsep=0pt,itemsep=2pt]
		\item\label{it:A-weight} If $\bm{\rweight}_{p}$ satisfies \eqref{condA} for some $p$-resistance form $E$ on $V_{0}$, then by \cite[Lemma 4.4-(a)]{CGQ22}, for any $p$-resistance form $\widetilde{E}$ on $V_{0}$ there exists $\tilde{c} \in (0,\infty)$ such that \eqref{condA} with $\widetilde{E},\tilde{c}$ in place of $E,c$ holds. Hence \eqref{condA} is a condition relying only on $\bm{\rweight}_{p}$. 
		\item The assertion $\mathcal{E}_{p}^{(0)} \in \mathcal{Q}_{p}(V_{0})$ follows from \eqref{e:eigenform.concrete}. Indeed, for any $n,l \in \mathbb{N} \cup \{ 0 \}$, by \eqref{V0bdry} and Proposition \ref{prop.trace-comp}, one can see that 
			\begin{align*}
				\mathcal{R}_{\bm{\rweight}_{p}}^{l}(E_{n}) 
				= \Biggl(\frac{1}{n+1}\sum_{j = 0}^{n}\lambda^{-j}\Lambda_{\bm{\rweight}_{p}}^{l + j}(E_{0})\Biggr)\Bigg|_{V_{0}} 
				\in \mathcal{Q}_{p}'(V_{0}). 
			\end{align*}
			Let $\varepsilon > 0$. 
			Then for all large enough $k \in \mathbb{N}$, we have 
			\[
			\abs{E'(u) - E_{n_{k}}(u)} \le \varepsilon \quad \text{whenever $u \in \mathbb{R}^{V_{0}}$ satisfies $\osc_{V_{0}}[u] \le 1$.}
			\]
			For such $k \in \mathbb{N}$ and $u \in \mathbb{R}^{V_{0}}$, since \eqref{mp} implies 
			\[
			\osc_{V_{l}}\Bigl[h_{V_{0}}^{\Lambda_{\bm{\rweight}_{p}}^{l}(E')}[u]\Bigr] \le \osc_{V_{0}}[u], 
			\quad \osc_{V_{l}}\Bigl[h_{V_{0}}^{\Lambda_{\bm{\rweight}_{p}}^{l}(E_{n_{k}})}[u]\Bigr] \le \osc_{V_{0}}[u], 
			\]
			we have 
			\[
			\abs{\mathcal{R}_{\bm{\rweight}_{p}}^{l}(E')(u) - \mathcal{R}_{\bm{\rweight}_{p}}^{l}(E_{n_{k}})(u)} \le \varepsilon\sum_{w \in W_{l}}\rweight_{p,w} \quad \text{whenever $u \in \mathbb{R}^{V_{0}}$ satisfies $\osc_{V_{0}}[u] \le 1$.}
			\]
			This shows that $d_{\widetilde{\mathcal{M}}_{p}(V_{0})}(\mathcal{R}_{\bm{\rweight}_{p}}^{l}(E'),\mathcal{R}_{\bm{\rweight}_{p}}^{l}(E_{n_k})) \to 0$ as $k \to \infty$, whence it follows that $\mathcal{R}_{\bm{\rweight}_{p}}^{l}(E') \in \mathcal{Q}_{p}(V_{0})$. Therefore, $\mathcal{E}_{p}^{(0)} \in \mathcal{Q}_{p}(V_{0})$. 
	\end{enumerate}
\end{rmk} 

In the rest of this subsection, we fix $\bm{\rweight}_{p} = (\rweight_{p,i})_{i \in S} \in (0,\infty)^{S}$. 
Let us introduce two important conditions on $\bm{\rweight}_{p}$, following \cite[Section 5]{CGQ22}: 
\begin{enumerate}[,align=left,leftmargin=*,topsep=2pt,parsep=0pt,itemsep=2pt]
	\item[(\textbf{A}')]\label{it:condA'} There exists a $p$-resistance form $\mathcal{E}_{p}^{(0)}$ on $V_{0}$ such that $\mathcal{R}_{\bm{\rweight}_{p}}(\mathcal{E}_{p}^{(0)}) = \mathcal{E}_{p}^{(0)}$. 
	\item[(\textbf{R})]\label{it:condR} (\textup{(\hyperref[it:condA']{\textbf{A}'})} holds and) $\min_{i \in S}\rweight_{p,i} > 1$. 
\end{enumerate}
Note that by Theorem \ref{thm.eigenform}, \textup{(\hyperref[it:condA']{\textbf{A}'})} implies \eqref{condA}, and \eqref{condA} implies that $\lambda^{-1} \bm{\rweight}_{p}$ satisfies (\hyperref[it:condA']{\textbf{A}'}) for some $\lambda \in (0,\infty)$.  

The following proposition is important to construct a self-similar $p$-resistance form as an ``inductive limit'' of discrete $p$-resistance forms as presented in \cite[Proposition 5.3]{CGQ22}, which is an adaptation of the relevant pieces of the theory of resistance forms due to \cite[Sections 2.2, 2.3 and 3.3]{Kig01}.
\begin{prop}\label{prop.inductive}
	Assume that \textup{(\hyperref[it:condA']{\textbf{A}'})} holds. 
    Define $\mathcal{E}_{p}^{(n)} \coloneqq \Lambda_{\bm{\rweight}_{p}}^{n}(\mathcal{E}_{p}^{(0)})$, i.e.,
    \begin{equation}\label{d:compatible}
        \mathcal{E}_{p}^{(n)}(u) \coloneqq \sum_{w \in W_{n}}\rweight_{p,w}\mathcal{E}_{p}^{(0)}(u \circ F_{w}), \quad u \in \mathbb{R}^{V_{n}}.
    \end{equation}
    Then $\mathcal{E}_{p}^{(n)}$ is a $p$-resistance form on $V_{n}$ and $\bigl.\mathcal{E}_{p}^{(n + m)}\bigr|_{V_{n}} = \mathcal{E}_{p}^{(n)}$ for any $n,m \in \mathbb{N} \cup \{ 0 \}$, i.e., $\bigl\{ (V_{n},\mathcal{E}_{p}^{(n)}) \bigr\}_{n \ge 0}$ is a compatible sequence of $p$-resistance forms. 
\end{prop}
\begin{proof}
    We will show $\bigl.\mathcal{E}_{p}^{(n + m)}\bigr|_{V_{n}} = \mathcal{E}_{p}^{(n)}$.
    (See \cite[Proposition 3.1.3]{Kig01} for the case of $p = 2$.)
    It suffices to prove $\bigl.\mathcal{E}_{p}^{(n + 1)}\bigr|_{V_{n}} = \mathcal{E}_{p}^{(n)}$ for any $n \in \mathbb{N} \cup \{ 0 \}$ by virtue of Proposition \ref{prop.trace-comp}.
    Note that the case of $n = 0$ is true by $\mathcal{R}_{\bm{\rweight}_{p}}(\mathcal{E}_{p}^{(0)}) = \mathcal{E}_{p}^{(0)}$, and that
    \begin{equation}\label{CGQ.press}
        \mathcal{E}_{p}^{(n + 1)}(u) = \sum_{i \in S}\rweight_{p,i}\,\mathcal{E}_{p}^{(n)}(u \circ F_{i}), \quad \text{for any $n \in \mathbb{N} \cup \{ 0 \}$ and $u \in \mathbb{R}^{V_{n + 1}}$.}
    \end{equation}
    Assume that $\bigl.\mathcal{E}_{p}^{(m)}\bigr|_{V_{m - 1}} = \mathcal{E}_{p}^{(m - 1)}$ for some $m \in \mathbb{N}$.
    Then for any $u \in \mathbb{R}^{V_{m}}$,
    \begin{align*}
        \mathcal{E}_{p}^{(m)}(u)
        &\overset{\eqref{CGQ.press}}{=} \sum_{i \in S}\rweight_{p,i}\,\mathcal{E}_{p}^{(m - 1)}(u \circ F_{i}) \\
        &= \sum_{i \in S}\rweight_{p,i}\min\Bigl\{ \mathcal{E}_{p}^{(m)}(v \circ F_{i}) \Bigm| v \in \mathbb{R}^{K_{i} \cap V_{m + 1}}, v|_{K_{i} \cap V_{m}} = u|_{K_{i}} \Bigr\} \\
        &\overset{\eqref{V0bdry}}{=} \min\Biggl\{ \sum_{i \in S}\rweight_{p,i}\,\mathcal{E}_{p}^{(m)}(v \circ F_{i}) \Biggm| v \in \mathbb{R}^{V_{m + 1}}, v|_{V_{m}} = u \Biggr\} \\
        &\overset{\eqref{CGQ.press}}{=} \min\Bigl\{ \mathcal{E}_{p}^{(m + 1)}(v) \Bigm| v \in \mathbb{R}^{V_{m + 1}}, v|_{V_{m}} = u \Bigr\}
        = \bigl.\mathcal{E}_{p}^{(m + 1)}\bigr|_{V_{m}}(u),
    \end{align*}
    which completes the proof.
\end{proof}

We can naturally construct a $p$-resistance form as an inductive limit on the countable set $V_{\ast}$ as described in the following proposition. 
\begin{prop}\label{prop.RFVast}
	Assume that \textup{(\hyperref[it:condA']{\textbf{A}'})} holds and let $\bigl\{ (V_{n},\mathcal{E}_{p}^{(n)}) \bigr\}_{n \ge 0}$ be the compatible sequence of $p$-resistance forms given in Proposition \ref{prop.inductive}.  
    Define a linear subspace $\mathcal{F}_{p,\ast}$ of $\mathbb{R}^{V_{\ast}}$ and $\mathcal{E}_{p,\ast} \colon \mathcal{F}_{p,\ast} \to [0,\infty)$ by
    \begin{align}
        \mathcal{F}_{p,\ast} &\coloneqq \Bigl\{ u \in \mathbb{R}^{V_{\ast}} \Bigm| \lim_{n \to \infty}\mathcal{E}_{p}^{(n)}(u|_{V_{n}}) < \infty \Bigr\}, \quad \text{and} \label{defn:Fpast}\\
            \mathcal{E}_{p,\ast}(u) &\coloneqq \lim_{n \to \infty}\mathcal{E}_{p}^{(n)}(u|_{V_{n}}), \quad u \in \mathcal{F}_{p,\ast}. \label{defn:Epast}
    \end{align}
    Then $(\mathcal{E}_{p,\ast},\mathcal{F}_{p,\ast})$ is a $p$-resistance form on $V_{\ast}$ satisfying $\mathcal{E}_{p,\ast}|_{V_{n}} = \mathcal{E}_{p}^{(n)}$ for any $n \in \mathbb{N} \cup \{ 0 \}$.
    Moreover, the following self-similarity properties hold:
    \begin{gather}
        \mathcal{F}_{p,\ast} = \bigl\{ u \in \mathbb{R}^{V_{\ast}} \bigm| \text{$u \circ F_{i} \in \mathcal{F}_{p,\ast}$ for any $i \in S$} \bigr\}, \label{SSE1.Vast} \\
            \mathcal{E}_{p,\ast}(u) = \sum_{i \in S}\rweight_{p,i}\,\mathcal{E}_{p,\ast}(u \circ F_{i}) \quad \text{for any $u \in \mathcal{F}_{p,\ast}$.} \label{SSE2.Vast}
    \end{gather}
    If in addition \textup{(\hyperref[it:condR]{\textbf{R}})} holds, then for any $u \in \mathcal{F}_{p,\ast}$ there exists a unique $\widehat{u} \in \contfunc(K)$ such that $\widehat{u}|_{V_{\ast}} = u$, and $\{ \widehat{u} \mid u \in \mathcal{F}_{p,\ast} \}$ is dense in $\contfunc(K)$. 
\end{prop}
\begin{proof}
	It is immediate from Theorem \ref{thm.Epcountable} that $(\mathcal{E}_{p,\ast},\mathcal{F}_{p,\ast})$ is a $p$-resistance form on $V_{\ast}$ with $\mathcal{E}_{p,\ast}|_{V_{n}} = \mathcal{E}_{p}^{(n)}$. 
	By the definition in \eqref{d:compatible}, it is easy to see that for any $n,k \in \mathbb{N} \cup \{ 0 \}$ and any $u \in \mathbb{R}^{V_{\ast}}$, 
	\[
	\mathcal{E}_{p}^{(n + k)}(u|_{V_{n + k}}) = \sum_{w \in W_{k}}\rweight_{p,w}\,\mathcal{E}_{p}^{(n)}(u \circ F_{w}|_{V_{n}}). 
	\]
	This immediately implies \eqref{SSE1.Vast} and \eqref{SSE2.Vast}. 
	The existence of unique continuous extensions of functions in $\mathcal{F}_{p,\ast}$ under \textup{(\hyperref[it:condR]{\textbf{R}})} is proved in \cite[Theorem 5.1-(b)]{CGQ22}. 
	A standard argument using the Stone--Weierstrass theorem shows that $\mathscr{C} \coloneqq \{ \widehat{u} \mid u \in \mathcal{F}_{p,\ast} \}$ is dense in $\contfunc(K)$.
	Indeed, $\mathscr{C}$ is an algebra since $\mathcal{F}_{p,\ast}$ is also an algebra by Proposition \ref{prop.GC-list}-\ref{GC.leibniz}. 
	For any $x,y \in K$ with $x \neq y$, choose $n \in \mathbb{N}$ and $v,w \in W_{n}$ so that $x \in K_{v}, y \in K_{w}$ and $K_{v} \cap K_{w} = \emptyset$. (Such $n,v,w$ exist by \eqref{ss-diam}.) 
	Then, by setting $v \coloneqq h_{V_{n}}^{\mathcal{E}_{p,\ast}}[\indicator{F_{v}(V_{0})}]$, we see that $\varphi_{xy} \coloneqq \widehat{v} \in \mathscr{C}$ satisfies $\varphi_{xy}(x) = 1$ and $\varphi_{xy}(y) = 0$, so we can use the Stone--Weierstrass theorem to conclude that $\mathscr{C}$ is dense in $\contfunc(K)$. 
\end{proof}

To extend $(\mathcal{E}_{p,\ast},\mathcal{F}_{p,\ast})$ to a $p$-energy form defined on $K$, we need to specify how to regard functions in $\mathcal{F}_{p,\ast}$ as functions defined on $K$, which is indeed a delicate problem and discussed in \cite[Theorems 5.1 and 5.2]{CGQ22}. 
In this paper, we are mainly interested in the case where $\mathcal{F}_{p,\ast}$ can be embedded into $\contfunc(K)$.
In other words, we always assume that \textup{(\hyperref[it:condR]{\textbf{R}})} holds.
(See \cite[Theorem 5.2]{CGQ22} and \cite[Appendix]{KS.scp} for details on a situation when we can identify a function $u \in \mathbb{R}^{V_{\ast}}$ satisfying $\lim_{n \to \infty}\mathcal{E}_{p}^{(n)}(u|_{V_{n}}) < \infty$ with a function on $K$ without \textup{(\hyperref[it:condR]{\textbf{R}})}.)   
To state a construction of self-similar $p$-resistance forms under \textup{(\hyperref[it:condR]{\textbf{R}})}, we need the following lemma.  
\begin{lem}\label{lem.identify}
	Assume that \textup{(\hyperref[it:condA']{\textbf{A}'})} and \textup{(\hyperref[it:condR]{\textbf{R}})} hold. 
	Let $(\mathcal{E}_{p,\ast},\mathcal{F}_{p,\ast})$ be the $p$-resistance form on $V_{\ast}$ given in Proposition \ref{prop.RFVast}. 
	Then $\id_{V_{\ast}}$ is uniquely extended to a continuous map $\overline{\id_{V_{\ast}}}$ from the completion of $(V_{\ast},R_{\mathcal{E}_{p,\ast}}^{1/p})$ to $K$, and $\overline{\id_{V_{\ast}}}$ is a homeomorphism. 
\end{lem}
\begin{proof}
	The proof is very similar to the arguments in \cite[Proposition 3.3.2, Lemma 3.3.5 and Theorem 3.3.4]{Kig01}, where similar statements in the case of $p = 2$ are shown. 
	Let $(\widehat{K},\widehat{d}\,)$ be the completion of $(V_{\ast},R_{p,\mathcal{E}_{p,\ast}}^{1/p})$ and let $(\overline{\mathcal{E}}_{p,\ast},\overline{\mathcal{F}}_{p,\ast})$ be the $p$-resistance form on $\widehat{K}$ defined by \eqref{e:defn.compatext.dom} and \eqref{e:defn.compatext.form} with $\mathcal{S} \coloneqq \bigl\{ (V_{n},\mathcal{E}_{p}^{(n)}) \bigr\}_{n \in \mathbb{N} \cup \{ 0 \}}$.
	Also, we fix a metric $d$ on $K$ which gives the original topology of $K$.
	Recall that $R_{\widehat{\mathcal{E}}_{p,\ast}}^{1/p} = \widehat{d}$ by Corollary \ref{cor.Epext}. 
	For $n \in \mathbb{N}$, we define 
	\[
	\delta_{n} \coloneqq \min_{v,w \in W_{n}; K_{v} \cap K_{w} = \emptyset}\biggl(\inf_{x \in F_{v}(V_{\ast}), y \in F_{w}(V_{\ast})}R_{\mathcal{E}_{p,\ast}}(x,y)\biggr). 
	\]
	Then $\delta_{n} > 0$ since $R_{\mathcal{E}_{p,\ast}}(x,y) \ge \mathcal{E}_{p,\ast}\bigl(h_{V_{n}}^{\mathcal{E}_{p,\ast}}[\indicator{F_{w}(V_{0})}]\bigr)^{-1}$ for any $(x,y) \in F_{v}(V_{\ast}) \times F_{w}(V_{\ast})$. 
	Let $\{ x_{n} \}_{n \ge 0}$ be a Cauchy sequence in $(V_{\ast},R_{\mathcal{E}_{p,\ast}}^{1/p})$.  
	For each $n \in \mathbb{N}$, choose $N(n) \in \mathbb{N}$ so that 
	\[
	\sup_{k,l \ge N(n)}R_{\mathcal{E}_{p,\ast}}(x_{k},x_{l}) < \delta_{n}. 
	\]
	Then there exists $w \in W_{n}$ such that $\{ x_{k} \}_{k \ge N(n)} \subseteq \bigcup_{v \in W_{n}; K_{v} \cap K_{w} \neq \emptyset}F_{v}(V_{\ast}) \eqqcolon A_{n,w}$. 
	Since $\lim_{n \to \infty}\max_{w \in W_{n}}\diam(A_{n,w},d) = 0$ by \eqref{ss-diam}, we conclude that $\id_{V_{\ast}} \colon (V_{\ast},R_{\mathcal{E}_{p,\ast}}^{1/p}) \to (V_{\ast},d|_{V_{\ast} \times V_{\ast}})$ is uniformly continuous.
	Now we define $\theta \colon (\widehat{K},\widehat{d}) \to (K,d)$ as the unique continuous map satisfying $\theta|_{V_{\ast}} = \id_{V_{\ast}}$. 
	Let us show that $\theta$ is injective. 
	Assume that $x,y \in \widehat{K}$ satisfy $\theta(x) = \theta(y)$. 
	Let $\{ x_{n} \}_{n \ge 0}, \{ y_{n} \}_{n \ge 0}$ be Cauchy sequences in $(V_{\ast},R_{\mathcal{E}_{p,\ast}}^{1/p})$ satisfying $\lim_{n \to \infty}\widehat{d}(x,x_{n}) = \lim_{n \to \infty}\widehat{d}(y,y_{n}) = 0$. 
	Then $\lim_{n \to \infty}d(\theta(x),x_{n}) = \lim_{n \to \infty}d(\theta(y),y_{n}) = 0$ since $\theta$ is continuous. 
	For any $u \in \overline{\mathcal{F}}_{p,\ast}$, let $\widehat{u}_{n} \in \contfunc(K)$ be the unique function satisfying $\widehat{u}_{n}|_{V_{\ast}} = h_{V_{n}}^{\mathcal{E}_{p,\ast}}[u|_{V_{n}}]$, which exists by Proposition \ref{prop.RFVast}. 
	Also, let $v_{n} \in \contfunc(\widehat{K})$ be the unique function satisfying $v_{n}|_{V_{\ast}} = h_{V_{n}}^{\mathcal{E}_{p,\ast}}[u|_{V_{n}}]$; recall the proof of Theorem \ref{thm.Rpcompletion}.  
	Then we see that 
	\begin{equation}\label{e:preidentity}
		v_{n}(x) = \lim_{k \to \infty}h_{V_{n}}^{\mathcal{E}_{p,\ast}}[u](x_{k}) = \widehat{u}_{n}(\theta(x)) 
		= \widehat{u}_{n}(\theta(y)) = \lim_{k \to \infty}h_{V_{n}}^{\mathcal{E}_{p,\ast}}[u](y_{k}) = v_{n}(y). 
	\end{equation}
	Let us fix $o \in V_{0} \subseteq V_{n}$. 
	By the estimate \eqref{R-basic} in Proposition \ref{prop.R-conseq}-\ref{it:RF.basic-ineq} for $(\overline{\mathcal{E}}_{p,\ast},\overline{\mathcal{F}}_{p,\ast})$, 
	\[
	\abs{u(x) - v_{n}(x)}^{p} \le R_{\overline{\mathcal{E}}_{p,\ast}}(x,o)\overline{\mathcal{E}}_{p,\ast}(u - v_{n}) = R_{\overline{\mathcal{E}}_{p,\ast}}(x,o)\mathcal{E}_{p,\ast}\bigl(u|_{V_{\ast}} - h_{V_{n}}^{\mathcal{E}_{p,\ast}}[u|_{V_{n}}]\bigr), 
	\]
	which together with \eqref{harmapprox} in Proposition \ref{prop.harmapprox} and \eqref{e:preidentity} implies that
	\[
	u(x) = \lim_{n \to \infty}v_{n}(x) \overset{\eqref{e:preidentity}}{=} \lim_{n \to \infty}v_{n}(y) = u(y). 
	\]
	Since $u \in \overline{\mathcal{F}}_{p,\ast}$ is arbitrary, we conclude by \ref{RF3} for $(\overline{\mathcal{E}}_{p,\ast},\overline{\mathcal{F}}_{p,\ast})$ that $R_{\overline{\mathcal{E}}_{p,\ast}}(x,y) = 0$ and hence $x = y$. 
	This means that $\theta$ is injective. 
	
	Next we see that $\{ F_{i} \}_{i \in S}$ yields a family of contractions on the complete (non-empty) metric space $(\widehat{K},\widehat{d}\,)$. 
	By virtue of \eqref{SSE2.Vast}, in the same way as the proof of \eqref{pRMss} in Proposition \ref{prop.pRMss}, one can show that for any $w \in W_{\ast}$ and any $x,y \in V_{\ast}$, 
	\[
	\widehat{d}(F_{w}(x),F_{w}(y))^{p} = R_{\mathcal{E}_{p,\ast}}(F_{w}(x),F_{w}(y)) \le \rweight_{p,w}^{-1}R_{\mathcal{E}_{p,\ast}}(x,y) = \rweight_{p,w}^{-1}\widehat{d}(x,y)^{p}. 
	\]
	In particular, $F_{w}|_{V_{\ast}} \colon (V_{\ast}, \widehat{d}\,) \to (V_{\ast},\widehat{d}\,)$ is uniformly continuous, and hence there exists a unique continuous map $F_{w}^{\widehat{K}} \colon \widehat{K} \to \widehat{K}$ such that $F_{w}^{\widehat{K}}\big|_{V_{\ast}} = F_{w}|_{V_{\ast}}$. 
	Then it is clear that 
	\begin{equation}\label{e:completion.contraction}
		\widehat{d}\,\bigl(F_{w}^{\widehat{K}}(x),F_{w}^{\widehat{K}}(y)\bigr) \le \rweight_{p,w}^{-1/p}\,\widehat{d}(x,y) \quad \text{for any $x,y \in \widehat{K}$,}
	\end{equation}
	and that $\theta \circ F_{w}^{\widehat{K}} = F_{w} \circ \theta$. 
	Now, by \textup{(\hyperref[it:condR]{\textbf{R}})} and \eqref{e:completion.contraction}, $\bigl\{ F_{i}^{\widehat{K}} \bigr\}_{i \in S}$ is a family of contractions on $(\widehat{K},\widehat{d}\,)$, so that by the standard result on the unique existence of self-similar sets (see, e.g., \cite[Theorem 1.1.4]{Kig01}) there exists a unique non-empty compact subset $\widehat{K}_{0}$ of $\widehat{K}$ such that $\widehat{K}_{0} = \bigcup_{i \in S}F_{i}^{\widehat{K}}(\widehat{K}_{0})$. 
	Let us fix $o \in \widehat{K}_{0}$ and set $A \coloneqq \bigcup_{w \in W_{\ast}}F_{w}^{\widehat{K}}(o) \subseteq \widehat{K}_{0}$. 
	Then $\theta(A) = \bigcup_{w \in W_{\ast}}F_{w}(\theta(o))$ is dense in $(K,d)$ by $\lim_{n \to \infty}\max_{w \in W_{n}}\diam(K_{w},d) = 0$ from \eqref{ss-diam}.
	Since $\theta(A) \subseteq \theta(\widehat{K}_{0}) \subseteq K$ and $\theta(\widehat{K}_{0})$ is compact by the continuity of $\theta$, we have $\theta(\widehat{K}_{0}) = K$ and thus $\theta(\widehat{K}) = K$. 
	Then $\widehat{K}$ turns out to be compact since $\widehat{K} = \widehat{K}_{0}$ by the injectivity of $\theta$. 
	Now $\theta$ turns out to be a homeomorphism between $\widehat{K}$ and $K$. 
\end{proof}

The following theorem describes a construction of a self-similar $p$-resistance form as the inductive limit of $\{ \mathcal{E}_{p}^{(n)} \}_{n \ge 0}$ under the condition \textup{(\hyperref[it:condR]{\textbf{R}})}. 
\begin{thm}\label{thm.CGQ}
	Assume that \textup{(\hyperref[it:condA']{\textbf{A}'})} and \textup{(\hyperref[it:condR]{\textbf{R}})} hold. 
	Let $\bigl\{ (V_{n},\mathcal{E}_{p}^{(n)}) \bigr\}_{n \ge 0}$ be the compatible sequence of $p$-resistance forms given in Proposition \ref{prop.inductive}, and define 
    \begin{gather}
        \mathcal{F}_{p} \coloneqq \Bigl\{ u \in \contfunc(K) \Bigm| \lim_{n \to \infty}\mathcal{E}_{p}^{(n)}(u|_{V_{n}}) < \infty \Bigr\}, \label{pRFinductive.dom} \\
        \mathcal{E}_{p}(u) \coloneqq \lim_{n \to \infty}\mathcal{E}_{p}^{(n)}(u), \quad u \in \mathcal{F}_{p}.  \label{pRFinductive.form}
    \end{gather}
    Then $(\mathcal{E}_{p},\mathcal{F}_{p})$ is a regular self-similar $p$-resistance form on $\mathcal{L}$ with weight $\bm{\rweight}_{p}$, $\mathcal{E}_{p}|_{V_{n}} = \mathcal{E}_{p}^{(n)}$ for any $n \in \mathbb{N} \cup \{ 0 \}$, and $R_{\mathcal{E}_{p}}^{1/p}$ is compatible with the original topology of $K$.   
\end{thm}
\begin{rmk}
    Similar to Proposition \ref{prop.ssenergy-GCinv}, by choosing a suitable $E_{0} \in \mathcal{S}_{p}(V_{0})$ in Theorem \ref{thm.eigenform}, we can verify nice properties like the \emph{symmetry-invariance} of $E_{\ast}$ in \eqref{e:CGQlimit}, $\mathcal{E}_{p}^{(0)}$ in \eqref{e:eigenform.concrete} and $(\mathcal{E}_{p},\mathcal{F}_{p})$ in Theorem \ref{thm.CGQ}; see Theorem \ref{thm.ANFsymform} for details, and see also \eqref{SGE.sym}.
\end{rmk}
\begin{proof}[Proof of Theorem \ref{thm.CGQ}]
	By Lemma \ref{lem.identify} and Corollary \ref{cor.Epext}, $(\mathcal{E}_{p},\mathcal{F}_{p})$ is a $p$-resistance form on $K$. 
	The self-similarity conditions, \eqref{SSE1} and \eqref{SSE2}, for $(\mathcal{E}_{p},\mathcal{F}_{p})$ are obvious from Proposition \ref{prop.inductive}. 
	By Lemma \ref{lem.identify} and Proposition \ref{prop.RFVast}, $R_{\mathcal{E}_{p}}^{1/p}$ is compatible with the original topology of $K$ and $(\mathcal{E}_{p},\mathcal{F}_{p})$ is regular (recall Definition \ref{dfn:regRF}). 
\end{proof}

Let us recall the following proposition from \cite{CGQ22}. This result is useful to verify \textup{(\hyperref[it:condR]{\textbf{R}})} for concrete examples. 
\begin{prop}[{\cite[Lemma 5.4]{CGQ22}}]\label{prop.goodweight}
	Assume that \textup{(\hyperref[it:condA']{\textbf{A}'})} holds. 
	If $w \in W_{\ast} \setminus \{ \emptyset \}$ satisfies $w^{\infty} \coloneqq www\ldots \in \mathcal{P}_{\mathcal{L}}$, then $\rweight_{p,w} > 1$.  
\end{prop} 

\subsection{Existence of eigenforms on strongly symmetric p.-c.f.\ self-similar sets}\label{sec.ANF}
Let us conclude this section by verifying the condition \eqref{condA} for the existence of an eigenform from Theorem \ref{thm.eigenform}, for a special class of p.-c.f.\ self-similar sets known as \emph{affine nested fractals}\index{affine nested fractal} in the literature and introduced in \cite{FHK94} as a generalization of the class of nested fractals introduced by Lindstr\o m \cite{Lin90}. 
More precisely, we will work on a wider class called \emph{strongly symmetric p.-c.f.\ self-similar sets}\index{strongly symmetric p.-c.f.\ self-similar set}.
A proof of \eqref{condA} for affine nested fractals is given in \cite[Theorem 6.3]{CGQ22}, but the description of the group of symmetries in the paper \cite{CGQ22} is not sophisticated\footnote{For a group of symmetries, say $\mathcal{G}$, one of the essential properties that is needed to prove the $\mathcal{G}$-invariance of the resulting self-similar $p$-energy form is Proposition \ref{prop.ANF-geom}-\ref{it:ANF.symcompos}. We have to be careful of whether this property holds for $\mathcal{G}$, but this point is not taken care of in \cite{CGQ22}.}, so in Theorem \ref{thm.eigenform-ANF} below we provide the details of the proof of \eqref{condA} and simultaneously extend the result of \cite[Theorem 6.3]{CGQ22} to the wider class of strongly symmetric p.-c.f.\ self-similar sets.   

We start with recalling the definitions of a group of symmetries, affine nested fractals and strongly symmetric p.-c.f.\ self-similar sets; see \cite[Section 3.8]{Kig01} for details. 
Throughout this subsection, we fix the setting of Framework \ref{frmwrk:ANF} below. 

\begin{framework}\label{frmwrk:ANF}
Let $D\in\mathbb{N}$ and let $S$ be a non-empty finite set with $\#S \ge 2$.
Let $\{c_{i}\}_{i \in S} \subseteq (0,1)$, $\{ a_{i} \}_{i \in S} \subseteq \mathbb{R}^{D}$ and $\{ U_{i} \}_{i \in S} \subseteq O(D)$, where $O(D)$ is the collection of orthogonal transformations of $\mathbb{R}^{D}$. 
Define $f_{i}\colon\mathbb{R}^{D}\to\mathbb{R}^{D}$ by $f_{i}(x) \coloneqq c_{i}U_{i}x + a_{i}$ for each $i\in S$. 
Let $K$ be the self-similar set associated with $\{ f_{i} \}_{i \in S}$, set $F_{i} \coloneqq f_{i}|_{K}$ for each $i \in S$ and assume that $K$ is connected and that $\mathcal{L} = (K,S,\{ F_{i} \}_{i \in S})$ is a p.-c.f.\ self-similar structure. 
We also assume that $\sum_{q \in V_{0}}q = \bm{0}$. 
Let $d\colon K\times K\to[0,\infty)$ be the Euclidean metric on $K$ given by $d(x,y)\coloneqq |x-y|$.  
\end{framework}

\begin{defn}[{\cite[Definitions 3.8.3 and 3.8.4]{Kig01}}]\label{defn.ANF}
    \begin{enumerate}[label=\textup{(\arabic*)},align=left,leftmargin=*,topsep=2pt,parsep=0pt,itemsep=2pt]
        \item\label{it:symgroup} We define 
        	\[
        	\mathcal{G}_{\mathrm{sym}} 
        	\coloneqq \mathcal{G}_{\mathrm{sym}}(\mathcal{L}) 
        	\coloneqq \Biggl\{ g|_{K} \Biggm| 
        	\begin{minipage}{210pt}
        		$g \in O(D)$, for any $n \in \mathbb{N} \cup \{ 0 \}$ and any $w \in W_{n}$ there exists $w' \in W_{n}$ such that $g(K_{w}) = K_{w'}$ and $g(F_{w}(V_{0})) = F_{w'}(V_{0})$
        	\end{minipage}
        	\Biggr\}. 
        	\] 
        	Note that for any $g \in \mathcal{G}_{\mathrm{sym}}$ and any $w \in W_{\ast}$, a word $w' \in W_{\ast}$ satisfying $\abs{w} = \abs{w'}$ and $g(K_{w}) = K_{w'}$ is uniquely determined thanks to \eqref{V0bdry} and $\#V_{0} < \infty = \#K$, so that $g(F_{w}(V_0)) = F_{w'}(V_0)$ by $g \in \mathcal{G}_{\mathrm{sym}}$.
			In particular, the map $\tau_{g} \colon W_{\ast} \to W_{\ast}$ defined by $\tau_{g}(w) \coloneqq w'$ gives a bijection such that $\abs{\tau_{g}(w)} = \abs{w}$ for any $w \in W_{\ast}$. 
        \item\label{it:gxy} For $x,y \in \mathbb{R}^{D}$ with $x \neq y$, let $g_{xy} \colon \mathbb{R}^{D} \to \mathbb{R}^{D}$ be the reflection in the hyperplane $H_{xy} \coloneqq \bigl\{ z \in \mathbb{R}^{D} \bigm| \abs{x - z} = \abs{y - z} \bigr\}$. 
        \item\label{it:distances-V0-ssset} Let $m_{\ast} \coloneqq \#\{ \abs{x-y} \mid \textrm{$x,y \in V_{0}$, $x \neq y$} \}$ and $l_{0} \coloneqq \min\{ \abs{x-y} \mid \textrm{$x,y \in V_{0}$, $x \neq y$} \}$. We define $\{ l_{i} \}_{i = 0}^{m_{\ast} - 1}$ inductively by $l_{i + 1} \coloneqq \min\{ \abs{x-y} \mid \textrm{$x,y \in V_{0}$, $\abs{x - y} > l_{i}$} \}$. 
		\item Let $m,n \in \mathbb{N} \cup \{ 0 \}$ and $(x_{i})_{i = 0}^{n} \in (V_{m})^{n+1}$. Then $(x_{i})_{i = 0}^{n}$ is called an \emph{$m$-walk}\index{$m$-walk} (between $x_{0}$ and $x_{n}$) if and only if there exist $w^{1},\dots,w^{n} \in W_{m}$ such that $\{ x_{i-1},x_{i} \} \subseteq F_{w^i}(V_{0})$ for all $i \in \{ 1,2,\ldots,n \}$. A $0$-walk $(x_{i})_{i = 0}^{n}$ is called a \emph{strict $0$-walk}\index{strict $0$-walk} (between $x_{0}$ and $x_{n}$) if and only if $\abs{x_{i} - x_{i-1}} = l_{0}$ for any $i \in \{ 1,2,\ldots,n \}$. 
		\item\label{it:StrongSym} $\mathcal{L}$ is called a \emph{strongly symmetric p.-c.f.\ self-similar set}\index{strongly symmetric p.-c.f.\ self-similar set} if and only if it satisfies the following four conditions: 
			\begin{enumerate}[label=\textup{(SS\arabic*)},align=left,leftmargin=*,topsep=2pt,parsep=0pt,itemsep=2pt]
				\item\label{SS1} For any $x,y \in V_{0}$ with $x \neq y$, there exists a strict $0$-walk between $x$ and $y$. 
				\item\label{SS2} If $x,y,z \in V_{0}$ and $\abs{x-y} = \abs{x-z}$, then there exists $g \in \mathcal{G}_{\mathrm{sym}}$ such that $g(x) = x$ and $g(y) = z$. 
				\item\label{SS3} For any $i \in \{ 1,\dots,m_{\ast}-1 \}$, there exist $x,y,z \in V_{0}$ such that $l_{i} = \abs{x - y} > \abs{x - z} > 0$ and $g_{yz}|_{K} \in \mathcal{G}_{\mathrm{sym}}$. 
				\item\label{SS4} $V_{0}$ is $\mathcal{G}_{\mathrm{sym}}$-transitive, i.e., for any $x,y \in V_{0}$ with $x \neq y$, there exists $g \in \mathcal{G}_{\mathrm{sym}}$ such that $g(x) = y$. 
			\end{enumerate}
        \item\label{it:ANF} $\mathcal{L}$ is called an \emph{affine nested fractal}\index{affine nested fractal} if $g_{xy}|_{K} \in \mathcal{G}_{\mathrm{sym}}(\mathcal{L})$ for any $x,y \in V_{0}$ with $x \neq y$.
    \end{enumerate}
\end{defn}
\begin{rmk}\label{rmk:defn.ANF}
	Definition \ref{defn.ANF} involves a few differences from the original definitions in \cite{Kig01}.
	In \cite[Definitions 3.8.3 and 3.8.4]{Kig01}, the following group of symmetries $\mathcal{G}_{s}$ is used instead of $\mathcal{G}_{\mathrm{sym}}$: 
	\[
	\mathcal{G}_{s} \coloneqq \mathcal{G}_{s}(\mathcal{L}) \coloneqq \Biggl\{ g|_{K} \Biggm| 
	\begin{minipage}{211pt}
		$g \in O(D)$, for any $n \in \mathbb{N} \cup \{ 0 \}$ there exists a bijection $\tau_{g,n} \colon W_{n} \to W_{n}$ such that $g(F_{w}(V_{0})) = F_{\tau_{g,n}(w)}(V_{0})$ for any $w \in W_{n}$
	\end{minipage}
	\Biggr\};  
	\]
	note that $\mathcal{G}_{\mathrm{sym}} \subseteq \mathcal{G}_{s}$ by the remarks in Definition \ref{defn.ANF}-\ref{it:symgroup}. 
	Under the assumption that 
	\begin{equation}\label{e:ANF.joint}
		\#(F_{i}(V_{0}) \cap F_{j}(V_{0})) \le 1 \quad \text{for any $i,j \in S$ with $i \neq j$,}
	\end{equation}
	we know that $\mathcal{G}_{\mathrm{sym}} = \mathcal{G}_{s}$ by \cite[Proposition 3.8.19]{Kig01}. 
	The difference between $\mathcal{G}_{\mathrm{sym}}$ and $\mathcal{G}_{s}$ does not affect the arguments in the parts of \cite{Kig01,CGQ22} that we need (Proposition \ref{prop.ANF-geom} and Theorem \ref{thm.eigenform-ANF} below). 
	Also, our version of \ref{SS3} is weaker than \cite[(SS3) in Definition 3.8.4-(1)]{Kig01}. Consequently, any strongly symmetric p.-c.f.\ self-similar set in the sense of \cite[Definition 3.8.4]{Kig01} is also a strongly symmetric p.-c.f.\ self-similar set in our sense.
\end{rmk} 

Let us recall a few properties of $\mathcal{G}_{\mathrm{sym}}$ and of affine nested fractals in the following proposition, which can be shown in the same ways as in \cite[Section 3.8]{Kig01}.
Let us emphasize that \emph{we do NOT assume the condition \eqref{e:ANF.joint}} unlike in \cite[Section 3.8]{Kig01}. 
\begin{prop}[{\cite[Propositions 3.8.7, 3.8.20 and Lemma 3.8.23]{Kig01}}]\label{prop.ANF-geom}
	\begin{enumerate}[label=\textup{(\arabic*)},align=left,leftmargin=*,topsep=2pt,parsep=0pt,itemsep=2pt]
		\item\label{it:ANF.ss} If $\mathcal{L}$ is an affine nested fractal, then it is a strongly symmetric p.-c.f.\ self-similar set.
		\item\label{it:ANF.symcompos} Let $w \in W_{\ast}$, $g \in \mathcal{G}_{\mathrm{sym}}$ and set 
		\[
		U_{g,w} \coloneqq F_{w'}^{-1} \circ g \circ F_{w}, 
		\] 
		where $w' \in W_{\ast}$ is the unique word satisfying $g(K_{w}) = K_{w'}$. 
		Then $U_{g,w} \in \mathcal{G}_{\mathrm{sym}}$. 
		\item\label{it:ANF.reflection} Assume that $\mathcal{L}$ satisfies \ref{SS4} in Definition \ref{defn.ANF}-\ref{it:StrongSym}, and let $x,y \in V_{0}$ satisfy $x \not= y$ and $g_{xy}|_{K} \in \mathcal{G}_{\mathrm{sym}}$. If $w \in W_{\ast}$ satisfies $\abs{a - y} < \abs{a - x}$ and $\abs{b - y} > \abs{b - x}$ for some $a,b \in F_{w}(V_{0})$, then $g_{xy}(K_{w}) = K_{w}$ and $g_{xy}(F_{w}(V_0)) = F_{w}(V_{0})$.
	\end{enumerate}
\end{prop}

Now we can present the following theorem establishing the existence of an eigenform on $V_{0}$ for strongly symmetric p.-c.f.\ self-similar sets and thereby improving the preceding result \cite[Theorem 6.3]{CGQ22} treating only the case of affine nested fractals. 
Note that the case of $p = 2$ corresponds to the existence of a harmonic structure on $\mathcal{L}$ in \cite[Theorem 3.8.10]{Kig01}.   
\begin{thm}\label{thm.eigenform-ANF}
	Assume that $\mathcal{L}$ is strongly symmetric, and let $\bm{\rweight}_{p} = (\rweight_{p,i})_{i \in S} \in (0,\infty)^{S}$. 
   	If
    \begin{equation}\label{rscale.sym}
        \rweight_{p,i} = \rweight_{p,i'} \quad \text{for any $i \in S$ and any $g \in \mathcal{G}_{\mathrm{sym}}$,}
    \end{equation}
    where $i' \in S$ is the unique element satisfying $K_{i'} = g(K_{i})$, then \eqref{condA} holds. 
    In particular, if $\rweight_{p,i} = \rweight_{p}$ for any $i \in S$ for some $\rweight_{p} \in (0,\infty)$, then \textup{(\hyperref[it:condA']{\textbf{A}'})} and \textup{(\hyperref[it:condR]{\textbf{R}})} with $\bigl(\lambda(\bm{\rweight}_{p})^{-1}\rweight_{p}\bigr)_{i \in S}$ in place of $\bm{\rweight}_{p}$ hold, where $\lambda(\bm{\rweight}_{p})$ is the number given in Theorem \ref{thm.eigenform}-\ref{it:CGQ.eigenvalue}. 
\end{thm}
\begin{proof}
	The proof is essentially the same as \cite[Proof of Theorem 6.3]{CGQ22}, but we give its details since we weaken the assumption of \cite[Theorem 6.3]{CGQ22}. 
	For $n \in \mathbb{N} \cup \{ 0 \}$, define $E_{p,n} \in \mathcal{S}_{p}(V_{n})$ by 
	\[
	E_{p,n}(u) \coloneqq \sum_{w \in W_{n}}\rweight_{p,w}\sum_{x,y \in V_{0}; \abs{x-y} = l_{0}}\abs{u(F_{w}(x)) - u(F_{w}(y))}^{p}, \quad u \in \mathbb{R}^{V_{n}}. 
	\]
	Note that, by Proposition \ref{prop.ANF-geom}-\ref{it:ANF.symcompos} and \eqref{rscale.sym}, $E_{p,n}$ is $\mathcal{G}_{\mathrm{sym}}$-invariant, i.e., $E_{p,n}(u \circ g|_{V_{n}}) = E_{p,n}(u)$ for any $u \in \mathbb{R}^{V_{n}}$ and any $g \in \mathcal{G}_{\mathrm{sym}}$. 
	Moreover, the following property holds: 
	\begin{equation}\label{e:RM.invariance}
		R_{E_{p,n}}(x,y) = R_{E_{p,n}}(x',y') \quad \text{if $x,y,x',y' \in V_{0}$ and $\abs{x-y} = \abs{x'-y'}$.}
	\end{equation}
	Indeed, by \ref{SS4} there exists $g_{1} \in \mathcal{G}_{\mathrm{sym}}$ such that $g_{1}(x) = x'$. 
	Since $g_{1} \in O(D)$ implies $\abs{x' - g_1(y)} = \abs{x - y} = \abs{x' - y'}$, we have from \ref{SS2} that $g_{2}(x') = x'$ and $g_{2}(g_{1}(y)) = y'$ for some $g_{2} \in \mathcal{G}_{\mathrm{sym}}$. 
	Define $g \coloneqq g_{2} \circ g_{1} \in \mathcal{G}_{\mathrm{sym}}$, so that $g(x) = x'$ and $g(y) = y'$.  
	From the $\mathcal{G}_{\mathrm{sym}}$-invariance of $E_{p,n}$, we obtain \eqref{e:RM.invariance}.
	
	We fix $a_{1},a_{2} \in V_{0}$ that satisfy $\abs{a_1 - a_2} = l_{0}$ and claim that for any $n \in \mathbb{N}$ and any $x,y \in V_{0}$ with $x \neq y$, 
	\begin{equation}\label{e:condA.ANF}
		\frac{1}{2^{m_{\ast} - 1}}R_{E_{p,n}}(a_{1},a_{2}) \le R_{E_{p,n}}(x,y) \le (\#V_{0})^{p}R_{E_{p,n}}(a_{1},a_{2}), 
	\end{equation} 
	which implies \eqref{condA} for $\bm{\rweight}_{p}$ with $c = 2^{1-m_{\ast}}(\#V_{0})^{-p}$. 
	To show the upper estimate in \eqref{e:condA.ANF}, let $(x_{i})_{i=0}^{k} \in (V_{0})^{k+1}$ be a strict $0$-walk between $x$ and $y$, which exists by \ref{SS1}. 
	Then, by \eqref{e:RM.invariance}, we have $R_{E_{p,n}}(x_{i-1},x_{i}) = R_{E_{p,n}}(a_{1},a_{2})$ for any $i \in \{ 1,\dots,k \}$. 
	Hence 
	\[
	R_{E_{p,n}}(x,y)^{1/p} 
	\le \sum_{i=1}^{k}R_{E_{p,n}}(x_{i-1},x_{i})^{1/p} 
	= kR_{E_{p,n}}(a_{1},a_{2})^{1/p}
	\le (\#V_{0})R_{E_{p,n}}(a_{1},a_{2})^{1/p}.  
	\] 
	Next we prove the lower estimate in \eqref{e:condA.ANF}. 
	The case of $\abs{x-y} = l_{0}$ is clear by \eqref{e:RM.invariance}, so we assume that $\abs{x-y} > l_{0}$.  
	By \ref{SS3} and \eqref{e:RM.invariance}, it suffices to consider the case where there exists $z \in V_0$ such that $0 < \abs{x - z} < \abs{x - y}$ and $g_{yz}|_{K} \in \mathcal{G}_{\mathrm{sym}}$. 
	Now we define $u \in \mathbb{R}^{V_{n}}$ by 
	\[
	u(a) \coloneqq 
	\begin{cases}
		h_{\{ x,z \}}^{E_{p,n}}[\indicator{x}](a) \quad &\text{if $a \in V_{n}$ satisfies $\abs{a-z} \le \abs{a-y}$,} \\
		h_{\{ x,z \}}^{E_{p,n}}[\indicator{x}](g_{yz}(a)) \quad &\text{if $a \in V_{n}$ satisfies $\abs{a-z} \ge \abs{a-y}$.} 
	\end{cases}
	\]   
	Since $\abs{x-z} < \abs{x-y}$, we have $u(x) = h_{\{ x,z \}}^{E_{p,n}}[\indicator{x}](x) = 1$. 
	Also, $u(y) = h_{\{ x,z \}}^{E_{p,n}}[\indicator{x}](z) = 0$. 
	Hence $R_{E_{p,n}}(x,y) \ge E_{p,n}(u)^{-1}$. 
	Setting $H_{1,n} \coloneqq \{ a \in V_{n} \mid \abs{a-z} \le \abs{a-y} \}$ and $H_{2,n} \coloneqq \{ a \in V_{n} \mid \abs{a-z} \ge \abs{a-y} \}$, we see from Proposition \ref{prop.ANF-geom}-\ref{it:ANF.symcompos} and the $\mathcal{G}_{\mathrm{sym}}$-invariance of $E_{p,n}$ that 
	\begin{align*}
		E_{p,n}(u)
		&= \sum_{\substack{w \in W_{n}; \\ F_{w}(V_{0}) \subseteq H_{1,n} \\ \text{or } F_{w}(V_{0}) \subseteq H_{2,n} }} \rweight_{p,w}E_{p,0}(u \circ F_{w}|_{V_{0}}) + \sum_{\substack{w \in W_{n}; \\ F_{w}(V_{0}) \not\subseteq H_{1,n} \\ \text{and } F_{w}(V_{0}) \not\subseteq H_{2,n}}} \rweight_{p,w}E_{p,0}(u \circ F_{w}|_{V_{0}}) \\
		&\leq 2\sum_{\substack{w \in W_{n}; \\ F_{w}(V_{0}) \subseteq H_{1,n}}}\rweight_{p,w}E_{p,0}\Bigl(h_{\{ x,z \}}^{E_{p,n}}[\indicator{x}] \circ F_{w}|_{V_{0}}\Bigr) + \sum_{\substack{w \in W_{n}; \\ F_{w}(V_{0}) \not\subseteq H_{1,n} \\ \text{and } F_{w}(V_{0}) \not\subseteq H_{2,n}}}\rweight_{p,w}E_{p,0}(u \circ F_{w}|_{V_{0}}). 
	\end{align*} 
	To estimate the second term in the right-hand side of the above inequality, let $w = w_{1} \ldots w_{n} \in W_{n}$ satisfy $F_{w}(V_{0}) \not\subseteq H_{1,n}$ and $F_{w}(V_{0}) \not\subseteq H_{2,n}$, so that there exist $a,b \in V_{0}$ such that $\abs{F_{w}(a) - z} < \abs{F_{w}(a) - y}$ and $\abs{F_{w}(b) - z} > \abs{F_{w}(b) - y}$. 
	Then by Proposition \ref{prop.ANF-geom}-\ref{it:ANF.reflection} we have $g_{yz}(K_{w}) = K_{w}$ and $g_{yz}(F_{w}(V_{0})) = F_{w}(V_{0})$. 
	We next claim that, if additionally $\abs{a - b} = l_{0}$, then $g_{yz}(F_{w}(a)) = F_{w}(b)$. 
	Indeed, if $\dist_{d}(\{F_{w}(a), F_{w}(b) \}, H_{yz}) < c_{w}l_{0}/2$ (recall that $c_{w} = c_{w_{1}} \cdots c_{w_{n}} \in (0,1)$ is the Lipschitz constant of $F_{w}$), then $\{ g_{yz}(F_{w}(a)), g_{yz}(F_{w}(b)) \} \subseteq F_{w}(V_{0})$ and $0 < \abs{a - F_{w}^{-1}(g_{yz}(F_{w}(a)))} \wedge \abs{b - F_{w}^{-1}(g_{yz}(F_{w}(b)))} < l_{0}$, contradicting the minimality of $l_{0}$. 
	Hence $\{ F_{w}(a), F_{w}(b) \} \subseteq \{ \alpha \in V_{n} \mid \dist_{d}(\alpha, H_{yz}) \ge c_{w}l_{0}/2 \}$. 
	Since $\abs{F_{w}(a) - z} < \abs{F_{w}(a) - y}$, $\abs{F_{w}(b) - z} > \abs{F_{w}(b) - y}$ and $\abs{F_{w}(a) - F_{w}(b)} = c_{w}l_{0}$, it follows that $F_{w}(a) - F_{w}(b)$ is orthogonal to $H_{yz}$ and that $\{ F_{w}(a), F_{w}(b) \} \subseteq \{ \alpha \in V_{n} \mid \dist_{d}(\alpha, H_{yz}) = c_{w}l_{0}/2 \}$, whence $g_{yz}(F_{w}(a)) = F_{w}(b)$ as claimed. 
	In particular, we have $u(F_{w}(a)) = u(F_{w}(b))$. 
	Combining this with Proposition \ref{prop.ANF-geom}-\ref{it:ANF.symcompos}, we have 
	\begin{align*}
		&\sum_{\substack{w \in W_{n}; \\ F_{w}(V_{0}) \not\subseteq H_{1,n} \\ \text{and } F_{w}(V_{0}) \not\subseteq H_{2,n}}}\rweight_{p,w}E_{p,0}(u \circ F_{w}|_{V_{0}}) \\
		&= \sum_{\substack{w \in W_{n}; \\ F_{w}(V_{0}) \not\subseteq H_{1,n} \\ \text{and } F_{w}(V_{0}) \not\subseteq H_{2,n}}}\rweight_{p,w}\sum_{\substack{a,b \in V_{0}; \abs{a-b} = l_{0}, \\ \{ F_{w}(a), F_{w}(b) \} \subseteq H_{1,n} \\ \text{or } \{ F_{w}(a), F_{w}(b) \} \subseteq H_{2,n}}}\abs{u(F_{w}(a)) - u(F_{w}(b))}^{p} \\
		&\le 2\sum_{\substack{w \in W_{n}; \\ F_{w}(V_{0}) \not\subseteq H_{1,n} \\ \text{and } F_{w}(V_{0}) \not\subseteq H_{2,n}}}\rweight_{p,w}E_{p,0}\Bigl(h_{\{ x,z \}}^{E_{p,n}}[\indicator{x}] \circ F_{w}|_{V_{0}}\Bigr),  
	\end{align*} 
	and we then deduce that $E_{p,n}(u) \leq 2 E_{p,n}\bigl(h_{\{ x,z \}}^{E_{p,n}}[\indicator{x}]\bigr)$ and thus that
	\[
	R_{E_{p,n}}(x,y) 
	\ge E_{p,n}(u)^{-1} 
	\ge \frac{1}{2}E_{p,n}\Bigl(h_{\{ x,z \}}^{E_{p,n}}[\indicator{x}]\Bigr)^{-1}
	= \frac{1}{2}R_{E_{p,n}}(x,z). 
	\]
	Repeating this estimate and using \eqref{e:RM.invariance}, we obtain the desired lower estimate in \eqref{e:condA.ANF}. 
\end{proof}

Combined with Theorem \ref{thm.eigenform-ANF}, the following theorem ensures the existence of $\mathcal{G}_{\mathrm{sym}}$-invariant self-similar $p$-resistance forms on strongly symmetric p.-c.f.\ self-similar sets. 
\begin{thm}\label{thm.ANFsymform}
	Let $\bm{\rweight}_{p} = (\rweight_{p,i})_{i \in S} \in (0,\infty)^{S}$ and assume that \textup{(\hyperref[it:condA']{\textbf{A}'})}, \textup{(\hyperref[it:condR]{\textbf{R}})} and \eqref{rscale.sym} hold.
	Then there exists a regular self-similar $p$-resistance form $(\mathcal{E}_{p},\mathcal{F}_{p})$ on $\mathcal{L}$ with weight $\bm{\rweight}_{p}$ and with $R_{\mathcal{E}_{p}}^{1/p}$ compatible with the original topology of $K$ such that it is $\mathcal{G}_{\mathrm{sym}}$-invariant, i.e., $u \circ g \in \mathcal{F}_{p}$ and $\mathcal{E}_{p}(u \circ g) = \mathcal{E}_{p}(u)$ for any $u \in \mathcal{F}_{p}$ and any $g \in \mathcal{G}_{\mathrm{sym}}$. 
\end{thm}
\begin{proof}
	Define $E_{0} \in \mathcal{S}_{p}(V_{0})$ by $E_{0}(u) \coloneqq \sum_{x,y \in V_{0}}\abs{u(x) - u(y)}^{p}$ for $u \in \mathbb{R}^{V_{0}}$. 
	Then $E_{0}(u) = E_{0}(u \circ g)$ for any $u \in \mathbb{R}^{V_{0}}$ and $g \in \mathcal{G}_{\mathrm{sym}}$. 
	By Theorem \ref{thm.eigenform} and in particular the explicit expressions \eqref{e:defn.CGGQ-En}--\eqref{e:eigenform.concrete} in Theorem \ref{thm.eigenform}-\ref{it:CGQ.newinitial},\ref{it:CGQ.eigenform}, there exists a $p$-resistance form $\mathcal{E}_{p}^{(0)}$ on $V_{0}$ such that $\mathcal{R}_{\bm{\rweight}_{p}}(\mathcal{E}_{p}^{(0)}) = \mathcal{E}_{p}^{(0)}$ and $\mathcal{E}_{p}^{(0)}(u) = \mathcal{E}_{p}^{(0)}(u \circ g)$ for any $u \in \mathbb{R}^{V_{0}}$ and any $g \in \mathcal{G}_{\mathrm{sym}}$. 
	Then $(\mathcal{E}_{p},\mathcal{F}_{p})$ defined by \eqref{pRFinductive.dom} and \eqref{pRFinductive.form} in Theorem \ref{thm.CGQ} has the desired properties;
	its $\mathcal{G}_{\mathrm{sym}}$-invariance follows from \eqref{rscale.sym}, Proposition \ref{prop.ANF-geom}-\ref{it:ANF.symcompos}, the fact that $\tau_{g}|_{W_{n}} \colon W_{n} \to W_{n}$ is a bijection for any $n \in \mathbb{N} \cup \{ 0 \}$, and the expressions \eqref{pRFinductive.dom} and \eqref{pRFinductive.form} for $(\mathcal{E}_{p},\mathcal{F}_{p})$.  
\end{proof}

\section{\texorpdfstring{$p$}{p}-Walk dimension of Sierpi\'{n}ski carpets/gaskets}\label{sec.p-walk}
In this section, we prove the strict inequality $\pwalk > p$ for the generalized Sierpi\'{n}ski carpets and the $D$-dimensional level-$l$ Sierpi\'{n}ski gasket as an application of the nonlinear potential theory developed in Sections \ref{sec.p-harm} and \ref{sec.compatible}.
In particular, we remove the \emph{planarity} in the hypothesis of the previous result \cite[Theorem 2.27]{Shi24}.   

\subsection{Generalized Sierpi\'{n}ski carpets}
Following \cite[Section 2]{Kaj23}, we recall the definition of the generalized Sierpi\'{n}ski carpets.
\begin{framework}\label{frmwrk:GSC}
Let $D,l\in\mathbb{N}$, $D\geq 2$, $l\geq 3$ and set $Q_{0} \coloneqq [0,1]^{D}$.
Let $S\subsetneq\{0,1,\ldots,l-1\}^{D}$ be non-empty, define
$f_{i}\colon\mathbb{R}^{D}\to\mathbb{R}^{D}$ by $f_{i}(x) \coloneqq l^{-1}i+l^{-1}x$ for each $i\in S$
and set $Q_{1} \coloneqq \bigcup_{i\in S}f_{i}(Q_{0})$, so that $Q_{1}\subsetneq Q_{0}$.
Let $K$ be the self-similar set associated with $\{f_{i}\}_{i\in S}$.
Note that $K\subsetneq Q_{0}$.
Set $F_{i} \coloneqq f_{i}|_{K}$ for each $i\in S$ and $\GSC(D,l,S) \coloneqq (K,S,\{F_{i}\}_{i\in S})$.
Let $d\colon K\times K\to[0,\infty)$ be the Euclidean metric on $K$ normalized so that $\diam(K,d) = 1$,
set $d_{\mathrm{f}}\coloneqq\log_{l}(\#S)$, and let $m$ be the self-similar measure on
$\GSC(D,l,S)$ with uniform weight $(1/\#S)_{i\in S}$.
\end{framework}
Recall that $d_{\mathrm{f}}$ is the Hausdorff dimension of $(K,d)$ and that
$m$ is a constant multiple of the $d_{\mathrm{f}}$-dimensional Hausdorff measure
on $(K,d)$; see, e.g., \cite[Proposition 1.5.8 and Theorem 1.5.7]{Kig01}.
Note that $d_{\mathrm{f}}<D$ by $S\subsetneq\{0,1,\ldots,l-1\}^{D}$.

The following definition is due to Barlow and Bass \cite[Section 2]{BB99}, except that
the non-diagonality condition in \cite[Hypotheses 2.1]{BB99} has been strengthened
later in \cite{BBKT} to fill a gap in \cite[Proof of Theorem 3.19]{BB99};
see \cite[Remark 2.10-1.]{BBKT} for some more details of this correction.
\begin{defn}[Generalized Sierpi\'{n}ski carpet]\label{dfn:GSC}
$\GSC(D,l,S)$ is called a \emph{generalized Sierpi\'{n}ski carpet}\index{generalized Sierpi\'{n}ski carpet}
if and only if the following four conditions are satisfied:
\begin{enumerate}[label=\textup{(GSC\arabic*)},align=left,leftmargin=*,topsep=2pt,parsep=0pt,itemsep=2pt]
\item\label{GSC1}(Symmetry) $f(Q_{1})=Q_{1}$ for any isometry $f$ of $\mathbb{R}^{D}$ with $f(Q_{0})=Q_{0}$.
\item\label{GSC2}(Connectedness) $Q_{1}$ is connected.
\item\label{GSC3}(Non-diagonality)
	$\interior_{\mathbb{R}^{D}}\bigl(Q_{1}\cap \prod_{k=1}^{D}[(i_{k}-\varepsilon_{k})l^{-1},(i_{k}+1)l^{-1}]\bigr)$
	is either empty or connected for any $(i_{k})_{k=1}^{D}\in\mathbb{Z}^{D}$ and
	any $(\varepsilon_{k})_{k=1}^{D}\in\{0,1\}^{D}$.
\item\label{GSC4}(Borders included) $[0,1]\times\{0\}^{D-1}\subseteq Q_{1}$.
\end{enumerate}
\end{defn} 

See \cite[Remark 2.2]{BB99} for a description of the meaning of each of the four
conditions \ref{GSC1}, \ref{GSC2}, \ref{GSC3} and \ref{GSC4} in Definition \ref{dfn:GSC}.
To be precise, \ref{GSC3} is slightly different from the formulation of the non-diagonality
condition in \cite[Subsection 2.2]{BBKT}, but they have been proved to be equivalent to
each other in \cite[Theorem 2.4]{Kaj10}; see \cite[Section 2]{Kaj10} for some other equivalent
formulations of the non-diagonality condition.

In this subsection, we assume that $\GSC(D,l,S)=(K,S,\{F_{i}\}_{i\in S})$
as introduced in Framework \ref{frmwrk:GSC} is a generalized Sierpi\'{n}ski carpet
as defined in Definition \ref{dfn:GSC}.

We next ensure the existence of a symmetry-invariant $p$-resistance form on $\GSC(D,l,S)$ for $p > \dim_{\textup{ARC}}(K,d)$ by applying Theorem \ref{thm.KSgood-ss}.
\begin{defn}\label{dfn:GSC-isometry}
We define
\begin{equation}\label{eq:GSC-isometry}
\mathcal{G}_{0}\coloneqq\{f|_{K}\mid\textrm{$f$ is an isometry of $\mathbb{R}^{D}$, $f(Q_{0})=Q_{0}$}\},
\end{equation}
which forms a finite subgroup of the group of homeomorphisms of $K$ by virtue of \ref{GSC1}.
\end{defn}
\begin{cor}\label{cor.GSCE}
Let $p \in (\dim_{\textup{ARC}}(K,d),\infty)$.
Then Assumption \ref{assum.BFss} holds with $r_{\ast} = l^{-1}$, $K$ is $p$-conductively homogeneous, and there exists a regular self-similar $p$-resistance form $(\mathcal{E}_{p},\mathcal{W}^{p})$ on $\GSC(D,l,S)$ with weight $(\sigma_{p})_{i \in S}$ such that it satisfies the conditions \ref{Kigss.equiv}--\ref{Kigss.RF} of Theorem \ref{thm.KSgood-ss}.
Moreover, $(\mathcal{E}_{p},\mathcal{W}^{p})$ has the following $\mathcal{G}_{0}$-invariance:
\begin{equation}\label{GSCE.sym}
\text{If $u\in \mathcal{W}^{p}$ and $g\in\mathcal{G}_{0}$
    then $u\circ g\in\mathcal{W}^{p}$ and $\mathcal{E}_{p}(u\circ g)=\mathcal{E}_{p}(u)$.}
\end{equation}
\end{cor}
\begin{proof}
    Assumption \ref{assum.BFss} and the $p$-conductive homogeneity for the generalized Sierpi\'{n}ski carpets in the case $p \in (\dim_{\textup{ARC}}(K,d),\infty)$ follow from \cite[Theorem 4.13]{Kig23} or \cite[Proposition 4.5 and Theorem 4.14]{Shi24}.
    Let $(\mathcal{E}_{p},\mathcal{W}^{p})$ be a self-similar $p$-resistance form given in Theorem \ref{thm.KSgood-ss}.
    Then the desired properties except for \eqref{GSCE.sym} are already proved. 
    We can easily see that $\widetilde{\mathcal{E}}_{p}^{n}(f \circ g) = \widetilde{\mathcal{E}}_{p}^{n}(f)$ for any $f \in L^p(K,m)$, any $g \in \mathcal{G}_{0}$ and any $n \in \mathbb{N}\cup\{ 0 \}$, and that the conditions \eqref{e:trans.bi-level}--\eqref{e:trans.weightinv} with $\mathcal{G}_{0}$ in place of $\mathscr{T}$ hold. 
	Hence the desired symmetry-invariance \eqref{GSCE.sym} follows from Theorem \ref{t:Kig-good}-\ref{Kig.inv}, \eqref{e:defn.Kigss} and Proposition \ref{prop.ssenergy-GCinv}-\ref{it:ssenergy.inv}. 
\end{proof}

Recall that $\sigma_{p}$ and $\pwalk$ are defined for any $p \in (0,\infty)$ (under Assumption \ref{assum.BFss}).
We know the following monotonicity on $\pwalk/p$ in $p \in (0,\infty)$. 
\begin{prop}\label{p:pwalk-mono.GSC}
    $\pwalk/p \ge d_{\mathrm{w},q}/q$ for any $p,q \in (0,\infty)$ with $p \le q$.
\end{prop}
\begin{proof}
	This follows from a H\"{o}lder-type estimate for conductance constants proved in \cite[Lemma 4.7.4]{Kig20} and the fact that $\hdim = \log_{l}(\#S)$. 
\end{proof}

The following definition is exactly the same as part of \cite[Definition 3.6]{Kaj23}.
\begin{defn}\label{dfn:GSC-V00V01-isometry-subgroup}
\begin{enumerate}[label=\textup{(\arabic*)},align=left,leftmargin=*,topsep=2pt,parsep=0pt,itemsep=2pt]
\item\label{it:GSC-V00V01} We set $V_{0}^{\varepsilon} \coloneqq K\cap(\{\varepsilon\}\times\mathbb{R}^{D-1})$ for each
	$\varepsilon\in\{0,1\}$ and $U_{0} \coloneqq K\setminus(V_{0}^{0}\cup V_{0}^{1})$.
\item\label{it:GSC-isometry-subgroup} We define $g_{\varepsilon}\in\mathcal{G}_{0}$ by $g_{\varepsilon}\coloneqq\tau_{\varepsilon}|_{K}$
	for each $\varepsilon=(\varepsilon_{k})_{k=1}^{D}\in\{0,1\}^{D}$,
	where $\tau_{\varepsilon}\colon\mathbb{R}^{D}\to\mathbb{R}^{D}$ is given by
	$\tau_{\varepsilon}((x_{k})_{k=1}^{D}) \coloneqq (\varepsilon_{k}+(1-2\varepsilon_{k})x_{k})_{k=1}^{D}$,
	and define a subgroup $\mathcal{G}_{1}$ of $\mathcal{G}_{0}$ by
	\begin{equation}\label{eq:GSC-isometry-subgroup}
	\mathcal{G}_{1}:=\{g_{\varepsilon}\mid\varepsilon\in\{0\}\times\{0,1\}^{D-1}\}.
	\end{equation}
\end{enumerate}
\end{defn}
In the rest of this subsection, we fix $p \in (\dim_{\textup{ARC}}(K,d), \infty)$ and a self-similar $p$-resistance form $(\mathcal{E}_{p},\mathcal{W}^{p})$ in Corollary \ref{cor.GSCE}.
Recall Theorem \ref{thm.RF-exist} and let $h_{0} \coloneqq h_{V_{0}^{0} \cup V_{0}^{1}}^{\mathcal{E}_{p}}\bigl[\indicator{V_{0}^{1}}\bigr] \in \mathcal{W}^{p}$. 
The strategy to prove $\pwalk > p$ is very similar to \cite{Kaj23}, that is, we will prove the \emph{non}-$\mathcal{E}_{p}$-harmonicity on $U_{0}$ of $h_{2}\coloneqq\sum_{w\in W_{2}}(F_{w})_{*}(l^{-2}h_{0}+q^{w}_{1}\indicator{K})\in \mathcal{W}^{p}$, which also satisfies $h_{2}|_{V_{0}^{i}} = i \, (i = 0, 1)$.
(See \cite[Figures 2 and 3]{Kaj23} for illustrations of $h_{0}$ and of $h_{2}$.)
Then the strict estimate $\pwalk > p$ will follow from $\mathcal{E}_{p}(h_{0})<\mathcal{E}_{p}(h_{2})$ and the self-similarity of $\mathcal{E}_{p}$.
Our arguments will be easier than that of \cite{Kaj23} by virtue of $\mathcal{W}^{p} \subseteq \contfunc(K)$.

The next proposition is a key ingredient.
Note that it requires our standing assumption that $S\not=\{0,1,\ldots,l-1\}^{D}$, which excludes the case of $K=[0,1]^{D}$ from the present framework.
\begin{prop}\label{prop:h2-non-harmonic}
Let $h_{2}\coloneqq\sum_{w\in W_{2}}(F_{w})_{*}(l^{-2}h_{0}+q^{w}_{1}\indicator{K})\in \mathcal{W}^{p}$.
Then $h_{2}$ is not $\mathcal{E}_{p}$-harmonic on $U_{0}$ and $h_{2}|_{V_{0}^{i}} = i$ for each $i \in \{ 0, 1 \}$.
\end{prop}
\begin{proof}
The proof is a straightforward modification of \cite[Proposition 3.11]{Kaj23} thanks to Theorem \ref{thm.RF-exist}. We present here a self-contained proof for the reader's convenience.

We claim that, if $h_{2}$ were $\mathcal{E}_{p}$-harmonic on $U_{0}$,
then $h_{0}\in\mathcal{W}^{p}$ would
turn out to be $\mathcal{E}_{p}$-harmonic on $K\setminus V_{0}^{0}$, which together with Proposition \ref{prop.equiv} would imply that $\mathcal{E}_{p}(h_{0}) = \mathcal{E}_{p}(h_{0};h_{0})=0$, which would be a contradiction by \ref{RF1} and $\mathcal{W}^{p} \subseteq \contfunc(K)$.

For each $\varepsilon=(\varepsilon_{k})_{k=1}^{D}\in\{1\}\times\{0,1\}^{D-1}$,
set $U^{\varepsilon} \coloneqq K\cap\prod_{k=1}^{D}(\varepsilon_{k}-1,\varepsilon_{k}+1)$ and
$K^{\varepsilon} \coloneqq K\cap\prod_{k=1}^{D}[\varepsilon_{k}-1/2,\varepsilon_{k}+1/2]$.
Fix $\varphi_{\varepsilon}\in \mathcal{W}^{p} \cap \contfunc_{c}(U^{\varepsilon})$ so that $\varphi_{\varepsilon}|_{K^{\varepsilon}}=\indicator{K^{\varepsilon}}$, which exists by \eqref{RM.comp}, \ref{RF1} and \ref{RF5}.
Let $v\in \mathcal{W}^{p} \cap \contfunc_{c}(K\setminus V_{0}^{0})$ and, taking an enumeration $\{\varepsilon^{(k)}\}_{k=1}^{2^{D-1}}$
of $\{1\}\times\{0,1\}^{D-1}$ and recalling Proposition \ref{prop.GC-list}-\ref{GC.leibniz}, define
$v_{\varepsilon}\in\mathcal{W}^{p} \cap \contfunc_{c}(U^{\varepsilon})$ for $\varepsilon\in\{1\}\times\{0,1\}^{D-1}$
by $v_{\varepsilon^{(1)}} \coloneqq v\varphi_{\varepsilon^{(1)}}$ and
$v_{\varepsilon^{(k)}} \coloneqq v\varphi_{\varepsilon^{(k)}}\prod_{j=1}^{k-1}(\indicator{K}-\varphi_{\varepsilon^{(j)}})$
for $k\in\{2,\ldots,2^{D-1}\}$.
Then $v-\sum_{\varepsilon\in\{1\}\times\{0,1\}^{D-1}}v_{\varepsilon}
	=v\prod_{\varepsilon\in\{1\}\times\{0,1\}^{D-1}}(\indicator{K}-\varphi_{\varepsilon})
	\in\mathcal{W}^{p}\cap\contfunc_{c}(U_{0})$,
hence $\mathcal{E}_{p}(h_{0};v)=\sum_{\varepsilon\in\{1\}\times\{0,1\}^{D-1}}\mathcal{E}_{p}(h_{0};v_{\varepsilon})$
by Proposition \ref{prop.equiv} (with $Y = K \setminus U_{0}$).
Therefore the desired $\mathcal{E}_{p}$-harmonicity of $h_{0}$ on $K\setminus V_{0}^{0}$ would be
obtained by deducing that $\mathcal{E}(h_{0};v_{\varepsilon})=0$ for any
$\varepsilon\in\{1\}\times\{0,1\}^{D-1}$.

To this end, set $\varepsilon^{(0)}\coloneqq(\indicator{\{1\}}(k))_{k=1}^{D}$, take
$i=(i_{k})_{k=1}^{D}\in S$ with $i_{1}<l-1$ and $i+\varepsilon^{(0)}\not\in S$,
which exists by $\emptyset\not=S\subsetneq\{0,1,\ldots,l-1\}^{D}$ and \ref{GSC1},
and let $\varepsilon=(\varepsilon_{k})_{k=1}^{D}\in\{1\}\times\{0,1\}^{D-1}$.
We will choose $i^{\varepsilon}\in S$ with $F_{ii^{\varepsilon}}(\varepsilon)\in F_{i}(K\cap(\{1\}\times(0,1)^{D-1}))$
and assemble $v_{\varepsilon}\circ g_{w}\circ F_{w}^{-1}$ with a suitable
$g_{w}\in\mathcal{G}_{1}$ for $w\in W_{2}$ with $F_{ii^{\varepsilon}}(\varepsilon)\in K_{w}$
into a function $v_{\varepsilon,2}\in \mathcal{W}^{p}\cap\contfunc_{c}(U_{0})$.
Specifically, set
$i^{\varepsilon,\eta}\coloneqq\bigl((l-1)(\indicator{\{1\}}(k)+1-\varepsilon_{k})+(2\varepsilon_{k}-1)\eta_{k}\bigr)_{k=1}^{D}$
for each $\eta=(\eta_{k})_{k=1}^{D}\in\{0\}\times\{0,1\}^{D-1}$ and
$I^{\varepsilon}\coloneqq\{\eta\in\{0\}\times\{0,1\}^{d-1}\mid i^{\varepsilon,\eta}\in S\}$,
so that $i^{\varepsilon} \coloneqq i^{\varepsilon,\zero_{D}}\in S$ by \ref{GSC4} and \ref{GSC1} and hence
$\zero_{D}\in I^{\varepsilon}$. Thanks to $v_{\varepsilon}\in \mathcal{W}^{p}\cap\contfunc_{c}(U^{\varepsilon})$ and
$i+\varepsilon^{(0)}\not\in S$ we can define $v_{\varepsilon,2}\in\contfunc(K)$ by setting
\begin{equation}\label{eq:vepsilon2}
v_{\varepsilon,2}|_{K_{w}}\coloneqq
	\begin{cases}
	v_{\varepsilon}\circ g_{\eta}\circ F_{w}^{-1}
		& \textrm{if $\eta\in I^{\varepsilon}$ and $w=ii^{\varepsilon,\eta}$}\\
	0 & \textrm{if $w\not\in\{ii^{\varepsilon,\eta}\mid\eta\in I^{\varepsilon}\}$}
	\end{cases}
	\quad\textrm{for each $w\in W_{2}$.}
\end{equation}
Then $\supp_{K}[v_{\varepsilon,2}]\subseteq K_{i}\setminus V_{0}^{0}\subseteq U_{0}$
by \eqref{eq:vepsilon2} and $i_{1}<l-1$.
In addition, $v_{\varepsilon,2}\circ F_{w}\in\mathcal{W}^{p}$ for any
$w\in W_{2}$ by \eqref{eq:vepsilon2}, $v_{\varepsilon}\in\mathcal{W}^{p}$
and the $\mathcal{G}_{0}$-invariance \eqref{GSCE.sym} of $(\mathcal{E}_{p},\mathcal{W}^{p})$.
Thus $v_{\varepsilon,2}\in\mathcal{F}_{p}$ by the self-similarity \eqref{SSE1} of $\mathcal{W}^{p}$ and therefore
$v_{\varepsilon,2}\in\mathcal{W}^{p}\cap\contfunc_{c}(U_{0})$.
Recall that $h_{2}\circ F_{w}=l^{-2}h_{0}+q^{w}_{1}\indicator{K}$ for any $w\in W_{2}$ and note that, by the uniqueness in Theorem \ref{thm.RF-exist}, $h_{0} \circ g_{\eta} = h_{0}$ for any $\eta \in I^{\varepsilon}$.
Then we have
\begin{align}\label{eq:Eh2vepsilon2}
    \mathcal{E}_{p}(h_{2};v_{\varepsilon,2})
    &= \sum_{\eta\in I^{\varepsilon}}\sigma_{p}^{2}l^{-2(p - 1)}\mathcal{E}_{p}(h_{0};v_{\varepsilon}\circ g_{\eta}) \nonumber\\
    &= \sum_{\eta\in I^{\varepsilon}}\sigma_{p}^{2}l^{-2(p - 1)}\mathcal{E}_{p}(h_{0}\circ g_{\eta};v_{\varepsilon})
    = (\#I^{\varepsilon})\sigma_{p}^{2}l^{-2(p - 1)}\mathcal{E}_{p}(h_{0};v_{\varepsilon}).
\end{align}
Now, supposing that $h_{2}$ were $\mathcal{E}_{p}$-harmonic on $U_{0}$, from \eqref{eq:Eh2vepsilon2},
$\#I^{\varepsilon}>0$, $v_{\varepsilon,2}\in\mathcal{F}_{p}\cap\contfunc_{c}(U_{0})$ and Proposition \ref{prop.equiv}, we would obtain
$\mathcal{E}_{p}(h_{0};v_{\varepsilon})=\sigma_{p}^{-2}l^{2(p - 1)}(\#I^{\varepsilon})^{-1}\mathcal{E}_{p}(h_{2};v_{\varepsilon,2})=0$,
which would imply a contradiction as explained in the last two paragraphs.
\end{proof}

\begin{thm}
    $\pwalk > p$ for any $p \in (0, \infty)$.
\end{thm}
\begin{proof}
It suffices to prove the case of $p \in (\dim_{\textup{ARC}}(K,d),\infty)$ by Proposition \ref{p:pwalk-mono.GSC}. 
Let $h_{0}, h_{2}\in\mathcal{W}^{p}$ be as in Proposition \ref{prop:h2-non-harmonic}.
By Proposition \ref{prop:h2-non-harmonic}, we have $\mathcal{E}_{p}(h_{0})<\mathcal{E}_{p}(h_{2})$.
This strict inequality combined with the self-similarity \eqref{SSE2} of $\mathcal{E}_{p}$ shows that
\begin{equation*}
\mathcal{E}_{p}(h_{0})<\mathcal{E}_{p}(h_{2})=\bigl(\sigma_{p}(\#S)l^{-p}\bigr)^{2}\mathcal{E}_{p}(h_{0}),
\end{equation*}
whence $l^{p}< \sigma_{p}(\#S)$.
Since $\sigma_{p} = l^{\pwalk - \hdim}$ and $\hdim = \log{(\#S)}/\log{l}$, we get $\pwalk=\log{\bigl(\sigma_{p}(\#S)\bigr)}/\log{l}>p$.
\end{proof}

\subsection{\texorpdfstring{$D$}{D}-dimensional level-\texorpdfstring{$l$}{l} Sierpi\'{n}ski gasket}\label{sec.SG}
Following \cite[Example 5.1]{Kaj13}, we introduce the $D$-dimensional level-$l$ Sierpi\'{n}ski gasket.
\begin{framework}[$D$-dimensional level-$l$ Sierpi\'{n}ski gasket\index{Sierpi\'{n}ski gasket}]\label{frmwrk:SG}
Let $D,l\in\mathbb{N}$, $D\geq 2$, $l\geq 2$ and let $\{ q_{k} \}_{k = 0}^{D} \subseteq \mathbb{R}^{D}$ be the set of the vertices of a regular $D$-dimensional simplex so that $q_{0}, \dots, q_{D - 1} \in \{ (x_{1},\dots,x_{D}) \in \mathbb{R}^{D} \mid x_{1} = 0\}$ and $q_{D} \in \{ (x_{1},\dots,x_{D}) \in \mathbb{R}^{D} \mid x_{1} \ge 0\}$.
Further let $S \coloneqq \bigl\{ (i_{k})_{k = 1}^{D} \bigm| i_{k} \in \mathbb{N} \cup \{ 0 \}, \sum_{k = 1}^{D}i_{k} \le l - 1 \bigr\}$, and for each $i = (i_{k})_{k = 1}^{D} \in S$ we set $q_{i} \coloneqq q_{0} + \sum_{k = 1}^{D}l^{-1}i_{k}(q_{k} - q_{0})$ and define $f_{i} \colon \mathbb{R}^{D} \to \mathbb{R}^{D}$ by $f_{i}(x) \coloneqq q_{i} + l^{-1}(x - q_{0})$.
Let $K$ be the self-similar set associated with $\{ f_{i} \}_{i \in S}$ and set $F_{i} \coloneqq f_{i}|_{K}$.
Let $\SG(D,l,S) = (K, S, \{ F_{i} \}_{i \in S})$, which is a self-similar structure.
Let $d\colon K\times K\to[0,\infty)$ be the Euclidean metric on $K$ normalized so that $\diam(K,d) = 1$,
set $d_{\mathrm{f}}\coloneqq\log_{l}(\#S)$, and let $m$ be the self-similar measure on
$\SG(D,l,S)$ with uniform weight $(1/\#S)_{i\in S}$.
\end{framework}
$\SG(D,l,S)$ is clearly an affine nested fractal (recall Framework \ref{frmwrk:ANF} and Definition \ref{defn.ANF}), and called the \emph{$D$-dimensional level-$l$ Sierpi\'{n}ski gasket}\index{Sierpi\'{n}ski gasket}. 
In the rest of this subsection, we fix the Sierpi\'{n}ski gasket $\SG(D,l,S)$ and the self-similar measure $m$ as in Framework \ref{frmwrk:SG}. 
We can easily verify \cite[Assumption 2.15]{Kig23} for $\SG(D,l,S)$. 
In addition, it is well known that $m$ is $d_{\textrm{f}}$-Ahlfors regular (see \cite[Proposition E.7]{Kig23} for example).
Similar to Corollary \ref{cor.GSCE}, we have a symmetry-invariant $p$-resistance form on $\SG(D,l,S)$ for any $p \in (1, \infty)$.
(The Ahlfors regular conformal dimension of $(K,d)$ is $1$; see Theorem \ref{thm.dARC-ANF}.)  

\begin{defn}\label{dfn:SG-isometry}
We define
\begin{equation}\label{eq:SG-isometry}
\mathcal{G}_{0}\coloneqq\{f|_{K}\mid\textrm{$f$ is an isometry of $\mathbb{R}^{D}$, $f(V_{0})=V_{0}$}\},
\end{equation}
which forms a finite subgroup of the group of homeomorphisms of $K$.
\end{defn}
\begin{cor}\label{cor.SGE}
Let $p \in (1,\infty)$.
Then Assumption \ref{assum.BFss} holds with $r_{\ast} = l^{-1}$, $K$ is $p$-conductively homogeneous, and there exists a regular self-similar $p$-resistance form $(\mathcal{E}_{p},\mathcal{W}^{p})$ on $\SG(D,l,S)$ with weight $(\sigma_{p})_{i \in S}$ such that it satisfies the conditions \ref{Kigss.equiv}--\ref{Kigss.RF} in Theorem \ref{thm.KSgood-ss}.
Moreover, $(\mathcal{E}_{p},\mathcal{W}^{p})$ has the following $\mathcal{G}_{0}$-invariance:
\begin{equation}\label{SGE.sym}
\text{If $u\in \mathcal{W}^{p}$ and $g\in\mathcal{G}_{0}$
    then $u\circ g\in\mathcal{W}^{p}$ and $\mathcal{E}_{p}(u\circ g)=\mathcal{E}_{p}(u)$.}
\end{equation}
\end{cor}

Similar to Proposition \ref{p:pwalk-mono.GSC}, we have the following monotonicity of $\pwalk/p$ in $p$. 
\begin{prop}\label{p:pwalk-mono.SGs}
    $\pwalk/p \ge d_{\mathrm{w},q}/q$ for any $p,q \in (0,\infty)$ with $p \le q$.
\end{prop}

We can prove the following main result by using compatible sequences.
\begin{thm}
    $\pwalk > p$ for any $p \in (0, \infty)$.
\end{thm}
\begin{proof}
    Let $p \in (1,\infty)$ and let $(\mathcal{E}_{p},\mathcal{W}^{p})$ be a self-similar $p$-resistance form on $\SG(D,l,S)$ as given in Corollary \ref{cor.SGE}.
    Define $u \in \contfunc(K)$ by $u(x_{1},\dots,x_{D}) \coloneqq x_{1}$ for any $(x_{1},\dots,x_{D}) \in K \subseteq \mathbb{R}^{D}$.
    Then $u|_{V_{n}} \in \mathcal{W}^{p}|_{V_{n}}$ for any $n \in \mathbb{N} \cup \{ 0 \}$ by Proposition \ref{prop.trace-dom}.
    We claim that if $u|_{V_{1}}$ were $\mathcal{E}_{p}|_{V_{1}}$-harmonic on $V_{1} \setminus V_{0}$, then $\mathcal{E}_{p}|_{V_{0}}(u|_{V_{0}}) = 0$, which would contradict \ref{RF1}.

    Suppose that $\mathcal{E}_{p}|_{V_{1}}(u|_{V_{1}};\varphi) = 0$ for every $\varphi \in \mathbb{R}^{V_{1}}$ with $\varphi|_{V_{0}} = 0$.
    Noting that $(u|_{V_{1}} \circ F_{i})|_{V_{0}} = l^{-1}u|_{V_{0}} + c_{i}\indicator{V_{0}}$ for some constant $c_{i} \in \mathbb{R}$ and using \eqref{eq:compat-seq}, we have
    \begin{equation}\label{SG-nonharm1}
        \mathcal{E}_{p}|_{V_{1}}(u|_{V_{1}};\varphi)
        = \sigma_{p}\sum_{i \in S}\mathcal{E}_{p}|_{V_{0}}(u|_{V_{1}} \circ F_{i};\varphi \circ F_{i})
        = l^{-(p - 1)}\sigma_{p}\sum_{i \in S}\mathcal{E}_{p}|_{V_{0}}(u|_{V_{0}}; \varphi \circ F_{i}).
    \end{equation}
    It is easy to see that $(V_{1} \setminus V_{0}) \cap \{ (x_{1},\dots,x_{D}) \in \mathbb{R}^{D} \mid x_{1} = 0 \} \neq \emptyset$.
    Let $z \in V_{1} \setminus V_{0}$ with $z \in \{ x_{1} = 0 \}$ and let $\varphi \coloneqq \indicator{\{ z \}}^{V_{1}} \in \mathbb{R}^{V_{1}}$.
    Since $u \circ g = u$ for any $g \in \mathcal{G}_{0}$ with $g(\{ x_{1} = 0 \} \cap K) = \{ x_{1} = 0 \} \cap K$, the $\mathcal{G}_{0}$-invariance \eqref{SGE.sym} implies $\mathcal{E}_{p}|_{V_{0}}\bigl(u|_{V_{0}}; \indicator{\{ q_{i} \}}^{V_{0}}\bigr) = \mathcal{E}_{p}|_{V_{0}}\bigl(u|_{V_{0}}; \indicator{\{ q_{j} \}}^{V_{0}}\bigr)$ for any $i,j \in \{ 0, \dots, D - 1 \}$.
    Since $\varphi \circ F_{j} = 0 \in \mathbb{R}^{V_{0}}$ for any $j \in S$ with $z \not\in K_{j}$, we have from \eqref{SG-nonharm1} that
    \begin{align*}
        0
        = \mathcal{E}_{p}|_{V_{1}}(u|_{V_{1}};\varphi)
        &= l^{-(p - 1)}\sigma_{p}\sum_{i \in S; z \in K_{i}}\mathcal{E}_{p}|_{V_{0}}(u|_{V_{0}}; \varphi \circ F_{i}) \nonumber \\
        &= l^{-(p - 1)}\sigma_{p}\cdot\bigl(\#\{ i \in S \mid z \in K_{i} \}\bigr)\mathcal{E}_{p}|_{V_{0}}\bigl(u|_{V_{0}}; \indicator{\{ q_{0}\}}^{V_{0}}\bigr).
    \end{align*}
    Hence we get $\mathcal{E}_{p}|_{V_{0}}\bigl(u|_{V_{0}}; \indicator{\{ q_{j} \}}^{V_{0}}\bigr) = 0$ for every $j \in \{ 0, \dots, D - 1 \}$.
    Therefore,
    \[
    \mathcal{E}_{p}|_{V_{0}}\bigl(u|_{V_{0}}; \indicator{\{ q_{D} \}}^{V_{0}}\bigr)
    = \mathcal{E}_{p}|_{V_{0}}\bigl(u|_{V_{0}}; \indicator{V_{0}}\bigr) = \sum_{j = 0}^{D - 1}\mathcal{E}_{p}|_{V_{0}}\bigl(u|_{V_{0}}; \indicator{\{ q_{j} \}}^{V_{0}}\bigr)
    = 0,
    \]
    which yields $\mathcal{E}_{p}|_{V_{0}}(u|_{V_{0}}; v) = 0$ for every $v \in \mathbb{R}^{V_{0}}$.
    In particular, $\mathcal{E}_{p}|_{V_{0}}(u|_{V_{0}}) = 0$, which is a contradiction and hence we conclude that $u|_{V_{1}}$ is \emph{not} $\mathcal{E}_{p}|_{V_{1}}$-harmonic on $V_{1} \setminus V_{0}$.
    Combining this observation with Proposition \ref{prop.trace-comp}, we see that
    \begin{equation}\label{SG-strict}
        \mathcal{E}_{p}|_{V_{0}}(u|_{V_{0}})
        = \mathcal{E}_{p}|_{V_{1}}|_{V_{0}}(u|_{V_{0}})
        = \mathcal{E}_{p}|_{V_{1}}\Bigl(h_{V_{0}}^{\mathcal{E}_{p}|_{V_{1}}}[u|_{V{0}}]\Bigr)
        < \mathcal{E}_{p}|_{V_{1}}(u|_{V_{1}}).
    \end{equation}
    Similar to \eqref{SG-nonharm1}, we have $\mathcal{E}_{p}|_{V_{1}}(u|_{V_{1}}) = l^{-p}\sigma_{p}(\#S)\mathcal{E}_{p}|_{V_{0}}(u|_{V_{0}})$.
    Hence the strict inequality \eqref{SG-strict} yields $1 < l^{-p}l^{\pwalk - \hdim}(\#S) = l^{\pwalk - p}$, which proves $\pwalk > p$ for any $p \in (1,\infty)$.
    By Proposition \ref{p:pwalk-mono.SGs}, we complete the proof.
\end{proof}


\section*{Acknowledgements}
The authors would like to thank Prof.\ Jun Kigami and Ms.\ Yuka Ota for some illuminating conversations related to this work; in particular, part of Subsection \ref{sec.ANF} is based on discussions with them. The authors also would like to thank Prof.\ Kazuhiro Kuwae for showing a version of draft of \cite{Kuw23+} during the preparation of this paper and for his comments on Proposition \ref{prop.GC-list}-\ref{GC.Cpsmall},\ref{GC.Cplarge} and Definition \ref{defn.RF-sl}-\ref{it:SL1s.pRF},\ref{it:SL2s.pRF} in an earlier version of this paper. 
Naotaka Kajino was supported in part by JSPS KAKENHI Grant Numbers JP21H00989, JP22H01128, JP23K22399. Ryosuke Shimizu was supported in part by JSPS KAKENHI Grant Numbers JP20J20207, JP23KJ2011. This work was supported by the Research Institute for Mathematical Sciences, an International Joint Usage/Research Center located in Kyoto University.

\appendix 
\section*{Appendix}
\addcontentsline{toc}{section}{Appendix}
\section{Symmetric Dirichlet forms and the generalized \texorpdfstring{$2$}{2}-contraction property}\label{sec.bilinear}
Recall the generalized $p$-contraction property \ref{GC} introduced in Definition \ref{defn.GC}.
In this section, we verify the generalized $p$-contraction property for various $p$-energy forms resulting from symmetric Dirichlet forms. 
\subsection{Symmetric Dirichlet forms satisfy the generalized \texorpdfstring{$2$-contraction}{2-contraction} property}
In this subsection, we show that any symmetric Dirichlet form satisfies \hyperref[GC]{(GC)$_{2}$}.
Throughout this subsection, we fix a measure space $(X,\mathcal{B},m)$. 

Let us recall the definition of the notion of symmetric Dirichlet form. 
See, e.g., \cite{CF,FOT,MR} for details of the theory of Dirichlet forms. 
\begin{defn}[Symmetric Dirichlet form\index{Dirichlet form}]
	Let $\mathcal{F}$ be a dense linear subspace of $L^{2}(X,m)$ and let $\mathcal{E} \colon \mathcal{F} \times \mathcal{F} \to \mathbb{R}$ be a non-negative definite symmetric bilinear form on $\mathcal{F}$. 
	The pair $(\mathcal{E},\mathcal{F})$ is said to be a \emph{symmetric Dirichlet form} on $L^{2}(X,m)$ if and only if $\mathcal{F}$ equipped with the inner product $\mathcal{E} + \langle \,\cdot\,,\,\cdot\,\rangle_{L^{2}(X,m)}$ is a Hilbert space, $u^{+} \wedge 1 \in \mathcal{F}$ and $\mathcal{E}(u^{+} \wedge 1,,u^{+} \wedge 1) \le \mathcal{E}(u,u)$ for any $u \in \mathcal{F}$. 
\end{defn}

We can show that a symmetric Dirichlet form $(\mathcal{E},\mathcal{F})$ satisfies \hyperref[GC]{(GC)$_{2}$} by modifying the proof of \cite[Theorem I.4.12]{MR}. 
\begin{prop}\label{prop.GC.DF}
	Let $(\mathcal{E},\mathcal{F})$ be a symmetric Dirichlet form on $L^{2}(X,m)$. 
	Then $(\mathcal{E}_{2},\mathcal{F})$ given by $\mathcal{E}_{2}(u) \coloneqq \mathcal{E}(u,u)$ is a $2$-energy form on $(X,m)$ satisfying \hyperref[GC]{\textup{(GC)$_{2}$}}.   
\end{prop}
\begin{proof}
	It is clear that $\mathcal{E}_{2}^{1/2}$ is a seminorm on $\mathcal{F}$, so we shall prove \hyperref[GC]{(GC)$_{2}$} for $(\mathcal{E}_{2},\mathcal{F})$. 
	Let $n_{1},n_{2} \in \mathbb{N}$, $q_{1} \in (0,2]$, $q_{2} \in [2,\infty]$ and $T = (T_{1},\dots,T_{n_{2}}) \colon \mathbb{R}^{n_{1}} \to \mathbb{R}^{n_{2}}$ satisfy \eqref{GC-cond}. 
	We consider the case of $q_{2} < \infty$ (the case of $q_{2} = \infty$ is similar). 
	Let $\{ G_{\alpha} \}_{\alpha > 0}$ be the strongly continuous resolvent on $L^{2}(X,m)$ associated with $(\mathcal{E},\mathcal{F})$; see, e.g., \cite[Theorem I.2.8]{MR}. 
	By \cite[Theorem I.2.13-(ii)]{MR}, it suffices to prove that for any $\bm{u} = (u_{1},\dots,u_{n_{1}}) \in L^{2}(X,m)^{n_{1}}$ and any $\alpha \in (0,\infty)$, 
	\begin{equation}\label{e:DF.GC.reduction}
		\left(\sum_{l = 1}^{n_{2}}\langle (1 - \alpha G_{\alpha})T_{l}(\bm{u}),T_{l}(\bm{u})\rangle_{L^2(X,m)}^{q_{2}/2}\right)^{1/q_{2}}
		\le \left(\sum_{k = 1}^{n_{1}}\langle (1 - \alpha G_{\alpha})u_{k},u_{k}\rangle_{L^2(X,m)}^{q_{1}/2}\right)^{1/q_{1}}. 
	\end{equation}
	By the linearity of $G_{\alpha}$ and \eqref{GC-cond}, it is enough to prove \eqref{e:DF.GC.reduction} in the case where $u_{k}$ is a simple function for each $k \in \{ 1,\dots,n_{1} \}$, so we assume that 
	\begin{equation}\label{e:uk.simple}
		u_{k} = \sum_{i = 1}^{N}\alpha_{ki}\indicator{A_{i}}, \quad k \in \{ 1,\dots,n_{1} \}, 
	\end{equation}
	where $N \in \mathbb{N}$, $(\alpha_{ki})_{i = 1}^{N} \subseteq \mathbb{R}$, $\{ A_{i} \}_{i = 1}^{N} \subseteq \mathcal{B}(X)$ with $m(A_{i}) < \infty$ and $A_{i} \cap A_{j} = \emptyset$ for $i \neq j$. 
	Fix $\alpha \in (0,\infty)$ and, for $i,j \in \{ 1,\dots,N \}$, we define 
	\[
	b_{i,j} \coloneqq \langle (1 - \alpha G_{\alpha})\indicator{A_{i}}, \indicator{A_{j}} \rangle_{L^2(X,m)}, 
	\quad \lambda_{i} \coloneqq m(A_{i})
	\quad \text{and} \quad a_{ij} \coloneqq \langle \alpha G_{\alpha}\indicator{A_{i}}, \indicator{A_{j}} \rangle_{L^2(X,m)}. 
	\]
	Then $b_{ij} = \lambda_{i}\delta_{ij} - a_{ij}$ by a simple calculation, and $a_{ij} = a_{ji}$ since $G_{\alpha}$ is a symmetric operator on $L^{2}(X,m)$ (see, e.g., \cite[Theorem I.2.8]{MR}). 
	Hence for any $(z_{1},\dots,z_{N}) \in \mathbb{R}^{N}$, 
	\begin{equation}\label{e:DF.GC.simple}
		\sum_{i,j = 1}^{N}z_{i}z_{j}b_{ij} = \sum_{i < j}a_{ij}(z_{i} - z_{j})^{2} + \sum_{j = 1}^{N}m_{j}z_{j}^{2}, 
	\end{equation}
	where $m_{j} \coloneqq \lambda_{j} - \sum_{i = 1}^{N}a_{ij}$. 
	Note that $a_{ij} \ge 0$ for any $i,j \in \{ 1,\dots,N \}$ since $\alpha G_{\alpha}\indicator{A_{i}} \ge 0$ by \cite[Theorem I.4.4]{MR}. 
	We set $A \coloneqq \bigcup_{i = 1}^{N}A_{i}$, and then we have $\alpha G_{\alpha}(\indicator{A}) \le 1$ by \cite[Theorem I.4.4]{MR} and 
	\begin{align*}
		\sum_{u = 1}^{N}a_{ij}
		= \alpha\int_{X}\indicator{A}G_{\alpha}(\indicator{A_{j}})\,dm
		= \alpha\int_{X}G_{\alpha}(\indicator{A})\indicator{A_{j}}\,dm 
		\le \int_{X}\indicator{A_{j}}\,dm
		= \lambda_{j}, 
	\end{align*} 
	whence $m_{j} \ge 0$. 
	With these preparations, we show \eqref{e:DF.GC.reduction} for $\bm{u}$ defined in \eqref{e:uk.simple}. 
	Set $T_{l,i} \coloneqq T_{l}(\alpha_{1i},\dots,\alpha_{u_{1}i})$ for each $l \in \{ 1,\dots,n_{2} \}$. 
	We see that 
	\begin{align*}
        &\sum_{l = 1}^{n_{2}}\langle (1 - \alpha G_{\alpha})T_{l}(\bm{u}),T_{l}(\bm{u})\rangle_{L^2(X,m)}^{q_{2}/2}
        = \sum_{l = 1}^{n_{2}}\left(\sum_{i,j = 1}^{N}T_{l,i}T_{l,j}b_{ij}\right)^{q_{2}/2} \\
        &\overset{\eqref{e:DF.GC.simple}}{=} \sum_{l = 1}^{n_{2}}\left(\sum_{i < j}a_{ij}(T_{l,i} - T_{l,j})^{q_{2} \cdot \frac{2}{q_{2}}} + \sum_{j = 1}^{N}m_{j}T_{l,j}^{q_{2} \cdot \frac{2}{q_{2}}}\right)^{q_{2}/2} \\
        &\overset{\eqref{reverse}}{\le} \left(\sum_{i < j}\left(a_{ij}^{q_{2}/2}\sum_{l = 1}^{n_{2}}(T_{l,i} - T_{l,j})^{q_{2}}\right)^{2/q_{2}} + \sum_{j = 1}^{N}\left(m_{j}^{q_{2}/2}\sum_{l = 1}^{n_{2}}T_{l,j}^{q_{2}}\right)^{2/q_{2}}\right)^{q_{2}/2} \nonumber\\ 
        &\overset{\eqref{GC-cond}}{\le} \left(\sum_{i < j}\left(a_{ij}^{q_{2}/2}\left(\sum_{k = 1}^{n_{1}}(\alpha_{ki} - \alpha_{kj})^{q_{1}}\right)^{q_{2}/q_{1}}\right)^{2/q_{2}} + \sum_{j = 1}^{N}\left(m_{j}^{q_{2}/2}\left(\sum_{k = 1}^{n_{1}}\alpha_{kj}^{q_{1}}\right)^{q_{2}/q_{1}}\right)^{2/q_{2}}\right)^{q_{2}/2} \nonumber\\ 
        &= \left(\sum_{i < j}\left(\sum_{k = 1}^{n_{1}}\bigl(a_{ij}(\alpha_{ki} - \alpha_{kj})^{2}\bigr)^{q_{1}/2}\right)^{2/q_{1}} + \sum_{j = 1}^{N}\left(\sum_{k = 1}^{n_{1}}\bigl(m_{j}\alpha_{kj}^{2}\bigr)^{q_{1}/2}\right)^{2/q_{1}}\right)^{\frac{q_{1}}{2}\cdot\frac{q_{2}}{q_{1}}} \nonumber\\ 
        &\overset{\textup{($\ast$)}}{\le} \left(\left(\sum_{k = 1}^{n_{1}}\left(\sum_{i < j}a_{ij}(\alpha_{ki} - \alpha_{kj})^{2} + \sum_{j = 1}^{N}m_{j}\alpha_{kj}^{2}\right)^{q_{1}/2}\right)^{2/q_{1}}\right)^{\frac{q_{1}}{2}\cdot\frac{q_{2}}{q_{1}}} \nonumber\\ 
         &= \left(\sum_{k = 1}^{n_{1}}\left(\sum_{i < j}a_{ij}(\alpha_{ki} - \alpha_{kj})^{2} + \sum_{j = 1}^{N}m_{j}\alpha_{kj}^{2}\right)^{q_{1}/2}\right)^{q_{2}/q_{1}} \nonumber\\ 
        &\overset{\eqref{e:DF.GC.simple}}{=} \left(\sum_{k = 1}^{n_{1}}\left(\sum_{i, = 1}^{N}\alpha_{ki}\alpha_{kj}b_{ij}\right)^{q_{1}/2}\right)^{q_{2}/q_{1}} 
        = \left(\sum_{k = 1}^{n_{1}}\langle (1 - \alpha G_{\alpha})u_{k},u_{k}\rangle_{L^2(X,m)}^{q_{1}/2}\right)^{\frac{2}{q_1}\cdot\frac{q_2}{2}},
    \end{align*}
    where we used the triangle inequality for $\ell^{2/q_{1}}$-norm in ($\ast$). 
    The proof is completed. 
\end{proof}

Next we will extend \hyperref[GC]{\textup{(GC)$_{2}$}} to $(\mathcal{E}_{2},\mathcal{F}_{e})$, where $\mathcal{F}_{e}$ is the \emph{extended Dirichlet space}; see Definition \ref{defn.extDF} below. 
(See, e.g., \cite[Section 1.5]{FOT} or \cite[Section 1.1]{CF} for details on the extended Dirichlet space.)
We need to recall the following result. 

\begin{prop}[{\cite[Proposition 2]{Sch.fatou}} and {\cite[Lemma 1]{Sch.ext}}\footnote{In \cite[Lemma 1]{Sch.ext}, Proposition \ref{prop.Fe-well} is stated and proved for a much wider class of $(\mathcal{E},\mathcal{F})$. The assumptions made in \cite[Lemma 1]{Sch.ext} are that $(X,\mathcal{B},m)$ is an arbitrary measure space, that $\mathcal{F}$ is a linear subspace of $L^{0}(X,m)$ and that $\mathcal{E} \colon \mathcal{F} \times \mathcal{F} \to \mathbb{R}$ is a non-negative definite bilinear form satisfying the \emph{strong sector condition} (see \cite[Definition 1]{Sch.ext}) and the Fatou property (see \cite[Definition 2]{Sch.ext}), both of which are satisfied if $(\mathcal{E},\mathcal{F})$ is a symmetric Dirichlet form. Indeed, the strong sector condition is immediate from the Cauchy--Schwarz inequality for $\mathcal{E}$, and the Fatou property for $(\mathcal{E},\mathcal{F})$ holds by \cite[Proposition 2]{Sch.fatou}.}]\label{prop.Fe-well}
	Let $(\mathcal{E},\mathcal{F})$ be a symmetric Dirichlet form on $L^{2}(X,m)$.
	If $\{ u_{n} \}_{n \in \mathbb{N}} \subseteq \mathcal{F}$ converges $m$-a.e.\ to $0$ as $n \to \infty$ and $\lim_{k \wedge l \to \infty}\mathcal{E}(u_{k} - u_{l},u_{k} - u_{l}) = 0$, then $\lim_{n \to \infty}\mathcal{E}(u_{n},u_{n}) = 0$. 
\end{prop}

Now we define the extended Dirichlet form $(\mathcal{E},\mathcal{F}_{e})$. 
\begin{defn}[Extended Dirichlet form\index{extended Dirichlet form}]\label{defn.extDF}
	Let $(\mathcal{E},\mathcal{F})$ be a symmetric Dirichlet form on $L^{2}(X,m)$. 
	We define the \emph{extended Dirichlet form} $(\mathcal{E},\mathcal{F}_{e})$ of $(\mathcal{E},\mathcal{F})$ by 
	\begin{align}
	\mathcal{F}_{e} &\coloneqq 
	\biggl\{ f \in L^{0}(X,m) \biggm|
	\begin{minipage}{225pt}
		$\lim_{n \to \infty}f_{n} = f$ $m$-a.e.\ for some $\{ f_{n} \}_{n \in \mathbb{N}} \subseteq \mathcal{F}$ with $\lim_{k \wedge l \to \infty}\mathcal{E}(f_{k} - f_{l}, f_{k} - f_{l}) = 0$
	\end{minipage} 
	\biggr\}, \label{e:defn.extDF} \\
	\mathcal{E}(f,f) &\coloneqq \lim_{n \to \infty}\mathcal{E}(f_{n},f_{n}), \quad \text{$f \in \mathcal{F}_{e}$, where $\{ f_{n} \}_{n \in \mathbb{N}}$ is a sequence as in \eqref{e:defn.extDF}.}
	\label{e:defn.extDF.form}
	\end{align}
	Each such $\{ f_{n} \}_{n \in \mathbb{N}}$ as in \eqref{e:defn.extDF} is called an \emph{approximating sequence for $f$}. 
	By Proposition \ref{prop.Fe-well}, the limit $\lim_{n \to \infty}\mathcal{E}(f_{n},f_{n})$ in \eqref{e:defn.extDF.form} does not depend on a particular choice of $\{ f_{n} \}_{n \in \mathbb{N}}$, and we have $\mathcal{F} = \mathcal{F}_{e} \cap L^{2}(X,m)$ by \cite[Proposition 2]{Sch.fatou}; see also \cite[Theorem 1.1.5]{CF}. 
\end{defn}

We also need the following proposition, which is proved by utilizing a version \cite[Theorem A.4.1-(ii)]{CF} of the Banach--Saks theorem. 
\begin{prop}[{\cite[Lemma 2]{Sch.ext}}\footnote{In \cite[Lemma 2]{Sch.ext}, Proposition \ref{prop.extcriteria} is stated and proved for the same class of bilinear forms $(\mathcal{E},\mathcal{F})$ as Proposition \ref{prop.Fe-well} is in \cite[Lemma 1]{Sch.ext}.}]\label{prop.extcriteria}
	Let $(\mathcal{E},\mathcal{F})$ be a symmetric Dirichlet form on $L^{2}(X,m)$, and let $\{ u_{n} \}_{n \in \mathbb{N}} \subseteq \mathcal{F}$. 
	If $\liminf_{n \to \infty}\mathcal{E}(u_{n},u_{n}) < \infty$ and $\{ u_{n} \}_{n \in \mathbb{N}}$ converges $m$-a.e.\ to $u \in L^{0}(X,m)$ as $n \to \infty$, then $u \in \mathcal{F}_{e}$ and $\mathcal{E}(u,u) \le \liminf_{n \to \infty}\mathcal{E}(u_{n},u_{n})$. 
\end{prop}

Now we can show that the extended Dirichlet form $(\mathcal{E},\mathcal{F}_{e})$ satisfies \hyperref[GC]{\textup{(GC)$_{2}$}}.  
\begin{prop}\label{prop.GC.DFext}
	Let $(\mathcal{E},\mathcal{F})$ be a symmetric Dirichlet form on $L^{2}(X,m)$. 
	Then $(\mathcal{E}_{2},\mathcal{F}_{e})$ given by $\mathcal{E}_{2}(u) \coloneqq \mathcal{E}(u,u)$ is a $2$-energy form on $(X,m)$ satisfying \hyperref[GC]{\textup{(GC)$_{2}$}}.  
\end{prop}
\begin{proof} 
	It is clear that $\mathcal{E}_{2}^{1/2}$ is a seminorm on $\mathcal{F}_{e}$.  
	Let us show \hyperref[GC]{\textup{(GC)$_{2}$}} for $(\mathcal{E}_{2},\mathcal{F}_{e})$. 
	As in the proof of Proposition \ref{prop.GC.DF}, let $n_{1},n_{2} \in \mathbb{N}$, $q_{1} \in (0,2]$, $q_{2} \in [2,\infty]$ and $T = (T_{1},\dots,T_{n_{2}}) \colon \mathbb{R}^{n_{1}} \to \mathbb{R}^{n_{2}}$ satisfy \eqref{GC-cond}. 
	Let $\bm{u} = (u_{1},\dots,u_{n_{1}}) \in \mathcal{F}_{e}^{n_{1}}$. 
	For each $k \in \{ 1,\dots,n_{1} \}$, let $\{ u_{k,n} \}_{n \in \mathbb{N}} \subseteq \mathcal{F}$ be an approximating sequence for $u_{k}$. 
	Set $\bm{u}_{n} \coloneqq (u_{1,n},\dots,u_{n_{1},n})$. 
	Since $T_{l} \in \contfunc(\mathbb{R}^{n_{1}})$ and $(\mathcal{E}_{2},\mathcal{F})$ satisfies \hyperref[GC]{\textup{(GC)$_{2}$}} (Proposition \ref{prop.GC.DF}), $\lim_{n \to \infty}T_{l}(\bm{u}_{n}) = T_{l}(\bm{u})$ $m$-a.e.\ and $\{ \mathcal{E}_{2}(T_{l}(\bm{u}_{n})) \}_{n \in \mathbb{N}}$ is bounded. 
	Then we have $T_{l}(\bm{u}) \in \mathcal{F}_{e}$ and $\mathcal{E}_{2}(T_{l}(\bm{u})) \le \liminf_{n \to \infty}\mathcal{E}_{2}(T_{l}(\bm{u}_{n}))$ by Proposition \ref{prop.extcriteria}, and see by \hyperref[GC]{\textup{(GC)$_{2}$}} for $(\mathcal{E}_{2},\mathcal{F})$ from Proposition \ref{prop.GC.DF} that 
	\begin{align*}
		\norm{\bigl(\mathcal{E}_{2}(T_{l}(\bm{u}))^{1/2}\bigr)_{l = 1}^{n_{2}}}_{\ell^{q_{2}}} 
		&\le \norm{\bigl(\liminf_{n \to \infty}\mathcal{E}_{2}(T_{l}(\bm{u}_{n}))^{1/2}\bigr)_{l = 1}^{n_{2}}}_{\ell^{q_{2}}} \\
		&\le \liminf_{n \to \infty}\norm{\bigl(\mathcal{E}_{2}(T_{l}(\bm{u}_{n}))^{1/2}\bigr)_{l = 1}^{n_{2}}}_{\ell^{q_{2}}} \\
		&\le \liminf_{n \to \infty}\norm{\bigl(\mathcal{E}_{2}(u_{k,n})^{1/2}\bigr)_{k = 1}^{n_{1}}}_{\ell^{q_{1}}}
		= \norm{\bigl(\mathcal{E}_{2}(u_{k})^{1/2}\bigr)_{k = 1}^{n_{1}}}_{\ell^{q_{1}}}, 
	\end{align*}
	proving that $(\mathcal{E}_{2},\mathcal{F}_{e})$ satisfies \hyperref[GC]{\textup{(GC)$_{2}$}}. 
\end{proof}

\subsection{The generalized \texorpdfstring{$2$}{2}-contraction property for energy measures}
In this subsection, under the standard topological assumptions on $(X,m)$, we verify \hyperref[GC]{\textup{(GC)$_{2}$}} for the ($2$-)energy measures associated with any regular symmetric Dirichlet form and for their densities and, under the strong locality, also \ref{GC} for the $p$-energy form defined as the integral of the $\frac{p}{2}$-th power of those densities, which has recently been studied by Kuwae \cite{Kuw23+}.

Throughout this subsection, we assume that $X$ and $m$ are as specified in \eqref{a:loccpt} and \eqref{a:fullRadon}, which are precisely the topological assumption \cite[(1.1.7)]{FOT} made almost throughout the book \cite{FOT}, and that $(\mathcal{E},\mathcal{F})$ is a symmetric Dirichlet form on $L^{2}(X,m)$ which is \emph{regular}, i.e., possesses a core in the sense of Definition \ref{d:core}-\ref{it:d-core}. 

A regular symmetric Dirichlet form is known to satisfy the following representation.
\begin{thm}[Beurling--Deny expression\index{Beurling--Deny expression (of regular symmetric Dirichlet form)}; see, e.g., {\cite[Theorem 3.2.1]{FOT}}]
	There exist a symmetric bilinear form $\mathcal{E}^{(c)}$ on $\mathcal{F} \cap \contfunc_{c}(X)$ satisfying $\mathcal{E}^{(c)}(u,v) = 0$ for any $u,v \in \mathcal{F} \cap \contfunc_{c}(X)$ with $v$ constant on a neighborhood of $\supp_{X}[u]$, a symmetric Radon measure $J$ on $X \times X \setminus \{ (x,x) \mid x \in X \}$ and a Radon measure $k$ on $X$ such that 
	\begin{equation}\label{e:BDdecomp}
		\mathcal{E}(u,v)
		= \mathcal{E}^{(c)}(u,v) + \mathcal{E}^{(j)}(u,v) + \mathcal{E}^{(k)}(u,v) \quad \text{for any $u,v \in \mathcal{F} \cap \contfunc_{c}(X)$,} 
	\end{equation}
	where 
	\[
	\mathcal{E}^{(j)}(u,v) \coloneqq \int_{X \times X}(u(x) - u(y))(v(x) - v(y))\,J(dxdy), \quad \mathcal{E}^{(k)}(u,v) \coloneqq \int_{X}u(x)v(x)\,k(dx). 
	\]
	Moreover, such $\mathcal{E}^{(c)}$, $J$ and $k$ are uniquely determined by $\mathcal{E}$. 
	We call $\mathcal{E}^{(c)}$ the \emph{strongly local part}\index{strongly local part (of regular symmetric Dirichlet form)} of $\mathcal{E}$, $J$ the \emph{jumping measure}\index{jumping measure (of regular symmetric Dirichlet form)} associated with $\mathcal{E}$ and $k$ the \emph{killing measure}\index{killing measure (of regular symmetric Dirichlet form)} associated with $\mathcal{E}$. 
\end{thm}

In the next two propositions, we extend each of $\mathcal{E}^{(c)},\mathcal{E}^{(j)},\mathcal{E}^{(k)}$ in \eqref{e:BDdecomp} to $\mathcal{F}_{e}$ and associate energy measures to each of them; see \cite[Section 3.2]{FOT} for their proofs. 
\begin{prop} 
	Let $u \in \mathcal{F}_{e}$ and $\{ u_{n} \}_{n \in \mathbb{N}} \subseteq \mathcal{F}$ be an approximating sequence for $u$. 
	Then, for any $\mathcal{E}^{\#} \in \{ \mathcal{E}^{(c)}, \mathcal{E}^{(j)}, \mathcal{E}^{(k)} \}$, $\{ \mathcal{E}^{\#}(u_n,u_n) \}_{n \in \mathbb{N}}$ is a Cauchy sequence in $[0,\infty)$ and the limit $\lim_{n \to \infty}\mathcal{E}^{\#}(u_{n},u_{n}) \in [0,\infty)$ does not depend on a particular choice of an approximating sequence $\{ u_{n} \}_{n \in \mathbb{N}}$ for $u$. 
\end{prop} 

\begin{prop}\label{prop.DFEM} 
	Let $\mathcal{E}^{\#} \in \{ \mathcal{E}, \mathcal{E}^{(c)}, \mathcal{E}^{(j)}, \mathcal{E}^{(k)} \}$. 
	For any $u \in \mathcal{F} \cap \contfunc_{c}(X)$, there exists a unique Radon measure $\mu_{\langle u \rangle}^{\#}$ on $X$ such that 
	\begin{equation}\label{e:defn.DFEM}
		\int_{X}\varphi\,d\mu_{\langle u \rangle}^{\#} = \mathcal{E}^{\#}(u, u\varphi) - \frac{1}{2}\mathcal{E}^{\#}(u^{2},\varphi) \quad \text{for any $\varphi \in \mathcal{F} \cap \contfunc_{c}(X)$.}
	\end{equation}
	Moreover, for any Borel measurable function $\varphi \colon X \to [0,\infty)$ with $\norm{\varphi}_{\sup} < \infty$, any $u \in \mathcal{F}_{e}$ and any approximating sequence $\{ u_{n} \}_{n \in \mathbb{N}} \subseteq \mathcal{F} \cap \contfunc_{c}(X)$ for $u$, $\bigl\{ \int_{X}\varphi\,d\mu_{\langle u_{n} \rangle}^{\#} \bigr\}_{n \in \mathbb{N}}$ is a Cauchy sequence in $[0,\infty)$,  $\lim_{n \to \infty}\int_{X}\varphi\,d\mu_{\langle u_{n} \rangle}^{\#}$ does not depend on the choice of $\{ u_{n} \}_{n}$, and $\int_{X}\varphi\,d\mu_{\langle u \rangle}^{\#} = \lim_{n \to \infty}\int_{X}\varphi\,d\mu_{\langle u_{n} \rangle}^{\#}$, where $\mu_{\langle u_{n} \rangle}^{\#}$ is the Radon measure on $X$ defined by $\mu_{\langle u \rangle}^{\#}(A) \coloneqq \lim_{n \to \infty}\mu_{\langle u_{n} \rangle}^{\#}(A)$ for $A \in \mathcal{B}(X)$. 
\end{prop}
\begin{defn}[Energy measures]
	Let $u \in \mathcal{F}_{e}$. 
	Let $\mu_{\langle u \rangle}$ denote the measure in the above proposition in the case $\mathcal{E}^{\#} = \mathcal{E}$. 
	We call $\mu_{\langle u \rangle}$ the \emph{energy measure} of $u$. 
	For each $w \in \{ c,j,k \}$, let $\mu_{\langle u \rangle}^{w}$ denote the measure in the above proposition in the case $\mathcal{E}^{\#} = \mathcal{E}^{(w)}$.
	For $u,v \in \mathcal{F}_{e}$, we also define $\mu_{\langle u,v \rangle}^{\#} \coloneqq \frac{1}{4}\bigl(\mu_{\langle u + v \rangle}^{\#} - \mu_{\langle u - v \rangle}^{\#}\bigr)$, where $\mu_{\langle \,\cdot\, \rangle}^{\#} \in \bigl\{ \mu_{\langle \,\cdot\, \rangle},\mu_{\langle \,\cdot\, \rangle}^{c},\mu_{\langle \,\cdot\, \rangle}^{j},\mu_{\langle \,\cdot\, \rangle}^{k} \bigr\}$.
\end{defn} 

The following lemma gives a Fatou-type property for energy measures. 
\begin{lem}\label{lem.DF.fatoulike}
	Let $\varphi \colon X \to [0,\infty)$ be a Borel measurable function with $\norm{\varphi}_{\sup} < \infty$ and let $\mu_{\langle \,\cdot\, \rangle}^{\#} \in \bigl\{ \mu_{\langle \,\cdot\, \rangle},\mu_{\langle \,\cdot\, \rangle}^{c},\mu_{\langle \,\cdot\, \rangle}^{j},\mu_{\langle \,\cdot\, \rangle}^{k} \bigr\}$. 
	If $\{ u_{n} \}_{n \in \mathbb{N}} \subseteq \mathcal{F}$ and $u \in \mathcal{F}_{e}$ satisfy $\lim_{n \to \infty}u_{n} = u$ $m$-a.e.\ and $\sup_{n \in \mathbb{N}}\mathcal{E}(u_{n},u_{n}) < \infty$, then 
	\begin{equation}\label{e:DF.fatoulike}
		\int_{X}\varphi\,d\mu_{\langle u \rangle}^{\#} \le \liminf_{n \to \infty}\int_{X}\varphi\,d\mu_{\langle u_{n} \rangle}^{\#}. 
	\end{equation}
\end{lem}
\begin{proof}
	By extracting a subsequence of $\{ u_{n} \}_{n}$ if necessary, we can assume that the limit $\lim_{n \to \infty}\int_{X}\varphi\,d\mu_{\langle u_{n} \rangle}^{\#}$ exists. 
	By using a version \cite[Theorem A.4.1-(ii)]{CF} of the Banach--Saks theorem, we can find a subsequence $\{ u_{n_{k}} \}_{k \in \mathbb{N}}$ such that $\{ v_{l} \}_{l \in \mathbb{N}} \subseteq \mathcal{F}$ defined by $v_{l} \coloneqq l^{-1}\sum_{k = 1}^{l}u_{n_{k}}$ satisfies $\lim_{k \wedge l \to \infty}\mathcal{E}(v_{k} - v_{l},v_{k} - v_{l}) = 0$. 
	Noting that $\lim_{l \to \infty}v_{l} = u$ $m$-a.e.\ and using Proposition \ref{prop.Fe-well}, we have $\lim_{l \to \infty}\mathcal{E}(u - v_{l},u - v_{l}) = 0$. 
	Hence $\lim_{l \to \infty}\int_{X}\varphi\,d\mu_{\langle v_{l} \rangle}^{\#} = \int_{X}\varphi\,d\mu_{\langle u \rangle}^{\#}$ by Proposition \ref{prop.DFEM}. 
	By the triangle inequality for $\left(\int_{X}\varphi\,d\mu_{\langle \,\cdot\, \rangle}^{\#}\right)^{1/2}$, 
    \[
    \left(\int_{X}\varphi\,d\mu_{\langle v_{l} \rangle}^{\#}\right)^{1/2} 
    \le \frac{1}{l}\sum_{k = 1}^{l}\left(\int_{X}\varphi\,d\mu_{\langle u_{n_k} \rangle}^{\#}\right)^{1/2},  
    \]
    which implies \eqref{e:DF.fatoulike} by letting $l \to \infty$.
\end{proof}

Now we can show that the integrals of non-negative bounded Borel measurable functions with respect to energy measures give $2$-energy forms satisfying \hyperref[GC]{\textup{(GC)$_{2}$}}. 
\begin{prop}\label{prop.GC.DF.emintext}
	Let $\varphi \colon X \to [0,\infty)$ be a Borel measurable function with $\norm{\varphi}_{\sup} < \infty$ and let $\mu_{\langle \,\cdot\, \rangle}^{\#} \in \bigl\{ \mu_{\langle \,\cdot\, \rangle},\mu_{\langle \,\cdot\, \rangle}^{c},\mu_{\langle \,\cdot\, \rangle}^{j},\mu_{\langle \,\cdot\, \rangle}^{k} \bigr\}$. 
	Then $(\int_{X}\varphi\,d\mu_{\langle \,\cdot\, \rangle}^{\#},\mathcal{F}_{e})$ is a $2$-energy form on $(X,m)$ satisfying \hyperref[GC]{\textup{(GC)$_{2}$}}.
\end{prop}
\begin{proof}
	Let $n_{1},n_{2} \in \mathbb{N}$, $q_{1} \in (0,2]$, $q_{2} \in [2,\infty]$ and $T = (T_{1},\dots,T_{n_{2}}) \colon \mathbb{R}^{n_{1}} \to \mathbb{R}^{n_{2}}$ satisfy \eqref{GC-cond}.
	It suffices to prove that for any $\bm{u} = (u_{1},\dots,u_{n_{1}}) \in (\mathcal{F} \cap \contfunc_{c}(X))^{n_{1}}$ and any $\varphi \in \mathcal{F} \cap \contfunc_{c}(X)$, 
	\begin{equation}\label{e:GC.DFEM.funcwise}
		\norm{\left(\left(\int_{X}\varphi\,d\mu_{\langle T_{l}(\bm{u}) \rangle}^{\#}\right)^{1/2}\right)_{l = 1}^{n_{2}}}_{\ell^{q_{2}}} 
		\le \norm{\left(\left(\int_{X}\varphi\,d\mu_{\langle u_{k} \rangle}^{\#}\right)^{1/2}\right)_{k = 1}^{n_{1}}}_{\ell^{q_{1}}}. 
	\end{equation}
	Indeed, we can extend \eqref{e:GC.DFEM.funcwise} to any $\bm{u} \in \mathcal{F}_{e}^{n_{1}}$ and any Borel measurable function $\varphi \colon X \to [0,\infty]$ as follows. 
    Let us start with the case of $\varphi = \indicator{A}$, where $A \in \mathcal{B}(X)$. 
    By \cite[Theorem 2.18]{Rud}, there exist sequences $\{ K_{n} \}_{n \in \mathbb{N}}$ and $\{ U_{n} \}_{n \in \mathbb{N}}$ such that $K_{n} \subseteq A \subseteq U_{n}$, $K_{n}$ is compact, $U_{n}$ is open and $\lim_{n \to \infty}\max_{v \in \{ T_{l}(\bm{u}) \}_{l} \cup \{ u_{k} \}_{k}}\mu_{\langle v \rangle}^{\#}(U_{n} \setminus K_{n}) = 0$. 
    By Urysohn's lemma, we can pick $\varphi_{n} \in \contfunc_{c}(X)$ so that $0 \le \varphi_{n} \le 1$, $\varphi_{n}|_{K_{n}} = 1$ and $\supp_{X}[\varphi_{n}] \subseteq U_{n}$. 
    By \eqref{e:GC.DFEM.funcwise} with $\varphi_{n}$ in place of $\varphi$, we obtain $\norm{\left(\mu_{\langle T_{l}(\bm{u}) \rangle}^{\#}(K_{n})^{1/2}\right)_{l = 1}^{n_{2}}}_{\ell^{q_{2}}} 
    \le \norm{\left(\mu_{\langle u_{k} \rangle}^{\#}(U_{n})^{1/2}\right)_{k = 1}^{n_{1}}}_{\ell^{q_{1}}}$. 
    By letting $n \to \infty$, we get \eqref{e:GC.DFEM.funcwise} with $\varphi = \indicator{A}$, i.e., 
	\begin{equation}\label{e:GC.DFEM.setwise}
		\norm{\bigl(\mu_{\langle T_{l}(\bm{u}) \rangle}^{\#}(A)^{1/2}\bigr)_{l = 1}^{n_{2}}}_{\ell^{q_{2}}} 
		\le \norm{\bigl(\mu_{\langle u_{k} \rangle}^{\#}(A)^{1/2}\bigr)_{k = 1}^{n_{1}}}_{\ell^{q_{1}}}. 
	\end{equation}
	By the reverse Minkowski inequality on $\ell^{q_{1}/2}$ and the Minkowski inequality on $\ell^{q_{2}/2}$ (see also \eqref{GC.sum}), we can extend \eqref{e:GC.DFEM.setwise} to \eqref{e:GC.DFEM.funcwise} for any non-negative Borel measurable simple function $\varphi$ on $X$,
	By the monotone convergence theorem, \eqref{e:GC.DFEM.funcwise} holds for any Borel measurable function $\varphi \colon X \to [0,\infty]$. 
	Next we will extend \eqref{e:GC.DFEM.funcwise} to $\bm{u} = (u_{1},\dots,u_{n_{1}}) \in \mathcal{F}_{e}^{n_{1}}$. 
	Since $\mathcal{F} \cap \contfunc_{c}(X)$ is dense in $(\mathcal{F}, \norm{\,\cdot\,}_{\mathcal{E},1})$, there exists an approximating sequence $\{ u_{k,n} \}_{n \in \mathbb{N}} \subseteq \mathcal{F} \cap \contfunc_{c}(X)$ for $u_{k}$ for each $k \in \{ 1,\dots,n_{1} \}$. 
	Set $\bm{u}_{n} \coloneqq (u_{1,n},\dots,u_{n_{1},n})$. 
	Then, for each $l \in \{ 1,\dots,n_{2} \}$, $\lim_{n \to \infty}T_{l}(\bm{u}_{n}) = T_{l}(\bm{u})$ $m$-a.e., $T_{l}(\bm{u}_{n}) \in \mathcal{F}$ and $\sup_{n \in \mathbb{N}}\mathcal{E}(T_{l}(\bm{u}_{n}),T_{l}(\bm{u}_{n})) < \infty$ by Proposition \ref{prop.GC.DF}. 
	Hence $T_{l}(\bm{u}) \in \mathcal{F}_{e}$ by Proposition \ref{prop.extcriteria}, and 
	\begin{align*}
		\norm{\left(\left(\int_{X}\varphi\,d\mu_{\langle T_{l}(\bm{u}) \rangle}^{\#}\right)^{1/2}\right)_{l = 1}^{n_{2}}}_{\ell^{q_{2}}}
		&\le \norm{\left(\left(\liminf_{n \to \infty}\int_{X}\varphi\,d\mu_{\langle T_{l}(\bm{u}_{n}) \rangle}^{\#}\right)^{1/2}\right)_{l = 1}^{n_{2}}}_{\ell^{q_{2}}} \\
		&\le \liminf_{n \to \infty}\norm{\left(\left(\int_{X}\varphi\,d\mu_{\langle T_{l}(\bm{u}_{n}) \rangle}^{\#}\right)^{1/2}\right)_{l = 1}^{n_{2}}}_{\ell^{q_{2}}} \\
		&\overset{\eqref{e:GC.DFEM.funcwise}}{\le} \liminf_{n \to \infty}\norm{\left(\left(\int_{X}\varphi\,d\mu_{\langle u_{k,n} \rangle}^{\#}\right)^{1/2}\right)_{k = 1}^{n_{1}}}_{\ell^{q_{1}}} \\
		&= \norm{\left(\left(\int_{X}\varphi\,d\mu_{\langle u_{k} \rangle}^{\#}\right)^{1/2}\right)_{k = 1}^{n_{1}}}_{\ell^{q_{1}}}, 
	\end{align*}
	where we used Lemma \ref{lem.DF.fatoulike} in the first inequality and Proposition \ref{prop.DFEM} in the last equality. 
	This implies that $(\int_{X}\varphi\,d\mu_{\langle \,\cdot\, \rangle}^{\#},\mathcal{F}_{e})$ is a $2$-energy form on $(X,m)$ satisfying \hyperref[GC]{\textup{(GC)$_{2}$}}. 
	
	Let us go back to the proof of \eqref{e:GC.DFEM.funcwise} in the case where $\bm{u} = (u_{1},\dots,u_{n_{1}}) \in (\mathcal{F} \cap \contfunc_{c}(X))^{n_{1}}$ and $\varphi \in \mathcal{F} \cap \contfunc_{c}(X)$. 
	Fix a metric $d$ on $X$ which is compatible with the given topology of $X$, an increasing sequence of relatively open sets $\{ G_{l} \}_{l \in \mathbb{N}}$ with $\bigcup_{l \in \mathbb{N}}G_{l} = X$ and a sequence of positive numbers $\{ \delta_{l} \}_{l \in \mathbb{N}}$ with $\delta_{l} \downarrow 0$ as $l \to \infty$. 
	Then there exist a sequence of positive numbers $\{ \beta_{n} \}_{n \in \mathbb{N}}$ with $\beta_{n} \uparrow \infty$ as $n \to \infty$, a family of Radon measures $\{ \sigma_{\beta} \}_{\beta > 0}$ on $X \times X$ and a family of Radon measures $\{ m_{\beta} \}_{\beta > 0}$ on $X$ with $m_{\beta} \ll m$ such that for any $v \in \mathcal{F} \cap \contfunc_{c}(X)$,
	\begin{equation}\label{e:DF.emfunc.total}
		\int_{X}\varphi\,d\mu_{\langle v \rangle} 
		= \lim_{\beta \to \infty}\left(\frac{\beta}{2}\int_{X \times X}\abs{v(x) - v(y)}^{2}\varphi(x)\,\sigma_{\beta}(dx,dy) + \frac{\beta}{2}\int_{X}\abs{v(x)}^{2}\varphi(x)\,m_{\beta}(dx)\right), 
	\end{equation}
	and 
	\begin{equation}\label{e:DF.emfunc.local}
		\int_{X}\varphi\,d\mu_{\langle v \rangle}^{c} 
		= \lim_{l \to \infty}\lim_{n \to \infty}\frac{\beta_{n}}{2}\int_{\{(x,y) \in G_{l} \times G_{l} \mid d(x,y) < \delta_{l} \}}\abs{v(x) - v(y)}^{2}\varphi(x)\,\sigma_{\beta_{n}}(dx,dy). 
	\end{equation}
	See \cite[the equations just before (3.2.13) and (3.2.19)]{FOT} for details. 
	Note that $T_{l}(\bm{u}) \in \mathcal{F} \cap \contfunc_{c}(X)$ for each $l \in \{ 1,\dots,n_{2} \}$ by Proposition \ref{prop.GC.DF} and $T_{l}(0) = 0$. 
	If $q_{2} < \infty$, then we have from \eqref{e:DF.emfunc.total} that 
	\begin{align*}
        &\sum_{l = 1}^{n_{2}}\left(\int_{X}\varphi\,d\mu_{\langle T_{l}(\bm{u}) \rangle}\right)^{q_{2}/2} \\
        &\overset{\eqref{reverse}}{\le} \lim_{\beta \to \infty}\biggl(\frac{\beta}{2}\int_{X \times X}\norm{T(\bm{u}(x)) - T(\bm{u}(y))}_{\ell^{q_{2}}}^{2}\varphi(x)\,\sigma_{\beta}(dx,dy) \nonumber \\
        &\hspace*{200pt}+ \frac{\beta}{2}\int_{X}\norm{T(\bm{u}(x))}_{\ell^{q_{2}}}^{2}\varphi(x)\,m_{\beta}(dx)\biggr)^{q_{2}/2} \nonumber\\ 
        &\overset{\eqref{GC-cond}}{\le} \lim_{\beta \to \infty}\left(\frac{\beta}{2}\int_{X \times X}\norm{\bm{u}(x) - \bm{u}(y)}_{\ell^{q_{1}}}^{2}\varphi(x)\,\sigma_{\beta}(dx,dy) + \frac{\beta}{2}\int_{X}\norm{\bm{u}(x)}_{\ell^{q_{1}}}^{2}\varphi(x)\,m_{\beta}(dx)\right)^{q_{2}/2} \nonumber\\
        &\overset{\textup{($\ast$)}}{\le} \lim_{\beta \to \infty}\Biggl(\sum_{k = 1}^{n_{1}}\Biggl[\frac{\beta}{2}\int_{X \times X}\abs{u_{k}(x) - u_{k}(y)}^{2}\varphi(x)\,\sigma_{\beta}(dx,dy) \nonumber \\
        &\hspace*{200pt}+ \frac{\beta}{2}\int_{X}\abs{u_{k}(x)}^{2}\varphi(x)\,m_{\beta}(dx)\Biggr]^{q_{1}/2}\Biggr)^{\frac{2}{q_{1}} \cdot \frac{q_{2}}{2}} \nonumber\\
        &= \left(\sum_{k = 1}^{n_{1}}\left(\int_{X}\varphi\,d\mu_{\langle u_{k} \rangle}\right)^{q_{1}/2}\right)^{q_{2}/q_{1}}, 
    \end{align*}
    where we used the triangle inequality for a suitable $L^{2/q_{1}}$-norm on $(X \times X) \sqcup X$ in ($\ast$) (here $\sqcup$ denotes the disjoint union). 
    The case of $q_{2} = \infty$ is similar, so we obtain the desired estimate \eqref{e:GC.DFEM.funcwise} for $\mu_{\langle \,\cdot\, \rangle}^{\#} = \mu_{\langle \,\cdot\, \rangle}$. 
    The other case $\mu_{\langle \,\cdot\, \rangle}^{\#} \in \{ \mu_{\langle \,\cdot\, \rangle}^{c}, \mu_{\langle \,\cdot\, \rangle}^{j}, \mu_{\langle \,\cdot\, \rangle}^{k} \}$ can be shown in a similar way by virtue of the expression in \cite[(3.2.23)]{FOT}. 
\end{proof}

Next we see that ``$\abs{\nabla u}$'' also satisfies a similar contraction property. 
To present the precise definition of the density, we recall the notion of minimal energy-dominant measure.  
\begin{defn}[Minimal energy-dominant measure; {\cite[Definition 2.1]{Hin10}}]
	A $\sigma$-finite Borel measure $\mu$ on $X$ is called a \emph{minimal energy-dominant measure}\index{minimal energy-dominant measure} of $(\mathcal{E},\mathcal{F})$ if and only if the following two conditions hold.
	\begin{enumerate}[label=\textup{(\roman*)},align=left,leftmargin=*,topsep=2pt,parsep=0pt,itemsep=2pt]
		\item\label{it:MED.dom} For any $f \in \mathcal{F}$, we have $\mu_{\langle f \rangle} \ll \mu$.
		\item\label{it:MED.min} If another $\sigma$-finite Borel measure $\mu'$ on $X$ satisfies \ref{it:MED.dom} with $\mu$ in place of $\mu'$, then $\mu \ll \mu'$. 
	\end{enumerate}
\end{defn}

The existence of a minimal energy-dominant measure is proved in \cite[Lemma 2.2]{Nak85} (see also \cite[Lemma 2.3]{Hin10}). 
For any minimal energy-dominant measure $\mu$ of $(\mathcal{E},\mathcal{F})$, the same argument as in \cite[Proof of Lemma 2.2]{Hin10} implies that $\mu_{\langle f \rangle} \ll \mu$ for any $f \in \mathcal{F}_{e}$. 
In addition, for $\mu_{\langle \,\cdot\, \rangle}^{\#} \in \{ \mu_{\langle \,\cdot\, \rangle}, \mu_{\langle \,\cdot\, \rangle}^{c}, \mu_{\langle \,\cdot\, \rangle}^{j}, \mu_{\langle \,\cdot\, \rangle}^{k} \}$, we easily see that $\mu_{\langle f,g \rangle}^{\#} \ll \mu$ for any $f,g \in \mathcal{F}_{e}$. 
We define $\Gamma_{\mu}^{\#}(u,v) \coloneqq \frac{d\mu^{\#}_{\langle u,v \rangle}}{d\mu}$ and $\Gamma_{\mu}^{\#}(u) \coloneqq \Gamma_{\mu}^{\#}(u,u)$ for $u,v \in \mathcal{F}_{e}$. 

\begin{prop}\label{prop.GCkuwae}
	Let $\mu$ be a minimal energy-dominant measure of $(\mathcal{E},\mathcal{F})$ and for each $f \in \mathcal{F}_{e}$, let $\Gamma_{\mu}(f) \coloneqq d\mu_{\langle f \rangle}/d\mu$ and $\Gamma_{\mu}^{w}(f) \coloneqq d\mu_{\langle f \rangle}^{w}/d\mu$ for each $w \in \{ c,j,k \}$. 
	Let $\Gamma_{\mu}^{\#}(\,\cdot\,) \in \{ \Gamma_{\mu}(\,\cdot\,), \Gamma_{\mu}^{c}(\,\cdot\,), \Gamma_{\mu}^{j}(\,\cdot\,), \Gamma_{\mu}^{k}(\,\cdot\,) \}$. 
	Then for any $n_{1},n_{2} \in \mathbb{N}$, $q_{1} \in (0,2]$, $q_{2} \in [2,\infty]$ and $T = (T_{1},\dots,T_{n_{2}}) \colon \mathbb{R}^{n_{1}} \to \mathbb{R}^{n_{2}}$ satisfying \eqref{GC-cond} and any $\bm{u} = (u_{1},\dots,u_{n_{1}}) \in \mathcal{F}_{e}^{n_{1}}$, 
	\begin{equation}\label{e:GC.DF.density}
		\norm{\bigl(\Gamma_{\mu}^{\#}(T_{l}(\bm{u}))(x)^{1/2}\bigr)_{l = 1}^{n_{2}}}_{\ell^{q_{2}}} 
		\le \norm{\bigl(\Gamma_{\mu}^{\#}(u_{k})(x)^{1/2}\bigr)_{k = 1}^{n_{1}}}_{\ell^{q_{1}}} \quad \text{for $\mu$-a.e.\ $x \in X$,}
	\end{equation}
	and in particular, for any $p \in [q_{1}, q_{2}] \cap (0,\infty)$ and any Borel measurable function $\varphi \colon X \to [0,\infty]$,
	\begin{equation}\label{e:GC2.DF.int}
		\norm{\Biggl(\biggl(\int_{X} \varphi \Gamma_{\mu}^{\#}(T_{l}(\bm{u}))^{\frac{p}{2}}\,d\mu\biggr)^{1/p}\Biggr)_{l = 1}^{n_{2}}}_{\ell^{q_{2}}} 
		\le \norm{\Biggl(\biggl(\int_{X} \varphi \Gamma_{\mu}^{\#}(u_{k})^{\frac{p}{2}}\,d\mu\biggr)^{1/p}\Biggr)_{k = 1}^{n_{1}}}_{\ell^{q_{1}}}. 
	\end{equation}
\end{prop}
\begin{proof}
	We first construct a good $\mu$-version of $\Gamma_{\mu}^{\#}(v)$ for each $v \in \mathcal{F}_{e}$. 
	Fix $\{ X_{n} \}_{n \in \mathbb{N}} \subseteq \mathcal{B}(X)$ such that $X_{n} \subseteq X_{n + 1}$, $X = \bigcup_{n \in \mathbb{N}}X_{n}$ and $\mu(X_{n}) \in (0,\infty)$ for each $n \in \mathbb{N}$. 
	Let $\{ A_{k} \}_{k \in \mathbb{N}}$ be a countable open base for the topology of $X$. 
	Set $A_{k}^{0} \coloneqq X \setminus A_{k}$ and $A_{k}^{1} \coloneqq A_{k}$ for each $k \in \mathbb{N}$, and define a non-decreasing sequence $\{ \mathcal{A}_{k} \}_{k \in \mathbb{N}}$ of $\sigma$-algebras in $X$ in the same way as \eqref{e:defn.cylinder}.  
	For $v \in \mathcal{F}_{e}$, $n,k \in \mathbb{N}$, $\alpha \in \{ 0,1 \}^{k}$, define $\Gamma_{\mu}^{\#}(v)_{n,k} \colon X \to [0,\infty)$ by, for $x \in A_{k}^{\alpha}$, 
	\begin{equation}\label{e:defn.goodversion}
		\Gamma_{\mu}^{\#}(v)_{n,k}(x) 
		\coloneqq 
		\begin{cases}
			\mu(A_{k}^{\alpha} \cap X_{n})^{-1}\mu_{\langle v \rangle}^{\#}(A_{k}^{\alpha} \cap X_{n}) \quad &\text{if $\mu(A_{k}^{\alpha} \cap X_{n}) > 0$,} \\
			0 \quad &\text{if $\mu(A_{k}^{\alpha} \cap X_{n}) = 0$.}
		\end{cases}
	\end{equation}
	We also set $\mu_{n} \coloneqq \mu(X_{n})^{-1}\mu((\cdot) \cap X_{n})$ and $v_{n}^{\#} \coloneqq \frac{d\mu_{\langle v \rangle}^{\#}((\cdot) \cap X_{n})}{\mu((\cdot) \cap X_{n})}$. 
	Then we easily see that $\mathbb{E}_{\mu_{n}}[v_{n}^{\#}\,|\,\mathcal{A}_{k}] = \Gamma_{\mu}^{\#}(v)_{n,k}$ $\mu$-a.e.\ on $X_{n}$ and hence $\lim_{k \to \infty}\Gamma_{\mu}^{\#}(v)_{n,k} = v_{n}^{\#}$ $\mu$-a.e.\ on $X_{n}$ by the martingale convergence theorem (see, e.g., \cite[Theorem 10.5.1]{Dud}) and the fact that $\bigcup_{k \in \mathbb{N}}\mathcal{A}_{k}$ generates $\mathcal{B}(X)$.  
	Now we define $\widetilde{\Gamma}_{\mu}^{\#}(v) \colon X \to [0,\infty)$ by $\widetilde{\Gamma}_{\mu}^{\#}(v)(x) \coloneqq v_{n}^{\#}(x)$ for $n \in \mathbb{N}$ and $x \in X_{n} \setminus X_{n - 1}$, where $X_{0} \coloneqq \emptyset$. 
	Then $\widetilde{\Gamma}_{\mu}^{\#}(v) = \Gamma_{\mu}^{\#}(v)$ $\mu$-a.e.\ on $X$. 
	
	Next we show \eqref{e:GC.DF.density}. 
	Let $n_{1},n_{2} \in \mathbb{N}$, $q_{1} \in (0,2]$, $q_{2} \in [2,\infty]$, $\bm{u} = (u_{1},\dots,u_{n_{1}}) \in \mathcal{F}_{e}^{n_{1}}$ and let $T = (T_{1},\dots,T_{n_{2}}) \colon \mathbb{R}^{n_{1}} \to \mathbb{R}^{n_{2}}$ satisfy \eqref{GC-cond} with $2$ in place of $p$. 
	By Proposition \ref{prop.GC.DF.emintext} and \eqref{e:defn.goodversion}, for any $n,m \in \mathbb{N}$ and any $x \in X$, 
	\[
	\norm{\bigl(\Gamma_{\mu}^{\#}(T_{l}(\bm{u}))_{n,m}(x)^{1/2}\bigr)_{l = 1}^{n_{2}}}_{\ell^{q_{2}}} 
		\le \norm{\bigl(\Gamma_{\mu}^{\#}(u_{k})_{n,m}(x)^{1/2}\bigr)_{k = 1}^{n_{1}}}_{\ell^{q_{1}}}. 
	\]
	By letting $m \to \infty$, we obtain 
	\[
	\norm{\Bigl(\widetilde{\Gamma}_{\mu}^{\#}(T_{l}(\bm{u}))(x)^{1/2}\Bigr)_{l = 1}^{n_{2}}}_{\ell^{q_{2}}} 
		\le \norm{\Bigl(\widetilde{\Gamma}_{\mu}^{\#}(u_{k})(x)^{1/2}\Bigr)_{k = 1}^{n_{1}}}_{\ell^{q_{1}}} \quad \text{for $\mu$-a.e.\ $x \in X$,} 
	\]
	whence \eqref{e:GC.DF.density} holds. 
	Lastly, if $p \in [q_{1},q_{2}] \cap (0,\infty)$ and $q_{2} < \infty$, then we see that for any Borel measurable function $\varphi \colon X \to [0,\infty]$,  
	\begin{align}\label{e:GC2forint}
		\sum_{l = 1}^{n_{2}}\biggl(\int_{X} \varphi \Gamma_{\mu}^{\#}(T_{l}(\bm{u}))^{\frac{p}{2}}\,d\mu\biggr)^{q_{2}/p}
		&\overset{\eqref{reverse}}{\le} \biggr(\int_{X} \varphi \norm{\bigl(\Gamma_{\mu}^{\#}(T_{l}(\bm{u}))(x)^{1/2}\bigr)_{l = 1}^{n_{2}}}_{\ell^{q_{2}}}^{p}\,\mu(dx)\biggr)^{q_{2}/p} \nonumber \\
		&\overset{\eqref{e:GC.DF.density}}{\le} \biggr(\int_{X} \varphi \norm{\bigl(\Gamma_{\mu}^{\#}(u_{k})(x)^{1/2}\bigr)_{k = 1}^{n_{1}}}_{\ell^{q_{1}}}^{p}\,\mu(dx)\biggr)^{q_{2}/p} \nonumber \\
		&\overset{\textup{($\ast$)}}{\le} \Biggl(\sum_{k = 1}^{n_{1}}\biggl(\int_{X} \varphi \Gamma_{\mu}^{\#}(u_{k})^{\frac{p}{2}}\,d\mu\biggr)^{q_{1}/p}\Biggr)^{q_{2}/q_{1}}, 
	\end{align}
	where we used the triangle inequality for the norm of $L^{p/q_{1}}(X,\varphi\,d\mu)$ in ($\ast$). 
	The case of $q_{2} = \infty$ is similar, so we obtain \eqref{e:GC2.DF.int}. 
\end{proof}

If $(\mathcal{E},\mathcal{F})$ is strongly local, then under a certain set of assumptions we can show \hyperlink{GC-em}{\textup{(GC)$^{\textup{EM}}_{p}$}} for $\Gamma_{\mu}(\,\cdot\,)^{\frac{p}{2}}\,d\mu$ on a suitable domain. 
To prove it, we need some preparations. 
The following proposition is the standard Minkowski integral inequality\index{Minkowski integral inequality} (see, e.g., \cite[Appendix B5]{DF}).
\begin{prop}\label{prop.intMin}
	Let $(X_{i},\mathcal{B}_{i},m_{i})$ be a $\sigma$-finite measure space for each $i \in \{ 1,2 \}$. 
	Let $q \in [1,\infty)$ and let $f \colon X_{1} \times X_{2} \to [-\infty,\infty]$ be measurable with respect to the product $\sigma$-algebra of $\mathcal{B}_{1}$ and $\mathcal{B}_{2}$. 
	Then 
	\begin{align}\label{e:intMin}
		\biggl(\int_{X_{1}}\biggl(\int_{X_{2}}\abs{f(x_{1},x_{2})}\,m_{2}(dx_{2})\biggr)^{q}\,m_{1}(dx_{1})\biggr)^{\frac{1}{q}} 
		\le \int_{X_{2}}\biggl(\int_{X_{1}}\abs{f(x_{1},x_{2})}^{q}\,m_{1}(dx_{1})\biggr)^{\frac{1}{q}}m_{2}(dx_{2}). 
	\end{align}
\end{prop}
 
Next we show a tensor-type inequality for non-negative definite symmetric bilinear forms. 
\begin{prop}\label{prop.tensorineq}
	Let $V$ be a finite-dimensional vector space over $\mathbb{R}$, let $E \colon V \times V \to \mathbb{R}$ be a non-negative definite symmetric bilinear form, let $n_{1},n_{2} \in \mathbb{N}$ and let $A = (A_{lk})_{1 \le l \le n_{2}, 1 \le k \le n_{1}}$ be a real matrix. 
	If $(u_{1},\dots,u_{n_{1}}) \in V^{n_{1}}$, $q_{1} \in (0,\infty)$, $q_{2} \in (0,\infty]$ and $q_{1} \le q_{2}$, then 
	\begin{equation}\label{e:tensorineq}
		\norm{\Biggl(E\biggl(\sum_{k = 1}^{n_{1}}A_{lk}u_{k}\biggr)^{1/2}\Biggr)_{l =1}^{n_{2}}}_{\ell^{q_{2}}} 
		\le \norm{A}_{\ell^{q_{1}}_{n_{1}} \to \ell^{q_{2}}_{n_{2}}}\norm{\bigl(E(u_{k})^{1/2}\bigr)_{k = 1}^{n_{1}}}_{\ell^{q_{1}}},  
	\end{equation}
	where we set $E(u) \coloneqq E(u,u)$ for $u \in V$ and $\norm{A}_{\ell^{q_{1}}_{n_{1}} \to \ell^{q_{2}}_{n_{2}}} \coloneqq \sup_{x \in \mathbb{R}^{n_{1}},\, \norm{x}_{\ell^{q_{1}}} \leq 1} \norm{Ax}_{\ell^{q_{2}}}$. 
\end{prop}
\begin{proof}
	The desired inequality follows from a Beckner-like result in \cite[7.9.]{DF}\ (see also \cite[Lemma 2]{Bec75}).
	We present a complete proof for convenience. 
	Let $\gamma_{n}$ be the Gaussian measure on $\mathbb{R}^{n}$, i.e., $\gamma_{n}(dx) \coloneqq (2\pi)^{-n/2}\exp{\bigl(-\abs{x}^{2}/2\bigr)}\,dx$, for each $n \in \mathbb{N}$ and set $n \coloneqq \dim(V/E^{-1}(0)) \in \mathbb{N} \cup \{ 0 \}$. 
	If $n = 0$, i.e., $E(u) = 0$ for any $u \in V$, then \eqref{e:tensorineq} is clear. 
	Hence we assume that $n \ge 1$ in the rest of the proof. 
	Let $\pi_{j} \colon \mathbb{R}^{n} \to \mathbb{R}$ be the projection map to the $j$-th coordinate for each $j \in \{ 1,\dots,n \}$. 
	Then we have from \cite[Proposition in 8.7.]{DF} that for any $(\alpha_{j})_{j = 1}^{n} \in \mathbb{R}^{n}$, 
	\begin{equation}\label{e:gaussisom}
		\norm{\pi_{1}}_{L^{q_{1}}(\mathbb{R},\gamma_{1})}^{-1}\left(\int_{\mathbb{R}^{n}}\abs{\sum_{j = 1}^{n}\alpha_{j}\pi_{j}(x)}^{q_{1}}\,\gamma_{n}(dx)\right)^{1/q_{1}}
		= \norm{(\alpha_{j})_{j = 1}^{n}}_{\ell^{2}}. 
	\end{equation}
	Indeed, \eqref{e:gaussisom} is obviously true in the case of $(\alpha_{j})_{j} = (\delta_{1j})_{j}$ and this together with the invariance of $\gamma_{n}$ under $\ell^{2}_{n}$-isometries implies the desired equality \eqref{e:gaussisom}.  
	
	Let us fix a basis $\{ e_{j} \}_{j = 1}^{n} \subseteq V$ of $V$ satisfying $E(e_{j},e_{j'}) = \delta_{jj'}$ for each $j,j' \in \{ 1,\dots,n \}$, which exists by the Gram--Schmidt orthonormalization. 
	Now we define $\iota \colon V \to L^{q_{1}}(\mathbb{R}^{n}, \gamma_{n})$ by 
	\begin{equation}\label{e:gaussemb}
		\iota(u) \coloneqq \norm{\pi_{1}}_{L^{q_{1}}(\mathbb{R},\gamma_{1})}^{-1}\sum_{j = 1}^{n}E(u,e_{j})^{1/2}\pi_{j}, \quad u \in V. 
	\end{equation}
	Then $\norm{\iota(u)}_{L^{q_{1}}(\mathbb{R}^{n},\gamma_{n})} = \bigl(\sum_{j = 1}^{n}E(u,e_{j})\bigr)^{1/2} = E(u,u)^{1/2}$ by \eqref{e:gaussisom}. 
	If $q_{2} < \infty$, then 
	\begin{align*}
		\norm{\left(E\left(\sum_{k = 1}^{n_{1}}A_{lk}u_{k}\right)^{1/2}\right)_{l =1}^{n_{2}}}_{\ell^{q_{2}}} 
		&= \left(\sum_{l = 1}^{n_{2}}\left(\int_{\mathbb{R}^{n}}\abs{\sum_{k = 1}^{n_{1}}A_{lk}\iota(u_{k})}^{q_{1}}\,d\gamma_{n}\right)^{q_{2}/q_{1}}\right)^{\frac{q_{1}}{q_{2}} \cdot \frac{1}{q_{1}}} \\
		&\overset{\textup{($\ast$)}}{\le} \left(\int_{\mathbb{R}^{n}}\left(\sum_{l = 1}^{n_{2}}\abs{\sum_{k = 1}^{n_{1}}A_{lk}\iota(u_{k})}^{q_{2}}\right)^{q_{1}/q_{2}}\,d\gamma_{n}\right)^{1/q_{1}} \\
		&\le \norm{A}_{\ell^{q_{1}}_{n_{1}} \to \ell^{q_{2}}_{n_{2}}}\left(\int_{\mathbb{R}^{n}}\sum_{k = 1}^{n_{1}}\abs{\iota(u_{k})}^{q_{1}}\,d\gamma_{n}\right)^{1/q_{1}} \\
		&= \norm{A}_{\ell^{q_{1}}_{n_{1}} \to \ell^{q_{2}}_{n_{2}}}\left(\sum_{k = 1}^{n_{1}}E(u_{k})^{q_{1}/2}\right)^{1/q_{1}}, 
	\end{align*}
	where we used \eqref{e:intMin} with $q = q_{1}/q_{2}$ in \textup{($\ast$)}. 
	Since the case of $q_{2} = \infty$ is similar, so we obtain \eqref{e:tensorineq}. 
\end{proof}

Let us recall the definition of $p$-energy forms introduced by Kuwae in \cite{Kuw23+}
\begin{defn}[{\cite[Definition 1.4]{Kuw23+}}]\label{defn.kuwae-form}
	Let $\mu$ be a minimal energy-dominant measure of $(\mathcal{E},\mathcal{F})$, $p \in (1,\infty)$ and $\mathscr{D} \subseteq \{ u \in \mathcal{F} \cap L^{p}(X,m) \mid \Gamma_{\mu}(u)^{\frac{1}{2}} \in L^{p}(X,\mu) \}$ a linear subspace. 
	Assume that $(\mathcal{E},\mathcal{F})$ is strongly local and that 
	\begin{equation}\label{assum.closability}
		\begin{minipage}{325pt}
			$\lim_{n \to \infty}\int_{X}\Gamma_{\mu}(u_{n})^{\frac{p}{2}}\,d\mu = 0$ for any $\{ u_{n} \}_{n \in \mathbb{N}} \subseteq \mathscr{D}$ satisfying $\lim_{n \wedge k \to \infty}\int_{X}\Gamma_{\mu}(u_{n} - u_{k})^{\frac{p}{2}}\,d\mu = 0$ and $\lim_{n \to \infty}\norm{u_{n}}_{L^{p}(X,m)} = 0$. 
		\end{minipage}
	\end{equation}
	We define the norm $\norm{\,\cdot\,}_{H^{1,p}}$ on $\mathscr{D}$ by $\norm{u}_{H^{1,p}} \coloneqq \bigl(\norm{u}_{L^{p}(X,m)}^{p} + \int_{X}\Gamma_{\mu}(u)^{\frac{p}{2}}\,d\mu\bigr)^{1/p}$, and let $(H^{1,p}(X),\norm{\,\cdot\,}_{H^{1,p}})$ denote the completion of $(\mathscr{D},\norm{\,\cdot\,}_{H^{1,p}})$, so that we may and do consider $H^{1,p}(X)$ as a linear subspace of $L^{p}(X,m)$ since the canonical bounded linear map from $H^{1,p}(X)$ to $L^{p}(X,m)$ extending $\id_{\mathscr{D}}$ is injective by \eqref{assum.closability}. 
	Then we can uniquely extend $\Gamma_{\mu}$ to $H^{1,p}(X)$ by defining $\Gamma_{\mu}(u)^{\frac{1}{2}} \in L^{p}(X,\mu)$ for $u \in H^{1,p}(X)$ as the $L^{p}(X,\mu)$-limit of $\Gamma_{\mu}(u_{n})^{\frac{1}{2}}$, where $\{ u_{n} \}_{n \in \mathbb{N}} \subseteq \mathscr{D}$ satisfies $\lim_{n \wedge k \to \infty}\int_{X}\Gamma_{\mu}(u_{n} - u_{k})^{\frac{p}{2}}\,d\mu = 0$ and $\lim_{n \to \infty}\norm{u - u_{n}}_{L^{p}(X,m)} = 0$. 
\end{defn}
\begin{rmk} \label{rmk.kuwae-form}
	The condition \eqref{assum.closability} always holds if $p \ge 2$ and $\mu(F_{n}) < \infty$ for any $n \in \mathbb{N}$ for some $\mathcal{E}$-nest $\{F_{n}\}_{n \in \mathbb{N}}$\footnote{Namely, a non-decreasing sequence $\{F_{n}\}_{n \in \mathbb{N}}$ of closed subsets of $X$ such that $\bigcup_{n \in \mathbb{N}}\mathcal{F}_{F_{n}}$ is dense in $(\mathcal{F},\norm{\,\cdot\,}_{\mathcal{E},1})$, where $\mathcal{F}_{F_{n}} := \{ u \in \mathcal{F} \mid \text{$u=0$ $m$-a.e.\ on $X \setminus F_{n}$} \}$; see, e.g., \cite[Definition 1.2.12-(i) and Theorem 1.3.14-(ii)]{CF}.} as proved in \cite[Proposition 1.1]{Kuw23+}; 
	the latter condition\footnote{Note that a minimal energy-dominant measure $\mu$ of $(\mathcal{E},\mathcal{F})$ does not satisfy this condition in general. Indeed, consider the case where $X = \mathbb{R}$, $m$ is the Lebesgue measure on $\mathbb{R}$ and $(\mathcal{E},\mathcal{F})$ is the Dirichlet form of the Brownian motion on $\mathbb{R}$, and let $\mu$ be a Borel measure on $\mathbb{R}$. Then it is easy to see from \cite[Theorem 9.9]{Kig12} that $\mu$ satisfies the condition in Remark \ref{rmk.kuwae-form} if and only if $\mu$ is a Radon measure on $\mathbb{R}$. On the other hand, since $\mathcal{F} = W^{1,2}(\mathbb{R})$ and $d\mu_{\langle u \rangle} = \abs{u'}^{2}\,dm$ for any $u \in W^{1,2}(\mathbb{R})$, it is clear that $\mu$ is a minimal energy-dominant measure of $(\mathcal{E},\mathcal{F})$ if and only if $\mu$ is $\sigma$-finite and satisfies $\mu \ll m$ and $m \ll \mu$. Of course, the latter class of $\mu$ contains plenty of measures which are not Radon measures on $\mathbb{R}$ and thereby are minimal energy-dominant measures of $(\mathcal{E},\mathcal{F})$ failing to satisfy the condition in Remark \ref{rmk.kuwae-form}.} on $\mu$ is not assumed there, but is necessary for \cite[Proof of Proposition 1.1]{Kuw23+} to make sense. 
\end{rmk}

Now we can conclude that the family $\bigl\{\Gamma_{\mu}(u)^{\frac{p}{2}}\,d\mu\bigr\}_{u \in H^{1,p}(X)}$ of $p$-energy measures on $(X,\mathcal{B}(X))$ dominated by $\bigl(\int_{X}\Gamma_{\mu}(\,\cdot\,)^{\frac{p}{2}}\,d\mu,H^{1,p}(X)\bigr)$ satisfies \hyperlink{GC-em}{\textup{(GC)$^{\textup{EM}}_{p}$}}.
More precisely, we obtain the following theorem.
\begin{thm}\label{thm.GCp-kuwae}
	Let $\mu$ be a minimal energy-dominant measure of $(\mathcal{E},\mathcal{F})$, $p \in (1,\infty)$ and $\mathscr{D} \subseteq \{ u \in \mathcal{F} \cap L^{p}(X,m) \mid \Gamma_{\mu}(u)^{\frac{1}{2}} \in L^{p}(X,\mu) \}$ a linear subspace. 
	Assume that $(\mathcal{E},\mathcal{F})$ is strongly local and that \eqref{assum.closability} holds. 
	In addition, we assume that 
	\begin{equation}\label{assum.corecontraction}
		\begin{minipage}{275pt}
			$\widehat{T}(u) \in \mathscr{D}$ for any $u \in \mathscr{D}^{n}$ and any $\widehat{T} \in \contfunc^{\infty}(\mathbb{R}^{n})$ satisfying $\widehat{T}(0) = 0$ and $\sup_{x,y \in \mathbb{R}^{n}; x \neq y}\frac{\abs{\widehat{T}(x) - \widehat{T}(y)}}{\abs{x - y}} < \infty$. 
		\end{minipage}
	\end{equation}
	If $n_{1},n_{2} \in \mathbb{N}$, $q_{1} \in (0,p]$, $q_{2} \in [p,\infty]$ and $T = (T_{1},\dots,T_{n_{2}}) \colon \mathbb{R}^{n_{1}} \to \mathbb{R}^{n_{2}}$ satisfy \eqref{GC-cond} and $\bm{u} = (u_{1},\dots,u_{n_{1}}) \in H^{1,p}(X)^{n_{1}}$, then $T(\bm{u}) \in H^{1,p}(X)^{n_{2}}$ and  
	\begin{equation}\label{e:GCp.DF.density}
		\norm{\bigl(\Gamma_{\mu}(T_{l}(\bm{u}))(x)^{1/2}\bigr)_{l = 1}^{n_{2}}}_{\ell^{q_{2}}} 
		\le \norm{\bigl(\Gamma_{\mu}(u_{k})(x)^{1/2}\bigr)_{k = 1}^{n_{1}}}_{\ell^{q_{1}}} \quad \text{for $\mu$-a.e.\ $x \in X$}.
	\end{equation}
	In particular, $\bigl\{\Gamma_{\mu}(u)^{\frac{p}{2}}\,d\mu\bigr\}_{u \in H^{1,p}(X)}$ is a family of $p$-energy measures on $(X,\mathcal{B}(X))$ dominated by $\bigl(\int_{X}\Gamma_{\mu}(\,\cdot\,)^{\frac{p}{2}}\,d\mu,H^{1,p}(X)\bigr)$ and satisfies \hyperlink{GC-em}{\textup{(GC)$^{\textup{EM}}_{p}$}}. 
\end{thm}
\begin{proof}
	Let us consider the same mollifiers as in \cite[the last paragraph on p.~3732]{Kuw23+}, i.e., define $j \colon \mathbb{R}^{n_{1}} \to \mathbb{R}$ by $j(x) \coloneqq \exp{\bigl(-\frac{1}{1 - \abs{x}^{2}}\bigr)}$ for $\abs{x} \le 1$ and $j(x) \coloneqq 0$ for $\abs{x} > 1$, set $j_{m}(x) \coloneqq m^{n_{1}}j(mx)$ for each $m \in \mathbb{N}$.
	We define $T_{l,n}(x) \coloneqq \int_{\mathbb{R}^{n_{1}}}(j_{n}(x - y) - j_{n}(y))T_{l}(y)\,dy = \int_{\mathbb{R}^{n_{1}}}j_{n}(y)(T_{l,n}(x - y) - T_{l,n}(y))\,dy$ so that $T_{l,n} \in \contfunc^{\infty}(\mathbb{R}^{n_{1}})$, $T_{l,n}(0) = 0$ and $\lim_{n \to \infty}T_{l,n}(x) = T_{l}(x)$ for any $x \in \mathbb{R}^{n_{1}}$. 
	Then \eqref{GC-cond} with $T^{(n)} \coloneqq (T_{1,n},\dots,T_{n_{2},n})$ in place of $T$ holds; indeed, for any $x,y \in \mathbb{R}^{n_{1}}$, 
	\begin{align}\label{e:GCcond.mollifier}
		\norm{T^{(n)}(x) - T^{(n)}(y)}_{\ell^{q_{2}}}
		&= \norm{\left(\int_{\mathbb{R}^{n_{1}}}j_{n}(z)(T_{l}(x - z) - T_{l}(y - z))\,dz\right)_{l = 1}^{n_{2}}}_{\ell^{q_{2}}} \nonumber \\
		&\overset{\textup{($\ast$)}}{\le} \int_{\mathbb{R}^{n_{1}}}j_{n}(z)\norm{T(x - z) - T(y - z)}_{\ell^{q_{2}}}\,dz \nonumber \\
		&\overset{\eqref{GC-cond}}{\le} \norm{x - y}_{\ell^{q_{1}}}\int_{\mathbb{R}^{n_{1}}}j_{n}(z)\,dz
		= \norm{x - y}_{\ell^{q_{1}}}, 
	\end{align}
	where we used \eqref{e:intMin} with $q = q_{2}$ in \textup{($\ast$)}. 
	Moreover, 
	\begin{equation}\label{e:GC-operatornorm}
		\norm{\left(\sum_{k = 1}^{n_{1}}\partial_{k}T_{l,n}(x)y_{k}\right)_{l = 1}^{n_{2}}}_{\ell^{q_{2}}}
		= \lim_{\varepsilon \downarrow 0}\varepsilon^{-1}\norm{T^{(n)}(x) - T^{(n)}(x + \varepsilon y)}_{\ell^{q_{2}}}
		\overset{\eqref{e:GCcond.mollifier}}{\le} \norm{y}_{\ell^{q_{1}}}, 
	\end{equation}
	whence $\norm{(\partial_{k}T_{l,n}(x))}_{\ell^{q_{1}}_{n_{1}} \to \ell^{q_{2}}_{n_{2}}} \le 1$ for any $x \in \mathbb{R}^{n_{1}}$. 
	
	We first prove \eqref{e:GCp.DF.density} with $T^{(n)}$ in place of $T$ under the assumption that $\bm{u} = (u_{1},\dots,u_{n_{1}}) \in \mathscr{D}^{n_{1}}$. 
	Set $\widetilde{\bm{u}} = (\widetilde{u}_{1},\dots,\widetilde{u}_{n_{1}})$ where $\widetilde{u}_{k}$ is a $\mathcal{E}$-quasicontinuous $m$-version of $u_{k}$ (see \cite[p.~69 and Theorem 2.1.3]{FOT}).
	We have $T_{l,n}(\bm{u}) \in \mathscr{D}$ by \eqref{assum.corecontraction} and  
	\begin{equation}\label{e:gammamu.chain}
		\Gamma_{\mu}(T_{l,n}(\bm{u}))(x) = \sum_{i,j = 1}^{n_{1}}\partial_{i}T_{l,n}(\widetilde{\bm{u}}(x))\partial_{j}T_{l,n}(\widetilde{\bm{u}}(x))\Gamma_{\mu}(u_{i},u_{j})(x) \quad \text{for $\mu$-a.e.\ $x \in X$}
	\end{equation} 
	by the chain rule in \cite[(7) on p.~3724]{Kuw23+}. 
	Let $\{ f_{\lambda} \}_{\lambda \in \Lambda} \subseteq \mathcal{F}$ be an algebraic basis of $\mathcal{F}$ over $\mathbb{R}$. 
	Then there exist $n \in \mathbb{N}$, $\{ \alpha_{k,j} \}_{j = 1}^{n} \subseteq \mathbb{R}$, $k \in \{ 1,\dots,n_{1} \}$, and $\{ g_{j} \}_{j = 1}^{n} \subseteq \{ f_{\lambda} \}_{\lambda \in \Lambda}$ such that $u_{k} = \sum_{j = 1}^{n}\alpha_{k,j}g_{j}$ for each $k \in \{ 1,\dots,n_{1} \}$. 
	Let $R$ be the finitely generated algebra over $\mathbb{Q}$ generated by $\{ \alpha_{k,j} \}_{1 \le j \le n, 1 \le k \le n_{1}} \cup \{ 1 \}$ so that $\mathbb{Q} \subseteq R$ and $R$ is countable. 
	We set 
	\[
	\mathcal{U} \coloneqq \Biggl\{ \sum_{j = 1}^{n}a_{j}g_{j} \Biggm| a_{j} \in R \text{ for each $j \in \{ 1,\dots,n \}$} \Biggr\}
	\] 
	so that $\{ u_{k} \}_{k = 1}^{n_{1}} \subseteq \mathcal{U}$ and $\mathcal{U}$ is countable. 
	Since $R$ is dense in $\mathbb{R}$, for any $x \in X$, $N \in \mathbb{N}$, $k \in \{ 1,\dots,n_{1} \}$ and $l \in \{ 1,\dots,n_{2} \}$, there exists $A_{lk,n}^{x,N} \in R$ such that $\abs{\partial_{k}T_{l,n}(\widetilde{\bm{u}}(x)) - A_{lk,n}^{x,N}} \le N^{-1}$. 
	Note that $\Gamma_{\mu}(\,\cdot\,,\,\cdot\,)(x) \colon \mathcal{U} \times \mathcal{U} \to \mathbb{R}$ is a non-negative definite symmetric bilinear form for $\mu$-a.e.\ $x \in X$ since $\mathcal{U}$ is countable. 
	By Proposition \ref{prop.tensorineq}, for $\mu$-a.e.\ $x \in X$, 
	\begin{align*}
		&\norm{\left(\left(\sum_{i,j = 1}^{n_{1}}A_{li,n}^{x,N}A_{lj,n}^{x,N}\Gamma_{\mu}(u_{i},u_{j})(x)\right)^{1/2}\right)_{l =1}^{n_{2}}}_{\ell^{q_{2}}} \nonumber \\
		&= \norm{\left(\Gamma_{\mu}\left(\sum_{k = 1}^{n_{1}}A_{lk,n}^{x,N}u_{k}\right)(x)^{1/2}\right)_{l =1}^{n_{2}}}_{\ell^{q_{2}}} \nonumber \\
		&\le \biggl(1 + \norm{(\partial_{k}T_{l,n}(\widetilde{\bm{u}}(x)))_{l,k} - (A_{lk,n}^{x,N})_{l,k}}_{\ell^{q_{1}}_{n_{1}} \to \ell^{q_{2}}_{n_{2}}}\biggr)\norm{\bigl(\Gamma_{\mu}(u_{k})(x)^{1/2}\bigr)_{k = 1}^{n_{1}}}_{\ell^{q_{1}}}. 
	\end{align*}
	Letting $N \to \infty$ in the estimate above and recalling \eqref{e:gammamu.chain}, we obtain 
	\begin{equation}\label{e:GCp.DF.density.pre}
		\norm{\bigl(\Gamma_{\mu}(T_{l,n}(\bm{u}))(x)^{1/2}\bigr)_{l = 1}^{n_{2}}}_{\ell^{q_{2}}} 
		\le \norm{\bigl(\Gamma_{\mu}(u_{k})(x)^{1/2}\bigr)_{k = 1}^{n_{1}}}_{\ell^{q_{1}}} \quad \text{for $\mu$-a.e.\ $x \in X$,}
	\end{equation}
	under the assumption that $\bm{u} \in \mathscr{D}^{n_{1}}$. 
	
	Next let $\bm{u} = (u_{1},\dots,u_{n_{1}}) \in H^{1,p}(X)^{n_{1}}$ and fix $\{ \bm{u}^{(n)} = (u_{1,n},\dots,u_{n_{1},n}) \}_{n \in \mathbb{N}} \subseteq \mathscr{D}^{n_{1}}$ so that $\lim_{n \to \infty}\max_{k \in \{ 1,\dots,n_{1} \}}\norm{u_{k} - u_{k,n}}_{H^{1,p}} = 0$. 
	Then \eqref{e:GCp.DF.density.pre} together with the same argument as in \eqref{e:GC2forint} implies that 
	\[
	\norm{\left(\left(\int_{X}\Gamma_{\mu}(T_{l,n}(\bm{u}^{(n)}))^{\frac{p}{2}}\,d\mu\right)^{1/p}\right)_{l = 1}^{n_{2}}}_{\ell^{q_{2}}} 
		\le \norm{\left(\left(\int_{X}\Gamma_{\mu}(u_{k,n})^{\frac{p}{2}}\,d\mu\right)^{1/p}\right)_{k = 1}^{n_{1}}}_{\ell^{q_{1}}}. 
	\]
	In particular, $\{ T_{l,n}(\bm{u}^{(n)}) \}_{n \in \mathbb{N}}$ is bounded in $H^{1,p}(X)$. 
	Noting that $H^{1,p}(X)$ is reflexive (\cite[Theorem 1.7]{Kuw23+}) and that $\lim_{n \to \infty}\int_{X}\Gamma_{\mu}(u_{k} - u_{k,n})^{\frac{p}{2}}\,d\mu = 0$, we find $\{ n_{j} \}_{j \in \mathbb{N}} \subseteq \mathbb{N}$ with $\inf_{j \in \mathbb{N}}(n_{j + 1} - n_{j}) \ge 1$ such that $T^{(n_{j})}(\bm{u}^{(n_{j})})$ converges weakly in $H^{1,p}(X)^{\oplus n_{2}}$ to some $v = (v_{1},\dots,v_{n_{2}}) \in H^{1,p}(X)^{\oplus n_{2}}$ and $\max_{k \in \{ 1,\dots,n_{1} \}}\Gamma_{\mu}(u_{k} - u_{k,n_{j}})(x) \to 0$ for $\mu$-a.e.\ $x \in X$ as $j \to \infty$.\footnote{The direct sum $H^{1,p}(X)^{\oplus n_{2}}$ is equipped with the norm $\norm{f}_{H^{1,p}(X)^{\oplus n_{2}}} \coloneqq \sum_{l = 1}^{n_{2}}\norm{f_{j}}_{H^{1,p}(X)}$ for any $f = (f_{1},\dots,f_{n_{2}}) \in H^{1,p}(X)^{\oplus n_{2}}$.}  
	Since $\lim_{n \to \infty}\norm{T_{l,n}(\bm{u}^{(n)}) - T_{l}(\bm{u})}_{L^{p}(X,m)} = 0$ by \eqref{e:GCcond.mollifier} and the dominated convergence theorem, we have $v_{l} = T_{l}(\bm{u})$. 
	By Mazur's lemma (Lemma \ref{lem.mazur}), there exist $\{ N(i) \}_{i \in \mathbb{N}} \subseteq \mathbb{N}$ and $\{ \alpha_{j} \} \subseteq [0,1]$ with $\inf_{i \in \mathbb{N}}(N(i) - i) \ge 1$ and $\sum_{j = i}^{N(i)}\alpha_{i,j} = 1$ such that $\widehat{v}_{l,i} \coloneqq \sum_{j = i}^{N(i)}\alpha_{i,j}T_{l,n_{j}}(\bm{u}^{(n_{j})})$ converges strongly in $H^{1,p}(X)$ to $T_{l}(\bm{u})$ for any $l \in \{ 1,\dots,n_{2} \}$ as $i \to \infty$. 
	Then we easily see that for $\mu$-a.e.\ $x \in X$ and any $i \in \mathbb{N}$, 
	\begin{align}\label{e:gammamu.GCp.mazur}
		\norm{\bigl(\Gamma_{\mu}(\widehat{v}_{l,i})(x)^{1/2}\bigr)_{l = 1}^{n_{2}}}_{\ell^{q_{2}}} 
		&\le \norm{\left(\sum_{j = i}^{N(i)}\alpha_{i,j}\Gamma_{\mu}(T_{l,n_{j}}(\bm{u}^{(n_{j})}))(x)^{1/2}\right)_{l = 1}^{n_{2}}}_{\ell^{q_{2}}} \nonumber \\
		&\le \sum_{j = i}^{N(i)}\alpha_{i,j}\norm{\bigl(\Gamma_{\mu}(T_{l,n_{j}}(\bm{u}^{(n_{j})}))(x)^{1/2}\bigr)_{l = 1}^{n_{2}}}_{\ell^{q_{2}}} \nonumber \\
		&\overset{\eqref{e:GCp.DF.density.pre}}{\le} \sum_{j = i}^{N(i)}\alpha_{i,j}\norm{\bigl(\Gamma_{\mu}(u_{k,n_{j}})(x)^{1/2}\bigr)_{k = 1}^{n_{1}}}_{\ell^{q_{1}}}, 
	\end{align}
	where we used the triangle inequality for the norm of $\ell^{q_{2}}$ in the second inequality. 
	Note that for $\mu$-a.e.\ $x \in X$, 
	\[
	\lim_{i \to \infty}\sum_{j = i}^{N(i)}\alpha_{i,j}\norm{\bigl(\Gamma_{\mu}(u_{k,n_{j}})(x)^{1/2}\bigr)_{k = 1}^{n_{1}}}_{\ell^{q_{1}}}
	= \norm{\bigl(\Gamma_{\mu}(u_{k})(x)^{1/2}\bigr)_{k = 1}^{n_{1}}}_{\ell^{q_{1}}}. 
	\]
	Since $\lim_{i \to \infty}\int_{X}\Gamma_{\mu}(\widehat{v}_{l,i} - T_{l}(\bm{u}))^{\frac{p}{2}}\,d\mu = 0$, there exists $\{ m_{i} \}_{i \in \mathbb{N}} \subseteq \mathbb{N}$ with $\inf_{i \in \mathbb{N}}(m_{i + 1} - m_{i}) \ge 1$ such that $\lim_{i \to \infty}\Gamma_{\mu}(\widehat{v}_{l,m_{i}} - T_{l}(\bm{u}))(x) = 0$ for $\mu$-a.e.\ $x \in X$ and any $l \in \{ 1,\dots,n_{2} \}$. 
	In view of the triangle inequality for $\Gamma_{\mu}(\,\cdot\,)^{\frac{1}{2}}$ (see \cite[(3) on p.~3724]{Kuw23+}), we have $\lim_{i \to \infty}\max_{l \in \{ 1,\dots,n_{2} \}}\abs{\Gamma_{\mu}(\widehat{v}_{l,m_{i}})(x) - \Gamma_{\mu}(T_{l}(\bm{u}))(x)} = 0$ for $\mu$-a.e.\ $x \in X$.
	Hence we obtain \eqref{e:GCp.DF.density} by \eqref{e:gammamu.GCp.mazur}.
	Once we get \eqref{e:GCp.DF.density}, we easily see by the same argument as in \eqref{e:GC2forint} that $\bigl\{\Gamma_{\mu}(u)^{\frac{p}{2}}\,d\mu\bigr\}_{u \in H^{1,p}(X)}$, which is obviously a family of $p$-energy measures on $(X,\mathcal{B}(X))$ dominated by $\bigl(\int_{X}\Gamma_{\mu}(\,\cdot\,)^{\frac{p}{2}}\,d\mu,H^{1,p}(X)\bigr)$, satisfies \hyperlink{GC-em}{\textup{(GC)$^{\textup{EM}}_{p}$}}. 
\end{proof}

\section{Some results for \texorpdfstring{$p$}{p}-resistance forms on p.-c.f.\ self-similar structures}\label{sec:pcfcollect}

\subsection{Existence of \texorpdfstring{$p$}{p}-resistance forms with non-arithmetic weights}\label{sec:gap}
In this subsection, we discuss a gap between the frameworks in Subsection \ref{sec.Kigss} and in Subsection \ref{sec.pcf} for p.-c.f.\ self-similar structures. 
As in Subsection \ref{sec.pcf}, we fix $p \in (1,\infty)$ and a p.-c.f.\ self-similar structure $\mathcal{L} = (K,S,\{ F_{i} \}_{i \in S})$ with $\#S \ge 2$ and $K$ connected. 

The following proposition about the ``eigenvalue'' $\lambda(\bm{\rweight}_{p})$ in Theorem \ref{thm.eigenform} is a key result. 
\begin{prop}\label{prop.eigenvalue}
	Let $\bm{\rweight}_{p} = (\rweight_{p,i})_{i \in S} \in (0,\infty)^{S}$. 
	Assume that $\bm{\rweight}_{p}$ satisfies \eqref{condA} \textup{(recall Remark \ref{rmk.condA}-\ref{it:A-weight})}. 
	\begin{enumerate}[label=\textup{(\alph*)},align=left,leftmargin=*,topsep=2pt,parsep=0pt,itemsep=2pt]
		\item\label{it:eigenvalue.homogeneous} For any $a \in (0,\infty)$, $a\bm{\rweight}_{p} \coloneqq (a\rweight_{p,i})_{i \in S}$ satisfies  \eqref{condA} and $\lambda(a\bm{\rweight}_{p}) = a\lambda(\bm{\rweight}_{p})$. 
		\item\label{it:eigenvalue.mono} Let $\widetilde{\bm{\rweight}}_{p} = (\widetilde{\rweight}_{p,i})_{i \in S} \in (0,\infty)^{S}$. If $\widetilde{\bm{\rweight}}_{p}$ satisfies \eqref{condA} and $\rweight_{p,i} \le \widetilde{\rweight}_{p,i}$ for any $i \in S$, then $\lambda(\bm{\rweight}_{p}) \le \lambda(\widetilde{\bm{\rweight}}_{p})$.
	\end{enumerate}
\end{prop}
\begin{proof}
	Throughout this proof, we fix a $p$-resistance form $E_0$ on $V_0$. 
	
	\ref{it:eigenvalue.homogeneous}: 
	Since $\mathcal{R}_{a\bm{\rweight}_{p}}^{n}(E_{0}) = a\mathcal{R}_{\bm{\rweight}_{p}}^{n}(E_{0})$ for any $n \in \mathbb{N} \cup \{ 0 \}$, we easily see that $a\bm{\rweight}_{p}$ satisfies \eqref{condA}. 
	Recall from Theorem \ref{thm.eigenform}-\ref{it:CGQ.eigenvalue} that $\lambda(a\bm{\rweight}_{p}) \in (0,\infty)$ is the unique number satisfying the following: there exists $C \in [1,\infty)$ such that 
	\begin{equation}\label{e:unique.eigenvalue}
		C^{-1}\lambda(a\bm{\rweight}_{p})^{n}E_{0}(u) \le \mathcal{R}_{a\bm{\rweight}_{p}}^{n}(E_{0})(u) \le C\lambda(a\bm{\rweight}_{p})^{n}E_{0}(u) \quad \text{for any $n \in \mathbb{N} \cup \{ 0 \}$, $u \in \mathbb{R}^{V_0}$.}
	\end{equation}
	Therefore, $\lambda(a\bm{\rweight}_{p}) = a\lambda(\bm{\rweight}_{p})$. 
	
	\ref{it:eigenvalue.mono}:
	Since $\mathcal{R}_{\bm{\rweight}_{p}}^{n}(E_{0})(u) \le \mathcal{R}_{\widetilde{\bm{\rweight}}_{p}}^{n}(E_{0})(u)$ for any $u \in \mathbb{R}^{V_{0}}$, by \eqref{e:unique.eigenvalue}, there exists $C \in [1,\infty)$ such that for any $n \in \mathbb{N} \cup \{ 0 \}$ and any $u \in \mathbb{R}^{V_0}$, 
	\[
	C^{-1}\lambda(\bm{\rweight}_{p})^{n}E_{0}(u) \le \mathcal{R}_{\bm{\rweight}_{p}}^{n}(E_{0})(u) \le \mathcal{R}_{\widetilde{\bm{\rweight}}_{p}}^{n}(E_{0})(u) \le C\lambda(\widetilde{\bm{\rweight}}_{p})^{n}E_{0}(u). 
	\]
	Since $n \in \mathbb{N} \cup \{ 0 \}$ is arbitrary and $E_{0}(u) > 0$ for $u \in \mathbb{R}^{V_{0}} \setminus \mathbb{R}\indicator{V_{0}}$, we conclude that $\lambda(\bm{\rweight}_{p}) \le \lambda(\widetilde{\bm{\rweight}}_{p})$. 
\end{proof}

Now we can show the existence of $p$-resistance forms with non-arithmetic weights on a class of strongly symmetric p.-c.f.\ self-similar sets as follows. (Recall the notation in Subsection \ref{sec.ANF}.) 
\begin{prop}\label{prop.irrational}
	Let $\mathcal{L}$ be a strongly symmetric p.-c.f.\ self-similar set. 
	Assume that there exists $i \in S$ such that 
	\begin{equation}\label{e:ANF.orbit}
		\bigcup_{g \in \mathcal{G}}\tau_{g}(i) \neq S. 
	\end{equation}
	Then there exists $\bm{\rweight}_{p} = (\rweight_{p,i})_{i \in S} \in (0,\infty)^{S}$ such that $\lambda(\bm{\rweight}_{p}) = 1$, $\rweight_{p,i} > 1$ for any $i \in S$, $\bm{\rweight}_{p}$ satisfies \eqref{rscale.sym} and 
	\begin{equation}\label{e:irrational}
		\frac{\log{\rweight_{p,i}}}{\log{\rweight_{p,j}}} \not\in \mathbb{Q} \quad \text{for some $i,j \in S$.}
	\end{equation}
	In particular, there exists a self-similar $p$-resistance form $(\mathcal{E}_{p},\mathcal{F}_{p})$ on $\mathcal{L}$ with weight $\bm{\rweight}_{p}$.  
\end{prop}
\begin{rmk}
	\begin{enumerate}[label=\textup{(\arabic*)},align=left,leftmargin=*,topsep=2pt,parsep=0pt,itemsep=2pt]
		\item Any weight $\bm{\rweight}_{p} = (\rweight_{p,i})_{i \in S}$ of a $p$-energy form constructed in Theorem \ref{thm.KSgood-ss} must satisfy $\rweight_{p,i} = \sigma_{p}^{n_{i}}$ for some $n_{i} \in \mathbb{N}$, where $\sigma_{p} \in (0,\infty)$ is the $p$-scaling factor. 
			Hence constructions of self-similar $p$-energy forms with weight $\rweight_{p}$ which satisfies \eqref{e:irrational} are not covered by Theorem \ref{thm.KSgood-ss} (or by \cite[Theorem 4.6]{Kig23}). 
		\item The condition \eqref{e:ANF.orbit} is not very restrictive. See Figure \ref{fig.nonallorbit} for examples of self-similar sets satisfying this condition. In Figure \ref{fig.allorbit}, we present examples of self-similar sets that do not satisfy  \eqref{e:ANF.orbit}. 
	\end{enumerate}
\end{rmk}
\begin{proof}[Proof of Proposition \ref{prop.irrational}]
	Fix $i \in S$ and set $S_{1} \coloneqq \bigcup_{g \in \mathcal{G}}\tau_{g}(i)$ and $S_{2} \coloneqq S \setminus S_{1}$, which is non-empty by \eqref{e:ANF.orbit}. 
	For $t \in \mathbb{R}$, we define $\bm{\rweight}_{p}(t) \coloneqq  (\rweight_{p,s}(t))_{s \in S}$ by 
	\[
	\rweight_{p,s}(t) \coloneqq 1 + t\indicator{S_{2}}(s) \quad \text{for $s \in S$.} 
	\]
	It is easy to see that $\rweight_{p}(t)$ satisfies \eqref{rscale.sym}. 
	Set $\lambda_{p}(t) \coloneqq \lambda(\bm{\rweight}_{p}(t))$ for simplicity. 
	By Proposition \ref{prop.eigenvalue}, for any $t \in \mathbb{R}$, any $\delta \in (0,\infty)$ and any $s \in S$, 
	\begin{align*}
		(1 - t - \delta)\lambda_{p}(0) 
		\le \lambda_{p}(t - \delta) 
		\le \lambda_{p}(t)  
		\le \lambda_{p}(t + \delta)  
		\le (1 + t + \delta)\lambda_{p}(0) ,  
	\end{align*}
	whence $\lambda_{p}(t)$ is continuous in $t$. \
	
	Fix $j \in S_{2}$ and define 
	\[
	r_{i,j}(t) \coloneqq \frac{\log{(\rweight_{p,i}(t)/\lambda_{p}(t))}}{\log{(\rweight_{p,j}(t)/\lambda_{p}(t)})} = \frac{-\log{(\lambda_{p}(t))}}{\log{(1 + t)} - \log{(\lambda_{p}(t))}}, \quad t \in \mathbb{R}. 
	\]
	Since $r_{i,j}(0) = 1$ and $r_{i,j}(t)$ is continuous in $t$, there exists $t_{\ast} \in \mathbb{R} \setminus \{ 0 \}$ such that $r_{i,j}(t_{\ast}) \not\in \mathbb{Q}$. 
	The existence of a self-similar $p$-resistance form on $\mathcal{L}$ with weight $\bm{\rweight}_{p}$ follows from Theorems \ref{thm.eigenform-ANF} and \ref{thm.ANFsymform}, so we complete the proof. 
\end{proof}
\begin{figure}[tb]\centering
	\includegraphics[width=100pt]{fig_SG2_reduced.pdf}\hspace*{10pt}
	\includegraphics[width=100pt]{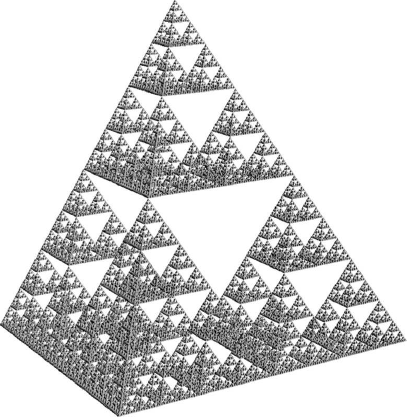}\hspace*{10pt}
	\includegraphics[width=100pt]{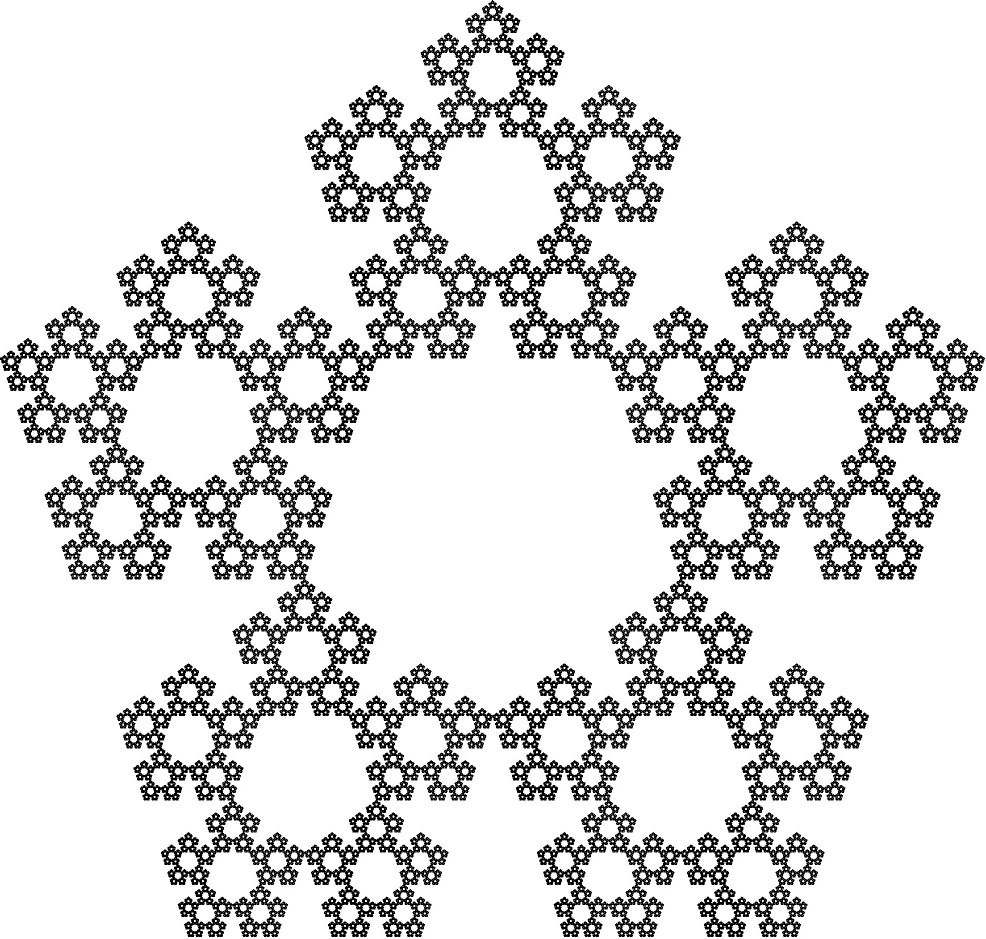}\hspace*{10pt}
    \includegraphics[width=100pt]{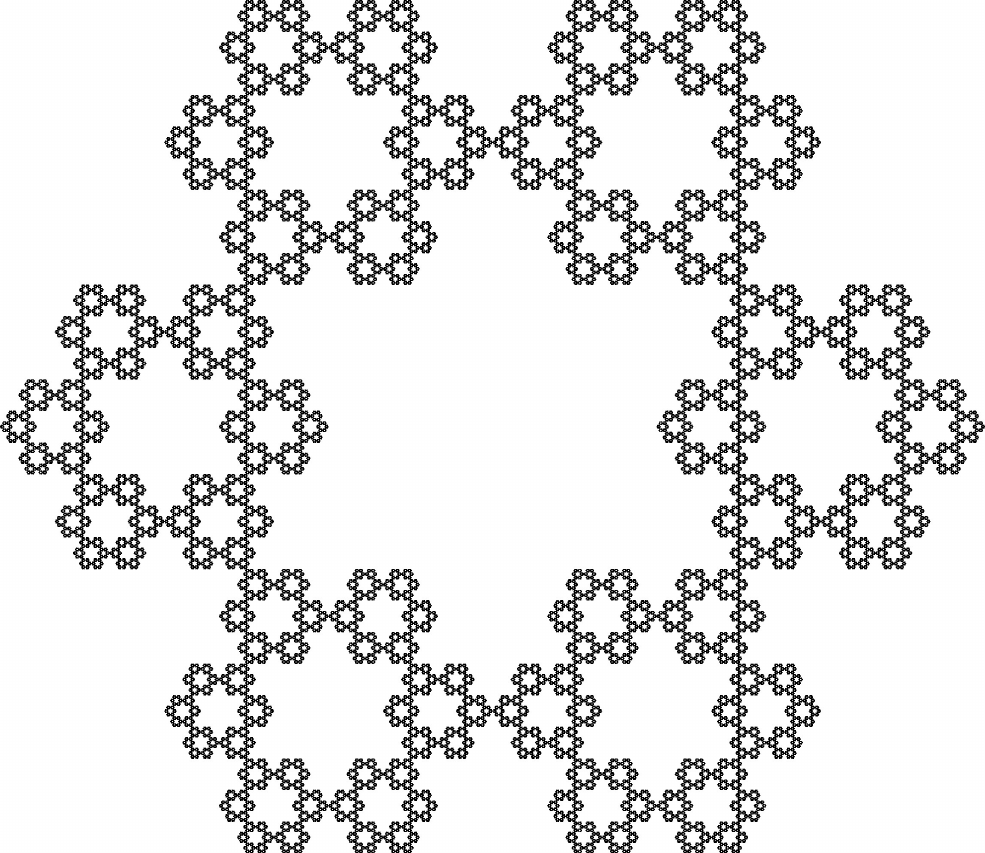}
	\caption{Examples of affine nested fractals that do \emph{NOT} satisfy \eqref{fig.nonallorbit}. From the left, $D$-dimensional level-$2$ Sierpi\'{n}ski gasket ($D = 2,3$), pentakun and hexagasket.}\label{fig.allorbit}
\end{figure}
\begin{figure}[tb]\centering
	\includegraphics[width=100pt]{fig_SG3_reduced.pdf}\hspace*{10pt}
	\includegraphics[width=100pt]{fig_SG4_reduced.pdf}\hspace*{10pt}
	\includegraphics[width=100pt]{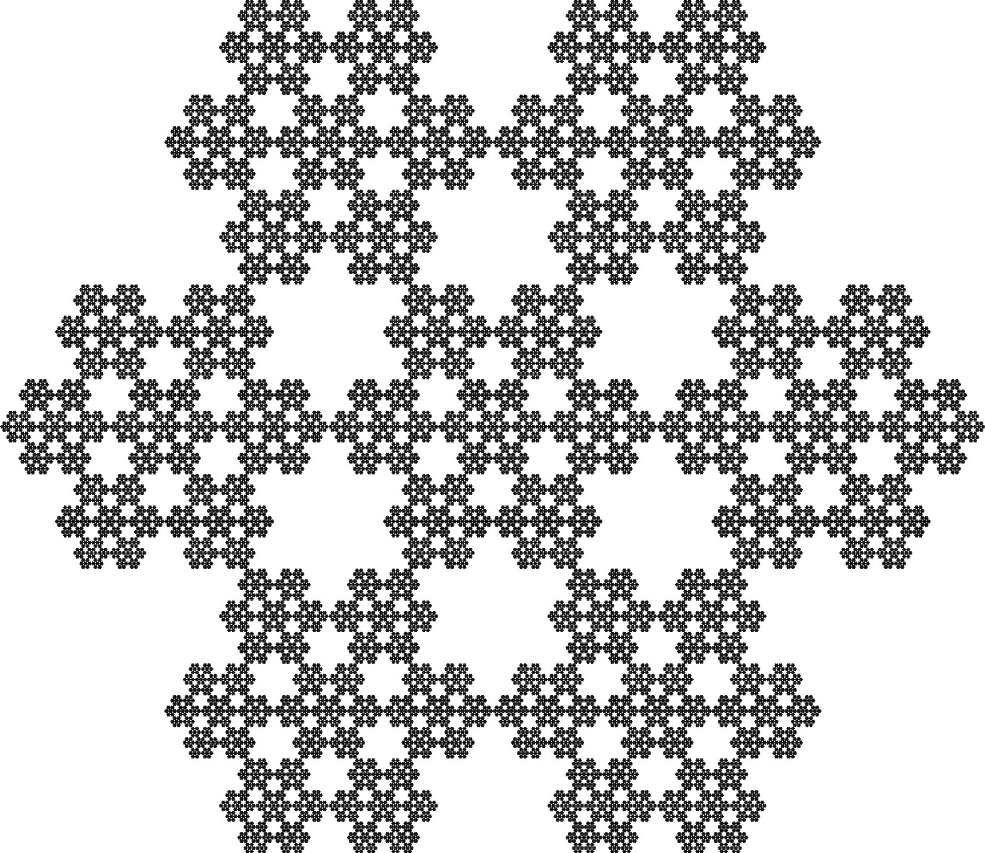}\hspace*{10pt}
    \includegraphics[width=100pt]{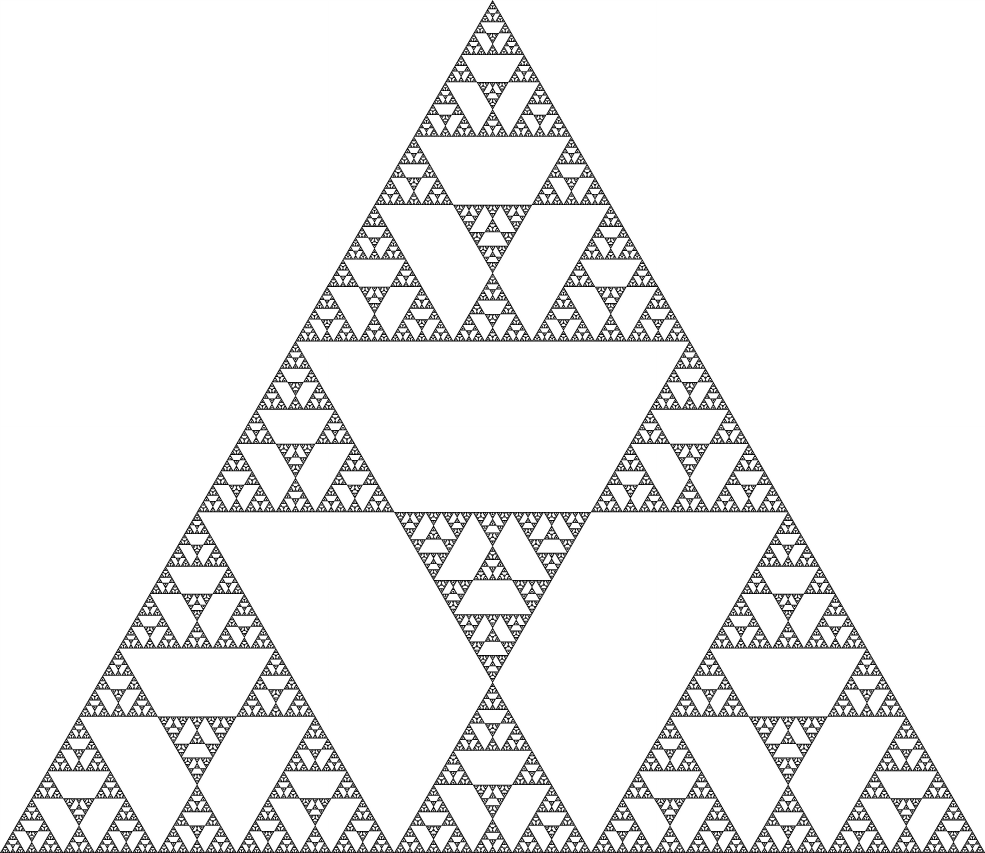}
	\caption{Examples of affine nested fractals that satisfy \eqref{fig.nonallorbit}. From the left, $2$-dimensional level-$l$ Sierpi\'{n}ski gasket ($l = 3,4$), snowflake and a Sierpi\'{n}ski gasket-type fractal.}\label{fig.nonallorbit}
\end{figure}

\subsection{Ahlfors regular conformal dimension of affine nested fractals}\label{sec.confdimANF}
In this subsection, we prove that the Ahlfors regular conformal dimension of any strongly symmetric self-similar set equipped with the $p$-resistance metric for any $p \in (1,\infty)$ is equal to one (Theorem \ref{thm.ARCpRM}). 
We also show that the Ahlfors regular conformal dimension with respect to the Euclidean metric is also equal to one under some geometric condition (Theorem \ref{thm.dARC-ANF}). 

Very similar results are already known in the literature.
Indeed, Tyson and Wu \cite[Theorems 1.3--1.5]{TW06} showed that the (quasi)conformal dimensions (as defined in \cite[p.~206]{TW06}) of the $D$-dimensional level-$2$ Sierpi\'nski gasket and of the $N$-polygasket with $N/4 \notin \mathbb{Z}$ are equal to one\footnote{According to \cite[the paragraph after Theorem 1.3]{TW06}, T.~J.~Laakso had shown before the work \cite{TW06} that the conformal dimension of the $2$-dimensional level-$2$ Sierpi\'nski gasket (equipped with the Euclidean metric) is equal to one.}. (The values of the conformal dimension and the Ahlfors regular conformal dimension coincide if the underlying metric space is compact, quasiself-similar \cite[Definition 2.4]{EB24}, connected and locally connected \cite[Theorem 1.6]{EB24}.) 
Also, Carrasco Piaggio \cite[Theorem 1.2]{CP14} provided a general criterion for a compact and metric doubling metric space to have Ahlfors regular conformal dimension one. 
This subsection is aimed at giving a new proof of a variant of these results in \cite{TW06,CP14} based on the existence of self-similar $p$-resistance forms proved in Theorem \ref{thm.eigenform-ANF}. 

Throughout this section, we assume that $\mathcal{L} = (K,S,\{ F_{i} \}_{i \in S})$ is a strongly symmetric p.-c.f.\ self-similar set (recall Framework \ref{frmwrk:ANF} and Definition \ref{defn.ANF}). 
Let $c_{i} \in (0,1)$ be the contraction ratio of $F_{i}$ for each $i \in S$. 
Note that $(c_{i})_{i \in S} \in (0,1)^{S}$ must satisfy 
\begin{equation}\label{contraction.sym} 
    c_{i} = c_{\tau_{g}(i)} \quad \text{for any $i \in S$ and any $g \in \mathcal{G}_{\mathrm{sym}}$,}
\end{equation}
because of the symmetry of $\mathcal{L}$.  
For each $p \in (1,\infty)$, we also fix a self-similar $p$-resistance form $(\mathcal{E}_{p}^{\#},\mathcal{F}_{p}^{\#})$ on $\mathcal{L}$ with weight $(\rweight_{\#,p})_{i \in S}$ for some $\rweight_{\#,p} \in (1,\infty)$, i.e., a $p$-resistance form $(\mathcal{E}_{p}^{\#},\mathcal{F}_{p}^{\#})$ on $K$ such that 
\begin{gather*}
	\mathcal{F}_{p}^{\#} = \{ u \in \contfunc(K) \mid u \circ F_{i} \in \mathcal{F}_{p}^{\#} \text{ for any $i \in S$} \}, \\
	\mathcal{E}_{p}^{\#}(u) = \rweight_{\#,p}\sum_{i \in S}\mathcal{E}_{p}^{\#}(u \circ F_{i}) \quad \text{for any $u \in \mathcal{F}_{p}^{\#}$.}
\end{gather*}
By Theorem \ref{thm.eigenform-ANF}, such a self-similar $p$-resistance form on $\mathcal{L}$ exists and the number $\rweight_{\#,p}$ is uniquely determined. 
Let $\pmetric_{p}^{\#}$ denote the $p$-resistance metric of $(\mathcal{E}_{p}^{\#},\mathcal{F}_{p}^{\#})$ (recall Definition \ref{defn:pResmet}).

The next proposition ensures that $\pmetric_{p}^{\#}$ is quasisymmetric (recall Definition \ref{defn.AR}-\ref{it:QS}) to the $q$-resistance metric of \emph{any} self-similar $q$-resistance form arising from Theorem \ref{thm.eigenform-ANF}.
\begin{prop}\label{prop.RpQS}
   Let $p,q \in (1,\infty)$ and assume that $\bm{\rweight}_{q} = (\rweight_{q,i})_{i \in S} \in (0,\infty)^{S}$ satisfies \eqref{rscale.sym}, $\rweight_{q,i} > 1$ for any $i \in S$ and $\lambda(\bm{\rweight}_{q}) = 1$, where $\lambda(\bm{\rweight}_{q}) \in (0,\infty)$ is the unique number given in Theorem \ref{thm.eigenform-ANF}.
   Let $(\mathcal{E}_{q},\mathcal{F}_{q})$ be a self-similar $q$-resistance form on $\mathcal{L}$ with weight $\bm{\rweight}_{q}$, which exists by Theorems \ref{thm.eigenform-ANF}, and let $\pmetric_{q}$ be the $q$-resistance metric associated with $(\mathcal{E}_{q},\mathcal{F}_{q})$.
   Then $\pmetric_{q,\mathcal{E}_{q}}$ is quasisymmetric to $\pmetric_{p}^{\#}$.
\end{prop}
\begin{proof}
    We will use \cite[Corollary 3.6.7]{Kig20} to show the desired statement. 
    We first show that there exist $\alpha_{1},\alpha_{2} \in (0,\infty)$ such that 
    \begin{equation}\label{pRMdiam}
    	\alpha_{1}\rweight_{q,w}^{-1/(p - 1)} \le \diam(K_{w},\pmetric_{q}) \le \alpha_{2}\rweight_{q,w}^{-1/(p - 1)} \quad \text{for any $w \in W_{\ast}$.}
    \end{equation}
    The upper estimate in \eqref{pRMdiam} is immediate from \eqref{pRMss}. 
	To prove the lower estimate in \eqref{pRMdiam}, note that we can easily find $m_{0} \in \mathbb{N}$ such that for any $w \in W_{\ast}$ there exist $v^{1},v^{2} \in W_{\abs{w} + m_{0}}$ with $v^{i} \le w$, $i = 1,2$, and $K_{v^{1}} \cap K_{v^{2}} = \emptyset$.  
    (It is enough to choose $m_{0}$ satisfying $2(\max_{i \in S}c_{i})^{m_{0}} < 1$.)
    Then, by the proof of Proposition \ref{prop:Rp-geom}-\ref{it:Rp.adapted} and $\rweight_{p,v^{i}} \le \rweight_{q,w}(\max_{i \in S}\rweight_{q,i})^{m_{0}}$, there exists $\alpha_{1} \in (0,\infty)$ that is independent of $w \in W_{\ast}$ such that 
    \[
    \inf_{(x,y) \in K_{v^{1}} \times K_{v^{2}}}\pmetric_{q}(x,y) \ge \alpha_{1}\rweight_{q,w}^{-1/(p - 1)}, 
    \]
    which implies the desired lower estimate in \eqref{pRMdiam}. 
    
    Next we note that $\mathcal{L}$ is a rationally ramified self-similar structure by  \cite[Proposition 1.6.12]{Kig09}; moreover, by combining \cite[Proposition 1.6.12]{Kig09}, $K_{v} \cap K_{w} = F_{v}(V_{0}) \cap F_{w}(V_{0})$ for any $v,w \in W_{\ast}$ with $\Sigma_{v} \cap \Sigma_{w} = \emptyset$ (see \cite[Proposition 1.3.5-(2)]{Kig01}) and the fact that each element of $V_{0}$ is a fixed point of $F_{i}$ for some $i \in S_{\mathrm{fix}} \coloneqq \{ i \in S \mid K_{i} \cap V_{0} \neq \emptyset \}$, $\mathcal{L}$ is rationally ramified with a relation set 
	\begin{equation}\label{ANF.relation}
		\mathcal{R}
	= \bigl\{ \{ (\{ w(j) \}, \{ v(j) \}, \varphi_{j}, x(j), y(j)) \mid w(j),v(j),x(j),y(j) \in W_{\ast} \setminus \{ \emptyset \} \} \bigr\}_{j = 1}^{k} 
	\end{equation}
	satisfying $w(j),v(j) \in S_{\mathrm{fix}}$.  
 	(See \cite[Sections 1.5 and 1.6 and Chapter 8]{Kig09} for details about rationally ramified self-similar structures.)
 	
 	With these preparations, we will apply \cite[Corollary 3.6.7]{Kig20} to $\pmetric_{q,\mathcal{E}_{q}}$ and $\pmetric_{p}^{\#}$. 
 	By Proposition \ref{prop:Rp-geom}-\ref{it:Rp.adapted} and \eqref{pRMdiam}, $\pmetric_{q,\mathcal{E}_{q}}$ is \emph{$1$-adapted} and \emph{exponential} (see \cite[Definition 2.4.7 and 3.1.15-(2)]{Kig20} for these definitions; see also Remark in \cite[p.\ 108]{Kig20}). 
 	Similarly, $\pmetric_{p}^{\#}$ is also $1$-adapted and exponential. 
 	Hence, by \cite[Corollary 3.6.7]{Kig20}, $\pmetric_{q,\mathcal{E}_{q}}$ is quasisymmetric to $\pmetric_{p}^{\#}$ if and only if $\pmetric_{q,\mathcal{E}_{q}}$ is \emph{gentle} with respect to $\pmetric_{p}^{\#}$ (see \cite[Definition 3.3.1]{Kig20} for the definition of the gentleness). 
 	Define $g_{q}(w) \coloneqq \rweight_{q,w}^{-1/(q - 1)}$ and $g_{\#,p}(w) \coloneqq \rweight_{\#,p}^{-\abs{w}}$ for $w \in W_{\ast}$. 
 	Since $g_{q}$ and $g_{\#,p}$ satisfy the condition (R1) in \cite[Theorem 1.6.6]{Kig09} by \eqref{rscale.sym} and \eqref{ANF.relation}, we obtain the desired gentleness by \cite[Theorem 1.6.6]{Kig09} and \eqref{pRMdiam}. 
 	This completes the proof. 
\end{proof}

Now we can determine the Ahlfors regular conformal dimension of $(K,\pmetric_{p}^{\#})$ by using the discrete characterization of the Ahlfors regular conformal dimension due to Keith and Kleiner (see \cite[the paragraph before Corollary 1.4]{CP13}).
\begin{thm}\label{thm.ARCpRM}
	$\dim_{\mathrm{ARC}}(K,\pmetric_{p}^{\#}) = 1$. 
\end{thm}
\begin{proof}
	We will use a version of the characterization of $\dim_{\mathrm{ARC}}(K,\pmetric_{p}^{\#})$ in \cite[Theorem 4.6.9]{Kig20}.
	Note that $(K,\pmetric_{p}^{\#})$ satisfies (BF1) and (BF2) in \cite[Section 4.3]{Kig20} by Proposition \ref{prop:Rp-geom}-\ref{it:Rp.adapted}, \eqref{pRMdiam}, \cite[Proposition 1.6.12, Lemmas 1.3.6 and 1.3.12]{Kig09}. 
 	We define a graph $G_{n} = (V_{n},E_{n})$ and $q$-energy $\mathcal{E}_{p}^{G_{n}}$, $q \in (1,\infty)$, on $G_{n}$ by
    \[
    E_{n} \coloneqq \{ (x,y) \mid \text{$x,y \in F_{w}(V_{0})$ for some $w \in W_{n}$} \}, 
    \]
    and 
    \[
    \mathcal{E}_{q}^{G_{n}}(f) \coloneqq \frac{1}{2}\sum_{(x,y) \in E_{n}}\abs{f(x) - f(y)}^{q}, \quad f \in \mathbb{R}^{V_{n}}. 
    \]
    Note that $\{ G_{n} \}_{n \ge 0}$ is a \emph{proper system of horizontal networks} with indices $(1,2(\#V_{0} - 1)\#V_{0},1,1)$ in the sense of \cite[Definition 4.6.5]{Kig20}. 
    Therefore by \cite[Theorem 4.6.9]{Kig20}, $\dim_{\mathrm{ARC}}(K,\pmetric_{p}^{\#}) = 1$ if and only if the following holds: for any $q \in (1,\infty)$, 
    \begin{equation}\label{ARCcharacterization}
    	\liminf_{k \to \infty}\sup_{w \in W_{\ast}}\inf\Bigl\{ \mathcal{E}_{q}^{G_{\abs{w} + k}}(f) \Bigm| f \in \mathbb{R}^{V_{\abs{w} + k}}, f|_{F_{w}(V_{k})} = 1, f|_{Z_{w,k}} = 0 \Bigr\} = 0, 	
    \end{equation}
	where $Z_{w,k} \coloneqq \{ x \in V_{\abs{w} + n} \mid \text{$x \in F_{v}(V_{k})$ for some $v \in W_{\abs{w}}$ with $K_{v} \cap K_{w} = \emptyset$} \}$.
	Since both $\mathcal{E}_{q}^{\#}\bigr|_{V_{0}}(\,\cdot\,)^{1/q}$ and $\mathcal{E}_{q}^{G_{0}}(\,\cdot\,)^{1/q}$ are norms on the finite-dimensional vector space $\mathbb{R}^{V_{0}}/\mathbb{R}\indicator{V_{0}}$, there exists $C \ge 1$ such that $C^{-1}\mathcal{E}_{q}^{\#}\bigr|_{V_{0}}(u) \le \mathcal{E}_{q}^{G_{0}}(u) \le C\mathcal{E}_{q}^{\#}\bigr|_{V_{0}}(u)$ for any $u \in \mathbb{R}^{V_{0}}$. 
	Hence, by Propositions \ref{prop.pRMss}-\ref{pRMcompatible} and \ref{prop.compatible}, we obtain $C^{-1}\mathcal{E}_{q}^{\#}\bigr|_{V_{n}}(u) \le \rweight_{\#,q}^{n}\mathcal{E}_{q}^{G_{n}}(u) \le C\mathcal{E}_{q}^{\#}\bigr|_{V_{n}}(u)$ for any $n \in \mathbb{N} \cup \{ 0 \}$ and any $u \in \mathbb{R}^{V_{n}}$. 
	Recall that $\Gamma_{1}(w) = \{ v \in W_{\abs{w}} \mid K_{v} \cap K_{w} \neq \emptyset \}$ for $w \in W_{\ast}$ (Definition \ref{defn.h-networks}). 
	Let $h_{q,w} \in \mathcal{F}_{q}^{\#}$ be the unique function satisfying $h_{q,w}|_{K_{w}} = 1$, $h_{q,w}|_{K_{v}} = 0$ for any $v \in W_{\abs{w}} \setminus \Gamma_{1}(w)$ and 
	\[
	\mathcal{E}_{q}^{\#}(h_{q,w}) = \inf\Bigl\{ \mathcal{E}_{q}^{\#}(u) \Bigm| u|_{K_{w}} = 1, \text{$u|_{K_{v}} = 0$ for any $v \in W_{\abs{w}} \setminus \Gamma_{1}(w)$} \Bigr\}. 
	\]
	Then we see from \eqref{Rp-capu}, \eqref{Rp-adapted} and \eqref{pRMdiam} that 
	\begin{align*}
		&\sup_{w \in W_{\ast}}\inf\Bigl\{ \mathcal{E}_{q}^{G_{\abs{w} + k}}(f) \Bigm| f \in \mathbb{R}^{V_{\abs{w} + k}}, f|_{F_{w}(V_{k})} = 1, f|_{Z_{w,k}} = 0 \Bigr\} \\
		&\le C\rweight_{\#,q}^{-(\abs{w} + k)}\sup_{w \in W_{\ast}}\mathcal{E}_{q}^{\#}\bigr|_{V_{\abs{w} + k}}(h_{q,w}|_{V_{\abs{w} + k}}) 	
		\le C\rweight_{\#,q}^{-(\abs{w} + k)}\sup_{w \in W_{\ast}}\mathcal{E}_{q}^{\#}(h_{q,w}) 
		\lesssim \rweight_{\#,q}^{-k}. 
	\end{align*}
	Since $\rweight_{\#,q} \in (1,\infty)$ for any $q \in (0,1)$, we obtain \eqref{ARCcharacterization}. 
	The proof is completed. 
\end{proof}

To discuss the Ahlfors regular conformal dimension of $K$ with respect to the Euclidean metric, we need the following assumption.
\begin{assum}\label{assum.ANF}
	We define $\Lambda_{1}^{d} \coloneqq \{ \emptyset \}$, 
	\[
	\Lambda_{s}^{d} \coloneqq \{ w \mid w = w_{1} \dots w_{n} \in W_{\ast} \setminus \{ \emptyset \}, \diam(K_{w_{1} \dots w_{n - 1}},d) > s \ge \diam(K_{w},d) \}
	\]
	for each $s \in (0,1)$. 
	For $s \in (0,1]$, $M \in \mathbb{N} \cup \{ 0 \}$ and $x \in K$, define 
	\[
	\Lambda_{s,M}^{d}(x) \coloneqq 
	\Biggl\{ v \Biggm|
	\begin{minipage}{280pt}
		$v \in \Lambda_{s}^{d}$, there exists $w \in \Lambda_{s}^{d}$ with $x \in K_{w}$ and $\{ z(j) \}_{j = 1}^{k} \subseteq \Lambda_{s}^{d}$ with $k \le M + 1$, $z(1) = w$, $z(k) = v$ such that $K_{z(j)} \cap K_{z(j + 1)} \neq \emptyset$ for any $j \in \{ 1,\dots, k - 1 \}$
	\end{minipage} 
	\Biggr\}, 
	\]
	and $U_{M}^{d}(x,s) \coloneqq \bigcup_{w \in \Lambda_{s,M}^{d}(x)}K_{w}$. 
	Then there exist $M_{\ast} \in \mathbb{N}$, $\alpha_{0},\alpha_{1} \in (0,\infty)$ such that 
	\[
	U_{M_{\ast}}^{d}(x,\alpha_{0}s) \subseteq B_{d}(x,s) \subseteq U_{M_{\ast}}^{d}(x,\alpha_{1}s) \quad \text{for any $(x,s) \in K \times (0,1]$.}
	\] 
    (Equivalently, $d$ is $M_{\ast}$-adapted; see \cite[Definition 2.4.1]{Kig20}.) 
\end{assum}
\begin{rmk}\label{rmk.ANFadapted}
	We do not know whether Assumption \ref{assum.ANF} is true for any strongly symmetric self-similar set. 
	Even for nested fractals, being $1$-adapted with respect to the Euclidean metric is required as an additional assumption in \cite[Assumption 4.41]{Kig23}.	
\end{rmk}

Now we can show the main result in this section under Assumption \ref{assum.ANF}. 
\begin{thm}\label{thm.dARC-ANF}
    Assume that Assumption \ref{assum.ANF} holds.
    Then $\dim_{\mathrm{ARC}}(K,d) = 1$.
\end{thm} 
\begin{proof} 
	Thanks to Theorem \ref{thm.ARCpRM}, it suffices to prove that $\pmetric_{p}^{\#}$ is quasisymmetric to $d$. 
	Obviously, $d$ is exponential since $\diam(K_{w},d) = c_{w}\diam(K,d)$. 
	By \eqref{contraction.sym}, a similar argument as in the proof of Proposition \ref{prop.RpQS} implies that $\pmetric_{p}^{\#}$ is gentle with respect to $d$. 
	Hence \cite[Corollary 3.6.7]{Kig20} together with Assumption \ref{assum.ANF} implies that $\pmetric_{p}^{\#}$ is quasisymmetric to $d$.  
\end{proof}

\subsection{An estimate on self-similar regular \texorpdfstring{$p$}{p}-resistance forms on p.-c.f.\ self-similar structures}\label{sec:pcf-contraction}
This subsection is devoted to proving the following theorem, which is a generalization of \cite[Theorem A.1]{Kig03}. 
\begin{thm}\label{thm.pcf-contraction}
	Let $p \in (1,\infty)$, let $\mathcal{L} = (K,S,\{ F_{i} \}_{i \in S})$ be a p.-c.f.\ self-similar structure with $\#S \ge 2$ and $K$ connected, and let $(\mathcal{E},\mathcal{F})$ be a self-similar $p$-resistance form on $\mathcal{L}$ with weight $\bm{\rweight} = (\rweight_{i})_{i \in S} \in (1,\infty)^{S}$. 
	Then there exists $c \in (0,1)$ such that for any $x,y \in K$ and any $w \in W_{\ast}$, 
	\begin{equation}\label{e:pcf-contraction}
		c\rweight_{w}^{-1}R_{\mathcal{E}}(x,y) \le R_{\mathcal{E}}(F_{w}(x),F_{w}(y)) \le \rweight_{w}^{-1}R_{\mathcal{E}}(x,y). 
	\end{equation}
\end{thm}

Since the upper estimate in \eqref{e:pcf-contraction} is obtained in \eqref{pRMss}, what matters is the lower estimate in \eqref{e:pcf-contraction}. 
To prove it, we need the following lemma. 
\begin{lem}\label{lem.ss-special}
	Assume the same conditions as in Theorem \ref{thm.pcf-contraction}. 
	Let $x,y \in K$ and $w \in W_{\ast}$. 
	Set $\Lambda \coloneqq \{ \tau = \tau_{1}\dots\tau_{n} \in W_{\ast} \mid (\rweight_{\tau_{1}\cdots\tau_{n - 1}})^{-1} > \rweight_{w} \ge \rweight_{\tau}^{-1} \}$, $U \coloneqq V_{0} \cup \{ x,y \}$, $V_{\Lambda} \coloneqq \bigcup_{w \in \Lambda}F_{w}(V_{0})$ and $V \coloneqq V_{\Lambda} \cup \{ F_{w}(x), F_{w}(y) \}$. 
	Then $\Lambda$ is a partition of $\Sigma$ and 
	\begin{equation}\label{e:ss-special}
		\mathcal{E}|_{V}(u) = \rweight_{w}\mathcal{E}|_{U}(u \circ F_{w}) + \sum_{\tau \in \Lambda \setminus \{ w \}}\rweight_{\tau}\mathcal{E}|_{V_{0}}(u \circ F_{\tau}) \quad \text{for any $u \in \mathcal{F}|_{V}$.} 
 	\end{equation} 
\end{lem}
\begin{proof}
    The proof is very similar to Proposition \ref{prop.compatible}. 
    It is clear that $\Lambda$ is a partition of $\Sigma$. 
    Note that, by Proposition \ref{prop.pRMss}-\ref{pRMcompatible}, $R_{\mathcal{E}}^{1/p}$ is compatible with the original topology of $K$ and thereby $\diam(K,R_{\mathcal{E}}^{1/p}) < \infty$. 
    For any $u \in \mathcal{F}|_{V}$,
    \begin{align*}
        &\mathcal{E}|_{V}(u) \\
        &= \min\bigl\{ \mathcal{E}(v) \bigm| v \in \mathcal{F}, v|_{V} = u \bigr\} \\
        &\overset{\eqref{ss.partition}}{=} \min\Biggl\{ \rweight_{w}\mathcal{E}(v \circ F_{w}) + \sum_{\tau \in \Lambda \setminus \{ w \}}\rweight_{\tau}\mathcal{E}(v \circ F_{\tau}) \Biggm| \text{$v \in \mathcal{F}$, $v|_{V} = u$} \Biggr\} \\
        &\ge \min\Biggl\{ \rweight_{w}\mathcal{E}(v \circ F_{w}) \Biggm| v \in \mathcal{F}, v|_{V} = u \Biggr\} + \min\Biggl\{\sum_{\tau \in \Lambda \setminus \{ w \}}\rweight_{\tau}\mathcal{E}(v \circ F_{\tau}) \Biggm| v \in \mathcal{F}, v|_{V} = u \Biggr\} \\
        &\ge \rweight_{w}\min\{ \mathcal{E}(v) \mid v \in \mathcal{F}, v|_{U} = u \circ F_{w} \} + \sum_{\tau \in\Lambda \setminus \{ w \}}\rweight_{\tau}\min\{\mathcal{E}(v) \mid v \in \mathcal{F}, v|_{V_{0}} = u \circ F_{\tau} \} \\
        &= \rweight_{w}\mathcal{E}|_{U}(u \circ F_{w}) + \sum_{\tau \in \Lambda \setminus \{ w \}}\rweight_{\tau}\mathcal{E}|_{V_{0}}(u \circ F_{\tau}).
    \end{align*}
    To prove the converse, let $v \in C(K)$ satisfy $v \circ F_{w} = h_{U}^{\mathcal{E}}[u \circ F_{w}]$ and, for $\tau \in \Lambda \setminus \{ w \}$, $v \circ F_{\tau} = h_{V_{0}}^{\mathcal{E}}[u \circ F_{\tau}]$.
    Such $v$ is well-defined since $K_{w} \cap K_{\tau} = F_{w}(V_{0}) \cap F_{\tau}(V_{0})$. 
    Also, we have $v|_{V} = u$ and $v \in \mathcal{F}$ by \eqref{SSE1}. 
    Moreover,
    \begin{align*}
        \mathcal{E}|_{V}(u)
        \le \mathcal{E}(v)
        &\overset{\eqref{ss.partition}}{=} \sum_{\tau \in \Lambda}\rweight_{\tau}\mathcal{E}(v \circ F_{\tau})
        = \rweight_{w}\mathcal{E}|_{U}(u \circ F_{w}) + \sum_{\tau \in \Lambda \setminus \{ w \}}\rweight_{\tau}\mathcal{E}|_{V_{0}}(u \circ F_{\tau}).
    \end{align*}
    This completes the proof. 
\end{proof}

\begin{proof}[Proof of Theorem \ref{thm.pcf-contraction}]
	Let $\Lambda,U,V_{\Lambda},V$ be the same as in Lemma \ref{lem.ss-special}. 
	Set $\Gamma_{1}(w; \Lambda) \coloneqq \{ \tau \in \Lambda \mid w \neq \tau, K_{w} \cap K_{\tau} \neq \emptyset \}$ for simplicity. 
	Then $\#\Gamma_{1}(w; \Lambda) \le \#(\mathcal{C}_{\mathcal{L}})\#(V_{0})$ by \cite[Lemma 4.2.3]{Kig01}. 
    Let $\psi_{xy} \in \mathcal{F}$ satisfy $\psi_{xy}(x) = 1$, $\psi_{xy}(y) = 0$ and $\mathcal{E}(\psi_{xy}) = R_{\mathcal{E}}(x,y)^{-1}$.
    Let $u_{\ast} \in \mathcal{F}$ satisfy $u_{\ast}(x) = 1$, $u_{\ast}(y) = 0$, $u|_{V \setminus F_{w}(U)} \in \mathbb{R}\indicator{V \setminus F_{w}(U)}$ and 
    \[
    \mathcal{E}(u_{\ast}) = \inf\{ \mathcal{E}(v) \mid v \in \mathcal{F}, (v \circ F_{w})|_{U} = \psi_{xy}, v|_{V \setminus F_{w}(U)} \in \mathbb{R}\indicator{V \setminus F_{w}(U)} \}.
    \]
    Such $u_{\ast}$ is uniquely exists by a standard argument in the variational analysis.
    Also, by Proposition \ref{prop.GC-list}-\ref{GC.lip}, we easily see that $0 \le u_{\ast} \le 1$.
    Since $\mathbb{R}^{V_{0}}/\mathbb{R}\indicator{V_{0}}$ is a finite dimensional vector space, there exists a constant $C \in (0,\infty)$ such that
    \begin{equation}\label{e:EpV0-osc}
        \mathcal{E}|_{V_{0}}(u)^{1/p} \le C\max_{z,z' \in V_{0}}\abs{u(z) - u(z')} \quad \text{for any $u \in \mathbb{R}^{V_{0}}$.}
    \end{equation}
    Then, by using Lemma \ref{lem.ss-special}, we see that 
    \begin{align*}
        R_{\mathcal{E}}(F_{w}(x),F_{w}(y))^{-1}
        \le \mathcal{E}(u_{\ast})
        &= \mathcal{E}|_{V}(u_{\ast}) \\
        &= \rweight_{w}\mathcal{E}|_{U}(u_{\ast} \circ F_{w}) + \sum_{\tau \in \Lambda \setminus \{ w \}}\rweight_{\tau}\mathcal{E}|_{V_{0}}(u_{\ast} \circ F_{\tau}) \\
        &= \rweight_{w}\mathcal{E}|_{U}(u_{\ast} \circ F_{w}) + \sum_{\tau \in \Gamma_{1}(w; \Lambda)}\rweight_{\tau}\mathcal{E}|_{V_{0}}(u_{\ast} \circ F_{\tau}) \\
        &\overset{\eqref{e:EpV0-osc}}{\le} \frac{\rweight_{w}}{R_{\mathcal{E}}(x,y)} + C^{p}\sum_{\tau \in \Gamma_{1}(w; \Lambda)}\rweight_{\tau} \\
        &\le \rweight_{w}\left(\frac{1}{R_{\mathcal{E}}(x,y)} + C^{p}\Bigl(\max_{i \in S}\rweight_{i}\Bigr)(\#\Gamma_{1}(w; \Lambda))\right) \\
        &= \rweight_{w}\left(\frac{1}{R_{\mathcal{E}}(x,y)} + C'\frac{R_{\mathcal{E}}(x,y)}{R_{\mathcal{E}}(x,y)}\right) \\
        &\le \rweight_{w}\biggl(1 + C'\sup_{z,z' \in K}R_{\mathcal{E}}(z,z')\biggr)R_{\mathcal{E}}(x,y)^{-1}, 
    \end{align*}
    which shows the desired lower estimate in \eqref{e:pcf-contraction}. 
\end{proof}


\noindent \textbf{Naotaka Kajino} \\
Research Institute for Mathematical Sciences, Kyoto University, Kitashirakawa-Oiwakecho, Sakyo-ku, Kyoto 606-8502, Japan.\\
E-mail: \texttt{nkajino@kurims.kyoto-u.ac.jp}\smallskip\\

\noindent \textbf{Ryosuke Shimizu} \\
Waseda Research Institute for Science and Engineering, Waseda University, 3-4-1 Okubo, Shinjuku-ku, Tokyo 169-8555, Japan. \\
Graduate School of Informatics, Kyoto University, Yoshida-honmachi, Sakyo-ku, Kyoto 606-8501, Japan (current address). \\
E-mail: \texttt{r.shimizu@acs.i.kyoto-u.ac.jp}

\section*{Index of symbols}
\addcontentsline{toc}{section}{Index of symbols}
\begin{enumerate}[label=\textup{$\bullet$},align=left,leftmargin=*,topsep=2pt,parsep=0pt,itemsep=2pt]
	\item $\indicator{A}=\indicator{A}^{X}$: indicator function of a subset $A$ of a set $X$ --- Notation \ref{notation.intro}-\ref{it:id-indicator}
	\item $\#A$: cardinality of a set $A$ --- Notation \ref{notation.intro}-\ref{it:cardinality}
	\item $\closure{A}^{X}$: closure of $A \subseteq X$ in $X$ --- Notation \ref{notation.intro}-\ref{it:Borel-contfunc}
	\item $a^{+}$: positive part of a number or function $a$ --- Notation \ref{notation.intro}-\ref{it:supinf}
	\item $a \vee b$: maximum of numbers or functions $a$ and $b$ --- Notation \ref{notation.intro}-\ref{it:supinf}
	\item $a \wedge b$: minimum of numbers or functions $a$ and $b$ --- Notation \ref{notation.intro}-\ref{it:supinf}
	\item $\mathcal{B}(X)$: Borel $\sigma$-algebra of a topological space $X$ --- Notation \ref{notation.intro}-\ref{it:Borel-contfunc}
	\item $\mathcal{B}|_{A}$: trace of a $\sigma$-algebra $\mathcal{B}$ on a subset $A$ --- Notation \ref{notation.intro}-\ref{it:measure-sp}
	\item $B_{d}(x,r)$: open metric ball of radius $r$ centered at $x$ in a metric space $(X,d)$ --- Notation \ref{notation.intro}-\ref{it:metric-sp}
	\item $B^{\mathcal{F}}$ --- Definition \ref{defn.RFp-general}
	\item $\contfunc(X)$: $\mathbb{R}$-linear space of $\mathbb{R}$-valued continuous functions on $X$ --- Notation \ref{notation.intro}-\ref{it:Borel-contfunc}
	\item $\contfunc_{b}(X)$: $\mathbb{R}$-linear space of bounded functions in $\contfunc(X)$ --- Notation \ref{notation.intro}-\ref{it:Borel-contfunc}
	\item $\contfunc_{c}(X)$: $\mathbb{R}$-linear space of functions in $\contfunc(X)$ with compact supports in $X$ --- Notation \ref{notation.intro}-\ref{it:Borel-contfunc}
	\item $\mathcal{C}_{\mathcal{L}}$: critical set of $\mathcal{L}$ --- \eqref{e:C-P} in Definition \ref{d:V0Vstar}-\ref{it:C-P}
	\item $\mathbb{C}_{\mathcal{F}}$ --- Proposition \ref{p:genbord}-\ref{fcn-top}
	\item $\mathrm{cap}_{p}^{n}(A_{0},A_{1};A)$: condenser capacity between $A_0$ and $A_1$ in $A \subseteq T_n$ --- Definition \ref{defn.con-const}-\ref{it:concap} 
	\item $\mathcal{D}^{\#}$: $\mathcal{E}$-closure of $\mathcal{D} \subseteq \mathcal{F}$ for a $p$-energy form $(\mathcal{E},\mathcal{F})$ --- \eqref{e:defn.ExtD}
	\item $\mathcal{D}_{\mathrm{loc}}(U)$: the set of functions belonging locally in $U$ to $\mathcal{D}$ --- \eqref{e:defn.Floc} in Definition \ref{defn.Flocal-ss}-\ref{it:Flocal-ss}
	\item $\diam(A,d)$: diameter of $A \subseteq X$ in $d$ for a metric space $(X,d)$ --- Notation \ref{notation.intro}-\ref{it:metric-sp}
	\item $\dim_{\mathrm{ARC}}(K,d)$: Ahlfors regular conformal dimension of $(K,d)$ --- Definition \ref{defn.AR}-\ref{it:ARCdim}
	\item $\dist_{d}(A,B)$: distance between $A,B \subseteq X$ in $d$ for a metric space $(X,d)$ --- Notation \ref{notation.intro}-\ref{it:metric-sp}
	\item $d_{\mathrm{f}}(\bm{\rho})$ --- Proposition \ref{prop:Rp-geom}-\ref{it:Rp.AR}
	\item $\hdim$: Hausdorff dimension of $(K,d)$ --- Assumption \ref{assum.BFss}
	\item $\pwalk$: $p$-walk dimension --- \eqref{eq:pwalk} in Definition \ref{d:values}
	\item $d_{\widetilde{M}_{p}(A)}$: metric on $\widetilde{M}_{p}(A)$ of local uniform convergence in $\mathbb{R}^{A}/\mathbb{R}\indicator{A}$ --- Definition \ref{defn.CGQ-Qp}
	\item $E_{n}^{\ast}, E_{n}^{\ast}(A)$: edge sets of $T_{n}$ --- Definition \ref{defn.h-networks} 
	\item $\mathcal{E}(f;g)$: derivative of $\frac{1}{p}\mathcal{E}$ for a $p$-energy form $(\mathcal{E},\mathcal{F})$ --- \eqref{exist-deriva} in Theorem \ref{thm.p-form}
	\item $\mathfrak{E}_{p}(\mathcal{F})$: the set of $p$-energy functionals $\mathcal{E}$ on $\mathcal{F}$ --- line before Definition \ref{defn.ssenergyoperator}
	\item $\mathcal{E}|_{B}$: trace of $\mathcal{E}$ on $B$ --- \eqref{eq:dfn-trace} in Theorem \ref{thm.RF-exist}
	\item $(\mathcal{E}_{p,\ast},\mathcal{F}_{p,\ast})$: limit $p$-resistance form on $V_{\ast}$ --- \eqref{defn:Fpast}, \eqref{defn:Epast} in Proposition \ref{prop.RFVast} 
	\item $(\mathcal{E}_{\mathcal{S}},\mathcal{F}_{\mathcal{S}})$: the $p$-resistance form associated to a compatible sequence $\mathcal{S}$ --- \eqref{e:defn.compat.dom}, \eqref{e:defn.compat.form} in Theorem \ref{thm.Epcountable} 
	\item $(\overline{\mathcal{E}},\overline{\mathcal{F}})$: the completion of a $p$-resistance form $(\mathcal{E},\mathcal{F})$ --- \eqref{e:defn.completion.dom}, \eqref{e:defn.completion.form} in Theorem \ref{thm.Rpcompletion} 
	\item $(\overline{\mathcal{E}}_{\mathcal{S}},\overline{\mathcal{F}}_{\mathcal{S}})$: the completion of $(\mathcal{E}_{\mathcal{S}},\mathcal{F}_{\mathcal{S}})$ --- \eqref{e:defn.compatext.dom}, \eqref{e:defn.compatext.form} in Corollary \ref{cor.Epext} 
	\item $\mathcal{E}_{p,A}^{n}, \mathcal{E}_{p}^{n}$: discrete $p$-energy form on  $T_n$ --- Definition \ref{defn.con-const}-\ref{it:sub.discreteform}
	\item $\mathcal{E}_{M, p, k}$: conductance constant --- Definition \ref{defn.con-const}-\ref{it:con-const}
	\item $\widetilde{\mathcal{E}}_{p,A}^{n}, \widetilde{\mathcal{E}}_{p}^{n}$: rescaled discrete $p$-energy form on $T_n$ --- Definition \ref{d:Kig-sob}-\ref{it:defn.rescaled-discrete-p-energy}
	\item $F_{w}$: self-similarity mapping corresponding to a word $w$ --- Definition \ref{d:shift}-\ref{it:FwKw}
	\item $F_{w}^{\ast}$: pulling-back operator by $F_{w}$ --- Remark \ref{rmk:pullback}
	\item $\CoreClosure$: the closure of $\mathcal{F} \cap C(K)$ in $\mathcal{F}$ --- the first paragraph of Subsection \ref{subsec:ext-ssem}
	\item $\mathcal{F}_{\mathrm{loc}}^{0}(U)$: the set of functions belonging locally in $U$ to $\CoreClosure$ --- Definition \ref{defn.Flocal-ss}-\ref{it:Flocal-ss-sspem}
	\item $\mathcal{F}_{\mathrm{loc}}(U)$: the set of functions belonging locally in $U$ to $\mathcal{F}$ --- Definition \ref{defn.RFlocal}-\ref{it:RF.localDir} 
	\item $\mathcal{F}|_{B}$: the set of the restrictions to $B$ of functions in $\mathcal{F}$ --- Definition \ref{dfn:restrdom}
	\item $\mathcal{F}^{0}(B)$: the set of functions in $\mathcal{F}$ that are zero on the complement of $B$ --- Definition \ref{defn.RFp-general} 
	\item $\mathcal{F}_{e}$: extended Dirichlet space --- \eqref{defn.extDF} in Definition \ref{defn.extDF}
	\item $\GSC(D,l,S)$ --- Framework \ref{frmwrk:GSC} 
	\item $\mathcal{G}_{0}$: the set of Euclidean isometries preserving $V_0$ --- \eqref{eq:GSC-isometry} in Definition \ref{dfn:GSC-isometry}, \eqref{eq:SG-isometry} in Definition \ref{dfn:SG-isometry}
	\item $\mathcal{G}_{1}$ --- \eqref{eq:GSC-isometry-subgroup} in Definition \ref{dfn:GSC-V00V01-isometry-subgroup}-\ref{it:GSC-isometry-subgroup}
	\item $\mathcal{G}_{s} = \mathcal{G}_{s}(\mathcal{L})$: a variant of $\mathcal{G}_{\mathrm{sym}}$ --- Remark \ref{rmk:defn.ANF}
	\item $\mathcal{G}_{\textrm{sym}} = \mathcal{G}_{\mathrm{sym}}(\mathcal{L})$: group of symmetries of a self-similar set $\mathcal{L}$ --- Definition \ref{defn.ANF}-\ref{it:symgroup}
	\item $g_{xy}$: reflection in the hyperplane $H_{xy}$ --- Definition \ref{defn.ANF}-\ref{it:gxy} 
	\item $\mathcal{H}_{\mathcal{E},B}$: the set of functions that are $\mathcal{E}$-harmonic on the complement of $B$ --- Definition \ref{dfn:part-harmonic}
	\item $h_{B}^{\mathcal{E}}[u]$: $\mathcal{E}$-harmonic extension of $u$ from $B$ --- Theorem \ref{thm.RF-exist}
	\item $H_{xy}$: perpendicular bisector hyperplane between $x$ and $y$ --- Definition \ref{defn.ANF}-\ref{it:gxy}
	\item $\id_{X}$: identity map of a set $X$ --- Notation \ref{notation.intro}-\ref{it:id-indicator}
	\item $K_{w}$: image of $F_{w}$ --- Definition \ref{d:shift}-\ref{it:FwKw}
	\item $\mathcal{L} = (K,S,\{ F_i \}_{i \in S})$: self-similar structure --- Definition \ref{d:sss}
	\item $L^{0}(X,\mathcal{B},m)$: the set of $m$-equivalence classes of $\mathcal{B}$-measurable functions on $X$ --- \eqref{L0-dfn}
	\item $L^{0}(X,m)$: abbreviation of $L^{0}(X,\mathcal{B},m)$ --- paragraph before \eqref{L0-dfn}
	\item $\{ l_{i} \}_{i = 0}^{m_{\ast} - 1}$: increasing sequence of the values of Euclidean distances between distinct points of $V_{0}$ --- Definition \ref{defn.ANF}-\ref{it:distances-V0-ssset} 
	\item $\mathcal{M}_{p}(A), \widetilde{\mathcal{M}}_{p}(A)$ --- Definition \ref{defn.CGQ-Mp} 
	\item $m|_{A}$: restriction of a measure $m$ to a measurable subset $A$ --- Notation \ref{notation.intro}-\ref{it:measure-sp}
	\item $\mathfrak{m}_{\mathcal{E}}^{(n)}\langle f \rangle, \mathfrak{m}_{\mathcal{E}}\langle f \rangle$ --- paragraph before Proposition \ref{prop.sspem-pre} 
	\item $\mathfrak{m}_{\mathcal{E}}\langle f; g \rangle$: derivative of $\mathfrak{m}_{\mathcal{E}}\langle \,\cdot\, \rangle$ --- \eqref{e:defn.emss.pre} in Proposition \ref{prop.sspem-pre}-\ref{it:sspem-pre.signed}
	\item $m_{\ast}$: number of the values of Euclidean distances between distinct points of $V_{0}$ --- Definition \ref{defn.ANF}-\ref{it:distances-V0-ssset}
	\item $\mathcal{N}_{p}(f)$ --- Definition \ref{d:Kig-sob}-\ref{it:NpWp}
	\item $\osc_{X}[u]$: oscillation over $X$ of a function $u$ on a set $X$ --- Notation \ref{notation.intro}-\ref{it:id-indicator}
	\item $\mathcal{P}_{\mathcal{L}}$: critical set of $\mathcal{L}$ --- \eqref{e:C-P} in Definition \ref{d:V0Vstar}-\ref{it:C-P}
	\item $P_{n,k}f$ --- Definition \ref{defn.nei-const}-\ref{it:Pnkf}
	\item $P_{n}f$ --- Definition \ref{d:Kig-sob}-\ref{it:Pnf}
	\item $\mathcal{Q}_{p}'(A), \mathcal{Q}_{p}(A)$ --- Definition \ref{defn.CGQ-Qp}-\ref{it:Qp} 
	\item $\mathbb{R}\indicator{X}$: $\mathbb{R}$-linear space of constant functions on a set $X$ --- Notation \ref{notation.intro}-\ref{it:id-indicator}
	\item $R_{\mathcal{E}}$: (pre-)$p$-resistance associated to a $p$-resistance form $(\mathcal{E},\mathcal{F})$ --- \eqref{R-def} in Definition \ref{defn.RFp}
	\item $\widehat{R}_{p,\mathcal{E}}$: $p$-resistance metric of $(\mathcal{E},\mathcal{F})$ --- Definition \ref{defn:pResmet}
	\item $\mathcal{R}_{\bm{\rho}_{p}}(E)$: renormalization operator --- \eqref{defn.renorm} in Definition \ref{defn.renorm-op}
	\item $\SG(D,l,S)$: $D$-dimensional level-$l$ Sierpi\'{n}ski gasket --- Framework \ref{frmwrk:SG}
	\item $S(w), S^{k}(w)$: sets of children of $w$ in a rooted tree $T$ --- Definition \ref{defn.tree-notation}-\ref{it:tree-vertices} 
	\item $\mathcal{S}_{p}(A)$ --- Definition \ref{defn.CGQ-Qp}-\ref{it:Sp}
	\item $\mathcal{S}_{\bm{\rho},n}$ --- \eqref{e:defn.renorm} in Definition \ref{defn.ssenergyoperator}
	\item $\sgn(a)$: the sign of $a$ --- Notation \ref{notation.intro}-\ref{it:sgn}
	\item $\supp_{X}[u]$: topological support in $X$ of $u \in \contfunc(X)$ --- Notation \ref{notation.intro}-\ref{it:Borel-contfunc}
	\item $\supp_{X}[m]$: topological support in $X$ of a Borel measure $m$ on $X$ --- Notation \ref{notation.intro}-\ref{it:support}
	\item $\supp_{m}[f]$: $m$-support of (the $m$-equivalence class of) a measurable function $f$ --- Notation \ref{notation.intro}-\ref{it:support}
	\item $T_{n}$: the set of elements in the $n$-th generation in $T$ --- Definition \ref{defn.tree-notation}-\ref{it:tree-levels}
	\item $T_{k}^{(r)}$: the set of elements in the $k$-th generation in $T^{(r)}$ --- Definition \ref{d:scale}
	\item $(T^{(r)},E_{T^{(r)}})$: the rooted tree associated with $\Lambda_{r^k}^{g}$ --- Definition \ref{d:scale}
	\item $\mathcal{U}_{p}^{\mathrm{GC}}$: the set of $p$-homogeneous functionals satisfying the generalized $p$-contraction property --- Definition \ref{d:GC-cone}
	\item $\mathcal{U}_{p,T}$: the set of $p$-homogeneous functionals contractive with respect to $T$ --- Definition \ref{d:GC-cone}
	\item $U_{M}^{\widehat{R}_{p}}(x,s)$ --- Definition \ref{defn.scale-Rp}-\ref{it:scale-Rp-UMxs}
	\item $\mathcal{U}_{p}(\mathcal{D})$ --- Definition \ref{d:Kig-sob}-\ref{it:defn.basecone}
	\item $U_{g,w}$ --- Proposition \ref{prop.ANF-geom}-\ref{it:ANF.symcompos} 
	\item $U_{M}^{d}(x,s)$ --- Assumption \ref{assum.ANF}
	\item $V_{0}$: ``boundary'' of $K$ for a self-similar structure $\mathcal{L} = (K,S,\{ F_i \}_{i \in S})$ --- Definition \ref{d:V0Vstar}-\ref{it:V0Vstar} 
	\item $V_{n}, V_{\ast}$ --- Definition \ref{d:V0Vstar}-\ref{it:V0Vstar} 
	\item $V_{n,\Lambda}$ --- Proposition \ref{prop.compatible}  
	\item $W_{n}$: set of words of length $n$ --- Definition \ref{d:shift}-\ref{it:word}
	\item $W_{\ast}$: set of words of finite length --- Definition \ref{d:shift}-\ref{it:word}
	\item $\mathcal{W}^{p}$ --- Definition \ref{d:Kig-sob}-\ref{it:NpWp}
	\item $\Gamma\langle f \rangle$: $p$-energy measure dominated by a $p$-energy form $(\mathcal{E},\mathcal{F})$ --- Definition \ref{defn.em-Cp}
	\item $\Gamma\langle f;g \rangle$: derivative of $\Gamma\langle \,\cdot\, \rangle$ --- \eqref{e:defn.emss} in Definition \ref{defn.emss}
	\item $\Gamma_{\mathcal{E}}\langle f \rangle$: self-similar $p$-energy measure of $f$ associated with a self-similar $p$-energy form $(\mathcal{E}, \mathcal{F})$ --- \eqref{e:defn.em.one}
	\item $\Gamma_{\mathcal{E}}\langle f; g \rangle$: derivative of $\Gamma_{\mathcal{E}}\langle \,\cdot\, \rangle$ --- \eqref{ss-emform} in Proposition \ref{prop.ss-pform-em}-\ref{ssem-Cp}
	\item $\Lambda_{\rho}^{n}(E)$ --- \eqref{RFrenorm} in Proposition \ref{prop.RFrenorm}
	\item $\Lambda_{s}^{\widehat{R}_{p}}$: scale associated with $\widehat{R}_{p}$ --- Definition \ref{defn.scale-Rp}-\ref{it:scale.res}
	\item $\Lambda_{s,M}^{\widehat{R}_{p}}(x)$ --- Definition \ref{defn.scale-Rp}-\ref{it:scale-Rp-UMxs}
	\item $\Lambda_{s}^{d}$,$\Lambda_{s,M}^{d}$ --- Assumption \ref{assum.ANF}
	\item $\Lambda_{r^{k}}^{g}$ --- Definition \ref{d:scale}
	\item $\Lambda_{\bm{\rho}_{p}}(E)$ --- \eqref{defn.renorm} in Definition \ref{defn.renorm-op}
	\item $\lambda(\bm{\rho}_{p})$: eigenvalue of $\mathcal{R}_{\bm{\rweight}_{p}}$ --- Theorem \ref{thm.eigenform}-\ref{it:CGQ.eigenvalue},\ref{it:CGQ.newinitial},\ref{it:CGQ.eigenform}
	\item $\pi$: map giving the parent vertex of each vertex in a rooted tree $T$ --- Definition \ref{defn.tree-notation}-\ref{it:tree-vertices}
	\item $\bm{\rweight} = (\rweight_{i})_{i \in S}$: weight of a self-similar $p$-energy form --- Definition \ref{defn.ssform}
	\item $\Sigma_{w}$: the set of elements of $\Sigma$ descended from $w$ --- Definitions \ref{d:shift}-\ref{it:FwKw}, \ref{defn.tree-notation}-\ref{it:tree-bdry-sections}
	\item $\Sigma_{A}$: the union of $\Sigma_{w}$ over $w \in A$ --- Definition \ref{defn.tree-notation}-\ref{it:tree-bdry-sections}
	\item $\sigma_{w}$: the map from $\Sigma = S^{\mathbb{N}}$ to itself adding a word $w$ in front --- Definition \ref{d:shift}-\ref{it:shift},\ref{it:FwKw} 
	\item $\sigma_{p,k}(A)$: neighbor disparity constant --- Definition \ref{defn.nei-const}-\ref{it:nei-const}
	\item $\sigma_{p,k,n}^{\mathscr{J}},\sigma_{p,k}^{\mathscr{J}}$ --- Definition \ref{defn.nei-const}-\ref{it:sigmaJpk}
	\item $\sigma_{p}$: $p$-scaling factor --- \eqref{p-factor} in Theorem \ref{t:pCH}
	\item $\bm{\sigma}_{p} = (\sigma_{p}^{j_s})_{s \in S}$: weight determined by the $p$-scaling factor --- Theorem \ref{thm.KSgood-ss}
	\item $\tau_p$ (for a triple $(K,d,\{ K_{w} \}_{w \in T})$) --- \eqref{e:defn.taup} in Definition \ref{d:values}
	\item $\tau_{g}$: canonical action of $g \in \mathcal{G}_{\mathrm{sym}}$ on $W_{\ast}$ --- Definition \ref{defn.ANF}-\ref{it:symgroup} 
	\item $\chi$: canonical projection map from $\Sigma = S^{\mathbb{N}}$ to $K$ for a self-similar structure $\mathcal{L} = (K,S,\{ F_i \}_{i \in S})$ --- Definition \ref{d:sss} 
	\item $[\omega]_{n}$: initial length-$n$ word or $n$-th level vertex of $\omega \in \Sigma$ --- Definitions \ref{d:shift}-\ref{it:shift}, \ref{defn.tree-notation}-\ref{it:tree-bdry-sections}
	\item $\partial_{k}\Phi$: first-order partial derivative of $\Phi$ in the $k$-th coordinate --- Notation \ref{notation.intro}-\ref{it:ellp-norm}
	\item $\nabla\Phi$: gradient of $\Phi$ --- Notation \ref{notation.intro}-\ref{it:ellp-norm}
	\item $\fint_{A}f\,dm$: mean value of an $m$-integrable function $f$ on a measurable subset $A$ --- Notation \ref{notation.intro}-\ref{it:measure-sp}
	\item $\abs{x}$: Euclidean norm of $x \in \mathbb{R}^{n}$ --- Notation \ref{notation.intro}-\ref{it:ellp-norm}
	\item $\abs{w}$: length of a word $w \in W_{\ast}$ or level of a vertex $w$ in a rooted tree $T$ --- Definitions \ref{d:shift}-\ref{it:word}, \ref{defn.tree-notation}-\ref{it:tree-levels}
	\item $\abs{\,\cdot\,}_{\widetilde{M}_{p}(A)}$ --- Definition \ref{defn.CGQ-Qp}
	\item $\norm{u}_{\sup} = \norm{u}_{\sup,X}$: supremum of $\abs{u}$ over $X$ for a function $u$ on a set $X$ --- Notation \ref{notation.intro}-\ref{it:id-indicator}
	\item $\norm{\,\cdot\,}_{\mathcal{E},\alpha}$: $(\mathcal{E},\alpha)$-norm --- \eqref{e:defn-e1norm} in Definition \ref{d:e1norm} 
	\item $\norm{x}_{\ell^{p}_{n}} = \norm{x}_{\ell^{p}}$: $\ell^{p}$-norm of $x \in \mathbb{R}^{n}$ --- Notation \ref{notation.intro}-\ref{it:ellp-norm}
	\item $\norm{\,\cdot\,}_{L^{p}(X,m)}$: norm of $L^{p}(X,m)$
	\item $\norm{\,\cdot\,}_{\mathcal{W}^{p}}$ --- Definition \ref{d:Kig-sob}-\ref{it:NpWp}
	\item $\langle \,\cdot\,,\,\cdot\,\rangle_{L^{2}(X,m)}$: inner product of $L^{2}(X,m)$
	\item $\underset{\textrm{QS}}{\sim}$: equivalence relation given by quasisymmetry on the set of metrics on a set --- Definition \ref{defn.AR}-\ref{it:QS}
	\item $\leq$ (for partitions of $\Sigma = S^{\mathbb{N}}$): one is a refinement of the other --- Definition \ref{d:shift}-\ref{it:partition}
	\item $\lesssim$: inequality up to positive constant multiple --- Notation \ref{notation.intro}-\ref{it:ineq-up-to-const}
	\item $\asymp$: comparability up to positive constant multiple --- Notation \ref{notation.intro}-\ref{it:ineq-up-to-const}
	\item $\ll$: absolutely continuity of measures --- Notation \ref{notation.intro}-\ref{it:measure-sp}
\end{enumerate}

\section*{Index of abbreviations}
\addcontentsline{toc}{section}{Index of abbreviations}
\begin{enumerate}[label=\textup{$\bullet$},align=left,leftmargin=*,topsep=2pt,parsep=0pt,itemsep=2pt]
	\item \textup{(\textbf{A})}: a (necessary and sufficient) condition for the existence of an eigenform --- Theorem \ref{thm.eigenform} 
	\item \textup{(\textbf{A}')}: the existence of an eigenform with eigenvalue one --- paragraph following Remark \ref{rmk.condA} 
	\item \textup{(CF)$_{m}$}: a condition on cutoff functions for $p$-energy forms --- Definition \ref{defn.CF-measure}
	\item \textup{(Cha1)},\textup{(Cha2)}: chain rules for $p$-energy measures --- Definition \ref{defn.chainrule} 
	\item \textup{(Cla)$_{p}$}: $p$-Clarkson's inequality for $p$-energy forms --- Definition \ref{d:Cp}
	\item \textup{(Cla)$^{\prime}_{p}$}: weaker version of $p$-Clarkson's inequality --- Remark \ref{rmk:Cp-weak}
	\item \textup{(Cla)$^{\textup{EM}}_{p}$}: $p$-Clarkson's inequality for $p$-energy measures --- Definition \ref{defn.em-Cp}
	\item \textup{(EM1)$_{p}$},\textup{(EM2)$_{p}$}: requirements for $\{ \Gamma\langle f \rangle \}_{f \in \mathcal{F}}$ to be a family of $p$-energy measures dominated by $(\mathcal{E},\mathcal{F})$ --- Definition \ref{defn.em-Cp}
	\item \textup{(GC)$_{p}$}: generalized $p$-contraction property for $p$-energy forms --- Definition \ref{defn.GC}
	\item \textup{(GC)$^{\textup{EM}}_{p}$}: generalized $p$-contraction property for $p$-energy measures --- Definition \ref{defn.em-Cp}
	\item $\GSC(D,l,S)$ --- Framework \ref{frmwrk:GSC}
	\item \textup{(GSC1)},\textup{(GSC2)},\textup{(GSC3)},\textup{(GSC4)}: requirements for $\GSC(D,l,S)$ to be a generalized Sierpi\'nski carpet --- Definition \ref{dfn:GSC} 
	\item \textup{(\textbf{R})}: a condition for $(\rweight_{p,i})_{i \in S}$ --- paragraph following Remark \ref{rmk.condA}
	\item \textup{(RF1)$_{p}$},\textup{(RF2)$_{p}$},\textup{(RF3)$_{p}$},\textup{(RF4)$_{p}$},\textup{(RF5)$_{p}$}: requirements for $(\mathcal{E},\mathcal{F})$ to be a $p$-resistance form --- Definition \ref{defn.RFp}
	\item $\SG(D,l,S)$: $D$-dimensional level-$l$ Sierpi\'{n}ski gasket --- Framework \ref{frmwrk:SG}
	\item \textup{(SL1)},\textup{(SL2)}: strong local properties of $p$-energy forms --- Definition \ref{defn.Epsl}
	\item \textup{(SL1)$_{\textup{s}}$},\textup{(SL2)$_{\textup{s}}$},\textup{(SL1)$_{\textup{w}}$},\textup{(SL2)$_{\textup{w}}$}: strong local properties of $p$-resistance forms --- Definition \ref{defn.RF-sl}
	\item \textup{(SS1)},\textup{(SS2)},\textup{(SS3)},\textup{(SS4)}: requirements for $\mathcal{L}$ to be a strongly symmetric p.-c.f.\ self-similar set --- Definition \ref{defn.ANF}-\ref{it:StrongSym}
\end{enumerate}

\printindex

\end{document}